\documentclass[10pt,a4paper]{amsbook}

\usepackage{amsmath,amsthm,amsfonts,amssymb,xcolor,mathtools,enumitem,comment,mathrsfs,subdepth}
\usepackage{hyperref}
\usepackage{microtype}
\usepackage[open=true,openlevel=2,depth=3]{bookmark}
\usepackage[utf8]{inputenc}

\newcommand{\R}{\mathbb R}
\newcommand{\C}{\mathbb C}

\newcommand{\GG}{\mathfrak{G}}
\newcommand{\UU}{\mathfrak{U}}
\newcommand{\N}{\mathbb{N}}
\newcommand{\G}{\mathcal{G}}
\newcommand{\Hh}{\mathcal{H}}
\newcommand{\Gs}{\mathcal{G}^s}
\newcommand{\U}{\mathcal{U}}
\newcommand{\J}{\mathcal{J}}
\renewcommand{\O}{\mathcal{O}}

\newcommand{\va}[1]{\left|#1\right|}
\newcommand{\set}[1]{\left\{#1\right\}}
\newcommand{\p}[1]{\left( #1 \right)}
\newcommand{\n}[1]{\left\| #1 \right\|}
\newcommand{\jap}[1]{\left\langle #1 \right\rangle}
\newcommand{\re}{\textup{Re}}

\newcommand{\FBI}{\mathrm{FBI}}

\let\Im\relax
\let\Re\relax
\DeclareMathOperator{\Im}{Im}
\DeclareMathOperator{\Re}{Re}
\DeclareMathOperator{\Op}{Op}
\DeclareMathOperator{\Tr}{Tr}
\DeclareMathOperator{\WF}{WF}
\DeclareMathOperator{\vol}{vol}
\DeclareMathOperator{\Dom}{Dom}

\newtheorem{theorem}{Theorem}
\newtheorem{prop}{Proposition}[chapter]
\newtheorem{lemma}{Lemma}[chapter]
\newtheorem{reduc}{Reduction}
\newtheorem{corollary}{Corollary}[chapter]

\theoremstyle{definition}
\newtheorem{definition}{Definition}[chapter]
\newtheorem{remark}{Remark}[chapter]

\author{Yannick Guedes Bonthonneau}
\address{LAGA, CNRS, Université Sorbonne Paris Nord}

\author{Malo Jézéquel}
\address{CNRS, Univ Brest, UMR 6205, Laboratoire de Mathématiques de Bretagne Atlantique, France}

\title{FBI transform in Gevrey classes and Anosov flows}

\begin{document}

\maketitle

\tableofcontents

\chapter*{Introduction}

\subsection*{Spectral theory and Anosov flows}

Dimitri V. Anosov introduced the flows that bear his name in \cite{anosovGeodesicFlowsClosed1967}. He wanted to study geodesic flows on the unit tangent bundle of compact Riemannian manifolds with (\emph{a priori} non-constant) negative sectional curvature. Since then, Anosov flows (and more generally smooth uniformly hyperbolic dynamics) have been widely studied by numerous authors using a large variety of different tools (Markov partitions, specification, tower constructions, coupling methods, \dots). In the last two decades, the so-called functional analytic approach to statistical properties of uniformly hyperbolic dynamics has been a very active field of research, due to the introduction in dynamical systems of the notion of \emph{Banach spaces of anisotropic distributions} (see \cite{quest,demersGentleIntroductionAnisotropic2018} for surveys on this subject).

Let us recall the basic ideas behind this approach. Let $\p{\phi_t}_{t \in \R}$ denote a $\mathcal{C}^\infty$ Anosov flow on a compact $\mathcal{C}^\infty$ manifold $M$ (the precise definition of an Anosov flow is recalled in the introduction of Chapter \ref{part:Anosov-flow}). As a dynamical system, $\phi_t$ is very chaotic, so much so that it is hopeless to try to describe the long-time behaviour of all of its orbits. The complexity of the pointwise dynamics suggests that we should consider instead the action of the flow $\phi_t$ on objects that are reminiscent of the smooth structure of $M$. This is one of the motivations for the introduction of the \emph{Koopman operator} defined, for $t \geq 0$, by 
\begin{equation}\label{eq:Koopman_the_first}
\begin{split}
\mathcal{L}_t : u \longmapsto u \circ \phi_t,
\end{split}
\end{equation}
acting for instance on $\mathcal{C}^\infty\p{M}$ or on the space $\mathcal{D}'\p{M}$ of distributions on $M$. The adjoint of the Koopman operator is the well-known \emph{transfer operator}
\[
 f \longmapsto \frac{f\circ\phi_{-t}}{|\det \mathrm{d}\phi_t| \circ \phi_{-t}}.
\]
Our aim is now to understand the asymptotics of the operator $\mathcal{L}_t$ when $t$ tends to $+ \infty$. The advantage of this approach is that if $u \in \mathcal{C}^\infty\p{M}$, we have:
\begin{equation*}
\begin{split}
\frac{\mathrm{d}}{\mathrm{d}t}\p{\mathcal{L}_t u} = X\p{\mathcal{L}_t u},
\end{split}
\end{equation*}
where $X$ denotes the generator of the Anosov flow $\phi_t$ (we identify $X$ with its Lie derivative). That is, we have replaced the non-linear dynamics of $\phi_t$ on the finite-dimensional manifold $M$ by a linear ODE on an infinite-dimensional space ($\mathcal{C}^\infty\p{M}$ for instance), which is presumably easier to understand. Indeed, we expect that, as in the case of finite-dimensional linear ODEs, the long-time behaviour of the operator $\mathcal{L}_t$ should be ruled by the spectrum of the differential operator $X$. 

Following this observation, much effort was invested in developing a spectral theory for Anosov flows. The first decisive steps in this field were made by David Ruelle \cite{Ruelle-86} and Mark Pollicott \cite{Pollicott-85}, so that \emph{Ruelle--Pollicott resonances} -- that will be defined below -- now bear their names. The main difficulty is that the spectrum of $X$ acting on $L^2\p{M}$ is in general very wild, in sharp contrast with the case of elliptic differential operators such as the Laplacian on a compact Riemannian manifold.

It is only with the introduction of the Banach spaces of anisotropic distributions that this conundrum was completely resolved. They are functional spaces tailored to fit the dynamics of $\phi_t$. Using the spectral properties of $X$ acting on such spaces, we may define a discrete spectrum for $X$ called \emph{Ruelle--Pollicott spectrum} (see the introduction of Chapter \ref{part:Anosov-flow} for more details). The article of Michael Blank, Gerhard Keller, and Carlangelo Liverani \cite{BKL} is often considered the first appearance of Banach spaces of anisotropic distributions in the dynamical literature. However the notion was already implicitly present in the works of Hans Henrik Rugh \cite{R1,R2} and David Fried \cite{Fried-95} in the analytic category. A consequence of the theory of Banach spaces of anisotropic distributions is that the resolvent of $X$, defined for $\Re z \gg 1$ and $u \in \mathcal{C}^\infty\p{M}$ by
\begin{equation}\label{eq:notre_amie_la_resolvante}
\begin{split}
\p{z- X}^{-1} u = \int_0^{+ \infty} e^{-z t} \mathcal{L}_t u \mathrm{d}t,
\end{split}
\end{equation}
admits, as an operator from $\mathcal{C}^\infty\p{M}$ to $\mathcal{D}'\p{M}$, a meromorphic continuation to $\C$ with residues of finite ranks. The poles of $\p{z-X}^{-1}$ are called the \emph{Ruelle--Pollicott resonances} of $X$ or $\phi_t$ (the rank of the residue being the multiplicity of the resonance). These resonances are powerful tools in the study of the statistical properties of the flow $\phi_t$ -- for example see the article of Oliver Butterley \cite{Butterley-16}.

In some particular cases, mostly of algebraic nature, it is possible to obtain quite detailed informations on the distribution of these resonances (geodesic flows in constant negative curvature or constant-time suspensions of cat maps for instance). However, in the general case, there is no reason that we should be able to obtain very precise results, much less compute explicitly the values of the Ruelle--Pollicott resonances of $X$. We have to settle for qualitative counting estimates.

It is well-known that much information about the global dynamics of $\phi_t$ may be retrieved from its periodic orbits -- description of invariant measures, relation between the pressure and Poincar\'e series, etc. It may not come as a surprise then that the Ruelle--Pollicott spectrum can be completely determined from the periodic orbits of $\phi_t$, via a sort of zeta function, as we explain now.

For each periodic orbit $\gamma$, let $T_\gamma$ denote (resp. $T_\gamma^\#$) its length (resp. primitive length). We also denote by $\mathcal{P}_\gamma$ its linearized Poincaré map, i.e. $\mathcal{P}_\gamma = \mathrm{d}\varphi_{T_\gamma}(x)_{|_{E_u \oplus E_s}}$ for some $x$ in $\gamma$ (the definition of the stable and unstable directions $E_s$ and $E_u$ is recalled in the introduction of Chapter \ref{part:Anosov-flow}). These data are gathered in the \emph{dynamical determinant}, defined for $\Re z \gg 1$ by
\begin{equation}\label{eq:def-det-dyn}
\zeta_{X}(z) \coloneqq \exp\left( -\sum_\gamma \frac{T_\gamma^{\#}}{T_\gamma}\frac{e^{-z T_\gamma} }{\va{\det(I- \mathcal{P}_\gamma)}}\right).
\end{equation}
Since the flow $\phi_t$ is $\mathcal{C}^\infty$, it follows from the work of Paolo Giulietti, Liverani and Pollicott \cite{GLP} (see also the article of Semyon Dyatlov and Maciej Zworski \cite{DZdet}) that $\zeta_X$ extends to a holomorphic function on $\C$, whose zeros \emph{are} the Ruelle--Pollicott resonances of $X$. It is then of the utmost importance to understand the complex analytic properties of $\zeta_X$ in order to apprehend the distribution of Ruelle--Pollicott resonances. On a historical note, the Ruelle--Pollicott resonances were originally introduced by Ruelle and Pollicott as the zeros of some zeta function closely related to \eqref{eq:def-det-dyn}. 

The regularity of the flow $\phi_t$ plays an important role in the approach described above. Indeed, if the flow $\phi_t$ were only $\mathcal{C}^k$ for some $k > 1$, then the meromorphic continuations of the resolvent \eqref{eq:notre_amie_la_resolvante} and of the dynamical determinant \eqref{eq:def-det-dyn} would \emph{a priori} only be defined on a half plane of the form $\set{z \in \C : \Re z > - A}$, with $A$ that goes to $+ \infty$ when $k$ tends to $+ \infty$ (see \cite{port}). On the contrary, if the flow $\phi_t$ is real-analytic then Rugh \cite{R1,R2} (in dimension $3$) and Fried \cite{Fried-95} (in higher dimension) are able to prove a bound on the growth of $\zeta_X$ that implies that its order is less than the dimension of $M$. What has to be understood here is that the functional analytic tools available in the real-analytic category are usually stronger than those available in the $\mathcal{C}^\infty$ category. These tools have also been used in hyperbolic dynamics to understand other systems than Anosov flows: this is the case in the pioneering work of Ruelle \cite{Ruelle-76} for instance. Functional analytic tools from the real-analytic category have been used a lot recently, when studying for instance expanding maps or perturbations of cat maps (see for example\cite{faureRuellePollicottResonancesReal2006,naudRuelleSpectrumGeneric2012,adamGenericNontrivialResonances2017,slipantschukCompleteSpectralData2017,bandtlowLowerBoundsRuelle2019}).

This situation suggests to study Anosov flows in regularity classes that are intermediate between $\mathcal{C}^\infty$ and real-analytic. The hope here is that certain properties that fail for $\mathcal{C}^\infty$ Anosov flow but hold for real-analytic ones could in fact be true in larger classes of regularity. Following this idea, the second author was able to establish a trace formula for certain ultradifferentiable Anosov flows in \cite{jezequelGlobalTraceFormula2019}. In this paper, we focus on the Gevrey classes of regularity. We will give later in this introduction a precise definition, but let us just recall for now that these are classes of regularity indexed by a parameter $s \geq 1$ that have been introduced by Maurice Gevrey \cite{Gevrey-18} in order to study the regularity of the solutions of certain PDEs. There are several reasons for which we decided to work with this particular level of regularity. First, it is convenient that when $s= 1$ the Gevrey regularity is in fact the real-analytic regularity: we will be able to confront our results with those available in the literature. Moreover, the Gevrey classes of regularity are notoriously very well-behaved, so that we could hope to develop practicable tools in this setting. Secondly and more importantly, we were interested in a very particular question: the finiteness of the order of the dynamical determinant \eqref{eq:def-det-dyn}, and we thought, for heuristic reasons explained below, that the Gevrey regularity was the good setting to understand this question. 
 
These expectations have been met, as our main result writes:
\begin{theorem}\label{thm:main}
Let $s \in \left[1, + \infty\right[$. Let $M$ be an $n$-dimensional compact $s$-Gevrey manifold endowed with a $s$-Gevrey Anosov flow $\p{\phi_t}_{t \in \R}$ generated by a vector field $X$. Then there is a constant $C>0$ such that for every $z\in \C$, 
\[
|\zeta_X(z)| \leq C \exp\left( C |z|^{ns}\right).
\]
In particular, $\zeta_X$ has finite order less than $ns$.
\end{theorem}
 
As already mentioned, when $s=1$, the Gevrey regularity coincides with the real-analytic regularity. In that case, Theorem \ref{thm:main} is a consequence of the results of Rugh \cite{R1,R2} and Fried \cite{Fried-95} mentioned above. However, when $s > 1$, our results is new. As far as we know, no tools were available to study specifically Anosov flows of Gevrey regularity before the present work.

According to Jensen's Lemma, the theorem implies that the number of resonances in disks centered at zero of radius $R$ is $\O(R^{ns})$. This gives information on the counting of resonances far from the $L^2$ spectrum, or, so to speak, deep in the complex region. It is to be contrasted with the much sharper local results obtained by Faure and Johannes Sj\"ostrand \cite{FauSjo} and more recently by Faure and Masato Tsujii \cite{Faure-Tsujii-17-FractalWeylLaw}. Their results hold only for counting in boxes near the $L^2$ spectrum, but are much better upper bounds. The only available general lower bound to our knowledge is almost linear, and due to Long Jin, Zworski (and Fr\'ed\'eric Naud) \cite{ltfzwor}. In the contact flow case, there are much sharper results, see \cite{faureBandStructureRuelle2013,Datchev-Dyatlov-Zworski}.

Regarding the optimality of Theorem \ref{thm:main}, it is relevant to observe that according to Ruelle's result \cite{Ruelle-76}, if the stable and unstable foliations of the flow are themselves analytic, the bound on the order is much better than the dimension. In particular, for the geodesic flow of a hyperbolic manifold of dimension $n$, acting on the unit cosphere of dimension $2n-1$, the order of the dynamical determinant is exactly $n+1$ -- instead of the dimension $2n-1$ -- as follows from the Selberg trace formula, linking Ruelle--Pollicott resonances to eigenvalues of the Laplacian. We suspect that the use of additional microlocal refinement could improve in general our bound, in the spirit of the Fractal Weyl law from \cite{Faure-Tsujii-17-FractalWeylLaw}.

The techniques employed here are quite different from those of Rugh and Fried. They enable us to study directly the resolvent $(z-X)^{-1}$ on relevant spaces. Several results follow, that are gathered in the beginning of Chapter \ref{part:Anosov-flow}. We obtain:
\begin{itemize}
	\item a weighted version of Theorem \ref{thm:main} (see Theorem \ref{thm:main-strong} and following discussion);
    \item a characterization of the resonant states by their Gevrey wavefront set (see Proposition \ref{prop:wfs_resonant_state});
	\item a ``global'' version of the stability of Ruelle--Pollicott resonances under stochastic perturbation, proven by Dyatlov and Zworski in \cite{dyatlovStochasticStabilityPollicott2015} (see Theorem \ref{theorem:viscosite});
	\item a Gevrey version of linear response, see \S \ref{sec:linear_response}.
\end{itemize}

The last three items are new even in the analytic case. In particular, in the real-analytic case, we are able to establish an actual Kato theory for perturbations of real-analytic Anosov flows, which implies in particular Theorem \ref{theorem:SRB} below. This result is hardly surprising, but we are not aware of any proof in the literature devoted to this subject (see \cite[Corollary 1]{Katok-Knieper-Pollicott-Weiss} however for a related result).
\begin{theorem}\label{theorem:SRB}
Let $\epsilon \mapsto X_\epsilon$ be a real-analytic family of real-analytic vector fields on a real-analytic manifold $M$, defined for $\epsilon$ near $0$. We assume that $X_0$ generates an Anosov flow that admits a unique SRB measure $\mu_0$. Let $\mu_\epsilon$ denote the unique SRB measure of the Anosov flow generated by $X_\epsilon$, for $\epsilon$ near zero. Then the map
\begin{equation*}
\begin{split}
\epsilon \mapsto \mu_\epsilon \in \U^1\p{M}
\end{split}
\end{equation*}
is real-analytic on a neighbourhood of zero. 
\end{theorem}
The space $\U^1\p{M}$ of analytic ultradistributions that appears in Theorem \ref{theorem:SRB} is defined in \S \ref{sec:def_Gev}.

The proof of Theorem \ref{thm:main} relies on quite advanced technical tools, whose construction takes up most pages of this article. However, we can give a (relatively) simple heuristic argument indicating that the Gevrey regularity is the good setting when trying to establish that the dynamical determinant \eqref{eq:def-det-dyn} has finite order. To present this argument, we have to catch a very brief glimpse of the technical side of the theory for $\mathcal{C}^\infty$ Anosov flows (mentioned in the beginning of this introduction).

Let us describe here the strategy introduced by Faure, Nicolas Roy and Sj\"ostrand in \cite{Faure-Roy-Sjostrand-08} to study $\mathcal{C}^\infty$ Anosov diffeomorphisms, and adapted to flows in \cite{FauSjo}. The anisotropic spaces from \cite{FauSjo} are of the form $\Op(e^{G}) \cdot L^2\p{M}$, where $G$ is a symbol of logarithmic size called an escape function and $\Op$ is a Weyl quantization. The action of $X$ on $\Op(e^{G}) \cdot L^2\p{M}$ is equivalent to the action on $L^2$ of
\begin{equation}\label{eq:FS_en_deux_deux}
\begin{split}
\Op\p{e^{G}}^{-1} X \Op\p{e^G} = X + A_G.
\end{split}
\end{equation}
Here, the operator $A_G$ is a pseudo-differential operator of logarithmic order, and $G$ has been chosen so that the real part of the symbol of $A_G$ is negative. This can be used with the sharp G\r{a}rding inequality to invert the operator $X - \lambda$, up to a compact operator. The meromorphic continuation of \eqref{eq:notre_amie_la_resolvante} follows then from Fredholm analytic theory. However, when trying to invert the conjugated operator $\Op(e^{G})^{-1} \p{X - \lambda} \Op(e^G) = X + A_G - \lambda$, the negative sign of the real part of the symbol of $A_G$ competes with the possible positivity of $- \Re \lambda$. Fortunately, we may change $G$, and consequently make $\Re A_G$ even more negative. Hence the strategy above allows to continue $\p{X - \lambda}^{-1}$ meromorphically to $\set{\lambda \in \C : \Re \lambda > - N_G}$, for arbitrarily large $N_G$. Notice in particular that the construction from \cite{FauSjo} requires the use of a scale of spaces, since the resolvent is meromorphically continued ``strip by strip''. This feature is essential in the $C^\infty$ setting. Indeed for the method to work, one needs that the growth of $e^G$ be matched by the decay of the Fourier transform of a typical $C^\infty$ function. Somehow, to obtain a continuation to $\C$ with only one space in the $C^\infty$ case, we would need a \emph{Fréchet} anisotropic space, whose theory has yet to be written.

In order to improve the situation that we just described, one could try to make that operator $A_G$ (or rather its symbol) ``larger''. However, this requires to take $G$ with a growth faster than logarithmic. It seems then natural to wonder what happens if we work in a smaller class of regularity. Indeed, Fourier transforms of functions that are more regular than $\mathcal{C}^\infty$ should decay faster, and, morally, it should allow to soften the constraint on the growth of $G$. We recall here that the question that we want to consider is the finiteness of the order of the dynamical determinant \eqref{eq:def-det-dyn}. It is classical to derive such a bound from Schatten estimates. Thus, we would like the resolvent of the operator \eqref{eq:FS_en_deux_deux} to be Schatten. A natural way to try to achieve that is to make the operator \eqref{eq:FS_en_deux_deux} hypo-elliptic. To do so, we want the pseudo-differential operator $A_G$ to have positive order $\epsilon > 0$, but this requires that $G$ itself has order $\epsilon > 0$. Intuitively, in order to handle such a $G$, we need to work in a class of regularity in which compactly supported functions have a Fourier transform that decays like a stretched exponential $\exp( - \langle \xi \rangle^\epsilon/C)$. This is exactly the decay which characterizes the Fourier transform of $1/\epsilon$-Gevrey functions. This is why this paper is set in the framework of Gevrey classes of regularity.

We used above the construction of anisotropic spaces from \cite{FauSjo} in order to explain the motivation behind Theorem \ref{thm:main}. However, this heuristic proof could have been formulated using any other construction of anisotropic spaces and the idea would be the same: one wants to use the extra regularity to replace some weak weight by a stronger one. Unfortunately, we have not been able to use the available constructions of anisotropic Banach spaces to make the heuristic proof explained above rigorous. The main difficulty is that the weight $e^G$ is not a symbol in the usual sense when $G$ is a symbol of order $\epsilon>0$. 

The first two directions we explored turned out to be unfruitful. The constructions using local adapted Littlewood--Paley decompositions as in \cite{jezequelGlobalTraceFormula2019} give spaces for which small time dynamics of the Koopman operator \eqref{eq:Koopman_the_first} is not sufficiently well controlled. On the other hand, microlocal constructions, as the one from \cite{FauSjo} described above, could maybe be adapted, using results of Luisa Zanghirati \cite{zanghiratiPseudodifferentialOperatorsInfinite1985} and Luigi Rodino \cite{Rodino-93-book} on spaces of infinite order Gevrey pseudo-differential operators. However, this would restrict our work to the case of large $s$, and it is not clear in the first place that the spaces constructed in that way carry dynamical information. It would also not yield a bound on the order as good as the one we obtain.

Thankfully, we were able to construct satisfying spaces -- see the more precise result Theorem \ref{thm:spectral-theory}:
\begin{theorem}
Let $s \in \left[1, + \infty\right[$. Let $M$ be an $n$-dimensional compact $s$-Gevrey manifold endowed with a $s$-Gevrey Anosov flow $\p{\phi_t}_{t \in \R}$ generated by a vector field $X$. There exists a Hilbert space $\Hh$ of $s$-Gevrey ultradistributions, and a constant $C$ such that $X+C$ is invertible on $\Hh$, and $(X+C)^{-1}$ is $p$-Schatten on $\Hh$ for any $p>ns$. Additionally, $\Hh$ contains a dense subspace of $s$-Gevrey functions, and $(\mathcal{L}_t)_{t\geq 0}$ is a strongly continuous semi-group of bounded operators on $\Hh$ (with generator $X$).
\end{theorem}
Further study of the spectral properties of $X$ acting on $\Hh$ enables to prove Theorems \ref{thm:main} and \ref{theorem:SRB}.

\subsection*{FBI transform}

Our construction relies on an FBI -- for Joseph Fourier, Jacques Bros and David Iagolnitzer -- transform. This is a linear map taking functions of a space variable $x\in M$ to functions on the phase space $\alpha \in T^\ast M$. It is an integral transform whose kernel has an oscillating Gaussian behaviour.

FBI Transforms are a sort of ``localized'' Fourier transform. Such transforms are common in the linear PDE literature, and we gathered some historical remarks to give a bit of context. Maybe the first popular occurrence is the one introduced by Dennis Gabor \cite{Gabor-46} for use in signal processing. Later on, a continuous version of his construction was given, taking the form
\[
G_x(\tau,\omega) = \int_{\R} x(t) e^{-\pi (t-\tau)^2 - i \omega t} \mathrm{d}t.
\]
To this day, the Gabor transform is a staple in signal analysis. It is a particular case of a Short Time Fourier Transform whose general expression is
\[
\mathrm{STFT}\{x\}(\tau,\omega) = \int_{\R^n} x(t) w(t-\tau) e^{-i \omega t} \mathrm{d}t,
\]
where $w$ is a window function.

Another occurrence of generalized Fourier transform, more closely related to our main interest, is the so-called Bargmann--Segal transform. It was introduced by Valentine Bargmann \cite{Bargmann-61} and Irving Ezra Segal \cite{Segal-63}, as
\[
\C^n \owns z \mapsto Bf(z) = \int_{\R^n} e^{-\frac{1}{2}( \jap{z, z} - 2\sqrt{2}\jap{z, x} + \jap{x, x})} f(x)\mathrm{d}x.
\]
It realizes an isometry between $L^2(\R^n)$ and the Bargmann--Segal space of holomorphic functions $F$ on $\C^n$ such that
\[
\int_{\C^n} |F(z)|^2 e^{-|z|^2} \mathrm{d}z < \infty.
\]

Bros and Iagolnitzer introduced their ``generalized Fourier transform'' in \cite{Bros-Iagolnitzer-75}. Its form was
\begin{equation}\label{eq:original-FBI}
\mathcal{F} u(v,v_0,X) = \int_{\R^n} u(x) e^{i \langle v, x\rangle - v_0 (x-X)^2} \mathrm{d}x,
\end{equation}
(for $u$ in some reasonable space, for example, $u$ can be a tempered distribution). Their purpose was to define and study an analytic version of the wavefront set. This was based upon previous works of Iagolnitzer and Henry Stapp \cite{Iagolnitzer-Stapp-69}. At that time, another notion of analytic wavefront set had been proposed by Mikio Sato, Takahiro Kawai and Masaki Kashiwara \cite{Sato-Kawai-Kashiwara-73}. Jean Michel Bony \cite{Bony-77} then proved that these definitions were equivalent. 

Nowadays, it is more common to use a semiclassical version of the transform, defined on $\R^n$ by 
\begin{equation}\label{eq:def-flat-FBI-transform}
T_{\R^n} u(x,\xi;h):= \frac{1}{(2\pi h)^{3n/2}} \int e^{\frac{i}{h}(\langle x-x',\xi\rangle + \frac{i}{2} |x-x'|^2) } u(x')\mathrm{d}x'.
\end{equation}
In the case of manifolds, to account for non-linear changes of variable, it is more convenient to take a slightly different scaling for the phase. 

In the second chapter of this article, given a compact real-analytic manifold $M$, we will construct an FBI transform $T$, which is an operator from $\mathcal{D}'(M)$ to $\mathcal{C}^\infty(T^\ast M)$ given for $\alpha = (\alpha_x,\alpha_\xi) \in T^\ast M$ by 
\begin{equation*}
Tu(\alpha) = \int_M K_T(\alpha,y) \mathrm{d}y.
\end{equation*}
Here $K_T$ is an analytic kernel, negligible away from the ``diagonal'' $\{y = \alpha_x\}$, and which near this diagonal has roughly the behaviour of
\begin{equation*}
e^{\frac{i}{h} (\langle \alpha_x - y, \alpha_\xi\rangle + \frac{i}{2}\langle\alpha_\xi\rangle(\alpha_x - y)^2) } a(\alpha,y),
\end{equation*}
where a is an analytic symbol, elliptic in the relevant class. Contrary to \eqref{eq:def-flat-FBI-transform}, we use here a phase that is a symbol of order $1$ in $\alpha$. The basic properties of this transform are investigated in Chapter \ref{part:FBI}.

Let us explain now why FBI transforms are interesting objects. What motivated the works of Bros and Iagolnitzer is the following simultaneous observations. Let $f$ be a tempered distribution in $\R^n$. Then, $(x,\xi)$ is \emph{not} in the $C^\infty$ wavefront set of $f$, if and only if, for all $(x',\xi')$ sufficiently close to $(x,\xi)$, as $h\to 0$,
\[
T_{\R^n} f(x',\xi';h) = \O(h^\infty).
\]
On the other hand, if $f$ is a real analytic function is some open set $U\subset \R^n$, then for $x\in U$, and $\xi\in \R^n$, there exists a constant $C>0$ such that:
\[
T_{\R^n} f(x,\xi;h) = \O\left(\exp\left( - \frac{1}{Ch} \right)\right).
\]
This suggests to define the \emph{analytic wavefront set} of $f$ as the set of $(x,\xi)\in \R^{2n}$ such that $T_{\R^n}f$ does \emph{not} satisfy this last bound uniformly in any neighbourhood of $(x,\xi)$.

The PDE specialist may wonder what is the purpose of introducing one more definition of the wavefront set, and the Dynamical Systems expert speculate why we have to consider wavefront sets in the first place. The answer is twofold. 

The first point is (very) closely related to the Segal--Bargmann transform. It is the observation that the FBI transform enables to represent (pseudo)differential operators as \emph{multiplication operators}. This feature alone makes it useful for studying elliptic regularity problems.

Usually, one can chose the transform $T$ such that $T^\ast T$ is the identity, so that $T$ is an isometry between $L^2(M)$ and its image in $L^2(T^\ast M)$. It turns out that, for a suitable class of functions $p\in C^\infty(T^\ast M)$,
\[
u \mapsto T^\ast p T u
\]
is a pseudo-differential operator with principal symbol $p$. Denoting $\Pi = TT^\ast$, it is the orthogonal projector on the image of $T$. We then have
\[
\Pi p \Pi (Tu) = T ( T^\ast p T u).
\]
This formula relates a Toeplitz-like operator -- $\Pi p \Pi$ -- with a pseudo-differential operator -- $T^\ast p T$.

On top of being a very practical tool, the fact that the FBI transform relates pseudo-differential operators with Toeplitz-like operators is thus a bridge between the quantization of cotangent spaces via the algebra of pseudo-differential operators, and the quantization of compact symplectic manifolds.

The second point is that for a function $u$, its transform $Tu(x,\xi)$ is a function of parameters belonging to the \emph{classical} phase space. In particular, one expects that (microlocal) propagation phenomenon can be observed directly on $Tu$. A particularly striking consequence is the flexibility with which escape functions can be used. This deserves a detailed explanation

\subsection*{Helffer--Sj\"ostrand theory.}

Our use of the FBI transform is deeply inspired by the work of Bernard Helffer and Sj\"ostrand \cite{Helffer-Sjostrand-86,Sjostrand-96-convex-obstacle}. We will rely in particular on the notion of \emph{complex Lagrangian deformation} of the cotangent bundle. This concept is closely related with the notion of \emph{escape function}, which is maybe more popular nowadays. The idea of escape function has a long history. It is linked to the technique of positive commutators, that appears in H\"ormander's proof of propagation of singularities or in \'Eric Mourre's estimates \cite{Mourre-81} for example. It was introduced by Helffer and Sj\"ostrand \cite{Helffer-Sjostrand-86} to study quantum tunneling effects for some electric Schr\"odinger operators. The central idea is the following: given a (pseudo-)differential operator $P$ with principal symbol $p$, if one can build a function $G$ that decreases along the Hamiltonian flow $(\Phi_t)_{t \in \R}$ of $p$ in some region of phase space, then one can gain some sub-principal micro-ellipticity in that region. Helffer and Sj\"ostrand introduced a framework involving the FBI transform to implement this idea. Escape functions have since been popularized and have become very much an independent technique.

We will give first a (very) condensed presentation of the Helffer-Sj\"ostrand method itself before coming back to historical considerations. If $M$ is a real analytic manifold, one can choose a complexification $\smash{\widetilde{M}}$ for $M$, and then $T^* \smash{\widetilde{M}}$ is a complexification for the cotangent space $T^* M$ of 
$M$. It is possible to construct an FBI transform $T$ whose kernel is real analytic. This implies that if $u\in \mathcal{D}'(M)$, then $Tu$ has a holomorphic extension to a complex neighbourhood of $T^\ast M$ in $T^\ast \smash{\widetilde{M}}$. In particular, if $\Lambda \subseteq T^\ast \smash{\widetilde{M}}$ is sufficiently close to $T^\ast M$, we can consider the restriction $Tu_{|\Lambda}$, but we need to choose $\Lambda$ wisely if we want this to be useful. In particular, there are some geometric conditions that have to be fulfilled by $\Lambda$.

Notice indeed that the complex manifold $T^\ast \smash{\widetilde{M}}$ is endowed with a rich \emph{real} symplectic structure, since both the real part $\omega_R$ and imaginary part $\omega_I$ of the canonical complex symplectic form $\omega$ on $T^\ast \smash{\widetilde{M}}$ are real symplectic forms on $T^\ast \smash{\widetilde{M}}$. Notions from the symplectic geometry of $\omega_I$ (resp. $\omega_R$) will be designed with an I (resp. R) -- I-Lagrangian, I-symplectic... Since $\omega$ is an exact symplectic form ($\omega = \mathrm{d}\theta$ if one denotes by $\theta = \xi \mathrm{d}x$ the canonical Liouville one-form), the $2$-forms $\omega_I$ and $\omega_R$ also are.

If $G$ is a real valued symbol of order $1$ on $T^\ast \smash{\widetilde{M}}$ --that will play the role of an escape function in our context --, we consider for $\tau \geq 0$ small enough
\[
\Lambda_\tau := \exp(\tau H^{\omega_I}_G)(T^\ast M) \subset T^\ast \widetilde{M}.
\]
Since $T^\ast M$ is I-Lagrangian, $\Lambda_\tau$ also is: this is the complex Lagrangian deformation of $T^* M$ that we announced. The one form $\Im \theta$ is thus closed on $\Lambda_\tau$. The global description as the image of $T^\ast M$ implies the existence of a global solution $H_\tau$ on $\Lambda_\tau$ to the equation $\mathrm{d}H_\tau = - \Im\theta_{|\Lambda_\tau}$. Since $G$ is a symbol, $H_\tau$ also is. Since $T^\ast M$ is R-symplectic, $\Lambda_\tau$ also is for $\tau$ small enough, so that $\omega_R^n / n!$ defines a volume form $\mathrm{d}\alpha$ on $\Lambda_\tau$.

Let $P$ be a semi-classical differential operator with analytic coefficients, with (analytic) principal symbol $p$. Then for $u$ analytic and $\tau$ small enough, we will see in Chapter \ref{part:FBI} that under relevant assumptions, we have
\begin{equation}\label{eq:flat-space-scalar-product}
\int_{\Lambda_\tau} T Pu \ \overline{Tu}\ e^{-2H_\tau /h} \mathrm{d}\alpha = \int_{\Lambda_\tau} (p+ O(h))\ |T u|^2\  e^{-2H_\tau/h} \mathrm{d}\alpha.
\end{equation}
Here, we have identified $p$ with its holomorphic extension. The idea of the method is to consider (instead of $L^2(M)$) the space
\begin{equation}\label{eq:first-def-H_Lambda}
\mathcal{H}_{\Lambda_\tau} := \left\{u\ \middle|\ Tu_{|\Lambda_\tau} \in L^2( e^{-H_\tau/h}\mathrm{d}\alpha) \right\}.
\end{equation}
(here $u$ is assumed to be in a space of hyperfunctions we will define precisely later on). On $\mathcal{H}_{\Lambda_\tau}$, the operator $P$ has an effective principal symbol $p_{|\Lambda_\tau}$, which may be (if $G$ is suitably chosen) more elliptic than $p_{|T^\ast M}$.

This method also applies to analytic pseudo-differential operators. It can be extended to the case of Gevrey operators, after the necessary adjustments we present at the end of the introduction. The rigorous statements and proofs of the results that we just mentioned may be found in Chapter \ref{part:FBI}.

Let us now give a bit of context and more explanations. We start by presenting what we call the ``Martinez method'', and explain why it is not suited to our needs. As we have seen before, if $\sigma$ is a function satisfying suitable estimates and $T$ some FBI transform, $T^\ast \sigma T$ is a pseudo-differential operator with principal symbol $\sigma$. However, the operator $T^\ast \sigma T$ is a well defined bounded operator on $L^2$, even if $\sigma$ is only a $L^\infty$ function. This observation suggests to quantize a class of functions much larger than the usual class of symbols. Instead of the usual construction of microlocal spaces in the form (for some function $G$ that will be called an \emph{escape function})
\[
\Op(e^{G}) L^2(M), 
\]
defined with $\Op$, the Weyl quantization, one can define a space $\Hh_G$ with
\[
\| f \|_{G}^2 := \int_{T^\ast M} e^{-2G} |Tu|^2.
\]
Formally, this is the space $(T^\ast e^G T) L^2(M)$. It is considerably easier to define than the space $\Hh_\Lambda$ introduced in \eqref{eq:first-def-H_Lambda}. This approach was taken in Faure and Tsujii's work \cite{Faure-Tsujii-17-FractalWeylLaw}. Similar spaces appear in the works of Andr\'e Martinez \cite{Martinez-94-interaction}, Shu Nakamura \cite{nakamuraMartinezMethodPhase1995} and Klaus Jung \cite{jungPhaseSpaceTunneling2000}. These authors study tunneling effects for the bottom of the spectrum of elliptic operators, and instead of $\exp(G)$ consider stronger weights in the form $\exp(G/h)$, with $h >0$ a semi-classical parameter.

The main technical point in Faure and Tsujii's paper is that the essential spectral radius of the Koopman operator \eqref{eq:Koopman_the_first} acting on the space $\Hh_G$ is bounded by 
\[
\sup_{x} \limsup_{\xi\to \infty}\ \exp( G\circ\Phi_{t} - G)(x,\xi),
\]
where $(\Phi_t)_{t \in \R}$ denotes the symplectic lift of $(\phi_t)_{t \in \R}$. This is a sort of ``integrated version'' of a G\r{a}rding-type estimate on which Martinez et. al. rely. Indeed, they prove that 
\begin{equation}\label{eq:multiplication-formula-Martinez}
\langle Pf, f\rangle_{\Hh_{G/h}} = \int_{T^\ast M} p_G(x,\xi)  e^{\frac{2G}{h}} |T u|^2,
\end{equation}
where $p_G$ is the principal symbol of $P$, shifted by $G$ according to some precise formula. When $p$ is real, the imaginary part of $p_G$ is given in first approximation, for $G$ small enough, by
\[
\nabla_\xi G \cdot \nabla_x p - \nabla_x G \cdot \nabla_\xi p.
\]
The appearance of this Poisson bracket is revealing of some symplectic phenomenon at play here. Actually, if one were to take $G_1$ to be the real part of an almost analytic extension of $G$, one would obtain (in the relevant symbol class)
\begin{equation}\label{eq:DL-symbol-rough}
p(\exp(\tau H_{G_1}^{\omega_I}(\alpha))) = p(\alpha) + i \tau(\nabla_\xi G \cdot \nabla_x p - \nabla_x G \cdot \nabla_\xi p) + \O(\tau^2).
\end{equation}
In particular, if $\Lambda$ is a I-Lagrangian corresponding to $G_1$, the spaces $\Hh_G$ and $\Hh_{\Lambda}$ have (at least approximately) the same microlocal behaviour. On the other hand, in the Helffer--Sjostrand method, one does not have to assume that the symbol $G$ is constructed as the extension of a symbol defined on $T^\ast M$. For our purpose, this point will be crucial, and justifies the use of I-symplectic geometry instead of the -- \emph{a priori} simpler -- point of view of Martinez. Indeed, unlike the references cited above, for the study of analytic flows, we need to assume that $G$ is symbol of order $1$. In that case, the estimate \eqref{eq:DL-symbol-rough} can be rewritten (for $p$ of order $1$)
\[
p(\exp(\tau H_{G_1}^{\omega_I}(\alpha))) = p(\alpha) + i \tau(\nabla_\xi G \cdot \nabla_x p - \nabla_x G \cdot \nabla_\xi p) + \O_{L^\infty}(\langle\alpha\rangle \tau^2).
\]
The remainder is small as a symbol of order $1$, but is \emph{not} bounded. For this reason the Koopman operator may lose the property of semi-group on the spaces $\Hh_G$. To circumvent this, we have to build our escape function directly on $T^\ast \widetilde{M}$.

Let us come back to the Helffer--Sj\"ostrand method. Our presentation is more general than what appeared in the original article \cite{Helffer-Sjostrand-86}, so let us present the novelty in our approach. In its first version, the method dealt with operators on the flat space $\R^n$, and the weight $G$ was assumed to be compactly supported in the $\xi$ variable. The theory was then adapted to manifolds, by Sj\"ostrand \cite{Sjostrand-96-convex-obstacle} and then Sj\"ostrand and Zworski \cite{Sjostrand-Zworski-99}, to obtain asymptotics for the counting function of resonances of an analytic convex obstacle, somehow closing the discussion opened by the Bardos--Lebeau--Rauch paper \cite{Bardos-Lebeau-Rauch-87}.

However, in these articles, there remained the restriction that $G$ be compactly supported in $\xi$. Additionally, the transform introduced did not have a globally analytic kernel, preventing the study of hyperfunctions. This was not a problem since the main interest were solutions of the wave equation, which are highly regular \emph{a priori}. In our study of analytic and, more generally, Gevrey Anosov flows, the eigenfunctions -- or \emph{resonant states} -- are not regular, and this limitation had to be lifted. We thus obtain a transform whose kernel is globally analytic. 

Additionally, we are able to deal with symbols $G$ of order $1$ -- this is the natural limit on arbitrary real analytic manifolds. Indeed, when $G$ is of order $1$, the manifold $\Lambda$ is contained in a Grauert tube of size $\sim 1$ of $T^\ast M$. If one tries to take a symbol of order larger than $1$, it would require to work with Grauert tubes of infinite radius, and only very particular real analytic manifolds possess such structure -- for example, it is not the case of the real hyperbolic space.

At the same time as we were elaborating this article, Jeffrey Galkowski and Zworski \cite{Zworski-Galkowski-1,Zworski-Galkowski-2} were studying a very similar extension in the analytic category. They obtain a version of the Helffer--Sj\"ostrand framework for symbols of order $1$, on tori. They have already found another application to the technique in \cite{Zworski-Galkowski-3}.

With \cite{lascarFBITransformsGevrey1997}, Bernard and Richard Lascar extended the Helffer--Sj\"ostrand method to the case of Gevrey regularity on manifolds. Again, they only considered symbols with compact support in $\xi$, while we are able to consider symbols of order $1/s$ for $s$-Gevrey problems. We will see that this is the natural limit. Moreover, the ranges of $s$'s that are allowed in \cite{lascarFBITransformsGevrey1997} is quite restrictive, a constraint that we will also lift.

We close this short presentation of the Helffer-Sj\"ostrand method with the following heuristic consideration. The gist of the technique is that given a linear (pseudo-)differential operator $P$, for the purpose of studying the regularity of the solutions to some equation involving $P$, the regularity of $P$ can be exhausted by taking the right I-Lagrangian $\Lambda$. Replacing $L^2(M)$ by $\Hh_\Lambda$ is akin to replacing $T^\ast M$ by $\Lambda$ for all practical purposes. Instead of using analytic microlocal analysis on $M$ -- which may be complicated -- one uses $C^\infty$ microlocal analysis on $\Lambda$ -- which may be considerably simpler. We hope that our application of the method to analytic Anosov flows in \S~\ref{part:Anosov-flow} can be an illustration of this principle.

For the reader wanting yet another perspective on FBI transform on manifolds, we recommend Jared Wunsch and Zworski's paper \cite{Wunsch-Zworski-01}, which deals with the $C^\infty$ case. Our results on analytic FBI transform on compact manifold are detailed at the beginning of Chapter \ref{part:FBI}.

\subsection*{Gevrey microlocal analysis}

Before we explain how the results in the analytic case have to be adapted in the Gevrey case, let us recall some definitions. The class of Gevrey functions was introduced by Maurice Gevrey in \cite{Gevrey-18} to study the regularity of solutions to certain linear PDEs, in particular the heat equation -- it turns out that the heat kernel is $2$-Gevrey with respect to the time variable.

Let $s \geq 1$ be fixed. Let $U$ be an open subset of $\R^n$. A function $f : U \to \C$ is said to be $s$-Gevrey, or, for short, $\G^s$, if $f$ is $\mathcal{C}^\infty$ and if, for every compact subset $K$ of $U$, there are constants $C,R > 0$, such that, for all $\alpha \in \N^n$ and $x \in K$, we have
\begin{equation}\label{eqdefgev_intro}
\begin{split}
\va{\partial^\alpha f (x)} \leq C R^{\va{\alpha}} \alpha!^s.
\end{split}
\end{equation}
The constant $R$ in \eqref{eqdefgev_intro} may be interpreted as the inverse of a (Gevrey) radius of convergence. Notice that when $s = 1$, this describes the class of real-analytic function on $U$. When $s > 1$, the class of $s$-Gevrey function is non-quasianalytic: it contains compactly supported functions. 

In \cite{Roumieu-58-Fourier}, Charles Roumieu made the crucial observation that given a $s$-Gevrey function $f$, compactly supported in $\R^n$, for some constants $c,C>0$, and $\xi\in \R^n$,
\[
|\hat{f}(\xi)| \leq C e^{-c |\xi|^{1/s}}.
\]
Conversely, a function whose Fourier transform satisfies such an estimate is $s$-Gevrey. He also initiated the study of Gevrey ultradistributions, i.e. the continuous linear functionals on spaces of Gevrey functions. Later on, Hikosaburo Komatsu \cite{Ko1,Ko2,Ko3} gave a systematic treatment of such objects.

Since Gevrey regularity can be characterized by the decay of Fourier transforms, it is only natural to expect that a specific version of microlocal analysis can be developed for this class of regularity. Such a study was inaugurated less than a decade later by Louis Boutet de Monvel and Paul Krée in \cite{Boutet-Kree-67}. They studied an algebra of pseudo-differential operators on $\R^n$ whose symbols $\sigma$ satisfy estimates of the form
\[
|\partial_x^\alpha \partial_\xi^\beta \sigma|\leq C R^{|\alpha|+|\beta|} \alpha !^s \beta! \langle\xi\rangle^{m-|\beta|}\ ,
\]
for all multi-indices $\alpha,\beta\in \N^n$. In other words, the symbols are $s$-Gevrey in the $x$ variable, and real analytic in the $\xi$ variable. We will denote this algebra by $\mathcal{G}^s_x\mathcal{G}^1_\xi\Psi$. It contains the differential operators with $s$-Gevrey coefficients. They prove that the usual statement of $C^\infty$ pseudo-differential calculus can be adapted, so that:
\begin{itemize}
	\item modulo some compactness of support assumption, $\G^s_x\G^1_\xi \Psi$ indeed forms an algebra;
	\item the usual $\O(\langle\xi\rangle^{-\infty})$ remainders, and $C^\infty$-smoothing properties can be replaced by $\O(\exp(- \jap{\xi}^{1/s}/C))$ remainders and $\G^s$ smoothing properties. In particular,
	\item given $P\in \G^s_x\G^1_\xi\Psi$ elliptic, there exists $E$ another such operator, so that $PE- I$ is $\G^s$ smoothing ($E$ is called a parametrix for $P$).
\end{itemize}

With the development and popularization of microlocal analysis, their study was furthered by many other authors. Notably, a group of Italian mathematicians, and among them Zanghirati \cite{zanghiratiPseudodifferentialOperatorsInfinite1985} and Rodino \cite{Rodino-93-book}. When compared with \cite{Boutet-Kree-67}, the main innovation is the following: they study pseudo-differential operators whose order is not finite, such as $\exp t (-\Delta)^{\epsilon}$, with $\epsilon>0$ small enough. This would not make sense when working with $C^\infty$ regularity. 

The striking result of Bardos, Lebeau and Rauch \cite{Bardos-Lebeau-Rauch-87} created new interest for Gevrey microlocal analysis. The theory of propagation of singularities was established by several authors, notably Yoshinori Morimoto with Kazuo Taniguchi \cite{Morimoto-Taniguchi-85} and Lascar \cite{B-Lascar-85}. In the latter, a different class of operators is considered, corresponding to symbols satisfying the estimates
\[
|\partial_x^\alpha\partial_\xi^\beta \sigma| \leq C R^{|\alpha|+|\beta|} (\alpha! \beta!)^s \langle\xi\rangle^{m-|\beta|}.
\]
Now, the symbol is Gevrey in $x$ and $\xi$, and we will refer to such operators as $\G^s_x\G^s_\xi$ or simply $\Gs$ pseudors. When working with propagation of singularities, since Hamiltonian flows do not necessarily preserve the fibers in $T^\ast M$, they cannot preserve the mixed $\G^s_x \G^1_\xi$ regularity from before. It is thus natural to consider such operators.

Finally, we already mentioned \cite{lascarFBITransformsGevrey1997}, where appears a semi-classical -- i.e. with an $h$ -- version of the theory. In this situation, the remainders of size $\O(\exp(-\langle \xi \rangle^{1/s}/C))$ from \cite{Boutet-Kree-67} become remainders of size $\exp( - 1/Ch^{1/s})$. The purpose of Lascar \& Lascar was to study the FBI transform in Gevrey classes, to improve upon the result of Sj\"ostrand \cite{Sjostrand-96-convex-obstacle}, itself a refinement of \cite{Bardos-Lebeau-Rauch-87}. In some sense, the present article completes \cite{lascarFBITransformsGevrey1997} in several directions. 

We already mentioned several technical extensions we will provide with respect to the I-Lagrangian spaces method. In Chapter \ref{part:Gevrey-microlocal}, the reader can also find a turnkey theory of $\Gs$ pseudors, which we did not find in the literature. Indeed, while algebraic stability of $\Gs_x\G^1_\xi$ pseudors was considered, it is not the case of the $\Gs$ class. To our understanding, when $\Gs$ pseudors appeared in articles such as \cite{lascarFBITransformsGevrey1997,jungPhaseSpaceTunneling2000}, the origin of such an operator was either not explicitly mentioned or it was actually a $\G^1_x\Gs_\xi$ operator. Since there was only one $\Gs$ operator, the authors were not concerned with $\Gs\Psi$ as a class of operators. We have thus taken the opportunity to prove that the class of $\Gs$ pseudors is stable by composition, and several other results that enable one to practice microlocal analysis in that class. 

Let us come back to the I-Lagrangian spaces. When studying Gevrey PDE problems instead of analytic ones, we have to rely on almost-analytic extensions instead of holomorphic extensions. A central fact is that if $f$ is $s$-Gevrey, it admits $s$-Gevrey almost analytic extensions, that satisfy estimates in the form
\[
\left| \overline{\partial} f(x+iy) \right| \leq C \exp\left( - \frac{1}{ C |y|^{\frac{1}{s-1}}} \right).
\]
For this to be a $s$-Gevrey remainder $\O(\exp(-1/Ch^{1/s}))$, we have to impose that $|y|\leq C h^{1-1/s}$. In the analytic case, for the I-Lagrangian spaces, we work in a neighbourhood on size $1$ of the reals. In the $s$-Gevrey case, we have to consider shrinking complex neighbourhoods of the reals as $h\to 0$. For this reason, in the $s$-Gevrey case, instead of working with I-Lagrangians $\Lambda= e^{\tau H_G}(T^\ast M)$ with $G$ a symbol of order $1$, we will have to consider
\[
\Lambda = \exp( \tau h^{1-1/s} G )(T^\ast M), 
\]
with $G$ a symbol of order $1/s$. For a $s$-Gevrey Anosov flows, this means that we can at most gain a $\jap{\xi}^{1/s}$ subprincipal term, leading to $ns$-Schatten estimates, and Theorem \ref{thm:main}.

Let us mention a direction in which our work can certainly be extended. Given a logarithmically convex sequence $A=(A_k)_{k\geq 0}$ of positive real numbers, the associated Denjoy--Carleman class on $\R^n$ is the set of smooth functions $f$ such that for some constant $C>0$, and for all $\alpha \in \N^n$,
\[
|\partial^\alpha f(x)|\leq C^{1+\va{\alpha}} A_{\va{\alpha}} \alpha!\ .
\]
Denjoy--Carleman classes thus generalize Gevrey classes, since the $s$-Gevrey class is the $(k!^{s-1})_{k\geq 0}$--Denjoy--Carleman class. Since $A$ is logarithmically convex, the associated regularity can be characterized with the Fourier transform, as was observed in \cite{Roumieu-58-Fourier}. It follows that one can probably reformulate all our results in some of these more general classes. In order to extend our results to more general Denjoy--Carleman classes, one would need to understand the almost analytic extension of functions in these classes. A discussion of this topic may be found in \cite{dyncprimekinPseudoanalyticContinuationSmooth1976}.

In another work of the second author \cite{jezequelGlobalTraceFormula2019}, it was observed that to obtain global trace formulae for flows, it was sufficient to work in Denjoy--Carleman classes much less regular than Gevrey. It is also probably the case that for Anosov diffeomorphisms, the right scale of regularity necessary to obtain Schatten estimates is much less than Gevrey, using I-Lagrangian spaces or another technique.

\subsection*{Structure of the paper}

Chapter \ref{part:Gevrey-microlocal} is devoted to recalling some basic facts about the Gevrey regularity and describing the algebra of pseudo-differential operators with Gevrey symbols acting on manifolds. Chapter \ref{part:FBI} of this monograph will be devoted to describing an FBI transform suited to our needs. To this end, we will revisit and expand the  \S 1 of \cite{Sjostrand-96-convex-obstacle}. This is quite independent of Anosov flows, and we hope that it can be usefully applied to other situations. Finally, we will apply this tool to the study of Gevrey Anosov flows in Chapter \ref{part:Anosov-flow} and prove in particular Theorem \ref{thm:main}.

Throughout the paper $h > 0$ will denote a small implicit parameter thought as tending to $0$.

\subsection*{Acknowledgements}

Thanks to M. Zworski's suggestion to use H\"ormander's $\overline{\partial}$ trick, we were able to include the analytic case ($s=1$), we are in his debt.

We would also like to thank Baptiste Morisse, San V\~{u} Ng\d{o}c and Johannes Sj\"ostrand for inspiring discussions.

Most of this work has been done while the second author was supported by the European Research Council (ERC) under the European Union’s Horizon 2020 research and innovation programme (grant agreement No 787304). Part of this work was accomplished during the second author's residence at the Mathematical Sciences Research Institute in Berkeley, California, during the Fall 2019 Micro-local Analysis Program supported by the National Science Foundation under Grant No. DMS-1440140. The second author benefits from the support of the French government “Investissements d’Avenir” program integrated to France 2030, bearing the following reference ANR-11-LABX-0020-01.

\renewcommand{\theequation}{\arabic{chapter}.\arabic{equation}}
\setcounter{equation}{0}

\chapter{Gevrey microlocal analysis on manifolds}
\label{part:Gevrey-microlocal}

This first chapter is divided in three sections. In \S \ref{sec:def_Gev}, the reader will find definitions, and some useful lemmas about spaces of Gevrey functions, symbols, and almost-analytic extensions thereof. In \S \ref{sec:oscillatory-integrals}, we present some techniques to deal with Gevrey oscillating integrals. In \S \ref{sec:toolbox_gevrey_pseudor}, we review the basic elements of a Gevrey theory of semi-classical microlocal analysis. It is in this section that we present the details of the algebraic properties of $\Gs$ pseudors announced in the introduction. 

Some classical microlocal questions that we do not deal with are the following: 
\begin{itemize}
	\item the existence of a good parametrix in the $\Gs$ class when $s > 1$;
	\item the related question of the functional calculus in that class;
	\item propagation of singularities and radial estimates.
\end{itemize}

\section{Gevrey spaces of functions and symbols}\label{sec:def_Gev}

\subsection{Gevrey functions and ultradistributions on manifolds}
\label{sec:appendix-Gevrey-regularity}

\subsubsection{First definitions}
\label{sec:def-Gevrey}

We start by recalling the definition of Gevrey classes of regularity. Let $s \geq 1$ be fixed. Let $U$ be an open subset of $\R^n$. A function $f : U \to \C$ is said to be $s$-Gevrey if $f$ is $\mathcal{C}^\infty$ and if, for every compact subset $K$ of $U$, there are constants $C,R > 0$ such that, for all $\alpha \in \N^n$ and $x \in K$, we have
\begin{equation}\label{eqdefgev}
\begin{split}
\va{\partial^\alpha f (x)} \leq C R^{\va{\alpha}} \alpha!^s.
\end{split}
\end{equation}
Notice that when $s = 1$, we retrieve the class of real-analytic functions on $U$. When $s > 1$, the class of $s$-Gevrey functions is non-quasianalytic: it contains compactly supported functions. We will often write $\G^s$ instead of $s$-Gevrey and denote by $\G^s\p{U}$ the space of $\G^s$ functions on $U$ and by $\G_c^s\p{U}$ the space of compactly supported elements of $\G^s\p{U}$. The non-quasianalyticity of $\G^s$ when $s > 1$ implies in particular that there are $\G^s$ partitions of unity. As pointed out in the introduction, the inverse of the constant $R$ in \eqref{eqdefgev} will sometimes be called a $\G^s$ radius of convergence.

The definition above extends immediately to the case of Banach valued functions. A function $f$ from $U$ to some Banach space $B$ is said to be $\G^s$ if \eqref{eqdefgev} holds with the modulus replaced by the norm of $B$. With this definition, the class of $\G^s$ functions is stable by composition, as was proved by Gevrey in his original paper \cite{Gevrey-18}. In fact, the class of $\G^s$ functions is very well-behaved and, when $s > 1$, quite flexible: for instance, it is Cartesian closed, stable by differentiation, solving ODEs, Implicit Function Theorem, etc, and there are versions of Borel's and Whitney's theorem for $\G^s$ functions (see \cite{Kriegl-Michor-Rainer-09} for details).

Since the class $\G^s$ is stable by composition, we have a natural definition of a $\G^s$ structure on a manifold: a $\G^s$ manifold is a Hausdorff topological space with countable basis endowed with a maximal $\G^s$ atlas. Here, a $\G^s$ atlas is defined to be an atlas with $\G^s$ change of charts (notice that we retrieve the usual notion of real-analytic manifold when $s=1$). As usual, if $M$ and $N$ are two $\G^s$ manifolds, then a map $f : M \to N$ is said to be $\G^s$ if it is $\G^s$ ``in charts". Since the Implicit Function Theorem holds in the class $\G^s$, most elementary results from differential geometry are easily checked to be true in the $\G^s$ category. In particular, there is a well-defined notion of $\G^s$ (vector-)bundle, and each usual bundle associated with a $\G^s$ manifold $M$ (tangent, cotangent, etc) admits a natural $\G^s$ structure. As a consequence, it makes sense to say that a vector field over $M$ is $\G^s$.

\begin{remark}\label{remark:structure_analytique}
Of course, a real-analytic manifold has a natural structure of $\G^s$ manifold, since real-analytic maps are $\G^s$. As pointed out in \cite{lascarFBITransformsGevrey1997}, all $\G^s$ manifolds may be described in this way. Indeed, there is a Gevrey version of the famous Whitney's Embedding Theorem \cite{whitneyDifferentiableManifolds} : \emph{every $\G^s$ compact manifold is $\G^s$-diffeomorphic to a real-analytic submanifold of an Euclidean space.} The adaptation of the proof of Whitney's Theorem to our setting is straightforward. Since the Inverse Function Theorem holds in the $\G^s$ category, it suffices to follow the lines of the proof of \cite[Chapter 4, Theorem 7.1]{hirschDifferentialTopology1994}, replacing $C^\infty$ by $\G^s$ at every step. This embedding produces an analytic structure on any compact $\G^s$ manifold, compatible with its $\G^s$ structure. 
\end{remark}

We want now to define ultradistributions on a $\G^s$ manifold $M$. To do so, we need to give a structure of topological vector space to the space $\G^s_c\p{M}$ of compactly supported $\G^s$ functions on $M$. If $\p{U,\kappa}$ is a $\G^s$ chart for $M$ and $K$ a relatively compact subset of $U$, then we define for every $R > 0$ and function $f$, infinitely differentiable on a neighbourhood of $K$, the semi-norm
\begin{equation}\label{eqnormrn}
\begin{split}
\n{f}_{s,R,K} \coloneqq \sup_{\substack{x \in \kappa\p{K} \\ \alpha \in \N^n}} \frac{\va{\partial^{\alpha}\p{f \circ \kappa^{-1}} \p{x}}}{R^{\va{\alpha}}\alpha!^s}.
\end{split}
\end{equation}
We extend this definition to any relatively compact subset $K$ of $M$ by covering $K$ by a finite number $K_1,\dots,K_N$ of compact sets included in some domains of charts and setting
\begin{equation*}
\begin{split}
\n{f}_{s,R,K} \coloneqq \sum_{j = 1}^N \n{f}_{s,R,K_j}.
\end{split}
\end{equation*}
Then, if $K$ is a relatively compact subset of $M$ and $R > 0$, we define $\widetilde{E}^{s,R}\p{K}$ to be the Banach space of functions $f\in\mathcal{C}^\infty(M)$, supported in $K$, such that $\n{f}_{s,R,K} < + \infty$, endowed with the norm $\n{\cdot}_{s,R,K}$. For $s=1$, consider this definition as temporary, as we will give a more practical but equivalent scale of spaces of real analytic functions in the next section \S \ref{sec:Grauert}. 

Notice that these spaces heavily depend on the choices of charts that we made above. It is, however, not the case of the inductive limit, when $U$ is an open subset of $M$,
\begin{equation}\label{eqtopgsigma}
\begin{split}
\G^s_c\p{U} \coloneqq \varinjlim_{K \Subset U} \varinjlim_{R > 0} \widetilde{E}^{s,R}\p{K}.
\end{split}
\end{equation}
Here, the first limit is taken over compact subsets $K$ of $U$, and the inductive limit is taken in the category of locally convex topological vector spaces. Notice that the underlying set of this limit is indeed the set of compactly supported $\G^s$ functions on $U$. In particular, when $s=1$, if $U$ is not compact itself then $\G^s_c\p{U}$ is zero-dimensional. In the case that $M$ itself is compact, we will write $\G^s(M)$ for $\G^s_c(M)$.

We define the space $\U^s\p{U}$ of ultradistributions on $U$ to be the strong dual of $\G^s_c\p{U}$. Notice that when $U= M$ and $M$ is compact, there is a nice description on $\U^s\p{M}$, that we detail now when $s > 1$ (the case $s = 1$ will be dealt with in \S \ref{sec:Grauert}). Define for $R > 0$ the closure $E^{s,R}(M)$ of analytic functions in $\widetilde{E}^{s,R}(M)$. This way, the inclusion of $E^{s,R}(M)$ in $E^{s,R'}(M)$ when $R' > R$ has dense image, and thus induces an injection from $(E^{s,R'}(M))'$ to $E^{s,R}(M))'$. This will be important when constructing functional spaces of Gevrey ultradistributions. One can then prove that \eqref{eqtopgsigma} becomes using this new scale of spaces
\begin{equation*}
\G^s\p{M} = \varinjlim_{R > 0} E^{s,R}\p{M}.
\end{equation*}
Indeed, for every $R > 0$, there is $R' > 0$ such that $\widetilde{E}^{s,R}(M)$ is a subset of $E^{s,R'}(M)$. This is for instance a consequence of the analysis in Chapter \ref{part:FBI} (see in particular Theorem \ref{thm:existence-good-transform} and Lemma \ref{lemma:S_BMT}).

Then, using \cite{dieudonne_schwartz}, we see that $\U^s\p{M}$ identifies with the projective limit
\begin{equation}\label{eq:inductive_ultradistribution}
\begin{split}
\U^s\p{M} = \varprojlim_{R > 0} \p{E^{s,R}\p{M}}'.
\end{split}
\end{equation}
In particular, $\U^s\p{M}$ is a Fr\'echet space. It is a Roumieu-type space of ultradistributions. More details on Roumieu and Beurling spaces can be found in \cite{Braun-Meise-Taylor-90}. These topological considerations will play almost no role here, as we will work directly with the spaces $E^{s,R}\p{M})'$ most of the time.

\begin{remark}\label{remark:choix_de_cartes}
The fact that the limit \eqref{eqtopgsigma} does not depend on the choice of charts in the definition of the norms $\n{\cdot}_{s,R,K}$ follows from the stability by composition of the class $\G^s$. Indeed, if we denote by $(\n{\cdot}_{s,R,K}')_{R > 0}$ the family of norms that is obtained with another choice of charts, one sees (applying Faa di Bruno's formula or adapting the original proof of Gevrey \cite{Gevrey-18}) that there is a constant $C > 0$ such that for every $R > 0$ we have
\begin{equation}\label{eq:presque_equivalence}
\begin{split}
\n{\cdot}_{s,C \max\p{1,R},K} \leq C \n{\cdot}_{s,R,K}' \textup{ and }  \n{\cdot}_{s,C \max\p{1,R},K}' \leq C \n{\cdot}_{s,R,K}.
\end{split}
\end{equation}
Consequently, when $U$ is an open subset of $\R^n$, we may and will assume that the chart $\kappa$ used to define the norms $\n{\cdot}_{s,R,K}$ is the inclusion of $U$ in $\R^n$.
\end{remark}

\begin{remark}
Let us explain how we will use Landau notations. We assume that $M$ is a compact manifold. If $w : \R_+^* \to \R_+^*$ is a function and $\p{u_h}_{h > 0}$ a family of function on $M$ depending on a small parameter $h > 0$, we will oftentimes write that $u_h$ is an $\O\p{w(h)}$ in $\G^s\p{M}$ (or even in $\G^s$), or that $u= \O_{\G^s}(w(h))$. It has to be understood in the following way: there is an $R > 0$, not depending on $h$, such that $u_h$ is an $\O\p{w(h)}$ in $E^{s,R}\p{M}$. 

If $U$ is a subset of $M$, we will also write that $u_h$ is an $\O\p{w(h)}$ in $\G^s$ \emph{outside of $U$}. It means that there is an $R > 0$ such that
\begin{equation*}
\begin{split}
\n{u_h}_{s,R,M \setminus U} = \O\p{w(h)}.
\end{split}
\end{equation*}
Notice that we retrieve the previous case by taking $U = \emptyset$. We will not be concerned with regularity with respect to the parameter $h$. However, when there are other parameters, in a function $f$ of $(x,y;h)$ for example, more care is required. We will thus use the notation $f_h(x,y) = \O_{\G^s_x}\p{w(h,y)}$ to indicate that for every $y$ the function $x \mapsto f(x,y)$ is an $\O_{\G^s}\p{w(h,y)}$, with constants uniform in $y$. It is not to be confused with $f_h(x,y)=\O_{\G^s}(w(h))$ which indicates Gevrey regularity in both $x$ and $y$.
\end{remark}

\begin{remark}\label{remark:Gs_sections}
Let $s \geq 1$. In order to discuss perturbations of $\G^s$ Anosov flows, we need to define a topology on the space of $\G^s$ sections of a $\G^s$ vector bundle. Let $p : F \to M$ denote a real $\G^s$ vector bundle on $M$ (the case of complex vector bundle is similar). Let $\p{U,\kappa}$ be a $\G^s$ chart for $M$ and $\p{\kappa,\Psi} : p^{-1}\p{U} \to \kappa\p{U} \times \R^d$ be a trivialization for $p : F \to M$. Then, if $K$ is a compact subset of $U$ and $R > 0$, we define for every $\mathcal{C}^\infty$ section $f$ of $F$ the semi-norm
\begin{equation*}
\begin{split}
\n{f}_{s,R,K} = \sup_{\substack{x \in \kappa\p{K} \\ \alpha \in \N^n}} \frac{\n{\partial^\alpha \Psi\p{ f \circ \kappa^{-1}} (x)}}{R^{\va{\alpha}} \alpha!^s},
\end{split}
\end{equation*}
where $\n{\cdot}$ denotes any norm on $\R^d$. Using this semi-norm to replace \eqref{eqnormrn}, the case of sections of the vector bundle $F$ is dealt with as the case of the trivial line bundle over $M$. In particular, when $M$ is a compact manifold, we have a definition of the spaces $E^{s,R}\p{M;F}$, for $R > 0$, and a topology on the space $\G^s\p{M;F}$ of $\G^s$ sections of $F$.
\end{remark}

\subsubsection{The particular case of real-analytic functions}\label{sec:Grauert}

We want now to rewrite the definitions of the previous paragraph in a way that may be more intuitive, in the case $s=1$. Indeed, it may seem more natural to describe real analytic functions as restrictions of holomorphic functions. If $K$ is a compact subset of $\R^n$, then we see that a smooth function $f$ defined on a neighbourhood of $K$ such that $\n{f}_{1,R,K} < + \infty$ admits a holomorphic extension to a complex neighbourhood of $K$ of size $\p{CR}^{-1}$ (for some $C > 0$ that does not depend on $R$), the $L^\infty$ norm of this extension being bounded by $C\n{f}_{1,R,K}$. Reciprocally, if $f$ admits a bounded holomorphic extension to a complex neighbourhood of size $C R^{-1}$, then $\n{f}_{1,R,K}$ is finite and controlled by the $L ^\infty$ norm of this extension (independently on $R$). We will now explain how this remark generalizes to the case of compact manifold.

Let then $M$ be a compact real-analytic manifold of dimension $n$. By a result of François Bruhat and Hassler Whitney \cite{Bruhat-Whitney-59}, the manifold $M$ admits a complexification $\smash{\widetilde{M}}$. That is, $\smash{\widetilde{M}}$ is a holomorphic manifold of complex dimension $n$ endowed with a real-analytic embedding $M \subseteq \widetilde{M}$, such that $M$ is a totally real submanifold of $\smash{\widetilde{M}}$. This means that at each $p\in M$, we have $T_p M \cap i T_p M = \{0\}$. It follows then that if $N$ is a complex manifold and $f : M \to N$ is a real-analytic map, then $f$ extends to a holomorphic map from a neighbourhood of $M$ in $\smash{\widetilde{M}}$ to $N$. In particular, if $\smash{\widetilde{M}'}$ is another complexification for $M$ then the identity of $M$ extends to a biholomorphism between a neighbourhood of $M$ in $\smash{\widetilde{M}}$ and a neighbourhood of $M$ in $\smash{\widetilde{M}'}$.

\begin{remark}\label{remark:analytique_intrinseque}
If $\smash{\widetilde{M}}$ is a complexification for $M$, let $B(\widetilde{M})$ denote the space of bounded holomorphic functions on $\smash{\widetilde{M}}$. Then, we may give a new definition of the space of real-analytic functions on $M$ by
\begin{equation*}
\begin{split}
\G^1\p{M} = \varinjlim_{\widetilde{M}} B\big(\widetilde{M}\big).
\end{split}
\end{equation*}
Here, the inductive limit (in the category of locally convex topological vector spaces) is over all the complexifications $\smash{\widetilde{M}}$ of $M$. This coincides with \eqref{eqtopgsigma} when $s = 1$. This definition may be quite appealing because it is very intrinsic. However, we will rather use a more concrete description of $\G^1(M)$, that boils down to choosing a particular basis of complex neighbourhoods for $M$.
\end{remark}

We will use particular complexifications of $M$ called Grauert tubes. The notion of Grauert tube first appeared in \cite{grauertLeviProblemImbedding1958}, but we will rely on the exposition from 
\cite{guilleminGrauertTubesHomogeneous1991}. First, according to \cite{morreyAnalyticEmbeddingAbstract1958}, there is a real-analytic embedding of $M$ into an Euclidean space. Hence, we may choose a real-analytic Riemannian metric $g$ on $M$. According to \cite{Bruhat-Whitney-59}, there exists a complexification $\smash{\widetilde{M}}$ of $M$ endowed with an anti-holomorphic involution $z \mapsto \bar{z}$ such that $M$ is the set of fixed point of $z \mapsto \bar{z}$. Then, since the square of the distance induced by $g$ on $M$ is real-analytic near the diagonal, it extends to a holomorphic function on a neighbourhood of the diagonal of $M$ in $\smash{\widetilde{M} \times \widetilde{M}}$. Following \cite{guilleminGrauertTubesHomogeneous1991}, we define $\rho$ on $\smash{\widetilde{M}}$ (up to taking $\smash{\widetilde{M}}$ smaller) by
\begin{equation*}
\begin{split}
\rho(z) = - \frac{1}{4} d(z,\bar{z})^2.
\end{split}
\end{equation*}
From \cite{guilleminGrauertTubesHomogeneous1991}, we know that $\rho$ defines a strictly plurisubharmonic function on $\smash{\widetilde{M}}$ such that $\smash{M = \{z \in \widetilde{M} : \rho(z) = 0\}}$. Then, if $\epsilon > 0$ is small, we define the Grauert tube $(M)_\epsilon$ as the sub-level of $\rho$:
\begin{equation}\label{eq:def_Grauert_tube}
\begin{split}
(M)_\epsilon \coloneqq \set{z \in \widetilde{M} : \rho(z) < \epsilon^2}.
\end{split}
\end{equation}
Notice that, since $\rho$ is strictly plurisubharmonic, the Grauert tube $(M)_\epsilon$ is strictly pseudo-convex. Moreover, the real $(1,1)$-form $i \partial \bar{\partial} \rho$ is Kähler and the associated hermitian form coincides with $g$ on $M$. We will consequently still denote this hermitian form by $g$.

Using the notion of Grauert tube, we can replace the spaces $\widetilde{E}^{1,R}\p{M}$ that appeared in \S \ref{sec:def-Gevrey} with a more convenient scale. If $R \geq 1$ is large enough, denote by $\widetilde{E}^{1,R}\p{M}$ the space of bounded holomorphic functions on $(M)_{1/R}$ (endowed with the sup norm). Here, we only work with $R$ large enough so that $(M)_{1/R}$ is well-defined and $\widetilde{E}^{1,R}(M)$ is non-trivial. The spaces $\widetilde{E}^{1,R}\p{M}$ defined in this way do not need to coincide with those of \S \ref{sec:def-Gevrey}, but they give rise to the same inductive limit (in the category of locally convex topological vector spaces):
\begin{equation}\label{eq:description_gone}
\begin{split}
\G^1\p{M} = \varinjlim_{R \to + \infty} \widetilde{E}^{1,R}\p{M}.
\end{split}
\end{equation}
Hence, we will always assume that we use this definition of the spaces $\widetilde{E}^{s,R}\p{M}$ in the case $s=1$. To see that we get the same topology on $\G^1\p{M}$ as in \S \ref{sec:def-Gevrey}, work as in Remark \ref{remark:choix_de_cartes} (using Cauchy's and Taylor's formulae to prove an estimate similar to \eqref{eq:presque_equivalence}). As in the case $s > 1$, it will be convenient to work with slightly smaller spaces than the $\widetilde{E}^{1,R}(M)$. Hence, we define $E^{1,R}(M)$ as the closure of $\widetilde{E}^{1,R'}(M)$ in $\widetilde{E}^{1,R}(M)$ for $R' > R$. This definition makes sense when $R$ is large enough (so that we can choose $R' > R$ with $\smash{(M)_{1/R'}}$ well defined) and it follows from the Oka--Weil Theorem that it does not depend on the choice of $R'$. Notice that we can then replace $\widetilde{E}^{1,R}(M)$ by $E^{1,R}(M)$ in \eqref{eq:description_gone} and that \eqref{eq:inductive_ultradistribution} also holds when $s = 1$. However, the choice of topologies on $\G^1\p{M}$ and $\U^1(M)$ is not very important, since we will have to work directly with the spaces $E^{1,R}\p{M}$ and $(E^{1,R}(M))'$ in order to get the best order for the dynamical determinant in Theorem \ref{thm:main}. As in the case $s > 1$, the point of using smaller spaces is that if $R < R'$ then we have an injection of $\smash{(E^{1,R'}(M))'}$ in $\smash{(E^{1,R}(M))'}$.

Let us point out that there is another way to describe Grauert tubes. Indeed, since the Riemannian metric $g$ is real-analytic, so is its exponential map. Consequently, if $x$ is a point in $M$, then $\exp_x : T_x M \to M$ extends to a holomorphic map, still denoted by $\exp_x$, from a neighbourhood of $0$ in $T_x M \otimes \C$ to $\smash{\widetilde{M}}$. Then, the map
\begin{equation}\label{eq:exp_Grauert}
\begin{split}
(x,v) \mapsto \exp_x(iv)
\end{split}
\end{equation}
defines a real analytic diffeomorphism between a neighbourhood of the zero section in $TM$ and a neighbourhood of $M$ in $\smash{\widetilde{M}}$. For $\epsilon > 0$ small enough, under the map in \eqref{eq:exp_Grauert}, the Grauert tube $(M)_\epsilon$ is the image of
\begin{equation*}
\begin{split}
\set{(x,v) \in TM : g(v,v) < \epsilon^2}
\end{split}
\end{equation*}
by the map \eqref{eq:exp_Grauert}. With this description of the Grauert tube $(M)_\epsilon$, we see that the projection $TM \to M$ induces a real-analytic projection from $(M)_\epsilon$ to $M$, that we will denote by $\Re$. We define also the function $\va{\Im}: (M)_\epsilon \to \R_+$ as the square root of $\rho$. We will sometimes, slightly abusively, write $\va{\Im z}$ instead of $\va{\Im}(z)$.

Since $M$ is compact, we could have chosen any decreasing basis of neighbourhoods for $M$ in $\smash{\widetilde{M}}$ to define the spaces $E^{1,R}\p{M}$. However, we will need to consider real-analytic functions defined on $T^* M$ (for instance symbols) or more generally on products of the type $M^{N_1} \times \p{T^* M}^{N_2}$. Since these manifolds are non-compact, the choice of a complex neighbourhood for $T^* M$ becomes non-trivial. As we want to consider symbols on $T^* M$, it seems natural to introduce Grauert tubes using a metric adapted to the Kohn--Nirenberg classes of symbol. Consequently, we define a metric $g_{KN}$ on $T^* M$, that will be called Kohn--Nirenberg metric, in the following way. The Levi--Civita connexion associated with $g$ gives a splitting $T(T^\ast M) = V \oplus H$ into vertical and horizontal bundles, where both subbundles are identified with $TM$, so that we can define
\[
g_{KN}(x,\xi)\coloneqq g_H(x,\xi) + \frac{1}{1 + |\xi|^2_x} g_V(x,\xi).
\] 
In charts, it is uniformly equivalent to its flat version
\[
g_{KN}^{\text{flat}} = \mathrm{d}x^2 + \frac{1}{1+\xi^2} \mathrm{d}\xi^2.
\]
The curvature of $g_{KN}$ is bounded, and so are all its covariant derivatives, as one can check using the tools presented in \cite{Gudmundsson-Kappos}, and one can check that it admits Grauert tubes, which look like conical neighbourhoods of $T^\ast M$ at infinity. 

More precisely, the cotangent space $T^* \smash{\widetilde{M}}$ of $\smash{\widetilde{M}}$ is a complexification of $T^* M$. Notice that there is a natural inclusion of $T^* M \otimes \C$ into $\smash{T^* \widetilde{M}}$ and that the anti-holomorphic involution that fixes $T^* M$ is given on $T^* M \otimes \C$ by $(x,\xi) \mapsto (x,\bar{\xi})$. As above, we find a strictly plurisubharmonic function $\rho_{KN}$ defined on a neighbourhood of $T^*M $ in $\smash{T^* \widetilde{M}}$. Then, mimicking \eqref{eq:def_Grauert_tube}, we set for $\epsilon > 0$ small enough
\begin{equation*}
\begin{split}
\p{T^* M}_\epsilon \coloneqq \set{ \alpha \in T^* \widetilde{M} : \rho_{KN}(\alpha) < \epsilon^2}.
\end{split}
\end{equation*}
As in the compact case, $\va{\Im \alpha}$ will denote the square root of $\rho_{KN}(\alpha)$. To describe these tubes in more concrete terms, we may examine them in local coordinates. Given a real-analytic chart for $M$, it extends holomorphically to a chart for $\smash{\widetilde{M}}$. It hence defines a holomorphic trivialization for $\smash{T^* \widetilde{M}}$, mapping $T^* M$ on $T^* \R^n$. If we denote these coordinates by $\tilde{x} = x + iy, \tilde{\xi} = \xi + i \eta$, then for a point $\alpha \in T^* \widetilde{M}$ that writes $(\tilde{x},\tilde{\xi})$ in local coordinates, the quantity $\va{\Im \alpha}$ is uniformly equivalent to 
\begin{equation*}
\begin{split}
\va{\Im \alpha} \asymp \va{y} + \frac{\va{\eta}}{\jap{\xi}}
\end{split}
\end{equation*}
when $\alpha$ remains in a small conical neighbourhood of $T^* M$. This gives a rough but tractable idea of the shape of Grauert tubes in local coordinates. 

Let us discuss some others notations. Since $x\mapsto g_x$ is a K\"ahler metric, the map $\alpha\mapsto g_{\alpha_x}(\alpha_\xi,\alpha_\xi)$ is real analytic and non-negative. On the other hand, we can consider the holomorphic extension $\tilde{g}$ of $g$, so that $\alpha\mapsto \tilde{g}_{\alpha_x}(\alpha_\xi,\alpha_\xi)$ is a holomorphic map. With the determination of the square root positive on the reals, we define for $\alpha$ in $\smash{T^* \widetilde{M}}$ the Japanese brackets
\begin{equation}\label{eq:def_jap_brack}
\begin{split}
\jap{\alpha}  = \sqrt{1 + \tilde{g}_{\alpha_x}(\alpha_\xi,\alpha_\xi)} \textup{ and } \jap{\va{\alpha}} = \sqrt{1 + g_{\alpha_x}(\alpha_\xi,\alpha_\xi)}.
\end{split}
\end{equation}
Hence, $\jap{\alpha}$ is holomorphic in $\alpha$, while $\jap{\va{\alpha}}$ is not. However, notice that on a Grauert tube $\p{T^* M}_\epsilon$, for $\epsilon > 0$ small enough, the positive quantities $\jap{\va{\alpha}},\va{\jap{\alpha}}$ and $\Re \jap{\alpha}$ are uniformly equivalent. Notice also that we may define a Kohn--Nirenberg metric on $T^* \widetilde{M}$ (since $T^* \widetilde{M}$ identifies with the cotangent bundle of $\smash{\widetilde{M}}$ seen as a real-analytic manifold).

Finally, let us mention that we will use simpler definitions of the notions above when working on $\R^n$. If $U$ is an open subset of $R^n$ and $\epsilon > 0$ small, we define the complex neighbourhood
\begin{equation}\label{eq:Grauert_rn}
\begin{split}
\p{T^* U}_\epsilon = \set{ (x,\xi) \in \C^n : d(x,U) < \epsilon \textup{ and } \va{\Im \xi} < \epsilon \jap{\Re \xi}}
\end{split}
\end{equation}
of $T^* U$. We think of $\p{T^* U}_\epsilon$ as (an approximation of) the Grauert tube of $T^* U$ for the Kohn--Nirenberg metric, even if this case is not covered by our discussion above (since $U$ is not a compact manifold). Similarly, for $\xi \in \C^n$ with $\Re \jap{\xi,\xi} \geq 0$ we define the Japanese brackets (we take the determination of the square root that is positive on $\R_+^*$)
\begin{equation*}
\begin{split}
\jap{\xi} = \sqrt{1 + \jap{\xi,\xi}} \textup{ and } \jap{\va{\xi}} = \sqrt{1 + \n{\xi}^2}.
\end{split}
\end{equation*}
Here, the norm in the definition of $\jap{\va{\xi}}$ is the one given by the identification $\C^n \simeq \R^{2n}$, while $\jap{\xi,\xi} = \sum_{j = 1}^n \xi_j^2$ if $\xi = (\xi_1,\dots,\xi_n)$. Hence $\jap{\xi}$ is holomorphic but $\jap{\va{\xi}}$ is not. As in the manifold case, for $U$ an open subset of $\R^n$ and $\p{x,\xi} \in (T^* U)_\epsilon$, the quantities $\jap{\va{\xi}}, \jap{\Re \xi}, \va{\jap{\xi}}$ and $\Re \jap{\xi}$ are equivalent.

\subsubsection{Almost analytic extensions}

Let $\p{M,g}$ be a compact real-analytic Riemannian manifold. To study deformations in the Grauert tube of $M$ (or of $T^\ast M$), we will make extensive use of almost-analytic extensions of smooth functions on $M$. The notion of almost analytic extension was introduced by Lars H\"ormander \cite{hormanderLectureNotes1969} and then by Louis Nirenberg \cite{nirenbergProofMalgrangePreparation1971}. It has become a very common notion in microlocal analysis, and are essential in \cite{Melin-Sjostrand-75} for instance.

Recall that if $f$ is a $\mathcal{C}^\infty$ function on $M$ then an almost-analytic extension for $f$ is a compactly supported $\mathcal{C}^\infty$ functions $\tilde{f}$ on some $(M)_\epsilon$ that coincides with $f$ on $M$ and such that $\bar{\partial} \tilde{f}$ vanishes to all orders on $M$. It is classical that such a $\tilde{f}$ exists \cite[Theorem 3.6]{zworskiSemiclassicalAnalysis2012}. While this is hardly surprising, it will be crucial in our analysis that if $f$ is $\G^s$ then $\tilde{f}$ may be chosen $\G^s$ as well. This will allow us to make the flatness of $\bar{\partial} \tilde{f}$ near $M$ quantitative using Lemma \ref{lemma:decay-extension-gevrey} below. 

Observe that in \cite{furdosAlmostAnalyticExtensions2019} (for example), weaker forms of almost analytic extensions were obtained with the adequate decay. However, they are only $C^1$, which is not sufficient for our purposes. One can also find in \cite{jungPhaseSpaceTunneling2000} another notion of almost analytic extension: instead of working with a single function $\tilde{f}$ such that $\bar{\partial}\tilde{f}$ vanishes to all orders on $M$, Jung introduces a family of functions $(\tilde{f}_h)_{h > 0}$, where $f_h$ is defined on a tube whose size depends on $h$ and such that $\bar{\partial}f_h$ vanishes when $h$ tends to $0$.

We start by proving that Gevrey functions admit Gevrey almost analytic extensions.
\begin{lemma}\label{lemma:almost-analytic-extension-Gs}
Given a Grauert tube $(M)_\epsilon \supset M$, for each $s>1$, there exists a constant $C_s$ and a compact neighbourhood $K \subseteq (M)_\epsilon$ of $M$ such that for all $R \geq 1$, there exists a bounded map $f \mapsto \tilde{f}$ from $\widetilde{E}^{s,R}(M)$ to $\widetilde{E}^{s,C_s R}(K)$, such that $\tilde{f}$ is an almost-analytic extension for $f$.
\end{lemma}

Covering $M$ by real-analytic charts and then using a $\G^{s}$ partition of unity, we reduce to the Euclidean case, that is, we only need to prove:

\begin{lemma}\label{lemma:aae_Gs_flat}
Let $s> 1$. Let $K$ be a compact subset of $\R^n$ and $K'$ be a compact neighbourhood of $K$ in $\C^n$. Then, there is $M > 0$, and for each $R > 0$ there is a continuous map $f \mapsto \tilde{f}$ from $\widetilde{E}^{s,R}\p{K}$ to $\widetilde{E}^{s,MR}\p{K'}$ such that for every $f \in \widetilde{E}^{s,R}\p{K}$, the function $\tilde{f}$ is an almost analytic extension for $f$.
\end{lemma}

It seems to be folklore that this may be deduced from results of Lennart Carleson on universal moment problems \cite{carlesonUniversalMomentProblems1961}, but we are not aware of any reference containing a proof, and thus we provide one.

\begin{proof}
First of all, we need to extend to higher dimensions the one dimensional results of Carleson. This is quite straightforward, since most difficulties are already present in the one dimensional case.  For $x \in \R$ we define
\[
w(x) = \p{1 + x^2}^{-1} \exp\p{- \frac{2 s \va{x}^{\frac{1}{s}}}{e}}.
\]
Let us denote by $H$ the Hilbert space of measurable functions (up to modification on zero measure sets) $u$ from $\R^n$ to $\C$ that satisfy
\begin{equation*}
\n{u}_H^2 \coloneqq \int_{\R^n} \va{u(x)}^2 \bar{w}(x) \mathrm{d}x < + \infty,
\end{equation*}
where $\bar{w}(x) = \prod_{j=1}^n w(x_j)$. Let $\p{P_m}_{m \in \N}$ be the sequence (depending on $s$) of orthogonal polynomials, defined, up to a sign, by
\begin{equation*}
\forall m,p \in \N : \int_{\R} P_m(t) P_p(t) w(t) \mathrm{d}t = \begin{cases} 1 & \textrm{ if } m = p, \\ 0 & \textrm{ otherwise}, \end{cases}
\end{equation*}
and $\deg P_m = m$ for every $m \in \N$. Then define for $\alpha \in \N^n$ the polynomial
\begin{equation*}
P_\alpha (x) = \prod_{j=1}^n P_{\alpha_j}(x_j),
\end{equation*}
and notice that the $P_\alpha$'s form an orthogonal family in $H$. According to \cite[(2.6)]{carlesonUniversalMomentProblems1961}, there are constants $C,r > 0$ such that for every $k \in \N$ we have
\begin{equation*}
\sum_{\ell=0}^{+ \infty} \va{P_\ell^{(k)}\p{0}}^2 \leq C r^{2k} k!^{2(1- s)}.
\end{equation*}
For some new constants $C$ and $r$, and for every $\beta \in \N^n$, it follows that
\begin{equation*}
\sum_{\alpha \in \N^n} \va{\partial^\beta P_{\alpha} (0)}^2 \leq C r^{2\va{\beta}} \beta!^{2(1 - s)}.
\end{equation*}
Now, let $S$ denote the Hilbert space of sequences $s = \p{s_\alpha}_{\alpha \in \N^n}$ of complex numbers such that 
\begin{equation*}
\n{s}_S^2 \coloneqq \sum_{\alpha \in \N^n} \p{\frac{\va{s_\alpha} \p{2r}^{\va{\alpha}}}{\alpha!^s}}^2 < + \infty.
\end{equation*}
If $s \in S$, define the sequence $\p{b_\alpha}_{\alpha \in \N^n}$ by
\begin{equation*}
b_\alpha = \sum_{\beta \in \N^n} \frac{\partial^\beta P_\alpha(0)}{\beta!} s_\beta, 
\end{equation*}
for $\alpha \in \N^n$ (notice that this sum is finite). Then we have
\begin{align*}
\sum_{\alpha \in \N^n} \va{b_\alpha}^2 & = \sum_{\alpha \in \N^n} \va{\sum_{\beta \in \N^n} \frac{\partial^\beta P_\alpha(0)}{\beta!} s_\beta }^2 \ , \\ 
    & \leq \n{s}_S^2 \sum_{\alpha \in \N^n} \p{\sum_{\beta \in \N^n} \frac{\va{\partial^\beta P_\alpha(0)}^2}{\beta!^2} \beta!^{2s} \frac{1}{\p{2 r}^{2\va{\beta}}}} \ , \\
    & \leq \n{s}_S^2 \sum_{\beta \in \N^n} \beta!^{2(s-1)} \frac{1}{\p{2 r}^{2 \va{\beta}}} \sum_{\alpha \in \N^n} \va{\partial^\beta P_\alpha(0)}^2 \ , \\
    & \leq C \n{s}_S^2 \sum_{\beta \in \N^n} 2^{- 2\va{\beta}} = C \p{\frac{4}{3}}^n \n{s}_S^2.
\end{align*}
Thus, if we set
\begin{equation*}
L(s) = \sum_{\alpha \in \N^n} b_\alpha P_\alpha,
\end{equation*}
then $L$ is a bounded operator from $S$ to $H$. The main point about $L$ is that for every $s \in S$ and $\alpha \in \N^n$ an elementary computation ensures that
\begin{equation}\label{eqL}
\int_{\R^n} L(s)(t) t^\alpha \bar{w}(t) \mathrm{d}t = s_\alpha.
\end{equation}

Now, if $f \in \widetilde{E}^{s,R}\p{K}$, we define for every $x \in \R^n$ the sequence $s(x) = \p{s_\alpha(x)}_{\alpha \in \N^n}$ by
\begin{equation*}
s_\alpha (x) = \frac{ \partial^\alpha f(x)}{b^{\va{\alpha}}}, 
\end{equation*}
where $b =  3\times 2^s R r$. Notice that $s(x) \in S$ and define $\tilde{f}$ for $x,y \in \R^n$ by
\begin{equation}\label{eq:def-almost-analytic-extension}
\tilde{f}(x+iy) = \chi(y) \int_{\R^n} e^{ib\jap{t,y}} L(s(x))(t) \bar{w}(t) \mathrm{d}t = \chi(y) \langle e^{i b \jap{y, \cdot}}, L(s(x)) \rangle_{H},
\end{equation}
where $\chi$ is a compactly supported $s'$-Gevrey function for some $1 < s' < s$, identically equal to $1$ on a neighbourhood of $0$. One easily checks that the map $x \mapsto s(x) \in S$ is $\mathcal{C}^\infty$ and supported in $K$. Moreover, if $x \in K$ and $\beta \in \N^n$ then we have
\begin{equation}\label{eqdevs}
\partial^\beta s(x) = \p{\frac{\partial^{\alpha + \beta} f(x)}{b^{\va{\alpha}}}}_{\alpha \in \N^n}.
\end{equation}
Thus, thanks to our choice of $b$ we have for all $\beta \in \N^n$ and $x \in K$
\begin{equation}\label{eqgev2}
\n{\partial^\beta s(x)}_S \leq 3^n \n{f}_{s,R,K} (2^s R)^{\va{\beta}} \beta!^s.
\end{equation}
Thus $x \mapsto s(x)$ is $\G^s$ with inverse radius $2^sR$. By dominated convergence, we see that the map $G : y \mapsto (t \to e^{ib\jap{t,y}})$ is $\mathcal{C}^\infty$ from $\R^n$ to $H$. Moreover, if $\alpha \in \N^n$ the derivative $\partial^\alpha G$ is $y \mapsto \p{t \mapsto (ibt)^{\va{\alpha}} e^{ib\jap{t,y}}}$ and its norm at $y \in \R^n$ is
\begin{align*}
\n{\partial^\alpha G(y)}_H^2 & = b^{2 \va{\alpha}} \int_{\R^n} \va{t^\alpha}^2 \bar{w}(t) \mathrm{d}t = b^{2 \va{\alpha}} \prod_{j=1}^n \int_{\R} \va{t}^{2 \alpha_j} w(t) \mathrm{d}t \ , \\
     & \leq b^{2\va{\alpha}} \p{\frac{e}{2s}}^{2s \va{\alpha}} \p{\frac{s e^s}{s^s}}^n \prod_{j=1}^n \Gamma\p{s(2 \alpha_j+1)}.
\end{align*}
Then, applying Stirling's formula, we see that there are constants $C,M > 0$ such that for every $\alpha \in \N^n$ and $y \in \R^n$ we have
\begin{equation}\label{eqgev1}
\n{\partial^\alpha G(y)}_H \leq C (MR)^{\va{\alpha}} \alpha!^s.
\end{equation}
Consequently, the map $F : x+iy \mapsto \langle e^{i b \jap{y, \cdot}}, L(s(x)) \rangle_{H}$ is $\mathcal{C}^\infty$ on $\C^n$ and if $\alpha, \beta \in \N^n$ and $x,y \in \R^n$ then we have 
\begin{equation}\label{eqderivees}
\partial_x^\alpha \partial_y^\beta F(x+iy) = \langle \partial^\beta G(y) , L(\partial^\alpha s(x)) \rangle_H.
\end{equation}
Thus, from \eqref{eqgev2} and \eqref{eqgev1}, we see that the map $f \mapsto \tilde{f}$ is continuous from $\tilde{E}^{s,R}\p{K}$ to $\tilde{E}^{s,MR}\p{K'}$ for $K'$ a compact complex neighbourhood of $K$ (we may and do assume that $M \geq 2^s$). 

It remains to see that $\tilde{f}$ is indeed an almost analytic extension for $f$. Since $\tilde{f}$ and $F$ coincide near $\R^n$, we only need to study $F$ to do so. Notice that for all $x \in \R^n$ and $\alpha,\beta \in \N^n$, according to \eqref{eqderivees}, \eqref{eqL} and \eqref{eqdevs}, we have
\begin{equation*}
\partial_x^\alpha \partial_y^\beta F(x) = \int_{\R^n} (ibt)^{\beta} L\p{\partial^\alpha s(x)}(t) \bar{w}(t) \mathrm{d}t = i^{\va{\beta}} \partial^{\alpha + \beta} f(x).
\end{equation*}
In particular, $F$ and $f$ coincide on $\R^n$. Now, if $j \in \set{1,\dots,n}$, we have $\frac{\partial F}{\partial \bar{z}_j} = \frac{1}{2} \p{\frac{\partial F}{\partial x_j} + i \frac{\partial F}{\partial y_j}}$, and thus if $x \in \R^n$ and $\alpha,\beta \in \N^n$ we have
\begin{align*}
\frac{\partial }{\partial \bar{z}_j}\p{\partial_x^\alpha \partial_y^\beta F}(x) & = \frac{1}{2}\p{\partial_x^{\alpha+e_j} \partial_y^\beta F + i \partial_x^\alpha \partial_y^{\beta+ e_j} F} \\
    & = \frac{1}{2}\p{i^{\va{\beta}} \partial^{\alpha + \beta + e_j} f(x) + i^{\va{\beta} + 2} \partial^{\alpha + \beta + e_j} f(x)} = 0.
\end{align*}
Consequently, $\bar{\partial}F$ vanishes to all orders on $\R^n$, and $\tilde{f}$ is indeed an almost analytic extension of $f$.
\end{proof}

In order to apply Lemma \ref{lemma:almost-analytic-extension-Gs}, we need to investigate the way a Gevrey function can be flat. To do so, we will apply the ``sommation au plus petit terme'', a method for regularizing certain divergent series that is particularly well suited for Taylor series of Gevrey functions. The interested reader may refer to \cite{ramisSeriesDivergentesTheories} for details and historical references.
\begin{lemma}\label{lemma:decay-extension-gevrey}
Let $U$ be an open subset of $\R^n$ and $K$ a compact and convex subset of $U$. Then for every $s > 1$, there are constants $C, C_0 > 0$ such that for every $R > 0$ and $f \in \mathcal{C}^\infty\p{U}$ such that the quantity $\n{f}_{s,R,K}$ defined by \eqref{eqnormrn} is finite, if $x ,y \in K$ then we have
\begin{equation}\label{eq:Taylor_Gevrey}
\begin{split}
\va{f(y) - \sum_{\va{\alpha} \leq \frac{\p{R \va{x-y}}^{- \frac{1}{s-1}}}{C_0}} \frac{\partial^\alpha f(x)}{\alpha!} \p{y - x}^\alpha} \leq C \n{f}_{s,R,K} \exp\p{- \frac{1}{C\p{R\va{x-y}}^{\frac{1}{s-1}}}}. 
\end{split}
\end{equation} 
\end{lemma}

We explain in Remark \ref{remark:size_d_bar} below how Lemma \ref{lemma:decay-extension-gevrey} allows to control the size of the Cauchy--Riemann operator applied to an almost analytic extension of a Gevrey function.

\begin{proof}[Proof of Lemma \ref{lemma:decay-extension-gevrey}]
Taylor's formula gives for every positive integer $k$
\begin{equation*}
f(y) = \sum_{\va{\alpha} < k} \frac{\partial^\alpha f(x)}{\alpha!} \p{y - x}^\alpha + \sum_{|\alpha|=k} \frac{k}{\alpha !}(y-x)^\alpha\int_0^{1} (1-t)^{k-1}  \partial^\alpha f(x+t(y-x)) \mathrm{d}t.
\end{equation*}
Thus, recalling the definition \eqref{eqnormrn} of the norm $\n{\cdot}_{s,R,K}$, we find by direct estimation
\begin{align*}
\va{f(y) - \sum_{\va{\alpha} < k} \frac{\partial^\alpha f(x)}{\alpha!} \p{y - x}^\alpha} 	&\leq \n{f}_{s,R,K} \p{n R \va{x-y}}^k \times k!^{s-1} \ ,\\
			&\leq \n{f}_{s,R,K} \p{n R \va{x-y}}^k k^{\p{s - 1}k}.
\end{align*}
This suggests to take $C_0 = e\ n^{1/(s-1)}$. When $n R|x-y| \geq e^{1-s}$, estimate \eqref{eq:Taylor_Gevrey} follows from a $L^\infty$ bound on $f$. When on the other hand $n R|x-y| < e^{1-s}$, we get $ k = \left\lfloor \frac{1}{e} \p{n R \va{x-y}}^{- \frac{1}{s - 1}} \right\rfloor + 1$. Studying the variations of the function $x\log x$ precisely, we get consequently, for some $C > 0$ depending only on $s$,
\begin{equation*}
\begin{split}
& \va{f(y)- \sum_{\va{\alpha} \leq \frac{1}{e}\p{n R \va{x - y}}^{- \frac{1}{s-1}}} \frac{\partial^\alpha f(x)}{\alpha!} \p{y - x}^\alpha} \\ & \qquad \qquad \qquad \quad \leq C \n{f}_{s,R,K} \exp\p{- \frac{(s-1)}{e \p{n R \va{x-y}}^{\frac{1}{s-1}}} },
\end{split}
\end{equation*}
and the lemma is proved.
\end{proof}

\begin{remark}\label{remark:choix_C0}
We will mostly use Lemma \ref{lemma:decay-extension-gevrey} with an $f$ that vanishes to infinite order at $x$. In that case, we get a control on the size of $f(y)$ for $y$ near $x$. However, it will sometimes be useful to have the general result at our disposal. Concerning the general result, notice that the constant $C_0$ in \eqref{eq:Taylor_Gevrey} may be chosen arbitrarily large (up to taking $C$ larger). Indeed, the proof of Lemma \ref{lemma:decay-extension-gevrey} gives $C_0 = 4 (nR)^{\frac{1}{s-1}}$, and we can always replace $R$ by a larger number in the proof.
\end{remark}

\begin{remark}\label{remark:size_d_bar}
Let us explain how Lemma \ref{lemma:decay-extension-gevrey} allows to control the size of an almost analytic extension of a Gevrey function. Let $s,K,K',M$ and $R$ be as in Lemma \ref{lemma:aae_Gs_flat}. We may assume that $K'$ is convex. Then, if $f \in \widetilde{E}^{s,R}\p{K}$ is a $\G^s$ function, we know that it admits a $\G^s$ almost analytic extension $\tilde{f} \in \widetilde{E}^{s,MR}\p{K'}$. Then if $R_1 > MR$, we see that the components of $\bar{\partial} \tilde{f}$ belongs to $\widetilde{E}^{s,R_1}\p{K'}$. By assumption, these components vanish at infinite order on $\R^n$.  Hence, it follows from Lemma \ref{lemma:decay-extension-gevrey} that there are constants $C > 0$ (that only depends on $K$ and $K'$) and $C_R$ (that may also depend on $R$) such that for $z \in \C^n$ we have
\begin{equation}\label{eq:size_d_bar}
\begin{split}
\va{\bar{\partial} \tilde{f}(z)} \leq C_R \n{f}_{s,R,K} \exp\p{- \frac{1}{C\p{R \va{\Im z}}^{\frac{1}{s-1}}}}.
\end{split}
\end{equation}
Here, we used the fact that the norms of the coordinates of $\bar{\partial} \tilde{f}$ in $\widetilde{E}^{s,R_1}\p{K'}$ are controlled by $\n{f}_{s,R,K}$. 

It will be useful to control also the derivatives of $\bar{\partial} \tilde{f}$. In fact, we can improve \eqref{eq:size_d_bar} into a Gevrey estimates. Indeed, if $\alpha,\beta \in \N^n$ then the components of $\partial_x^\alpha \partial_y^\beta \bar{\partial} \tilde{f}$ are $\G^s$ (we write $z = x+iy$ for the coordinates in $\C^n$). To see so, just notice that for $\alpha',\beta' \in \N^n, j \in \set{1,\dots,n}$ and $z \in K'$ we have ($C > 0$ only depends on $K$ and $K'$ and may vary from one line to another, $C_R$ may also depend on $R$)
\begin{equation*}
\begin{split}
& \va{\partial_x^{\alpha + \alpha'} \partial_y^{\beta + \beta'} \frac{\partial \tilde{f}}{\partial \bar{z}_j}(z)} \\ & \qquad \qquad \leq \n{\tilde{f}}_{s,MR,K'} \p{MR}^{1+\va{\alpha} + \va{\alpha'} + \va{\beta} + \va{\beta'}} \p{\alpha + \alpha' + e_j}!^s\p{\beta + \beta' + e_j}!^s \\
    & \qquad \qquad \leq C_R \n{f}_{s,R,K} \p{C R}^{\va{\alpha} + \va{\beta}} \alpha!^s \beta!^s \p{CR}^{\va{\alpha'} + \va{\beta'}} \p{\alpha'}!^s \p{\beta'}!^s.
\end{split}
\end{equation*}
This estimate can be rewritten as
\begin{equation*}
\begin{split}
\n{\partial_x^\alpha \partial_y^\beta \frac{\partial \tilde{f}}{\partial \bar{z}_j}}_{s,CR,K'} \leq C_R \n{f}_{s,R,K} \p{C R}^{\va{\alpha} + \va{\beta}} \alpha!^s \beta!^s.
\end{split}
\end{equation*}
Since $\partial_x^\alpha \partial_y^\beta \frac{\partial \tilde{f}}{\partial \bar{z}_j}$ vanishes at all orders on $\R^n$, it follows from Lemma \ref{lemma:decay-extension-gevrey} that there are constants $C > 0$, that only depends on $K$ and $K'$, and $C_R > 0$ that may also depend on $R$, such that for every $z \in \C^n$ and $\alpha,\beta \in \N^n$ we have
\begin{equation*}
\begin{split}
\va{\partial_x^\alpha \partial_y^\beta \bar{\partial}\tilde{f}(z)} \leq C_R \n{f}_{s,R,K} \p{C R}^{\va{\alpha} + \va{\beta}} \alpha!^s \beta!^s \exp\p{- \frac{1}{C\p{R \va{\Im z}}^{\frac{1}{s-1}}}}.
\end{split}
\end{equation*}
\end{remark}

To close this section, we present here a trick that will be useful to derive Gevrey bounds from $L^\infty$ bounds on almost analytic extensions. This trick is a slight refinement of the proof of \cite[Proposition 3.8]{furdosAlmostAnalyticExtensions2019} in the Gevrey case (which is in some sense the reciprocal of Lemma \ref{lemma:aae_Gs_flat}). 

\begin{lemma}[Bochner--Martinelli trick]\label{lemma:Bochner-Martinelli-trick}
Let $s \geq 1$. Let $D$ be a ball in $\R^n$. There is a constant $\widetilde{C}> 0$ such that, for every $R \geq 1$, there is a constant $C_R$ such that the following holds. For every $\lambda \geq 1$, let $f_\lambda$ be a function from $\R^n$ to $\C$. We assume that for all $\lambda \geq 1 $ the map $f_\lambda$ admits a $\mathcal{C}^1$ extension $F_\lambda$ to a neighbourhood of $D$ in $\C^n$, and there exist $R > 0$ and $C > 0$ such that
\begin{enumerate}[label=(\roman*)]
	\item for every $\lambda >1$ and $x \in \C^n$ if $x$ is at distance less than $R^{-\frac{1}{s}} \lambda^{\frac{1}{s}-1}$ of $\overline{D}$ then 
\begin{equation*}
\va{F_\lambda(x)} \leq C \exp\p{- \p{\frac{\lambda}{R}}^{\frac{1}{s}}};
\end{equation*}
	\item for every $\lambda >1$ and $x \in \C^n$ if $x$ is at distance less than $R^{-1/s} \lambda^{\frac{1}{s}-1}$ of $\overline{D}$ then
\begin{equation*}
\va{\bar{\partial} F_\lambda(x)} \begin{cases} \leq C \exp\p{- (R |\Im x|)^{- \frac{1}{s-1}}} & \textrm{ if } s > 1 \\ =  0 & \textrm{ if } s = 1. \end{cases}
\end{equation*}
\end{enumerate}
Then for $R' = \widetilde{C} R$ and every $\lambda \geq 1$ we have 
\begin{equation*}
\n{f_\lambda}_{s,R',\overline{D}} \leq C_R C \exp\p{-\p{\frac{\lambda}{\widetilde{C} R}}^{\frac{1}{s}}}.
\end{equation*}
\end{lemma}

\begin{proof}
The case $s = 1$ is just an application of Cauchy's formula. Consequently, we focus on the case $s > 1$ which is more crucial.

As the name suggests, the proof in the case $s > 1$ is based on an application of the Bochner--Martinelli formula. For every $\lambda \geq 1$ denote by $D_\lambda$ the $R^{-1/s}\lambda^{\frac{1}{s}-1}/2$ neighbourhood of $D$. Then choose $x \in D$ and write the Bochner--Martinelli formula for $F_\lambda$ at $x$ to find
\begin{equation}\label{eqBMf}
\begin{split}
f_\lambda(x) = \int_{\partial D_\lambda} F_\lambda(z) \omega(z,x) - \int_{D_\lambda} \bar{\partial}F_\lambda(z) \wedge \omega(z,x),
\end{split}
\end{equation}
where $\omega(\cdot,x)$ is the $(n,n-1)$ form (the hat means that the corresponding factor is omitted)
\begin{equation*}
\begin{split}
\omega(z,x) = \frac{\p{n-1}!}{\p{2 i \pi}^n} \frac{1}{\va{z - x}^{2n}} \sum_{j=1}^n \p{\bar{z}_j - \bar{x}_j} \mathrm{d}\bar{z}_1 \wedge \mathrm{d}z_1 \wedge \dots \wedge \widehat{\mathrm{d}\bar{z}_j} \wedge \dots \wedge \mathrm{d}\bar{z}_n \wedge \mathrm{d}z_n.
\end{split}
\end{equation*}
By Fa\`a di Bruno's formula and the Leibniz rule, we get as in \cite[Proposition 3.8]{furdosAlmostAnalyticExtensions2019}
\begin{equation*}
\begin{split}
\va{\partial^\alpha_x\p{\frac{\bar{z}_j - \bar{x}_j}{\va{z - x}^{2n}}}} \leq \frac{C(n)^{\va{\alpha}} \alpha!}{\va{x- z}^{2n + \va{\alpha} - 1}},
\end{split}
\end{equation*}
where the constant $C(n)$ only depends on the dimension $n$. Hence, by differentiation under the integral, we see that $f_\lambda$ is $\mathcal{C}^\infty$ with the estimate (for $\lambda \geq 1$, the constant $C > 0$ is from the statement of the lemma and $C_1$ does not depend on $C$)
\begin{equation}\label{eq:vers_BMT}
\begin{split}
\va{\partial^\alpha f_\lambda(x)} & \leq C C_1^{1 + \va{\alpha}} C(n)^{\va{\alpha}} \alpha!\p{R^{\frac{1}{s}} \lambda^{1 -\frac{1}{s}}}^{2n + \va{\alpha} -1} \exp\p{-  \p{\frac{\lambda}{R}}^{\frac{1}{s}}} \\ 
&  \qquad \qquad \qquad+ C C_1 C(n)^{\va{\alpha}} \alpha! \sup_{t \in \left[0,\frac{1}{2}R^{- \frac{1}{s}}\lambda^{\frac{1}{s} - 1} \right]} \frac{\exp\p{- \p{R t}^{- \frac{1}{s-1}}}}{t^{2n + \va{\alpha} - 1}}.
\end{split}
\end{equation}
Then, notice that $\sup_{r \in \R_+^*} e^{-r} r^{(s-1)(2n + \va{\alpha} - 1)} \leq C(s,n)^{1+|\alpha|}\alpha!^{s-1} $, so that
\begin{equation*}
\begin{split}
 & \exp\p{- \frac{ 1 }{2} \p{\frac{\lambda}{R}}^{\frac{1}{s}}} \p{ R^{\frac{1}{s} - 1} \lambda^{1 - \frac{1}{s}}}^{2n + \va{\alpha} -1} \\ & \qquad \qquad \qquad \qquad \qquad \qquad \qquad \leq \sup_{t \in \R_+^*} \exp\p{-\frac{t^{1/s}}{2}} t^{(1 - 1/s)(2n + \va{\alpha} - 1)} \\
    \\ & \qquad \qquad \qquad \qquad \qquad \qquad \qquad \leq  2^{(s-1)(2n + \va{\alpha} - 1)} \sup_{r \in \R_+^*} e^{-r} r^{(s-1)(2n + \va{\alpha} - 1)} \\
    \\ & \qquad \qquad \qquad \qquad \qquad \qquad \qquad \leq C_1^{1 + \va{\alpha}} \alpha!^{s-1},
\end{split}
\end{equation*}
where the constant $C_1 > 0$ only depends on $n$ and $s$. We also have for $t \in \left[0,\frac{1}{2} R^{- \frac{1}{s}} \lambda^{\frac{1}{s} - 1} \right]$ that
\begin{equation*}
\begin{split}
\frac{\exp\p{- \p{R t}^{- \frac{1}{s-1}}}}{t^{2n + \va{\alpha} - 1}} & \leq \exp\p{- 2^{\frac{s}{s-1}} \p{\frac{\lambda}{R}}^{\frac{1}{s}}} \sup_{\tau \in \R_+} \frac{\exp\p{- \frac{1}{2}\p{R \tau}^{- \frac{1}{s-1}}}}{\tau^{2n + \va{\alpha} - 1}} \\
    &\hspace{-20pt} \leq \p{2^{s-1} R}^{2n + \va{\alpha} - 1} \exp\p{- 2^{\frac{s}{s-1}} \p{\frac{\lambda}{R}}^{\frac{1}{s}}} \sup_{r \in \R_+} e^{-r} r^{(s-1)(2n + \va{\alpha} - 1)} \\
    &\hspace{-20pt} \leq C_1^{1 + \va{\alpha}} \alpha!^{s-1} R^{2n + \va{\alpha} - 1} \exp\p{- 2^{\frac{s}{s-1}} \p{\frac{\lambda}{R}}^{\frac{1}{s}}}.
\end{split}
\end{equation*}
The announced result follows by plugging the last two estimates in \eqref{eq:vers_BMT}.
\end{proof}

The Bochner--Martinelli trick will be quite useful when studying oscillatory integrals, for the following reason:
\begin{lemma}\label{lemma:norm-exponentials}
Let $U\subset \R^n$ be a bounded open set, and $\Phi : U \to \C$ be a $\G^s$ function. Let $K\subset U$ be a compact set. Then there exists $C,R_0,h_0>0$ such that for $R\geq R_0$, there exists a constant $C_R>0$, so that for $h\in ]0,h_0]$,
\[
\left\| x \mapsto e^{\frac{i}{h}\Phi(x)} \right\|_{s,R,K} \leq C_R \exp\left( - \frac{1}{h}\inf_{K} \Im \Phi + \frac{1}{C(hR)^{1/s}} \right).
\]
\end{lemma}
In particular, taking $R>0$ large enough, we can obtain an arbitrarily small constant before the $h^{-1/s}$ term.

\begin{proof}
Up to replacing $\Phi$ by $\Phi - i \inf_{K} \Im \Phi$, we may assume that $\inf_{K} \Im \Phi = 0$. Then, we pick $\tilde{\Phi}$ an almost analytic extension for $\Phi$, and we observe that for $x$ at distance at most $ R^{-1/s} h^{1-1/s}$ from $K$, we have
\[
\Im \tilde{\Phi} \geq - C  \frac{h^{1-\frac{1}{s}}}{R^{\frac{1}{s}}}.
\]
On the other hand,
\begin{align*}
\left| \overline{\partial}_x \left( e^{\frac{i}{h}\widetilde{\Phi}(x)}\right) \right| &\leq C \exp\left( C \frac{h^{-1/s}}{R^{\frac{1}{s}}} -  \frac{1}{\p{R_0 |\Im x|}^{\frac{1}{s-1}}} \right),\\
				& \leq C \exp\left( C \frac{h^{-1/s}}{R^{\frac{1}{s}}} -  \frac{1}{\p{R |\Im x|}^{\frac{1}{s-1}}} \right),
\end{align*}
provided that $R \geq R_0$. We can then apply the Bochner--Martinelli trick to $\exp(- 2 C R^{-1/s} h^{ - 1/s}) e^{i \frac{\Phi(x)}{h}}$ and the result follows.
\end{proof}

\subsection{Gevrey and analytic symbols}\label{sec:Gevrey-symbols}

\subsubsection{Definitions}

In anticipation of the study of $\G^s$ pseudors in \S \ref{sec:toolbox_gevrey_pseudor}, we define here Gevrey symbol classes. As explained in the introduction, we will work with larger classes of symbols (and hence of pseudors) than in the classical references dealing with Gevrey pseudo-differential calculus \cite{Boutet-Kree-67,zanghiratiPseudodifferentialOperatorsInfinite1985,Rodino-93-book}. Let us fix an open subset $U$ of $\R^n$ for the remainder of this section.

\begin{definition}\label{def:gevrey_symbol}
Let $m \in \R$, $s\geq 1$. A $\mathcal{C}^\infty$ function $a :T^* U \simeq U \times \R^n \to \C$ belongs to the symbol class $S^{s,m}\p{T^* U}$ if, for every compact subset $K$ of $U$, there are constants $C,R > 0$ such that, for every $\alpha,\beta \in \N^n$ and $(x,\xi) \in K \times \R^n$, we have
\begin{equation}\label{eq:estimee_symbole}
\begin{split}
\va{\partial_x^\alpha \partial_\xi^\beta a(x,\xi)} \leq C R^{\va{\alpha} + \va{\beta}} \p{\alpha! \beta!}^s \jap{\xi}^{m - \va{\beta}}.
\end{split}
\end{equation}
In a more transparent way, we will also write that $a$ is \emph{a $\Gs$ symbol (of order $m$)}.
\end{definition}

\begin{remark}
Symbols are allowed to depend on the small implicit parameter $h > 0$ but then the estimates in Definition \ref{def:gevrey_symbol} (and Definition \ref{def:gevrey_formal_symbol} below) are assumed to hold uniformly in $h > 0$.
\end{remark}

\begin{remark}\label{remark:symboles_generaux}
For $N_1, N_2 \in \N$, we define \emph{mutatis mutandis} classes of symbols 
\[
S^{s,m}\p{\p{T^* U}^{N_1} \times U^{N_2}}.
\]
It is understood that the results and definitions below extend naturally to those more general classes of symbols. Later on, we will mostly consider symbols $a(\alpha,x)$ in $S^{s,m}\p{T^* M \times M}$ that are in fact only defined when $x$ and $\alpha_x$ are close to each other.
\end{remark}

As we did for functions in \S \ref{sec:Grauert}, one can characterize analytic symbols using holomorphic extensions. To do so, we refer to the notion of approximate Grauert tube of an open subset of $\R^n$ introduced in \eqref{eq:Grauert_rn}. Using Taylor's and Cauchy's Formula we can prove (the proof is left as an exercise to the reader):

\begin{lemma}
Let $m \in \R$. Let $a : T^* U \to \C$ be a function. Then $a$ belongs to the symbol class $S^{1,m}\p{T^* U}$ if and only if, for every relatively compact open subset $\Omega$ of $U$, there is $\epsilon > 0$ such that $a$ admits a holomorphic extension to $\p{T^* \Omega}_\epsilon$, which satisfies for some $C > 0$ and every $\p{x,\xi} \in \p{T^* \Omega}_\epsilon$
\begin{equation*}
\begin{split}
\va{a(x,\xi)} \leq C \jap{\Re \xi}^m.
\end{split}
\end{equation*}
\end{lemma}

\begin{remark}\label{remark:aae_for_symbol}
In the absence of holomorphic extensions in the case $s>1$, we will work with almost analytic extensions. For this, we will use many times that for $a \in S^{s,m}\p{T^* U}$ and every relatively compact open subset $\Omega$ of $U$, the symbol $a$ admits an almost analytic extension on $\p{T^* \Omega}_\epsilon$ for some $\epsilon > 0$. This extension satisfies $\G^s$ symbolic estimates as in Definition \ref{def:gevrey_symbol}. In particular, by rescaling Lemma \ref{lemma:decay-extension-gevrey}, we find that there is a constant $C > 0$ such that for every $(x,\xi) \in \p{T^* \Omega}_\epsilon$ we have
\begin{equation}\label{eq:size_dbar_symbol}
\begin{split}
\va{\bar{\partial} a (x,\xi)} \leq C \jap{\Re \xi}^m \exp\p{ - \frac{1}{C} \p{\va{\Im x} + \frac{\va{\Im \xi}}{\jap{\Re \xi}}}^{- \frac{1}{s-1}}}.
\end{split}
\end{equation}
The proof of the existence of such an extension for $a$ is based on Lemma \ref{lemma:almost-analytic-extension-Gs} and a standard rescaling argument, using for instance a partition of unity of Paley--Littlewood type. Observe that the quantity in the exponential is comparable to $\| \Im (x,\xi) \|_{g_{KN}}^{-1/(s-1)}$. As in Remark \ref{remark:size_d_bar}, the estimate \eqref{eq:size_dbar_symbol} can be upgraded to a Gevrey estimates: there are constants $C,R > 0$ such that for every $(x,\xi) \in \p{T^* \Omega}_\epsilon$ and $\alpha,\beta \in \N^{2n}$ we have -- under the identification $\R^{2n} \simeq \C^n$,
\begin{equation*}
\begin{split}
& \va{\partial_x^\alpha \partial_\xi^\beta \bar{\partial} a (x,\xi)} \\
& \qquad \qquad \leq C R^{\va{\alpha} + \va{\beta}} \alpha!^s \beta!^s \jap{\Re \xi}^{m - \va{\beta}} \exp\p{ - \frac{1}{C} \p{\va{\Im x} + \frac{\va{\Im \xi}}{\jap{\Re \xi}}}^{- \frac{1}{s-1}}}.
\end{split}
\end{equation*}
\end{remark}

In the case $s=1$, we will construct in \S \ref{sssecparametrix} a parametrix for elliptic pseudors. To do so, we need the notions of formal analytic symbols and their realization, that we define now.

\begin{definition}\label{def:gevrey_formal_symbol}
Let $m \in \R$. We define an element of the class of formal symbols $FS^{1,m}\p{T^* U}$ to be a formal series $\sum_{k \geq 0} h^k a_k$ where the $a_k$'s are $\mathcal{C}^\infty$ functions from $T^* U$ to $\C$ such that, for every open relatively compact subset $\Omega$ of $U$, there are $\epsilon > 0$ and, for each $k \geq 0$, a holomorphic extension of $a_k$ to $\p{T^* \Omega}_\epsilon$,  and constants $C,R > 0$ such that for every $k \in \N$ and $(x,\xi) \in \p{T^* \Omega}_\epsilon$ we have
\begin{equation}\label{eq:formal_symbol}
\begin{split}
\va{a_k(x,\xi)} \leq C R^k k! \jap{\Re \xi}^{m - k}.
\end{split}
\end{equation}
\end{definition}

\begin{remark}\label{remark:slight_abuse}
If $V$ is an open subset of $\p{T^* U}_\epsilon$ and $\p{a_k}_{k \geq 0}$ is a family of holomorphic functions on $V$ that satisfies \eqref{eq:formal_symbol} then we will also say that $\sum_{k \geq 0} h^k a_k$ is a formal symbol on $V$.
\end{remark}

\begin{definition}\label{def:realisation_Gevrey}
Let $m \in \R$ and $\underline{a} = \sum_{k \geq 0} h^k a_k \in FS^{1,m}\p{T^* U}$. Let $\Omega$ be an open subset of $U$. A realization of $\underline{a}$ on $\Omega$ is an element $a$ of $S^{1,m}\p{T^*\Omega}$ such that, for every open relatively compact subset $\Omega'$ of $\Omega$, there is $\epsilon > 0$ such that $a$ and the $a_k$'s have holomorphic extensions to $\p{T^* \Omega'}_\epsilon$ such that, for every $C_0 > 0$ large enough, there is a constant $C > 0$, such that, for every $(x,\xi) \in \p{T^* \Omega'}_\epsilon$, we have
\begin{equation}\label{eq:def_realisation}
\begin{split}
\va{a(x,\xi) - \sum_{0  \leq k \leq \frac{\langle\Re \xi\rangle}{C_0h} } h^k a_k(x,\xi)} \leq C \exp\p{ - \frac{\jap{\Re \xi}}{C h}}.
\end{split}
\end{equation}
\end{definition}

We will prove in \S \ref{sec:HormanderGrauert} that any formal analytic symbol admits a realization, namely:

\begin{lemma}\label{lemma:existence_realisation}
Let $m \in \R$. Let $\underline{a} \in FS^{1,m}\p{T^* U}$ and let $\Omega$ be a relatively compact open subset of $U$. Then $\underline{a}$ admits a realization $a \in S^{1,m}\p{T^* \Omega}$ on $\Omega$.
\end{lemma}

\begin{remark}\label{remark:formal-Gs-symbols}
Let us explain why we do not introduce formal $\G^s$ symbols. The natural notion of such a symbol would be a formal sum
\[
\sum h^k a_k,
\]
with an estimate of the form
\[
\left| \partial_x^\alpha \partial_\xi^\beta a_k\right| \leq C R^{|\alpha|+|\beta|} (\alpha+\beta+2k)!^s \langle\xi\rangle^{m-|\beta|}.
\]
To build an actual symbol from a formal symbol, the only procedure we have is the so-called \emph{sommation au plus petit terme}. The problem is that the result is a symbol that is well-defined modulo an error in the class
\[
\O_{\G^s_{x,\xi}}\p{\exp\p{- \p{\frac{\langle\xi\rangle}{Ch}}^{\frac{1}{2s-1}} }}.
\]
While this has the right \emph{regularity}, it does not decay fast enough to produce a $\G^s$ smoothing remainder. In the class of $\Gs_x \G^1_\xi$ symbols, we could consider formal symbols, and this is what is done in \cite{Boutet-Kree-67,zanghiratiPseudodifferentialOperatorsInfinite1985,Rodino-93-book}.
\end{remark}

We will also need Gevrey symbols defined on manifolds. There is a natural definition that is valid for any $s \geq 1$: we ask for Definition \ref{def:gevrey_symbol} to be satisfied in any charts. This definition is not empty. Indeed, it follows from Fa\`a di Bruno's formula, that the class of symbols presented in Definition \ref{def:gevrey_symbol} is stable by $\G^s$ change of coordinates.

\begin{definition}\label{def:symbole_manifold}
Let $M$ be a $\G^s$ manifold and $m \in \R$. We say that $a : T^* M \to \C$ belongs to the class of symbol $S^{s,m}\p{T^* M}$ if for every $\G^s$ chart $\p{U,\kappa}$ for $M$ the function
\begin{equation*}
\begin{split}
a^{\kappa}(x,\xi) \coloneqq a\p{\kappa^{-1}(x), {}^{T} \mathrm{d}_{\kappa^{-1}(x)} \kappa \cdot \xi}
\end{split}
\end{equation*}
belongs to $S^{s,m}\p{T^*\p{\kappa(U)}}$. Similarly, we say that $\underline{a} = \sum_{k \geq 0} h^k a_k$ is a formal analytic symbol in $FS^{1,m}\p{T^*M}$ if, for every analytic chart $\p{U,\kappa}$ for $M$, the formal symbol $\sum_{k \geq 0} h^k a_k^{\kappa}$ belongs to $FS^{1,m}\p{T^*\p{\kappa(U)}}$.
\end{definition}

\begin{remark}
As in the Euclidean case, we extend these definitions in the natural way to symbols defined on products of the form $\p{T^* M}^{N_1} \times M^{N_2}$. Moreover, if $M$ is compact and $s > 1$, we see by taking a partition of unity that if $a$ is in $S^{s,m}\p{T^* M}$ then it admits an almost analytic extension to $\p{T^* M}_{\epsilon}$ (for some $\epsilon > 0$) in the sense of Remark \ref{remark:aae_for_symbol}. If $s=1$, we do not need to use a partition of unity (by the principle of analytic continuation) to prove that $a$ admits a holomorphic extension to some $\p{T^* M}_\epsilon$ with growth at most polynomial. The following characterization of analytic symbol follows. 
\end{remark}

\begin{lemma}\label{lemma:symb_analytique_manifold}
Let $(M,g)$ be a compact real-analytic Riemannian manifold and $m \in \R$. Then $a : T^* M \to \C$ belongs to $S^{1,m}\p{T^* M}$ if and only if there are $C,\epsilon > 0$ such that $a$ admits a holomorphic extension to $\p{T^* M}_\epsilon$ that satisfies for every $\alpha \in \p{T^*M}_\epsilon$
\begin{equation*}
\begin{split}
\va{a(\alpha)} \leq C \jap{\va{\alpha}}^m.
\end{split}
\end{equation*}
Similarly, $\underline{a} = \sum_{k \geq 0} h^k a_k$ belongs to $FS^{1,m}\p{T^* M}$ if and only if there are constants $C,R,\epsilon > 0$ such that the $a_k$'s admit holomorphic extensions to $\p{T^* M}_\epsilon$ that satisfy for every $k \in \N$ and $\alpha \in \p{T^* M}_\epsilon$:
\begin{equation*}
\begin{split}
\va{a_k(\alpha)} \leq C R^k k! \jap{\va{\alpha}}^{m-k}.
\end{split}
\end{equation*}
\end{lemma}

\begin{remark}
As in Definition \ref{def:symbole_manifold}, the notion of realization for symbols on manifolds is deduced immediately from the notion on Euclidean spaces by taking charts. In the case of analytic symbols on a compact manifold, we may use Lemma \ref{lemma:symb_analytique_manifold} to reformulate this notion: $a \in S^{1,m}\p{T^*M}$ is a realization for $\sum_{k \geq 0}h^k a_k$ if and only if for every $C_0$ large enough, there are constants $C,\epsilon > 0$ such that for every $\alpha \in \p{T^* M}_\epsilon$ we have
\begin{equation*}
\begin{split}
\va{a(\alpha) - \sum_{0 \leq k \leq \frac{\jap{\va{\alpha}}}{C_0 h}} h^k a_k(\alpha)} \leq C \exp\p{- \frac{\jap{\va{\alpha}}}{Ch}}.
\end{split}
\end{equation*}
As stated below, realizations of formal analytic symbols on manifold do exist. We will indeed use H\"ormander's solution to the $\bar{\partial}$-equation to prove the following lemma (the proof is delayed to \S \ref{sec:HormanderGrauert}).
\end{remark}

\begin{lemma}\label{lemma:existence_realisation_mfld}
Let $M$ be a compact $\G^1$ manifold and $m \in \R$. Let $\underline{a} = \sum_{k \geq 0} h^k a_k$ be a formal analytic symbol in $FS^{1,m}\p{T^* M}$. Then $\underline{a}$ admits a realization on $M$.
\end{lemma}

\subsubsection{H\"ormander's problem on Grauert tube and analytic approximations}\label{sec:HormanderGrauert}

In order to define the kernel of the FBI transform $T$ and to find realizations of formal analytic symbols, we need a tool to construct globally defined real analytic functions. This will be performed using the version of H\"ormander's estimates \cite{Hormander-65-dbar,Hormander-66-complex-analysis} exposed in \cite{Demailly-book-website}. Notice that it is common to apply H\"ormander's solution to the $\bar{\partial}$ equation to solve such problems, see for instance the appendix of \cite{Sjostrand-96-I-lagrangian} or \cite[Proposition 1.2]{Sjostrand-82-singularite-analytique-microlocale}.  We will prove the following lemma.

\begin{lemma}\label{lemma:approximation_analytique}
Let $\p{M,g}$ be a compact real analytic Riemannian manifold. Let $\epsilon > 0$ be small enough. Let $f$ be a $(0,1)$-form on $\p{T^* M}_{\epsilon}$ (that may depend on the implicit parameter $h$) with $\mathcal{C}^\infty$ coefficients and such that $\bar{\partial} f = 0$. Assume that there is a constant $C > 0$ such that for every $\alpha \in \p{T^* M}_{\epsilon}$ we have
\begin{equation}\label{eq:petit_dbar}
\begin{split}
\va{f(\alpha)} \leq C \exp\p{- \frac{\jap{\va{\alpha}}}{C h}}.
\end{split}
\end{equation}
Then, there are $\epsilon_1 \in \left]0,\epsilon\right[$ and a constant $C' > 0$ such that, for $h$ small enough, there is a $\mathcal{C}^\infty$ function $r$ on $\p{T^* M}_{\epsilon_1}$ such that $\bar{\partial} r = f$ and
\begin{equation*}
\begin{split}
r(\alpha) = \O\p{\exp\p{- \frac{\jap{\va{\alpha}}}{C' h}}}
\end{split}
\end{equation*}
on $(T^* M)_{\epsilon_1}$. Moreover, the same result holds with the natural modifications if we replace $\p{T^* M}_\epsilon$ by a manifold of the form $\p{T^* M}_\epsilon^{N_1} \times \p{M}_\epsilon^{N_2}$.
\end{lemma}

\begin{remark}\label{remark:approximation_analytique}
Most of the time, we will apply Lemma \ref{lemma:approximation_analytique} with $f = -\bar{\partial} F$ where $F$ is a $\mathcal{C}^\infty$ function on $\p{T^* M}_\epsilon$. Then, if we define $\widetilde{F} = F + r$ on $\p{T^* M}_{\epsilon_1}$, we see that $\widetilde{F}$ is holomorphic and satisfies
\begin{equation*}
\begin{split}
\widetilde{F}\p{\alpha} = F\p{\alpha} + \O\p{\exp\p{- \frac{\jap{\va{\alpha}}}{C' h}}}
\end{split}
\end{equation*}
on $\p{T^* M}_{\epsilon_1}$. This is the method that we will use to construct a FBI transform with analytic kernel for instance (see Lemma \ref{lemma:existence-transform}).
\end{remark}

\begin{remark}\label{rmq:Grauet_est_Kahler}
In order to make sense of assumption \eqref{eq:petit_dbar} in Lemma \ref{lemma:approximation_analytique}, we need to specify which metric we use to measure the $(0,1)$-form $f$. To do so, recall from \S \ref{sec:Grauert} that, associated with the Kohn--Nirenberg metric $g_{KN}$, there is a non-negative plurisubharmonic function $\rho_{KN}$ on the complexification of $T^* M$. The $(1,1)$-form $\omega_{KN} = i \partial \bar{\partial} \rho$ defines a K\"ahler metric on $\p{T^* M}_\epsilon$ that coincides with $g_{KN}$ on $T^* M$. Then $\omega_{KN}$ induces a hermitian metric on $T\p{T^* M}_\epsilon$, hence on $T^* \p{T^* M}_{\epsilon}$ and following on $T^*\p{T^* M}_\epsilon \otimes_\R \C$. This is the metric that appears in \eqref{eq:petit_dbar}. 
\end{remark}

We will also need the following version of Lemma \ref{lemma:approximation_analytique} on $\R^n$, whose proof is slightly easier (one may rely directly on the original version of H\"ormander's argument \cite{Hormander-65-dbar}) and will consequently be omitted.

\begin{lemma}\label{lemma:approximation_analytique_Rn}
Let $U$ be an open subset of $\R^n$. Let $\Omega$ be a relatively compact open subset of $U$. Let $\epsilon > 0$ be small enough. Let $f$ be a $(0,1)$-form on $\p{T^* U}_{\epsilon}$ (that may depend on the implicit parameter $h$) with $\mathcal{C}^\infty$ coefficients and such that $\bar{\partial} f = 0$. Assume that there is a constant $C > 0$ such that for every $\alpha \in \p{T^* U}_{\epsilon}$ we have
\begin{equation*}
\begin{split}
\va{f(\alpha)} \leq C \exp\p{- \frac{\jap{\va{\alpha}}}{C h}}.
\end{split}
\end{equation*}
Then, there are $\epsilon_1 \in \left]0,\epsilon\right[$ and a constant $C' > 0$ such that, for $h$ small enough, there is a $\mathcal{C}^\infty$ function $r$ on $\p{T^* \Omega}_{\epsilon_1}$ such that $\bar{\partial} r = f$ and
\begin{equation*}
\begin{split}
r(\alpha) = \O\p{\exp\p{- \frac{\jap{\va{\alpha}}}{C' h}}}
\end{split}
\end{equation*}
on $(T^* \Omega)_{\epsilon_1}$. Moreover, the same result holds with the natural modifications if we replace $\p{T^* U}_\epsilon$ by a manifold of the form $\p{T^* U}_\epsilon^{N_1} \times \p{U}_\epsilon^{N_2}$.
\end{lemma}

Let us mention that there are no particular difficulties to deal with a manifold of the form $\p{T^* M}_\epsilon^{N_1} \times \p{M}_\epsilon^{N_2}$ rather than with $\p{T^* M}_\epsilon$ in Lemma \ref{lemma:approximation_analytique} and thus we will only write the proof in the latter case. Our arguments rely on Theorem \ref{thm:resoudre_lequation} below, which is \cite[Theorem VIII.6.5]{Demailly-book-website} applied to the holomorphic line bundle $ \p{X \times \C} \otimes \Lambda^n T X$ where $X \times \C$ denotes the trivial line bundle on $X$ (see the remark below Theorem VIII.6.5 in \cite{Demailly-book-website}). Notice that we apply here a result which is true in a very general context (namely weakly pseudo-convex K\"ahler manifold) to a much more favorable case (Grauert tubes are strictly pseudo-convex). In particular, one could probably retrieve ``by hand'' Lemma \ref{lemma:approximation_analytique} by adapting H\"ormander's argument \cite{Hormander-65-dbar,Hormander-66-complex-analysis}. However, we thought that it was easier to rely on an available powerful technology, since it avoids to rediscuss boundary and non-compactness issues. Moreover, it seemed to us that a more elementary proof of Lemma \ref{lemma:approximation_analytique} would have been mostly a repetition of some technical estimates that are very well detailed in \cite{Demailly-book-website}.

\begin{theorem}[\cite{Demailly-book-website}]\label{thm:resoudre_lequation}
Let $(X,\omega)$ be a weakly pseudo-convex K\"ahler manifold of complex dimension $d$. Let $\varphi$ be a $\mathcal{C}^\infty$ function from $X$ to $\R$. For every $x \in X$ denotes by $\lambda_1(x) \leq \dots \leq \lambda_d(x)$ the eigenvalues of $\textup{Ricci}(\omega) + i \partial \bar{\partial} \varphi$ and assume that $\lambda_1(x) > 0$. Then, if $f$ is a form of type $(0,1)$ on $X$, with $\mathcal{C}^\infty$ (resp. $L^2_{\textup{loc}}$) coefficients, that satisfies $\bar{\partial}f = 0$ and
\begin{equation}\label{eq:estimee_L2_forme}
\begin{split}
\int_X \frac{1}{\lambda_1} \va{f}^2 e^{- \varphi} \mathrm{d Vol} < + \infty.
\end{split}
\end{equation}
Then, there is a $\mathcal{C}^\infty$ (resp. $L^2_{\textup{loc}}$) function $r$ on $X$ such that $\bar{\partial}r = f$ and 
\begin{equation}\label{eq:estimee_L2_fonction}
\begin{split}
\int_X \va{r}^2 e^{- \varphi} \mathrm{d Vol} \leq \int_X \frac{1}{\lambda_1} \va{f}^2 e^{- \varphi} \mathrm{d Vol}.
\end{split}
\end{equation}
\end{theorem}

\begin{remark}
We refer to \cite{Demailly-book-website} for precise definitions and extensive discussions of the complex analytic notions that appear in the statement of Theorem \ref{thm:resoudre_lequation} (see also, for instance, \cite{demailly_1982}). Let us just mention that $\mathrm{d Vol}$ denotes the Lebesgue measure associated with the Hermitian form associated with $\omega$ (this measure is identified with the volume form $\omega^d/d!$) and that $\textup{Ricci}(\omega)$ denotes the Ricci curvature of $\omega$.
\end{remark}

We want to apply Theorem \ref{thm:resoudre_lequation} with $X = \p{T^* M}_\epsilon$ and $\omega = \omega_{KN}$ (as defined in Remark \ref{rmq:Grauet_est_Kahler}). To do so, we will start by checking that:

\begin{lemma}\label{lemma:conical-is-pseudo-convex}
$\p{ \p{T^* M}_\epsilon, \omega_{KN}}$ is a weakly pseudo-convex K\"ahler manifold.
\end{lemma}

\begin{proof}
We already saw in \S \ref{sec:Grauert} that $\p{ \p{T^* M}_\epsilon, \omega_{KN}}$ is K\"ahler. Recall that a complex manifold is said to be pseudoconvex if it admits a plurisubharmonic exhaustion function $\psi$. In the case of $\p{T^*M}_\epsilon$, we can take for instance
\begin{equation}\label{eq:exhaustons}
\begin{split}
\psi : \alpha \mapsto -\log\p{\epsilon^2 - \rho(\alpha)} + \Re g_{\alpha_x}(\alpha_\xi,\alpha_\xi).
\end{split}
\end{equation}
The first term in \eqref{eq:exhaustons} is plurisubharmonic because it is a convex increasing function of a plurisubharmonic function. The second term is plurisubharmonic because it is pluriharmonic, as the real part of a holomorphic function ($g$ denotes here the holomorphic extension of the Riemannian metric on $M$). To see that $\psi$ is indeed an exhaustion function, notice that both terms in \eqref{eq:exhaustons} are bounded from below. Thus, the first term ensures that the sublevel sets $\set{\psi < c}$ remain away from the boundary of $\p{T^* M}_\epsilon$ while the second term ensures that they remain bounded in the fibers (recall that in $\p{T^* M}_\epsilon$ the quantity $\Re g_{\alpha_x}(\alpha_\xi,\alpha_\xi)$ is equivalent to the size of $\alpha_\xi$).
\end{proof}

We need then to construct the function $\varphi$ that appears in Theorem \ref{thm:resoudre_lequation}. Before doing so, recall that if $u$ is a real $(1,1)$-form on a complex finite-dimensional vector space $V$ there is a unique hermitian form $\mathfrak{h}_u$ on $V$ that satisfies 
\begin{equation*}
\begin{split}
\forall \xi,\eta \in V : u(\xi,\eta) = - 2 \Im \mathfrak{h}_u(\xi,\eta).
\end{split}
\end{equation*}
Moreover,if we choose coordinates $(z_1,\dots,z_d)$ on $V$, the map $u \mapsto \mathfrak{h}_u$ writes
\begin{equation*}
\begin{split}
\frac{i}{2} \sum_{1 \leq j,k \leq n} h_{jk} \mathrm{d}z_j \wedge \mathrm{d}\bar{z}_k \mapsto \sum_{1 \leq j,k \leq n} h_{jk} \mathrm{d}z_j \otimes \mathrm{d}\bar{z}_k.
\end{split}
\end{equation*}
Now, if $V$ is the tangent space of $\p{T^* M}_{\epsilon}$ at some point $\alpha$, we define the eigenvalues of $u$ as the diagonal elements of the matrix of $\mathfrak{h}_u$ in an orthonormal basis for $\mathfrak{h}_{\omega_{KN}}$ that diagonalizes $\mathfrak{h}_u$ (these are the eigenvalues that appear in the statement of Theorem \ref{thm:resoudre_lequation}). We can now construct the weight $\varphi$ that we are going to use.

\begin{lemma}\label{lemma:plurisousharmonique}
Assume that $\epsilon > 0$ is small enough. Then there is a strictly plurisubharmonic function $\varphi$ on $\p{T^* M}_\epsilon$ with the following properties:
\begin{enumerate}[label=(\roman*)]
\item there is $C > 0$ such that for all $\alpha \in \p{T^*M}_\epsilon$ we have
\begin{equation*}
\begin{split}
-C \jap{\va{\alpha}} \leq \varphi(\alpha) \leq - \frac{\jap{\va{\alpha}}}{C};
\end{split}
\end{equation*}
\item there is $C > 0$ such that for all $\alpha \in \p{T^* M}_\epsilon$ we have
\begin{equation*}
\begin{split}
\mathfrak{h}_{i \partial \bar{\partial} \varphi} \geq \frac{\jap{\va{\alpha}}}{C} \mathfrak{h}_{\omega_{KN}}.
\end{split}
\end{equation*}
\end{enumerate}
\end{lemma}

\begin{proof}
Choose $A > 0$. We will see that if $A$ is large enough and $\epsilon$ is small enough then the function
\begin{equation*}
\begin{split}
\varphi = -\jap{\va{\alpha}}\p{1 - A \rho}
\end{split}
\end{equation*}
satisfies the desired properties. To do so, we compute
\begin{equation}\label{eq:Levy}
\begin{split}
i \partial \bar{\partial} \varphi = - i \partial \bar{\partial} \p{\jap{\va{\alpha}}} + A\p{\jap{\va{\alpha}} \omega_{KN} + u},
\end{split}
\end{equation}
where
\begin{equation*}
\begin{split}
u = i \partial \rho \wedge \bar{\partial}\p{\jap{\va{\alpha}}} + i \partial\p{\jap{\va{\alpha}}} \wedge \bar{\partial} \rho + i \rho \partial \bar{\partial} \jap{\va{\alpha}}.
\end{split}
\end{equation*}
In terms of Hermitian forms, \eqref{eq:Levy} rewrites
\begin{equation*}
\begin{split}
\mathfrak{h}_{i \partial \bar{\partial} \varphi} = \mathfrak{h}_{-i \partial \bar{\partial}\p{\jap{\va{\alpha}}}} + A\p{\jap{\va{\alpha}} \mathfrak{h}_{\omega_{KN}} + \mathfrak{h}_u}.
\end{split}
\end{equation*}
Working in local coordinates, we see that the regularity of $g_{KN}$ (it has bounded derivatives with respect to itself, i.e. it satisfies symbolic estimates) implies that $\mathfrak{h}_{\omega_{KN}}$ is uniformly equivalent to the Kohn--Nirenberg metric on $T^* \C^n$. For the same reason, $\rho$ is a symbol of order $0$. Using that the differential of $\rho$ vanishes on the real and that $\jap{\va{\alpha}}$ is a symbol of order $1$, we see that
\begin{equation*}
\begin{split}
\mathfrak{h}_u = \O\p{\epsilon \jap{\va{\alpha}}} \mathfrak{h}_{\omega_{KN}}.
\end{split}
\end{equation*}
Hence, if $\epsilon$ is small enough, we have
\begin{equation*}
\begin{split}
\jap{\va{\alpha}} \mathfrak{h}_{\omega_{KN}} + \mathfrak{h}_u \geq \frac{\jap{\va{\alpha}}}{2} \mathfrak{h}_{\omega_{KN}}.
\end{split}
\end{equation*}
From the symbolic behavior of $\jap{\va{\alpha}}$, we deduce that
\begin{equation*}
\begin{split}
\mathfrak{h}_{-i \partial \bar{\partial}\p{\jap{\va{\alpha}}}} = \O\p{\jap{\va{\alpha}}} \mathfrak{h}_{\omega_{KN}},
\end{split}
\end{equation*}
and thus (ii) holds if $A$ is large enough (we may even take $C$ arbitrarily small). Finally, if we impose that $\epsilon < \frac{1}{2A}$, it is immediate that (i) holds.
\end{proof}

We are now ready to prove Lemma \ref{lemma:approximation_analytique}.

\begin{proof}[Proof of Lemma \ref{lemma:approximation_analytique}]
We want to apply Theorem \ref{thm:resoudre_lequation} with $\varphi$ replaced by $\frac{\nu \varphi}{h}$, where $\nu$ is a small positive constant and $\varphi$ is from Lemma \ref{lemma:plurisousharmonique}. To do so, we need to check that the eigenvalues $\lambda_1 \leq \dots \leq \lambda_{2n}$ of
\begin{equation*}
\begin{split}
\textup{Ricci}(\omega_{KN}) + \frac{i \nu \partial \bar{\partial} \varphi}{h}
\end{split}
\end{equation*}
are positive. Using again the regularity of the Kohn--Nirenberg metric, we see that
\begin{equation*}
\begin{split}
\mathfrak{h}_{\textup{Ricci}(\omega_{KN})} = \mathcal{O}\p{ \mathfrak{h}_{\omega_{KN}}}.
\end{split}
\end{equation*}
Hence, with (ii) in Lemma \ref{lemma:plurisousharmonique}, we have for some constant $C > 0$
\begin{equation*}
\begin{split}
\mathfrak{h}_{\textup{Ricci}(\omega_{KN}) + \frac{i \nu \partial \bar{\partial} \varphi}{h}} \geq \p{\frac{\nu \jap{\va{\alpha}}}{Ch} - C} \mathfrak{h}_{\omega_{KN}}.
\end{split}
\end{equation*}
Hence, if $h$ is small enough (depending on $\nu$), we have
\begin{equation}\label{eq:minoration_vp}
\begin{split}
\lambda_1 \geq \frac{\nu \jap{\va{\alpha}}}{2 C h} > 0,
\end{split}
\end{equation}
and we may apply Theorem \ref{thm:resoudre_lequation} to the $(0,1)$-form $f$.

Indeed, thanks to \eqref{eq:petit_dbar}, (i) in Lemma \ref{lemma:plurisousharmonique} and \eqref{eq:minoration_vp}, we see that for $\nu$ and $h$ small enough we have
\begin{equation}\label{eq:estimeeL2dbar}
\begin{split}
\int_{\p{T^* M}_\epsilon} \frac{1}{\lambda_1} \va{f}^2 e^{ - \frac{\nu \varphi}{h}} \mathrm{d Vol} \leq C < + \infty,
\end{split}
\end{equation}
where $C$ does not depend on $h$. Hence, according to Theorem \ref{thm:resoudre_lequation}, there is a $\mathcal{C}^\infty$ function $r$ on $\p{T^* M}_\epsilon$ such that $\bar{\partial} r = f$ and
\begin{equation}\label{eq:estimeeL2}
\begin{split}
\int_{\p{T^* M}_\epsilon} \va{r}^2 e^{- \frac{\nu \varphi}{h}} \mathrm{d Vol} \leq C.
\end{split}
\end{equation}
It remains to show that $r$ is an $\O(\exp( - \jap{\va{\alpha}}/C' h ))$ in $(T^* M)_{\epsilon_1}$ for some $\epsilon_1 \in \left]0,\epsilon\right[$. However, this fact follows easily from \eqref{eq:estimeeL2dbar} and \eqref{eq:estimeeL2} (recall that $\bar{\partial} r = f$) by averaging Bochner--Martinelli's Formula.
\end{proof}

\begin{remark}\label{remark:regularite_solution_d_bar}
It will be useful to know the dependence of $r$ on $f$ in Lemma \ref{lemma:approximation_analytique}. We will see that the solution to the $\bar{\partial}$ equation in Theorem \ref{thm:resoudre_lequation} may be obtained by the application of a continuous linear operator. In particular, if $f$ depends smoothly on a parameter in the space of forms that satisfy \eqref{eq:petit_dbar} then $r$ has the same dependence on this parameter (in a space of very rapidly decaying functions). We will need this fact to check that the measurability assumption in Fubini's Theorem is satisfied in the proof of Theorem \ref{theorem:functional_calculus} below.

To see that in Theorem \ref{thm:resoudre_lequation}, the function $r$ may be deduced from $f$ by the application of a continuous linear operator, denote by $\mathcal{H}_{2}$ the Hilbert space of forms $f$ of type $\p{0,1}$ with $L^2_{\textup{loc}}$ coefficients that satisfy \eqref{eq:estimee_L2_forme} and $\bar{\partial}f = 0$ (in the sense of distributions). Define also the space $\mathcal{H}_1$ of $L^2_{\textup{loc}}$ functions $r$ on $X$ such that $\bar{\partial} r \in \mathcal{H}_2$ and
\begin{equation*}
\begin{split}
\int_{X} \va{r}^2 e^{- \varphi} \mathrm{dVol} < + \infty.
\end{split}
\end{equation*}
The space $\mathcal{H}_1$ is made a Hilbert space when endowed with the norm
\begin{equation*}
\begin{split}
\n{r}_{\mathcal{H}_2} \coloneqq \sqrt{\int_{X} \va{r}^2 e^{- \varphi} \mathrm{dVol} + \n{\bar{\partial} r}_{\mathcal{H}_2}^2}.
\end{split}
\end{equation*}
Then the operator $\bar{\partial}$ is clearly a bounded operator between $\mathcal{H}_1$ and $\mathcal{H}_2$. Denote by $p$ the orthogonal projection in $\mathcal{H}_2$ on the closed subspace $\p{\ker \bar{\partial}}^{\perp}$. Now, if $f \in \mathcal{H}_2$, we know by Theorem \ref{thm:resoudre_lequation} that there is an $r \in \mathcal{H}_1$ such that $\bar{\partial} r = f$ with the estimates \eqref{eq:estimee_L2_fonction}. Then $p\p{r}$ is the unique element of $\p{\ker \bar{\partial}}^{\perp}$ such that $\bar{\partial} \p{p\p{r}} = f$, and consequently we may define a linear operator $L$ from $\mathcal{H}_2$ to $\mathcal{H}_1$ by $L f = p\p{r}$. For every $f \in \mathcal{H}_2$, we have $\bar{\partial} Lf = r$ by definition. Moreover we have
\begin{equation*}
\begin{split}
\int_X \va{Lf}^2 e^{- \varphi} \mathrm{dVol} & = \n{L f}_{\mathcal{H}_1}^2 - \n{f}_{\mathcal{H}_2}^2 \\
      & \leq \n{r}_{\mathcal{H}_1}^2 - \n{f}_{\mathcal{H}_2}^2 = \int_X \va{r}^2 e^{- \varphi} \mathrm{dVol} \\
      & \leq \n{f}_{\mathcal{H}_2}^2.
\end{split}
\end{equation*}
Hence, $Lf$ also satisfies \eqref{eq:estimee_L2_fonction}. Finally, when $r$ has $\mathcal{C}^\infty$ coefficients then so does $Lf$ by ellipticity of the $\bar{\partial}$ operator.
\end{remark}

We are now ready to prove Lemma \ref{lemma:existence_realisation}. The proof of Lemma \ref{lemma:existence_realisation_mfld} is very similar, applying Lemma \ref{lemma:approximation_analytique} instead of Lemma \ref{lemma:approximation_analytique_Rn}, and we will omit it.

\begin{proof}[Proof of Lemma \ref{lemma:existence_realisation}]
Let us choose a relatively compact open subset $\Omega'$ of $U$ such that $\overline{\Omega} \subseteq \Omega'$. Let also $\chi : \R \to \left[0,1\right]$ be a $\mathcal{C}^\infty$ function such that $\chi(t) = 0$ if $t \leq \frac{1}{2}$ and $\chi(t) = 1$ if $t \geq 1$. Then define for $k \geq 1$
\begin{equation*}
\begin{split}
\chi_{k}(\xi) = \chi\p{\frac{\jap{\va{\xi}}}{C_1 k h}}
\end{split}
\end{equation*}
for $\xi \in \C^n$. Here $C_1$ is a large constant to be fixed later. We also define $\chi_0$ to be identically equal to $1$. Then for $\p{x,\xi} \in \p{T^* \Omega'}_\epsilon$ we define
\begin{equation*}
\begin{split}
b(x,\xi)= \sum_{k \geq 0} \chi_k(\xi) h^k a_k(x,\xi).
\end{split}
\end{equation*}
Notice that if $(x,\xi) \in \p{T^* \Omega'}$ we have for any $C_0 \geq C_1$
\begin{equation}\label{eq:premier_essai}
\begin{split}
\va{b(x,\xi) - \sum_{0 \leq k \leq \frac{\jap{\va{\xi}}}{C_0 h}}h^k a_k(x,\xi)} & \leq \sum_{k > \frac{\jap{\va{\xi}}}{C_0 h}} \va{\chi_k(\xi)} h^k \va{a_k(x,\xi)}.
\end{split}
\end{equation}
By assumption there are constants $C,R > 0$ such that 
\begin{equation*}
\begin{split}
h^k \va{a_k(x,\xi)} \leq C R^{k} k^{k} h^k \jap{\va{\xi}}^{m - k}.
\end{split}
\end{equation*}
Notice then that if $\chi_k(\xi) \neq 0$ then $kh\jap{\va{\xi}}^{-1}< 2/ C_1$ and consequently
\begin{equation*}
\begin{split}
h^k \va{a_k(x,\xi)} \leq C \jap{\va{\xi}}^{m} \exp\p{ k \ln\p{\frac{2 R}{C_1}}}.
\end{split}
\end{equation*}
By taking $C_1$ large enough, we ensure that $\ln\p{2R / C_1} < 0$ and we find then from \eqref{eq:premier_essai} that
\begin{equation*}
\begin{split}
\va{b(x,\xi) - \sum_{0 \leq k \leq \frac{\jap{\va{\xi}}}{C_0 h}}h^k a_k(x,\xi)} & \leq C \jap{\va{\xi}}^m \sum_{k > \frac{\jap{\va{\xi}}}{C_0 h}} \exp\p{ k \ln\p{\frac{2 R}{C_1}}} \\
    & \leq C \jap{\va{\xi}}^m \frac{\exp\p{ \ln\p{\frac{2 R}{C_1}} \frac{\jap{\va{\xi}}}{C_0 h}} }{1 - \exp\p{ \ln\p{\frac{2 R}{C_1}}}}.
\end{split}
\end{equation*}
Hence, $b$ satisfies the bound \eqref{eq:def_realisation}, but $b$ is not analytic in $\xi$. We will correct this using Hörmander's solution to the $\bar{\partial}$ equation. Notice indeed that
\begin{equation*}
\begin{split}
\bar{\partial} b(x,\xi) = \sum_{k \geq \frac{\jap{\va{\xi}}}{C_1 h}} h^k a_k(x,\xi) \bar{\partial} \chi_k(\xi).
\end{split}
\end{equation*}
The same computation than above shows then that, provided that $C_1$ is large enough, $\bar{\partial} b(x,\xi)$ is a $\O(\exp(- \jap{\va{\xi}}/Ch ))$. Thus, we end the proof by applying Lemma \ref{lemma:approximation_analytique_Rn}.
\end{proof}

\section{Gevrey oscillatory integrals}
\label{sec:oscillatory-integrals}

Behind most results of this exposition, lies a careful analysis of oscillatory integrals with parameter, of the form
\begin{equation}\label{eq:oscillatory_general}
(x,h)\mapsto \int_{\Omega} e^{\frac{i}{h}\Phi(x,y)}a(x,y) \mathrm{d}y.
\end{equation}
We will either want to prove that they are small, in some functional space, or give a precise expansion as $h\to 0$ (and possibly as $x$ tends to $\infty$ in some sense). Such estimates will fall in the usual category of non-stationary and stationary phase estimates. We expect the reader to be accustomed with these techniques, at least in the case that $\Phi$ is a real valued function, and the integrands are $C^\infty$. However, since we will be working in Gevrey regularity, there are some particularities, and tricks that we gathered here. We hope that this can help in understanding the details of the proofs.

To study oscillatory integrals, it is usual to use many integration by parts. This can be a useful technique in the Gevrey setting. However, it imposes to count derivatives precisely, and introduce some formal norms. It is the approach of \cite{Boutet-Kree-67,Gramchev-87-stationary-phase-gevrey,lascarFBITransformsGevrey1997}, and many others. On the other hand, in the analytic case, it is often more convenient to use Cauchy estimates. Instead of having to find $\mathcal{C}^k$ estimates for real values of the parameter, only $L^\infty$ estimates are necessary, but for values of the parameter in a \emph{complex neighbourhood} of the reals. 

Using almost analytic extensions, one can mimic this tactic in the Gevrey case. It is the point of view that we have adopted as much as possible throughout the paper. Naturally, this should also be possible in the $\mathcal{C}^\infty$ case, and we leave this to the reader's imagination. Integrations by parts still appear, but they only serve one purpose: showing that some non-absolutely converging oscillatory integrals are well defined. 

\subsection{Non-stationary phase}\label{sec:non-stationary}

The assumptions on the functions $\Phi$, $a$, and the set $\Omega$ in \eqref{eq:oscillatory_general} will depend on the context. In the general case, $\Im \Phi$ will be allowed to be slightly negative. We will have to distinguish between the \emph{stationary} and \emph{non-stationary} regions for $\Phi$. We deal with the latter case in this section. Let us start with the easy case: the phase is non-stationary of the whole of $U\times \Omega$. 

\begin{prop}\label{prop:non-stationary}
Let $s \geq 1$, and $U\times V$ be an open subset of $\R^n \times \R^k$. Let $\Phi : U\times V \to \C$ be a $\G^s$ function and $\Omega$ be a relatively compact open subset of $V$. For $\lambda \geq 1$ and $a \in \G^s\p{U \times V}$ let
\begin{equation}\label{eq:integrale_non_stationnaire}
\begin{split}
I_\lambda(a) : x  \mapsto \int_{\Omega} e^{i \lambda \Phi(x,y)} a(x,y) \mathrm{d}y.
\end{split}
\end{equation}
We make the following assumptions:
\begin{enumerate}[label=(\roman*)]
	\item the imaginary part of $\Phi$ is non-negative on $U\times \Omega$;
	\item there exists $\epsilon>0$ such that $\Im \Phi(x,y) > \epsilon$ for $x\in U$ and $y\in \partial \Omega$;
	\item the differential $\mathrm{d}_y\Phi$ does not vanish on $U \times V$.
\end{enumerate} 
Let $K$ be a compact subset of $U$. Denote respectively by $K_1$ and $K_2$ compact neighbourhoods of $K$ in $U$ and of $\overline{\Omega}$ in $V$. Then, there exists $C>0$ such that, for every $R \geq 1$, there is a constant $C_R > 0$ such that for all $\lambda \geq 1$ and $a \in \G^s\p{U}$ we have
\begin{equation}\label{eq:avant_rescaling}
\begin{split}
\|I_\lambda(a)\|_{s,C R,K} \leq C_R\|a\|_{s,R, K_1 \times K_2 } \exp\left( - \left(\frac{\lambda}{C R}\right)^{\frac{1}{s}} \right).
\end{split}
\end{equation} 
\end{prop}

We will give a proof of Proposition \ref{prop:non-stationary} based on an application of the Bochner--Martinelli Trick \ref{lemma:Bochner-Martinelli-trick}. Thus, we will give a bound on an almost analytic extension of $I_\lambda(a)$. However, when $a$ and $\Phi$ are analytic in $x$ then $I_\lambda$ has a natural holomorphic extension that satisfies the same bound as the almost analytic extension of $I_\lambda(a)$. It will be useful when studying FBI transforms to have this bound isolated, and this is the point of the following proposition. 

\begin{prop}\label{prop:non-stationary-analytic}
Under the hypothesis of Proposition \ref{prop:non-stationary}, assume in addition that $\Phi$ is analytic and that $a$ is of the form $a(x,y) = b(x,y) u(y)$ with $b \in \G^1\p{U \times V}$ fixed and $u \in \G^s\p{V}$ (we write then $I_\lambda(a) = I_\lambda(u)$).

Let $K$ be a compact subset of $U$ and $K'$ be a compact neighbourhood of $\Omega$ in $V$. Then, there exists a complex neighbourhood $W$ of $K$ such that for every $u \in \G^s\p{V}$, the function $I_\lambda(u)$ has a holomorphic extension to $W$. Moreover, there is $C > 0$ such that, for every $R \geq 1$, there is $C_R > 0$ such that, if $\lambda \geq 1$ and $z \in W$ are such that $\va{\Im z} \leq  (CR)^{-1/s} \lambda^{1 - 1/s}$, then for every $u \in \G^s\p{V}$, we have
\begin{equation}\label{eq:decay_complex}
\va{I_\lambda(u)(z)} \leq C_R \n{u}_{s,R,K'}\exp\left( - \left(\frac{\lambda}{C R}\right)^{\frac{1}{s}} \right).
\end{equation}
\end{prop}

\begin{proof}[Proof of Proposition \ref{prop:non-stationary} and \ref{prop:non-stationary-analytic}]
We start with the proof of Proposition \ref{prop:non-stationary} and then explain how the same argument also gives Proposition \ref{prop:non-stationary-analytic}.

Let $\widetilde{\Phi}$ denote either the holomorphic extension of $\Phi$ (if $s=1$) or a $\G^s$ almost analytic extension of $\Phi$ (if $s > 1$). This almost analytic extension exists by Lemma \ref{lemma:aae_Gs_flat}, since when $s > 1$ we may assume that $\Phi$ is compactly supported (once $K_1$ and $K_2$ are fixed).

Then, we fix $R \geq 1$. In the following, $C > 0$ will denote a generic constant that may depend on $\Phi, K_1,K_2$ but not on $R$, while $C_R > 0$ will denote a generic constant that may depend on $\Phi_,K_1,K_2$ and $R$. The value of $C$ and $C_R$ may change from one line to another.

Let $a \in \G^s\p{U \times V}$ be such that $\n{a}_{s,R,K_1 \times K_2} < + \infty$. We denote by $\tilde{a}$ either the holomorphic extension of $a$ if $s = 1$ (which is defined for $x,y$ in $\C^n$ at distance at most $C^{-1} R^{-1}$ from $K_1 \times K_2$) or an almost analytic extension for $a$ given by Lemma \ref{lemma:aae_Gs_flat} if $s > 1$ (without loss of generality, we may assume in that case that $a$ is supported in $K_1 \times K_2$). We can then extend the definition of $I_\lambda(a)$ to a smooth function on a complex neighbourhood of $K$ in $\C^n$ just by replacing $\Phi$ and $a$ in \eqref{eq:integrale_non_stationnaire} respectively by $\widetilde{\Phi}$ and $\widetilde{a}$.

It will be useful to pick a cutoff $\chi \in C^\infty(\R^+)$, such that $\chi$ equals $1$ in $[0,1/2]$, and $0$ in $[1,+\infty[$. We want to perform a contour deformation in the definition of $I_\lambda(a)$ in order to make the imaginary part of the phase positive. Since near the boundary of $\Omega$, the imaginary part of the phase is positive, we will only have to consider deformations in the complex domain for $y \in \Omega$ that are away from $\partial \Omega$. For some small $\epsilon_1 > 0$ and all $t \in \R$, we define the contour 
\[
\Gamma_t : (x,y) \mapsto y + t  \p{1-\chi\p{\frac{d\p{y,\partial\Omega}}{\epsilon_1}}}\nabla_y \Im \widetilde{\Phi}(x,y).
\]
Here, the gradient is defined using the identification $\C^n \simeq \R^n$.

Since the differential of $\Phi$ does not vanish on $K_1 \times K_2$ and $\widetilde{\Phi}$ satisfy the Cauchy--Riemann equations on $\R^n$, we see that $\nabla_y \Im \widetilde{\Phi}$ does not vanish on $K_1 \times K_2$, and hence on a complex neighbourhood $W_1 \times W_2$ of $K_1 \times K_2$. Thus, applying Taylor's formula, we see that there is $t_0 > 0$ such that for all $t \in \left[0,t_0\right], x\in W_1$ and $y \in K_2$ we have
\begin{equation*}
\Im \widetilde{\Phi} \p{x,\Gamma_t(x,y)} \geq \frac{t}{C} \p{1-\chi\p{\frac{d\p{y,\partial\Omega}}{\epsilon_1}}} + \Im \widetilde{\Phi}(x,y).
\end{equation*}
For $x\in K_1$, the imaginary part of the phase is strictly positive when $y$ is near the boundary of $\Omega$ , so that, for $x \in W_1$, this remains true, and we may assume that $\Im \widetilde{\Phi}(x,y)>\epsilon/2$ for $d(y,\partial\Omega) < \epsilon_1/2$ and $x\in W_1$. Away from the boundary, we have $\Im \widetilde{\Phi}(x,y) \geq - C |\Im x|$, for some constant $C>0$, when $x\in W_1$ and $y \in K_2$ (this is because we assumed that the imaginary part of the phase is non-negative for reals $x$ and $y$). Summing up, we have for $t\in [0,t_0],x\in W_1$ and $y \in \Omega$,
\begin{equation}\label{eqgrandim}
\Im \widetilde{\Phi}(x, \Gamma_t(x,y)) \geq \frac{t}{C} - C \va{\Im x}.
\end{equation}

\begin{figure}[h]
\centering
\def\svgwidth{0.7\linewidth}
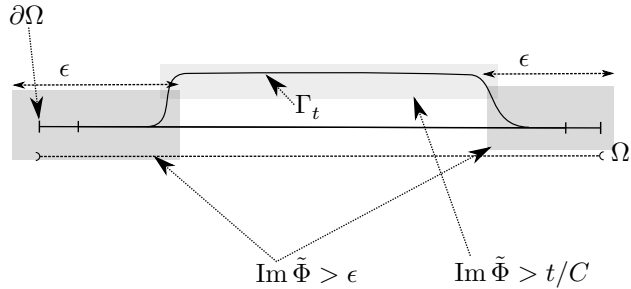
\caption{\label{fig:non-stationary-contour} Sketch of the contour deformation.}
\end{figure}

Now, choose $t_1 \in \left[0,t_0\right]$ and $x \in W_1$ and let $H$ be the map from $M_{t_1} = \{x\}\times \Omega \times \left[0,t_1\right]$ to $\C^n$ defined by $H(x,y,t) = \Gamma_t(x,y)$. If $s=1$, we assume in addition that $t_1 \leq (CR)^{-1}$ and $d(x,K) \leq (CR)^{-1}$ in order to ensure that $\tilde{a}(x,\Gamma_t(x,y))$ is well-defined for every $y \in \Omega$ and $t \in \left[0,t_1\right]$. We apply Stokes' formula to the $n$-form $H^*(e^{i \lambda \tilde{\Phi}(z)} \tilde{a}(z) \mathrm{d}z )$ to get
\begin{equation}\label{eqstokes}
\begin{split}
I_\lambda(a)(x) = \int_{\Omega} e^{i \lambda \widetilde{\Phi}\p{x,\Gamma_{t_1}(x,y)}} & \tilde{a}(x,\Gamma_{t_1}(x,y)) J \Gamma_{t_1} (x,y) \mathrm{d}y  \\ & \qquad \qquad \qquad + \int_{M_{t_1}}\hspace{-12pt} \mathrm{d} H^* \Big(e^{i \lambda \widetilde{\Phi}(x,z)} \tilde{a}(x,z) \mathrm{d}z\Big),
\end{split}
\end{equation}
where $J \Gamma_{t_1}$ denotes the Jacobian of $\Gamma_{t_1}$ as a function of $y$. We did not make any regularity assumption on $\Omega$, but this is not a problem when applying Stokes' Formula, since we only deform away from the boundary of $\Omega$. Notice that there is a uniform bound on the Jacobian $J \Gamma_{t_1}$ for $t_1 \in \left[0,t_0\right]$ and that the supremum of $\tilde{a}$ is controlled by $C_R \n{a}_{s,R,K_1 \times K_2}$ (provided that $d(x,K) \leq (CR)^{-1}$ in the case $s = 1$). Consequently, using \eqref{eqgrandim}, we see that
\begin{equation}\label{eq:bound-deformed-term}
\begin{split}
& \va{\int_{\Omega} e^{i \lambda \widetilde{\Phi}\p{x,\Gamma_{t_1}(x,y)}} \tilde{a}\p{x,\Gamma_{t_1}(x,y)} J \Gamma_{t_1} (x) \mathrm{d}y} \\ & \qquad \qquad \qquad \qquad \qquad \qquad \leq C_R \n{a}_{s,R,K_1 \times K_2} \exp\p{\lambda \p{C|\Im x|-  \frac{t_1}{C}}}.
\end{split}
\end{equation}
This only gets better when $t_1$ increases. However, this will not be the case of the second term if $s > 1$. To bound the second term in \eqref{eqstokes}, notice that
\begin{equation*}
\begin{split}
\mathrm{d}(e^{i \lambda \tilde{\Phi}(x,z)} \tilde{a}(x,z) \mathrm{d}z)& = \bar{\partial}_z\p{e^{i \lambda \tilde{\Phi}(x,z)} \tilde{a}(x,z)} \wedge \mathrm{d}z \\
     & = i \lambda e^{i \lambda \tilde{\Phi}(x,z)} \tilde{a}(x,z) \bar{\partial}_z \widetilde{\Phi}(x,z) \wedge \mathrm{d}z + e^{i \lambda \tilde{\Phi}(x,z)} \bar{\partial}_z \tilde{a}(x,z) \wedge \mathrm{d}z.
\end{split}
\end{equation*}
Hence, the second term in \eqref{eqstokes} is null if $s=1$, since $\tilde{a}$ is holomorphic. Thus, in the case $s=1$, if $x \in \C^n$ is at distance at most $(AR)^{-1}$ of $K$ (for some large $A > 0$ that does not depend on $R$), we find taking $t_1 = \min\p{t_0, (C R)^{-1}}$ that,
\begin{equation*}
\begin{split}
\va{I_\lambda(a)(x)} & \leq C_R \n{a}_{1,R,K_1 \times K_2} \exp\p{\lambda\p{\frac{C}{A R} - \frac{1}{CR}}} \\
    & \leq C_R \n{a}_{1,R,K_1 \times K_2} \exp\p{- \frac{\lambda}{C R}}.
\end{split}
\end{equation*}
On the second line, we assumed that $A$ was large enough and changed the value of $C$. Proposition \ref{prop:non-stationary} in the case $s=1$ follows, since $I_\lambda(a)$ is holomorphic by differentiation under the integral.

We turn now to the case $s > 1$. According to Remark \ref{remark:size_d_bar}, the integrand in the second term in \eqref{eqstokes} is controlled by
\begin{equation}\label{eq:borne_second_terme}
C_R \n{a}_{s,R,K_1\times K_2} \exp\left( C \lambda |\Im x|  - \frac{1}{C |R (t + |\Im x|)|^{\frac{1}{s-1}}}\right).
\end{equation}
Here, we recall that $t \in \left[0,t_1\right]$ for some $t_1 > 0$ that still need to be chosen. Here, we ignored the term $- t/C$ in \eqref{eqgrandim}, this term was needed only to bound the first term in \eqref{eqstokes}. If we take $t_1 = A_0^{-1} R^{-1/s} \lambda^{1/s-1}$ and assume that $\va{\Im x} \leq A_1^{-1} R^{-1/s} \lambda^{\frac{1}{s}-1}$ (for some large $A_0,A_1 > 0$), then the bound \eqref{eq:borne_second_terme} becomes
\begin{equation*}
\begin{split}
& C_R \n{a}_{s,R,K_1\times K_2} \exp\p{ \p{\frac{C}{A_1} - \frac{1}{C(A_0^{-1} + A_1^{-1})^{\frac{1}{s-1}}}} \frac{\lambda^{\frac{1}{s}}}{R^{\frac{1}{s}}}  } \\
   & \qquad \qquad \qquad \qquad \qquad \leq   C_R \n{a}_{s,R,K_1\times K_2} \exp\left( - \left( \frac{\lambda}{CR}\right)^{\frac{1}{s}} \right).
\end{split}
\end{equation*}
We assumed that $A_0$ and $A_1$ were large enough on the second line. Together with \eqref{eqstokes} and \eqref{eq:bound-deformed-term}, this proves that for $|\Im x| < A_1^{-1} \lambda^{1/s-1} R^{-1/s}$ (recall that $t_1 = A_0^{-1} \lambda^{1/s-1} R^{-1/s}$ and assume that $A_0 \gg A_1$),
\begin{equation}\label{eq:le_vrai_truc}
| I_\lambda(a)(x)| \leq C_R \n{a}_{s,R,K_1 \times K_2} \exp\left( - \left( \frac{\lambda}{CR}\right)^{\frac{1}{s}} \right).
\end{equation}
On the other hand, to estimate $\overline{\partial}_x I_\lambda(a)$, we have to give a bound for
\[
\left|\int_{\Omega} e^{i \lambda \widetilde{\Phi}(x,y)}( \overline{\partial}_x \tilde{a}(x,y) + i \lambda \tilde{a}(x,y) \overline{\partial}_x \widetilde{\Phi}(x,y)) \mathrm{d}y \right|.
\]
For this, we do not use a contour shift but a direct $L^1$ bound. Recalling Remark \ref{remark:size_d_bar} and that the imaginary part of $\widetilde{\Phi}(x,y)$ is non-negative when $x$ and $y$ are real, we find that this integral is bounded by
\[
C_R \n{a}_{s,R,K_1 \times K_2}\lambda \exp\left(C \lambda  |\Im x| - \frac{1}{C \p{R \va{\Im x}}^{\frac{1}{s-1}}} \right).
\]
But if $\va{\Im x} < A_1^{-1} \lambda^{1/s-1} R^{- 1/s}$ then
\[
\lambda \leq A_1^{- \frac{s}{s-1}} \va{\Im x}^{- \frac{s}{s-1}} R^{-\frac{1}{s-1}}
\]
and thus the estimate above gives if $A_1$ is large enough
\[
\left|\overline{\partial} I_\lambda\p{ a}(x) \right| \leq C_{R} \| a \|_{s,R,K_1 \times K_2} \exp\left(-\frac{1}{C \p{R |\Im x|}^{\frac{1}{s-1}}} \right).
\]
Proposition \ref{prop:non-stationary} follows then from the Bochner--Martinelli Trick, recalling \eqref{eq:le_vrai_truc}.

To prove Proposition \ref{prop:non-stationary-analytic}, just notice that in that case, one may take $\widetilde{\Phi} = \Phi$ and $\tilde{a}(x,y) = b(x,y) \tilde{u}(y)$ where $\tilde{u}$ is an almost analytic extension for $u$. Then, $I_\lambda\p{u}$ has an holomorphic extension by differentiation under the integral and the estimate that we want to prove is just \eqref{eq:le_vrai_truc} (there is no difference in the proof).
\end{proof}

\begin{remark}
The same result holds if instead of $U\times \Omega\subset \R^{n+k}$, we assume that $U\times \Omega\subset M\times N$, with $M,N$ two compact Riemannian manifolds. The proof is essentially the same. The main difference is that we have to give a more ``geometric'' expression for the contour shift. For example, using the integral lines of $\nabla_y\Im \widetilde{\Phi}$, or using the exponential map. 
\end{remark}

\subsection{Gevrey Stationary Phase}\label{sec:stationary}

It is a classical observation that the Stationary Phase method in Gevrey regularity suffers from a loss of regularity. This comes essentially from the fact that in the usual expansion
\[
\int_{\R} e^{-\lambda w^2} f(w) \mathrm{d}w \sim \sqrt{\frac{\pi}{\lambda}}\sum_{k\geq 0} \frac{1}{(4\lambda)^k k!}f^{(2k)}(0),
\]
to gain a power of $\lambda$, we have to differentiate $f$ \emph{twice}. Let us be a bit more precise. We consider $f$ a $\G^s$ function, defined in a neighbourhood of $0$ in $\R^n \times \R^k$, and for $ U$ a small enough neighbourhood of $0$ in $\R^k$,we define
\[
g_\lambda: x\mapsto \frac{\lambda^{k/2}}{\pi^{k/2}}\int_U e^{- \lambda w^2 } f(x,w) \mathrm{d}w.
\]
By differentiation under the integral, we find that $g_\lambda$ is uniformly $\Gs$ as $\lambda\to +\infty$. Integrating by parts, we also find (for $C,C_0>0$ large enough)
\begin{equation}\label{eq:stationary-phase-sharpest-gevrey-estimate}
g_\lambda(x) = \sum_{0 \leq \ell \leq \frac{\lambda^{\frac{
1}{2s-1}}}{C_0}} \frac{1}{\ell !} \left(\frac{\Delta_w}{4 \lambda}\right)^\ell f(x,0) + \O_{\Gs} \left(\exp\left( - \frac{\lambda^{\frac{1}{2s-1}}}{C}\right) \right).
\end{equation}
Here, notice that while the remainder is measured in $\G^s$, it has the size of an error term in the $\G^{2s- 1}$ category (compare for instance with the bound given in Proposition \ref{prop:non-stationary}). In some sense, the reason behind this is that $I_\lambda$ is (if only formally) $\G^{2s}$ in the small parameter $1/\lambda$; we already encountered a similar problem in Remark \ref{remark:formal-Gs-symbols}. The case of more general phases is considerably worse: if instead of $w^2$ we allow a general phase function $\Phi(x,w)$, we will only recover a $\G^{2s-1}$ estimate in $x$. One can find a general bound in \cite[Theorem 3.1]{Gramchev-87-stationary-phase-gevrey}, for example.

Before we discuss this further, we recall the analytic version of Stationary Phase. This is a well-known tool (see for instance \cite[Th\'eor\`eme 2.8 and Remarque 2.10]{Sjostrand-82-singularite-analytique-microlocale}. We will use the following version of it (which is stated on a real contour in \cite{Sjostrand-82-singularite-analytique-microlocale}, the generalization that we state here is straightforward).

\begin{prop}[Holomorphic Stationary Phase]\label{prop:HSP}
Let $U \times V$ be an open subset of $\C^n \times \C^k$. Let $\Gamma$ be a $k$-dimensional compact real submanifold with boundary of $V$. Let $a,\Phi : U \times V \to \C$ be holomorphic functions. Let $x_0 \in U$ and assume that
\begin{enumerate}[label=(\roman*)]
\item the imaginary part of $\Phi(x_0,y)$ is non-negative for every $y \in \Gamma$;
\item the imaginary part of $\Phi(x_0,y)$ is strictly positive for every $y \in \partial \Gamma$;
\item the function $ y \mapsto \Phi(x_0,y)$ has a unique critical point $y_0$ in $\Gamma$, which is in the interior of $\Gamma$ and is non-degenerate;
\item $\Phi(x_0,y_0) = 0$.
\end{enumerate}
Then there is a neighbourhood $W$ of $x_0$ in $U$ such that for every $x \in W$, the function $y \mapsto \Phi(x,y)$ has a unique critical point $y_c(x)$ near $\Gamma$. We denote by $\Psi(x) = \Phi(x,y_c(x))$ the associated critical value. Then, there is a family $\p{b_m}_{m \geq 0}$ of holomorphic functions on $W$ such that there are constants $C, C_0$ such that, for every $m \in \N$,
\begin{equation*}
\begin{split}
\sup_{x \in W} \va{b_m(x)} \leq C^{1+m} m!
\end{split}
\end{equation*}
and, for every $\lambda \geq 1$,
\begin{equation*}
\begin{split}
\sup_{x \in W} \va{  e^{-i \lambda \Psi(x)} \lambda^{k/2} \int_{\Gamma} e^{i \lambda \Phi(x,y)} a(x,y) \mathrm{d}y - \sum_{0 \leq m \leq \lambda/C_0} \lambda^{-m} b_m(x)} \leq C \exp\p{- \frac{\lambda}{C}}.
\end{split}
\end{equation*}
\end{prop}

Let us come back to the non-analytic case $s>1$. In some cases, the loss of regularity can be mitigated, and we can obtain an estimate as sharp as \eqref{eq:stationary-phase-sharpest-gevrey-estimate}. For example, we obtain this statement, used in the proof of Lemma \ref{lemma:calcul_pseudo}:
\begin{lemma}\label{lemma:stationary_gevrey_real}
Let $s > 1$. Let $U \times V$ be an open subset of $\R^n \times \R^k$. Let $K = K_1 \times K_2$ be a compact subset of $U \times V$. Let $\Phi : U \times V \to \R$ be a $\G^s$ function on $U \times V$. Assume that for every $x \in U$ the function $y \mapsto \Phi(x,y)$ has a unique critical point $y_c(x)$ in $V$ that satisfies   $\Phi(x,y_c(x)) = 0$. Let $a \in E^{s,R}\p{K}$ for some $R > 0$. Define then for $\lambda \geq 1$ the function
\begin{equation}\label{eq:stationary_def}
\begin{split}
I_\lambda(a) : x \mapsto \lambda^{\frac{k}{2}}\int_{V} e^{i \lambda \Phi(x,y)} a(x,y) \mathrm{d}y
\end{split}
\end{equation}
on $U$. Then the function $I_\lambda(a)$ is $\G^s$ uniformly in $\lambda$. More precisely, there is $R_1 > 0$, a family $\p{b_m}_{m \geq 0}$ of functions in $E^{s,R_1}\p{K_1}$ and constants $C, C_0  > 0$ such that, for every $m \in \N$,
\begin{equation*}
\begin{split}
\n{b_m}_{s,R_1,K_1} \leq C^{1 + m} m!^{2s - 1}
\end{split}
\end{equation*}
and, for every $\lambda \geq 1$,
\begin{equation}\label{eq:dev_statio_reel}
\begin{split}
\n{I_\lambda(a) - \sum_{0 \leq m \leq \frac{\lambda^{\frac{1}{2s-1}}}{C_0}} \lambda^{-m} b_m}_{s,R_1,K_1} \leq C \exp\p{- \frac{\lambda^{\frac{1}{2s-1}}}{C}}.
\end{split}
\end{equation}
\end{lemma}

\begin{proof}
We can work locally in $x$. By the Morse Lemma, there is a change of coordinates $\rho_x$ near $y_c(x)$ that changes the phase in $\Phi(x,y) = \frac{y_+^2 - y_-^2}{2}$ for some decomposition $y = (y_+,y_-)$ of $\R^k$. Moreover, since the Implicit Function Theorem is true in the Gevrey category so is the Morse Lemma and hence $\rho_x$ is $\G^s$ (with a $\G^s$ dependence on $x$). Since $y_c(x)$ is the only critical point of $y \mapsto \Phi(x,y)$, we may introduce a cut off function to split the integral $I_\lambda(a)(x)$ into an integral over $y$'s that are close to $y_c(x)$ and $y$'s that are away from $y_c(x)$. The latter is negligible thanks to Proposition \ref{prop:non-stationary}. Hence, we may assume that $a(x,y)$ vanishes when $y$ is too far from $y_c(x)$, in particular if it does not belong to the domain of definition of $\rho_x$. We can then use $\rho_x$ as a change of variable in $I_\lambda(a)(x)$. We just reduced to the case where $\Phi(x,y) = \Phi(y) = \frac{y_+^2 - y_-^2}{2}$.

From now on, the proof is just the Gevrey version of the argument from the proof of \cite[Theorem 3.13]{zworskiSemiclassicalAnalysis2012} in the $\mathcal{C}^\infty$ case. We introduce the new parameter $h = \lambda^{-1}$ and take Fourier Transform in $y$ in order to write (for some constant $c \neq 0$)
\begin{equation*}
\begin{split}
I_\lambda(a)(x) = f(a,h)(x) = c \int_{\R^k} e^{i h \frac{\xi_+^2 - \xi_-^2}{2}} \mathcal{F}_ya (x,\xi) \mathrm{d}\xi.
\end{split}
\end{equation*}
In this form, it is easy to prove that the function $h \mapsto f(a,h)$ is a smooth function from $\left]0,1\right]$ to $\widetilde{E}^{s,R'}\p{K_1}$ where $R' > R$ (we may differentiate under the integral in $x$). To obtain a more precise statement, we introduce the differential operator $P = - \frac{i}{2}(\Delta_{y_+} - \Delta_{y_-})$ and notice that
\begin{equation*}
\begin{split}
\frac{\mathrm{d}^m}{\mathrm{d} h^m} f(a,h) = f\p{P^m a,h}.
\end{split}
\end{equation*}
However, reasoning as in Remark \ref{remark:size_d_bar}, we see that for some $R_0 > 0$, there is a constant $C > 0$ such that for every $m \in \N$, we have
\begin{equation*}
\begin{split}
\n{P^m a}_{s,R_0,K} \leq C^{1 + m} \p{2 m}!^s.
\end{split}
\end{equation*}
It follows that $h \mapsto f(a,h)$ is in fact a $\G^{2 s}$ function from $\left]0,1\right]$ to $\widetilde{E}^{s,R_1}\p{K_1}$ where $R_1 > R_0$. Since the estimates on the derivatives are uniform, it extends to a $\G^{2s}$ function on $\left[0,1\right]$, we can then apply Lemma \ref{lemma:decay-extension-gevrey} (or rather its straightforward extension to Banach valued functions) to get \eqref{eq:dev_statio_reel} with
\begin{equation*}
\begin{split}
b_m(x) = \frac{P^m a(x,0)}{m!}.
\end{split}
\end{equation*}
\end{proof}

Another situation in which we are able to get a $\G^s$ remainder in the stationary phase method is when the phase $\Phi(x,y)$ in \eqref{eq:stationary_def} is purely imaginary and satisfies some sort of positivity condition. More generally, if $\Phi$ is a real analytic phase, with $\Im \Phi\geq 0$ for real parameters, and if $x_0$ is itself a non-degenerate critical point of the critical value $\Phi_c(x)$, then one gets a $\Gs$ remainder; we will not need this result.

In the key Lemma \ref{lemma:TPS-close-diagonal} in the next chapter, a somewhat different phenomenon will take place. The amplitude will be analytic in $x$, and $\Gs$ in $w$, so that we can consider the holomorphic extension of $I_\lambda(a)$ in the $x$ variable. We will be able to control the stationary integral in a region in the complex bigger than a ``$\G^{2s-1}$'' region $\{ |\Im x| \ll \lambda^{1/(2s-1)-1} \}$, but smaller than a ``$\G^s$'' region $\{ |\Im x| \ll \lambda^{1/s -1}\}$. 

\begin{remark}\label{remark:expression_coefficients}
All of the avatars of the Stationary Phase Method that we use are proved using the same global strategy: we reduce to the case of a quadratic phase using some version of the Morse Lemma. Consequently, the coefficients in the expansion
\begin{equation*}
\begin{split}
e^{- i \lambda \Psi(x)}\lambda^{\frac{k}{2}} \int e^{i \lambda \Phi(x,y)} a(x,y) \mathrm{d}y \sim \sum_{m \geq 0} \lambda^{-m} b_m(x)
\end{split}
\end{equation*}
that appear in Proposition \ref{prop:HSP} or Lemma \ref{lemma:stationary_gevrey_real}, are given by the same expression
\begin{equation*}
\begin{split}
b_m(x) = \frac{\p{2 \pi}^{\frac{k}{2}}}{m!} P_m a(x,y_c(x)).
\end{split}
\end{equation*}
Here, $y_c(x)$ denotes the critical point of $y \mapsto \Phi(x,y)$, and $P_m$ is the differential operator acting on the $y$ variable defined using Morse coordinates $\rho_x$, such that
\begin{equation}\label{eq:Morse_cequecest}
\begin{split}
\Phi(x,y) = \Phi(x,y_c(x)) + i \frac{\rho_x(y)^2}{2},
\end{split}
\end{equation}
by the expression
\begin{equation}\label{eq:Pm_thefirst}
\begin{split}
P_m u =  \left[ \p{\frac{\Delta}{2}}^m \p{ u \circ \rho_x^{-1} J \rho_x^{-1}} \right] \circ \rho_x.
\end{split}
\end{equation}
Here, $J \rho_x^{-1}$ denotes the Jacobian of $\rho_x^{-1}$. Concerning \eqref{eq:Morse_cequecest}, let us point out that if $z = (z_1,\dots,z_k) \in \C^k$, we write $ z^2 = \jap{z,z} = \sum_{j = 1}^n z_j^2$. The Laplacian in \eqref{eq:Pm_thefirst} is defined on $\C^k$ using holomorphic derivatives: $\Delta = \sum_{j  = 1}^n \frac{\partial^2}{\partial z_j^2}$ (we retrieve the usual Laplacian when restricting this operator to $\R^k$).

One can check that these are the same operators that appear in the proof of Lemma \ref{lemma:stationary_gevrey_real}. Of particular importance is the first term in the expansion
\begin{equation*}
\begin{split}
b_0(x) = (2 \pi)^{\frac{k}{2}}\frac{a(x,y_c(x))}{\sqrt{\det\p{-i \mathrm{d}_y^2 \Phi(x,y_c(x))}}}.
\end{split}
\end{equation*} 
The choice of the determination of the square root here depends on the choice of orientation on the contour of integration. We will always arrange so that the argument of the square root has non-negative real part and choose the corresponding holomorphic determination of the square root (see \cite{Sjostrand-82-singularite-analytique-microlocale} for details).
\end{remark}

Let us now state the version of the Morse Lemma that will be needed in the Stationary Phase Argument in the proof of Proposition \ref{lemma:TPS-close-diagonal}. The phase there will be analytic rather than Gevrey, so that we will rely on the following version of the Holomorphic Morse Lemma. (This is also the version of the Morse Lemma that is used in the proof of Proposition \ref{prop:HSP}).

\begin{lemma}[Holomorphic Morse Lemma]\label{lemma:holomorphic_morse}
Let $U$ and $V$ be open subsets respectively of $\C^n$ and $\C^k$. Let $\Phi : U \times V \to \C$ be a holomorphic function. Assume that there is $(x_0,y_0) \in U \times V$ such that $y \mapsto \Phi(x_0,y)$ has a non-degenerate critical point at $y_0$. Then there exist open neighbourhoods $U_0, V_0$ and $W$ of $x_0,y_0$ and $0$ respectively in $\C^n,\C^k $ and $\C^k$ and holomorphic maps $\rho : U_0 \times V_0 \mapsto \C^k$ and $y_c : U_0 \to V_0$ such that
\begin{enumerate}[label = (\roman*)]
\item for every $x \in U_0$, the point $y_c(x)$ is the unique critical point of $y \mapsto \Phi(x,y)$ in $V_0$;
\item for every $x \in U_0$, the map $V_0 \ni y \mapsto \rho(x,y)$ is a diffeomorphism onto its image, which contains $W$;
\item for every $x \in U_0$ and $y \in V_0$, we have
\begin{equation*}
\begin{split}
\Phi(x,y) = \Phi(x,y_c(x)) + \frac{i}{2} \rho(x,y)^2.
\end{split}
\end{equation*}
\end{enumerate}
\end{lemma}

The proof of this classical result may be found for instance in \cite{Sjostrand-82-singularite-analytique-microlocale} (see Lemme 2.7 there and the remark just after it). In several proofs, in order to apply a method of steepest descent, it will be important to control the imaginary part of the critical value of the phase. The tool for this is the so-called ``fundamental'' lemma, presented in an elaborate version in \cite[Lemme 3.1]{Sjostrand-82-singularite-analytique-microlocale}. We give here a more elementary version.
\begin{lemma}\label{lemma:fundamental-lemma}
Let $\Phi$ be a phase as in Lemma \ref{lemma:holomorphic_morse}. Assume in addition that $x_0$ and $y_0$ are real points, that $\Im \Phi(x_0,y_0) = 0$ and that $\Im \Phi(x,y)\geq 0$ for real $x,y$. Then for $x$ real close to $x_0$,
\[
\Im\Phi(x,y_c(x))\geq 0.
\]
If additionally, we assume that $\Im \mathrm{d}^2\Phi (x_0,y_0)$ is non-degenerate, then for some constant $C>0$, and $x$ real close to $x_0$,
\[
\Im\Phi(x,y_c(x))\geq \frac{1}{C} (x-x_0)^2.
\]
\end{lemma}

\begin{proof}
Let $\rho$ be as in Lemma \ref{lemma:holomorphic_morse} and write $\rho_x = \rho(x,\cdot)$ for $x \in U_0$. We start by decomposing the Jacobian matrix of $\rho_{x_0}$ at $y_0$ into real and imaginary part 
\[
D \rho_{x_0}(y_0) = A +i B.
\]
According to the definition of Morse coordinates, we obtain
\[
i {}^t (A+i B) (A+ iB) = D_{y,y}^2 \Phi(x_0,y_0).
\]
Identifying the imaginary part, it follows that ${}^t A A - {}^t B B = \Im D_{y,y}^2 \Phi(x_0,y_0) \geq 0$. Since $A+iB$ has to be invertible, this implies that $A$ is actually invertible. We deduce that, up to taking $V_0$ smaller, $\rho_{x_0}(V_0 \cap \R^k)$ is a graph over the reals, and this remains true by perturbation for $x$ close to $x_0$. We can thus write
\[
\rho_x(V_0 \cap \R^k) = \{ w + i  F_x(w)\ |\ w\in \Re \rho_x(V_0 \cap \R^k) \}.
\]
Here, $F_x$ is a real analytic map in all its variables. In particular, if $x$ is real, at the point $y=\rho_x^{-1}(i F_x(0))$, we have 
\[
\Phi(x,y) = \Phi(x,y_c(x)) - \frac{i}{2} \va{F_x(0)}^2,
\]
so that $\Im \Phi(x,y_c(x)) \geq \va{F_x(0)}^2/2 \geq 0$ since $\Im \Phi(x,y) \geq 0$ by assumption. 

Now, we study the particular case in which $\Im \mathrm{d}^2 \Phi(x_0,y_0)$ is non-degenerate. Since $\Im\Phi(x,y)\geq 0$ for real $x,y$,  the quantity $\Im\Phi(x,y_c(x))$ has to be critical at $x=x_0$. Hence it suffices for our purpose to show that the imaginary part of the Hessian of the critical value is invertible at $x=x_0$. For this, it is easier to directly work with $\Phi$. According to the Inverse function theorem,
\[
D y_c = - (D_{y,y}^2 \Phi)^{-1} D_{y,x}^2 \Phi.
\]
Here, our convention regarding matrices of second differential is that the rows of $D_{y,x}^2 \Phi$ have the size of $x$ and its columns the size of $y$. Letting $\Phi_c : x \mapsto \Phi(x,y_c(x))$, it comes that $\nabla_x \Phi_c (x) = \nabla_x \Phi(x,y_c(x))$, and
\[
D^2 \Phi_c(x) = D_{x,x}^2 \Phi (x, y_c(x)) -  D_{x,y}^2 \Phi(x)(D_{y,y}^2 \Phi(x))^{-1} D_{y,x}^2 \Phi(x).
\]
This is the Schur complement of $D_{y,y}^2 \Phi$ in $D^2 \Phi$. We deduce that $D^2\Phi$ is invertible if and only if $D^2\Phi_c $ is invertible. Since it is the case, with $x=x_0$ and $y=y_0$, the inverse of the Hessian $D^2 \Phi$ is given by
\begin{equation*}
\begin{split}
& \begin{pmatrix}
1 & 0 \\  -\p{D_{y,y}^2 \Phi}^{-1} D_{y,x}^2\Phi & 1
\end{pmatrix}\begin{pmatrix}
\p{D^2 \Phi_c}^{-1} & 0 \\ 0 & \p{D_{y,y}^2\Phi}^{-1} 
\end{pmatrix} \\ & \qquad \qquad \qquad \qquad \qquad \qquad \qquad \qquad \qquad \times \begin{pmatrix}
1 & -  D_{x,y}^2 \Phi \p{D_{y,y}^2 \Phi}^{-1} \\ 0 & 1
\end{pmatrix}.
\end{split}
\end{equation*}
For a vector $u\in \R^n$, we deduce
\[
\jap{ \p{D^2\Phi_c}^{-1} u, u} = \left\langle \p{D^2 \Phi}^{-1} \begin{pmatrix} u\\ 0 \end{pmatrix} ,  \begin{pmatrix} u\\ 0 \end{pmatrix} \right\rangle.
\]
Since $\Im D^2\Phi > 0$, we deduce that $\Im( \p{D^2\Phi}^{-1}) < 0$, so that $\Im \p{D^2\Phi_c}^{-1}<0$ and finally $\Im D^2\Phi_c >0$. This ends the proof.
\end{proof}

\subsection{Further tricks}\label{subsec:further_tricks}

So far, we have considered integrals of the form
\[
\int_\Omega e^{i \lambda \Phi(x,y)} a(x,y) dy,
\]
where $\Omega\subset \R^k$ is relatively compact and $x$ is assumed to be varying in a compact set. However, sometimes, we will need to replace the parameter $x$ by a parameter $\p{x,\xi}$ where $x$ varies in a compact set but $\xi \in \R^n$. In that case, the phase $\Phi$ and the amplitude $a$ will have symbolic behavior with respect to $\xi$ and the large parameter $\lambda$ will be replaced by $\jap{\xi}/h$ where $h > 0$ is small as usual. In this context, Propositions \ref{prop:non-stationary} and \ref{prop:non-stationary-analytic} rewrite:

\begin{prop}\label{prop:non_stationary_symbols}
Let $s \geq 1, m \in \R$ and $U\times V$ be an open subset of $\R^n \times \R^k$. Let $\Phi$ be an element of $S^{s,1}\p{T^* U \times V}$ and $\Omega$ be a relatively compact open subset of $V$. Let $a \in S^{s,m}\p{T^*U \times V}$ and, for $0 < h \leq 1$, define
\begin{equation}\label{eq:integrale_non_stationnaire_symbols}
\begin{split}
I_h(a) : \p{x,\xi}  \mapsto \int_{\Omega} e^{\frac{i}{h}  \Phi(x,\xi,y)} a(x,\xi,y) \mathrm{d}y.
\end{split}
\end{equation}
We make the following assumptions:
\begin{enumerate}[label=(\roman*)]
	\item $\Im \Phi \geq 0$ on $T^*U\times \Omega$;
	\item there exists $\epsilon>0$ such that $\Im \Phi(x,\xi,y) > \epsilon \jap{\xi}$ for $\p{x,\xi} \in T^* U$ and $y\in \partial \Omega$;
	\item there is a $C > 0$ such that for every $\p{x,\xi,y} \in T^* U \times V$ such that $\va{\xi} \geq C$ we have
\begin{equation*}
\begin{split}
\va{\mathrm{d}_y \Phi(x,\xi,y)} \geq  \frac{\jap{\xi}}{C}.
\end{split}
\end{equation*}
\end{enumerate} 
Let $K$ be a compact subset of $U$. Then, there are constants $C,R > 0$ such that, for every $(x,\xi) \in K \times \R^n$ with $\va{\xi} \geq C$ and every $0 < h \leq 1$, we have
\begin{equation}\label{eq:apres_rescaling}
\sup_{\alpha,\beta \in \N^n} \va{\frac{\jap{\xi}^{\va{\beta}} \partial_x^\alpha \partial_\xi^\beta I_h(a)(x,\xi)}{R^{\va{\alpha} + \va{\beta}}\alpha!^s \beta!^s}} \leq C \exp\left( - \left( \frac{\langle\xi\rangle}{Ch}\right)^{\frac{1}{s}} \right).
\end{equation}
\end{prop}

\begin{prop}\label{prop:non_stationary_analytic_symbols}
Let $s \geq 1, m \in \R$ and $U\times V$ be an open subset of $\R^n \times \R^k$. Let $\Phi$ be an element of $S^{1,1}\p{T^* U \times V}$, $a$ be an element of $S^{1,m}\p{T^* U \times V}$ and  $\Omega$ be a relatively compact open subset of $V$. Then, for $u \in \G^s\p{V}$ and $0 < h \leq 1$, define
\begin{equation}\label{eq:integrale_non_stationnaire_analytic_symbols}
\begin{split}
I_h(u) : \p{x,\xi}  \mapsto \int_{\Omega} e^{\frac{i}{h}  \Phi(x,\xi,y)} a(x,\xi,y) u(y) \mathrm{d}y.
\end{split}
\end{equation}
We make the following assumptions:
\begin{enumerate}[label=(\roman*)]
	\item $\Im \Phi \geq 0$ on $T^*U\times \Omega$;
	\item there exists $\epsilon>0$ such that $\Im \Phi(x,\xi,y) > \epsilon \jap{\xi}$ for $\p{x,\xi} \in T^* U$ and $y\in \partial \Omega$;
	\item there is a $C > 0$ such that for every $\p{x,\xi,y} \in T^* U \times V$ such that $\va{\xi} \geq C$ we have
\begin{equation*}
\begin{split}
\va{\mathrm{d}_y \Phi(x,\xi,y)} \geq \frac{\jap{\xi}}{C}.
\end{split}
\end{equation*}
\end{enumerate} 
Let $K$ be a compact subset of $U$ and $K'$ be a compact neighbourhood of $\Omega$ in $V$. Then, there is an $\epsilon > 0$ such that for every $u \in \G^s\p{U}$, the function $I_h(u)$ has a holomorphic extension to $\p{T^* K}_\epsilon$. Moreover, there is $C > 0$ such that, for every $R \geq 1$, there is $C_R > 0$, such that if $0 < h \leq 1$ and $(x,\xi) \in \p{T^* K}_{\epsilon}$ are such that
\begin{equation*}
\begin{split}
\va{\xi} \geq C \textup{ and } \va{\Im x} + \frac{\va{\Im \xi}}{\jap{\Re \xi}} \leq  \frac{1}{(CR)^{\frac{1}{s}}} \p{\frac{h}{\jap{\Re \xi}}}^{1 - \frac{1}{s}}
\end{split}
\end{equation*}
and $u \in \G^s\p{V}$ then
\begin{equation*}
\begin{split}
\va{I_h(u)(x,\xi)} \leq C_R \n{u}_{s,R,K'} \exp\p{-  \p{\frac{\jap{\Re \xi}}{C R h}}^{\frac{1}{s}}}.
\end{split}
\end{equation*}
\end{prop}

Proofs of Proposition \ref{prop:non_stationary_symbols} and \ref{prop:non_stationary_analytic_symbols} may be deduced easily from the proofs of Proposition \ref{prop:non-stationary} and \ref{prop:non-stationary-analytic}, replacing the estimate from Remark \ref{remark:size_d_bar} by the bound from Remark \ref{remark:aae_for_symbol} on the $\bar{\partial}$ of almost analytic extension of Gevrey symbols.

However, there is also a standard rescaling argument that allows to deduce Propositions \ref{prop:non_stationary_symbols} and \ref{prop:non_stationary_analytic_symbols} from Propositions \ref{prop:non-stationary} and \ref{prop:non-stationary-analytic}. It also applies to stationary phase estimates. Let us recall this argument, for instance in the case of Proposition \ref{prop:non_stationary_analytic_symbols} . To deal with $\va{\xi} \sim M \gg 1$, we want to use the rescaling $\eta = M^{-1} \xi$ and thus define the phase and amplitude
\begin{equation*}
\begin{split}
\Phi_M\p{x,\eta,y} = \frac{\Phi\p{x,M \eta,y}}{M} \textup{ and } a_M(x,\eta,y) = M^{-m} a\p{x, M \eta,y}
\end{split}
\end{equation*}
that satisfy uniform Gevrey estimates when $\p{x,y}$ is in a compact subset of $U \times V$ and $ C^{-1} \leq \va{\eta} \leq C$ for $C > 0$ arbitrarily large. Then, in the notations on Proposition \ref{prop:non_stationary_symbols}, we have
\begin{equation}\label{eq:rescaling_non_stationary}
\begin{split}
I_h\p{a} (x,\xi) = M^m \int_{\Omega} e^{i \lambda \Phi_M\p{x,\eta,y}} a_M\p{x, \eta ,y} \mathrm{d}y,
\end{split}
\end{equation}
where $\lambda = h^{-1} M$ and $\eta = M^{-1} \xi$. Notice then that when $\va{\xi} \sim M$ we have $\lambda \sim h^{-1} \jap{\xi}$ and the point $\eta$ remains in the domain $C^{-1} \leq \va{\eta} \leq C$ where we have uniform Gevrey estimates on $a_M$ and $\Phi_M$. Hence, the integral in the right hand side of \eqref{eq:rescaling_non_stationary} falls into the domain of application of Proposition \ref{prop:non-stationary}, and \eqref{eq:apres_rescaling} is deduced from \eqref{eq:avant_rescaling} (taking into account the scaling $\xi = M \eta$ and neglecting the factor $M^m$ in \eqref{eq:rescaling_non_stationary} whose growth is annihilated by the decay of the stretched exponential).

The same rescaling argument also gives a symbolic version of Proposition \ref{prop:HSP}.

\begin{prop}\label{prop:HSP_symbol}
Let $U \times V \times W$ be an open subset of $\R^n \times \R^k \times \R^d$. Let $r \in \R$. Let $a \in S^{1,r}(T^* U \times V \times W)$ and $\Phi \in S^{1,0}(T^* U \times V \times W)$. Let $\Gamma$ be a $d$-dimensional compact real submanifold with boundary of $\C^d$. Assume that $\Gamma$ is contained in a small enough complex neighbourhood of $W$ so that $a$ and $\Phi$ are well-defined (and satisfy symbolic estimates) on a complex neighbourhood of $T^* U \times V \times \Gamma$. Let $(x_0,z_1) \in U \times V $, $y_0$ be in the interior of $\Gamma$ and $F \subseteq \R^n$. Assume that
\begin{enumerate}[label=(\roman*)]
\item the imaginary part of $\Phi(x_0,\xi,z_1,y)$ is non-negative for every $y \in \Gamma$ and $\xi \in F$;
\item there is $\epsilon > 0$ such that the imaginary part of $\Phi(x_0,\xi,z_1,y)$ is greater than $\epsilon$ for every $y \in \partial \Gamma$ and $\xi \in F$;
\item for every $\xi \in F$, the function $y \mapsto \Phi(x_0,\xi,z_1,y)$ has a unique critical point in $\Gamma$, which is $y_0$ and is non-degenerate, with symbolic estimates in $\xi$;
\item for every $\xi \in F$ we have $\Phi(x_0,\xi,z_1,y_0) = 0$.
\end{enumerate}
Then there is a neighbourhood $U_0 \times V_0$ of $(x_0,z_1)$ in $\C^{n+k}$ and $\eta > 0$ such that, if we denote by $G$ the set of $\xi \in \C^n$ at distance less than $\eta$ of $F$ for the Kohn--Nirenberg metric, then, for every $(x,z) \in (U_0,V_0)$ and $\xi\in G$, the function $y \mapsto \Phi(x,\xi,z,y)$ has a unique critical point $y_c(x,\xi,z)$ near $\Gamma$. We denote by $\Psi(x,\xi,z) = \Phi(x,\xi,z,y_c(x,\xi,z))$ the associated critical value. Then, there is a formal analytic symbol $\sum_{m \geq 0} h^m b_m$ of order $r$ on $U_0 \times G \times V_0 $ such that for $C_0 > 0$ large enough, there is a constant $C > 0$ such that, for every $0 < h \leq 1, (x,z) \in U_0 \times V_0$ and $\xi \in G$, we have
\begin{equation}\label{eq:estimee_phase_stationnaire}
\begin{split}
& \va{  e^{-i \frac{\jap{\xi}}{h} \Psi(x,\xi,z)}\p{\frac{\jap{\xi}}{2 \pi h}}^{k/2} \int_{\Gamma} e^{i \frac{\jap{\xi}}{h} \Phi(x,\xi,z,y)} a(x,\xi,z,y) \mathrm{d}y -\sum_{0 \leq m \leq \jap{|\xi|}/C_0 h} h^{m} b_m(x,\xi,z)} \\ 
& \hspace{210pt} \leq C \exp\p{- \frac{\jap{\va{\xi}}}{Ch}}.
\end{split}
\end{equation}
\end{prop}

\begin{remark}\label{remark:stationary_symbol}
In most of the applications, we will have $G = \R^n$ so that $\sum_{m \geq 0} h^m b_m$ is a proper formal analytic symbol on $T^* U_0 \times V_0$ in the sense of Definition \ref{def:gevrey_formal_symbol} (otherwise, see Remark \ref{remark:slight_abuse}). Then, we see that the integral 
\begin{equation}\label{eq:integrale_phase_stationnaire_symbole}
\begin{split}
 e^{-i \frac{\jap{\xi}}{h} \Psi(x,\xi,z)}\p{\frac{\jap{\xi}}{2 \pi h}}^{k/2} \int_{\Gamma} e^{i \frac{\jap{\xi}}{h} \Phi(x,\xi,z,y)} a(x,\xi,z,y) \mathrm{d}y
\end{split}
\end{equation}
is a realization of $\sum_{m \geq 0} h^m b_m$ in the sense of Definition \ref{def:realisation_Gevrey}.
\end{remark}

\begin{remark}\label{remark:stationary_contour}
It will sometimes happen when applying Proposition \ref{prop:HSP_symbol} that the contour $\Gamma$ depends on $x$ or $\xi$ so that the integral \eqref{eq:integrale_phase_stationnaire_symbole} will not necessarily be a holomorphic functions of the corresponding variables. However, $\Gamma$ will always remain uniformly smooth and the assumptions of Proposition \ref{prop:HSP_symbol} will hold uniformly. Thus, the estimate \eqref{eq:estimee_phase_stationnaire} will still hold.

In particular, in the case when $G = \R^n$, then the integral \eqref{eq:integrale_phase_stationnaire_symbole} is approximated with a precision $\exp\p{- \jap{\va{\xi}}/h}$ by any realization of the formal symbol $\sum_{m \geq 0} h^m b_m$.
\end{remark}

We end this section by stating a rescaled version of Lemma \ref{lemma:stationary_gevrey_real}.

\begin{lemma}\label{lemma:stationary_gevrey_real_symbol}
Let $s > 1,r \in \R$. Let $U \times V$ be an open subset of $\R^n \times \R^k$. Let $K_2$ be a compact subset of $V$. Let $\Phi : V \to \R$ be a $\G^s$ function on $V$ and $a \in S^{s,r}\p{T^* U \times V}$. Assume that
\begin{enumerate}[label = (\roman*)]
\item if $(x,\xi) \in T^* U$ and $y \in V \setminus K_2$ then $a(x,\xi,y) = 0$;
\item the function $\Phi$ has a unique critical point $y_0 \in V$, which is non-degenerate and satisfies $\Phi(y_0) = 0$.
\end{enumerate} 
Then
\begin{equation*}
\begin{split}
b(x,\xi) = \p{\frac{\jap{\xi}}{(2 \pi h)}}^{\frac{k}{2}} \int_{V} e^{i \frac{\jap{\xi}}{h} \Phi(y) } a(x,\xi,y) \mathrm{d}y
\end{split}
\end{equation*}
defines a $\G^s$ symbol $b \in S^{s,r}\p{T^* U}$ given at first order by 
\begin{equation*}
\begin{split}
b(x,\xi) = \frac{e^{i \frac{\pi}{4}q}}{\va{\det\p{\mathrm{d}^2 \Phi(y_0)}}^{\frac{1}{2}}}a(x,\xi,y_0) \textup{ mod } h S^{s,r-1}\p{T^* U},
\end{split}
\end{equation*}
where $q$ denotes the signature of $\mathrm{d}^2 \Phi(y_0)$.
\end{lemma}

\begin{remark}\label{remark:developpement_symbole_gs}
We do not give here the most general version possible of Lemma \ref{lemma:stationary_gevrey_real_symbol}. In particular, we could make $\Phi$ depend on $x$ and $\xi$ (assuming that $\Phi$ is a $\G^s$ symbol of order $0$).

Notice also that we can give a full asymptotic expansion
\begin{equation*}
\begin{split}
b \sim \sum_{m \geq 0} h^m b_m,
\end{split}
\end{equation*}
where $b_m \in S^{s,r - m}\p{T^* U}$. Here, the asymptotic expansion holds in quite a strong sense (not as satisfactory as in the case $s=1$ though). We will not need the full strength of this expansion, and, consequently, let us just mention that, for every $N \geq 0$, we have
\begin{equation*}
\begin{split}
b - \sum_{m = 0}^N h^m b_m \in h^{N+1} S^{s,r - N-1}\p{T^* U}.
\end{split}
\end{equation*}
For the expression of the $b_m$'s we may use either the one that is given in the proof of Lemma \ref{lemma:stationary_gevrey_real} or in Remark \ref{remark:expression_coefficients}.
\end{remark}

\begin{remark}\label{remark:coefficients_symboles}
Taking into account the rescaling, we may deduce from Remark \ref{remark:expression_coefficients} an expression for the coefficients in the expansion
\begin{equation*}
\begin{split}
 e^{-i \frac{\jap{\xi}}{h} \Psi(x,\xi,z)}\p{\frac{\jap{\xi}}{2 \pi h}}^{\frac{k}{2}}\int e^{i \frac{\jap{\xi}}{h} \Phi(x,\xi,z,y)} a(x,\xi,z,y) \mathrm{d}y \sim \sum_{m \geq 0} h^m b_m
\end{split}
\end{equation*}
from Proposition \ref{prop:HSP_symbol} (or Remark \ref{remark:developpement_symbole_gs}). The coefficients $b_m$'s are given by the expression
\begin{equation*}
\begin{split}
b_m(x,\xi,z) = \frac{P_m a(x,\xi,z,y_c(x,\xi,z))}{m!}.
\end{split}
\end{equation*}
Here, the differential operator $P_m$ acts on the $y$ variable and is defined by the formula
\begin{equation*}
\begin{split}
P_m u = \jap{\xi}^{- \frac{k}{2}} \left[ \p{\frac{\Delta}{2}}^m\p{ u \circ \rho_{x,\xi,z}^{-1} J \rho_{x,\xi,z}^{-1} } \right] \circ \rho_{x,\xi,z},
\end{split}
\end{equation*}
where $\rho_{x,\xi,z}$ denotes Morse coordinate such that
\begin{equation*}
\begin{split}
\Phi(x,\xi,z,y) = \Psi(x,\xi,z,y_c(x,\xi,z)) + i \frac{\rho_{x,\xi,z}(y)^2}{2}.
\end{split}
\end{equation*}
For later use, notice that the first term in the expansion is given by
\begin{equation*}
\begin{split}
b_0(x,\xi,z) = \frac{a(x,\xi,z,y_c(x,\xi,z))}{\sqrt{\det\p{-i \mathrm{d}_y^2 \Phi(x,\xi,z,y_c(x,\xi))}}},
\end{split}
\end{equation*}
the choice of the determination of the square root will be dealt with as we explained in Remark \ref{remark:expression_coefficients}.
\end{remark}

\section{Gevrey pseudo-differential operators}\label{sec:toolbox_gevrey_pseudor}

\subsection{Gevrey and analytic pseudo-differential operators on compact manifolds}\label{subsection:Gevrey_pseudor_manifold}

Now that we have introduced classes of $\G^s$ symbols, we are ready to define and study the basic properties of $\G^s$ pseudo-differential operators (or from now on, pseudors). Thus, let $s \geq 1$ and let $M$ be a $n$ dimensional compact $\G^s$ manifold. According to Remark \ref{remark:structure_analytique}, we may endow $M$ with a compatible structure of real-analytic manifold and choose a real-analytic metric $g$ on $M$. Then, we endow $M$ with the Lebesgue measure associated with $g$, so that we may consider Schwartz kernels of operators from $\mathcal{C}^\infty
\p{M}$ to $\mathcal{D}'\p{M}$.

We will work in the semi-classical setting. Hence, throughout the paper, $h > 0$ will be a small implicit parameter, thought as tending to $0$. See \cite{zworskiSemiclassicalAnalysis2012} for an introduction to semi-classical pseudo-differential operators in the $\mathcal{C}^\infty$ category. To lighten the redaction, it is customary to refer to classical (i.e without an $h$) pseudo-differential operators as ``pseudors'' and the semi-classical ones as ``$h$-pseudors''. Since there will be very few classical operators in this text, we will drop the $h$, and write only pseudors in the rest of the paper.

These pseudors are defined by two properties: pseudo-locality and a fixed local model for their kernels near the diagonal.
\begin{definition}[Pseudo-locality]\label{def:pseudo-local}
Let $s\geq 1$, and $P=(P_h)_{h>0}$ a family of operators from $\mathcal{C}^\infty
\p{M}$ to $\mathcal{D}'\p{M}$, with Schwartz kernel $K_P$. We say that $P$ is $\G^s$ \emph{pseudo-local} if its kernel is small in $\G^s$ outside of the diagonal. More precisely, $P$ is $\G^s$ pseudo-local, if, for every compact subset $K$ of $M \times M$ that does not intersect the diagonal of $M \times M$, there are constants $C,R > 0$ such that 
\begin{equation}\label{eq:pseudor_away_diag}
\begin{split}
\n{K_P}_{s,R,K} \leq C \exp\p{- \frac{1}{C h^{\frac{1}{s}}}}.
\end{split}
\end{equation}
\end{definition}

\begin{definition}[Gevrey pseudo-differential operators]\label{def:gevrey_pseudor}
Let $m \in \R$, $s\geq 1$, and $P=(P_h)_{h>0}$ be a family of operators from $\mathcal{C}^\infty
\p{M}$ to $\mathcal{D}'\p{M}$, with Schwartz kernel $K_P$. Let us assume that:
\begin{enumerate}[label=(\roman*)]
\item $P$ is $\G^s$ pseudo-local;
\item near the diagonal, the kernel of $P$ is given (in any relatively compact open subset $V$ of a local $\G^s$ coordinate patch $U$) by the oscillatory integral
\begin{equation}\label{eq:local_pseudor}
\begin{split}
K_P(x,y) =  \frac{1}{(2\pi h)^n} \int_{\R^n} e^{i\frac{\jap{x-y,\xi}}{h}} p(x,\xi) \mathrm{d}\xi + \O_{\G^s}\p{\exp\p{- \frac{1}{C h^{\frac{1}{s}}}}},
\end{split}
\end{equation}
where $p \in S^{s,m}\p{T^* V}$ (in the sense of Definition \ref{def:gevrey_symbol}).
\end{enumerate}
Then we say that $P$ is a $\G^s$ pseudor of order $m$. The set of $\G^s$ pseudors (resp. of order $m$) on $M$ will be denoted by $\G^s \Psi\p{M}$ (resp. $\G^s \Psi^m\p{M}$).
\end{definition}

\begin{remark}\label{remark:g1gs_results}
As we explained before, one could define a smaller class of pseudors by requiring in (ii) that the symbol $p$ is a $\G^s_x \G^1_\xi$ symbol as defined in the introduction. Of course, when $s=1$, this is the same class of symbol. Most results from the current section (Propositions \ref{prop:composition_pseudors}, \ref{prop:adjoint_pseudor}, \ref{prop:quantization_Gs} and \ref{prop:mapping_pseudor}) remain true in the class of $\G^s_x \G^1_\xi$ pseudors (and are in fact more standard). 

Concerning the technical results from the following sections, let us list here the main differences for $\G^s_x \G^1_\xi$ pseudors (these results were present in an earlier version of this paper). In Lemma \ref{lemma:loin_de_la_diagonale}, we have the better bound $\mathcal{O}_{\G^s}\p{\exp(-1/Ch)}$. In Lemma \ref{lemma:calcul_pseudo}, we have an efficient asymptotic expansion for the symbol $b$. The Kuranishi trick Lemma \ref{lmkurtrick} is also true also for $s > 1$ when considering $\G^s_x \G^1_\xi$ pseudors (since it only relies on a contour shift in the $\xi$ variable), and there is a weaker version of Proposition \ref{prop:pseudor_lagrangien_alt} valid when $s > 1$. We also expect that the parametrix construction Theorem \ref{thm:parametrix} can be adapted to the case $s > 1$ if we consider $\G^s_x \G^1_\xi$ pseudors. The reader interested with the pseudo-differential calculus associated to $\G^s_x \G^1_\xi$ symbols may refer to \cite{Boutet-Kree-67,zanghiratiPseudodifferentialOperatorsInfinite1985,Rodino-93-book}.
\end{remark}

We state now the main properties of $\G^s$ pseudors. These results will be proved in \S \ref{sec:pseudor_technique} after technical preparations in \S \ref{sssecKuranishi}. Notice that Definition \ref{def:gevrey_pseudor} implies that $\G^s$ pseudors are in particular pseudors in the usual (i.e. $\mathcal{C}^\infty$) sense. To start with, a $\G^s$ pseudor sends $\mathcal{C}^\infty$ into $\mathcal{C}^\infty$. Hence, it makes sense to compose $\G^s$ pseudor and without surprise we obtain that:
\begin{prop}[Product of $\G^s$ pseudors]\label{prop:composition_pseudors}
Let $P$ and $Q$ be $\G^s$ pseudors on $M$. Then $P Q$ is a $\G^s$ pseudor. Consequently, $\G^s\Psi$ is a subalgebra of $\mathcal{C}^\infty\Psi$.
\end{prop}

This algebra is stable by taking adjoints:
\begin{prop}[Adjoint of a $\G^s$ pseudor]\label{prop:adjoint_pseudor}
Let $A$ be a $\G^s$ pseudor on $M$, and $\mathrm{d}\mu$ a volume form with $\G^s$ density on $M$. Then the adjoint $A^*$ of $A$  with respect to $\mathrm{d}\mu$ is also a $\G^s$ pseudor.
\end{prop}

As one could expect, $\G^s$ pseudors have better mapping properties than $\mathcal{C}^\infty$ pseudors.
\begin{prop}[Mapping properties for $\G^s$ pseudors]\label{prop:mapping_pseudor}
Let $P$ be a $\G^s$ pseudor on $M$. Then $P$ is continuous from $\G^s\p{M}$ to itself and extends continuously from $\U^s\p{M}$ to itself.
\end{prop}

We will also need to discuss the principal symbol of a $\G^s$ pseudor. We will use the following result. The proof in the case $s > 1$ is elementary once some technical facts on $\G^s$ pseudors have been established. The proof in the analytic case will be a bit more involved due to the lack of partition of unity.
\begin{prop}[Principal symbols for $\G^s$ pseudors]\label{prop:existence_principal_symbol}
Let $A$ be a $\G^s$ pseudor of order $m \in \R$ on $M$. Then, there is a symbol $a \in S^{s,m}\p{T^* M}$ such that, for every small enough coordinate patch $U$, if $p$ denotes the symbol such that the kernel of $A$ has the form \eqref{eq:local_pseudor} on $U$, then, after change of coordinates, we have $ p = a_{|T^* U} \textup{ mod } h S^{s,m-1}$.

We say that $a$ is \emph{a $\G^s$ principal symbol for $A$}.
\end{prop}

Here, it is important to notice that being a $\G^s$ principal symbol for $A$ is a slightly stronger property than being a representative of the principal symbol of $A$ (defined in the $\mathcal{C}^\infty$ sense) of regularity $\G^s$.

Finally, one may wonder if there are any $\G^s$ pseudors on a given compact manifold $M$. It is straightforward to check that differential operators with $\G^s$ coefficients are indeed $\G^s$ pseudors, but we can do better than that. Indeed, we can quantize any $\G^s$ symbol on $M$.

\begin{prop}[Quantization of $\G^s$ symbols]\label{prop:quantization_Gs}
There is a linear map $\Op$ from $\cup_{m \in \R} S^{s,m}$ to the space of continuous operators from $\mathcal{C}^\infty\p{M}$ to $\mathcal{D}'\p{M}$ such that, for every $p \in \cup_{m \in \R} S^{s,m}$, the operator $\Op\p{p}$ is a $\G^s$ pseudor on $M$ such that $p$ is a principal symbol for $P$.
\end{prop}

To close this section, we state some results in the analytic class that will be proved in \S \ref{sec:elliptic-pseudors}. In order to prove Theorem \ref{thm:existence-good-transform}, we will need to construct a parametrix for an elliptic $\G^1$ pseudor. We say that a $\G^1$ pseudor $P$ of order $m$ is (semi-classically) elliptic if there are $c,\epsilon > 0$ and a $\G^1$ principal symbol $p \in S^{1,m}\p{T^* M}$ for $P$ such that for \emph{every} $\alpha \in \p{T^* M}_\epsilon$ we have
\begin{equation}\label{eq:def_elliptique}
\begin{split}
\va{p(\alpha)} \geq c \jap{\va{\alpha}}^m.
\end{split}
\end{equation}
We say that $P$ is classically elliptic if \eqref{eq:def_elliptique} only holds when $\jap{\va{\alpha}}$ is large enough. With this definition, we have the following result, whose proof is given in \S \ref{sssecparametrix}.

\begin{theorem}[Boutet--Krée]\label{thm:parametrix}
Assume that $P$ is a a semi-classically elliptic $\G^1$ pseudor of order $m \in \R$ on $M$. Then, for $h$ small enough, there is a $\G^1$ pseudor $Q$ of order $-m$ such that $PQ = Q P = I$.
\end{theorem}

\begin{remark}
We cannot find precisely this statement in \cite{Boutet-Kree-67}, since this reference only deals with pseudors on open sets of Euclidean spaces and discuss polyhomogeneous rather than semi-classical symbolic calculus. However, the proof that we give of Theorem \ref{thm:parametrix} relies in a fundamental way on Sj\"ostrand's version \cite[Theorem 1.5]{Sjostrand-82-singularite-analytique-microlocale} of the argument of \cite{Boutet-Kree-67}. Notice finally that the parametrix construction from \cite{Boutet-Kree-67} is valid for any $s \geq 1$, but in a $\Gs_x \G^1_\xi$ class. 
\end{remark}

Finally, in \S \ref{subsec:functional_calculus}, we deduce from Theorem \ref{thm:parametrix} the following result about the functional calculus for elliptic semi-classical $\G^1$ pseudors. 

\begin{theorem}\label{theorem:functional_calculus}
Let $P$ be a self-adjoint $\G^1$ pseudor, classically elliptic of order $m > 0$. Assume that the spectrum of $P$ is contained in $\left[\varpi,+ \infty\right[$ for some $\varpi \in \R$ and $h$ small enough. Let $\epsilon > 0$ and $f$ be a holomorphic function on
\begin{equation*}
\begin{split}
U_{\varpi,\epsilon} = \set{z \in \C : \Re z > \varpi - \epsilon \textup{ and } \va{\Im z} < \epsilon \jap{\Re z}}.
\end{split}
\end{equation*}
Assume in addition that there is $N \in \R$ and a constant $C > 0$ such that for every $z \in U_{\varpi,\epsilon}$ we have
\begin{equation*}
\begin{split}
\va{f(z)} \leq C \jap{\va{z}}^N.
\end{split}
\end{equation*}
Then, for $h$ small enough, the operator $f(P)$, defined via the spectral theorem, is a $\G^1$ pseudor of order $mN$. Moreover, if $p$ is a $\G^1$ principal symbol for $P$, then $f(p)$ is a $\G^1$ principal symbol for $f(P)$.
\end{theorem}

\subsection{Gevrey pseudo-differential tricks}\label{sssecKuranishi}

Arguably, microlocal analysis is an application of the study of oscillatory integrals, as exposed in \S \ref{sec:oscillatory-integrals}. However, to apply efficiently the results of that section, many technical tricks are required. They form the core of microlocal tools, and we have gathered (some of) them in this section in the Gevrey setting.

We consider the local model for $\G^s$ pseudors. Hence, let $U$ be an open subset of $\R^n$, let $m \in \R$ and let $a$ a $\G^s$ symbol of order $m$ in $T^* U \times U$. We want to establish the basic properties of the distributional kernel
\begin{equation}\label{eq:pseudo_Rn}
\begin{split}
K_a(x,y) = \frac{1}{\p{2 \pi h}^n} \int_{\R^n} e^{\frac{i\jap{x-y,\xi}}{h}} a(x,\xi,y) \mathrm{d}\xi.
\end{split}
\end{equation}
The integral in \eqref{eq:pseudo_Rn} must be interpreted as an oscillating integral as usual. For technical purposes, we allowed here the symbol $a$ to depend on an additional variable $y$ (that does not appear in \eqref{eq:local_pseudor}), but we will see in Lemma \ref{lemma:calcul_pseudo} that it does not introduce more generality. This will be the main technical result of this section.

Since we required $\G^s$ pseudors to have a negligible kernel away from the diagonal, we need to prove that the local model \eqref{eq:local_pseudor} for their kernels has this property.

\begin{lemma}\label{lemma:loin_de_la_diagonale}
Let $m \in \R$, $s\geq 1$. Let $a$ be a $\G^s$ symbol of order $m$ in $T^* U \times U$. Let $K$ be a compact subset of $U \times U$ that does not intersect the diagonal. Then the kernel $K_a$ defined by \eqref{eq:pseudo_Rn} is $\mathcal{C}^\infty$ on a neighbourhood of $K$, and there are $C,R > 0$ such that
\begin{equation*}
\begin{split}
\n{K_a}_{s,R,K} \leq C \exp\p{- \frac{1}{C h^{\frac{1}{s}}}}.
\end{split}
\end{equation*}
\end{lemma}

\begin{proof}
Notice that we want to prove a local statement in $U \times U$. Hence, we may cut $K$ in small pieces and then assume that $K$ is contained in the interior of $K_1 \times K_2$, where $K_1$ and $K_2$ are two disjoint balls, and that there is $v \in \R^n$, such that for every $\p{x,y} \in K_1 \times K_2$ we have $(x-y)v \geq 1$. A priori, $K_a$ is not a converging integral. However, when $x\neq y$, we can perform a finite number of integration by parts to reduce to that case, and find
\begin{equation}\label{eq:int_oscill}
\begin{split}
K_a(x,y) = \frac{(-1)^N h^{2N}}{\p{2 \pi h}^n |x-y|^{2N}} \int_{\R^n } e^{\frac{i\jap{x-y,\xi}}{h}} \Delta_\xi^N a(x,\xi,y) \mathrm{d}\xi,
\end{split}
\end{equation}
with $N > \frac{m+n}{2}$. The integral is now convergent, and we can drop the prefactor. We denote by $a_N$ the new symbol. We observe that in the $\xi$ variable, the phase is non-stationary; however, we cannot apply Proposition \ref{prop:non-stationary} because the domain of integration is non-compact. We will have to employ a slightly different method. We consider $\delta>0$ and
\[
\Gamma_t := \set{ \xi + i \delta t h^{1-\frac{1}{s}}\langle\xi\rangle^{\frac{1}{s}} v\ |\ \xi\in\R^n}.
\]
Then, we denote by $\tilde{a}_N$ an almost analytic extension for $a_N$ in the sense of Remark \ref{remark:aae_for_symbol} (or the holomorphic extension of $a_N$ if $s = 1$) and obtain 
\begin{equation}\label{eq:change-of-contour-without-BMT}
\begin{split}
\int_{\R^n} e^{\frac{i\jap{x-y,\xi}}{h}}& a_N(x,\xi,y) \mathrm{d}\xi = \\
			&\int_{\Gamma_1} e^{\frac{i\jap{x-y,\xi}}{h}} \tilde{a}_N(x,\xi,y) \mathrm{d}\xi + \int_{\left[0,1\right] \times \R^n} H^*\p{ e^{\frac{i\jap{x-y,\xi}}{h}} \overline{\partial}_\xi \tilde{a}_N(x,\xi,y) \wedge \mathrm{d}\xi},
\end{split}
\end{equation}
where $H$ denotes the homotopy $\p{t,\xi} \mapsto \xi + i \delta t h^{1-1/s}\langle\xi\rangle^{1/s} v$. Along $\Gamma_t$, the imaginary part of the phase is 
\begin{equation}\label{eq:minoration_im}
\Im\langle x-y,\xi\rangle = \delta t h^{1-\frac{1}{s}}\langle \Re \xi\rangle^{\frac{1}{s}} \langle x-y,v\rangle \geq \delta t h^{1-\frac{1}{s}}\langle \Re \xi\rangle^{\frac{1}{s}}.
\end{equation}
Using the estimate on $\overline{\partial}_\xi\tilde{a}_N$ given by Remark \ref{remark:size_d_bar}, we obtain, for some $C > 0$ that may vary from one line to another,
\[
\begin{split}
& \left|\int_{\R^n}\hspace{-5pt} e^{\frac{i\jap{x-y,\xi}}{h}} a_N(x,\xi,y) \mathrm{d}\xi\right| \\ & \qquad \qquad \qquad \leq C \int_{\Gamma_1} \exp\p{- \frac{\Im\langle x-y,\xi\rangle}{h} } \mathrm{d}\va{\xi} \\
	 & \qquad \qquad \qquad \qquad \qquad + C \sup_{t\in[0,1]} \int_{\Gamma_t} \exp\p{- \frac{\Im\langle x-y,\xi\rangle}{h} - \frac{1}{C}\left|\frac{\Im \xi}{\langle\xi\rangle}\right|^{-\frac{1}{s-1}} } \mathrm{d}\va{\xi} 
	\\ & \qquad \qquad \qquad \leq  C \exp\p{-\frac{1}{C h^{\frac{1}{s}}}},
\end{split}
\]
provided that $\delta$ is small enough . Here $\mathrm{d}\va{\xi}$ denotes the total variation of the measure associated to the $n$-form $\mathrm{d}\xi$ restricted to the contour of integration. Now, we have to obtain a similar bound in $\G^s$ instead of $L^\infty$. Actually, this is not much further effort.

It follows from Lemma \ref{lemma:norm-exponentials} and \eqref{eq:minoration_im} that, for some $C,R > 0$, the $\n{\cdot}_{s,R,K_1 \times K_2}$ norm of $(x,y) \mapsto \exp\p{i \jap{x-y,\xi}/h}$ is an $\mathcal{O}(\exp(- C^{-1} (\jap{\va{\xi}}/h)^{1/s}))$ when $\xi \in \Gamma_1$. Thus, differentiating under the integral and applying Leibniz formula, we find that the first term in \eqref{eq:change-of-contour-without-BMT} is an $\mathcal{O}(\exp(-C^{-1} h^{- 1/s}))$ in $\G^s$. 

To deal with the second term in \eqref{eq:change-of-contour-without-BMT}, we use Lemma \ref{lemma:norm-exponentials} again: given $\epsilon > 0$, we can find $R > 0$ such that the $\n{\cdot}_{s,R,K_1 \times K_2}$ norm of $(x,y) \mapsto \exp\p{i \jap{x-y,\xi}/h}$ is an $\mathcal{O}(\exp(\epsilon (\jap{\va{\xi}}/h)^{1/s}))$ for $\xi \in \cup_{t \in \left[0,1\right]} \Gamma_t$. Recalling the bound on the derivatives of the coordinates of $\bar{\partial}_\xi \tilde{a}_N$ given in Remark \ref{remark:aae_for_symbol}, we see that for some $C, R > 0$ the $\n{\cdot}_{s,R, K_1 \times K_2}$ norm of $(x,y) \mapsto \bar{\partial}_\xi \tilde{a}_N (x,\xi,y)$ is an $\mathcal{O}(\exp(- C^{-1} (\jap{\va{\xi}}/h)^{1/s}))$ for $\xi \in \cup_{t \in \left[0,1\right]} \Gamma_t$. Indeed in that domain, we have
\begin{equation}
\left(\frac{|\Im \xi|}{\langle \Re \xi\rangle} \right)^{-\frac{1}{s-1}} = (\delta t)^{-\frac{1}{s-1}} \left(\frac{h^{1-1/s}}{\langle\Re \xi\rangle^{1-1/s}} \right)^{-\frac{1}{s-1}} \geq \delta^{-\frac{1}{s-1}} \left( \frac{\langle \Re \xi\rangle}{h}\right)^{1/s}
\end{equation} 
Since we may assume that $\epsilon < C^{-1}$, we can differentiate under the integral in the second term of \eqref{eq:change-of-contour-without-BMT} and apply the Leibniz formula to find that the second term in \eqref{eq:change-of-contour-without-BMT} is also an $\mathcal{O}(\exp(-C^{-1} h^{- 1/s}))$ in $\G^s$ (for some other $C > 0$).
\end{proof}

In order to control the action of $\G^s$ pseudors acting on $\G^s$ functions, we will use the following lemma.
\begin{lemma}\label{lemma:basic-Gs-boundedness-pseudo}
Let $m \in \R, s \geq 1$ and $a \in S^{s,m}\p{T^* U}$. Let $D,D'$ and $D''$ be balls such that $\overline{D} \subseteq D', \overline{D}' \subseteq D''$ and $\overline{D}'' \subseteq U$. Then, for every $R > 1$ and $\delta > 0$, there exists a constant $C_R>0$ and $R'  > 0$ such that, for all functions $u$ on $D''$ such that $\n{u}_{s,R,D''} < + \infty$, the function
\begin{equation}\label{eq:basic-Gs-boundedness-pseudo}
\begin{split}
x \mapsto \int_{D'} K_a(x,y) u(y) \mathrm{d}y
\end{split}
\end{equation} 
is $\G^s$ on $D$, with $\n{\cdot}_{s,R',D}$ norm less than
\begin{equation*}
\begin{split}
C_R \n{u}_{s,R,D''} \exp\p{ \frac{\delta}{h^{\frac{1}{s}}}}.
\end{split}
\end{equation*}
\end{lemma}

\begin{proof}
Start by noticing that the integral in \eqref{eq:basic-Gs-boundedness-pseudo} is well-defined for $x \in D$ since the singular support of $K_a$ is contained in the diagonal of $U \times U$. We denote this integral by $Au(x)$. Let $\chi : \R^n \to \left[0,1\right]$ be a $\mathcal{C}^\infty$ function supported in $D'$ and identically equals to $1$ on a neighbourhood of $D$. Then, we may write
\begin{equation*}
\begin{split}
Au(x) = \underbrace{\int_{D'} K_a(x,y) \chi(y) u(y) \mathrm{d}y}_{\coloneqq A_1u(x)} + \underbrace{\int_{D'} K_a(x,y) (1- \chi(y))u(y) \mathrm{d}y}_{\coloneqq A_2u(x)}.
\end{split}
\end{equation*}
The term $A_2u(x)$ may be ignored because $K_a$ is $\Gs$-small outside the diagonal (Lemma \ref{lemma:loin_de_la_diagonale}). For this we do not need $u$ to be $\Gs$. Then, notice that, for $N$ large enough, we have
\begin{equation*}
\begin{split}
A_1 u(x) = \frac{1}{\p{2 \pi h}^n} \int_{D' \times \R^n} \frac{e^{i \frac{\jap{x-y,\xi}}{h}}a(x,\xi)}{\p{1 + \xi^2 }^N} \p{ - h^2 \Delta_y + 1}^N \p{\chi(y) u(y)} \mathrm{d}y \mathrm{d}\xi.
\end{split}
\end{equation*}
Then, let $\tilde{u}$ be an almost analytic extension for $u$ given by Lemma \ref{lemma:almost-analytic-extension-Gs} if $s >1$, or the holomorphic extension of $u$ if $s = 1$. We also define $\tilde{\chi}$ on $\C^n$ by $\tilde{\chi}(z) = \chi\p{\Re z}$. Then, we let $\Gamma_\xi$ be the contour
\begin{equation*}
\begin{split}
\Gamma_\xi = \set{ y - i \epsilon \left(\frac{h}{\langle\xi\rangle}\right)^{1 -\frac{1}{s}}\frac{\xi}{\jap{\xi}} : y \in D'},
\end{split}
\end{equation*}
where $\epsilon > 0$ is small. Then by Stokes' Formula we find that
\begin{equation}\label{eq:decoupeA2}
\begin{split}
A_2 u(x) = \underbrace{\frac{1}{(2 \pi h)^n} \int_{\R^n} \p{\int_{\Gamma_\xi} \frac{e^{i \frac{\jap{x-y,\xi}}{h}}a(x,\xi)}{\p{1 + \xi^2}^N} Fu(y) \mathrm{d}y} \mathrm{d}\xi}_{\coloneqq R_0u(x)} + R_1u(x) + R_2u(x). 
\end{split}
\end{equation}
Here, $Fu(y)$ is a shorthand for $(- h^2 \Delta_y + 1)^N \p{\tilde{\chi}(y) \tilde{u}(y)}$ , and we split the error term coming from Stokes' Formula into the integral $R_1$ that implies application of the Cauchy--Riemann operator $\overline{\partial}_y$ to $\tilde{\chi}$ and its derivatives, and the integral $R_2$ that implies application of $\overline{\partial}_y$ to $\tilde{u}$ and its derivatives.

We start by dealing with $R_1u(x)$. Our definition of $\tilde{\chi}$ implies that if $\Re z$ is near $D$ (in particular near $x$) then $\tilde{\chi}$ is locally constant near $z$. Consequently, in the definition of $R_1u(x)$, we only integrate over points $y$ that are uniformly away from $D$ (and hence $x$). Thus, the argument from the proof of Lemma \ref{lemma:loin_de_la_diagonale} applies: we may shift the contour (in $\xi$) in the integral defining $R_1u(x)$ to prove that $R_1u$ decays like $\exp(- C^{-1} h^{-1/s})$ in $\G^s$, and this does not require $\Gs$ regularity on $\tilde{u}$ nor $\tilde{\chi}$.

We deal now with $R_0u(x)$. Notice that if $y \in \Gamma_\xi$ then 
\begin{equation}\label{eq:encore_un_decalage}
\begin{split}
\Im \p{\frac{- y \xi}{h}} = \epsilon \left(\frac{\langle\xi\rangle}{h}\right)^{1/s} \frac{\xi^2}{\jap{\xi}^{2}}.
\end{split}
\end{equation}
Then, we apply Lemma \ref{lemma:norm-exponentials}, and observe that for some $R'>1$, the $\n{\cdot}_{s,R',D}$ norm of $x \mapsto \exp\p{i \langle x, \xi \rangle / h} a(x,\xi)$ grows at most like $\mathcal{O}(\exp( \delta'(\jap{\xi}/h)^{1/s} ))$ for $\delta' > 0$ arbitrarily small. Finally, with \eqref{eq:encore_un_decalage} and differentiation under the integral, we see that $R_0$ is $\G^s$ with the announced estimates.

It remains to deal with $R_2$. This is very similar to the case of $R_0$, but instead of using the positivity of the imaginary part of the phase to prove the decay of the integrand, we apply Lemma \ref{lemma:decay-extension-gevrey} to control the Cauchy--Riemann operator applied to $\tilde{u}$ and its derivatives: see Remark \ref{remark:size_d_bar}. Since we consider points with imaginary part less than $\epsilon (h/\langle\xi\rangle)^{1 - 1/s}$, we get from the application of the Cauchy--Riemann operator to $\tilde{u}$ and its derivatives a decay in $\G^s$ of the integrand of $R_2$ of order:
\begin{equation*}
\begin{split}
\exp\p{- \frac{1}{C} \p{\frac{\jap{\xi}}{h}}^{\frac{1}{s}}}.
\end{split}
\end{equation*}
Then, we may work as in the case of $R_0$ (or in the proof of Lemma \ref{lemma:loin_de_la_diagonale}), differentiating under the integral to prove that $R_2 u$ decays like $\O(\exp(- C/h^{1/s}))$ in $\G^s$, ending the proof of the lemma. Notice that the growth $Au$ is only due to the term $R_0 u$ in \eqref{eq:decoupeA2}, and more precisely to small frequencies in the integral defining $R_0$.
\end{proof}

Another preliminary that will come in handy is the following
\begin{lemma}\label{lemma:Extension-kernel-s=1}
Let $U\subset \R^n$ be open, and $\Omega\subset U$ a relatively compact open set. Let $a$ be a $\G^1$ symbol on $T^\ast U\times U$. Then we can find $C,\delta>0$ and a distribution $\widetilde{K}_a(x,u)$, defined for $x\in \Omega$, $u\in\R^n$, such that $\widetilde{K}_a$ has a holomorphic extension to 
\begin{equation}\label{eq:domaine_extension}
\{(x,u)\in \C^n\ |\ d(x,\Omega)< \delta,\ u\in\C^n\setminus\{0\},\ |\Im u|< \delta |\Re u|\ \},
\end{equation}
with
\begin{equation}\label{eq:extension_noyau_analytique}
\begin{split}
\widetilde{K}_a(x,u) 	&= K_a(x,x+u) + \O_{\G^1}\p{\exp\p{- \frac{1}{Ch}}} \text{ for $|u|<\delta$, $u\in\R^n$}\\
|\widetilde{K}_a(x,u)|	&\leq  C \exp\p{ - \frac{\va{u}}{Ch}}  \text{ for $|u|>\delta/2$.}
\end{split}
\end{equation}
\end{lemma}

\begin{proof}
\begin{figure}[h!]
\centering
\def\svgwidth{0.7\linewidth}
\begingroup%
  \makeatletter%
  \providecommand\color[2][]{%
    \errmessage{(Inkscape) Color is used for the text in Inkscape, but the package 'color.sty' is not loaded}%
    \renewcommand\color[2][]{}%
  }%
  \providecommand\transparent[1]{%
    \errmessage{(Inkscape) Transparency is used (non-zero) for the text in Inkscape, but the package 'transparent.sty' is not loaded}%
    \renewcommand\transparent[1]{}%
  }%
  \providecommand\rotatebox[2]{#2}%
  \newcommand*\fsize{\dimexpr\f@size pt\relax}%
  \newcommand*\lineheight[1]{\fontsize{\fsize}{#1\fsize}\selectfont}%
  \ifx\svgwidth\undefined%
    \setlength{\unitlength}{552.37412865bp}%
    \ifx\svgscale\undefined%
      \relax%
    \else%
      \setlength{\unitlength}{\unitlength * \real{\svgscale}}%
    \fi%
  \else%
    \setlength{\unitlength}{\svgwidth}%
  \fi%
  \global\let\svgwidth\undefined%
  \global\let\svgscale\undefined%
  \makeatother%
  \begin{picture}(1,0.45570804)%
    \lineheight{1}%
    \setlength\tabcolsep{0pt}%
    \put(0,0){\includegraphics[width=\unitlength,page=1]{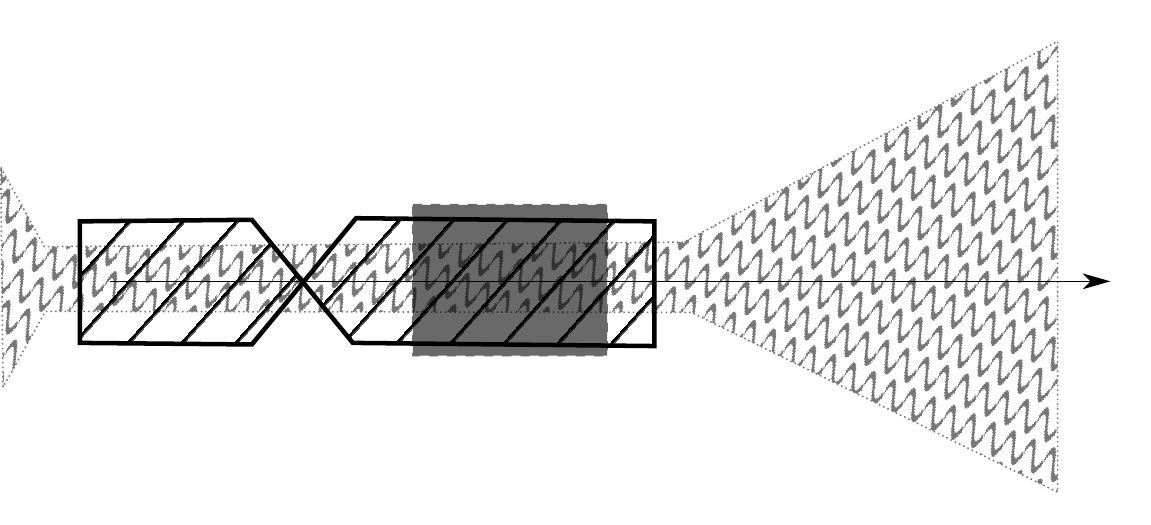}}%
    \put(0.00428189,0.06481696){\color[rgb]{0,0,0}\makebox(0,0)[lt]{\lineheight{1.25}\smash{\begin{tabular}[t]{l}Region of holomorphy of $K_a$\end{tabular}}}}%
    \put(0.35790766,0.00301256){\color[rgb]{0,0,0}\makebox(0,0)[lt]{\lineheight{1.25}\smash{\begin{tabular}[t]{l}Region where ${\chi}$ is not constant\end{tabular}}}}%
    \put(0.39319055,0.43744242){\color[rgb]{0,0,0}\makebox(0,0)[lt]{\lineheight{1.25}\smash{\begin{tabular}[t]{l}Conical neighbourhood $(T^\ast \Omega)_{\epsilon_1}$\end{tabular}}}}%
    \put(0,0){\includegraphics[width=\unitlength,page=2]{elimination-y-extension-kernel.pdf}}%
    \put(0.94578803,0.17008266){\color[rgb]{0,0,0}\makebox(0,0)[lt]{\lineheight{1.25}\smash{\begin{tabular}[t]{l}$\Re u$\end{tabular}}}}%
    \put(0,0){\includegraphics[width=\unitlength,page=3]{elimination-y-extension-kernel.pdf}}%
    \put(0.15806438,0.44195353){\color[rgb]{0,0,0}\makebox(0,0)[lt]{\lineheight{1.25}\smash{\begin{tabular}[t]{l}$\Im u$\end{tabular}}}}%
  \end{picture}%
\endgroup%

\caption{\label{fig:extension-kernel} The complex neighbourhoods involved in the extension of the kernel for $s=1$. }
\end{figure}

For this, we will use the $\overline{\partial}$ inversion trick Lemma \ref{lemma:approximation_analytique}. First of all, it follows from the proof of Lemma \ref{lemma:loin_de_la_diagonale} that, for $\epsilon > 0$ small enough, the distribution $ (x,u) \mapsto K_a(x,x+u)$ has a holomorphic extension to 
\begin{equation*}
\begin{split}
\set{\p{x,u} \in \C^{2n} : d(x,\Omega) < \epsilon, 0 < \va{u} < \epsilon \textup{ and } \va{\Im u} < \epsilon \va{\Re u}}.
\end{split}
\end{equation*} 
Let us choose a $\mathcal{C}^\infty$ function $\chi : \C^{n} \to \left[0,1\right]$ such that $\chi(u) = 0$ if $\va{u} \geq \epsilon/2$ and $\chi(u) = 1$ if $\va{u} \leq \epsilon/4$. Then, we use the bump function $\chi$ to define a form $f$ of type $(1,0)$ by
\begin{equation*}
\begin{split}
f(x,u) = -K_a(x,x+u) \bar{\partial} \chi (u).
\end{split}
\end{equation*}
The $1$-form $f$ is defined for $(x,u) \in \p{T^* \Omega}_{\epsilon_1}$ with $\epsilon_1 \ll \epsilon$ and it follows from Lemma \ref{lemma:loin_de_la_diagonale} that $f(x,u)$ is an $\O(\exp(- \langle|u|\rangle/(Ch)))$. Hence, it follows from Lemma \ref{lemma:approximation_analytique} that there is a smooth function $r : (T^* \Omega)_{\epsilon_2} \to \C$, for some small $\epsilon_2 > 0$, such that $\bar{\partial} r = f$ and $r$ is an $\O(\exp(- \langle\va{u}\rangle/(Ch)))$. We can then define $\widetilde{K}_a$ by
\begin{equation*}
\begin{split}
\widetilde{K}_a(x,u) = \chi(u) K_a(x,x+u)  + r(x,u).
\end{split}
\end{equation*}
For $\delta > 0$ small enough, this makes sense simultaneously as a distribution for $(x,u)$ real in $\Omega\times \R^n$, and as a smooth function on \eqref{eq:domaine_extension}. It follows from Lemma \ref{lemma:loin_de_la_diagonale} and the arguments above that estimates \eqref{eq:extension_noyau_analytique} are satisfied on \eqref{eq:domaine_extension} (provided $\delta<\epsilon/4$ is small enough). It remains to check that $\tilde{K}_a$ is holomorphic on \eqref{eq:domaine_extension}; this follows directly from $\bar{\partial}r=f$.
\end{proof}

After these preliminary steps, we turn to the most basic trick: elimination of variables in symbols. We will see that we can eliminate the dependence on $y$ of $a$ in \eqref{eq:pseudo_Rn}. We will then state the usual consequences of Lemma \ref{lemma:calcul_pseudo} (namely Lemmas \ref{lemma:droite_et_gauche} and \ref{lemma:changement_de_variable}). As the knowledgeable reader will observe, the proof is more subtle than in the $\mathcal{C}^\infty$ case, or even the $\Gs_x \G^1_\xi$ case. For the latter setting, one can find proofs for the $h=1$, $s>1$ case in \cite[Theorem 2.25]{zanghiratiPseudodifferentialOperatorsInfinite1985} (see also the proof of \cite[Theorem 3.2.24]{Rodino-93-book}).

\begin{lemma}\label{lemma:calcul_pseudo}
Let $m \in \R$, $s\geq 1$. Let $a$ be a $\Gs$ symbol of order $m$ in $T^* U \times U$. Let $\Omega$ be a relatively compact open subset of $U$. Then there exists a $\G^s$ symbol $b$ of order $m$ in $T^* \Omega$ such that for every compact subset $K$ of $\Omega \times \Omega$, there are constants $C,R > 0$ such that
\begin{equation}\label{eq:error-removing-variable-y}
\begin{split}
\n{K_a - K_b}_{s,R,K} \leq C \exp\p{ - \frac{1}{Ch^{\frac{1}{s}}}},
\end{split}
\end{equation}
where $K_a$ and $K_b$ are defined by \eqref{eq:pseudo_Rn}.
\end{lemma}

\begin{proof}
The first observation is that for $b \in \mathcal{S}'(\R^{2n})$, 
\[
K_b(x,y) = \frac{1}{(2\pi h)^n} \int_{\R^n} e^{\frac{i}{h}\langle x-y,\xi\rangle} b(x,\xi) \mathrm{d}\xi,
\]
is a partial Fourier transform of $b$. It follows that $b$ may be recovered from the kernel $K_b$ by
\[
b(x,\xi) = \int_{\R^n} e^{\frac{i}{h}\langle u, \xi\rangle} K_b(x,x+u) \mathrm{d}u.
\]
We would like to replace $K_b$ by the kernel $K_a$ in this formula and take this as a definition for $b$, but there are several issues with this formula, and we need consequently to make a few corrections in this expression. 

The first problem is that for $x \in \Omega$, the distribution $K_a(x,x+u)$ is not defined for all $u\in\R^n$, and hence we cannot use this formula directly. When $s>1$, we can introduce a $\G^s$ cutoff $\chi$, equal to $1$ on $[0,\delta/2]$, and supported in $[0,\delta]$ for $\delta>0$ small enough so that the $5\delta$-neighbourhood of $\Omega$ is included in $U$. Then, we replace $a$ by $a_1 =a(x,\xi,y)\chi(|x-y|)$ and set
\[
\widetilde{K}_a(x,u):= K_a(x,x+u) \chi(|u|) = K_{a_1}(x,x+u).
\] 
Without loss of generality, we can assume that $a_1=a$ (the error is negligible considering Lemma \ref{lemma:loin_de_la_diagonale}). When $s=1$, we change the definition of $\widetilde{K}_a(x,u)$, and use instead the kernel given by Lemma \ref{lemma:Extension-kernel-s=1}. In both cases, we let
\begin{equation}\label{eq:def-b_0}
b_0(x,\xi) = \int_{\R^n} e^{\frac{i}{h}\langle u, \xi\rangle} \widetilde{K}_a(x,u) \mathrm{d}u.
\end{equation}
One can see that $\widetilde{K}_a\p{x,y-x}$ is the kernel of a pseudor, and it follows consequently from the usual theory of $\mathcal{C}^\infty$ pseudors that $b_0$ is a symbol in the usual Kohn--Nirenberg classes. Hence, the kernel $K_{b_0}$ is well-defined and from the Fourier Inversion Formula we have $\widetilde{K}_a(x,u) = K_{b_0}\p{x, x+ u}$. Would we know that $b_0 \in S^{s,m}$, we could take $b= b_0$ and the bound \eqref{eq:error-removing-variable-y} would be satisfied. However, it is not obvious that $b_0$ is a $\G^s$ symbol.

In order to bypass this difficulty, we will provide a decomposition:
\begin{equation}\label{eq:decomposition-b_0-symbol+small}
b_0(x,\xi) = \underset{\text{ a $\Gs$ symbol}}{\underbrace{b(x,\xi)}} + \O_{\Gs_{x}}\left(\exp\left(-\frac{ \langle\xi\rangle^{1/s}}{C h^{1/s}}\right)\right).
\end{equation}
Here, the important point is that we are not requiring any regularity beyond $L^\infty$ in the $\xi$ variable for the error. Indeed, it follows from Lemma \ref{lemma:norm-exponentials} that for any $\epsilon > 0$ we have
\begin{equation*}
\begin{split}
e^{i \frac{\jap{x-y,\xi}}{h}} = \O_{\G^s_{x,y}}\p{\exp\p{\epsilon \p{\frac{\jap{\xi}}{h}}^{\frac{1}{s}}}}.
\end{split}
\end{equation*}
Consequently, by taking $\epsilon > 0$ small enough, we have by Leibniz Formula (the value of $C > 0$ may change from one line to another)
\begin{equation*}
\begin{split}
e^{i \frac{\jap{x-y,\xi}}{h}} \times \O_{\Gs_{x}}\left(\exp\left(-\frac{ \langle\xi\rangle^{1/s}}{C h^{1/s}}\right)\right) & = \O_{\G^s_{x,y}}\p{\exp\p{\p{\epsilon - \frac{1}{C}}\p{\frac{\jap{\xi}}{h}}^{\frac{1}{s}}}} \\
     & = \O_{\G^s_{x,y}}\p{\exp\p{ - \frac{1}{C}\p{\frac{\jap{\xi}}{h}}^{\frac{1}{s}}}}.
\end{split}
\end{equation*}
Hence, multiplying \eqref{eq:decomposition-b_0-symbol+small} by $\exp(i \langle x-y,\xi \rangle/h)$ and integrating over $\xi$, we find that
\begin{equation*}
\begin{split}
K_{b_0}(x,y) = K_{b}\p{x,y} + \O_{\G^s_{x,y}}\p{\exp\p{ - \frac{1}{C h^{\frac{1}{s}}}}},
\end{split}
\end{equation*}
and the estimate \eqref{eq:error-removing-variable-y} is then satisfied. Indeed, near the diagonal $K_{b_0}$ and $K_a$ agree (up to an $\O_{\G^1}(\exp( - 1/(Ch)))$ in the case $s=1$) and away from the diagonal both kernels are negligible due to Lemma \ref{lemma:loin_de_la_diagonale}.

Before going into the core of the proof, let us make another reduction. We choose a large integer $N$, then we perform integration by parts in the oscillating integral defining the kernel $K_a$ to write
\begin{equation*}
\begin{split}
K_a = \sum_{j = 1}^{\ell} h^{\va{\alpha_j}} \partial_{y}^{\alpha_j} K_{a_j},
\end{split}
\end{equation*}
where $\ell \in \N$ and, for $j \in \set{1,\dots,\ell}$, we have $a_j \in S^{m - 2N,s}\p{T^* U \times U}$ and $\alpha_j \in \N^n$ satisfies $\va{\alpha_j} \leq 2N$. Hence, we may assume in the following that the order $m$ of $a$ is less than $-n-1$, for if $K_{a_j} - K_{b_j}$ satisfies the bound \eqref{eq:error-removing-variable-y} for $j = 1,\dots,N$, then \eqref{eq:error-removing-variable-y} also holds for $a$ if we set
\begin{equation*}
\begin{split}
b(x,\xi) = \sum_{j = 1}^\ell \p{- i \xi}^{\alpha_j} b_j(x,\xi).
\end{split}
\end{equation*}

\textbf{Analytic case.} We establish \eqref{eq:decomposition-b_0-symbol+small} first in the case $s = 1$. In the case $s=1$, it seems natural to work using holomorphic extension, thus we should rather prove
\begin{equation}\label{eq:decomposition-b_0-symbol+small-alternative}
b_0(x,\xi) = \underset{\text{ a $\G^1$ symbol}}{\underbrace{b(x,\xi)}} + \O_{L^\infty}\left(\exp\left(-\frac{ \langle\xi\rangle}{C h}\right)\right)
\end{equation}
for $\xi \in \R^n$ and $x \in \C^n$ at distance less than $C^{-1}$ of $\Omega$ (for some large $C > 0$). To do so, we will proceed to a succession of contour deformations, in order to decompose $b_0$ into a small remainder, and an oscillatory integral, to which we will be able to apply Proposition \ref{prop:HSP_symbol}. 

We want to get rid of the non-diagonal terms first. For this, we can introduce the contour shift ($\chi$ is the same cutoff function as above)
\begin{equation}\label{eq:Gamma_1}
\Gamma_1 \coloneqq \left\{ u + i \epsilon \p{1-\chi\p{\frac{|u|}{4}}} \frac{\xi}{\langle\xi\rangle}\ \middle|\ u\in \R^n \right\}.
\end{equation}
This deformation is not trivial only for $|u|> 2\delta$, so that deforming from $\R^n$ to $\Gamma_1$, we find
\begin{equation}\label{eq:premiere_expression_b0}
b_0(x,\xi) = \int_{\Gamma_1} e^{\frac{i}{h}\langle u, \xi\rangle} \widetilde{K}_a(x,u) \mathrm{d}u.
\end{equation}
We only deform $\R^n$ away from the singularity $u = 0$ of $\widetilde{K}_a$ and thus there is no error in Stokes' Formula. 

Along $\Gamma_1$, when $|\Re u|> 4\delta$, the imaginary part of the phase is $\epsilon \jap{\xi}^{-1} \va{\xi}^2 $, and the kernel $\widetilde{K}_a$ is $\O(\exp(-1/ C h))$. As a consequence, we can remove the part of the integral corresponding to $|u|> 4\delta $, since this is an $\O(\exp(- C^{-1} \jap{\va{\xi}}/h))$ when $|\Im x| \leq C^{-1}$.

In the remaining region $|u|<4\delta$, we can use the explicit expression for $\widetilde{K}_a$. We must consequently deal with the error term in \eqref{eq:extension_noyau_analytique}. We want to see that a term of the form
\begin{equation}\label{eq:non_stationnaire_contour}
\begin{split}
\int_{\Gamma_1, \va{\Re u} \leq 4 \delta} e^{i \frac{\jap{u,\xi}}{h}} \O_{\G^1_{x,u}}\p{\exp\p{- \frac{1}{Ch}}} \mathrm{d}u
\end{split}
\end{equation}
is negligible. When $\va{\xi}$ is smaller than some given constant, the decay of  the factor $\O_{\G^1_{x,u}}(\exp(- 1/(Ch)))$ is sufficient. When $\xi$ is large, the phase in \eqref{eq:non_stationnaire_contour} is non-stationary, has non-negative imaginary part and is positive on the boundary of the contour of integration. Hence, by a non-stationary phase argument, we find that the term \eqref{eq:non_stationnaire_contour} is in fact an $\O\p{\exp\p{- C^{-1} \frac{\jap{\xi}}{h}}}$. Here, we cannot apply directly Proposition \ref{prop:non_stationary_analytic_symbols} because we integrate on a contour rather than on a compact subset of $\R^n$. However, the same contour shift strategy may be applied, and the reader should have no difficulty completing the argument (adapting for instance the proof of Proposition \ref{prop:non-stationary}).

We may consequently neglect the error term in \eqref{eq:extension_noyau_analytique}. Thus, we may assume that when $\va{u} \leq 4 \delta$ is real the kernel $\widetilde{K}_a(x,u)$ is given by the formula
\begin{equation}\label{eq:en_fait_noscille_pas}
\begin{split}
\widetilde{K}_a(x,u) = \frac{1}{\p{2 \pi h}^n} \int_{\R^n} e^{- i \frac{\jap{u ,\eta}}{h}} a(x,\eta,x+u) \mathrm{d}\eta.
\end{split}
\end{equation}
In general, the integral in the right hand side of \eqref{eq:en_fait_noscille_pas} would be oscillating, but we reduced to the case when the order of $a$ is less than $-n-1$ so that this integral is in fact convergent. Let us give an explicit holomorphic extension for \eqref{eq:en_fait_noscille_pas} when $u$ is away from $0$ before plugging \eqref{eq:en_fait_noscille_pas} into \eqref{eq:premiere_expression_b0}. To do so, we introduce for $u \in \R^n$ such that $\va{u} \leq 4 \delta$ the contour ($\epsilon_1 > 0$ is some small constant)
\begin{equation*}
\begin{split}
\Gamma_2\p{u} \coloneqq \set{ \eta - i \epsilon_1\p{1 - \chi\p{\frac{\va{u}}{2}}} \jap{\eta} u : \eta \in \R^n}.
\end{split}
\end{equation*}
Notice that when $\va{u} \leq 2 \delta$ then $\Gamma_2\p{u}$ coincides with $\R^n$. Then, we may shift contour in \eqref{eq:en_fait_noscille_pas} in order to replace $\R^n$ by $\Gamma_2\p{u}$. We find consequently that for $u \in \C^n$ such that $\va{\Re u} \geq 2 \delta$ and $\va{\Im u} \leq \epsilon_1 \delta^2$, the holomorphic extension of $\widetilde{K}_a$ is given by
\begin{equation}\label{eq:extension_analytique_noyau}
\begin{split}
\widetilde{K}_a(x,u) = \frac{1}{\p{2 \pi h}^n} \int_{\Gamma_2\p{\Re u}} e^{- i \frac{\jap{u ,\eta}}{h}} a(x,\eta,x+u) \mathrm{d}\eta.
\end{split}
\end{equation}
It may not be clear at first sight that the expression given by \eqref{eq:extension_analytique_noyau} is indeed holomorphic (indeed, the dependence of the contour on $u$ is not). However, due to the invariance of the integral under contour shift, we may replace $\Gamma_2\p{\Re u}$ by a locally constant contour when proving holomorphicity. Now, we assume that the constant $\epsilon > 0$ in the definition \eqref{eq:Gamma_1} of $\Gamma_1$ satisfies $\epsilon \ll \epsilon_1$ so that the equality \eqref{eq:extension_analytique_noyau} holds for $u \in \Gamma_1$. Thus, if we form the contour
\begin{equation*}
\begin{split}
\Gamma_3 \coloneqq \set{\p{u, \eta} : u \in \Gamma_1, \eta \in \Gamma_2\p{\Re u}},
\end{split}
\end{equation*}
then $b_0$ is given up to negligible terms by the integral
\begin{equation}\label{eq:deuxieme_expression_b0}
\begin{split}
b_0(x,\xi) \simeq \frac{1}{(2\pi h)^n}\int_{\substack{\p{u,\eta} \in \Gamma_3 \\ \va{\Re u} \leq 4 \delta}} e^{i \frac{\jap{u,\xi - \eta}}{h}} a(x,\eta,x+u) \mathrm{d}\eta \mathrm{d}u.
\end{split}
\end{equation}
We need a further contour deformation, to get rid of the lack of compactness in the $\eta$ variable. To this end, we observe that the phase from \eqref{eq:en_fait_noscille_pas} phase is non-stationary in the $u$ variable when $\eta\neq \xi$. This suggest to consider the contour
\[
\Gamma_4 \coloneqq \left\{ \left(u + i\epsilon_2 \frac{\overline{\xi-\eta}}{\jap{\xi}} ,\ \eta\right)\ \middle|\ \p{u, \eta} \in \Gamma_3 \right\}.
\]
Replacing $\Gamma_3 $ by $\Gamma_4$ in \eqref{eq:deuxieme_expression_b0}, the only error comes from the boundary of the contour of integration in \eqref{eq:deuxieme_expression_b0}. However, we see that for $\p{u,\eta} \in \Gamma_3$ such that $\va{\Re u} = 4 \delta$ the imaginary part of the phase at $\p{u,\eta}$ is given by (provided that $\epsilon \ll \epsilon_1 \ll 1$)
\begin{equation*}
\begin{split}
\Im \p{u,\xi - \eta} & = \epsilon\p{\frac{\jap{\xi}^2}{\jap{\xi}} - \frac{\jap{\xi,\Re \eta}}{\jap{\xi}}} + \epsilon_1 16 \delta^2 \jap{\Re \eta} \\
      & \geq \frac{\p{\jap{\xi} + \jap{\va{\eta}}}}{C}.
\end{split}
\end{equation*}
Since the imaginary part of the phase increases along the homotopy, we see that the boundary term when deforming from $\Gamma_3$ to $\Gamma_4$ is negligible. Hence, we have
\begin{equation}\label{eq:troisime_expression_b0}
\begin{split}
b_0(x,\xi) \simeq \frac{1}{(2\pi h)^n}\int_{\p{u,\eta} \in \Gamma_4, \va{\Re u} \leq 4 \delta} e^{i \frac{\jap{u,\xi - \eta}}{h}} a(x,\eta,x+u) \mathrm{d}\eta \mathrm{d}u.
\end{split}
\end{equation}

When deforming the contour of integration from $\Gamma_3$ to $\Gamma_4$, we increased the imaginary part $\Im \p{u,\eta}$ of the phase by a term $\epsilon_2 \jap{\xi}^{-1} \va{\xi - \eta}^2$. Since the imaginary part of the phase was already non-negative on $\Gamma_3$ (thanks to the assumption $\epsilon \ll \epsilon_1$), we see that the part of the integral in \eqref{eq:troisime_expression_b0} where $\va{\xi - \eta} \geq \delta \jap{\xi}$ is negligible and thus
\[
b_0(x,\xi) \simeq \frac{1}{(2\pi h)^n}\int_{\substack{\p{u,\eta} \in \Gamma_4\\  \va{\Re u} \leq 4 \delta, \va{\xi - \eta} \leq \delta \jap{\xi}}} e^{i \frac{\jap{u,\xi - \eta}}{h}} a(x,\eta,x+u) \mathrm{d}\eta \mathrm{d}u.
\]
The right hand side is now an integral over a bounded set, it is convenient to rescale this integral, writing
\begin{equation}\label{eq:quatrieme_expression_b0}
\begin{split}
b_0(x,\xi) \simeq  \frac{\langle\xi\rangle^n}{(2\pi h)^n} \int_{\Gamma_5} e^{\frac{i\langle\xi\rangle}{h}\langle u,v\rangle} a(x, \xi- \langle\xi\rangle v, x +u )\mathrm{d}u \mathrm{d}v,
\end{split}
\end{equation}
where the contour $\Gamma_5$ is defined by
\begin{equation*}
\begin{split}
\Gamma_5 \coloneqq \set{ \p{u,v} : \p{u, \xi - \jap{\xi}v} \in \Gamma_4, \va{\Re u} \leq 4 \delta \textup{ and} \va{v} \leq \delta}.
\end{split}
\end{equation*}
We want now to approximate the right hand side of \eqref{eq:quatrieme_expression_b0} by a $\G^1$ symbol. It follows from the Holomorphic Stationary Phase Method --- Proposition \ref{prop:HSP_symbol} or \cite[Théorème 2.8]{Sjostrand-82-singularite-analytique-microlocale} --- that there is a formal analytic symbol $\sum_{k \geq 0} h^k c_k$ of order $m$ such that the difference between $\sum_{0 \leq k \leq C^{-1} \jap{\xi}/h} h^k c_k(x,\xi)$, for $C > 0$ large, and the right hand side of \eqref{eq:quatrieme_expression_b0} is negligible. Thus, we can take for $b$ a realization of the formal analytic symbol $\sum_{k \geq 0} h^ k c_k$ (given by Lemma \ref{lemma:existence_realisation}).

\textbf{Gevrey case.} We turn now to the case $s > 1$. In that case, we may directly write
\begin{equation*}
\begin{split}
b_0(x,\xi) = \frac{1}{(2 \pi h)^n} \int_{\R^n \times \R^n} e^{i \frac{\jap{u, \xi - \eta}}{h}} a(x,\eta,x +u) \mathrm{d}\eta \mathrm{d}u.
\end{split}
\end{equation*}
Indeed, $a(x,\eta,x+u)$ vanishes when $\va{u} > \delta$ and the order of $a$ is less than $ - n - 1$. Then, changing variables, we find that
\begin{equation*}
\begin{split}
b_0(x,\xi) = \frac{\jap{\xi}^n}{(2 \pi h)^n} \int_{\R^n \times \R^n} e^{- i \frac{\jap{\xi}}{h}\jap{u,v}} a(x,\xi + \jap{\xi} v,x+u) \mathrm{d}u \mathrm{d}v.
\end{split}
\end{equation*} 
Using the cutoff functions $\chi$ again, we can write
\begin{equation}\label{eq:decomposition_b0_gs}
\begin{split}
b_0(x,\xi) = b(x,\xi) + \frac{\jap{\xi}^n}{(2 \pi h)^n} \int_{\R^n \times \R^n}\hspace{-20pt} e^{- i \frac{\jap{\xi}}{h}\jap{u,v}} \p{1 - \chi(\va{v})}a(x,\xi + \jap{\xi} v,x+u) \mathrm{d}u \mathrm{d}v,
\end{split}
\end{equation}
where 
\begin{equation}\label{eq:def_b_gs}
\begin{split}
b(x,\xi) = \frac{\jap{\xi}^n}{(2 \pi h)^n} \int_{\R^n \times \R^n} e^{- i \frac{\jap{\xi}}{h}\jap{u,v}} \chi(\va{v})a(x,\xi + \jap{\xi} v,x+u) \mathrm{d}u \mathrm{d}v
\end{split}
\end{equation}
is a symbol in $S^{s,m}\p{T^* \Omega}$ according to Lemma \ref{lemma:stationary_gevrey_real_symbol}. Here, it is very important that $\delta > 0$ is small enough so that the integrand in \eqref{eq:def_b_gs} is uniformly $\G^s$ in $v$. Indeed, without the cutoff functions the estimates on the derivatives would deteriorate for $v \simeq - \xi/\jap{\xi}$. 

Now, we need to prove that the remainder term in \eqref{eq:decomposition_b0_gs} is as small as announced in \eqref{eq:decomposition-b_0-symbol+small}. To do so, we will apply the non-stationary phase method to the integral
\begin{equation}\label{eq:non_stationnaire_v}
\begin{split}
\int_{\R^n} e^{- i \frac{\jap{\xi} \jap{v}}{h}\jap{u,\frac{v}{\jap{v}}}} a(x,\xi + \jap{\xi} v,x+u) \mathrm{d}u.
\end{split}
\end{equation}
Here, $v$ and $\xi$ are dealt with merely as parameters, so that the lack of regularity in these variable will not matter. However, the integrand is uniformly $\G^s$ in $x$. When $\va{v} > \delta / 2$, the phase in \eqref{eq:non_stationnaire_v} is non-stationary, so that we can apply Proposition \ref{prop:non-stationary} (with a rescaling arguments as in \S \ref{subsec:further_tricks}) to find that \eqref{eq:non_stationnaire_v} is an
\begin{equation*}
\begin{split}
\O_{\G_x^s}\p{\exp\p{- \p{\frac{\jap{\xi} \jap{v}}{C h}}^{\frac{1}{s}}}}.
\end{split}
\end{equation*}
Thus, \eqref{eq:decomposition-b_0-symbol+small} follows from \eqref{eq:decomposition_b0_gs} after integration over $v$.
\end{proof}

\begin{remark}\label{remark:developpement_asymptotique}
Since the proof of Lemma \ref{lemma:calcul_pseudo} is based on an application of the Stationary Phase Method (Proposition \ref{prop:HSP_symbol} or Lemma \ref{lemma:stationary_gevrey_real_symbol}), we see that the symbol $b$ has the usual asymptotic expansion (recall Remark \ref{remark:developpement_symbole_gs} and Remark \ref{remark:coefficients_symboles})
\begin{equation*}
\begin{split}
b(x,\xi) \sim \sum_{\alpha \in \N^n} h^{\va{\alpha}}\frac{\partial_\xi^\alpha \partial_y^\alpha a(x,\xi,x)}{i^{\va{\alpha}}\alpha!}.
\end{split}
\end{equation*}
When $s = 1$, this expansion holds in the sense of realization of formal analytic symbol (see Definition \ref{def:realisation_Gevrey}). When $s > 1$, the estimate is not as good, but we still know that, for every $N \in \N$, we have
\begin{equation*}
\begin{split}
b(x,\xi) = \sum_{\va{\alpha} < N} h^{\va{\alpha}} \frac{\partial_\xi^\alpha \partial_y^\alpha a(x,\xi,x)}{i^{\va{\alpha}}\alpha!} \textup{ mod } h^N S^{m-N,s}.
\end{split}
\end{equation*}
Consequently, we will have the usual asymptotic expansion for the symbol of the adjoint, the composition, etc, since the proofs that they are pseudors all rely on applications of Lemma \ref{lemma:calcul_pseudo}.
\end{remark}

With essentially the same proof, we could also reduce the kernels to the form
\[
\int_{\R^n} e^{i\frac{\langle x-y,\xi\rangle}{h}} b( tx + (1-t) y, \xi) \mathrm{d}\xi + \O_{\G^s}\p{\exp\p{\frac{1}{C h^{\frac{1}{s}}}}},
\]
for $t\in [0,1]$. Here, we will only need:
\begin{lemma}\label{lemma:droite_et_gauche}
Let $a \in S^{s,m}\p{T^* U}$ and $\Omega$ be a relatively compact subset of $U$. Then there is $b \in S^{s,m}\p{T^* \Omega}$ such that for $x,y \in \Omega$ we have
\begin{equation*}
\begin{split}
K_a(x,y) = \frac{1}{\p{2 \pi h}^n} \int_{\R^n} e^{i \frac{\jap{x-y,\xi}}{h}} b(y,\xi) \mathrm{d}\xi + \O_{\G^s}\p{\exp\p{- \frac{C}{h^{\frac{1}{s}}}}}.
\end{split}
\end{equation*}
\end{lemma}

\begin{proof}
Apply Lemma \ref{lemma:calcul_pseudo} to the kernel
\begin{equation*}
\begin{split}
\overline{K_a(y,x)} = \frac{1}{(2 \pi h)^n} \int_{\R^n} e^{i \frac{\jap{x-y,\xi}}{h}} \bar{a}(y,\xi) \mathrm{d}\xi
\end{split}
\end{equation*}
to find $b \in S^{s,m}\p{T^* \Omega}$ such that 
\begin{equation*}
\begin{split}
\overline{K_a(y,x)} = K_b(x,y) + \O_{\G^s}\p{\exp\p{- \frac{1}{C h^{\frac{1}{s}}}}}.
\end{split}
\end{equation*}
The result follows then by switching $x$ and $y$ back and taking complex conjugate.
\end{proof}

Using the Kuranishi trick, we deduce from Lemma \ref{lemma:calcul_pseudo} the following result that justifies Definition \ref{def:gevrey_pseudor}. This is a very classical procedure, see for instance \cite[Theorem 9.3]{zworskiSemiclassicalAnalysis2012}. In particular, the proof of Lemma \ref{lemma:changement_de_variable} does not rely on the analytic version of the Kuranishi trick Lemma \ref{lmkurtrick} that we will discuss soon (there is no complex phase or contour shift here).

\begin{lemma}\label{lemma:changement_de_variable}
Let $V$ be another open subset of $\R^n$ and let $\kappa : U \to V$ be a $\G^s$ diffeomorphism. Let $a \in S^{s,m}\p{T^* U}$ and $\Omega$ be a relatively compact subset of $U$. Then there is a symbol $b \in S^{s,m}\p{T^* \kappa\p{\Omega}}$ such that 
\begin{equation*}
\begin{split}
K_a \p{\kappa^{-1}(x),\kappa^{-1}(y)} J \kappa^{-1}(y) = K_b(x,y) + \O_{\G^s}\p{\exp\p{- \frac{1}{C h^{\frac{1}{s}}}}}
\end{split}
\end{equation*}
on $\kappa\p{\Omega}$. ($J \kappa^{-1}$ is the Jacobian of $\kappa^{-1}$).
\end{lemma}

For several reasons, we may encounter operators whose kernel is an oscillatory integral, with a phase that is not exactly the phase $\langle x-y,\xi\rangle$, but close to it. It is then crucial to be able to recognize the kernel of a pseudor, and this is what the Kuranishi trick is for. Let us be more precise and define the class of phases that we will use.

\begin{definition}\label{def:nonstandardphase}
Let $M$ be either a compact real-analytic manifold or an open subset of $\R^n$. A \emph{phase} on $M$ is a holomorphic symbol $\Phi$ of order $1$ defined for $(\alpha,y) \in (T^* M)_\epsilon \times (M)_\epsilon$ with $d(\alpha_x,y) < \delta$ (for some $\epsilon,\delta > 0$) such that
\begin{enumerate}[label = (\roman*)]
\item if $\p{\alpha,y} \in T^* M \times M$ then the imaginary part of $\Phi(\alpha,y)$ is non-negative;
\item $\Phi(\alpha,\alpha_x) = 0$ for $\alpha = (\alpha_x,\alpha_\xi) \in T^* M$;
\item for $\alpha \in T^* M$, we have $\mathrm{d}_y \Phi(\alpha,\alpha_x) = - \alpha_\xi$.
\end{enumerate}
We say that $\Phi$ is an \emph{admissible phase} if it satisfies in addition the coercivity condition:
\begin{enumerate}[label = (\roman*)]
\setcounter{enumi}{3}
\item there is a constant $C > 0$ such that if $\alpha,y$ are real and $\Phi(\alpha,y)$ is defined, then $\Im \Phi(\alpha,y) \geq C^{-1} \jap{\alpha} d\p{\alpha_x,y}^2$.
\end{enumerate}
\end{definition}
The point (ii) and (iii) in Definition \ref{def:nonstandardphase} may be stated loosely as
\begin{equation}\label{eq:Behaviour-Taylor-Phase-admissible}
\begin{split}
\Phi(\alpha_x, \alpha_\xi ,y) = \jap{\alpha_x-y,\alpha_\xi} + \O\p{\jap{\va{\alpha_\xi}} \va{\alpha_x-y}^2}.
\end{split}
\end{equation}
Notice that, with this definition, the standard phase $\Phi(\alpha,y) = \jap{\alpha_x-y,\alpha_\xi}$ is not an admissible phase: it does not satisfy the condition (iv). In some sense, (iv) is satisfied by the standard phase after a contour shift (this is basically the argument in the proof of Lemma \ref{lemma:loin_de_la_diagonale}). Let us also mention that from any $\G^1$ metric $g$ on $M$, we may build an admissible phase, by setting
\begin{equation*}
\Phi(\alpha,y) = -\alpha_\xi(\exp_{\alpha_x}^{-1}(y)) + \langle |\alpha_\xi|_{\alpha_x}\rangle d_g(\alpha_x,y)^2.
\end{equation*}

If $\Phi$ is a phase in the sense of Definition \ref{def:nonstandardphase} on some open subset $U$ of $\R^n$ and $a \in S^{1,m}\p{T^* U \times U}$ is an analytic symbol, we may see using standard integration by parts that the oscillating integral
\begin{equation}\label{eq:noyau_non_standard}
\begin{split}
K_{\Phi,a}(x,y) = \frac{1}{\p{2 \pi h}^n} \int_{\R^n} e^{i \frac{\Phi(x,\xi,y)}{h}} a(x,\xi,y) \mathrm{d}\xi
\end{split}
\end{equation}
defines a distribution near the diagonal of $U \times U$. Adapting the proof of Lemma \ref{lemma:loin_de_la_diagonale}, one can see that if $x,y \in U$ are close enough to the diagonal -- so that $K_{\Phi,a}$ is defined near $(x,y)$ -- but not \emph{on} the diagonal, then $K_{\Phi,a}$ is an $\O(\exp(-1/(Ch)))$ in $\G^1$ near $(x,y)$.

We will say that an operator whose kernel (modulo a small $\G^1$ error) takes the form $K_{\Phi,a}$, is a \emph{pseudor with non-standard phase}. There are two reasons for which we want to study such operators. The first one is that an analytic pseudor with non-standard phase appears in the proof of Theorem \ref{thm:existence-good-transform}. The other one is that we will use pseudors with non-standard phases to prove Proposition \ref{prop:quantization_Gs}. The main result about those operators is the following:

\begin{lemma}[Kuranishi trick]\label{lmkurtrick}
Let $U$ be a bounded open subset of $\R^n$. Let $\Phi_1$ and $\Phi_2$ be phases on $U$ in the sense of Definition \ref{def:nonstandardphase}. Let $m \in \R$ and $a_1 \in S^{1,m}\p{T^* U \times U}$. Let $\Omega$ be a relatively compact open subset of $U$. Then there is $a_2 \in S^{1,m}\p{T^* \Omega \times \Omega}$ such that
\begin{equation*}
\begin{split}
K_{\Phi_1,a_1}(x,y) = K_{\Phi_2, a_2}(x,y)
\end{split}
\end{equation*}
near the diagonal of $\Omega \times \Omega$.
\end{lemma}

\begin{proof}[Proof of Lemma \ref{lmkurtrick}]
Instead of considering the general case, it suffices to consider the case when one of the phases $\Phi_{1}$ and $\Phi_2$ is the standard phase $\langle x-y,\xi\rangle$. Thus, we pick some phase $\Phi$ as in Definition \ref{def:nonstandardphase}. According to Taylor's formula (with integral remainder) we have
\begin{equation}\label{eq:factorisation_phase}
\begin{split}
\Phi(x,\xi,y) = \jap{x-y , \theta_{x,y}(\xi)},
\end{split}
\end{equation}
where 
\begin{equation*}
\begin{split}
\theta_{x,y}(\xi) & = \xi - \int_0^1 (1-t) D^2_{y,y} \Phi(x,\xi, x + t(y-x)) (y-x) \mathrm{d}t \\ & = \xi + \O\p{\jap{\va{\xi}} \va{x-y}}.
\end{split}
\end{equation*}
From this, provided $x$ and $y$ are close enough, one may show that there are $\epsilon_1,\epsilon_2 > 0$ such that $\theta_{x,y}$ is injective on $\p{\R^n}_{\epsilon_1}$ and has a right inverse $\Upsilon_{x,y} : \p{\R^n}_{\epsilon_2} \to \p{\R^n}_{\epsilon_1}$ (here, we use the Kohn--Nirenberg metric to define the Grauert tubes of $\R^n$). We may then use $\Upsilon_{x,y}$ as a change of variable to write
\begin{equation}\label{eqapreskur}
K_{\Phi,a}(x,y) = \frac{1}{\p{2 \pi h}^n} \int_{\theta_{x,y}\p{\R^n}} e^{i \frac{\jap{x-y,\xi}}{h}} a(x,\Upsilon_{x,y}\p{\xi},y) J \Upsilon_{x,y}(\xi) \mathrm{d}\xi.
\end{equation}
This is to be understood as an oscillatory integral, which can be put in convergent form using integration by parts. Then, if $x$ and $y$ are close enough, $\theta_{x,y}(\R^n)$ remains uniformly transverse to $i\R^n$, and may thus be written as a graph over $\R^n$. We may thus shift the contour in \eqref{eqapreskur} to replace it by $\R^n$, using for instance the homotopy $[0,1] \times \theta_{x,y}(\R^n) \ni (t,\xi) \mapsto \Re \xi + i t \Im \xi$. We obtain
\begin{equation}\label{eq:deforming-contour-Ktrick}
\begin{split}
K_{\Phi,a}(x,y) &= \frac{1}{\p{2 \pi h}^n} \int_{\R^n } e^{i \frac{\jap{x-y,\xi}}{h}} a(x,\Upsilon_{x,y}\p{\xi},y) J \Upsilon_{x,y}(\xi) \mathrm{d}\xi \\
                  & = K_{\Phi_0, b}(x,y),
\end{split}
\end{equation}
where $\Phi_0(x,\xi,y) = \jap{x-y,\xi}$ denotes the standard phase and the symbol $b$ is defined by $b(x,\xi,y) = a(x,\Upsilon_{x,y}\p{\xi},y) J \Upsilon_{x,y}(\xi)$. We end the proof by noticing that we can reverse the argument since $a$ is retrieved from $b$ using the formula $a(x,\xi,y) = b(x,\theta_{x,y}\p{\xi},y) J \theta_{x,y}\p{\xi}$.
\end{proof}

In order to prove the existence of an analytic principal symbol for a $\G^1$ pseudor (Proposition \ref{prop:existence_principal_symbol}), it will be useful to understand the action of a pseudor on an analytic family of coherent states. This is closely related to the FBI transform that will be the subject of the next chapter. 
\begin{definition}[Coherent states]\label{def:coherent-state}
A $\G^1$ \emph{family of coherent states} of order $\ell \in \R$ on $M$ is a holomorphic function $(\alpha,x) \mapsto u_\alpha(x)$ on $\p{T^* M}_\epsilon \times (M)_\epsilon$ (for some $\epsilon > 0$) such that there exist $\eta,C > 0$ such that for $\alpha = (\alpha_x,\alpha_\xi) \in (T^* M)_\epsilon$ and $x \in (M)_\epsilon$ we have:
\begin{enumerate}[label=(\roman*)]
\item if $d(\alpha_x,x) \geq \eta$ then
\begin{equation*}
\begin{split}
\va{u_\alpha(x)} \leq C \exp\p{ - \frac{\jap{\va{\alpha}}}{Ch}};
\end{split}
\end{equation*}
\item if $d(\alpha_x,x) \leq 2 \eta$ then
\begin{equation*}
\begin{split}
u_\alpha(x) = e^{i \frac{\Phi(\alpha,x)}{h}} a(\alpha,x) + \O\p{\exp\p{- \frac{\jap{\va{\alpha}}}{C h}}},
\end{split}
\end{equation*}
where $a$ and $\Phi$ are analytic symbols respectively in $S^{1,\ell}\p{T^*M \times M}$ and in $S^{1,1}\p{T^* M \times M}$ (maybe defined only for $(\alpha_x,x)$ near the diagonal) that satisfy:
\item $\Phi\p{\alpha,\alpha_x} = 0$;
\item the derivative $\mathrm{d}_x \Phi(\alpha,\alpha_x)$ is elliptic of order $1$ in the sense that
\begin{equation*}
\begin{split}
\va{\mathrm{d}_x \Phi(\alpha,\alpha_x)} \geq  \frac{\jap{\va{\alpha}}}{C}
\end{split}
\end{equation*}
when $\jap{\va{\alpha}} \geq C$;
\item if $\alpha$ and $x$ are real then we have,
\begin{equation*}
\begin{split}
\Im \Phi(\alpha,x) \geq \frac{d\p{\alpha_x,x}^2}{C}.
\end{split}
\end{equation*}
We say that $\Phi$ and $a$ are respectively the phase and the amplitude of $(u_\alpha)$.
\end{enumerate}
\end{definition} 

We can now describe the action of $\G^1$ pseudors on coherent states. Notice that when $P$ is a differential operator, Proposition \ref{prop:pseudor_lagrangien_alt} is just a consequence of the Leibniz formula.
\begin{prop}\label{prop:pseudor_lagrangien_alt}
Let $P$ be a $\G^1$ pseudor of order $m \in \R$ and $u_\alpha$ be a $\G^1$ coherent state of order $\ell \in \R$ on $M$. Then
\begin{equation*}
\begin{split}
(Pu)_\alpha(x):= (P (u_\alpha))(x)
\end{split}
\end{equation*}
defines a $\G^1$ family of coherent states of order $m + \ell$ on $M$ with the same phase as $u_\alpha$. We denote by $b$ the amplitude of $Pu_\alpha$. If the kernel of $P$ is described in a coordinate patch $U$ by \eqref{eq:local_pseudor}, with $p \in S^{1,m}\p{T^* U}$, then $b$ is given at first order on $U$ by
\begin{equation}\label{eq:first-order-expansion-Pu_alpha}
b(\alpha,x) = p\p{x,\mathrm{d}_x \Phi\p{\alpha,x}} a(\alpha,x) \textup{ mod } h S^{1,m+\ell - 1}.
\end{equation}
\end{prop}

\begin{proof}
We take $\epsilon > 0$ very small and want to understand $Pu_\alpha(x)$ for $\p{\alpha, x} \in \p{T^* M}_\epsilon \times \p{M}_\epsilon$. Choose $\eta > 0$ small (with $\epsilon \ll \eta$). We start by investigating $Pu_\alpha(x)$ when the distance between $x$ and $\alpha_x$ is larger than $10^{-2} \eta$. We start by assuming that $x \in M$ (i.e is real). In that case, we may choose balls $D_1$ and $D_2$ of radius $10^{-3} \eta$ that do not intersect, with $x\in D_1$ and $\Re \alpha_x\in D_2$. We may also assume that $x$ and $\Re \alpha_x$ remain at distance at least $10^{-4} \eta$ from the boundaries of $D_1$ and $D_2$ (we only need a finite number of such disks to cover all cases). Then write
\begin{equation}\label{eq:decompo_tildek_loin_diag_alt}
\begin{split}
Pu_\alpha(x) = \p{\int_{D_1} + \int_{D_2} + \int_{M \setminus \p{D_1 \cup D_2}}} K_P(x,y) u_\alpha(y) \mathrm{d}y,
\end{split}
\end{equation}
where $K_P(x,y)$ denotes the kernel of $P$. Notice that all the terms in \eqref{eq:decompo_tildek_loin_diag_alt} make sense since the singular support of $K_P$ is contained in the diagonal of $M \times M$. The integral over $M \setminus \p{D_1 \cup D_2}$ is easily dealt with: by assumption, the kernel of $P$ is an $\O(\exp(-C/h))$ and $u_\alpha$ is an $\O(\exp(- C\langle|\alpha|\rangle/h))$ in $\G^1$ there. To tackle the integral over $D_1$, just notice that $u_\alpha$ is an $\O(\exp(-C\langle|\alpha|\rangle/h))$ in $\G^1$ there. Hence, we may apply Lemma \ref{lemma:basic-Gs-boundedness-pseudo} and find that the integral over $D_1$ is an $\O(\exp(- C\langle|\alpha|\rangle/h))$ in $\G^1$. It remains to deal with the integral over $D_2$. Notice that, up to smooth and small terms that we may neglect, the integral over $D_2$ writes
\begin{equation}\label{eq:une_integrale_qui_est_petite}
\begin{split}
\int_{D_2} e^{i \frac{\Phi(\alpha,y)}{h}} a(\alpha,y) K_P(x,y) \mathrm{d}y.
\end{split}
\end{equation}
Here, since $x$ and $y$ are uniformly away from each other the kernel of $K_P$ is an $\O(\exp(- C/h))$ in $\G^1$. For small $\alpha_\xi$, if $\alpha$ is sufficiently close to the reals, $\Im \Phi$ cannot be too negative, so this is sufficient to find a $\O(\exp(- C/h))$ bound in $\G^1$ for the integral. Thanks to the assumption of ellipticity on the derivative of $\Phi$, when $\alpha_\xi$ is large enough the phase is non-stationary in \eqref{eq:une_integrale_qui_est_petite}. Hence, the integral \eqref{eq:une_integrale_qui_est_petite} is an $\O(\exp(- C\langle|\alpha|\rangle/h))$ as an application of Proposition \ref{prop:non_stationary_analytic_symbols}. Hence, we proved that $P u_\alpha(x)$ is negligible when $x$ and $\alpha_x$ are away from each other and $x$ is real. However, we proved an estimate in $\G^1$ so that this estimate remains true when $x \in (M)_\epsilon$ with $\epsilon \ll \eta$.

We want now to understand the kernel $Pu_\alpha(x)$ when $d(x,\Re \alpha_x) < \eta$ -- as above, we may assume  $x$ real and then prove a bound in $\G^1$. To do so, we may assume that $x$ and $\Re \alpha_x$ are contained in a ball $D$ of radius $10 \eta$ and then denote by $D'$ the ball of radius $100 \eta$ with the same center as $D$ (we only need a finite number of such $D$ to cover the whole manifold). Then, write
\begin{equation*}
\begin{split}
Pu_\alpha(x) = \p{\int_{D'} + \int_{M \setminus D'}} K_P(x,y) u_\alpha(y) \mathrm{d}y
\end{split}
\end{equation*}
We obtain a $\O(\exp(-C\langle|\alpha|\rangle/h))$ bound in $\G^1$ for the integral over $M \setminus D'$ as in the previous case (both kernels are smooth and small there). By taking $\eta$ small enough, we may assume that $D'$ is contained in a coordinate patch and work in these coordinates. The arguments from the first part of the proof apply again to show that the contribution to the integral over $D'$ from the error terms in \eqref{eq:local_pseudor} and in the definition of $u_\alpha$ are negligible. Hence, we want to study the kernel given by the oscillating integral
\begin{equation*}
\begin{split}
Pu_\alpha(x)^0 = \frac{1}{(2 \pi h)^n} \int_{\R^n \times D'} e^{i \frac{\jap{x-y,\xi} + \Phi(\alpha,y)}{h}} p(x,\xi) a\p{\alpha,y} \mathrm{d}\xi \mathrm{d}y.
\end{split}
\end{equation*}
The oscillating integral here is well-defined because the singular support of a pseudo-differential operator is contained in the diagonal of $M \times M$. In order to manipulate this integral, we will approximate it by converging integrals. Indeed, we have
\begin{equation}\label{eq:approsimation_hatk_alt}
Pu_\alpha(x)^0 = \lim_{\delta \to 0} \underbrace{\frac{1}{(2 \pi h)^n} \int_{\R^n \times D'} e^{i\frac{\jap{x-y,\xi} + \Phi(\alpha,y) }{h}}  e^{- \delta \va{\xi}^2} p(x,\xi) a\p{\alpha,y} \mathrm{d}\xi \mathrm{d}y}_{\coloneqq k_\delta(x,\alpha)}.
\end{equation}
We would like to deal with \eqref{eq:approsimation_hatk_alt} using a stationary phase method. There are two obstacles that make this computation slightly subtle. The first one is to remove the regularization, which will be done by a contour deformation for large $\langle \xi\rangle$. The second is that when $\alpha_\xi$ grows large, the contributions from $|\xi| \ll |\alpha_\xi|$, must be $\mathcal{O}(\exp(-\langle\alpha\rangle/Ch))$. This low frequency estimate will crucially depend on the ellipticity of $d_x \Phi$. 

We will start by considering the case that $\alpha_\xi$ remains in a bounded set. Let us set
\[
\Psi_{x,\alpha}(y,\xi) =  \langle x-y,\xi\rangle + \Phi(\alpha,y).
\]
Let us introduce then for $\nu>0$ the contour
\begin{equation*}
\begin{split}
\Gamma_\nu^0 = \set{ \p{\xi + i \nu \langle\xi\rangle \overline{(x-y)}, y + i \chi(y) \nu \frac{\overline{\nabla_y \Psi_{x,\alpha}(\xi,y)}}{\jap{\va{\nabla_y \Psi_{x,\alpha}(\xi,y)}}}} : \xi \in \R^n,y \in D'},
\end{split}
\end{equation*}
where the bump function $\chi : \R^n \to \left[0,1\right]$ takes value $1$ on a neighbourhood of $D$ and is supported in the interior of $D'$. We observe that $\Gamma_\nu^0$ remains in the region where the amplitude in \eqref{eq:approsimation_hatk_alt} satisfies uniform symbolic estimates. If $\xi,y \in \R^n$ and $\xi$ is large enough then
\begin{equation*}
\begin{split}
\va{\nabla_y \Psi_{x,\alpha}\p{\xi, y}} \geq  \frac{\jap{\va{\xi}}}{C}.
\end{split}
\end{equation*}
Hence, we find by an application of Taylor's formula that for $\p{y,\xi} \in \Gamma_\nu^0$ with $\xi$ large we have (the constant $C > 0$ may vary from one line to another)
\begin{equation}\label{eq:on_fait_converger_tout_ca}
\begin{split}
\Im (&\Psi_{x,\alpha}(\xi,y)) \\
	&= \Im\p{\Psi_{x,\alpha}(\Re \xi,\Re y)} + \nu \langle \Re \xi\rangle \va{x - \Re y}^2 \\ 
	& \qquad + \nu \chi\p{\Re y} \frac{\va{\nabla_y \Psi_{x,\alpha}\p{\Re \xi, \Re y}}^2}{\jap{|\nabla_y \Psi_{x,\alpha}\p{\Re \xi, \Re y}|}} + \O\p{\nu^2 \chi(\Re y) \langle\xi\rangle} \\
   	&\geq \frac{\nu}{C} \p{\va{x - \Re y}^2 + \chi\p{\Re y}} \jap{\va{\xi}} - C \nu^2 \jap{\va{\xi}} - C \epsilon \\
   	&\geq \p{\nu \p{\frac{1}{C} - C \nu} - C \epsilon} \jap{\va{\xi}} \geq \frac{\nu \jap{\va{\xi}}}{C}. 
\end{split}
\end{equation}
Here, we took $x$ real, but allowed $\alpha$ to have imaginary part of size $\mathcal{O}(\epsilon)$. We used that if $\chi(\Re y) \neq 1$ then $\va{x - \Re y}$ is uniformly bounded from below (because $x$ is in $D$) and we assumed that $\nu$ was small enough and that $\epsilon \ll \nu$. Beware that \eqref{eq:on_fait_converger_tout_ca} is only valid for large $\xi$.

Accordingly, we can get rid of the approximation in \eqref{eq:approsimation_hatk_alt}. For this, first we can change the contour of integration from $\R^n\times D'$ to $\Gamma_\nu^0$ for $\nu>0$ small enough; there is no error term because we are in the $\G^1$ setting, and near the boundary of $\R^n_\xi \times D'_y$, we are only changing contour in the $\xi$ variable. Then we apply the Dominated Convergence Theorem and let $\delta\to 0$. Next, the integral over $\Gamma_\nu^0$ can be split into two parts. The piece corresponding to $|\xi|> C$ for some $C>0$ large enough will be $\mathcal{O}(\exp(-C/h))$, and this can be upgraded to a bound in $\G^1$ using Lemma \ref{lemma:norm-exponentials}. The resulting estimate is
\[\begin{split}
Pu_\alpha(x)^0 = \frac{1}{(2\pi h)^n} \int_{\substack{(y,\xi)\in\Gamma_\nu^0,\\ |\xi|\leq C}} e^{\frac{i}{h} \Psi_{x,\alpha} (\xi,y)} p(x, \xi) a(\alpha,y) \mathrm{d}\xi \mathrm{d}y + \mathcal{O}_{\mathcal{G}^1}\left(\exp\left(-\frac{1}{Ch} \right)\right).
\end{split}
\]
For the integral in the right hand side, we observe that the phase has positive imaginary part near the boundary of the domain, and a single critical point at $y=x$, $\xi= d_x \Phi(\alpha,x)$. At the critical point, $\Psi_{\alpha_x,x}$ vanishes and the Hessian is 
\[
\begin{pmatrix}
D_{x,x}^2 \Phi(\alpha,\alpha_x) & -1 \\ - 1 & 0
\end{pmatrix},
\]
which is invertible with determinant $(-1)^n$. To apply Proposition \ref{prop:HSP}, it remains to check that the imaginary part of the phase is non-negative on the whole contour for some value of the parameters. In the region where $\chi(\Re y)\neq 1$ this imaginary part is already strictly positive before deformation thanks to item (v) of Definition \ref{def:coherent-state}. In the region where $\chi(\Re y)=1$, a Taylor expansion argument shows that the imaginary part of $\Psi_{x,\alpha}$ can only increase going from $\R^n$ to $\Gamma_\nu^0$ along the contour deformation (including near non-degenerate critical points). It remains to observe that before the contour deformation, taking $x=\alpha_x$ and $\alpha_\xi$ real give parameters for which $\Im\Psi_{x,\alpha}\geq 0$ everywhere. We deduce that $Pu_\alpha$ behaves indeed like a $\G^1$ family of coherent states, at least for $\alpha_\xi$ in bounded domains, with the same phase as $u_\alpha$. Additionally, we find that the amplitude, as announced, is given by
\[
b = p(x,d_x\Phi(\alpha,x))a(\alpha,x) \mod h\G^1. 
\]

Now, we need to derive similar estimates, when $\alpha_\xi$ becomes large. Since we are only searching for uniform symbolic estimates, we can work locally in $\alpha_\xi$ (provided we indeed get uniform estimates): we pick $\xi_0\in\R^n$ satisfying $|\xi_0|=1$, $\lambda>C$, and we assume that $\lambda^{-1} \alpha_\xi$ remains in $B(\xi_0, \eta)_\epsilon$. Changing variable in \eqref{eq:approsimation_hatk_alt}, we find
\begin{equation}\label{eq:k_epsilon}
\begin{split}
k_{\delta}(x,\alpha) = \p{\frac{\lambda}{2 \pi h}}^n \int_{\R^n \times D'} \hspace{-10pt} e^{i \frac{\lambda}{h} \Psi_{x,\alpha,\lambda}(\xi,y) } e^{- \delta \lambda^2 \va{\xi}^2} p\p{x,\lambda \xi} a\p{\alpha,y} \mathrm{d}\xi \mathrm{d}y,
\end{split}
\end{equation}
where the phase $\Psi_{x,\alpha,\lambda}$ is defined by
\begin{equation*}
\begin{split}
\Psi_{x,\alpha,\lambda}(\xi,y) = \langle x-y, \xi \rangle + \frac{\Phi(\alpha,y)}{\lambda}.
\end{split}
\end{equation*}
In order to regularize the integral in \eqref{eq:k_epsilon} and get rid of the approximation, we introduce for $\nu > 0$ the contour
\begin{equation*}
\begin{split}
\Gamma_\nu^1 = \set{ \p{\xi + i \nu \va{\xi} \overline{(x-y)}, y + i \chi(y) \nu \frac{\overline{\nabla_y \Psi_{x,\alpha,\lambda}(\xi,y)}}{\jap{\va{\nabla_y \Psi_{x,\alpha,\lambda}(\xi,y)}}}} : \xi \in \R^n,y \in D'}.
\end{split}
\end{equation*}
Along $\Gamma_\nu^1$, the imaginary part of the phase becomes positive, as we explain now. The easiest observation is that estimate \eqref{eq:on_fait_converger_tout_ca} remains valid for $\Psi_{x,\alpha,\lambda}$ when $|\xi|$ is large enough, uniformly in $\lambda$. This is sufficient to see, as before, that
\begin{equation*}
\begin{split}
Pu_\alpha(x)^0 = \p{\frac{\lambda}{2 \pi h}}^n \int_{\Gamma_\nu^1} e^{i \frac{\lambda}{h} \Psi_{x,\alpha,\lambda}(\xi,y) } p\p{x,\lambda \xi} a\p{\alpha,y} \mathrm{d}\xi \mathrm{d}y.
\end{split}
\end{equation*}
The difference with the first part of the argument is that because of the rescaling, we only have uniform estimates on the amplitude on domains (in which $\Gamma_\nu^1$ is thankfully contained) where $\va{\Im \xi} \leq C^{-1} \va{\Re \xi}$ for some $C > 0$ (rather than for $\va{\Im \xi} \leq C^{-1} \jap{\Re \xi}$).

Using the same arguments as before (i.e the lower bound \eqref{eq:on_fait_converger_tout_ca} and Lemma \ref{lemma:norm-exponentials}), we observe again that for some $C>0$, (recall that $\lambda\simeq \jap{\va{\alpha}}$)
\begin{equation*}
 \begin{split}
 Pu_\alpha(x)^0 = \p{\frac{\lambda}{2 \pi h}}^n \int_{\set{ \substack{\xi \in \Gamma_\nu^1 \\ \jap{\va{\xi}} \leq C} }} e^{i \frac{\lambda}{h} \Psi_{x,\alpha,\lambda}(\xi,y) } & p\p{x, \lambda \xi} a\p{\alpha,y} \mathrm{d}\xi \mathrm{d}y \\ & + \O_{\mathcal{G}^1}\p{\exp\p{- \frac{\jap{\va{\alpha}}}{C h}}}.
\end{split}
\end{equation*} 

Contrary to the first case, we do not directly apply a stationary phase lemma at this stage. Instead we observe that for $y,\xi$ real and $\xi$ small enough, and using the ellipticity of $d_x\Phi$,
\begin{equation*}
\left| \nabla_y \Psi_{x,\alpha,\lambda}(\xi,y)\right|\geq \frac{1}{C}.
\end{equation*}
With a computation similar to that in \eqref{eq:on_fait_converger_tout_ca}, we obtain that in $\Gamma_\nu^1$ for $|\xi|$ small enough,
\begin{equation*}
\Im(\Psi_{x,\alpha,\lambda}(\xi,y))\geq \nu/C.
\end{equation*}
Here, the ellipticity gained from $\nabla_y \Psi_{x,\alpha,\lambda}$ is useful when $\chi(\Re y) = 1$, and in the regime where $\chi(\Re y)\neq 1$, we rely on the fact that $y$ cannot be too close to $\alpha_x$, and thus $\Im \Phi(\alpha,x)\geq \lambda \eta^2 /C$. Applying again Lemma \ref{lemma:norm-exponentials} we deduce that
\begin{equation*}
 \begin{split}
 Pu_\alpha(x)^0 = \p{\frac{\lambda}{2 \pi h}}^n \int_{\set{ \substack{\xi \in \Gamma_\nu^1 \\ 1/C\leq  \jap{\va{\xi}} \leq C} }} e^{i \frac{\lambda}{h} \Psi_{x,\alpha,\lambda}(\xi,y) } & p\p{x, \lambda \xi} a\p{\alpha,y} \mathrm{d}\xi \mathrm{d}y \\ & + \O_{\mathcal{G}^1}\p{\exp\p{- \frac{\jap{\va{\alpha}}}{C h}}}.
\end{split}
\end{equation*} 
The method of stationary phase can now be applied to the integral in the right hand side without further problem. Indeed, the amplitude and phase satisfy uniform bounds with respect to the parameters, the phase is uniformly positive on the boundary of the domain on integration, has a unique non-degenerate critical point. We also obtain that $\Im \Psi_{x,\alpha,\lambda}\geq 0$ on the whole contour for $x=\alpha_x$ and $\alpha_\xi$ real with the same justifications as before. We deduce that
\begin{equation*}
Pu_\alpha(x)^0 = e^{i\frac{\Phi(\alpha,x)}{h}}b(\alpha,x) + \mathcal{O}_{\G^1}\left(\exp\left(-\frac{\lambda}{Ch}\right)\right).
\end{equation*}
Actually, we obtain a phase and an amplitude formally depending on $\lambda$, satisfying uniform bounds. However the coefficients appearing in the stationary phase expansion do not depend on the rescaling that we performed, so that they do not actually depend on $\lambda$, and the uniform bounds translate into the expected analytic symbol estimates. We also obtain the expected formula for the first term asymptotics of $b \mod h S^{1,m+\ell-1}$.
\end{proof}

\subsection{Basic properties of Gevrey pseudors}\label{sec:pseudor_technique}

Now that we have obtained the necessary basic results regarding the kernels of Gevrey pseudors, we can prove their basic properties as operators that we stated in \S \ref{subsection:Gevrey_pseudor_manifold}. The elimination of variables Lemma \ref{lemma:calcul_pseudo} will be essential here. We start with the easy observation that the space of Gevrey pseudors is stable by taking adjoints. 

\begin{proof}[Proof of Proposition \ref{prop:adjoint_pseudor}]
Let us denote by $d\mu_0$ the analytic reference volume form with respect to which we have computed kernels so far. Certainly, for some $\G^s$ non-vanishing function $f$, $\mathrm{d}\mu = f \mathrm{d}\mu_0$. If $K_A(x,y)$ denotes the kernel of $A$, then the kernel of the adjoint $A^*$ of $A$ with respect to $d\mu$ is $\overline{K_A(y,x)}f(y)/f(x)$. Hence, it is clear that the kernel of $A^*$ is $\O_{\G^s}(\exp(- 1 /(C h)^{1/s}))$ away from the diagonal (because so is $K_A$). Then, it follows from Lemma \ref{lemma:calcul_pseudo} that the kernel of $A^*$ is of the form \eqref{eq:local_pseudor} in local coordinates (that we may choose volume preserving) as in the proof of Lemma \ref{lemma:droite_et_gauche}.
\end{proof}

It is then straightforward to deduce the mapping properties of $\G^s$ pseudors (Proposition \ref{prop:mapping_pseudor}) from Lemma \ref{lemma:basic-Gs-boundedness-pseudo} and Proposition \ref{prop:adjoint_pseudor}. We turn now to the stability under composition (Proposition \ref{prop:composition_pseudors}). The proof relies on Lemmas \ref{lemma:loin_de_la_diagonale},\ref{lemma:basic-Gs-boundedness-pseudo} and \ref{lemma:calcul_pseudo}.
\begin{proof}[Proof of Proposition \ref{prop:composition_pseudors}]
We start by proving that the kernel $K_{PQ}$ of $PQ$ is small in $\G^s$ away from the diagonal. We only need to consider the $\Gs$ norm of $K_{PQ}$ on $K = D_1 \times D_2$ when $D_1$ and $D_2$ are disjoint closed balls in $M$. Then, we may choose disjoint closed balls $D_1'$ and $D_2'$ such that $D_1$ and $D_2$ are contained respectively in the interior of $D_1'$ and $D_2'$. If $x \in D_1$ and $y \in D_2$, we may write
\begin{equation*}
\begin{split}
K_{PQ}(x,y) & = \int_{M} K_P(x,z) K_Q(z,y) \mathrm{d}z \\
            & = \p{\int_{D_1'} + \int_{D_2'} + \int_{M \setminus \p{D_1' \cup D_2'}}}  K_P(x,z) K_Q(z,y) \mathrm{d}z.
\end{split}
\end{equation*}
This splitting of the integral makes sense, considering the singular supports of $K_P$ and $K_Q$. The integral over $M\setminus \p{D_1' \cup D_2'}$ is easily dealt with since the integrand is an $\O(\exp(-1/( C h)^{1/s} ))$ in $\G^s$. The integral over $D_1'$ is dealt with by Lemma \ref{lemma:basic-Gs-boundedness-pseudo}. The integral over $D_2'$ is dealt with by Lemma \ref{lemma:basic-Gs-boundedness-pseudo} after using Lemma \ref{lemma:calcul_pseudo} to write $K_Q$ as the right quantization of a $\G^s$ symbol.

We need now to understand the kernel $K_{PQ}(x,y)$ near the diagonal, and this can be done in local coordinates. Here, the kernel $K_P$ may be written as a left quantization modulo a small $\Gs$ error, as in \eqref{eq:local_pseudor}. We may decompose $K_Q$ similarly and then, applying Lemma \ref{lemma:calcul_pseudo}, we can express $K_Q$ as a right quantization, i.e.
\begin{equation}\label{eq:local_twisted}
\begin{split}
K_Q(x,y) = \frac{1}{(2 \pi h)^n} \int_{\R^n} e^{i \frac{\jap{x-y,\xi}}{h}} q(\xi,y) \mathrm{d}\xi + \O_{\G^s}\p{\exp\p{- \frac{1}{C h^{\frac{1}{s}}}}},
\end{split}
\end{equation}
where $q$ is a symbol as in Definition \ref{def:gevrey_symbol}. In order to fix notations, let us say that \eqref{eq:local_pseudor} and \eqref{eq:local_twisted} hold in a ball $D'$ (in local coordinates) and that we want to compute the kernel $K_{PQ}(x,y)$ of $PQ$ in a strictly smaller ball $D$. The remainder terms in \eqref{eq:local_pseudor} and \eqref{eq:local_twisted} are dealt with as in the non-diagonal case, applying Lemma \ref{lemma:basic-Gs-boundedness-pseudo}, and we will ignore them.

Let then $\psi$ be a $\mathcal{C}^\infty$ function from $M$ to $\left[0,1\right]$ that vanishes outside $D'$ and is identically equal to $1$ on a neighbourhood of $D$. We may write for $x,y \in D$
\begin{equation}\label{eq:decomposition_KPQ}
\begin{split}
K_{PQ}(x,y) = \int_{D'} K_P(x,z)K_Q(z,y) \psi(z) \mathrm{d}z + \int_{M}K_P(x,z)K_Q(z,y) \p{1 -\psi(z)} \mathrm{d}z.
\end{split}
\end{equation}
The second term in the right hand side of \eqref{eq:decomposition_KPQ} is a $\O(\exp(-Ch^{1/s}))$ in $\G^s$ because it only involves off diagonal parts of $K_P$ and $K_Q$. In order to deal with the first term in \eqref{eq:decomposition_KPQ}, we apply \eqref{eq:local_pseudor} and \eqref{eq:local_twisted} (ignoring the remainder terms as we announced it) to write the oscillatory integral
\begin{equation}\label{eq:noyau_composition}
\begin{split}
		\int_{D'} K_P&(x,z) K_Q(z,y) \psi(z) \mathrm{d}z \\ 
		&=\frac{1}{(2 \pi h)^{2n}} \int_{\R^n \times \R^n \times \R^n}\hspace{-20pt} e^{i \frac{\jap{x-z,\xi} + \jap{z - y,\eta}}{h}} p(x,\xi) \psi(z) q(\eta,y) \mathrm{d}\xi \mathrm{d}z \mathrm{d}\eta \\
		&= \frac{1}{(2 \pi h)^n} \int_{\R^n} e^{i \frac{\jap{x - y,\xi}}{h}} p(x,\xi)q(\xi,y) \mathrm{d}\xi \\
		&\quad + \frac{1}{\p{2 \pi h}^{2n}} \int_{\R^n \times \R^n \times \R^n}\hspace{-20pt} e^{i \frac{\jap{x-z,\xi} + \jap{z - y,\eta}}{h}} p(x,\xi) \p{\psi(z) - 1} q(\eta,y) \mathrm{d}\xi \mathrm{d}z \mathrm{d}\eta.
\end{split}
\end{equation}
Here, we recognized a Fourier inversion formula. The first term in the right hand side of \eqref{eq:noyau_composition} is of the form prescribed by (ii) in Definition \ref{def:gevrey_pseudor} thanks to Lemma \ref{lemma:calcul_pseudo}. Hence, we only need to prove that the second term in the right hand side of \eqref{eq:noyau_composition} is an $\O(\exp(-1/(C h)^{1/s}))$ in $\G^s$. To do so, notice that the factor $1 - \psi(z)$ in the integrand ensures that we only integrate over $z$'s that are uniformly away from $x$ and $y$. Hence, the phase is non-stationary in $\eta$ and $\xi$, and we may shift the contour in $\eta$ and $\xi$ to prove that the integral is negligible (as we did in the proofs of Lemmas \ref{lemma:loin_de_la_diagonale} and \ref{lemma:basic-Gs-boundedness-pseudo}, we do not repeat the details). Notice that we do not need to shift contour in $z$, so that it does not matter that $\psi$ is only $\mathcal{C}^\infty$.
\end{proof}

Now, we use Proposition \ref{prop:pseudor_lagrangien_alt} to construct a principal symbol for a $\G^1$ pseudor and prove Proposition \ref{prop:existence_principal_symbol}.
\begin{proof}[Proof of Proposition \ref{prop:existence_principal_symbol}]
When $s > 1$, the existence of Gevrey partitions of unity make the proof an easy consequence of Lemmas \ref{lemma:loin_de_la_diagonale} and \ref{lemma:changement_de_variable} (see also Remark \ref{remark:developpement_asymptotique}). This is a quite standard procedure and thus we do not detail it (see for instance \cite[Theorem 14.1]{zworskiSemiclassicalAnalysis2012} for the $\mathcal{C}^\infty$ version). Consequently, we turn now to the case $s=1$ which relies on an application of Proposition \ref{prop:pseudor_lagrangien_alt}.

Let $\Phi^\ast$ be some admissible phase, and define
\begin{equation*}
\begin{split}
\Phi(\alpha,x) = - \overline{\Phi^*(\bar{\alpha},\bar{x})}.
\end{split}
\end{equation*}
Certainly, we have
\[
\Phi(\alpha,x)  = \langle x  - \alpha_x ,\alpha_\xi\rangle + \O(\langle|\alpha|\rangle |\alpha_x-x|^2).
\]
Comparing this with \eqref{eq:Behaviour-Taylor-Phase-admissible}, we see that $\Phi$ behaves as an admissible phase, were it not for some sign change. Using the $\bar{\partial}$ trick (Lemma \ref{lemma:approximation_analytique}), we can find a $\G^1$ family of coherent states of order $0$ with phase $\Phi$ and amplitude $1$, denoted $u_\alpha$. 

Let $P\in \G^1\Psi^m$. According to Proposition \ref{prop:pseudor_lagrangien_alt}, $Pu_\alpha$ is another $\G^1$ family of coherent states, and this implies that
\[
p(\alpha):= P u_\alpha(\alpha_x)
\]
defines an element of $S^{1,m}(T^\ast M)$. Since $\mathrm{d}_x\Phi(\alpha,\alpha_x) = \alpha_\xi$, then \eqref{eq:first-order-expansion-Pu_alpha} implies that $p$ is a $\G^1$ principal symbol for $P$.
\end{proof}

\begin{remark}
Notice that a $\G^1$ principal symbol for the $\G^1$ pseudor with kernel $K_{\Phi,a}$ defined by \eqref{eq:noyau_non_standard} is $a(x,\xi,x)$.
\end{remark}

Finally, we use the notion of pseudor with non-standard phase that we introduced in \S \ref{sssecKuranishi} to define a quantization procedure for analytic symbols.

\begin{proof}[Proof of Proposition \ref{prop:quantization_Gs}]
Thanks to Lemma \ref{lemma:changement_de_variable}, the case $s > 1$ is dealt with as in the $\mathcal{C}^\infty$ case: there are $\G^s$ partitions of unity. In order to give a construction for the case $s=1$, let us set for $p$ a $\G^1$ symbol
\begin{equation*}
\begin{split}
K_P(x,\alpha_x):= \frac{1}{(2 \pi h)^n} \int_{T_{\alpha_x}^* M} u_\alpha(x) p(\alpha) \mathrm{d}\alpha_\xi,
\end{split}
\end{equation*}
where $u_\alpha$ is the $\G^1$ family of coherent states from the previous proof. This is well defined as an oscillatory integral. It is a distribution on $M\times M$. It is $\O(\exp(- 1/Ch))$ in $\G^1$ away from the diagonal. Moreover, near the diagonal we may apply the Kuranishi trick (Lemma \ref{lmkurtrick}) in local coordinates to see that $K_P$ is the kernel of a $\G^1$ pseudor $P$ with $\G^1$ principal symbol $p$ (notice that we retrieve an admissible phase in local coordinates, because we changed the sign but we also switched variables $x$ and $\alpha_x$). We may then set $\Op(p) = P$ and the proof is over.
\end{proof}

\begin{remark}
The $\Op$ we have constructed does not satisfy the usual relations $\Op(1)=1$. We can remedy this. Indeed using the parametrix construction from the next section, one can see that the inverse $\Op(1)^{-1}$, which is defined when $h$ is small enough, is a $\G^1$ pseudor, and then define a new quantization by $\widetilde{\Op}(p) = \Op(1)^{-1} \Op(p)$.
\end{remark}

\subsection{Elliptic analytic pseudors}\label{sec:elliptic-pseudors}

When restricting to $\G^1$ pseudors, we can obtain stronger statements, in particular when they are elliptic. In this section we prove Theorem \ref{thm:parametrix} and \ref{theorem:functional_calculus}.

\subsubsection*{A parametrix construction}\label{sssecparametrix}

We study now the construction of a parametrix for an elliptic $\G^1$ pseudor $A$. We will start by focusing on the construction of the formal symbol of the parametrix of $A$ in the Euclidean case. We will then use this local construction to prove Theorem \ref{thm:parametrix}.

Notice that given formal analytic symbols $\underline{a} = \sum_{k \geq 0} h^k a_k$ and $\underline{b} = \sum_{k \geq 0} h^k b_k$, we can always define the formal symbol of the composition $\underline{a} \# \underline{b}$ by
\[
\underline{a} \# \underline{b} = \sum_{k\geq 0} \sum_{\va{\alpha} + \ell + m = k} \frac{h^k}{\alpha!} \partial_\xi^\alpha a_\ell \partial_x^\alpha b_m. 
\]
It is easy to show that it defines in fact a formal analytic symbol. Let $a$ and $b$ be realizations respectively of $\underline{a}$ and $\underline{b}$ (in the sense of Definition \ref{def:realisation_Gevrey}). Let $A$ and $B$ be $\G^1$ pseudors given in local coordinates by the the formula \eqref{eq:local_pseudor} with symbols respectively $a$ and $b$. We know by Proposition \ref{prop:composition_pseudors} that $C = A B$ is a $\G^1$ pseudor, so that it is given in local coordinates by the formula \eqref{eq:local_pseudor} with a certain symbol $c$. Then, recall Remark \ref{remark:developpement_asymptotique}, the symbol $c$ is obtained in the proof of Proposition \ref{prop:composition_pseudors} by an application of Lemma \ref{lemma:calcul_pseudo}, and, consequently, $c$ is a realization of the formal symbol $\underline{a} \# \underline{b}$. The first step in a parametrix construction is hence the following:

\begin{lemma}\label{lmparametrix}
Let $U$ be an open subset of $\R^n$ and let $m \in \R$. Let the formal symbol $\underline{a} = \sum_{k \geq 0} h^k a_k \in FS^{1,m}\p{T^* U}$ be elliptic in the sense that $a_0$ satisfies \eqref{eq:def_elliptique} in any relatively compact subset of $U$ (with a constant $C$ that may depend on the subset). If $\Omega$ is a relatively compact open subset of $U$, then there is a formal analytic symbol $\underline{b}$ of order $-m$ on $\Omega$ such that $\underline{b} \# \underline{a} = 1$.
\end{lemma}

The usual $\mathcal{C}^\infty$ parametrix construction gives a candidate for $\underline{b}$, built by induction, and implies that there can be at most one solution to this problem. However, it is not clear then that the formal symbol $\underline{b}$ that we obtain is indeed a formal analytic symbol. To do so, we could probably work as in \cite{Boutet-Kree-67}, but we rather adapt the proof of \cite[Theorem 1.5]{Sjostrand-82-singularite-analytique-microlocale}. Since we only want to establish the regularity of the solution, we may assume that $\Omega$ is a ball.

With a formal analytic symbol $\underline{a} = \sum_{k \geq 0} h^k a_k \in F S^{1,m}\p{T^*U}$, we associate the formal sum of differential operators
\begin{equation*}
\begin{split}
F(\underline{a}) = \sum_{k \geq 0} h^k A_k,
\end{split}
\end{equation*}
where $A_k$ is defined by the finite sum
\begin{equation}\label{eqdefAk}
\begin{split}
A_k = \sum_{ \ell + \va{\alpha} = k} \frac{1}{i^{\va{\alpha}} \alpha!} \partial_\xi^\alpha a_\ell(x,\xi) \partial_x^\alpha.
\end{split}
\end{equation}
If $\underline{b} = \sum_{k \geq 0} h^k b_k$ belongs to $F S^{1,m}\p{T^*U}$, we may define similarly a formal sum $F(\underline{b}) = \sum_{k \geq 0} h^k B_k$ of differential operator. Then, it is natural to define the composition $F(\underline{a}) \circ F(\underline{b})$ by 
\begin{equation}\label{eqcompositionform}
\begin{split}
F(\underline{a}) \circ F(\underline{b}) = \sum_{k \geq 0} h^k \sum_{k_1 + k_2 = k} A_{k_1} \circ B_{k_2}.
\end{split}
\end{equation}
Notice that with this convention we have that $F(\underline{a}) \circ F(\underline{b}) = F(\underline{a} \# \underline{b})$. Moreover, we may retrieve the formal symbol $\underline{a}$ from the operator $F(\underline{a})$ by noticing that $a_k = A_k(1)$. Hence, the associativity of $\circ$ implies the associativity of $\#$.

Now, choose $\epsilon_0 > 0$ small and define for $t \in \left[0, \epsilon_0\right]$ and $\lambda \in \N$ the open subset $V_{t,\lambda}$ of $\p{T^*U}_{\epsilon_0}$ by
\begin{equation*}
\begin{split}
V_{t,0} = \set{(x,\xi) \in \p{T^* U}_{\epsilon_0} : d(x,\Omega) < t \textup{ and } \va{\Re \xi} < 2}
\end{split}
\end{equation*}
and for $\lambda \geq 1$
\begin{equation*}
\begin{split}
V_{t,\lambda} = \set{(x,\xi) \in \p{T^* U}_{\epsilon_0} : d(x,\Omega) < t \textup{ and } 2^{\lambda - 1} < \va{\Re \xi} < 2^{\lambda+1}}.
\end{split}
\end{equation*}
Then, we denote by $B\p{V_{t,\lambda}}$ the Banach space of bounded holomorphic functions on $V_{t,\lambda}$. For $s,t \in [0,\epsilon_0]$, we let $\n{\cdot}_{s,t,\lambda}$ denote the operator norm for operators from $B\p{V_{t,\lambda}}$ to $B\p{V_{s,\lambda}}$. It follows then by Cauchy's formula that there are constants $C,R > 0$ such that for all $s < t \in \left[0 ,\epsilon_0 \right], \lambda \in \N$ and $\alpha \in \N^d$, we have
\begin{equation*}
\begin{split}
\n{\partial_x^\alpha}_{t,s,\lambda} \leq C R^{\va{\alpha}} \alpha! \p{t-s}^{- \va{\alpha}}.
\end{split}
\end{equation*}
Now, using Cauchy's formula again and recalling that $\underline{a}$ is a formal analytic symbol, we see that, provided $\epsilon_0$ is small enough, there are constants $C,R > 0$ such that for all $\ell, \lambda \in \N$ and $\alpha \in \N^d$ the function $\partial_\xi^\alpha a_\ell $ is bounded by $C R^{\va{\alpha} + \ell} \alpha! \ell! 2^{\lambda\p{m - \ell - \va{\alpha}}}$. Hence, up to taking $C,R > 0$ larger, the norm of $A_k$, defined by \eqref{eqdefAk}, as an operator from $B(V_{t,\lambda})$ to $B\p{V_{s,\lambda}}$ satisfies
\begin{equation*}
\begin{split}
\n{A_k}_{t,s,\lambda} \leq C R^k k! 2^{\lambda(m-k)} \p{t-s}^{- k},
\end{split}
\end{equation*}
for all $k,\lambda \in \N$ and $0 \leq s < t \leq \epsilon_0$. Define then for every $k \in \N$ the quantity (the index $m$ reminds us that $\underline{a}$ has order $m$)
\begin{equation}\label{eqdefpseudcoeff}
\begin{split}
f_{k,m}(\underline{a}) = \sup_{\substack{ 0 \leq s < t \leq \epsilon_0 \\ \lambda \in \N}} \frac{\p{t-s}^k \n{A_k}_{s,t,\lambda}}{k^k 2^{\lambda(m-k)}}.
\end{split}
\end{equation}
We have then $f_{k,m}\p{\underline{a}} \leq C R^k$ and hence for some $\rho > 0$ the pseudo-norm
\begin{equation}\label{eqdefpseudonorm}
\begin{split}
\n{\underline{a}}_{m,\rho} \coloneqq \sum_{k \geq 0} \rho^k f_{k,m}\p{\underline{a}}
\end{split}
\end{equation}
is finite. Now, notice that if $\underline{a} = \sum_{k \geq 0} h^k a_k$ is a formal sum such that $a_k \in S^{m-k}\p{ T^* \Omega }$ is defined on $\p{T^* \Omega}_{\epsilon_0}$, then we may still define the operator $A_k$ by \eqref{eqdefAk}, thus the coefficient $f_{k,m}\p{\underline{a}}$ by \eqref{eqdefpseudcoeff} and then the pseudo-norm $\n{\underline{a}}_{m,\rho}$ by \eqref{eqdefpseudonorm}. Then, if there is $\rho > 0$ such that $\n{\underline{a}}_{m,\rho}$ is finite, there are $C,R > 0$ such that $f_{k,m}\p{\underline{a}} \leq C R^k$ and hence using that $a_k = A_k\p{1}$, we find that $\underline{a}$ is a formal analytic symbol in $FS^{1,m}(T^* \Omega)$. From now on, we work with such formal sums of symbols. The advantage is that we have the following lemma.

\begin{lemma}\label{lmpasexactementassez}
Let $\underline{a}$ and $\underline{b}$ be formal sums of symbols as above. Assume that there are $m_1,m_2 \in \R$ and $\rho > 0$ such that $\n{\underline{a}}_{m_1,\rho}$ and $\n{\underline{b}}_{m_2,\rho}$ are finite. Then $\n{\underline{a} \# \underline{b}}_{m_1+m_2,\rho}$ is finite and
\begin{equation*}
\begin{split}
\n{\underline{a} \# \underline{b}}_{m_1+m_2,\rho} \leq \n{\underline{a}}_{m_1,\rho} \n{\underline{b}}_{m_2,\rho}.
\end{split}
\end{equation*}
\end{lemma}

\begin{proof}
To do so, we only need to prove that for all $k \in \N$ we have
\begin{equation}\label{eqrusesjo}
\begin{split}
f_{k,m_1+m_2}\p{\underline{a} \# \underline{b}} \leq \sum_{k_1+ k_2=k} f_{k_1,m_1}\p{\underline{a}} f_{k_2,m_2}\p{\underline{b}}.
\end{split}
\end{equation}
We recall that $F\p{\underline{a} \# \underline{b}} = F\p{\underline{a}} \circ F\p{\underline{b}}$ is given by the formula \eqref{eqcompositionform}. But then, if $0 \leq s < t \leq \epsilon_0$ and $\lambda \in \N$, we have (with $k_1 + k_2 = k$ and any $s < u < t$)
\begin{equation*}
\begin{split}
& \n{A_{k_1} \circ B_{k_2}}_{t,s,\lambda}\leq \n{A_{k_1}}_{t,u,\lambda} \n{B_{k_2}}_{u,s,\lambda} \\
      & \qquad \qquad \leq k_1^{k_1} 2^{\lambda(m_1-k_1)}(t-u)^{-k_1} f_{k_1,m_1}\p{\underline{a}} k_2^{k_2} 2^{\lambda(m_2-k_2)} (u-s)^{-k_2} f_{k_2,m_2}\p{\underline{b}} \\
      & \qquad \qquad \leq k_1^{k_1} k_2^{k_2} (t-u)^{-k_1}(u-s)^{-k_2} 2^{\lambda(m_1+m_2 - k)} f_{k_1,m_1}\p{\underline{a}} f_{k_2,m_2}\p{\underline{b}}.
\end{split}
\end{equation*}
Then, we may choose $u$ such that $t-u = (t-s) \frac{k_1}{k_1 + k_2}$ and $u-s = (t-s) \frac{k_2}{k_1 + k_2}$ to find that
\begin{equation*}
\begin{split}
\n{A_{k_1} \circ B_{k_2}}_{t,s,\lambda} \leq k^k (t-s)^{-k} 2^{\lambda(m_1+m_2 -k)} f_{k_1,m_1}\p{\underline{a}} f_{k_2,m_2}\p{\underline{b}}.
\end{split}
\end{equation*}
From this, \eqref{eqrusesjo} follows, and the proof of the lemma is thus over.
\end{proof}

We can now prove Lemma \ref{lmparametrix}.

\begin{proof}[Proof of Lemma \ref{lmparametrix}]
Recall that $\underline{a}$ is supposed to be an elliptic symbol of order $m$. Hence, if we define $b_0 = \frac{1}{a_0}$ then $b_0$ is an analytic symbol of order $-m$. Thus, for some $\rho_0 > 0$ we have $\n{\underline{a}}_{m,\rho_0} < + \infty$ and $\n{b_0}_{-m,\rho_0} < + \infty$. Then, by symbolic calculus we have
\begin{equation*}
\begin{split}
b_0 \# \underline{a} = 1 - \underline{r},
\end{split}
\end{equation*}
where $\underline{r}$ is a formal analytic symbol of order $0$, with no leading coefficient (that is $\underline{r}$ is $h$ times a formal analytic symbol of order $-1$). Hence, we deduce that for $0<\rho \leq \rho_0$, 
\[
\| r \|_{0,\rho} \leq \frac{\rho}{\rho_0}\|1 - b_0 \# \underline{a}\|_{0,\rho_0} \leq \frac{\rho}{\rho_0}\p{1+\|\underline{a}\|_{m,\rho_0}\| b_0 \|_{-m, \rho_0}},
\]
so that $\n{\underline{r}}_{0,\rho} < 1$ if $\rho$ is small enough. Thus, we want to use Von Neumann's argument to say that $\underline{b} = \sum_{k \geq 0} \underline{r}^{\# k} \# b_0$ is the symbol of a parametrix for $\underline{a}$. Applying Lemma \ref{lmpasexactementassez}, we see that for every $\ell \in \N$
\begin{equation*}
\begin{split}
\n{\sum_{p = 0}^{\ell} \underline{r}^{\# p} \# b_0}_{-m,\rho} \leq \frac{\n{b_0}_{-m,\rho}}{1 - \n{\underline{r}}_{0,\rho}}.
\end{split}
\end{equation*}
Since $\underline{r}$ has no leading coefficient, if $\ell$ is large enough (depending on $k$), then $f_{k,-m}\p{\underline{b}} = f_{k,-m}\p{\sum_{p = 0}^{\ell} \underline{r}^{\# p} \# b_0}$. In particular, for all $N\geq 0$, for $\ell_N$ large enough,
\[
\sum_{k=0}^N \rho^k f_{k,-m}(\underline{b}) = \sum_{k=0}^N \rho^k f_{k,-m}\left(\sum_{p = 0}^{\ell_N} \underline{r}^{\# p} \# b_0\right) \leq \frac{\n{b_0}_{-m,\rho}}{1 - \n{\underline{r}}_{0,\rho}}.
\]
Letting $N\to +\infty$, we deduce
\[
\| \underline{b} \|_{-m,\rho} \leq \frac{\n{b_0}_{-m,\rho}}{1 - \n{\underline{r}}_{0,\rho}} < + \infty,
\]
so that $\underline{b}$ is a formal analytic symbol for the reason described above.
\end{proof}

\begin{remark}\label{remark:uniform-bounds-resolvent}
It will be useful to obtain a version with parameters of Lemma \ref{lmparametrix}. We pick a symbol $a\in S^{1,m}(T^\ast U)$, elliptic of order $m > 0$. We assume that there is $V\subseteq \C$, an open set, conical at infinity, so that $\overline{V}\cap a(T^\ast \R^n) = \emptyset$. For $\epsilon > 0$, we see that if we define
\begin{equation*}
\begin{split}
W = \set{ z \in \C : d(z, \C \setminus V) > \epsilon \va{z}},
\end{split}
\end{equation*}
then, for some $\delta>0$ and every $(x,\xi)\in (T^\ast \R^n)_\delta$ , we have $|a(x,\xi)-s|> \delta|a(x,\xi)|$ for all $s\in W$. Then for $s\in W$, the family
\[
s \mapsto \frac{1}{a- s}
\]
is a bounded family of $\G^1$ symbols of order $-m$. In particular, denoting by $\underline{b}(s)$ the formal symbol such that $\underline{b}(s)\#(a-s) = 1$, we see that $\underline{b}(s)$ is a bounded family of formal symbols when $s$ varies in $W$. Moreover, from Remark \ref{remark:regularite_solution_d_bar} and Lemma \ref{lemma:existence_realisation}, we see that we may choose for $s \in W$ a realization $b(s)$ of $\underline{b}(s)$, so that $b(s)$ defines a bounded family of analytic symbol with measurable dependence on $s$.
\end{remark}

We are now ready for the:
\begin{proof}[Proof of Theorem \ref{thm:parametrix}]
We only construct a left parametrix for $P$. The construction of a right parametrix is similar, and then a standard argument ensures that the left and right parametrix coincide. 

By compactness of $M$, we may cover $M$ by a finite number of balls $\p{D_j}_{j \in J}$ such that each $D_j$ is a relatively compact subset of a coordinate patch $U_j$. Since $P$ is an elliptic $\G^1$ pseudor, for each $j \in J$ there is a symbol $a_j$, elliptic of order $m$ on $U_j$ such that in the coordinates on $U_j$ the kernel of $P$ is
\begin{equation*}
\begin{split}
K_j(x,y) = \frac{1}{(2 \pi h)^n} \int_{\R^n} e^{i \frac{\jap{x-y,\xi}}{h}} a_j(x,\xi) \mathrm{d}\xi + \O_{\G^1}\p{\exp\p{- \frac{1}{C h}}}.
\end{split}
\end{equation*}
We apply Lemma \ref{lmparametrix} to $a_j$ to find a formal symbol $\underline{b}_j$ such that $\underline{b}_j \# a_j = 1$. Then, choose a realization $b_j$ for $\underline{b}_j$, and define for $x,y$ in a neighbourhood of $D_j$,
\[
Q_j(x,y) := \frac{1}{(2 \pi h)^n} \int_{\R^n} e^{i \frac{\jap{x-y,\xi}}{h}} b_j(x,\xi) \mathrm{d}\xi.
\]
(this in the coordinate charts of $U_j$). Let us assume that $D_j\cap D_\ell\neq \emptyset$, and consider $D_j \cap D_\ell \subset D \subset D' \subset U_j\cap U_\ell$ open balls relatively compact one in another. Thanks to the definition of $b_j$, and since away from the diagonal the kernel of $\G^1$ pseudors are $\O(\exp(-1/(Ch)))$ in $\G^1$, if $\chi$ is a function supported in $D'$ and identically equal to $1$ on a  neighbourhood of $D$, we find that the kernels of the operators
\begin{equation*}
Q_j - Q_\ell \textup{ and } (Q_j - Q_\ell) \chi P \chi Q_j
\end{equation*}
differ by an $\O_{\G^1}\p{\exp\p{-\frac{1}{C h}}}$ in $D \times D$. To prove this, we start by arguing as in the proof of Proposition \ref{prop:composition_pseudors} that the contributions from off-diagonal term yield $\O_{\G^1}\p{\exp\p{-\frac{1}{C h}}}$ errors even though we added $C^\infty$ cutoffs. This implies that in a neighbourhood of $D\times D$, the kernel of $\chi P \chi Q_j$ takes the form \eqref{eq:pseudo_Rn} modulo a $\O_{\G^1}\p{\exp\p{-\frac{1}{C h}}}$ error. Then we apply Remark \ref{remark:developpement_asymptotique} to see that the corresponding symbol is a realization of the constant symbol $1$, by construction of $\underline{b}_j$. In particular, in a neighbourhood of $D\times D$, the kernel of $\chi P \chi Q_j$ is that of the identity up to a $\O_{\G^1}\p{\exp\p{-\frac{1}{C h}}}$ error. Using again the same arguments as in the proof of Proposition \ref{prop:composition_pseudors}, we deduce the announced estimate.

However, the same arguments apply if we decompose the product as
\begin{equation*}
(Q_j \chi P\chi - Q_\ell P\chi)Q_j,
\end{equation*}
and this proves that in a neighbourhood of $D\times D$,
\begin{equation}\label{eq:local-parametrix-are-close}
(Q_j - Q_\ell)(x,y) =\O_{\G^1}\p{\exp\p{-\frac{1}{C h}}}.
\end{equation}
Hence, the distribution $g_{j,\ell}(x,y) \coloneqq (Q_j - Q_\ell)(x,y)$, for $j,\ell \in J$, is in fact an analytic function. Moreover, $g_{j,\ell}$ has a holomorphic extension to a complex neighbourhood of $D_j \cap D_\ell$ which is an $\O(\exp(- 1/(Ch)))$. It follows from the analytic continuation principle that, wherever it makes sense, these holomorphic extensions satisfy
\begin{equation}\label{eq:cocycle}
\begin{split}
Q_j(x,y) = Q_\ell(x,y) + g_{j,\ell}(x,y) \textup{ and } g_{j,k}(x,y) = g_{j,\ell}(x,y) + g_{\ell,k}(x,y).
\end{split}
\end{equation}
For $j \in J$, we choose a complex neighbourhood $W_j \subseteq \widetilde{M}$ of $\overline{D}_j$ small enough so that $g_{j,\ell}$ is well-defined on $\p{W_j \cap W_\ell}^2$ for $\ell \in J$. Then, we pick $U \subseteq \widetilde{M} \times \widetilde{M}$ an open set such that $\overline{U}$ does not encounter the diagonal. Up to taking the complexification $\widetilde{M}$ of $M$ smaller, we may assume that $U \cup \{ W_j\times W_j \ |\ j\in J\}$ is an open cover of $\widetilde{M}\times \widetilde{M}$. Then, we pick a corresponding $C^\infty$ partition of unity $((\chi_j)_{j \in J},\chi_U)$. As a first approximation for the kernel of the parametrix $Q$, we define the distribution
\begin{equation*}
Q^0(x,y) = \sum_{j} \chi_j(x,y)Q_j(x,y)
\end{equation*}
on $M \times M$. Of course, due to the use of a $\mathcal{C}^\infty$ partition of unity, it is very unlikely for $Q^0(x,y)$ to be the kernel of $\G^1$ pseudor. We will use a $\bar{\partial}$ trick to correct that.

For $\epsilon > 0$ small enough, we define a form $f$ of type $\p{0,1}$ on $(M \times M)_\epsilon$ by
\begin{equation*}
\begin{split}
f(x,y) = \begin{cases}  Q_j(x,y) \bar{\partial} \chi_U(x,y) - \sum\limits_{\ell \in J} g_{\ell,j}(x,y) \bar{\partial} \chi_{\ell}(x,y) & \textup{ if } (x,y) \in W_j \times W_j \\
     0 & \textup{ if } (x,y) \notin \bigcup\limits_{j \in J} W_j \times W_j. \end{cases}
\end{split}
\end{equation*}
Here, the holomorphic extension of $Q_j(x,y)$ away from the diagonal is given by Lemma \ref{lemma:loin_de_la_diagonale} and if one of the $g_{\ell,j}$ does not makes sense then the $\bar{\partial}$ in factor vanishes so that the sum is well-defined. The fact that $f$ is well-defined despite the possible intersection between cases in its definition comes from \eqref{eq:cocycle} (use also the fact that $\bar{\partial}\p{\chi_U + \sum_{j \in J} \chi_j} = 0$). From the estimate on the $g_{\ell,j}$ and Lemma \ref{lemma:loin_de_la_diagonale}, we see that $f$ is an $\O(\exp(- 1/Ch))$ on $(M \times M)_\epsilon$.

It follows then from Lemma \ref{lemma:approximation_analytique} that there are $\epsilon_1 > 0$ and a smooth function $r$ on $(T^* M)_{\epsilon_1}$ which is an $\O(\exp(- 1/Ch))$ and such that $\bar{\partial}r = f$. We use $r$ to define a new approximation of the kernel of our parametrix:
\begin{equation*}
\begin{split}
Q^1(x,y) = Q^0(x,y) + r(x,y).
\end{split}
\end{equation*}
Notice that for $x,y$ in a neighbourhood of $D_j$, we have
\begin{equation}\label{eq:local_presque}
\begin{split}
Q^1(x,y) - Q_j(x,y) = - Q_j(x,y) \chi_U(x,y) + \sum_{\ell \in J} \chi_{\ell}(x,y)g_{\ell,j}(x,y) + r(x,y).
\end{split}
\end{equation}
The right hand side of \eqref{eq:local_presque} has a smooth extension to $W_j \times W_j$ which is holomorphic (because $\bar{\partial}r = f$) and an $\O(\exp(- 1/Ch))$ (due to Lemma \ref{lemma:loin_de_la_diagonale} and the estimates on the $g_{\ell,j}$'s and on $r$). We just proved that
\begin{equation}\label{eq:local_Q1}
\begin{split}
Q^1(x,y) = Q_j(x,y) + \O_{\G^1}\p{\exp\p{- \frac{1}{Ch}}}
\end{split}
\end{equation}
on a neighbourhood of $D_j \times D_j$. Working similarly, we see that $Q^1(x,y)$ is an $\O(\exp(- 1/Ch))$ in $\G^1$ away from the diagonal. Thus, $Q^1(x,y)$ is the kernel of a $\G^1$ pseudor that we also denote by $Q^1$. From \eqref{eq:local_Q1} and the definition of the $Q_j$'s, we see that 
\[
Q^1 P  = 1 + \O_{\G^1} \p{\exp\p{-\frac{1}{Ch}}}.
\]
The right hand side here takes the form $1+ R$, with $\| R\| < 1$ as a bounded operator from $(E^{1,R_0})'$ to $E^{1,R_0}$ (for some large $R_0$ and provided $h$ is small enough). It follows that $1+R$ is invertible as an operator from $(E^{1,R_0})'$ to itself, with inverse $1+R'$, so that
\[
R' = - R - R(1+R')R.
\]
From this formula, we deduce that $R'$ has kernel $\O_{\G^1}(\exp(-1/Ch))$, so that $Q = (1+R') Q^1$ is a $\G^1$ pseudor (due to Proposition \ref{prop:composition_pseudors}), inverse of $P$. 
\end{proof}

\subsubsection*{Functional calculus for analytic pseudors}\label{subsec:functional_calculus}

This last section is dedicated to the proof of Theorem \ref{theorem:functional_calculus}. Consequently, we denote by $A$ a self-adjoint $\G^1$ pseudor, classically elliptic of order $m > 0$ whose spectrum is contained in $[\varpi,+\infty[$. 

Let us start a few reductions. First of all, up to adding a constant to $A$, we may assume that $\varpi = 0$. Then, for any integer $\ell \in \N$, we may write
\begin{equation*}
\begin{split}
f(A) = \p{1 + A}^\ell \p{1 + A}^{- \ell} f(A),
\end{split}
\end{equation*}
hence, we may replace the function $f$ by the function $z \mapsto \frac{f(z)}{(1 + z)^{\ell}}$, and assume that $N < - 3 - \frac{n}{m}$. Let us state two lemmas before proving Theorem \ref{theorem:functional_calculus}. Recall that $A$ admits a $\G^1$ principal symbol $a$ in the sense of Proposition \ref{prop:existence_principal_symbol}. Since $A$ is self-adjoint, we may assume that $a$ is real-valued (up to replace it by $(a+\bar{a})/2$).

\begin{lemma}\label{lemma:ouestlesymbole}
Let $\delta >0$. For $h$ small enough, $a$ is valued in $\left[-\delta, + \infty \right[$.
\end{lemma}

\begin{proof}
The proof is by contradiction. During this proof, we will write $a =a_h$ and $A = A_h$ to make the dependence on $h$ clearer. Assume that there is a sequence $\p{h_k}_{k \in \N}$ converging to $0$ and a sequence of points $\alpha_k \in T^* M$ such that for every $k \in \N$ we have $a_{h_k}(\alpha_k) < - \delta$. Since we assumed that $A_h$ is classically elliptic, there is a constant $C > 0$ such that for every $\alpha \in T^*M$ and $h > 0$ small enough, we have
\begin{equation*}
\begin{split}
\va{a_h(\alpha)} \geq \frac{\jap{\alpha}^m}{C} - C.
\end{split}
\end{equation*}
Hence, there is a constant $B > 0$ such that $\va{a_h(\alpha)} \geq 1$ if $\jap{\alpha} \geq B$. Consequently, the sign of $a_h$ is constant on the connected components of $\set{\alpha \in T^* M : \jap{\alpha} \geq B}$. This sign has to be positive, otherwise it would follow from the Weyl's law that $A$ has negative eigenvalues.

It follows that the $\alpha_k$'s remain in a compact subset of $T^*M$. Hence, up to extracting a subsequence, we may assume that $\p{\alpha_k}_{k \in \N}$ converges to some $\alpha_\infty = \p{\alpha_{\infty,x},\alpha_{\infty,\xi}} \in T^* M$. By uniform continuity of the $a_{h_k}$'s we see that there is a neighbourhood $V$ of $\alpha_\infty$ in $T^* M$ such that for $k$ large enough and $\alpha \in V$ we have $a_{h_k}(\alpha) < - \frac{\delta}{2}$. Now let $\varphi$ be a $\mathcal{C}^\infty$ function from $M$ to $\R$ such that $\mathrm{d} \varphi(\alpha_{\infty,x}) = \alpha_\infty$. Then, we choose a $\mathcal{C}^\infty$ function $\chi$ from $M$ to $\R$ such that $\chi(\alpha_{\infty,x}) = 1$ and $\mathrm{d}\varphi(x) \in V$ for every $x$ in the support of $\chi$. By the ($\mathcal{C}^\infty$) stationary phase method, we find that
\begin{equation*}
\begin{split}
A_{h_k} \p{ \chi e^{i \frac{\varphi}{h_k}}} \underset{k \to + \infty}{=} a_{h_k}\p{\mathrm{d}\varphi} \chi + \O_{\mathcal{C}^\infty}\p{h_k}.
\end{split}
\end{equation*} 
However, for $k$ large enough, the spectrum of $A_{h_k}$ is by assumption contained in $\R_+$ and thus
\begin{equation*}
\begin{split}
0 \leq \jap{A_{h_k}\p{ \chi e^{i \frac{\varphi}{h_k}}} , \chi e^{i \frac{\varphi}{h_k}} } & \underset{k \to + \infty}{=} \int_{M} a_{h_k} \p{\mathrm{d}\varphi(x)} \va{\chi(x)}^2 \mathrm{d}x + \O\p{h_k} \\
       & \leq - \frac{\delta}{4} \int_M \va{\chi(x)}^2 \mathrm{d}x,
\end{split}
\end{equation*}
where the last line holds for $k$ large enough and is a contradiction.
\end{proof}

\begin{lemma}\label{lemma:fACauchy}
Let $\epsilon_1 \in \left]0,\epsilon \right[$. For $h$ small enough, we have
\begin{equation}\label{eq:fACauchy}
\begin{split}
f(A) = \frac{1}{2 i \pi} \int_{\partial U_{0,\epsilon_1}} f(w) \p{w - A}^{-1} \mathrm{d}w,
\end{split}
\end{equation}
where the integral converges in $L^2$ operator norm for instance (recall that we reduced to the case $N < -3 - \frac{n}{m}$).
\end{lemma}

\begin{proof}
Let $\p{\lambda_k}_{k \in \N}$ denotes the sequence of eigenvalues of $A$ (ordered increasingly and counted with multiplicity) and $\p{\psi_k}_{k \in \N}$ denotes an associated orthonormal systems of eigenvectors. By a ($\mathcal{C}^\infty$) parametrix construction, we see that the resolvent of $A$ is bounded from $L^2$ to the Sobolev space $H^{m}$, hence we have (it also follows from Weyl's law)
\begin{equation*}
\begin{split}
\sum_{k \in \N} \frac{1}{\va{\lambda_k}^{1 + \frac{n}{m}}} < + \infty.
\end{split}
\end{equation*}
The operator $f(A)$ is defined to be 
\begin{equation*}
\begin{split}
f(A) = \sum_{k \geq 0} f\p{\lambda_k} \jap{\cdot,\psi_k} \psi_k, 
\end{split}
\end{equation*}
and, since $N \leq -1 - \frac{n}{m}$, we see that this sum converges absolutely in the $L^2$ operator norm. Now, notice that
\begin{equation*}
\begin{split}
\sum_{k \geq 0} \mathrm{Leb}\p{\left[\lambda_k - \lambda_{k}^{-\p{1 + \frac{n}{m}}}, \lambda_k + \lambda_{k}^{-\p{1 + \frac{n}{m}}}\right]} < + \infty,
\end{split}
\end{equation*}
hence, since $\mathrm{Leb}\p{\R_+} = + \infty$ but the Lebesgue measure of any compact subset of $\R_+$ is finite, we may find a sequence $\p{r_\ell}_{\ell \in \N}$ of positive real numbers that tends to $+ \infty$, such that for every natural integers $k,\ell$ we have
\begin{equation*}
\begin{split}
\va{r_\ell - \lambda_k} \geq \lambda_{k}^{-\p{1 + \frac{n}{m}}}.
\end{split}
\end{equation*}
It follows that for $\ell$ large enough and all $k \in \N$ we have
\begin{equation}\label{eq:loinduspectre}
\begin{split}
\va{r_\ell - \lambda_k} \geq r_\ell^{- \p{2 + \frac{n}{m}}}.
\end{split}
\end{equation}
Now, set for $\ell \in \N$
\begin{equation*}
\begin{split}
V_\ell = U_{0,\epsilon_1} \cap \set{z \in \C : \Re z < r_\ell}, I_\ell = U_{0,\epsilon_1} \cap \set{z \in \C : \Re z = r_\ell},
\end{split}
\end{equation*}
and
\begin{equation*}
\begin{split}
J_\ell = \set{z \in \C : \Re z \leq r_l} \cap \partial U_{0,\epsilon_1}.
\end{split}
\end{equation*}
Then, for $\ell$ large, we write Cauchy's formula
\begin{equation}\label{eq:Cauchyintermediaire}
\begin{split}
\sum_{\substack{k \geq 0 \\ \lambda_k \leq r_\ell}} f\p{\lambda_k} \jap{\cdot,\psi_k} \psi_k & = \frac{1}{2 i \pi} \int_{\partial V_\ell} f(w) \p{w - A}^{-1} \mathrm{d}w \\
     & = \frac{1}{2 i \pi} \int_{J_\ell} f(w) \p{w - A}^{-1} \mathrm{d}w \\ & \qquad \qquad \qquad \qquad \qquad+ \frac{1}{2 i \pi} \int_{I_\ell} f(w) \p{w - A}^{-1} \mathrm{d}w.
\end{split}
\end{equation}
Using the bound $\| (w - A)^{-1}\|_{L^2 \to L^2} \leq |\Im w|^{-1}$, we see that $(w-A)^{-1}$ is uniformly bounded in the $L^2$ operator norm on $\partial U_{0,\epsilon_1}$. Hence, since $N < -1$, the integral in the right hand side of \eqref{eq:fACauchy} converges. Moreover, we see that the integral over $J_\ell$ in \eqref{eq:Cauchyintermediaire} tends to the right hand side of \eqref{eq:fACauchy} when $\ell$ tends to $+ \infty$. Consequently, we only need to prove that the integral over $I_\ell$ in \eqref{eq:Cauchyintermediaire} tends to $0$ when $\ell$ tends to $+ \infty$ to end the proof of the lemma.

To do so, notice that, for $\ell$ large enough, when $w \in I_\ell$ then, thanks to \eqref{eq:loinduspectre}, the distance between $w$ and the spectrum of $A$ is greater than $r_\ell^{-(2 + n/m)}$, so that we have $\| (w - A)^{-1}\|_{L^2 \to L^2} \leq r_\ell^{2 + n/m}$. Then, using the decay assumption on $f$ and the fact that the length of $I_\ell$ is a $\O\p{r_\ell}$, we find that the integral over $I_\ell$ in \eqref{eq:Cauchyintermediaire} is a $\O(r_\ell^{N+3 + n/m})$. Hence, this integral tends to $0$ since we reduced to the case $N < - 3 - n/m $.
\end{proof}

We are now ready to prove Theorem \ref{theorem:functional_calculus}.

\begin{proof}[Proof of Theorem \ref{theorem:functional_calculus}]
We want to apply Lemma \ref{lemma:fACauchy} with $\epsilon_1 = \frac{3}{4} \epsilon$. To do so, we need to check that for $w \in \partial U_{0,3\epsilon/4}$, the operator $\p{w - A}^{-1}$ is a semi-classical $\G^1$ pseudor with principal symbol $\p{w-a}^{-1}$. Recall Remark \ref{remark:uniform-bounds-resolvent} and apply Lemma \ref{lemma:ouestlesymbole} with $\delta = \frac{\epsilon}{4}$ to see that $w-a$ for $w \in \partial U_{0,3\epsilon/4}$ is uniformly elliptic. We may then apply Theorem \ref{thm:parametrix} to see that $(w - A)^{-1}$ is a $\G^1$ pseudor.

We explain now why Definition \ref{def:gevrey_pseudor} remains valid after averaging over $\partial U_{0,3\epsilon/4}$. We prove that $f(A)$ satisfies (ii) in Definition \ref{def:gevrey_pseudor} (the point (i) is easier to check). Since $(w - A)^{-1}$ for $w \in \partial U_{0,3\epsilon/4}$ is a $\G^1$ pseudor, its kernel may be written in local coordinates:
\begin{equation}\label{eq:resolvante_pseudo}
\begin{split}
(w- A)^{-1} (x,y) = \frac{1}{(2 \pi h)^n} \int_{\R^n} e^{i \frac{\jap{x-y,\xi}}{h}} p_w(x,\xi) \mathrm{d}\xi + R_w(x,y),
\end{split}
\end{equation}
where $p_w$ is an analytic symbol and $R_w$ is an $\O(\exp(-1/Ch))$ in $\G^1$. Moreover, it follows from Remark \ref{remark:uniform-bounds-resolvent} (see also Remark \ref{remark:regularite_solution_d_bar}) and the proof of Theorem \ref{thm:parametrix} that $p_w$ and $R_w$ satisfies uniform bound in $w$ and have a measurable dependence on $w$. Consequently, we may average \eqref{eq:resolvante_pseudo} over $\partial U_{0,3\epsilon/4}$ and find with Lemma \ref{lemma:fACauchy} and Fubini's Theorem that (the decay assumption on $f$ ensures the integrability)
\begin{equation}\label{eq:fApseudo}
\begin{split}
f(A)(x,y) = \frac{1}{(2 \pi h)^n} \int_{\R^n} e^{i \frac{\jap{x-y,\xi}}{h}} & \underbrace{\p{\int_{\partial U_{0,3\epsilon/4}} f(w) p_w(x,\xi) \mathrm{d}w}}_{\coloneqq q(x,\xi)}  \mathrm{d}\xi\\ & \qquad \qquad+ \underbrace{\int_{\partial U_{0,3\epsilon/4}} f(w) R_w(x,y) \mathrm{d}w}_{\coloneqq R(x,y)}.
\end{split}
\end{equation}
Due to the uniform bounds on $p_w$ and $R_w$, we see that $q$ is an analytic symbol and that $R$ is an $\O(\exp(-1/Ch))$ in $\G^1$. Here, we applied Fubini's Theorem in an oscillating integral, this is made rigorous by integrating against test functions and doing a finite number of integration by parts. Finally, notice that $q$ is given at first order by $f(a)$ since $p_w$ is given at first order by $(w-a)^{-1}$.
\end{proof}

\setcounter{equation}{0}

\chapter{FBI transform on compact manifolds}
\label{part:FBI}

The purpose of this chapter is to present some features of FBI transforms in the Gevrey classes. We consider fixed $(M,g)$ a compact analytic Riemannian manifold without boundary. In \S\ref{sec:FBI-basics}, we recall some symplectic geometry in the complexification of $T^\ast M$, introduce the so-called I-Lagrangians, and the corresponding I-Lagrangian (functional) spaces. We also study the relation between the regularity of distributions and the decay of their FBI transforms.

In the next section \S \ref{sec:FBI-et-pseudo}, we study the action of pseudors on I-Lagrangian spaces. The most technical and crucial arguments of the paper are contained in \S \ref{sec:TPS}.  Then, \S \ref{sec:Bergman} is devoted to the precise description of the Bergman projector, and its applications. In particular \S \ref{sec:Toeplitz} contains the proof of the representation formula \eqref{eq:flat-space-scalar-product}. It is closely related to a Toeplitz representation of pseudors on $M$.

For the convenience of the reader, we start with a summary of the technical results.

\subsection*{Results: I-Lagrangian deformations and Gevrey pseudors}

We recall here some definitions and notations that are exposed in greater detail in \S \ref{sec:Grauert}. We will denote the Riemannian volume on $M$ by $\mathrm{d}\vol_g$, or just $\mathrm{d}x$ as we will take charts with Jacobian identically equal to $1$. The cotangent space $T^\ast M$ will be endowed with the corresponding (analytic) Kohn--Nirenberg metric $g_{KN}$. In local charts, it is equivalent to
\[
g_{KN}^{\mathrm{flat}}= \mathrm{d}x^2 + \frac{\mathrm{d}\xi^2}{\langle\xi\rangle^2}.
\]
Given an analytic manifold $X$, we will denote by $(X)_\epsilon$ its $\epsilon$-Grauert tube (see \S \ref{sec:Grauert} for its definition). It is a geometric version of a complex neighbourhood of $X$. In the case of $T^\ast M$, the tube $(T^\ast M)_\epsilon$ is an asymptotically conical complex neighbourhood. It will be convenient to introduce for $\alpha\in (T^\ast M)_\epsilon$ the Japanese brackets $\jap{\alpha}$ and $\jap{\va{\alpha}}$ defined by \eqref{eq:def_jap_brack} in \S \ref{sec:Grauert}. Beware that the first is holomorphic, and not the second one. However, $|\langle\alpha\rangle|$, $\Re \langle\alpha\rangle$ and $\langle|\alpha|\rangle$ are comparable on $(T^\ast M)_\epsilon$ for $\epsilon>0$ small enough. In Definition \ref{def:analytic-FBI} and below, we consider a small implicit parameter $h > 0$. Unless specified, all the estimates are supposed to be uniform in $h$.

We recall that an admissible phase in the sense of Definition \ref{def:nonstandardphase}, is a holomorphic symbol $\Phi(\alpha,x)$ of order $1$ on $ T^\ast M \times M$, defined for $x$ close to $\alpha_x$, so that
\[
\Phi(\alpha,x) = \langle \alpha_x - x ,\alpha_\xi\rangle + \mathcal{O}(\langle|\alpha|\rangle
|\alpha_x - x|^2),\qquad \Im\Phi \geq \frac{1}{C}\langle |\alpha| \rangle |\alpha_x - x|^2.
\]
Here, we recall that we write $\alpha = (\alpha_x,\alpha_\xi)$ for $\alpha \in T^* M$.

In general, an FBI transform $T$ (on $M$) is a linear map between $\mathcal{D}'(M)$ and $C^\infty(T^\ast M)$. Away from the ``diagonal'' $\{x=\alpha_x\}$, its kernel is assumed to decay at a sufficient rate (which depends on the context), and near the diagonal, (possibly modulo a remainder) it takes the form
\[
e^{i\frac{\Phi_T(\alpha,x)}{h}} a(\alpha,x),
\]
where $a$ is elliptic in some symbol class, and $\Phi_T$ is an admissible phase.
\begin{definition}\label{def:analytic-FBI}
An \emph{analytic FBI transform} is an FBI transform $T$ such that for some $C,\epsilon_0,\epsilon_1,\eta>0$, the kernel $K_T$ of $T$ is holomorphic in $(M\times T^\ast M)_{\epsilon_0}$ and for $(x,\alpha)$ therein satisfies
\begin{enumerate}[label=(\roman*)]
	\item for $d(x,\alpha_x)>\epsilon_1$, we have $|K_T(\alpha,x)| \leq C e^{-\eta\frac{ \langle | \alpha| \rangle}{h}}$;
	\item for $d(x,\alpha_x)\leq\epsilon_1$, 
	\begin{equation}\label{eq:def-KT}
	\left|K_T(\alpha,x)- e^{i\frac{\Phi_T(\alpha,x)}{h}} a(\alpha,x)\right| \leq C e^{-\eta\frac{\langle|\alpha|\rangle}{h}};
	\end{equation}
\end{enumerate}
$\Phi_T$ being an admissible phase (as defined in Definition \ref{def:nonstandardphase}), and $a$ being a semi-classical analytic symbol, elliptic in the symbol class $h^{-3n/4}S^{1,n/4}$ (with the subtlety that $a(\alpha,x)$ is maybe only defined for $\alpha_x$ and $x$ close to each other).

An \emph{adjoint analytic FBI transform} is an operator $S:C^\infty_c(T^\ast M)\to C^\infty(M)$ whose kernel $K_S(x,\beta)$ satisfies that $(\alpha,x) \mapsto \overline{K_S(x,\alpha)}$ is the kernel of an analytic FBI transform.
\end{definition}

\begin{remark}
While it does not appear in the notation, the symbol $a$ is allowed to depend on the small implicit parameter $h > 0$ (but it has to satisfy estimates uniformly in $h$). We say that $\Phi_T$ is the phase of $T$ and that $a$ is its symbol. The class of analytic symbol $S^{1, n/4}$ is defined in Definition \ref{def:symbole_manifold}.

If $S$ is an adjoint analytic FBI transform, we say that $\Phi_S(x,\alpha)$ and $b(x,\alpha)$ are respectively the phase and the symbol of $S$ if $(\alpha,x) \mapsto - \overline{\Phi_T(x,\alpha)}$ and $\p{\alpha,x} \mapsto \overline{b(x,\alpha)}$ are the phase and the symbol of the analytic FBI transform with kernel $(\alpha,x) \mapsto \overline{K_S(x,\alpha)}$.

Recall that the fact that $K_T$ is the kernel of $T$ means that if $u$ is a smooth function on $M$ then $Tu$ is defined by the formula
\begin{equation}\label{eq:def_concrete_T}
\begin{split}
Tu(\alpha) = \int_M K_T(\alpha,x) u(x) \mathrm{d}x.
\end{split}
\end{equation}

Notice that if $T$ is an analytic FBI transform, its adjoint $T^\ast$ with respect to the $L^2$ spaces on $M$ and $T^\ast M$ is an adjoint analytic FBI transform, since the kernel of $T^*$ is given by $\p{x,\alpha} \mapsto \overline{K_{T}(x,\alpha)}$ (the cotangent bundle $T^* M$ is endowed with its canonical volume form).
\end{remark}

\begin{remark}\label{remark:ce_quest_PS}
If $S$ is an adjoint analytic FBI transform and $P$ is a $\G^1$ order $0$ elliptic pseudor, it follows from Proposition \ref{prop:pseudor_lagrangien_alt} in \S \ref{sec:toolbox_gevrey_pseudor} that $PS$ is still an adjoint analytic FBI transform. In fact, if $P$ is a general $\G^1$ pseudor (not necessarily elliptic of order $0$), then $PS$ still satisfies all the points in the definition of an adjoint analytic FBI transform but one: the symbol may not be elliptic in the good symbol class. 

Similarly, if $T$ is an analytic FBI transform and $P$ a $\G^1$ pseudor, then we may write $TP = \p{P^* T^*}^*$. We know that $T^*$ is an adjoint analytic FBI transform and that $P^*$ is a $\G^1$ pseudor (according to Proposition \ref{prop:adjoint_pseudor}). Hence, the discussion above implies that $TP$ satisfies the conditions of Definition \ref{def:analytic-FBI}, except that the symbol $a$ may not be in the good symbol class, and is not necessarily elliptic.
\end{remark}

\begin{remark}\label{remark:pour_localiser}
Notice that if $T$ is an analytic FBI transform, then it satisfies Definition \ref{def:adapted-Lagrangian} with arbitrarily small $\epsilon_1$, up to taking smaller $\eta$ and $\epsilon_0$ (with $\epsilon_0 \ll \epsilon_1$) and larger $C$. This elementary fact will be useful later. It implies that, while $T$ is a global object, its most relevant properties are local.
\end{remark}

Our first result, which is an extension of claims (1.10)-(1.12) in \cite{Sjostrand-96-convex-obstacle}, is the following (its proof may be found in \S \ref{subsec:inversion}). It is not clear \emph{a priori} that the composition of an analytic FBI transform and an adjoint analytic FBI transform makes sense, we will see in Proposition \ref{prop:regularite-decay-s=1} how to define this composition.
\begin{theorem}\label{thm:existence-good-transform}
There exists an analytic FBI transform $T$, and $h_0,R_0>0$ such that for $0<h<h_0$, the composition $T^\ast T$ makes sense and is the identity operator on $(E^{1,R_0})'$.
\end{theorem}

The space $(E^{1,R_0})'$ that appears in Theorem \ref{thm:existence-good-transform} is defined precisely in \S \ref{sec:appendix-Gevrey-regularity}. This is a space of hyperfunctions that may be understood loosely as ``the dual of the space of real-analytic functions with radius of convergence larger than $\frac{1}{R_0}$''. The difference with the transform in \cite{Sjostrand-96-convex-obstacle} is that our transform is globally analytic instead of only near the diagonal, so that it acts on $(E^{1,R_0})'$ instead of just usual distributions. Notice that we will in fact prove that $T^* T$ is well-defined as an operator from $(E^{1,R_0})'$ to $(E^{1,R_1})'$ for some $R_1 < R_0$, and it would then be more precise to say that $T^* T$ is the injection between these spaces.

We will from now on take $T$ to be an analytic FBI transform given by Theorem \ref{thm:existence-good-transform}, and denote $S=T^\ast$. Since its kernel has a holomorphic extension to a complex neighbourhood of $M\times T^\ast M$, for some hyperfunction $u$, it makes sense to consider holomorphic extensions of $Tu$ to $\p{T^* M}_{\epsilon_0}$. We will prove (Proposition \ref{prop:M-vers-Lambda}) that if $u$ is $\G^s$ for some $s\geq 1$, when $\alpha$ is at distance $\lesssim C^{-1} h^{1-1/s}\langle|\alpha|\rangle^{1/s}$ from $T^\ast M$, with $C>0$ depending on $u$,
\[
|T u(\alpha)| \leq C \exp\left(- \left(\frac{\langle|\alpha|\rangle}{C h}\right)^{\frac{1}{s}} \right) \text{ when } \langle|\alpha|\rangle \text{ is large enough.}
\]

The relation between the FBI transform and the symplectic geometry on $T^\ast M$ is quite deep. The canonical symplectic form $\omega = \mathrm{d}\xi\wedge\mathrm{d}x$ extends to $(T^\ast M)_{\epsilon_0}$ as a complex symplectic form. Its imaginary part $\omega_I= \Im \omega$ is a real symplectic form. Following Schapira, we say that a submanifold $\Lambda \subset (T^\ast M)_{\epsilon_0}$ is I-Lagrangian if it is Lagrangian for $\omega_I$. Given a real-valued function $G$, we will denote by $H_G^{\omega_I}$ its Hamiltonian vector field with respect to $\omega_I$, in the sense that $\mathrm{d}G(\cdot) = \omega_I(\cdot,H_G^{\omega_I})$. 

To be able to use almost analytic extensions of Gevrey symbols defined on $T^\ast M$, we will have to restrict our attention to families of Lagrangian submanifolds $(\Lambda_h)_{h>0}$ such that as $h\to 0$,
\begin{itemize}
	\item the $\Lambda_h$'s are uniformly close to $T^* M$ in $C^\infty$ (for the Kohn--Nirenberg metric),
	\item in $C^2$ topology, $\Lambda_h$ is at distance $\leq \tau_0 h^{1-1/s}\langle|\alpha|\rangle^{1/s}$ from $T^\ast M$ (with $\tau_0$ small).
\end{itemize}
A particular case of such families are given by $\Lambda := e^{H^{\omega_I}_G}(T^\ast M)$, where $G$ is a symbol of order $\leq 1/s$, so that $H^{\omega_I}_G$ is $\O(\tau_0 h^{1 - 1/s} \jap{\va{\alpha}}^{1/s})$ (for the Kohn--Nirenberg metric). For such examples, we can find a global function $H$ on $\Lambda$ such that if $\theta=\xi\cdot \mathrm{d}x$ is the holomorphic Liouville one form, $\mathrm{d}H = -\Im \theta_{|\Lambda}$. We will call such examples \emph{$(\tau_0,s)$-Gevrey adapted Lagrangians} (see Definition \ref{def:adapted-Lagrangian} for a precise statement).

In \S \ref{sec:toolbox_gevrey_pseudor}, we have gathered the main elements of a microlocal theory of $\mathcal{G}^s$ pseudors, in loose terms those pseudors whose symbol satisfies estimates of the form
\[
|\partial_x^\alpha\partial_\xi^\beta p| \leq C R^{|\alpha|+|\beta|} (\alpha ! \beta !)^s \langle\xi\rangle^{m-|\beta|}.
\] 
The main result of the present chapter is the following (see Proposition \ref{propmultform} for a slightly more general statement):
\begin{theorem}\label{thm:deforming-Gs-pseudors}
Let $P$ be a $\mathcal{G}^s$ pseudor of order $m$. Denote by $p$ its $\G^s$ principal symbol, and also one of its Gevrey almost analytic extensions. Then, there exists $R,\tau_0>0$ such that, if $\Lambda$ is a $(\tau_0,s)$-Gevrey adapted Lagrangians, then there is $h_0 > 0$ such that for $0 < h <h_0$ and for $u\in E^{1,R}(M)$
\[
\int_{\Lambda} T (P u)(\alpha)\overline{Tu(\alpha)} e^{-2H(\alpha)/h} \mathrm{d}\alpha = \int_{\Lambda} ( p(\alpha) + \O(h \langle|\alpha|\rangle^{m-1})) |T u(\alpha)|^2 e^{-2H(\alpha)/h}\mathrm{d}\alpha.
\]
Additionally, $E^{1,R}(M)$ is dense for all $h>0$ sufficiently small in 
\[
\mathcal{H}_\Lambda^0 := \left\{ u \in (E^{1,R})'\ \middle| \ \int_\Lambda |T u|^2 e^{-2H/h} \mathrm{d}\alpha < \infty\right\}.
\]
\end{theorem}

The space $E^{1,R}(M)$ is a space of analytic functions with radius of convergence $\sim 1/R$ -- see \S \ref{sec:appendix-Gevrey-regularity} for the definition and a reminder on the notion of almost analytic extension. The reason for which we need to introduce the function $H$ in Theorem \ref{thm:deforming-Gs-pseudors} is the following. Certain projectors on the image of $T$ have a structure of FIO with complex phase. However, to apply the techniques of \cite{Melin-Sjostrand-75}, we need  the imaginary part of some phase function to be positive. It is only after the conjugation by $e^{H/h}$ that this property is satisfied -- see Lemma \ref{lemma:PhiTS-Lambda} and more explanations in \S \ref{sec:action-pseudo-FBI}.

\section{Basic properties of the FBI transform}\label{sec:FBI-basics}

\subsection{I-Symplectic geometry and I-Lagrangian functional spaces}
\label{sec:symplectic-geometry}

Let us start by recalling a few facts on the symplectic geometry of $(T^\ast M)_{\epsilon_0}$ for $\epsilon_0 > 0$. For a reference, we suggest \cite{CannasdaSilva-01}. The usual symplectic form $\omega = \mathrm{d}\xi\wedge \mathrm{d}x$ of $T^\ast M$ can be extended to a complex-linear symplectic form on $(T^\ast M)_{\epsilon_0}$, still denoted by $\omega$. The Liouville 1-form $\theta = \xi \cdot \mathrm{d}x$ can also be extended, so that $\omega = \mathrm{d}\theta$ still is an exact form. 

We let $\omega_R = \Re \omega$ and $\omega_I = \Im \omega$. Notice that $\omega_R$ and $\omega_I$ are \emph{real} symplectic forms on $(T^\ast M)_{\epsilon_0}$. In local charts with $\tilde x= x+ i y$, $\tilde\xi = \xi+ i\eta$, the expression for $\omega$ is given by
\[
\omega = \underset{=\omega_R}{\underbrace{\mathrm{d}\xi\wedge \mathrm{d}x  - \mathrm{d}\eta \wedge \mathrm{d}y}} + i \underset{=\omega_I}{\underbrace{( \mathrm{d}\eta \wedge \mathrm{d}x + \mathrm{d}\xi \wedge \mathrm{d}y )}}.
\]
We can also express the Liouville 1-form:
\[
\theta = \xi\cdot \mathrm{d}x - \eta\cdot \mathrm{d}y + i(\xi\cdot \mathrm{d}y + \eta\cdot \mathrm{d}x).
\]
Following Schapira, we will denote with an I objects of symplectic geometric defined through the use of $\omega_I$. For example, the I-Hamiltonian (i.e. w.r.t. $\omega_I$) vector field of a $\mathcal{C}^1$ function $f$ is given in the coordinates above by
\begin{equation}\label{eq:I-hamilton-field}
\begin{split}
H_f^{\omega_I} & = \nabla_\eta f \cdot \frac{\partial }{\partial x} - \nabla_x f \cdot \frac{\partial }{\partial \eta} + \nabla_\xi f \cdot \frac{\partial }{\partial y} - \nabla_y f \cdot \frac{\partial }{\partial \xi} \\
               & = \sum_{j = 1}^n \frac{\partial f}{\partial \eta_j} \frac{\partial}{\partial x_j}  - \frac{\partial f}{\partial x_j} \frac{\partial}{\partial \eta_j} + \frac{\partial f}{\partial \xi_j} \frac{\partial}{\partial y_j} - \frac{\partial f}{\partial y_j} \frac{\partial}{\partial \xi_j},
\end{split}
\end{equation}
so that $\mathrm{d}f = \omega_I(\cdot, H^{\omega_I}_f)$.

One finds directly that $T^\ast M$ is a I-Lagrangian submanifold of $(T^\ast M)_{\epsilon_0}$. The idea of \cite{Helffer-Sjostrand-86} is to replace it by another I-Lagrangian submanifold. However, we will not work with \emph{any} I-Lagrangian subspace of $\p{T^* M}_{\epsilon_0}$, we will only consider adapted I-Lagrangians as we define now (we explain in Remark \ref{remark:lagrangienne_pas_exacte} how one could deal with slightly more general I-Lagrangians).

\begin{definition}\label{def:adapted-Lagrangian}
Let $s \geq 1$ and $\tau_0 \geq 0$. Let $\Lambda$ be an I-Lagrangian in $\p{T^* M}_{\epsilon_0}$. We say that $\Lambda$ is a $\p{\tau_0,s}$-\emph{adapted Lagrangian} if it takes the form
\begin{equation}\label{eq:adapted_Lagrangian}
\begin{split}
\Lambda = e^{H_G^{\omega_I}}\p{T^* M}.
\end{split}
\end{equation}
Here, we assume that $G$ is a real-valued function so that $G_0 := h^{-1+1/s}G$ is a symbol (in the usual Kohn--Nirenberg class of symbol) on $\p{T^*M}_{\epsilon_0}$ of order $1/s$, supported in some $(T^\ast M)_{\epsilon_1}$ with $\epsilon_1<\epsilon_0$. Additionally, we require that (we use the covariant derivatives associated to the Kohn--Nirenberg metric to measure the derivatives of $G_0$)
\begin{equation}\label{eq:premieres_derivees_de_G}
\begin{split}
\sup_{\substack{\alpha \in \p{T^* M}_{\epsilon_0} \\ k \leq 3}} \frac{\n{\nabla^k G_0(\alpha)}_{KN}}{\jap{\va{\alpha}}^{\frac{1}{s}}} \leq \tau_0.
\end{split}
\end{equation}
\end{definition}

\begin{remark}
If $G$ is as in Definition \ref{def:adapted-Lagrangian}, then one easily sees that the vector field $H_G^{\omega_I}$ is complete, so that we can define a $\p{\tau_0,s}$-adapted Lagrangian $\Lambda$ by the formula \eqref{eq:adapted_Lagrangian} (the manifold $\Lambda$ is then I-Lagrangian since $\exp(H_G^{\omega_I})$ is an $I$-symplectomorphism). Notice also that the assumptions on the symbol $G$ impose that it depends on $h$, but in a uniform fashion as $h\to 0$ (in particular, $\Lambda$ is uniformly smooth with respect to the Kohn--Nirenberg metric when $h$ tends to $0$).

In the applications, the dependence of $G$ on $h$ and $\tau_0$ will be fairly explicit since $G$ will be of the form $\tau_0 h^{1 - 1/s} G_0$ with $G_0$ of order $1/s$ satisfying the assumptions of Definition \ref{def:adapted-Lagrangian} with $\tau_0 = h = 1$.

One may notice that if $\tau_1 \geq \tau_0$ and $s \geq \tilde{s}$ then any $\p{\tau_0,s}$-adapted Lagrangian is also $\p{\tau_1,\tilde{s}}$-adapted. 
\end{remark}

\begin{remark}
Let $m \in \R$ and $\Omega$ be a manifold on which there is a notion of Kohn--Nirenberg metric and of Japanese bracket (the main examples are the manifolds $T^* M, (T^* \widetilde{M})_{\epsilon_0}$ and an adapted Lagrangian $\Lambda$), then we define as usual the Kohn--Nirenberg class of symbol $S_{KN}^m\p{\Omega}$ as the space of $\mathcal{C}^\infty$ functions $a : \Omega \to \C$ such that for every $k \in \N$ we have (using the covariant derivative associated to the Kohn--Nirenberg metric):
\begin{equation*}
\begin{split}
\sup_{\alpha \in \Omega} \frac{\n{\nabla^k a(\alpha)}_{KN}}{\jap{\va{\alpha}}^m} < + \infty.
\end{split}
\end{equation*}
For instance, in Definition \ref{def:adapted-Lagrangian} we ask for $G_0 \in S_{KN}^{1/s}\p{\p{T^* M}_{\epsilon_0}}$.

Since the adapted Lagrangians are uniformly smooth submanifolds with respect to the Kohn--Nirenberg metric, and image of $T^\ast M$ under a uniformly smooth flow, the symbol class $S_{KN}^m\p{\Lambda}$ is well defined, and to check that $a\in S_{KN}^m\p{\Lambda}$, we can compute the derivatives either directly on $\Lambda$, or through the pullback by $\exp(H_G^{\omega_I})$, with covariant derivatives or with partial derivatives in coordinates.
\end{remark}

The notion of $\p{\tau_0,s}$-adapted Lagrangian is tailored so that it makes sense to restrict the almost analytic extension of a $\G^s$ symbol (as defined in Remark \ref{remark:aae_for_symbol}) to a $\p{\tau_0,s}$-adapted Lagrangians when $\tau_0$ is small. More precisely, if $G$ is as in Definition \ref{def:adapted-Lagrangian}, it follows from the local expression \ref{eq:I-hamilton-field} for $H_G^{\omega_I}$ that the norm of $H_G^{\omega_I}$ for the Kohn--Nirenberg metric is $\O(\tau_0 h^{1 - 1/s} \jap{\va{\alpha}}^{1/s - 1})$. This essentially proves:
 
\begin{lemma}\label{lmdist}
Let $s \geq 1$ and $T \geq 0$. There is a constant $C > 0$ such that, for every $\tau_0 \in \left[0,T\right]$, if $\Lambda$ is a $\p{\tau_0,s}$-adapted Lagrangian and $\alpha \in \Lambda$ then
\begin{equation*}
|\Im \alpha| \leq C \tau_0 h^{1 - \frac{1}{s}} \jap{\va{\alpha}}^{\frac{1}{s}-1}
\end{equation*}
(On a Grauert tube, we can define $|\Im \alpha|$ in a coordinate-free way, see \S \ref{sec:Grauert}). In particular, for every $\epsilon > 0$, there is a $\tau_1 > 0$ such that, for every $\tau_0 \in \left[0,\tau_1\right]$, any $\p{\tau_0,s}$-adapted Lagrangian is contained in $\p{T^* M}_\epsilon$.
\end{lemma}
In view of this lemma, it will be convenient to consider for $\epsilon > 0$ and $\delta\in [0,1]$,
\begin{equation}\label{eq:def-Omega-tau-delta}
(T^\ast M)_{\epsilon,\delta}:= \left\{ \alpha \in (T^\ast M)_{\epsilon_0}\ |\ |\Im \alpha| \leq \epsilon h^{1-\delta}\langle|\alpha|\rangle^{\delta-1}\right\}.
\end{equation}
Observe that $(T^\ast M)_{\epsilon,1}=(T^\ast M)_{\epsilon}$ for $\epsilon$ small enough.

In order to define weighted $L^2$ spaces on adapted Lagrangians, we need to define a volume form on such Lagrangians. Since we have the Kohn--Nirenberg metric on Lagrangians, we could take the associated volume form; however it is not intrinsically related to the geometry of the problem, so it is more relevant to use another one. To do so, we use the following lemma.
\begin{lemma}\label{lemma:uniformity-lagrangians}
Let $s \geq 1$. There exists $T > 0$ and a constant $C > 0$ such that, for every $\tau_0 \in \left[0,T\right]$, if $\Lambda$ is a $\p{\tau_0,s}$-adapted Lagrangian then the restriction of $\omega_R$ to $\Lambda$ (that we will also denotes by $\omega_R$) is symplectic. 

Moreover, if $G$ is as in Definition \ref{def:adapted-Lagrangian}, and $J_G$ denotes the Jacobian of $\exp(H_G^{\omega_I})$ from $T^* M$ to $\Lambda$ (both Lagrangians are endowed with the volume form $\omega_R^n/n!$), then for every $\alpha \in T^* M$ we have
\begin{equation*}
\begin{split}
\va{J_G(\alpha) - 1} \leq C \tau_0 h^{1 - \frac{1}{s}} \jap{\va{\alpha}}^{\frac{1}{s} - 1}.
\end{split}
\end{equation*}
\end{lemma}

\begin{proof}
Let us use the notations from Definition \ref{def:adapted-Lagrangian}. The first idea here would be to try and prove that $\exp(H^{\omega_I}_G)^\ast \omega_R$ is close to $\omega_R$. For this, one would compute in local coordinates
\begin{align*}
\mathcal{L}_{H^{\omega_I}_G} \omega_R 	&= \mathrm{d}(\imath_{H^{\omega_I}_G} \omega_R), \\
										&= \mathrm{d}\left[ \nabla_\xi G \cdot \mathrm{d}\eta - \nabla_\eta G \cdot \mathrm{d}\xi + \nabla_x G \cdot \mathrm{d}y - \nabla_y G \cdot \mathrm{d}x \right], \\
										&= i \mathrm{d}(\partial - \overline{\partial}) G = i (-\partial\overline{\partial} +\overline{\partial}\partial)G,\\
										&= - 2 i \partial\overline{\partial}G.
\end{align*}
If $G_{|T^\ast M}$ is real-analytic, and $G = \Re \tilde{G}$, where $\tilde{G}$ is a holomorphic extension of the restriction of $G$ to $T^\ast M$, then this vanishes. However, in general, the coefficients of $\partial\overline{\partial} G$ are symbols of order $1/s$, and thus not bounded: it seems we cannot deduce anything directly from this identity.

However, this problems stems from the fact that we are not measuring the size of $\mathcal{L}_{H^{\omega_I}_G} \omega_R$ in a sensible way. Indeed, the Kohn--Nirenberg metric gives us a metric on the bundle of forms, so that
\[
\mathrm{d}x_j, \mathrm{d}y_j, \frac{\mathrm{d}\xi_j}{\langle|\alpha|\rangle}, \text{ and }  \frac{\mathrm{d}\eta_j}{\langle|\alpha|\rangle},\ j=1\dots n,
\]
form a basis in which the matrix (and its inverse) of $g_{KN}$ is bounded. We can thus replace it by the metric that makes them an orthogonal basis. Computing norms with this new metric, we find
\[
\| \partial\overline{\partial}G \| = \O\p{\tau_0 h^{1 - \frac{1}{s}} \langle|\alpha|\rangle^{\frac{1}{s}}},
\]
and
\[
\| \omega_R \| = \O( \langle|\alpha|\rangle),\ \| \omega_R^{-1}\| = \O( \langle|\alpha|\rangle^{-1}).
\]
According to Taylor's formula, $(\omega_R)_{|\Lambda}$ is thus symplectic, and the Jacobian $J_G$ close to $1$.
\end{proof}

From now on, if $\Lambda$ is an adapted Lagrangian, we will just denote by $\mathrm{d}\alpha$ the $2n$-form $\omega_R^n/n!$, which induces a volume form on $\Lambda$. We will denote the corresponding duality pairing
\begin{equation}\label{eqpairing}
\jap{f,g}_{\Lambda} = \int_{\Lambda} f g\ \mathrm{d}\alpha.
\end{equation}
The natural space in our setting will not be $L^2(\Lambda, \mathrm{d}\alpha)$ but rather $L^2(\Lambda, e^{-2H/h}\mathrm{d}\alpha)$, where $H$ is an action associated with $\Lambda$, solving
\begin{equation}\label{eq:propH}
\mathrm{d} H = - \Im \theta_{|\Lambda}.
\end{equation}
Since $\Lambda$ is I-Lagrangian, we deduce that there are local solutions to this equation. However, since $\Lambda$ is assumed to be of the form \eqref{eq:adapted_Lagrangian}, we can find an explicit \emph{global} solution, given by 
\begin{equation}\label{eq:explicit_solution_H}
H := \int_0^1 \p{e^{(\tau-1)H_G^{\omega_I}}}^*(G - \Im \theta(H_G^{\omega_I})) \mathrm{d} \tau.
\end{equation}
For a proof, we follow the arguments after equation 1.17 in \cite{Sjostrand-96-convex-obstacle}. Observe that
\[
\mathrm{d}( \Im \theta(H_G^{\omega_I}) ) = \Im \mathrm{d}( \imath_{H^{\omega_I}_G} \theta)  = \Im( \mathcal{L}_{H^{\omega_I}_G} \theta - \imath_{H^{\omega_I}_G} \mathrm{d}\theta) = \mathcal{L}_{H^{\omega_I}_G} \Im\theta + \mathrm{d}G.
\]
In particular, we get
\begin{align*}
\mathrm{d}H = \int_0^\tau \p{e^{(\tau-1)H_G^{\omega_I}}}^* \mathcal{L}_{H^{\omega_I}_G}(- \Im\theta ) \mathrm{d} \tau =  -\Im \theta + \p{e^{- H_G^{\omega_I}}}^*  \Im\theta.
\end{align*}
The second term vanishes when restricting to $\Lambda$ because $\Im \theta$ vanishes on $T^\ast M$.  Computing in local coordinates, we get
\[
\Im \theta (H_G^{\omega_I}) = \Im[ (\xi + i \eta)(\nabla_\eta G + i\nabla_\xi G)] = \eta \cdot \nabla_\eta G + \xi \cdot \nabla_\xi G.
\]
Hence, $\Im \theta (H_G^{\omega_I})$ is a symbol of order $1/s$ and, from \eqref{eq:I-hamilton-field}, we get that the derivative $H_G^{\omega_I}\p{\Im \theta\p{H_G^{\omega_I}}}$ is a symbol of order $- 1+2/s$.  In particular, we find that
\begin{equation}\label{eq:approximation_of_H}
H = G - \Im \theta(H_G^{\omega_I}) + \O_{\mathcal{C}^{1}}\p{\p{\tau_0 h^{1 - \frac{1}{s}}}^2\langle |\alpha|\rangle^{\frac{2}{s} - 1}}.
\end{equation}
The explicit formula \eqref{eq:explicit_solution_H} for $H$ has another consequence. For every $s \geq 1$ and $T \geq 0$, there is a constant $C > 0$ such that if $\Lambda$ is $\p{\tau_0,s}$-adapted Lagrangian with $\tau_0 \leq T$ then, for every $\alpha \in \Lambda$, we have
\begin{equation}\label{eq:taille_de_H}
\begin{split}
\va{H(\alpha)} \leq C \tau_0 h^{1- \frac{1}{s}} \jap{\va{\alpha}}^{\frac{1}{s}}.
\end{split}
\end{equation} 

Let us give a word on the simplest way the symbol $G$ in Definition \ref{def:adapted-Lagrangian} can be chosen. Let us consider a real valued symbol $G_0$ of order $1/s$ on $T^\ast M$. Denoting by $\tilde{G}_0$ one of its almost analytic extensions, we put $G=\tau_0 h^{1 - 1/s}\Re \tilde{G}_0$. In that case, if $p\in S^m$ is another real valued symbol with almost analytic extension $\widetilde{p}$, we find that for $\alpha\in T^\ast M$ and $\tau$ near $0$ we have
\begin{equation}\label{eq:shifted-symbol}
\widetilde{p}(e^{\tau H_G^{\omega_I}}(\alpha)) = p(\alpha) + i \tau_0 h^{1 - 1/s} \{G_0, p\}(\alpha) + \O(\tau_0^2 h^{2 - \frac{2}{s}} \langle |\alpha|\rangle^{m + \frac{2}{s} - 2}),
\end{equation}
where
\[
\{ G_0, p\} = \nabla_\xi G_0 \cdot \nabla_x p - \nabla_x G_0 \cdot \nabla_\xi p.
\]
More generally, to obtain Formula \eqref{eq:shifted-symbol}, instead of taking $G=\Re \tilde{G}_0$, one can take, as Sj\"ostrand did in \cite{Sjostrand-96-convex-obstacle}, any symbol $G$ that extends $G_0$, and such that $\partial_{y,\eta} G =0$ on $T^\ast M$. 

When $m=1$, in the case that $s \geq 2$, the remainder in \eqref{eq:shifted-symbol} is the symbol of a bounded operator. However, we will be interested in the case that $s$ can be any number in $\left[1,+ \infty\right[$. Then, the remainder is not bounded as $\alpha\to \infty$. To circumvent this, in Section \ref{sec:escape_function}, we will build $G$ directly on $(T^\ast M)_{\epsilon_0}$. This construction will not yield a $G$ such that $\nabla_{y,\eta} G =0$ on $T^\ast M$. Since Sj\"ostrand only considered compactly supported deformations, he could ignore this subtlety. In \cite{Sjostrand-96-convex-obstacle}, the fact that $\nabla_{y,\eta} G$ vanishes on $T^\ast M$ served two purposes. The first is Formula \eqref{eq:shifted-symbol}, and the second is a quick proof that the Jacobian of $\exp( H^{\omega_I}_G)$ is close to $1$. For the rest, that assumption is not necessary, and we can remove it.

Equipped with this viaticum of symplectic geometry, we can define the functional spaces associated with adapted Lagrangians. Beforehand, we have to check that there actually exists an analytic FBI transform in the sense of Definition \ref{def:analytic-FBI}. First, we recall that
\[
-\langle \alpha_\xi, \exp_{\alpha_x}^{-1}(x)\rangle + i\frac{\langle\alpha\rangle}{2}d(x,\alpha_x)^2
\]
defines an admissible phase in the sense of Definition \ref{def:nonstandardphase}. From now on, we fix an admissible phase $\Phi_T$ (not necessarily equal to the phase above).
\begin{lemma}\label{lemma:existence-transform}
Let $a$ be an elliptic symbol in $h^{- \frac{3n}{4}} S^{1,\frac{n}{4}}$. There exist an analytic FBI transform with symbol $a$ and phase $\Phi_T$.
\end{lemma}

\begin{proof}
We consider $e^{i\Phi_T(\alpha,x)/h} a(\alpha,x)$. It is holomorphic and well-defined when $\p{\alpha,x} \in \p{T^* M}_\epsilon \times \p{M}_\epsilon$ satisfies $d(x,\alpha_x) < \delta$, for some $\epsilon,\delta > 0$. Moreover, for some $C,\epsilon_1>0$, it is smaller than $C \exp(- C^{-1} \langle|\alpha|\rangle/h)$ when $d(\Re x, \Re \alpha_x) > \epsilon_1$. Hence, if we choose a bump function $\chi : M \times M \to \left[0,1\right]$ such that $\chi(x,y) = 1$ if $d(x,y) \leq \epsilon_1$ and $\chi(x,y) = 0$ if $d(x,y) \geq \delta/2$ (we may assume that $\epsilon_1$ is arbitrarily small), then the Cauchy--Riemann operator applied to
\begin{equation}\label{eq:approx_K_T}
\begin{split}
\p{\alpha,x} \mapsto \chi\p{\Re \alpha_x, \Re x} e^{i \frac{\Phi_T\p{\alpha,x}}{h}} a(\alpha,x)
\end{split}
\end{equation}
gives an $\O(\exp(- C^{-1} \jap{\va{\alpha}}/h))$ for $\p{\alpha,x} \in \p{T^* M}_\epsilon \times \p{M}_\epsilon$. In particular, we can apply Lemma \ref{lemma:approximation_analytique} to \eqref{eq:approx_K_T}, and we find a globally analytic kernel $K_T(\alpha,x)$ satisfying the desired properties.
\end{proof}

From now on, in this section, we will assume that $T$ is some fixed analytic FBI transform, and $S$ a fixed adjoint analytic FBI transform (see Definition \ref{def:analytic-FBI}), with symbols and phases respectively $a$ and $\Phi_T$ and $b$ and $\Phi_S$. Since $K_T$ is analytic, there is a $R>0$ such that $T u(\alpha)$ is well-defined by Formula \eqref{eq:def_concrete_T} for $u \in (E^{1,R})'$, and $\alpha\in (T^\ast M)_{\epsilon_0}$ (where $\epsilon_0$ is from Definition \ref{def:analytic-FBI}). 

We will also, for $\tau_0 > 0$ small, fix $\Lambda$ a $(\tau_0,1)$-Gevrey adapted Lagrangian (with associated symbol $G\in S^{1}_{KN}(\p{T^* M}_{\epsilon_0})$, and action $H$). Notice that, since $\Lambda \subset \p{T^* M}_{\epsilon_0}$, for $u \in (E^{1,R})'$ the FBI transform $Tu$ is well-defined on $\Lambda$ (provided that $\tau_0$ is small enough).

We define the FBI transform $T_\Lambda$ associated with $\Lambda$ by the formula $T_\Lambda u = \p{Tu}_{|\Lambda}$. We also define $S_\Lambda$ on $C^\infty_c(\Lambda)$ by
\[
S_\Lambda v(x)= \int_\Lambda K_S(x,\alpha) v(\alpha)\mathrm{d}\alpha.
\]
We can now define the scale of spaces that we are going to use in the following. First of all, let us denote for $k \in \R$ the weighted $L^2$ space on $\Lambda$
\begin{equation}\label{eq:def_espace_L2}
\begin{split}
L^2_k\p{\Lambda} = L^2\p{\Lambda, \jap{\va{\alpha}}^{2k} e^{- \frac{2H}{h}} \mathrm{d}\alpha},
\end{split}
\end{equation}
where we recall that $\mathrm{d}\alpha$ denotes the volume form associated with the symplectic form $\omega_R$ on $\Lambda$ and that $H$ is the action defined by \eqref{eq:explicit_solution_H} and that satisfies \eqref{eq:propH}. For $k \in \R$ we define then the space (for $R > 0$ large enough) 
\begin{equation}\label{eq:def-HMk}
\mathcal{H}_\Lambda^k = \set{u \in (E^{1,R})' : T_\Lambda u \in L^2_k\p{\Lambda}}
\end{equation}
endowed with the norm (we will see later that this is indeed a norm)
\begin{equation}\label{eq:def-norme-HMk}
\n{u}_{\mathcal{H}_{\Lambda}^k} = \n{T_\Lambda u}_{L^2_k\p{\Lambda}}
\end{equation}
and its analogue on the FBI side
\begin{equation}\label{eq:def-HLambdak}
\mathcal{H}_{\Lambda,\FBI}^k = \set{T_\Lambda u : u \in \mathcal{H}_{\Lambda}^k}\subset L^2_k(\Lambda).
\end{equation}
We will also need the spaces
\begin{equation}\label{eq:def_H_infini}
\begin{split}
\mathcal{H}_\Lambda^\infty = \bigcap_{k \in \R} \mathcal{H}_\Lambda^k \textup{ and } \mathcal{H}_{\Lambda,\FBI}^\infty = \bigcap_{k \in \R} \mathcal{H}_{\Lambda,\FBI}^k.
\end{split}
\end{equation}
We cannot say much about these spaces yet, but their basic properties will follow from the study of the FBI transform $T_\Lambda$ in \S \ref{sec:continuity-T-S} and \S \ref{subsec:inversion} below. Before that, let us end this section with two remarks concerning the classes of manifolds $\Lambda$ allowed in Definition \ref{def:adapted-Lagrangian}.

\begin{remark}\label{remark:lagrangienne_exacte}
Definition \ref{def:adapted-Lagrangian} of adapted Lagrangians may seem restrictive. There are two fundamental items that we will need in order to work with a manifold $\Lambda$. First, we need to control the distance between $\Lambda$ and $T^* M$ (this is done using Lemma \ref{lmdist}) and ensure that some transversality conditions satisfied by $T^* M$ remain true for $\Lambda$. Hence, it is natural to ask for $\Lambda$ to be close to $T^* M$ in some $\mathcal{C}^k$ sense.

The second required item is a global solution to the equation \eqref{eq:propH} on $\Lambda$ (we discuss in Remark \ref{remark:lagrangienne_pas_exacte} how one could relax this condition). This is ensured by the fact that the Lagrangian $\Lambda$ is of the form \eqref{eq:adapted_Lagrangian} for a real-valued function $G$. Let us now explain briefly why it is reasonable to ask for $\Lambda$ to be of the form \eqref{eq:adapted_Lagrangian} if we want \eqref{eq:propH} to have a global solution $H$ on $\Lambda$.

By the Weinstein Tubular Neighbourhood Theorem, 
$T^* M$ has a neighbourhood in $\p{T^* M}_{\epsilon_0}$ which is symplectomorphic (when endowed with $\omega_I$) to a neighbourhood of the zero section in $T^*\p{T^* M}$ (endowed with its canonical symplectic form). Then, if $\Lambda$ is a I-Lagrangian $\mathcal{C}^\infty$ close to $T^* M$, it corresponds through this symplectomorphism to a Lagrangian submanifold $\widetilde{\Lambda}$ of $T^* \p{T^* M}$ close to the zero section. Hence $\widetilde{\Lambda}$ is the graph of a $1$-form $\gamma$ on $T^* M$ (we ignore non-compactness issues in this informal discussion). The fact that $\widetilde{\Lambda}$ is Lagrangian implies that $\gamma$ is closed. In these new coordinates, the existence of a global solution to \eqref{eq:propH} is equivalent to the exactness of $\gamma$. Hence, there is a function $f$ defined on $T^* M$ such that $\mathrm{d}f = - \gamma$. Now, if we define $G$ on $T^* \p{T^* M}$ by $G(x,\xi) = f(x)$, we see that $\widetilde{\Lambda}$ is the image of $T^* M$ by the time $1$ map of the Hamiltonian flow defined by $G$.
 
Consequently, having $\Lambda$ of the form \eqref{eq:adapted_Lagrangian} is the exact geometric condition that we need if we want \eqref{eq:propH} to have a global solution $H$ on $\Lambda$.
\end{remark}

\begin{remark}\label{remark:lagrangienne_pas_exacte}
As we explained in Remark \ref{remark:lagrangienne_exacte}, we make the assumption that $\Lambda$ is of the form \eqref{eq:adapted_Lagrangian} in order to ensure that \eqref{eq:propH} as a global solution $H$ on $\Lambda$. Let us describe a possible way to work without the existence of such a $H$. Since $ - \Im \theta$ is closed, it defines an element of the homology group $H^1\p{\Lambda}$. Thus, we can use it to define a line bundle over $\Lambda$ in the following way: if $\widehat{\Lambda}$ denotes the universal cover of $\Lambda$ then we define an action of the fundamental group $\pi_1\p{\Lambda}$ of $\Lambda$ on $\widehat{\Lambda} \times \C$ by
\begin{equation*}
\begin{split}
c \cdot \p{x,u}  = \p{c \cdot x, e^{\frac{\int_{c} \rho }{h}}u},
\end{split}
\end{equation*}
where $c \cdot x$ denotes the action of $\pi_1\p{\Lambda}$ on $\widehat{\Lambda}$ and $\rho$ is the lift of $- \Im \theta $ to $\widehat{\Lambda}$. Then, the quotient of $\widehat{\Lambda} \times \C$ by this action defines a complex line bundle $L$ over $\Lambda$ (that may depend on $h$). Then, let $H : \widehat{\Lambda} \to \R$ be such that $\mathrm{d}H = \rho$ and defines the map $A$ from $\widehat{\Lambda} \times \C$ to itself by
\begin{equation*}
\begin{split}
A\p{x,u} = \p{x,e^{- \frac{H(x)}{h}} u},
\end{split}
\end{equation*}
and notice that for $\p{x,u} \in \widehat{\Lambda} \times \C$ we have
\begin{equation*}
\begin{split}
A\p{ c \cdot \p{x,u}} & = \p{c \cdot x, e^{\frac{- H(c \cdot x) + \int_c \rho}{h}} u } \\
      & = \p{c \cdot x, e^{- \frac{H(x)}{h}} u} = c \ast A\p{(x,u)},
\end{split}
\end{equation*}
where $\ast$ denotes the action of $\pi_1\p{\Lambda}$ on the first coordinate only (that is $c \ast \p{x,u} = \p{c \cdot x,u}$). Consequently, $A$ defines in the quotient a vector bundle morphism $\mathcal{A} : L \to \Lambda \times \C$ that can be used to replace the multiplication by $e^{- \frac{H}{h}}$. For instance, instead of working with the space $L^2(\Lambda, e^{-2H/h} \mathrm{d}\alpha)$ (as we will do below), we would have to work in this case with the space of sections $u$ of $L$ such that $\mathcal{A}u$ belongs to $L^2\p{\Lambda,\mathrm{d}\alpha}$. 

Of course, to deal with such a Lagrangian, we should use a FBI transform that sends functions on $M$ to sections of $L$. 
It is likely that most of the analysis may be adapted to this more general case. Notice however that some additional difficulties arise (for instance, one has to say something about the dependence of $L$ on $h$ when applying H\"ormander's solution to the $\bar{\partial}$ equation in order to construct the kernel of the FBI transform). 
Since we were not aware of any problem that would require to work with Lagrangians that are not of the form \eqref{eq:adapted_Lagrangian}, we did not pursue this line of work.
\end{remark}

\subsection{The FBI transform acting on Gevrey functions and ultradistributions}
\label{sec:continuity-T-S}

\subsubsection{Gevrey regularity on the FBI side}

To simplify manipulations later on, let us here study the continuity of analytic FBI transforms on some functional spaces. We will develop the idea that the regularity of an ultradistribution $u$ on $M$ translates into decay for its FBI transform $Tu$.

The notations are the same as in the previous paragraph \S \ref{sec:symplectic-geometry}: $T$ is some fixed analytic FBI transform, and $S$ a fixed adjoint analytic FBI transform (see Definition \ref{def:analytic-FBI}), with symbols and phases respectively $a$ and $\Phi_T$ and $b$ and $\Phi_S$. For some small $\tau_0 > 0$, we denote by $\Lambda$ a $(\tau_0,1)$-Gevrey adapted Lagrangian (with associated symbol $G\in S^{1}_{KN}\p{\p{T^* M}_{\epsilon_0}}$, and action $H$).

In order to highlight the idea that the regularity of an ultradistribution on $M$ may be understood through the decay of its FBI transform, let us introduce the following norms.  Let $s \geq 1$, $r \in \R$. If $f$ is a function from $\Omega\subset(T^\ast M)_{\epsilon_0}$ to $\C$, let
\begin{equation*}
\n{f}_{\Omega,s,r} := \sup_{\alpha \in \Omega} \va{f(\alpha)} \exp\p{- r \jap{ \va{\alpha}}^{\frac{1}{s}}} \in \R_+ \cup \set{+ \infty}.
\end{equation*}
Then we define $F^{s,r}\p{\Omega}$ to be the Banach space of continuous function $f : \Omega \to \C$ such that $\n{f}_{\Omega,s,r} < + \infty$.

We define then the following spaces of functions on the Lagrangian $\Lambda$. For $s \geq 1$ introduce the space of functions decaying at least exponentially 
\begin{equation*}
\GG^s\p{\Lambda} \coloneqq \bigcup_{r < 0} F^{s,r} \p{\Lambda},
\end{equation*}
and the space of functions diverging slower than any exponential
\begin{equation*}
\UU^s\p{\Lambda} \coloneqq \bigcap_{r>0} F^{s,r} \p{\Lambda}.
\end{equation*}
These spaces are endowed respectively with the inductive and projective limit structure (in the category of locally convex topological vector spaces). Notice that $\GG^s\p{\Lambda}$ is dense in $\UU^s\p{\Lambda}$ (just multiply by a bump function) and that the $L^2$ pairing \eqref{eqpairing} in $\Lambda$ gives a natural duality bracket between $\UU^s\p{\Lambda}$ and $\GG^s\p{\Lambda}$. The spaces $\GG^s\p{\Lambda}$ and $\UU^s\p{\Lambda}$ are natural analogues of $\G^s\p{M}$ and $\U^s\p{M}$ on the FBI side. Indeed, we have the following results. 
\begin{prop}\label{prop:M-vers-Lambda}
Let $\tilde{s} > s \geq 1$ and assume that $\Lambda$ is a $(\tau_0,\tilde{s})$-Gevrey adapted Lagrangian with $\tau_0$ small enough. Then, the transform $T_\Lambda$ is continuous from $\G^s\p{M}$ to $\GG^s\p{\Lambda}$ and from $\U^s\p{M}$ to $\UU^s\p{\Lambda}$.
\end{prop}

\begin{prop}\label{prop:Lambda-vers-M}
Let $\tilde{s} > s \geq 1$ and assume again that $\Lambda$ is a $(\tau_0,\tilde{s})$-Gevrey adapted Lagrangian with $\tau_0$ small enough. Then, the transform $S_\Lambda$ is continuous from $\GG^s\p{\Lambda}$ to $\G^s\p{M}$ and admits a continuous extension from $\UU^s\p{\Lambda}$ to $\U^{s}\p{M}$.
\end{prop}

If Propositions \ref{prop:M-vers-Lambda} and \ref{prop:Lambda-vers-M} give a good idea of the link between regularity on $M$ and decay on the FBI side, they are not precise enough for what we intend to do. In particular, when dealing with $\G^s$ Anosov flows, we want to consider $\p{\tau_0,s}$-adapted Lagrangians, in order to get the best results possible. Thus, we need more precise estimates that explain how $\p{\tau_0,s}$-adapted Lagrangians relate with $\G^s$ functions and associated ultradistributions. We will then deduce Propositions \ref{prop:M-vers-Lambda} and \ref{prop:Lambda-vers-M} from these estimates.

\subsubsection{Explicit estimates.}

The main goal of this section set, we go into details. We want to understand first the growth of $T_\Lambda u$ when  $u$ is either a Gevrey function or an ultradistribution. Since $T_\Lambda u$ is the restriction of $T u$ to $\Lambda$, and in view of Lemma \ref{lmdist}, we will study the growth of $Tu$ inside a sub-conical neighbourhood of $T^* M$, as defined in \eqref{eq:def-Omega-tau-delta}.

We start by studying $Tu$ when $u$ is smooth. To do so, we use the non-stationary phase method -- Proposition \ref{prop:non_stationary_analytic_symbols}. Since the phase of $T$ is non-stationary only for large $\alpha$, we need another bound for small $\alpha$'s. We  use the following elementary bound: there is $C > 0$ such that, for every $\alpha \in (T^* M)_{\epsilon_0}$ and every bounded function $u$ on $M$, we have
\begin{equation}\label{eq:trivial_bound_T}
\begin{split}
\va{Tu (\alpha)} \leq C \n{u}_{L^\infty\p{M}} h^{- \frac{3n}{4}} \jap{\va{\alpha}}^{\frac{n}{4}} \exp\p{C \frac{\jap{\va{\alpha}} \va{\Im \alpha}}{h}}.
\end{split}
\end{equation}
The bound \eqref{eq:trivial_bound_T} follows from an immediate majoration of the kernel of $T$, using the fact that the imaginary part of $\Phi_T$ is positive on $T^* M \times M$. Notice also that, since $Tu$ is holomorphic, the bound \eqref{eq:trivial_bound_T} is in fact a symbolic estimate.

To understand $Tu(\alpha)$ for large $\alpha$ when $u$ is Gevrey, we will rely on the following estimate.

\begin{lemma}\label{lemma:T_non_stationary}
Let $s \geq 1$. For every $R > 0$, there are constants $C,\tau_1 > 0$ such that, for every $\alpha \in \p{T^* M}_{\tau_1,1/s}$ such that $\jap{\va{\alpha}} \geq C$, and every $u \in \widetilde{E}^{s,R}\p{M}$, we have
\begin{equation}\label{eq:T_non_stationary}
\begin{split}
\va{Tu (\alpha)} \leq C \n{u}_{s,R,M} \exp\p{- \frac{1}{C} \p{\frac{\jap{\va{\alpha}}}{h}}^{\frac{1}{s}}}.
\end{split}
\end{equation}
In particular, if $r \geq - 1 / C h ^{1/s}$ and if $\Lambda$ is a $\p{\tau_0,s}$-adapted Lagrangian with $\tau_0$ small enough, then $T_\Lambda$ is bounded from $\widetilde{E}^{s,R}\p{M}$ to $F^{s,r}\p{\Lambda}$.
\end{lemma}

Then, we want to understand the growth of $Tu(\alpha)$ when $u$ is an ultradistribution. To do so, we only need to understand the size of the kernel of $T$ in Gevrey norms. We will prove the following estimate, using the Bochner--Martinelli Trick, Lemma \ref{lemma:Bochner-Martinelli-trick}.

\begin{lemma}\label{lemma:T_BMT}
Let $s \geq 1$. For every $\epsilon > 0$, there are constants $C,\tau_1, R > 0$ such that for every $\alpha \in (T^* M)_{\tau_1,1/s}$ we have
\begin{equation}\label{eq:bound_gevrey_noyau_T}
\begin{split}
\n{y \mapsto K_T(\alpha,y)}_{s,R,M} \leq C \exp\p{\epsilon\p{\frac{\jap{\va{\alpha}}}{h}}^{\frac{1}{s}}}.
\end{split}
\end{equation}
In particular, for $u \in (E^{s,R}\p{M})'$ and $\alpha \in (T^* M)_{\tau_1,1/s }$, we have
\begin{equation*}
\begin{split}
\va{Tu(\alpha)} \leq C \n{u}_{\p{E^{s,R}}'} \exp\p{\epsilon\p{\frac{\jap{\va{\alpha}}}{h}}^{\frac{1}{s}}}.
\end{split}
\end{equation*}
Consequently, if $r \geq \epsilon h^{- 1/s}$ and $\Lambda$ is a $\p{\tau_0,s}$-adapted Lagrangian for $\tau_0$ small enough, then $T_\Lambda$ is bounded from $(E^{s,R}\p{M})'$ to $F^{s,r}\p{\Lambda}$.
\end{lemma}

The same argument gives an estimate on the kernel of $S$.

\begin{lemma}\label{lemma:S_BMT}
Let $s \geq 1$. For every $\epsilon > 0$, there are constants $C,\tau_1, R > 0$ such that for every $\alpha \in (T^* M)_{\tau_1,1/s}$ we have
\begin{equation}\label{eq:bound_gevrey_noyau_S}
\begin{split}
\n{y \mapsto K_S(\alpha,y)}_{s,R,M} \leq C \exp\p{\epsilon\p{\frac{\jap{\va{\alpha}}}{h}}^{\frac{1}{s}}}.
\end{split}
\end{equation}
In particular, if $r < - \epsilon h^{- 1/s}$ and $\Lambda$ is a $\p{\tau_0,s}$-adapted Lagrangian with $\tau_0$ small enough then $S_\Lambda$ is bounded from $F^{s,r}\p{\Lambda}$ to $\G^s\p{M}$.
\end{lemma}

It is understood that, since $Tu$ is holomorphic, the estimates from Lemmas \ref{lemma:T_non_stationary} and \ref{lemma:T_BMT} are in fact $\mathcal{C}^\infty$ estimates, due to Cauchy's formula. 

\begin{proof}[Proof of Lemma \ref{lemma:T_non_stationary}]
Choose $\tau_1 \ll 1$ and, for $\alpha \in \p{T^* M}_{\tau_1,1/s}$, write
\begin{equation}\label{eq:rappel-def-T}
\begin{split}
T u (\alpha) = \int_{M} K_T(\alpha,y) u (y) \mathrm{d}y.
\end{split}
\end{equation}
Then, we choose a small ball $D$ in $M$ and assume that $\Re \alpha_x$ remains in $D$, uniformly away from the boundary of $D$ (at a distance much larger than $\tau_1$). The results for all $\alpha$ follows then immediately from a compactness argument. We split then the integral \eqref{eq:rappel-def-T} into the integral over $D$ and the integral over $M \setminus D$, that we denote respectively by $T^D u(\alpha)$ and $T^{M \setminus D}u(\alpha)$. The quantity $T^{M \setminus D} u(\alpha)$ is defined by integration over $y$'s that remain uniformly away from $\Re \alpha_x$ and hence the definition of $K_T$ implies that (recall Remark \ref{remark:pour_localiser}, we bound the $L^\infty$ norm by the $\n{\cdot}_{s,R,M}$ norm)
\begin{equation}\label{eq:borne-TMD}
\begin{split}
\va{T^{M \setminus D} u \p{\alpha}} \leq C \n{u}_{s,R,M} \exp\p{ - \frac{\jap{\va{\alpha}}}{Ch}}.
\end{split}
\end{equation}
Consequently, we may focus on $T^D u (\alpha)$. Up to an error term which satisfies the same kind of bound as $T^{M \setminus D}u(\alpha)$, and that we will consequently ignore, we have
\begin{equation}\label{eq:presque-TD}
\begin{split}
T^D u (\alpha) = \int_{D} e^{i \frac{\Phi_T(\alpha,y)}{h}} a(\alpha,y) u(y) \mathrm{d}y.
\end{split}
\end{equation}
By taking $D$ small enough, we may work in local coordinates. We want now to apply Proposition \ref{prop:non_stationary_analytic_symbols}. To do so, we need to check the requirements (i)-(iii). Item (i) is an immediate consequence of the point (i) in Definition \ref{def:nonstandardphase} of an admissible phase that is satisfied by $\Phi_T$. Item (ii) follows from  the point (iv) in Definition \ref{def:nonstandardphase} and the fact that $\alpha_x$ remains uniformly away from the boundary of $D$. Finally, item (iii) is satisfied because $\mathrm{d}_y\Phi_T(\alpha,\alpha_x) = - \alpha_\xi$.

We can consequently apply Proposition \ref{prop:non_stationary_analytic_symbols}, ending the proof of the lemma.
\end{proof}

\begin{proof}[Proof of Lemma \ref{lemma:T_BMT}]
We will apply the Bochner--Martinelli trick Lemma \ref{lemma:Bochner-Martinelli-trick} to the function
\begin{equation*}
\begin{split}
f_{\alpha,h} : y \mapsto K_T(\alpha,y) \exp\p{- \epsilon \p{\frac{\jap{\va{\alpha}}}{h}}^{\frac{1}{s}}}
\end{split}
\end{equation*}
with ``$\lambda = \jap{\va{\alpha}}/ h$''. Since the imaginary part of $\Phi_T(\alpha,y)$ is non-negative when $\alpha$ and $y$ are real, we see that if $\alpha \in (T^* M)_{\tau_1,1/s}$ and $y$ is at distance at most $\tau_1 (\jap{\va{\alpha}}/h)^{1/s - 1}$ of $M$ we have
\begin{equation*}
\begin{split}
\va{f_{\alpha,h}(y)} \leq C \exp\p{- \frac{\epsilon}{2} \p{\frac{\jap{\va{\alpha}}}{h}}^{\frac{1}{s}}},
\end{split}
\end{equation*}
for some $C > 0$. Point (ii) in Lemma \ref{lemma:Bochner-Martinelli-trick} is trivially satisfied since $f_{\alpha,h}$ is holomorphic. We can consequently apply the Bochner--Martinelli Trick with ``$\lambda = \jap{\va{\alpha}}/ h$'' to find that $f_{\alpha,h}$ is uniformly bounded in $E^{s,R}\p{M}$ for some $R > 0$. The bound \eqref{eq:bound_gevrey_noyau_T} follows.
\end{proof}

The proof of Lemma \ref{lemma:S_BMT} is similar and we omit it. We are now in position to prove Propositions \ref{prop:M-vers-Lambda} and \ref{prop:Lambda-vers-M}.

\begin{proof}[Proof of Proposition \ref{prop:M-vers-Lambda}]
We assume that $\tilde{s} > s \geq 1$ and that $\Lambda$ is a $\p{\tau_0,\tilde{s}}$-adapted Lagrangian with $\tau_0 > 0$ small. Let $R > 0$ and apply Lemma \ref{lemma:T_non_stationary}. Using the notations from Lemma \ref{lemma:T_non_stationary}, we see that if $\alpha \in \Lambda$ is large enough then $\alpha \in (T^* M)_{\tau_1,1/s}$ (this is a consequence of Lemma \ref{lmdist} because $\tilde{s} > s$). Consequently, if $u \in \widetilde{E}^{s,R}\p{M}$, we can use the estimate \eqref{eq:T_non_stationary} to bound $T_\Lambda u(\alpha)$ for large $\alpha$. For small $\alpha$, we can always use the bound \eqref{eq:trivial_bound_T}. Hence, we see that $T_\Lambda$ is bounded from $\widetilde{E}^{s,R}\p{M}$ to $F^{s,r}\p{\Lambda}$ for any $r > - C^{-1} h^{-1/s}$. The operator norm of $T_\Lambda$ may depend on $h$, but we do not care about it here. We just proved that $T_\Lambda$ is bounded from $\G^s\p{M}$ to $\GG^s\p{\Lambda}$.

We turn to the continuity from $\U^s\p{M}$ to $\UU^s\p{\Lambda}$. To do so choose $r > 0$ and apply Lemma \ref{lemma:T_BMT} for $\epsilon = r h^{1/s}/2$ (as above, we do not care about the dependence on $h$ of the bound, and can consequently work with $h$ fixed). Using the notations from Lemma \ref{lemma:T_BMT}, we see as above that if $\alpha \in \Lambda$ is large enough then $\alpha \in (T^* M)_{\tau_1,1/s}$. Consequently, if $u \in \U^s\p{M} \subseteq (E^{s,R}(M))'$ (where $R > 0$ is given by Lemma \ref{lemma:T_BMT}), we can use the estimate \eqref{eq:T_non_stationary} to bound $T_\Lambda u(\alpha)$ for large $\alpha$ in term of the norm of $u$ in $(E^{s,R}(M))'$. For small $\alpha$, we just apply an elementary bound, using that the kernel of $T$ is analytic and a compactness argument. Hence, we see that $T_\Lambda$ is bounded from $\U^s\p{M}$ to $F^{s,r}\p{\Lambda}$. Since $r > 0$ is arbitrary, $T_\Lambda$ is bounded from $\U^s\p{M}$ to $\UU^s\p{\Lambda}$.
\end{proof}

\begin{proof}[Proof of Proposition \ref{prop:Lambda-vers-M}]
The continuity of $S_\Lambda$ from $\GG^s\p{\Lambda}$ to $\G^s\p{M}$ follows from Lemma \ref{lemma:S_BMT}, using Lemma \ref{lmdist} as in the proof of Proposition \ref{prop:M-vers-Lambda}.

To extend $S_\Lambda$ from $\UU^s\p{\Lambda}$ to $\U^s\p{M}$, just notice that the formal adjoint of $S$ is (by definition) an analytic FBI transform. Consequently, the formal adjoint of $S_\Lambda$ is bounded from $\G^s\p{M}$ to $\GG^s\p{\Lambda}$ by Proposition \ref{prop:M-vers-Lambda}. Since the duality pairing \eqref{eqpairing} embeds $\GG^s\p{\Lambda}$ into the strong dual of $\UU^s\p{\Lambda}$, it follows that $S_\Lambda$ has a continuous extension from $\UU^s\p{\Lambda}$ to $\U^s\p{M}$.
\end{proof}

The next section is dedicated to the study of the composition $S_\Lambda T_\Lambda$. However, due to the sometimes intricate boundedness properties of $T_\Lambda$ and $S_\Lambda$, it is not always clear that this composition makes sense. We end this section by a result that explains why this composition is well-defined.

\begin{prop}\label{prop:regularite-decay-s=1}
Let $R_1 > 0$ be large enough. Then for $R_0 > 0$ large enough and $\tau_0$ small enough, there is a $r > 0$ such that, if $\Lambda$ is a $\p{\tau_0,1}$-adapted Lagrangian, then  $T_\Lambda$ is continuous from $(E^{1,R_0})'$ to $F^{1, r}(\Lambda)$, and $S_\Lambda$ has a continuous extension from $F^{1,r}(\Lambda)$ to $(E^{1,R_1})'$, so that $S_\Lambda T_\Lambda$ has a continuous extension from $(E^{1,R_0})'$ to $(E^{1,R_1})'$ (here $r$ may depend on $h$ but the quantification on $R_0$ and $\tau_0$ does not).

In the same fashion, for every $R_1 > 0$ large enough, there is a $r > 0$, such that, provided that $\tau_0$ is small enough, for every $k \in \R$, the operator $S_\Lambda$ is bounded from $L^2_k\p{\Lambda}$ to $(E^{1,R_1})'$ and $T_\Lambda$ is bounded from $(E^{1,R_1})'$ to $F^{1,r}\p{\Lambda}$, so that $T_\Lambda S_\Lambda$ is bounded from $L^2_k\p{\Lambda}$ to $F^{1,r}\p{\Lambda}$.
\end{prop}

\begin{proof}
Choose $R_1 > 0$. Let $\Lambda$ be a $\p{\tau_0,1}$-adapted Lagrangian for $\tau_0 > 0$ small. From Lemma \ref{lemma:T_non_stationary} and a duality argument as in the proof of Proposition \ref{prop:Lambda-vers-M}, we see that there is $C > 0$ such that, provided that $\tau_0$ is small enough, $S_\Lambda$ is bounded from $(E^{1,R_1}(M))'$ to $F^{s,r}(\Lambda)$ with $r = (Ch)^{-1}$. Now, applying Lemma \ref{lemma:T_BMT} with $\epsilon = 1/C$, we see that, provided that $\tau_0$ is small enough, there is $R_0 > 0$ such that $T_\Lambda$ is bounded from $(E^{1,R_0}(M))'$ to $F^{1,r}\p{\Lambda}$. Consequently, the composition $S_\Lambda T_\Lambda$ is well-defined and continuous from $(E^{1,R_0})'$ to $(E^{1,R_1})'$. Of course, we can always take $R_0$ larger, since it only makes $(E^{1,R_0})'$ smaller.

The proof of the second statement is similar. One just need in addition to notice that it follows from \eqref{eq:taille_de_H} that, for any $c > 0$, if $\Lambda$ is a $\p{\tau_0,1}$-adapted Lagrangian with $\tau_0$ small enough then we may use \eqref{eqpairing} to pair elements of $L^2_k\p{\Lambda}$ with elements of $F^{1,r}\p{\Lambda}$ where $r = - ch$. Thus $S_\Lambda$ has a continuous extension from $L^2_k\p{\Lambda}$ to $\p{E^{1,R_1}}'$ if $\tau_0$ is small enough.
\end{proof}

\begin{remark}\label{remark:C_infini}
Integrating by parts, one can show that $T$ maps continuously $C^\infty(M)$ to the space of maps on $T^\ast M$ that decay faster than any power of $\langle|\alpha|\rangle^{-1}$ at infinity ($h$ being fixed). Similarly, basic estimates on the kernel of $S$ shows that it maps the space of maps on $T^\ast M$ that decay faster than any power of $\langle|\alpha|\rangle^{-1}$ at infinity continuously into $\mathcal{C}^\infty\p{M}$.
\end{remark}

\subsection{Finding a good FBI transform}\label{subsec:inversion}

\subsubsection{Consequences of Theorem \ref{thm:existence-good-transform} and reduction of its proof}
This section is devoted to proving the existence of an FBI transform $T$ such that $T^\ast T = 1$ (Theorem \ref{thm:existence-good-transform}), and some consequences. The first corollary of Theorem \ref{thm:existence-good-transform} is
\begin{corollary}\label{cor:density-E1R}
There is an $R_0>0$ such that $E^{1,R_0}(M)$ is dense in $C^\infty(M)$, and, for every $R_1 > 0$ large enough, there is an $R_2 > 0$ such that, $E^{1,R_0}$ is dense in $(E^{1,R_2})'$ for the topology induced by $(E^{1,R_1})'$.
\end{corollary}

\begin{proof}
If $u$ belongs to $\mathcal{C}^\infty\p{M}$ or $(E^{1,R_2})'$, we want to approximate $u$ by
\[
u_n:= T^\ast \Big(\chi\big(\frac{1}{n} \langle|\alpha|\rangle\big) Tu\Big),
\]
where $\chi \in \mathcal{C}^\infty_c\p{\R}$ is equal to $1$ around $0$. It follows from Lemma \ref{lemma:T_BMT} that the $u_n$'s belongs to $E^{1,R_0}$ for $R_0$ large enough (that does not depend on $u$). If $u \in \mathcal{C}^\infty\p{M}$, then it follows from Remark \ref{remark:C_infini} that $u_n$ tends, when $n$ tends to $+ \infty$, to $T^* T u = u$ in $\mathcal{C}^\infty\p{M}$. Now, if $u \in (E^{1,R_2})'$, we see as in the proof of Proposition \ref{prop:regularite-decay-s=1} that $u_n$ tends to $T^* Tu = u$ in $(E^{1,R_1})'$ provided that $R_2 \gg R_1$.
\end{proof}

Now that we have this density statement, we can consider the composition $S_\Lambda T_\Lambda$ for $\Lambda$ an adapted Lagrangian. 
\begin{lemma}\label{lemma:shift-ST}
Assume that $T$ is given by Theorem \ref{thm:existence-good-transform}, and that $S=T^\ast$ is the corresponding adjoint transform. Then, for every $R > 0$ large enough, there is $R' \leq R$ such that if $\Lambda$ is a $\p{\tau_0,1}$-Gevrey adapted Lagrangian with $\tau_0$ small enough, then, for $h$ small enough, $S_\Lambda T_\Lambda$ when acting on $(E^{1,R})'$ is the injection of $(E^{1,R})'$ inside $(E^{1,R'})'$.
\end{lemma}

\begin{proof} 
First of all, recall that, thanks to Proposition \ref{prop:regularite-decay-s=1}, the operator $S_\Lambda T_\Lambda$ is well-defined on $(E^{1,R})'$ when $R$ is large enough (taking value in a space which is \emph{a priori} larger).

From Proposition \ref{prop:regularite-decay-s=1}, we see that, provided that $\tau_0$ is small enough, there are $R_1,R_2 > 0$ such that $S_\Lambda T_\Lambda$ is continuous from $(E^{1,R_1})'$ to $(E^{1,R_2})'$. Moreover, we may assume that $R_1$ is arbitrarily large (since it makes $(E^{1,R_1})'$ smaller) and hence apply Corollary \ref{cor:density-E1R} to see that, provided that $R >0$ is large enough, $E^{1,R_0}$ is dense in $(E^{1,R})'$ for the topology of $(E^{1,R_1})'$. Consequently, we only need to prove that $S_\Lambda T_\Lambda u = u$ for $u \in E^{1,R_0}$.

However, for $u \in E^{1,R_0}$, it follows from Lemmas \ref{lmdist} and \ref{lemma:T_non_stationary} that, provided that $\tau_0$ is small enough (depending on $R_0$), the integral in the following right-hand side is convergent
\begin{equation}\label{eq:notons-le}
S_\Lambda T_\Lambda u (x) = \int_{\Lambda} K_S(x,\alpha) T u(\alpha) \mathrm{d}\alpha.
\end{equation}
Moreover, the integrand in \eqref{eq:notons-le} is holomorphic and rapidly decreasing (due to Lemmas \ref{lemma:T_non_stationary} and \ref{lemma:S_BMT}). Hence, we may apply Stokes' Formula to shift contours in \eqref{eq:notons-le} and replace the integral over $\Lambda$ by an integral over $T^* M$, proving that $S_\Lambda T_\Lambda u = T^* T u = u$, and hence the lemma.
\end{proof}

\begin{remark}
Mind that $S_\Lambda$ is not a right inverse for $T_\Lambda$. Indeed, $T_\Lambda$ is not surjective, since its image only contains smooth function. In fact, we will form the projector $\Pi_\Lambda = T_\Lambda S_\Lambda$ on the image of $T_\Lambda$. The study of the operator $\Pi_\Lambda$, and more generally of operators of the form $T_\Lambda P S_\Lambda$, will be carried out in \S \ref{sec:TPS}.
\end{remark}

From Lemma \ref{lemma:shift-ST}, the basic properties of the functional spaces introduced in \S \ref{sec:symplectic-geometry} follow.

\begin{corollary}\label{cor:completude}
Let $s \geq 1$ and $\Lambda$ be a $\p{\tau_0,s}$-adapted Lagrangian with $\tau_0$ small enough. Under the assumptions of the previous lemma, for every $k \in \R$, the space $\mathcal{H}_\Lambda^k$ (defined by \eqref{eq:def-HMk}) is a Hilbert space and the space $\mathcal{H}_{\Lambda,\FBI}^k$ is a closed subspace of $L^2_k\p{\Lambda}$.

Moreover, for $R_0$ large enough, $\tau_0$ small enough (depending on $R_0$), and all $k \in \R$,
\begin{equation}\label{eq:des_inclusions}
E^{s,R_0} \subseteq \mathcal{H}_{\Lambda}^k\subseteq \p{E^{s,R_0}\p{M}}'.
\end{equation}
If $1 \leq \tilde{s} < s$ then, provided that $\tau_0$ is small enough, for every $k \in \R$, we have
\begin{equation*}
\begin{split}
\G^{\tilde{s}}\p{M} \subseteq \mathcal{H}_\Lambda^k \subseteq \U^{\tilde{s}}\p{M}.
\end{split}
\end{equation*}
All these inclusions are continuous.
\end{corollary}

\begin{proof}
Since $T_\Lambda$ has a left inverse, the equation \eqref{eq:def-norme-HMk} indeed defines a norm (provided that $R$ is chosen large enough in \eqref{eq:def-HMk} so that Lemma \ref{lemma:shift-ST} applies). Notice then that the spaces $\mathcal{H}_{\Lambda}^k$ and $\mathcal{H}_{\Lambda,\FBI}^k$ are isometric, so that the completeness of $\mathcal{H}_\Lambda^k$ is equivalent to the closedness of $\mathcal{H}_{\Lambda,\FBI}^k$.

In order to prove the closedness of $\mathcal{H}_{\Lambda,\FBI}^k$ in $L^2_k\p{\Lambda}$, we will explain how $R$ is chosen in \eqref{eq:def-HLambdak}. First of all, choose $R$ large enough so that Lemma \ref{lemma:shift-ST} applies. We want also that $R$ is large enough so that, by Proposition \ref{prop:regularite-decay-s=1} and for $\tau_0$ small enough, $S_\Lambda$ is bounded from $L^2_k\p{\Lambda}$ to $(E^{1,R})'$, and $\Pi_\Lambda$ is bounded from $L^2_k\p{\Lambda}$ to some $F^{1,r}$.

With these conditions in mind, let $\p{u_n}_{n \in \N}$ be a sequence in $\mathcal{H}_{\Lambda,\FBI}^k$ that converges to $u$ in the space $L^2_k\p{\Lambda}$. Then $\p{\Pi_\Lambda u_n}_{n \in \N}$ converges to $\Pi_\Lambda u$ in $F^{1,r}$, locally in $L^2$. But, since $\Pi_\Lambda u_n = u_n$ for $n \in \N$ (because our choice of $R$ ensures that Lemma \ref{lemma:shift-ST} applies), we see that $u = \Pi_\Lambda u = T_\Lambda S_\Lambda u \in \mathcal{H}_{\Lambda,\FBI}^0$ (here we use the fact the $S_\Lambda u \in (E^{1,R})'$ thanks to our choice of $R$). It follows that $\mathcal{H}_{\Lambda,\FBI}^k$ is closed in $L^2_k\p{\Lambda}$.

The last statement follows from Propositions \ref{prop:M-vers-Lambda}, \ref{prop:Lambda-vers-M}. To get the more precise estimates \eqref{eq:des_inclusions}, one just need to apply the more precise estimate Lemma \ref{lemma:T_non_stationary}, recalling \eqref{eq:taille_de_H} to control the size of $H$. This gives immediately the first inclusion, the second one is obtained by a duality argument.
\end{proof}

The proof of Theorem \ref{thm:existence-good-transform} relies on the following lemma: 
\begin{lemma}\label{lemma:ST-pseudo}
Let $T$ and $S$ be respectively an analytic and an adjoint analytic FBI transform, whose phases satisfy $\Phi_S(x,\alpha)= - \overline{\Phi_T(\alpha,x)}$. Then the operator $ST$ is a $\G^1$ elliptic pseudor of order $0$ on $M$. Moreover, its principal symbol is
\begin{equation*}
\begin{split}
\alpha \mapsto 2^n (h\pi)^{ \frac{3n}{2}} \frac{a(\alpha,\alpha_x) b(\alpha_x,\alpha)}{\sqrt{\det \mathrm{d}^2_x \Im \Phi_T(\alpha,\alpha_x)}},
\end{split}
\end{equation*}
where $a$ and $b$ are the symbols of $T$ and $S$ respectively.
\end{lemma}

\begin{proof}[Proof of Theorem \ref{thm:existence-good-transform}]
We start by taking an analytic transform $T^0$ with symbol $a(\alpha,x) = 2^{- \frac{n}{2}} \p{h \pi}^{- \frac{3n}{4}} (\det \mathrm{d}^2_x\Im \Phi_T(\alpha,\alpha_x))^{1/4}$, and applying Lemma \ref{lemma:ST-pseudo}, we deduce that $P= (T^0)^\ast T^0$ is an elliptic $\G^1$ pseudor with principal symbol $1$. Additionally, it is self-adjoint on $L^2$. Since the principal symbol of $P$ is $1$, it follows from G{\aa}rding's inequality that there is $C > 0$ such that, for $h$ small enough, the spectrum of $P$ is contained in $[C^{-1},C]$. In particular, the inverse square root of $P$ is well-defined (by the functional calculus for self-adjoint operators) and given by the formula
\begin{equation}\label{eq:racine_par_Cauchy}
\begin{split}
P^{- \frac{1}{2}} = \frac{1}{2 i \pi} \int_\gamma (w-P)^{-1} \frac{\mathrm{d}w}{\sqrt{w}},
\end{split}
\end{equation}
where $\gamma$ is a curve in $\set{z \in \C : \Re z > 0}$ that turns around the segment $[C^{-1},C]$. Since the principal symbol of $P$ is $1$, we see that all the $w-P$ for $w \in \gamma$ are semi-classically elliptic and hence it follows from Theorem \ref{thm:parametrix} that $(w-P)^{-1}$ is a semi-classical $\G^1$ pseudor. Moreover, this is true with uniform estimates when $w \in \gamma$, and thus it follows from \eqref{eq:racine_par_Cauchy} that $P^{- \frac{1}{2}}$ is a semi-classical $\G^1$ pseudor (we apply Definition \ref{def:gevrey_pseudor} of a $\G^1$ pseudor and Fubini Theorem as in the proof of Theorem \ref{theorem:functional_calculus}). It follows from Proposition \ref{prop:pseudor_lagrangien_alt} (see also Remark \ref{remark:ce_quest_PS}) that $T = T_0 P^{- \frac{1}{2}}$ is an analytic FBI transform, and moreover it satisfies that
\begin{equation*}
\begin{split}
T^* T = P^{- \frac{1}{2}} P P^{- \frac{1}{2}} = I.
\end{split}
\end{equation*}
\end{proof}

\subsubsection{Product of a transform and an adjoint transform}

This subsection is devoted to the proof of Lemma \ref{lemma:ST-pseudo}.
 
In order to manipulate converging integrals, we start by introducing a regularization procedure. From \eqref{eq:notons-le} with $\Lambda = T^* M$, we notice that if $u \in \G^1\p{M}$ then (by dominated convergence and Fubini's Theorem)
\begin{equation*}
\begin{split}
ST u(x) & = \lim_{\epsilon \to 0}  \int_{T^\ast M} e^{- \epsilon \jap{\alpha}^2 } K_S(x,\alpha) T u(\alpha) \mathrm{d}\alpha \\
       & = \lim_{\epsilon \to 0} \int_{M} K_\epsilon(x,y) u(y) \mathrm{d}y,
\end{split}
\end{equation*}
where the kernel $K_\epsilon$ is defined by
\begin{equation}\label{eq:def-Kepsilon}
\begin{split}
K_\epsilon (x,y) = \int_{T^* M} e^{- \epsilon \jap{\alpha}^2} K_S(x,\alpha) K_T(\alpha,y) \mathrm{d}\alpha.
\end{split}
\end{equation}

We start by showing that, outside of a small neighbourhood of the diagonal, the kernel of $ST$ is a rapidly decaying (when $h$ tends to $0$) analytic function. To do so, choose some small $\eta > 0$ (whose precise value will be fixed later). Then, there is a constant $C > 0$ such that, if $x,y \in M$ are such that $d(x,y) > \eta$, then
\begin{equation}\label{eq:produit-noyau-loin-de-la-diagonale}
\begin{split}
\va{K_S(x,\alpha) K_T(\alpha,y)} \leq C \exp\p{ -\frac{\jap{\va{\alpha}}}{Ch}}.
\end{split}
\end{equation}
This follows from the definition of $K_T$ and $K_S$ since either $x$ or $y$ is at distance at least $\eta/2$ from $\alpha_x$. Estimate \eqref{eq:produit-noyau-loin-de-la-diagonale} implies that for such $x$ and $y$ the kernel $K_\epsilon(x,y)$ converges when $\epsilon$ tends to $0$ to a kernel $K(x,y)$ which is a $\O(\exp(- 1/Ch))$. Since Estimate \eqref{eq:produit-noyau-loin-de-la-diagonale} remains true when $x$ and $y$ are in a small tube $(M)_{\tilde{\epsilon}}$, this is in fact an estimate in $\G^1$.

Hence, we only need to understand the kernel of $ST$ in an arbitrarily small neighbourhood of the diagonal in $M \times M$. By a standard compactness argument, we only need to understand the kernel of $ST$ for $x$ and $y$ in a fixed ball $D$ of radius $10 \eta$. To do so, let $D'$ and $D''$ be respectively a ball of radius $100 \eta$ and $1000 \eta$ with the same center as $D$. When $x$ and $y$ are in $D$, we can split the integral \eqref{eq:def-Kepsilon} into the integral over $T^* D'$ and the integral over $T^* \p{M \setminus D'}$. The integral over $T^* \p{M \setminus D'}$ is dealt with as in the case $d(x,y) > \eta$. By taking $\eta$ small enough, we may assume that when $x,y$ and $\alpha_x$ belongs to $D'$, the kernels $K_T$ and $K_S$ both satisfy an estimate as equation \eqref{eq:def-KT}, denoting by $a$ the symbol of $T$, and by $b$ that of $S$. The $\O(\exp(-C^{-1} \langle\alpha\rangle/h))$ remainders are tackled as above, so that we focus now on the kernel
\begin{equation*}
\begin{split}
\widetilde{K}_\epsilon(x,y) = \int_{T^* D'} e^{- \epsilon \jap{\alpha}^2} e^{i \frac{\Phi_S(x,\alpha) + \Phi_T(\alpha,y)}{h}} a(\alpha,y) b(x,\alpha) \mathrm{d}\alpha,
\end{split}
\end{equation*}
for $x,y \in D$. We want to recognize, when $\epsilon$ tends to $0$, the kernel of an analytic pseudor (with non-standard phase) and to compute its symbol. By taking $\eta$ small enough, we may assume that we work in local coordinates $\alpha = (z,\xi)$. We will assume (as we may) that these coordinates satisfy $\det g_z = 1$. We can also replace the regularization factor $\exp(- \epsilon\jap{\alpha}^2)$ by the simpler $\exp( - \jap{\xi}^2)$: testing against an analytic function, we see that the limiting distribution is independent on the regularization. Hence, we have 
\begin{equation}\label{eq:dans_un_coin}
\begin{split}
& (2 \pi h)^{-n} \widetilde{K}_\epsilon(x,y) \\
    & \qquad = \int_{\R^n} e^{- \epsilon \jap{\xi}^2} \p{\p{ \frac{\jap{\xi}}{ 2 \pi h}}^{\frac{n}{2}} \int_{D'} e^{i \frac{\jap{\xi}}{h}\Psi_{x,\xi,y}(z)} \frac{ (2 \pi h)^{\frac{3n}{2}} a(z,\xi,y) b(x,z,\xi)}{\jap{\xi}^{\frac{n}{2}}} \mathrm{d}z} \mathrm{d}\xi,
\end{split}
\end{equation}
where the phase $\Psi_{x,\xi,y}$ is defined by
\begin{equation*}
\begin{split}
\Psi_{x,\xi,y}(z) 	&= \jap{\xi}^{-1}\p{\Phi_S(x,(z,\xi)) + \Phi_T((z,\xi),y)} \\
					&= \jap{\xi}^{-1}\p{\Phi_T((z,\xi),y)  - \overline{\Phi_T((\bar{z},\bar{\xi}),\bar{x})}}.
\end{split}
\end{equation*}
To put the kernel into pseudo-differential form, it suffices formally to eliminate the variable $z$, and this is done by Holomorphic Stationary Phase, Proposition \ref{prop:HSP_symbol} (see also Remark \ref{remark:symboles_generaux}). Let us check that the hypotheses of Proposition \ref{prop:HSP_symbol} are satisfied. Hypothesis (i) follows from Condition (i) in Definition \ref{def:nonstandardphase} of an admissible phase. Hypothesis (ii) follows from Condition (iv) in the Definition \ref{def:nonstandardphase} and the fact that $x$ and $y$ remains at uniform distance from the boundary of $D'$. Finally, from (iii) in Definition \ref{def:nonstandardphase}, we see that $\Psi_{x,\xi,x}$ has a critical point at $z = x$, which is uniformly non-degenerate thanks to (iv) in Definition 1.7 (the imaginary part of the Hessian is uniformly definite positive). Moreover, the associated critical value is $0$. Thus, provided that $\eta$ is small enough, we find that the inner integral in \eqref{eq:dans_un_coin} is given for $(x,\xi,y) \in \p{T^* D}_{\epsilon_1} \times \p{D}_{\epsilon_1}$ (for some $\epsilon_1 > 0$) and arbitrarily large $C_0$ by
\begin{equation}\label{eq:expansion_inner_integral}
\begin{split}
e^{i \frac{\Phi_{ST}(x,\xi,y)}{h}} \p{\sum_{0  \leq k \leq \jap{\va{\xi}}/(C_0 h)} h^k c_k(x,\xi,y) + \mathcal{O}\p{\exp\p{- \frac{\jap{\va{\xi}}}{C h}}}}.
\end{split}
\end{equation}
Here, $\sum_{k \geq 0} h^k c_k$ is a formal analytic symbol of order $1$ and $\Phi_{ST}(x,\xi,y)$ denotes the critical value of $z \mapsto \jap{\xi} \Psi_{x,\xi,y}(z)$.

Let us check that $\Phi_{ST}$ is a phase in the sense of Definition \ref{def:nonstandardphase}. The holomorphy of $\Phi_{ST}$ follows from the Implicit Function Theorem, and the symbolic estimates from the fact that $\Psi_{x,\xi,y}$ is uniformly bounded. Point (ii) in Definition \ref{def:nonstandardphase} is satisfied because when $x=y$ the critical point of $\Psi_{x,\xi,x}$ is $x$ so that $\Phi_{ST}(x,\xi,x) = \jap{\xi} \Psi_{x,\xi,x}(x) = 0$. From the Implicit Function Theorem, we also see that
\begin{equation*}
\begin{split}
\mathrm{d}_y \Phi_{ST}(x,\xi,x) = \mathrm{d}_y \Phi_T(x,\xi,x) = - \xi,
\end{split}
\end{equation*}
since $\Phi_T$ is itself a phase. This proves that $\Phi_{ST}$ satisfies point (iii) of Definition \ref{def:nonstandardphase}. It remains to check (i), i.e. that for $\alpha,x$ real, $\Im \Phi_{ST}\geq 0$. This follows from the ``Fundamental'' Lemma \ref{lemma:fundamental-lemma}

Now, let $c$ be a realization of the formal symbol $\sum_{k \geq 0} h^k c_k$. Since the imaginary part of $\Phi_{ST}$ is non-negative for real $(x,\xi,y)$, we see that for $(x,\xi,y) \in \p{T^* D}_{\epsilon_1} \times \p{D_{\epsilon_1}}$, the factor $e^{i \Phi_{ST}(x,\xi,y)/h}$ in \eqref{eq:expansion_inner_integral} is an $\O(\exp(C \epsilon_1 \jap{\va{\xi}}/h))$. Consequently, by taking $\epsilon_1$ small enough, we may move the error term in \eqref{eq:expansion_inner_integral} out of the parenthesis. Thus, the inner integral in \eqref{eq:dans_un_coin} is given by
\begin{equation*}
\begin{split}
e^{i\frac{\Phi_{ST}(x,\xi,y)}{h}} c(x,\xi,y)
\end{split}
\end{equation*}
up to an $\mathcal{O}\p{\exp\p{- C^{-1} \jap{\va{\xi}}/ h}}$ error. Plugging this expression into \eqref{eq:dans_un_coin} and using Lemma \ref{lemma:norm-exponentials}, we see that $\widetilde{K}_{\epsilon}(x,y)$ converges when $\epsilon$ tends to $0$ to the kernel $K_{\Phi_{TS},c}$ defined by \eqref{eq:noyau_non_standard}, up to an $\O_{\G^1}(\exp(- 1/(Ch)))$. By Lemma \ref{lmkurtrick}, we know that $K_{\Phi_{TS},c}$ is in fact the kernel of a $\G^1$ pseudor. Thus, $ST$ is itself a $\G^1$ pseudor. According to Remark \ref{remark:coefficients_symboles} and the proof of Lemma \ref{lmkurtrick}, the principal symbol of $ST$ is given in our coordinates by
\begin{equation*}
\begin{split}
c_0(x,\xi,x) =  2^n (h\pi)^{ \frac{3n}{2}} \frac{a(x,\xi,x) b(x,\xi,x)}{\sqrt{\det \mathrm{d}^2_x \Im \Phi_T(x,\xi,x)}},
\end{split}
\end{equation*}
for $x,\xi$ real.
\qed

\section{Lifting Pseudo-differential operators}\label{sec:FBI-et-pseudo}

In this section, the purpose of defining the $\mathcal{H}_\Lambda^k$'s will be explained. We fix an analytic FBI transform $T$ given by Theorem \ref{thm:existence-good-transform} and denote by $S = T^*$ its adjoint. For an adapted Lagrangian $\Lambda$, we will study the operator $\Pi_\Lambda = T_\Lambda S_\Lambda$ defined before. We will find that it acts naturally on the space $L^2_{0}(\Lambda)$. More generally, given $P$ a continuous operator on some suitable spaces of ultradistributions, it make sense to consider the operator $T_\Lambda P S_\Lambda$ acting on $L^2_0\p{\Lambda}$, or equivalently the action of $P$ on $\mathcal{H}_\Lambda^0$. In this section, we will consider the case that $P$ is a $\mathcal{G}^s$ pseudor. In \S \ref{sec:TPS}, we will study asymptotics of the kernel of $T_\Lambda P S_\Lambda$. In \S \ref{sec:action-pseudo-FBI}, we will deduce some basic boundedness results, and in \S \ref{sec:wfs}, we will obtain results on the relation between Lagrangian spaces and Gevrey wave front sets. 

\subsection{The kernel of operators on the FBI side}\label{sec:TPS}

To study the continuity properties of pseudors on the spaces $\mathcal{H}_\Lambda^k$ associated to I-Lagrangians, we will need to understand precisely the asymptotics of the kernel of $T_\Lambda P S_\Lambda$. In this section we obtain such technical estimates. 

As a precaution, we start by observing that, according to the results of \S \ref{sec:continuity-T-S}, the operator $T_\Lambda P S_\Lambda$ maps $C^\infty_c(\Lambda)$ to elements of $\GG^s(\Lambda)\subset \mathcal{D}'(\Lambda)$, so it has a well-defined Schwartz kernel $K_{TPS}$ with respect to $\mathrm{d}\alpha$, which we can study. The analyticity of the FBI transform implies that $K_{TPS}$, \emph{a priori} defined only on $\Lambda\times \Lambda$, is actually the restriction of a holomorphic function on $(T^\ast M)_{\epsilon_0}\times (T^\ast M)_{\epsilon_0}$ that does not depend on $\Lambda$, as long as $\Lambda \subset (T^\ast M)_{\epsilon_0}$. We still denote it by $K_{TPS}$. 

The kernel $K_{TPS}(\alpha,\beta)$ behaves as an oscillatory integral, with large parameter
\[
\lambda := \frac{ \langle|\alpha|\rangle  + \langle|\beta|\rangle }{ h }. 
\]
For convenience, we recall that $\gamma\in (T^\ast M)_{\tau_0, 1/s}$ means that $\gamma \in T^* \widetilde{M}$ and
\[
|\Im \gamma| \leq \tau_0 (h/\langle|\gamma|\rangle)^{1-1/s}.
\]
When estimating $K_{TPS}$, we have to distinguish between two regimes: far and close from the diagonal. For the first regime, we obtain
\begin{lemma}\label{lemma:TPS-far-diagonal}
Let $s \geq 1$. Let $P$ be a $\G^s$ semi-classical pseudo-differential operator of order $m$ on $M$. Let $\eta > 0$ be small enough. Then there are constants $C_0,\tau_1 > 0$ such that for $\alpha,\beta$ in $(T^\ast M)_{\tau_1,1/s}$, and $d_{KN}(\alpha,\beta) > \eta/2$, the following holds:
\begin{equation*}
K_{TPS}(\alpha,\beta) = \O_{\mathcal{C}^\infty} \p{\exp\p{- \p{\frac{\jap{\va{\alpha}} + \jap{\va{\beta}}}{C_0 h}}^{\frac{1}{s}}}}
\end{equation*}
\end{lemma}

The proof of this Lemma relies mostly on non-stationary phase estimates. For the second regime, we have to use a stationary phase method. For this reason, the proof is more subtle. It is the most technical part of this article. 
\begin{lemma}\label{lemma:TPS-close-diagonal}
Let $s \geq 1$. Let $P$ be a $\G^s$ semi-classical pseudo-differential operator of order $m$ on $M$. Let $\eta > 0$ be small enough. Let $\Lambda$ be a $(\tau_0,s)$-adapted Lagrangian with $\tau_0$ small enough. Then there exists a constant $C_0 > 0$ such that for $\alpha,\beta$ in $\Lambda$ with $d_{KN}(\alpha,\beta)\leq \eta$, with $\lambda = \langle|\alpha|\rangle/h$,
\begin{equation}\label{eq:noyau_reduit_TPS}
\begin{split}
K_{TPS}(\alpha,\beta) 	 = e^{\frac{i}{h}\Phi_{TS}(\alpha,\beta)} e(\alpha,\beta) + \O_{\mathcal{C}^\infty}\left( \exp\left( - \frac{ \lambda^{ \frac{1}{2s-1} }  }{C_0  } - \frac{\lambda^{\frac{1}{s}} d_{KN}(\alpha,\beta)}{C_0  } \right) \right),
\end{split}
\end{equation}
where $\Phi_{TS}(\alpha,\beta)$ is the critical value of $y \mapsto \Phi_T(\alpha,y) + \Phi_{S}(\beta,y)$ and $e$ is a symbol of order $m$ in the Kohn--Nirenberg class $h^{-n} S_{KN}^m\p{\Lambda \times \Lambda}$, given at first order on the diagonal by
\begin{equation}\label{eq:premier_ordre_e}
e(\alpha,\alpha) = \frac{1}{(2 \pi h)^n} p(\alpha) \textup{ mod } h^{-n +1} S_{KN}^{m-1}\p{\Lambda},
\end{equation}
where $p$ is an almost analytic extension of a representative of the principal symbol of $P$ (see Remark \ref{remark:aae_for_symbol}).
\end{lemma}

For the proof of Lemma \ref{lemma:TPS-close-diagonal}, it will not be useful that $\Lambda$ is Lagrangian, only that it is $\mathcal{C}^1$ close to the reals. Actually, in the proof of Lemma \ref{lemma:TPS-close-diagonal}, the main point will be that $(\alpha,\beta)$ is in the region described in Figure \ref{fig:region-TPS}, which contains a neighbourhood of the diagonal of $\Lambda\times\Lambda$, if $\Lambda$ is $(\tau_0,s)$-adapted.
\begin{figure}[h]
\centering
\def\svgwidth{\linewidth}
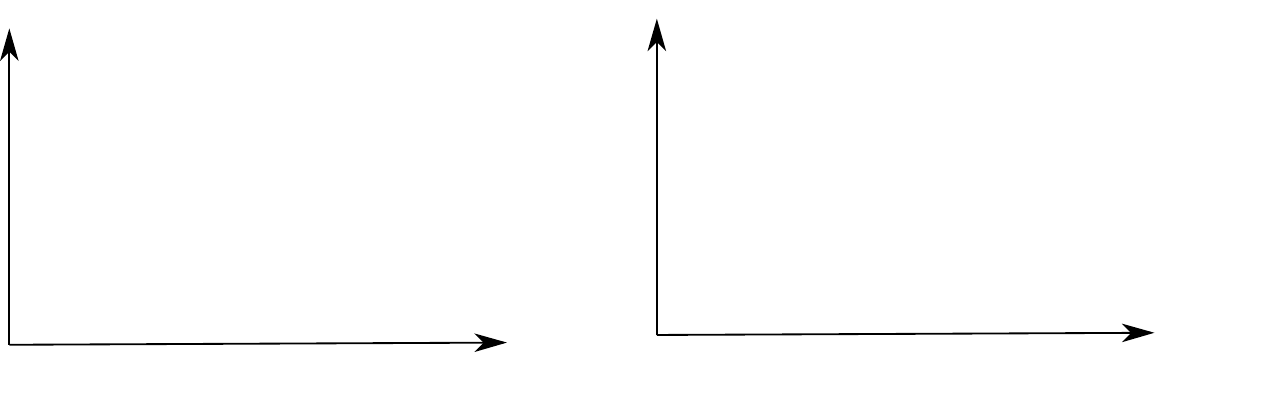
\caption{\label{fig:region-TPS} Region near the diagonal in the complex where $K_{TPS}$ can be controlled. Here, $\tilde{h} = h/\langle|\alpha|\rangle$.}
\end{figure}

\begin{remark}
There is a similar statement in \cite{lascarFBITransformsGevrey1997}. However, we make three improvements of this result. First, we get rid of the condition $s = 3$, which was made in \cite{lascarFBITransformsGevrey1997} in view of the application they had in mind. Then, we do not assume that $G$ is compactly supported and hence deal with both the limit $h \to 0$ and $\jap{\va{\alpha}} \to + \infty$. Finally, we deal here with a pseudo-differential operator $P$ while \cite{lascarFBITransformsGevrey1997} only considered differential operators. On a technical level, there is another notable difference: we implement the non-stationary phase method in Gevrey classes using almost analytic extensions and shifting of contours rather than integrations by parts and formal norms. We hope that this makes the proof easier to understand. 
\end{remark}

\begin{remark}
Since it is not necessary for our purposes, we did not investigate ultradifferentiability of the kernel. However, the proof we give should imply a $\G^{2s-1}$ regularity without much effort if $\Lambda$ were itself a $\G^{2s-1}$ manifold. We suspect that there is a way to obtain a $\Gs$ estimate in the case that $\Lambda$ is a $\Gs$ manifold, but for want of application we did not investigate further.
\end{remark}

The proof of Lemmas \ref{lemma:TPS-far-diagonal} and \ref{lemma:TPS-close-diagonal} are slightly different in the cases $s > 1$ and $s = 1$. To apply a combination of stationary and non-stationary phase method, we split the domain of integration in several regions. Several times, the imaginary part of the phase possibly vanishes near the boundary of those domains. Since it is non-stationary when this happens, one solution is to use contour deformations near those boundaries. Another solution, only available when $s>1$, is to introduce $\Gs$ cutoff functions. Since this lightens the proof, we will cover in full detail the case $s> 1$ (the most difficult case). Then we sketch the proof of both lemmas in the case $s=1$ at the end of this section.

\begin{proof}[Proof of Lemma \ref{lemma:TPS-far-diagonal} in the case $s>1$]
In the proof, $\eta>0$ and $\tau_0>0$ will have to be small enough \emph{independently} of $h>0$, $\alpha$ and $\beta$; it will be the case that $\tau_0 \ll \eta$. The cutoffs are assumed to be $\Gs$, throughout the proof. We will denote $K_{TPS}$ just as $K$.

Throughout this proof, we assume that that $\alpha,\beta$ are in $(T^\ast M)_{\tau_0,1/s}$, and that $d(\alpha_x,\beta_x) > \eta/2$. Since $\tau_0 \ll \eta$, we have that $d(\alpha_x,\beta_x) \sim d(\Re \alpha_x, \Re \beta_x)$ in this region. We may find $\G^s$ bump functions $\chi_\alpha$ and $\chi_\beta$ that take value $1$ on a neighbourhood respectively of $\Re \alpha_x$ and $\Re \beta_x$ and whose support do not intersect (we only need a finite number of such pairs of bump functions to deal with the whole manifold by compactness). Then we split the integral
\begin{equation}\label{eq:noyau_TPS_integrale}
\begin{split}
K(\alpha,\beta) = \int_{M \times M} K_T(\alpha,x) K_P(x,y) K_S\p{y,\beta} \mathrm{d}x \mathrm{d}y
\end{split}
\end{equation}
into
\begin{equation}\label{eq:on_decoupe_une_integrale_pour_changer}
\begin{split}
K(\alpha,\beta) 	& = \int_{M \times M} K_T(\alpha,x) K_P(x,y)\p{1 - \chi_{\beta}(y)} K_S\p{y,\beta} \mathrm{d}x \mathrm{d}y \\ 
						& \quad + \int_{M \times M} K_T(\alpha,x) \p{1 - \chi_{\alpha}(x)} K_P(x,y) \chi_\beta(y) K_S\p{y,\beta} \mathrm{d}x \mathrm{d}y \\ 
						& \quad + \int_{M \times M} K_T(\alpha,x) \chi_\alpha(x) K_P(x,y) \chi_\beta(y) K_S\p{y,\beta} \mathrm{d}x \mathrm{d}y.
\end{split}
\end{equation}
We notice that the first two integrals in the right hand side have an almost symmetric behaviour, so that it suffices to deal with the first one. For this, we also notice that it is nothing else than
\begin{equation*}
\begin{split}
T \Bigg( P\Big( y \mapsto \p{1 - \chi_{\beta}(y)} K_S\p{y,\beta} \Big) \Bigg) \p{\alpha}.
\end{split}
\end{equation*}
Since $\chi_\beta$ is equal to $1$ near $\Re \beta_x$, it follows from the assumptions on $K_S$ (see Definition \ref{def:analytic-FBI}) that the function $y \mapsto \p{1 - \chi_{\beta}(y)} K_S\p{y,\beta}$ is an $\O(\exp( -\jap{\va{\beta}}/ Ch ))$ in some $E^{s,R}\p{M}$. Then it follows from Lemmas \ref{lemma:loin_de_la_diagonale} and \ref{lemma:basic-Gs-boundedness-pseudo} that
\begin{equation*}
\begin{split}
P\Big( y \mapsto \p{1 - \chi_{\beta}(y)} K_S\p{y,\beta} \Big).
\end{split}
\end{equation*}
is an $\O(\exp(- C^{-1} (\jap{\va{\beta}} /h)^{1/s}))$ in some $E^{s,R}\p{M}$. Finally, it follows from Lemma \ref{lemma:T_non_stationary} and the bound \eqref{eq:trivial_bound_T} that the first term (and hence also the second one) in \eqref{eq:on_decoupe_une_integrale_pour_changer} is 
\[
\O_{C^\infty}\p{\exp\p{-  \p{\frac{\jap{\va{\alpha}} + \jap{\va{\beta}}}{C h}}^{\frac{1}{s}}}},
\]
provided that $\tau_0$ is small enough. To deal with the last integral in \eqref{eq:on_decoupe_une_integrale_pour_changer}, we start by noticing that, since $P$ is $\Gs$ pseudo-local (according to Definition \ref{def:gevrey_pseudor}), the function
\begin{equation*}
\begin{split}
\p{x,y} \mapsto \chi_\alpha(x) K_P(x,y) \chi_\beta(y)
\end{split}
\end{equation*}
is an $\O(\exp(- C^{-1} h^{-1/s}))$ in some space $E^{s,R}\p{M \times M}$. Then we may apply the same argument as in the proof of Lemma \ref{lemma:T_non_stationary} to see that this last integral is also of the expected size. More precisely, we perform a non-stationary phase (Proposition \ref{prop:non_stationary_analytic_symbols}) separately in both variables $x$ and $y$. 

Now, we move to the study of the kernel $K_{TPS}(\alpha,\beta)$ when the distance between $\Re \alpha_x$ and $\Re \beta_x$ is less than $10 \eta$. To do so, the compactness of $M$ allows us to assume that both $\Re \alpha_x$ and $\Re \beta_x$ belongs to a same ball $D$ of radius $100 \eta$, and that they remain uniformly away from the boundary of $D$ (recall that we assume $\tau_0 \ll \eta$). We introduce then a $\G^s$ bump function $\chi$ supported in $D$ that takes value $1$ on a neighbourhood of $\Re \alpha_x$ and $\Re \beta_x$. Reasoning as in the previous case, we see that the kernel $K_{TPS}(\alpha,\beta)$ is given, up to a negligible term, by the integral
\begin{equation*}
\begin{split}
\int_{D \times D} K_T(\alpha,x) \chi(x) K_P(x,y) \chi(y) K_S(y,\beta) \mathrm{d}x  \mathrm{d}y.
\end{split}
\end{equation*}
We may rewrite this, up to negligible terms, as the oscillating integral
\begin{equation}\label{eq:integrale_interessante}
\begin{split}
\frac{1}{\p{2 \pi h}^n} \int_{D \times \R^n \times D} e^{i \frac{\Phi_T(\alpha,x) + \jap{x-y,\xi} + \Phi_S(y,\beta)}{h}} \underbrace{a(\alpha,x) \chi(x) p(x,\xi) \chi(y) b(y,\beta)}_{\coloneqq f_{\alpha,\beta}(x,\xi,y)} \mathrm{d}x \mathrm{d}\xi \mathrm{d}y,
\end{split}
\end{equation}
where $p$ is the symbol of $P$ as in \eqref{eq:local_pseudor} from Definition \ref{def:gevrey_pseudor}, $\Phi_T$ and $\Phi_S$ are the phases of $T$ and $S$ respectively and $a$ and $b$ their symbols. The phase
\begin{equation*}
\begin{split}
\Psi_{\alpha,\beta}(x,\xi,y) = \Phi_T(\alpha,x) + \jap{x-y,\xi} + \Phi_S(y,\beta)
\end{split}
\end{equation*}
will play a crucial role in the following. From now on, we may and will assume that we are working in coordinates (by choosing $\eta$ very small). 

Let $A > 0$ be large. In order to end the proof of the lemma, we want to show that the integral \eqref{eq:integrale_interessante} is negligible when $\va{\alpha_\xi - \beta_\xi} \geq A \eta \max\p{\jap{\va{\alpha}},\jap{\va{\beta}}}$ (of course, we need that $A$ does not depend on $\eta$, provided that $\eta$ is small enough). To do so, we compute the gradient of the phase $\Psi_{\alpha,\beta}$, which is given by
\begin{equation*}
\begin{split}
\nabla \Psi_{\alpha,\beta} (x,\xi,y) = \begin{pmatrix} \nabla_x \Phi_T(\alpha,x) + \xi\\  x-y \\ \nabla_y \Phi_S(y,\beta) - \xi\end{pmatrix}.
\end{split}
\end{equation*}
Notice that if $x$ and $y$ belong to the support of $\chi$ then we have
\begin{equation}\label{eq:approximation_gradient_phase}
\begin{split}
\nabla_{x} \Phi_T\p{\alpha,x} + \xi &= \xi - \alpha_\xi + \O\p{\eta \jap{\va{\alpha}}}  \\
 \nabla_y \Phi_S(y,\beta) - \xi 	&= \beta_\xi - \xi + \O\p{\eta \jap{\va{\beta}}}.
\end{split}
\end{equation}
Hence, when $\va{\alpha_\xi - \beta_\xi} \geq A \eta \max\p{\jap{\va{\alpha}},\jap{\va{\beta}}}$, we have, for some constant $C > 0$,
\begin{equation*}
\begin{split}
\va{\nabla_{x} \Phi_T\p{\alpha,x} + \xi} + \va{\nabla_y \Phi_T(y,\beta) - \xi} & \geq \va{\alpha_\xi - \xi} + \va{\beta_\xi - \xi} - C \eta \max \p{\jap{\va{\alpha}},\jap{\va{\beta}}} \\
& \geq \va{\alpha_\xi - \beta_\xi} -C \eta \max \p{\jap{\va{\alpha}},\jap{\va{\beta}}} \\ & \geq \p{A - C} \eta \max \p{\jap{\va{\alpha}},\jap{\va{\beta}}}.
\end{split}
\end{equation*}
Thus, if $A \geq 2C$, then when $\va{\alpha_\xi - \beta_\xi} \geq A \eta \max\p{\jap{\va{\alpha}},\jap{\va{\beta}}}$, the phase $\Psi_{\alpha,\beta}$ is non-stationary. We can then apply Proposition \ref{prop:non-stationary} to find that \eqref{eq:integrale_interessante} is negligible (using the large parameter $\lambda = \max\p{\jap{\va{\alpha}},\jap{\va{\beta}}} / h$ and a suitable rescaling as in \S \ref{subsec:further_tricks}). There are several issues to fix before applying this lemma. First of all, the integral in \eqref{eq:integrale_interessante} is oscillating: this is fixed by a finite number of integration by parts. Then, we do not integrate over a compact set, but the phase is non-stationary in $x$ and $y$, so that we can apply Proposition \ref{prop:non-stationary} at fixed $\xi$ and then integrate over $\xi$. For this we notice that as $|\xi|$ grows, 
\begin{equation*}
|\nabla_{x,y} \Psi_{\alpha,\beta}(\bullet, \xi)|\geq (1/C)\max( \langle|\alpha|\rangle,\langle|\beta|\rangle,\ |\xi|).
\end{equation*}
In particular, as $\xi\to\infty$, the effective large parameter is $|\xi|/h$, and we gain sufficient decay to justify the integration.

Finally, the imaginary part of the phase is not necessarily positive on the boundary of the domain of integration. However, the amplitude vanishes near the boundary, so this is not a problem either. This closes the proof of Lemma \ref{lemma:TPS-far-diagonal}
\end{proof}

We now turn to the proof of Lemma \ref{lemma:TPS-close-diagonal}. We will isolate the technical difficulties in two lemmas -- \ref{lemma:estimee_reine} and \ref{lemma:estimee_positivite_phase}.
\begin{proof}[Proof of Lemma \ref{lemma:TPS-close-diagonal} in the case $s> 1$]
For this proof, we may work locally in the cotangent space (provided that we get uniform estimates) and assume that we are considering $\alpha = (\alpha_x,\alpha_\xi), \beta = (\beta_x,\beta_\xi) \in \Lambda$ which lie in a neighbourhood of size $\eta>0$ of a point $\alpha_0 = (\alpha_{0,x},\alpha_{0,\xi}) \in T^\ast M$ for the Kohn--Nirenberg metric (as in the proof of Lemma \ref{lemma:TPS-far-diagonal}, we assume that $\tau_0 \ll \eta$). Then we have $\langle|\alpha|\rangle \sim \langle|\beta|\rangle \sim \langle \alpha_0 \rangle$, and we will be looking for estimates suitably uniform as $\langle \alpha_0 \rangle \to \infty$. The large parameter in the oscillating integrals will be $\lambda = \langle \alpha_0 \rangle/h$. As in the proof of Lemma \ref{lemma:TPS-far-diagonal}, we will denote $K= K_{TPS}$. 

Since $\alpha_x$ and $\beta_x$ are close, we can work in local coordinates and rewrite the kernel $K_{TPS}$ as in \eqref{eq:integrale_interessante}. In order to get rid of the oscillating integral, we introduce a cutoff function $\theta(\xi)$ that takes value $1$ when $\xi$ is at distance less than $100 \eta$ of $\alpha_{0,\xi}$ and supported in a ball of radius proportional to $\eta$ (all these distances have to be understood using the Kohn--Nirenberg metric). The cutoff $\theta$ depend on $\alpha_0$, but it belongs uniformly to the symbol class $S^{s,0}\p{T^* \R^n}$ (once $\eta$ has been fixed), so that this dependence will not matter. Then, instead of considering the integral \eqref{eq:integrale_interessante}, we may study
\begin{equation}\label{eq:adieu_oscillation}
\begin{split}
\widehat{K}(\alpha,\beta) :=\frac{1}{\p{2 \pi h}^n} \int_{D \times \R^n \times D} e^{i \frac{\Psi_{\alpha,\beta}(x,\xi,y)}{h}} \theta(\xi) f_{\alpha,\beta}(x,\xi,y) \mathrm{d}x \mathrm{d}\xi \mathrm{d}y.
\end{split}
\end{equation}
Indeed, it follows from \eqref{eq:approximation_gradient_phase} and the same non-stationary argument as in the proof of Lemma \ref{lemma:TPS-far-diagonal} that the difference between \eqref{eq:integrale_interessante} and \eqref{eq:adieu_oscillation} is negligible (controlled by $\exp(- \lambda^{1/s} )$). It is then useful to change variable and write
\begin{equation}\label{eq:noyau_avant_Morse}
\begin{split}
\widehat{K}(\alpha,\beta) = \p{\frac{\lambda}{2 \pi}}^{3n/2} \int_{D \times \R^n \times D}  e^{i \lambda \widetilde{\Psi}_{\alpha,\beta}(x,\xi,y)} g_{\alpha,\beta}(x,\xi,y) \mathrm{d}x \mathrm{d}\xi \mathrm{d}y
\end{split}
\end{equation}
where
\begin{equation}\label{eq:la_phase_psi_tilde}
\begin{split}
\widetilde{\Psi}_{\alpha,\beta}(x,\xi,y) = \frac{\Phi_T(\alpha,x) + \Phi_{S}(y,\beta)}{\jap{ \alpha_0 }} + \jap{x-y,\xi}
\end{split}
\end{equation}
and
\begin{equation}\label{eq:amplitude}
g_{\alpha,\beta}(x,\xi,y) = \left( \frac{\lambda}{2\pi} \right)^{- n/2}\theta\p{\jap{ \alpha_0 } \xi} f_{\alpha,\beta}\p{x,\jap{ \alpha_0 } \xi, y} = \O\left( \frac{\langle \alpha_0 \rangle^m }{h^n} \right).
\end{equation}
This estimate on the amplitude follows from the fact that the symbol of $T$ is a $\G^1$ symbol of order $h^{-3n/4}\langle|\alpha|\rangle^{n/4}$ -- recall $m$ is the order of $P$. Here, we assume that $\eta$ is small enough so that the size of any point in the support of $\theta$ is comparable to  $\jap{\alpha_0}$ when $\jap{\alpha_0}$ goes to $+ \infty$. The estimate \eqref{eq:amplitude} holds as a $\Gs$ estimate in all variables. We will denote by $\widetilde{D}$ a disc of radius proportional to $\eta$ (for the Euclidean metric) such that $\theta\p{\jap{\alpha_0} \cdot}$ is supported in $\widetilde{D}$. Notice that we may assume that $D$ and $\widetilde{D}$ only depend on $\alpha_0$.

In order to apply the steepest descent method to study $\widehat{K}\p{\alpha,\beta}$, we need to understand the critical points of the phase $\widetilde{\Psi}_{\alpha,\beta}$. Notice first that when $\alpha = \beta$, the phase $\widetilde{\Psi}_{\alpha,\alpha}$ has a critical point at $ x= y = \alpha_x$ and $\xi = \alpha_\xi / \jap{\alpha_0}$. Moreover, this critical point is non-degenerate since the Hessian of $\widetilde{\Psi}_{\alpha,\alpha}$ is given there by
\begin{equation}\label{eq:la_Hessienne}
\begin{pmatrix}
\jap{ \alpha_0 }^{-1} D_{x,x}^2 \Phi_T(\alpha,\alpha_x) & I & 0 \\
I & 0 & - I \\
0 & - I & \jap{ \alpha_0 }^{-1} D_{x,x}^2 \Phi_S\p{\alpha_x,\alpha}
\end{pmatrix}
\end{equation}
and the determinant of this matrix is
\begin{equation*}
\begin{split}
\p{-1}^n \det\p{ \jap{ \alpha_0 }^{-1} \p{D_{x,x}^2 \Phi_T(\alpha,\alpha_x) + D_{x,x}^2 \Phi_S\p{\alpha_x,\alpha}}} \neq 0.
\end{split}
\end{equation*}
This determinant is in fact bounded away from zero uniformly in $\alpha$ and $\alpha_0$ since the imaginary part of the matrix $\jap{ \alpha_0 }^{-1} \p{D_{x,x}^2 \Phi_T(\alpha,\alpha_x) + D_{x,x}^2 \Phi_S\p{\alpha_x,\alpha}}$ is uniformly definite positive (this is part of Definition \ref{def:nonstandardphase} of an admissible phase). Hence, the Hessian \eqref{eq:la_Hessienne} of the phase $\widetilde{\Psi}_{\alpha,\alpha}$ is uniformly invertible.

We return now to our situation where $d_{KN}(\alpha,\alpha_0), d_{KN}(\beta,\alpha_0) \leq \eta$. To obtain the expected bounds, we need to study precisely the phase, and for this we start as in the proof of the ``Fundamental lemma'' \ref{lemma:fundamental-lemma}. It follows from the Implicit Function Theorem that, provided $\eta$ is small enough, $\widetilde{\Psi}_{\alpha,\beta}$ has a unique critical point $z(\alpha,\beta) = \p{x(\alpha,\beta), \xi(\alpha,\beta), y(\alpha,\beta)}$ near $(\alpha_{0,x},\alpha_{0,\xi}/\jap{\alpha_0},\alpha_{0,x})$, and this critical point is non-degenerate. Notice that if $y_c(\alpha,\beta)$ denotes the critical point of the phase $y \mapsto \Phi_T(\alpha,y) + \Phi_S(y,\beta)$ (which is defined by similar considerations) then we have
\begin{equation*}
x(\alpha,\beta) = y (\alpha,\beta) = y_c(\alpha,\beta)
\end{equation*} 
and
\begin{equation*}
\xi(\alpha,\beta) = - \frac{\nabla_x \Phi_T(\alpha,y_c(\alpha,\beta))}{\jap{ \alpha_0 }} = \frac{\nabla_y \Phi_S(y_c(\alpha,\beta), \beta)}{\jap{ \alpha_0 }}.
\end{equation*}
The critical value of $\widetilde{\Psi}_{\alpha,\beta}$ is given by 
\begin{equation*}
\begin{split}
 \widetilde{\Psi}_{\alpha,\beta}\p{x(\alpha,\beta), \xi(\alpha,\beta), y(\alpha,\beta)} =  \langle\alpha_0\rangle^{-1} \Phi_{TS}(\alpha,\beta),
\end{split}
\end{equation*}
where the phase $\Phi_{TS}$ is the critical value of $y \mapsto \Phi_T(\alpha,y) + \Phi_S(y,\beta)$. We can apply the Holomorphic Morse Lemma \ref{lemma:holomorphic_morse} to $\widetilde{\Psi}$. Provided $\eta$ is small enough, there exist holomorphic Morse coordinates $\rho_{\alpha,\beta}$ defined on a fixed complex neighbourhood $W$ of $D \times \widetilde{D} \times D$ and whose image contains a ball of fixed radius in $\C^n$ such that for every $z = (x,\xi,y) \in W$ we have
\begin{equation}\label{eq:Morse_encore_et_toujours}
\begin{split}
\widetilde{\Psi}_{\alpha,\beta}(z) = \frac{\Phi_{TS}(\alpha,\beta)}{\langle \alpha_0 \rangle} + \frac{i}{2} \rho_{\alpha,\beta}(z)^2.
\end{split}
\end{equation}
Since the proof of Lemma \ref{lemma:holomorphic_morse} is constructive -- see \cite[Lemma 2.7]{Sjostrand-82-singularite-analytique-microlocale} -- we obtain bounds on the derivatives of $\rho_{\alpha,\beta}$ and $\rho_{\alpha,\beta}^{-1}$ in function of the derivatives of $\widetilde{\Psi}_{\alpha,\beta}$. Via symbolic estimates, we deduce that these estimates are uniform as $\langle \alpha_0 \rangle \to \infty$. 

Differentiating \eqref{eq:Morse_encore_et_toujours}, we find that for $z = z(\alpha,\beta)$,
\begin{equation}\label{eq:Morse_differentiel}
i {}^t\p{ D \rho_{\alpha,\beta} (z(\alpha,\beta))} D \rho_{\alpha,\beta} (z(\alpha,\beta)) = D_{z,z}^2 \widetilde{\Psi}_{\alpha,\beta}(z\p{\alpha,\beta}).
\end{equation}
We denote $D \rho_{\alpha,\alpha} = A + i B$ the decomposition into real and imaginary part at $z=z(\alpha,\alpha)$, and deduce
\begin{equation*}
{}^t A A - {}^t B B = \Im D_{z,z}^2 \widetilde{\Psi}_{\alpha,\alpha}(z\p{\alpha,\alpha}).
\end{equation*}
When $\alpha = \beta$, we know that $\Im D_{z,z}^2 \widetilde{\Psi}_{\alpha,\alpha}(z\p{\alpha,\alpha})\geq 0$. Hence ${}^t A A \geq {}^t B B$. It follows that the kernel of $A$ is contained in that of $B$. However, since $D\rho_{\alpha,\alpha}(z(\alpha,\alpha))$ is invertible, these two kernels cannot intersect non-trivially. Hence $A$ is invertible, and this comes with uniform estimates, so it remains true when $\alpha \neq \beta$, provided $\eta$ is small enough.

Then, it follows by the Implicit Function Theorem that $V_{\alpha,\beta}:=\rho_{\alpha,\beta}(D \times \widetilde{D} \times D)$ is a graph over the reals, i.e.
\begin{equation*}
V_{\alpha,\beta} = \set{ w + i F_{\alpha,\beta}(w) : w \in \Re V_{\alpha,\beta} },
\end{equation*}
where the function $F_{\alpha,\beta}$ from a neighbourhood of $0$ in $\R^{3n}$ to $\R^{3n}$ is real-analytic (with symbolic estimates in $\alpha$ and $\beta$). Then, we use the coordinates $\rho_{\alpha,\beta}$ to change variables and write $\widehat{K}(\alpha,\beta)$ as the contour integral
\begin{equation}\label{eq:noyau_post_Morse}
\begin{split}
\widehat{K}(\alpha,\beta)  = e^{\frac{i}{h}\Phi_{TS}(\alpha,\beta)} \p{ \frac{\lambda}{2\pi} }^{\frac{3n}{2}} \int_{ V_{\alpha,\beta} }   e^{- \lambda \frac{w^2}{2}} g_{\alpha,\beta} \circ \rho_{\alpha,\beta}^{-1}(w) J \rho_{\alpha,\beta}^{-1}(w) \mathrm{d}z,
\end{split}
\end{equation}
where $J \rho_{\alpha,\beta}^{-1} = \det\p{D \rho_{\alpha,\beta}^{-1}}$ denotes the Jacobian of $\rho_{\alpha,\beta}^{-1}$. Then, we use the homotopy $ (t,w) \mapsto w + i t F_{\alpha,\beta}(w)$ to shift contour in \eqref{eq:noyau_post_Morse} and replace the integral over $V_{\alpha,\beta}$ by an integral over $\Re V_{\alpha,\beta}$. We denote by $\sigma_{\alpha,\beta}(w)$ the amplitude in the previous integral, to find, using Stokes' Formula that
\begin{equation}\label{eq:steepest_descendu}
\widehat{K}(\alpha,\beta) = e^{\frac{i}{h}\Phi_{TS}(\alpha,\beta) } \p{ \frac{\lambda}{2\pi} }^{\frac{3n}{2}}  \int_{\Re V_{\alpha,\beta} }\hspace{-10pt} e^{- \lambda \frac{w^2}{2}}   \tilde{\sigma}_{\alpha,\beta}(w) \mathrm{d}w + \overline{D}(\alpha,\beta),
\end{equation}
where $\tilde{\sigma}_{\alpha,\beta}$ denotes a $\G^s$ almost analytic extension for $\sigma_{\alpha,\beta}$ (obtained by replacing $\chi, \theta$ and $p$ by their almost analytic extensions given either by Lemma \ref{lemma:almost-analytic-extension-Gs} or Remark \ref{remark:aae_for_symbol}) and $\overline{D}(\alpha,\beta)$ is the main part in Stokes' Formula. The main technical point in the proof of Lemma \ref{lemma:TPS-close-diagonal} will be to get the following bound on $\overline{D}(\alpha,\beta)$.
\begin{lemma}\label{lemma:estimee_reine}
Assume that $\eta$ and $\tau_0$ are small enough. Then there exists $C>0$ such that
\begin{equation*}
\begin{split}
\overline{D}(\alpha,\beta) = \O_{C^\infty}\left(\exp\left( - \frac{\lambda^{\frac{1}{2s-1}}}{C} - \frac{\lambda^{\frac{1}{s}} d_{KN}(\alpha,\beta) }{ C}  \right) \right).
\end{split}
\end{equation*}
\end{lemma}

Let us postpone the proof of Lemma \ref{lemma:estimee_reine} to the end of this section and explain how it allows us to end the proof of Lemma \ref{lemma:TPS-close-diagonal}. From Lemma \ref{lemma:estimee_reine}, the term $\overline{D}(\alpha,\beta)$ may be considered as part of the error term in \eqref{eq:noyau_reduit_TPS}. Let us study the integral in the right hand side of \eqref{eq:steepest_descendu}. We know that $\Re V_{\alpha,\beta}$ contains a fixed ball $D_0$. From coercivity of the phase $w \mapsto w^2$, we may split the integral in \eqref{eq:steepest_descendu} between the integral over $D_0$ and the integral over $\Re V_{\alpha,\beta} \setminus D_0$ and consider the latter as part of the error term in \eqref{eq:steepest_descendu}. Indeed, it follows from an $L^1$ estimate that the integral over $\Re V_{\alpha,\beta} \setminus D_0$ is a
\begin{equation*}
\begin{split}
\O_{C^\infty}\p{\exp\p{-  \frac{ \lambda }{C}}}.
\end{split}
\end{equation*}
Notice indeed that the dependence on $\alpha$ and $\beta$ of the domain of integration is superficial since the integrand is supported away from the boundary of the domain of integration $\Re V_{\alpha,\beta}$. The integral over $D_0$
\begin{equation*}
\begin{split}
e(\alpha,\beta):=\p{\frac{\lambda}{2 \pi }}^{\frac{3n}{2}}  \int_{D_0} e^{- \lambda \frac{w^2}{2}}  \tilde{\sigma}_{\alpha,\beta}(w) \mathrm{d}w
\end{split}
\end{equation*}
is a Gaussian integral with a $\Gs$ amplitude, so that it is $\Gs$ symbol. It will be sufficient to study it via the $\mathcal{C}^\infty$ stationary phase method. In particular, it is given at first order by
\begin{equation*}
\begin{split}
e(\alpha,\beta) = \frac{1}{(2 \pi h)^n} e_0(\alpha,\beta) \textup{ mod } h^{-n+1} S_{KN}^{m-1},
\end{split}
\end{equation*}
where the symbol $e_0(\alpha,\beta)$ is given on the diagonal by
\begin{equation*}
\begin{split}
e_0(\alpha,\alpha) 	& = (2\pi h)^n \tilde{\sigma}_{\alpha,\alpha}(0) = (2\pi h)^n \tilde{g}_{\alpha,\alpha} \circ \rho_{\alpha,\alpha}^{-1}(0)  J \rho_{\alpha,\beta}^{-1}(0))\\
    				& = \tilde{p}\p{\alpha} (2\pi h)^{\frac{3n}{2}} \jap{\alpha_0}^{- \frac{n}{2}} J \rho_{\alpha,\alpha}^{-1}\p{0} a\p{\alpha,\alpha_x} b(\alpha_x,\alpha). 
\end{split}
\end{equation*}
Here, the bump function $\chi$ and $\theta$ do not play any role since they take the value $1$ respectively at $\alpha_x$ and $\alpha_\xi$. Recalling \eqref{eq:Morse_differentiel} and the expression \eqref{eq:la_Hessienne} for the Hessian of $\widetilde{\Psi}_{\alpha,\alpha}$ we find that
\begin{equation*}
\begin{split}
\jap{\alpha_0}^{- \frac{n}{2}} J \rho_{\alpha,\alpha}^{-1}\p{0} = \p{\det\p{-i \p{ D_{x,x}^2 \Phi_T(\alpha,\alpha_x) + D_{x,x}^2 \Phi_S\p{\alpha,\alpha_x}}}}^{- \frac{1}{2}}, 
\end{split}
\end{equation*}
for a certain determination of the square root. Then, we have
\begin{equation}\label{eq:une_expression_e0}
\begin{split}
e_0\p{\alpha,\alpha} = \tilde{p}\p{\alpha} \p{2 \pi h }^{\frac{3n}{2}} \frac{a\p{\alpha,\alpha_x} b(\alpha_x,\alpha)}{\p{\det\p{-i \p{ D_{x,x}^2 \Phi_T(\alpha,\alpha_x) + D_{x,x}^2 \Phi_S\p{\alpha,\alpha_x}}}}^{\frac{1}{2}}}.
\end{split}
\end{equation}
In this expression, we may replace $a$ and $b$ by their principal part. It follows then from the fact that $ST = I$ and from the expression given in Lemma \ref{lemma:ST-pseudo} for the principal symbol of $ST$ that
\begin{equation}\label{eq:blabla}
\begin{split}
\p{2 \pi h }^{\frac{3n}{2}} \frac{a\p{\alpha,\alpha_x} b(\alpha_x,\alpha)}{\p{\det\p{-i \p{ D_{x,x}^2 \Phi_T(\alpha,\alpha_x) + D_{x,x}^2 \Phi_S\p{\alpha,\alpha_x}}}}^{\frac{1}{2}}} = 1
\end{split}
\end{equation}
when $\alpha$ is real. Then, by analytic continuation principle, this formula remains true for complex $\alpha$. We see that the good determination of the square root in \eqref{eq:une_expression_e0} gives a right hand side equals to $1$ in \eqref{eq:blabla} by a homotopy argument (we have the good sign on the reals).

In fact, if we have the wrong determination of the square root in \eqref{eq:une_expression_e0}, it means that we misoriented $V_{\alpha,\beta}$ in \eqref{eq:noyau_post_Morse}, and these two sign error cancel. Finally, with \eqref{eq:une_expression_e0} and \eqref{eq:blabla}, we find that $e_0(\alpha,\alpha) = \tilde{p}\p{\alpha}$, and Lemma \ref{lemma:TPS-close-diagonal} is proved in the case $s > 1$ (up to the proof of the key estimate Lemma \ref{lemma:estimee_reine}).
\end{proof}

In order to complete the proof of Lemma \ref{lemma:TPS-close-diagonal} in the case $s> 1$, we need now to establish Lemma \ref{lemma:estimee_reine}.
\begin{proof}[Proof of Lemma \ref{lemma:estimee_reine}]
We recall that we are studying points $\alpha = (\alpha_x,\alpha_\xi)$  and $\beta = (\beta_x,\beta_\xi) \in \Lambda$ that lie in a neighbourhood of $\alpha_0 = (\alpha_{0,x},\alpha_{0,\xi}) \in T^\ast M$. First of all, notice that we have
\begin{equation*}
\overline{D}\p{\alpha,\beta} =  e^{\frac{i}{h}\Phi_{TS}(\alpha,\beta)} \left(\frac{ \lambda}{2 \pi } \right)^{\frac{3n}{2}} \int B^*\p{ e^{ - \lambda \frac{\tilde{w}^2}{2}  } (\rho_{\alpha,\beta})_\ast (\bar{\partial} \tilde{g}_{\alpha,\beta} \wedge \mathrm{d}z)(\tilde{w})},
\end{equation*}
where $B$ denotes the homotopy $\p{t,w} \mapsto w + i t F_{\alpha,\beta}\p{w}$ and the integral is over $\left[0,1\right] \times \Re V_{\alpha,\beta}$ . We need to understand how the imaginary part of the phase and the decay of $\bar{\partial} \tilde{g}_{\alpha,\beta}$ collaborate to make the integrand small. The idea will be to reduce this to some positivity estimate on the phase. To do so write $z^t_{\alpha,\beta}(w) = \rho_{\alpha,\beta}^{-1}\p{B(t,w)}$. In view of the estimate given in Remark \ref{remark:size_d_bar}, we observe that 
\[
 (\rho_{\alpha,\beta})_\ast (\bar{\partial} \tilde{g}_{\alpha,\beta} \wedge \mathrm{d}z)(B(t,w)) = \O_{ C^\infty} \left( \exp\left( - \frac{1}{ C |\Im z^t_{\alpha,\beta} |^{\frac{1}{s-1}}} \right) \right).
\]
Next, we observe that
\[
\Im\left( \frac{\Phi_{TS}(\alpha,\beta)}{\langle \alpha_0 \rangle} + \frac{i}{2}B(t,w)^2 \right) = \Im\left( \widetilde{\Psi}_{\alpha,\beta}(z^1_{\alpha,\beta}(w)) \right) + \frac{1-t^2}{2}|F_{\alpha,\beta}(w)|^2.
\]
From this, we deduce that we have pointwise bounds
\[
\begin{split}
e^{\frac{i}{h}\Phi_{TS}(\alpha,\beta)} &\left(\frac{\lambda}{2\pi} \right)^{\frac{3n}{2}}  B^*\p{ e^{ - \lambda \frac{w^2}{2}  } (\rho_{\alpha,\beta})_\ast (\bar{\partial} \tilde{g}_{\alpha,\beta} \wedge \mathrm{d}z)(w)} \\
		= \O_{L^\infty}&\left(\lambda^{\frac{3n}{2}}\exp\left( - \lambda \Im\left( \widetilde{\Psi}_{}(z^1(w)) \right) - \lambda \frac{1-t^2}{2}|F(w)|^2 - \frac{1}{ C |\Im z^t |^{\frac{1}{s-1}}} \right) \right).
\end{split}
\]
(we omitted the $\alpha,\beta$ dependence in the second line). Each time we differentiate with respect to $\alpha,\beta$, we lose at most a power of $\lambda$ in this estimate, so that it will suffice for our purpose to prove that, for $w \in \Re V_{\alpha,\beta}$ and $t \in \left[0,1\right]$, the quantity
\begin{equation}\label{eq:quantity-to-be-estimated-1}
\lambda \Im\left( \widetilde{\Psi}_{\alpha,\beta}(z^1_{\alpha,\beta}(w)) \right) + \lambda \frac{1-t^2}{2}|F_{\alpha,\beta}(w)|^2 + \frac{1}{ C|\Im z^t_{\alpha,\beta}(w) |^{\frac{1}{s-1}}} 
\end{equation}
is larger than 
\begin{equation*}
\begin{split}
C_1^{-1} \lambda^{\frac{1}{2s - 1}} + C_1^{-1} \lambda^{\frac{1}{s}} d_{KN}(\alpha,\beta)
\end{split}
\end{equation*}
for some $C_1 > 0$. To obtain such an estimate, we start by observing that $z^1_{\alpha,\beta}(w)$ is always a real point. Since $\rho_{\alpha,\beta}$ is uniformly bi-Lipschitz, we deduce that, for some $C > 0$,
\[
|\Im z^t_{\alpha,\beta}(w)| \leq  |\Im z^1_{\alpha,\beta}(w)| + C( |t-1| |F_{\alpha,\beta}(w)| ) \leq C (1-t) |F_{\alpha,\beta}(w)|.
\]
This suggests to consider for $u = \va{F_{\alpha,\beta}(w)}$ the infimum (we perform the change of variable ``$r = 1 - t$'')
\begin{equation}\label{eq:un_infimum_2}
\inf_{ r \in[0,1]} \lambda r u^2 + \frac{1}{(r u)^{\frac{1}{s-1}}}.
\end{equation}
The function of $r$ above has a global minimum over $\R_+^*$, attained at 
\[
r_* = \left( \frac{1}{s-1}\frac{1}{\lambda u^{\frac{2s-1}{s-1}}} \right)^{1- \frac{1}{s}}.
\]
Thus, the infimum we look for is always greater than the global minimum which happens to be
\begin{equation}\label{eq:global_infimum_2}
s (s-1)^{\frac{1}{s} - 1} (\lambda u)^{1/s}.
\end{equation}
However, if the argument $r_*$ of the minimum is greater than $1$, then the infimum \eqref{eq:un_infimum_2} is attained at $r= 1$ (we compute the infimum of a decreasing function of $r$ over $\left[0,1\right]$). Hence, if $r_* \geq 1$, the infimum \eqref{eq:un_infimum_2} is greater than 
\begin{equation*}
\begin{split}
\lambda u^2 + u^{- \frac{1}{s-1}} & \geq u^{- \frac{1}{s-1}} = \p{r_*^{\frac{s}{s-1}} (s-1) \lambda}^{\frac{1}{2s-1}} \\
       & \geq \p{s-1}^{\frac{1}{2s -1}} \lambda^{\frac{1}{2 s- 1}}.
\end{split}
\end{equation*}
On the other hand, if $r_* \leq 1$, then we may bound from below the global infimum \eqref{eq:global_infimum_2} by noticing that
\begin{equation*}
\begin{split}
\p{\lambda u}^{\frac{1}{s}} = \frac{r_*^{- \frac{1}{2s-1}}}{\p{s-1}^{\frac{s-1}{s\p{2s-1}}}} \lambda^{\frac{1}{2s-1}} \geq \frac{1}{\p{s-1}^{\frac{s-1}{s\p{2s-1}}}} \lambda^{\frac{1}{2s-1}}.
\end{split}
\end{equation*}
Finally, we find that there is a constant $C$ that only depends on $s$ such that the infimum \eqref{eq:un_infimum_2} is bounded from below by
\begin{equation}\label{eq:lower_bound_2}
\frac{1}{C}\p{\p{\lambda u}^{\frac{1}{s}} + \lambda^{\frac{1}{2s-1}} }.
\end{equation}
This gives a lower bound for the quantity in \eqref{eq:quantity-to-be-estimated-1} in the form
\begin{equation}\label{eq:quantity-to-be-estimated-2}
\lambda \Im\left( \widetilde{\Psi}_{\alpha,\beta}(z^1_{\alpha,\beta}(w) \right) + \frac{\lambda^{\frac{1}{2s-1}}}{C} +   \frac{\left( \lambda |F_{\alpha,\beta}(w)| \right)^{\frac{1}{s}}}{C}, 
\end{equation}
for some $C > 0$. Consequently, the key estimate Lemma \ref{lemma:estimee_reine} will be a consequence of the following positivity estimate on the phase $\widetilde{\Psi}_{\alpha,\beta}$:
\begin{lemma}\label{lemma:estimee_positivite_phase}
Let $C > 0$. Then, if $\eta$ and $\tau_0$ are small enough, there is a constant $C' > 0$ such that for $w \in \Re V_{\alpha,\beta}$ we have
\begin{equation}\label{eq:lower-bound-on-the-phase}
\begin{split}
\lambda \Im\left( \widetilde{\Psi}_{\alpha,\beta}(z_{\alpha,\beta}^1\p{w}) \right) &+ \frac{\lambda^{\frac{1}{2s-1}}}{C} +   \frac{\left( \lambda |F_{\alpha,\beta}(w)|  \right)^{\frac{1}{s}}}{C} \\
	& \qquad \qquad \geq  \frac{\lambda^{\frac{1}{2s-1}}}{C'} + \frac{\lambda^{\frac{1}{s}} d_{KN}(\alpha,\beta)}{C'}.
\end{split}
\end{equation}
\end{lemma}
\end{proof}

In the regions where the imaginary part of the phase alone is not sufficient to control the right hand side, we will see that $|F|$ is of the same order as $|\nabla_z \Phi|$, so that this estimates expresses the fact that $\Im \Phi$ and $\nabla_z \Phi$ cannot be simultaneously too small.
\begin{proof}[Proof of Lemma \ref{lemma:estimee_positivite_phase}]
Recall that the phase $\widetilde{\Psi}_{\alpha,\beta}$ is defined by \eqref{eq:la_phase_psi_tilde}. We set $z_{\alpha,\beta}^1(w) = (x_{\alpha,\beta}^1(w),\xi_{\alpha,\beta}^1(w),y_{\alpha,\beta}^1(w))$. When there is no ambiguity, we will write respectively $z,x,\xi$ and $y$ instead of $z_{\alpha,\beta}^{1}\p{w}$, $x_{\alpha,\beta}^{1}\p{w}$, $\xi_{\alpha,\beta}^{1}\p{w}$ and $y_{\alpha,\beta}^{1}\p{w}$. Notice that, since $x,y$ and $\xi$ are real
\begin{equation*}
\begin{split}
\Im \widetilde{\Psi}_{\alpha,\beta}(z) = \Im\left(\frac{\Phi_{T}(\alpha,x) + \Phi_{S}\p{y,\beta}}{\jap{\alpha_0 }}\right).
\end{split}
\end{equation*}

Let us now obtain some useful preliminary estimates. First, notice that $\Lambda$ is $\mathcal{C}^1$ close to $T^*M$. The arguments that led to Lemma \ref{lmdist} show that since $\alpha,\beta \in \Lambda$ we have, for some $C > 0$ (recall that $\lambda = \jap{\alpha_0}/h$),
\begin{equation}\label{eq:controle_partie_imaginaire}
\begin{split}
d_{KN}(\Im \alpha,\Im \beta) \leq C \tau_0 \lambda^{\frac{1}{s} - 1} d_{KN}(\alpha,\beta),
\end{split}
\end{equation}
in addition to Lemma \ref{lmdist} of course. In particular, provided $\tau_0$ is small enough, for some $C > 0$, we have
\[
\frac{1}{C}d_{KN}(\Re \alpha,\Re\beta) \leq d_{KN}(\alpha,\beta) \leq C d_{KN}(\Re \alpha,\Re\beta).
\]

The first estimate we prove on the phase is that (for some $C > 0$)
\begin{equation}\label{eq:lower-bound-Phase-first-take}
\begin{split}
\Im\left(\frac{\Phi_{T}(\alpha,x) + \Phi_{S}\p{y,\beta}}{\jap{\alpha_0 }} \right)  &\geq  \frac{1}{C} \p{\va{ x - \alpha_x}^2 + \va{y - \beta_x}^2} - C \tau_0 \lambda^{\frac{1}{2s-1}-1}\\ 
&\hspace{-60pt} - C \tau_0 \lambda^{\frac{1}{s}-1}  \left(\va{\alpha_x -x} + \va{y - \beta_x} + d_{KN}(\alpha,\beta) \right).
\end{split}
\end{equation}
For this, denote respectively by $Q_{T,\alpha}$ and $Q_{S,\beta}$ the second derivatives of $ r \mapsto 2\Phi_{T}(\alpha,r)$ at $r=\alpha_x$ and $r \mapsto 2\Phi_{S,\beta}(r,\beta)$ at $r=\beta_x$. Taylor's formula at order $2$ gives then 
\begin{equation}\label{eq:approximation_phase}
\begin{split}
\Phi_{T}(\alpha,x) + \Phi_{S}\p{y,\beta} & = \jap{\alpha_\xi, \alpha_x - x} + \jap{\beta_\xi,y - \beta_x} \\
     	&+ Q_{T,\alpha}\p{y - \alpha_x,y - \alpha_x} + Q_{S,\beta}\p{y - \beta_x,y-\beta_x} \\ 
     	&+ \langle\alpha_0\rangle\O\p{\va{x- \alpha_x}^3 + \va{y - \beta_x}^3}.
\end{split}
\end{equation} 
We start by estimating the imaginary part of the first line in the right hand side of \eqref{eq:approximation_phase}. To do so, we write
\begin{equation*}
\begin{split}
& \Im\p{\jap{\alpha_\xi,\alpha_x - x} + \jap{\beta_\xi,y - \beta_x}} \\ & \qquad \qquad \qquad = \jap{\Im \alpha_\xi ,\Re \alpha_x - x} + \jap{\Im \beta_\xi, y - \Re \beta_x} + \jap{ \Re \alpha_\xi - \Re \beta_\xi, \Im \alpha_x} \\ & \qquad \qquad \qquad \qquad \qquad + \jap{\Re \beta_\xi ,\Im \alpha_x - \Im \beta_x}. 
\end{split}
\end{equation*}
Since $|\Im\alpha|, |\Im\beta| \lesssim \tau_0 \lambda^{1/s-1}$, and according to \eqref{eq:controle_partie_imaginaire}, we estimate each term to find
\begin{equation*}
\begin{split}
& \frac{1}{\langle \alpha_0 \rangle} \left| \Im\p{\jap{\alpha_\xi,\alpha_x - x} + \jap{\beta_\xi,y - \beta_x}}  \right| \\
		&\qquad \qquad \qquad \qquad \qquad \leq C \tau_0 \lambda^{\frac{1}{s}-1} \left(\va{\Re \alpha_x -x} + \va{y - \Re \beta_x} + d_{KN}(\alpha,\beta) \right) .
\end{split}
\end{equation*}
We also compute
\begin{equation*}
\begin{split}
 \Im (Q_{T,\alpha}(x - \alpha_x, x - \alpha_x)) & = \Im Q_{T,\alpha}\p{ x- \Re \alpha_x, x- \Re \alpha_x}\\ 
  - \Im Q_{T,\alpha} & \p{\Im \alpha_x, \Im \alpha_x}  + 2 \Re Q_{T,\alpha}\p{\Im \alpha_x, x- \Re \alpha_x}.
\end{split}
\end{equation*}
By definition of an admissible phase, the matrix $\Im Q_{T,\alpha}/ \jap{\alpha_0}$ is uniformly definite positive (provided that $\tau_0$ is small enough). Using $|\Im\alpha| \lesssim \tau_0 \lambda^{1/s-1}$ again, we find that
\begin{equation*}
\begin{split}
& \frac{\Im \p{Q_{T,\alpha}\p{x - \alpha_x, x - \alpha_x}}}{\langle\alpha_0\rangle}  \\ 
& \qquad \geq \frac{1}{C} \va{x - \Re \alpha_x}^2 - C \tau_0 \lambda^{\frac{1}{s}-1} \va{ x- \Re \alpha_x} - C \tau_0^2 \lambda^{\frac{2}{s}-2} \\
    & \qquad \geq \frac{1}{C} \va{x - \alpha_x}^2 - C \tau_0 \lambda^{\frac{1}{s}-1} \va{ x- \alpha_x} - C \tau_0^2 \lambda^{\frac{2}{s}-2}.
\end{split}
\end{equation*}
We have a similar estimate for the term involving $Q_{S,\beta}$ and \eqref{eq:lower-bound-Phase-first-take} follows (the $\O$ in \eqref{eq:approximation_phase} is negligible thanks to the positive quadratic term). The $\lambda^{2/s-2}$ term  is controlled by $\lambda^{1/(2s-1)-1}$ because for $s\geq 1$, we have $2/s-1 \leq 1/(2s-1)$.

Examining \eqref{eq:lower-bound-Phase-first-take}, we see that, in order to establish \eqref{eq:lower-bound-on-the-phase}, we only need to prove that, given $C > 0$, for $\tau_0$ small enough we have:
\begin{equation}\label{eq:still_another_bound}
\begin{split}
&\lambda\p{\va{x - \alpha_x}^2 + \va{y - \beta_x}^2} + \lambda^{\frac{1}{2s-1}} + \left( \lambda |F_{\alpha,\beta}(w)|  \right)^{\frac{1}{s}} \\ 
& \qquad \qquad \qquad \qquad \qquad \geq C \tau_0 \lambda^{\frac{1}{s}}  \left(\va{\alpha_x -x} + \va{y - \beta_x} + d_{KN}(\alpha,\beta) \right)
\end{split}
\end{equation}
We get rid first of the term $C \tau_0 \lambda^{1/s} (\va{\alpha_x -x} + \va{y - \beta_x})$ on the right hand side. To do so, there are two possibilities, the first one is that
\begin{equation}\label{eq:petit_ou_pas}
\begin{split}
\va{\alpha_x -x} + \va{\beta_x - y} \geq \lambda^{\frac{1}{s} - 1},
\end{split}
\end{equation}
in which case $C \tau_0 \lambda^{\frac{1}{s}} (\va{\alpha_x -x} + \va{y - \beta_x})$ is controlled by the term $C^{-1}\lambda (\va{\alpha_x -x}^2 + \va{y - \beta_x}^2)$ on the left hand side (provided that $\tau_0$ is small enough). If \eqref{eq:petit_ou_pas} does not hold then 
\begin{equation*}
\begin{split}
C \tau_0 \lambda^{\frac{1}{s}} \p{\va{\alpha_x -x} + \va{\beta_x -y}}\leq C \tau_0 \lambda^{\frac{2}{s} -1} \leq C \tau_0 \lambda^{\frac{1}{2s - 1}}
\end{split}
\end{equation*}
and this term is controlled by $C^{-1} \lambda^{\frac{1}{2s - 1}}$ on the left hand side of \eqref{eq:still_another_bound}.

Dealing with this first term, we have further reduced our problem, and it suffices now to prove that, given $C>0$, for $\tau_0>0$ small enough, we have
\begin{equation}\label{eq:lower-bound-to-prove}
\lambda^{\frac{1}{2s-1}} + \left( \lambda |F_{\alpha,\beta}(w)|\right)^{\frac{1}{s}} + \lambda (|x-\alpha_x|^2  + | y - \beta_x|^2 ) \geq C \tau_0 \lambda^{\frac{1}{s}} d_{KN}(\alpha,\beta).
\end{equation}
Without loss of generality, we may assume that 
\begin{equation}\label{eq:une_hyp}
\begin{split}
d_{KN}(\alpha,\beta) \geq \lambda^{\frac{1}{s}- 1},
\end{split}
\end{equation}
since otherwise the right hand side of \eqref{eq:lower-bound-to-prove} is controlled by $\lambda^{\frac{1}{2s-1}}$ in the left hand side. We may also assume that 
\begin{equation}\label{eq:deux_hyp}
\begin{split}
|x-\alpha_x| + |y-\beta_x| \leq \epsilon d_{KN}(\alpha,\beta),
\end{split}
\end{equation}
where $\epsilon > 0$ is very small (to be chosen later). Indeed, otherwise, $\lambda(|x-\alpha_x|^2 + |y-\beta_x|^2)$ would control $\lambda d_{KN}(\alpha,\beta)^2$, itself larger than $\lambda^{1/s}d_{KN}(\alpha,\beta)$, thanks to \eqref{eq:une_hyp}. It follows from \eqref{eq:deux_hyp} that
\[
|\alpha_x - \beta_x| \leq |x-\alpha_x| + |y-\beta_x|  + |x-y| \leq |x-y| + \epsilon d_{KN}(\alpha,\beta),
\]
and thus $|\alpha_x - \beta_x| \leq C |x-y| + C\epsilon |\alpha_\xi - \beta_\xi|/\langle\alpha_0\rangle$. We obtain, using \eqref{eq:controle_partie_imaginaire} and assuming that $\tau_0$ is small enough, that
\begin{equation}\label{eq:consequence-une-hyp}
d_{KN}(\alpha,\beta) \leq C|x-y| + C\frac{|\Re \alpha_\xi- \Re\beta_\xi|}{\langle \alpha_0 \rangle}.
\end{equation}

We want now to give a lower bound on $\va{F_{\alpha,\beta}\p{w}}$. To do so, recall the gradient of the phase $\widetilde{\Psi}_{\alpha,\beta}$:
\begin{equation*}
\begin{split}
\nabla_{z} \widetilde{\Psi}_{\alpha,\beta}\p{z}  = \begin{pmatrix}\frac{\nabla_x\Phi_T\p{\alpha,x }}{ \jap{ \alpha_0 } } + \xi \\
x - y \\ 
 \frac{\nabla_y \Phi_S\p{y,\beta}}{ \jap{ \alpha_0 } } - \xi \end{pmatrix}
\end{split}
\end{equation*}
Hence, it follows from Definition \ref{def:nonstandardphase} of an admissible phase and Taylor's formula that (the value of $C > 0$ may change from one line to another)
\begin{equation}\label{eq:grande_partie_reelle}
\begin{split}
\left| \Re \nabla_{z} \widetilde{\Psi}_{\alpha,\beta}\p{z}\right| &\geq \frac{1}{C}\va{\frac{\Re \alpha_\xi}{\jap{\alpha_0}} - \xi} + \frac{1}{C}\va{\xi - \frac{\Re \beta_\xi}{\jap{\alpha_0 }}} + \frac{1}{C}|x-y|\\
		& \qquad \qquad - C \p{\va{x - \alpha_x} + \va{y - \beta_x}} \\ 
		& \geq \frac{1}{C}\p{\frac{ |\Re \alpha_\xi - \Re \beta_\xi|}{ \langle\alpha_0\rangle} + \va{x -y}} - C \epsilon d_{KN}(\alpha,\beta) \\
		& \geq \p{\frac{1}{C} - C \epsilon} d_{KN}(\alpha,\beta) \\
		& \geq \frac{1}{C} d_{KN}(\alpha,\beta) \\
\end{split}
\end{equation}
where we applied \eqref{eq:deux_hyp} and \eqref{eq:consequence-une-hyp} (assuming that $\epsilon > 0$ was small enough). From a similar computation, we get, recalling $|\Im\alpha|, |\Im\beta| \lesssim \tau_0 \lambda^{1/s - 1}$, \eqref{eq:deux_hyp} and \eqref{eq:consequence-une-hyp}, that
\begin{equation}\label{eq:small_partie_imaginaire}
\begin{split}
\va{\Im \nabla_z \widetilde{\Psi}_{\alpha,\beta}\p{z}} \leq C \p{\tau_0 + \epsilon} d_{KN}(\alpha,\beta).
\end{split}
\end{equation}
Now we compute the gradient of the phase $\widetilde{\Psi}_{\alpha,\beta}$ in Morse coordinates and find
\begin{equation*}
\begin{split}
-i {}^t D \rho_{\alpha,\beta}(z)^{-1} \nabla_z \widetilde{\Psi}_{\alpha,\beta}\p{z} =  w + i F_{\alpha,\beta}\p{w}.
\end{split}
\end{equation*}
Hence, we have
\begin{equation*}
\begin{split}
F_{\alpha,\beta}(w) & = -\Re( {}^t D\rho_{\alpha,\beta}(z)^{-1} ) \Re \nabla_z \widetilde{\Psi}_{\alpha,\beta}(z) \\
   & \qquad \qquad + \Im( {}^t D\rho_{\alpha,\beta}(z)^{-1} )  \Im \nabla_z\widetilde{\Psi}_{\alpha,\beta}(z).
\end{split}
\end{equation*}
Following the argument after \eqref{eq:Morse_differentiel}, applied to $D\rho_{\alpha,\beta}^{-1}$, we see that its real part is uniformly invertible. Indeed, it suffices to notice that the imaginary part of $- (D_{z,z}^2 \widetilde{\Psi}_{\alpha,\alpha})^{-1}$ is positive at the critical point. 

Hence, it follows from \eqref{eq:grande_partie_reelle} and \eqref{eq:small_partie_imaginaire} that (for some $C > 0$ that may change from one line to another, and provided that $\tau_0$ and $\epsilon$ are small enough)
\begin{equation*}
\begin{split}
\va{F_{\alpha,\beta}\p{w}}& \geq \frac{1}{C} d_{KN}(\alpha,\beta) - C\p{\tau_0 + \epsilon} d_{KN}(\alpha,\beta) \\
    & \geq \frac{1}{C} d_{KN}(\alpha,\beta).
\end{split}
\end{equation*}
Notice also that, since the quantity $|F_{\alpha,\beta}\p{w}|$ remains bounded, we have the lower bound $|F_{\alpha,\beta}\p{w}|^{1/s} \geq |F_{\alpha,\beta}\p{w}| / C$, and \eqref{eq:lower-bound-to-prove} holds for $\tau_0$ small enough, so that the proof of Lemma \ref{lemma:estimee_positivite_phase} is complete. 
\end{proof}

It remains to explain why Lemmas \ref{lemma:TPS-far-diagonal} and \ref{lemma:TPS-close-diagonal} remains true in the case $s = 1$. The main difficulty (that makes us distinguish that case) is the lack of partition of unity in the real-analytic category. This issue is solved by an application of Proposition \ref{prop:pseudor_lagrangien_alt}. In fact, most of the proof of Lemma \ref{lemma:TPS-close-diagonal} in the case $s>1$ was devoted to the estimation of integrals involving the Cauchy--Riemann operator applied to the symbol of $P$. In the case $s = 1$, these integrals are null, hence a much less complicated proof.

\begin{proof}[Sketch of proof of Lemmas \ref{lemma:TPS-far-diagonal} and \ref{lemma:TPS-close-diagonal} in the case $s = 1$]
We start\\with an application of Proposition \ref{prop:pseudor_lagrangien_alt} to compute the kernel of the operator $PS$. We see, as in Remark \ref{remark:ce_quest_PS}, that the kernel of $PS$ has the same properties as the kernel of an adjoint analytic FBI transform, except maybe the order and the ellipticity of the symbol. Hence, we will assume for simplicity that $P = I$ and compute only the kernel of $TS$. The only difference in the proof in the general case is when computing the principal part of the symbol, but this computation has already been tackled in the case $s > 1$ and there is no difference when $s = 1$.

The kernel of $TS$ is given by the formula
\begin{equation}\label{eq:noyau_TS_s1}
\begin{split}
K_{TS}\p{\alpha,\beta} = \int_{M} K_T\p{\alpha,y} K_S\p{y,\beta} \mathrm{d}y.
\end{split}
\end{equation}
This is much simpler than the formula \eqref{eq:noyau_TPS_integrale} that we used in the case $s >1$ since there is no distribution involved. To prove that $K_{TS}\p{\alpha,\beta}$ is negligible when $\alpha_x$ and $\beta_x$ are away from each other, we apply the same non-stationary phase argument as in the proof of Lemma \ref{lemma:T_non_stationary} (or as in the case $s > 1$). When $y$ is near $\alpha_x$, the phase in $K_T\p{\alpha,y}$ is non-stationary and $K_T\p{y,\beta}$ is small in $\G^1$. The converse happens when $y$ is near $\beta_x$, and when $y$ is away from $\alpha_x$ and $\beta_x$ both $K_T\p{\alpha,y}$ and $K_S\p{y,\beta}$ are small. Here, notice that we do not need partitions of unity to split the integral \eqref{eq:noyau_TS_s1} due to the coercivity of the imaginary parts of the phases $\Phi_T$ and $\Phi_S$.

When $\alpha_x$ and $\beta_x$ are closed, we may just neglect the $y$'s in \eqref{eq:noyau_TS_s1} that are away from $\alpha_x$ and $\beta_x$, and work in local coordinates, studying the integral on some ball $D$ (the points $\alpha_x$ and $\beta_x$ remain uniformly away from the boundary of $D$)
\begin{equation}\label{eq:le_truc_local_s1}
\begin{split}
\widehat{K}\p{\alpha,\beta} \coloneqq \int_{D} e^{i \frac{\Phi_T\p{\alpha,y} + \Phi_S\p{y,\beta}}{h}} a(\alpha,y) b\p{\beta,y} \mathrm{d}y. 
\end{split}
\end{equation}
If $\alpha_\xi$ and $\beta_\xi$ are away from each other (for the Kohn--Nirenberg metric), it follows from Definition \ref{def:nonstandardphase} of an admissible phase that the phase in \eqref{eq:le_truc_local_s1} is non-stationary (provided that $\alpha_x$ and $\beta_x$ are close enough), and we may apply the non-stationary phase method Proposition \ref{prop:non-stationary-analytic} after a proper rescaling  as in \S \ref{subsec:further_tricks}. This ends the proof of Lemma \ref{lemma:TPS-far-diagonal} in the case $s=1$.

We turn now to the proof of Lemma \ref{lemma:TPS-close-diagonal}: we want consequently to understand \eqref{eq:le_truc_local_s1} when $\alpha$ and $\beta$ are near the diagonal. As in the case $s > 1$, we will rely on an application of the Stationary Phase Method. However, the situation is much simpler here, since we can apply Proposition \ref{prop:HSP_symbol}. In particular, we do not need to restrict to $\alpha,\beta \in \Lambda$: the estimate \eqref{eq:noyau_reduit_TPS} hold for $\alpha,\beta \in (T^* M)_{\epsilon_1}$ for some small $\epsilon_1 > 0$ (depending on $P$) in the case $s=1$. Moreover, the error term in \eqref{eq:noyau_reduit_TPS} takes the simpler form $\O(\exp( - \lambda / C))$ in the case $s= 1$ (we recall that $\lambda$ denotes the large parameter $\jap{\va{\alpha}}/h$). Finally, notice that, applying Proposition \ref{prop:HSP_symbol}, we only need to study $\alpha = \beta$ real, since this implies an estimate for $\widehat{K}(\alpha,\beta)$ when $\alpha,\beta \in (T^* M)_{\epsilon_1}$ are close to each other.

Choose a reference point $z$ in the interior of $D$. We apply Proposition \ref{prop:HSP_symbol} with  ``$\Phi(x,\xi,y) = \Phi(\alpha_x,\beta_x,\alpha_\xi,\beta_\xi,y) = (\Phi_T(\alpha,y) + \Phi_S(y,\beta))/ \jap{\alpha}$'', ``$x_0 = (z,z)$'' and ``$F = \set{(\alpha_\xi,\alpha_\xi) : \alpha_\xi\in \R^n}$''. This will allow us to understand $\widehat{K}(\alpha,\beta)$ for $\alpha,\beta \in (T^* M)_{\epsilon_1}$ close to the diagonal and such that $\alpha_x$ and $\beta_x$ are close to $z$. Lemma \ref{lemma:TPS-close-diagonal} follows then by compactness of $M$ and a partition of unity argument (we do not claim holomorphicity for the symbol $e$). Let us check the hypotheses of Proposition \ref{prop:HSP_symbol}.

When $\alpha = \beta$ is real the phase $ y \mapsto (\Phi_T(\alpha,y) + \Phi_S(y,\beta)) / \jap{\alpha}$ has non-negative imaginary part for real $y$, because so do $\Phi_T(\alpha,y)$ and $\Phi_S(y,\alpha)$ by assumption. When $y$ is in $\partial D$, the imaginary part of the phase is uniformly positive since $z \notin \partial D$ and $\Phi_T$ is admissible. Finally, the phase has a critical point at $y = z$ which is non-degenerate (the Hessian of the phase has uniformly definite positive imaginary part). Hence, this critical point is isolated by the Implicit Function Theorem, and thus the only critical point of the phase in $D$ if $\eta$ is small enough. 

From Proposition \ref{prop:HSP_symbol}, we see that for $\alpha,\beta \in (T^* M)_{\epsilon_1}$ near the diagonal we have
\begin{equation}\label{eq:K_somme_partielle}
\begin{split}
\widehat{K}(\alpha,\beta) = e^{i \frac{\Phi_{TS}(\alpha,\beta)}{h}} \p{\sum_{0 \leq k \leq \lambda / C_0} h^k e_k(\alpha,\beta) + \O\p{\exp\p{ - \frac{\lambda}{C}}}},
\end{split}
\end{equation}
for some formal analytic symbol $\sum_{k \geq 0} h^k e_k$ and arbitrarily large $C_0$. We may replace the sum in \eqref{eq:K_somme_partielle} by any realization $e$ of the formal symbol $\sum_{k \geq 0} h^k e_k$. It follows from  Lemma \ref{lemma:fundamental-lemma} that $\Im \Phi_{TS}(\alpha,\beta)$ is non-negative when $\alpha$ and $\beta$ are real. Consequently, if $\alpha,\beta \in (T^* M)_{\epsilon_1}$, the factor $e^{i \frac{\Phi_{TS}(\alpha,\beta)}{h}}$ is an $\O(\exp( C \epsilon_1 \lambda))$. Hence, taking $\epsilon_1 > 0$ small enough, we may move the error term in \eqref{eq:K_somme_partielle} out of the parenthesis, and we find that $K_{TS}$ has the announced local structure near the diagonal. In fact, we get a much better estimate since we have an asymptotic expansion for $e$ that holds in the sense of realization of formal analytic symbols (since $\Phi_{TS}$ has non-negative imaginary part on the real).
\end{proof}

\subsection{Action of pseudors on the FBI side}\label{sec:action-pseudo-FBI}

Now that we have established a precise description of the kernel $K_{TPS}$, we will be able to understand the action of pseudors $P$ on the spaces $\mathcal{H}_\Lambda^k$, or equivalently the action of $T_\Lambda P S_\Lambda$ on $L^2_k(\Lambda)$. Recall that this is the exponentially weighted space $L^2 (\Lambda, \langle|\alpha|\rangle^{2k} e^{-2H(\alpha)/h} \mathrm{d}\alpha)$. The main result of this section is:
\begin{prop}\label{lemma:boundedness}
Let $P$ be a $\G^s$ pseudor of order $m$ on $M$. If $\tau_0$ is small enough and $\Lambda$ is a $\p{\tau_0,s}$-Gevrey adapted Lagrangian, then, for $h$ small enough and every $k \in \R$, the operator $T_\Lambda P S_\Lambda $ is bounded from the space $L^2_k\p{\Lambda}$ to the space $L^2_{k-m}\p{\Lambda}$, with norm uniformly bounded when $h$ tends to $0$.

In particular, $P$ is bounded from $\mathcal{H}_\Lambda^k$ to $\mathcal{H}_\Lambda^{k-m}$ and $\Pi_\Lambda = T_\Lambda S_\Lambda$ is bounded from $L^2_k\p{\Lambda}$ to itself.
\end{prop}

In fact, we can say a bit more about the operator $T_\Lambda P S_\Lambda$. We are indeed already able to prove a weak version of the multiplication formula Proposition \ref{propmultform} (which is itself a precise version of Theorem \ref{thm:deforming-Gs-pseudors}). We recall that the notion of $\G^s$ principal symbol has been defined in Proposition \ref{prop:existence_principal_symbol}, and that the associated almost analytic extensions are discussed in Remark \ref{remark:aae_for_symbol}.

\begin{prop}\label{lemma:weak-mult}
Under the assumption of Proposition \ref{lemma:boundedness}, if we denote by $p_{\Lambda}$ the restriction of an almost analytic extension of a $\G^s$ principal symbol $p$ for $P$ to $\Lambda$, then we have
\begin{equation*}
\begin{split}
T_\Lambda P S_\Lambda  & = p_{\Lambda} \Pi_\Lambda + \O_{L^2_k\p{\Lambda} \to L^2_{k-m + \frac{1}{2}}\p{\Lambda}}(h^\frac{1}{2}) \\
         & = \Pi_\Lambda p_{\Lambda} + \O_{L^2_k\p{\Lambda} \to L^2_{k - m + \frac{1}{2}}\p{\Lambda}}(h^\frac{1}{2}).
\end{split}
\end{equation*}
\end{prop}

We will also need the following corollary of Proposition \ref{lemma:boundedness}, which implies that the spaces $\mathcal{H}^k_\Lambda$ are legitimate functional spaces.
\begin{corollary}\label{corollary:densite}
Assume that $\tau_0$ is small enough and that $\Lambda$ is a $(\tau_0,1)$-adapted Lagrangian. Then, for $h$ small enough, there is $R > 0$ such that, for all $k \in \R$, the space $E^{1,R}(M)$ is dense in $\mathcal{H}_\Lambda^k$. Moreover, if $u \in L^2\p{M} \cap \mathcal{H}_\Lambda^k$, then there is a sequence $\p{u_n}_{n \in \N}$ in $E^{1,R}\p{M}$ such that $\p{u_n}_{n \in \N}$ converges to $u$ both in $L^2\p{M}$ and in $\mathcal{H}_\Lambda^k$.
\end{corollary}
Finally, we will prove that in the absence of deformation, the $\mathcal{H}_\Lambda^k$'s are just the usual Sobolev spaces.
\begin{corollary}\label{corollary:sobolev}
Let $h > 0$ be small enough. Then, for every $k \in \R$, the space $\mathcal{H}_{T^* M}^k$ is the usual semi-classical Sobolev space of order $k$ on $M$, with uniformly equivalent norms as $h$ tends to $0$.
\end{corollary}

For the proof of these results, it is convenient to introduce the ``reduced'' kernel
\begin{equation}\label{eq:def_reduced_kernel}
K_{TPS}^\Lambda(\alpha,\beta) := e^{\frac{H(\beta)-H(\alpha)}{h}}K_{TPS}(\alpha,\beta).
\end{equation}
This is only defined on $\Lambda\times \Lambda$, contrary to $K_{TPS}$. Now we have all the pieces to explain the reason why we have to work with I-Lagrangians, and why the weight $H$ was introduced in \S \ref{sec:symplectic-geometry}. 

According to Lemma \ref{lemma:TPS-far-diagonal}, the kernel $K_{TPS}$ is exponentially small far from the diagonal of $\Lambda\times \Lambda$, sufficiently so that it suffices to study its behaviour near the diagonal. There, we have an expansion in the form
\begin{equation}\label{eq:expression_KTPS}
K_{TPS}(\alpha,\beta) = e^{\frac{i}{h}\Phi_{TS}(\alpha,\beta)}e(\alpha,\beta) + \O\left(\exp\left( - \frac{\lambda^{\frac{1}{2s-1}} + \lambda^{\frac{1}{s}}d_{KN}(\alpha,\beta)}{C} \right) \right).
\end{equation}
This expansion suggests to deal with $TPS$ as an FIO with complex phase on $\Lambda$, in the spirit of \cite{Melin-Sjostrand-75}. For this machinery to work, we need to have a lower bound on $\Im \Phi_{TS}$. When $\alpha,\beta$ are real, $\Im\Phi_{TS}\geq 0$, and only vanishes when $\alpha = \beta$. This implies that $TPS$ is microsupported near the diagonal. However, the restriction of $ \Im \Phi_{TS}$ to $\Lambda \times \Lambda$ is not necessarily non-negative (a necessary assumption to apply the machinery from \cite{Melin-Sjostrand-75}). Thankfully, we can get around this using the action $H$ introduced in \S \ref{sec:symplectic-geometry}.
\begin{lemma}\label{lemma:PhiTS-Lambda}
There is $\eta_0>0 $ such that, if $\tau_0$ is small enough and $\Lambda$ is a $\p{\tau_0,1}$-adapted Lagrangian, then there is $C > 0$ such that, for every $\alpha,\beta \in \Lambda$ with $d_{KN}(\alpha,\beta) \leq \eta_0$, we have
\begin{equation}\label{eq:ineg-phase}
\begin{split}
\frac{1}{C}\jap{\va{\alpha}} d_{KN}(\alpha,\beta)^2 \leq \Im \Phi_{TS}(\alpha,\beta) + H(\alpha) - H(\beta) \leq C \jap{\va{\alpha}} d_{KN}(\alpha,\beta)^2.
\end{split}
\end{equation}
\end{lemma}

The idea behind this statement is that for $\alpha,\beta$ close enough, 
\[
\Phi_{TS} = \langle \alpha_x - \beta_x, \beta_\xi\rangle + \mathcal{O}(\langle|\alpha|\rangle d_{KN}(\alpha,\beta)^2),
\]
and
\[
\Im( \Phi_{TS} - \langle \alpha_x - \beta_x, \beta_\xi\rangle ) \geq \frac{\langle|\alpha|\rangle}{C} d_{KN}(\alpha,\beta)^2.
\]
In particular, the methods of \cite{Melin-Sjostrand-75} would prove that $\Pi_\Lambda$ is bounded on an exponentially weighted space $L^2(\Lambda, e^{-2H(\alpha)/h}d\alpha)$, if we had for $\epsilon>0$ small enough
\[
| H(\beta) - H(\alpha) - \Im( \langle \alpha_x - \beta_x, \beta_\xi\rangle )| \leq \epsilon \langle|\alpha|\rangle d_{KN}(\alpha,\beta)^2 .
\]
This implies exactly that $\mathrm{d}H = - \Im \theta_{|\Lambda}$, where $\theta$ is the canonical Liouville $1$-form. The existence of a solution implies that $\Lambda$ is an exact I-Lagrangian, and the desired bound holds if $\Lambda$ is sufficiently $\mathcal{C}^2$ close to the reals. This is the reason why we have to work with adapted Lagrangians, and introduce the weight $H$. 

\begin{proof}[Proof of Lemma \ref{lemma:PhiTS-Lambda}]
Define the function $A(\alpha,\beta) = \Im \Phi_{TS}(\alpha,\beta) + H(\alpha) - H(\beta)$ near the diagonal on $\Lambda \times \Lambda$. By definition, $A$ satisfies symbolic estimates (of order $1$). First, since $\Phi_{TS}(\alpha,\alpha)=0$, we have that $A(\alpha,\alpha)=0$. Then using \eqref{eq:propH} and the fact that
\begin{equation}\label{eq:diffphase1}
\begin{split}
\mathrm{d}_\alpha \Phi_{TS}(\alpha,\alpha) = \mathrm{d}_\alpha \Phi_T(\alpha,\alpha_x) = \theta_{|\Lambda}(\alpha)
\end{split}
\end{equation}
and
\begin{equation}\label{eq:diffphase2}
\begin{split}
\mathrm{d}_\beta \Phi_{TS}(\alpha,\alpha) = \mathrm{d}_\alpha \Phi_S(\alpha,\alpha_x) = -\theta_{|\Lambda}(\alpha),
\end{split}
\end{equation}
we find that the differential of $A$ vanishes on the diagonal. The expressions \eqref{eq:diffphase1} and \eqref{eq:diffphase2} follow from the definition of the phase $\Phi_{TS}$. Since $A$ vanishes at order $2$ on the diagonal of $\Lambda \times \Lambda$ and is a symbol of order $1$, the inequality of the right in \eqref{eq:ineg-phase} holds. The left inequality is equivalent to the existence of a constant $C > 0$ such that for every $\alpha \in \Lambda$ the Hessian $\nabla^2_{\beta,\beta} A_{|\beta=\alpha}$ is larger than $C^{-1} \jap{\va{\alpha}} g_{KN}$. Hence, the general case follows from the case $\Lambda = T^* M$ by a perturbation argument. Consequently, we only need to make the proof in the case $\Lambda = T^* M$.

Notice that in the case $\Lambda = T^* M$, we have $H = 0$, and hence $A = \Im \Phi_{TS}$. Recall that $\Phi_{TS}(\alpha,\beta)$ is the critical value of $y\mapsto \Psi(\beta,y) \coloneqq \Phi_T(\alpha,y) + \Phi_{S}(y,\beta)$. We fix the variable $\alpha$, and denote by $y(\beta)$ the critical point of this function. From the Implicit Function Theorem, we have 
\[
D_\beta y(\beta) = - (D_{y,y}^2 \Psi(\beta,y(\beta)))^{-1}D_{y,\beta}^2 \Psi(\beta,y(\beta)).
\]
Here, we work in coordinates and our convention regarding the notation $D$ is that $D_\beta y$ is the matrix of the differential of $\beta \mapsto y(\beta)$ (in the canonical basis) and that $D_{y,\beta}^2 \Psi$ is such that its columns have the size of $y$ and its lines the size of $\beta$.

Differentiating the expression $\Phi_{TS}(\alpha,\beta) = \Psi(\beta,y(\beta))$, it appears that we have $D_\beta \Phi_{TS}(\alpha,\beta) = D_\beta \Psi(\beta,y(\beta))$, and then that
\begin{equation*}
\begin{split}
D^2_{\beta,\beta}\Phi_{TS}&(\alpha,\beta) = D_{\beta,\beta}^2 \Psi (\beta, y(\beta)) + D_{\beta,y}^2 \Psi(\beta,y(\beta)) D_\beta y (\beta) \\
  &= D_{\beta,\beta}^2 \Psi (\beta, y(\beta)) -  D_{\beta,y}^2 \Psi(\beta,y(\beta)) (D_{y,y}^2 \Psi(\beta,y(\beta)))^{-1}D_{y,\beta}^2 \Psi(\beta,y(\beta)).
\end{split}
\end{equation*}
Then, differentiating $\Phi_S(\beta_x,\beta) = 0$ and $\mathrm{d}_y \Phi_S(\beta_x,\beta) = \beta_\xi$, we find that for $\beta=\alpha$, and $y=y(\alpha)=\alpha_x$, we have (using the decomposition $\beta = (\beta_x,\beta_\xi)$)
\[
D_{\beta,\beta}^2 \Psi(\alpha,\alpha_x) = \begin{pmatrix} D_{y,y}^2 \Phi_{S}(\alpha_x,\alpha) & - I \\ - I & 0 \end{pmatrix},\  D_{\beta,y}^2 \Psi(\alpha,\alpha_x) = \begin{pmatrix} - D_{y,y}^2 \Phi_S(\alpha_x,\alpha) \\ I\end{pmatrix},
\]
and $D_{y,y}^2\Psi(\alpha,\alpha_x) = 2 i \Im D_{y,y}^2 \Phi_S(\alpha_x, \alpha)$ (recall that $\Phi_T = - \overline{\Phi}_S$). Let us write the decomposition into real and imaginary part: $D_{y,y}^2 \Phi_S(\alpha_x,\alpha) = A+ iB$. Recall that $B$ is definite positive by assumption. Summing up, we find that
\[
\Im D_{\beta,\beta}^2 \Phi_{TS}(\alpha,\alpha) = \frac{1}{2} \begin{pmatrix}
AB^{-1} A + B & -A B^{-1}  \\ -B^{-1}A & B^{-1}
\end{pmatrix}
\]
We compute
\[
\begin{split}
\begin{pmatrix}
 B^{-1/2} u & B^{1/2} v 
\end{pmatrix} & \begin{pmatrix}
AB^{-1} A + B & -A B^{-1}  \\ -B^{-1}A & B^{-1}
\end{pmatrix}  \begin{pmatrix}
B^{-1/2} u \\ B^{1/2} v 
\end{pmatrix} \\
&=v^2 + u^2 + |B^{-1/2} A B^{-1/2} u|^2 - 2\langle B^{-1/2} A B^{-1/2} u, v\rangle\\
&= u^2 + ( B^{-1/2} A B^{-1/2} u - v)^2.
\end{split}
\]
This certainly is a positive definite quadratic form. Since $A+ iB$ is a symbol of order $1$, with $B$ being elliptic, this is uniform. We deduce
\[
\Im D_{\beta,\beta}^2 \Phi(\alpha,\alpha) \geq \frac{1}{C}\langle|\alpha|\rangle g_{KN}.
\]
\end{proof}

\begin{remark}
We proved Lemma \ref{lemma:PhiTS-Lambda} by computing directly the Hessian of the phase $\Phi_{TS}$. However, there is also a geometric interpretation of this lemma, in the spirit of the ``fundamental'' Lemma \ref{lemma:fundamental-lemma} and of the estimate Lemma \ref{lemma:estimee_positivite_phase} that was crucial in the proof of Lemma \ref{lemma:TPS-close-diagonal}. 

Let us explain a geometric view of Lemma \ref{lemma:PhiTS-Lambda} goes. We only need to consider the case $\Lambda = T^* M$. When $\alpha,\beta \in T^* M$, then the phase $y \mapsto \Phi_T(\alpha,y) + \Phi_S(y,\beta)$ has non-negative imaginary part for real $y$'s. The imaginary part of this phase is pluriharmonic and consequently the critical point of the phase is a saddle point for its imaginary part.  The critical value $\Phi_{TS}(\alpha,\beta)$ of the phase $y \mapsto \Phi_T(\alpha,y) + \Phi_S(y,\beta)$ is attained on the steepest descent contour. Since the imaginary part of the phase is non-negative for real $y$'s, we see that the reals are roughly aligned with the steepest phase descent. It implies that $\Im \Phi_{TS}(\alpha,\beta)$ is larger than the minimum of the imaginary part of $\Phi_T(\alpha,y) + \Phi_S(y,\beta)$ for real $y$. This proves that $\Im \Phi_{TS}(\alpha,\beta)$ is essentially larger than $\jap{\va{\alpha}}\va{\alpha_x - \beta_x}^2$. 

When $\va{\alpha_\xi - \beta_\xi} / \jap{\va{\alpha}}$ is much larger than $\va{\alpha_x - \beta_x}$, then the real part of $\nabla_{y}(\Phi_T(\alpha,y) + \Phi_S(y,\beta))$ must be large when $y$ is real. It implies that the critical point of the phase is away from the real (at a distance essentially larger than $\va{\alpha_\xi - \beta_\xi} / \jap{\va{\alpha}}$) and consequently, we can improve the previous estimate by the square of this distance multiplied by $\jap{\va{\alpha}}$ (the size of the Hessian).
\end{remark}

We now deduce Propositions \ref{lemma:boundedness} and \ref{lemma:weak-mult} from Lemma \ref{lemma:PhiTS-Lambda}.

\begin{proof}[Proof of Proposition \ref{lemma:boundedness}]
The proof is an application of the celebrated Schur's test \cite[Lemma 2.8.4]{Martinez-02-book-IntroductionSemiclassicalMicrolocal}. We start by observing that near the diagonal, using \eqref{eq:approximation_of_H},
\[
\frac{|H(\beta)-H(\alpha)|}{h}\leq C \tau_0 \left(\frac{\langle|\alpha|\rangle}{h}\right)^{1/s}d_{KN}(\alpha,\beta).
\]
Together with \eqref{eq:taille_de_H}, this proves that the factor $\exp((H(\beta)-H(\alpha)/h)$ in \eqref{eq:def_reduced_kernel} can be absorbed in the remainders from \eqref{eq:expression_KTPS}. It follows that, for $\tau_0$ small enough, the reduced kernel of $T_\Lambda P S_\Lambda$ satisfies (here $\alpha$ and $\beta$ are on $\Lambda$ and $e$ is a symbol supported near the diagonal of $\Lambda \times \Lambda$)
\begin{equation*}
\begin{split}
K_{TPS}^\Lambda\p{\alpha,\beta} = e^{i \frac{\Phi_{TS}(\alpha,\beta)+ i H(\alpha) - i H(\beta)}{h}} e(\alpha,\beta) + \O\p{\p{\frac{h}{\jap{\va{\alpha}}+ \jap{\va{\beta}}}}^{\infty}}.
\end{split}
\end{equation*}
In order to understand the action of $T_\Lambda P S_\Lambda$ from $L^2_k\p{\Lambda}$ to $L^2_{k-m}\p{\Lambda}$ we study the kernel
\begin{equation}\label{eq:reducedkernel}
\begin{split}
K(\alpha,\beta) & \coloneqq K_{TPS}^\Lambda\p{\alpha,\beta} \frac{\jap{\va{\alpha}}^{k-m}}{\jap{\va{\beta}}^k} \\ 
				& =  e^{ \frac{i \Phi_{TS}(\alpha,\beta) + H(\beta) - H(\alpha)}{h}} e(\alpha,\beta) \frac{\jap{\va{\alpha}}^{k-m}}{\jap{\va{\beta}}^k} + \O\p{\p{\frac{h}{\jap{\va{\alpha}} + \jap{\va{\beta}}}}^{\infty}}.
\end{split}
\end{equation}
In order to apply Schur's test, we need to find uniform bounds on
\begin{equation*}
\begin{split}
\int_{\Lambda} \va{K(\alpha,\beta)} \mathrm{d}\beta \textrm{ and } \int_{\Lambda} \va{K(\alpha,\beta)} \mathrm{d}\alpha.
\end{split}
\end{equation*}
We will only deal with the first integral, hence we fix $\alpha \in \Lambda$. We will also ignore the contribution of the error term in \eqref{eq:reducedkernel}, that is easily dealt with. Hence, we want a bound, uniform in $\alpha$ and $h$, on
\begin{equation}\label{eq:Schurtest}
\begin{split}
\int_{\Lambda} e^{-\frac{\Im \Phi_{TS}(\alpha,\beta) - H(\beta) + H(\alpha)}{h}} \va{e(\alpha,\beta)} \frac{\jap{\va{\alpha}}^{k-m}}{\jap{\va{\beta}}^k} \mathrm{d}\beta.
\end{split}
\end{equation}
Recalling that $e$ belongs to the symbol class $h^{-n} S_{KN}^m\p{\Lambda \times \Lambda}$, since $e$ is supported near the diagonal we have that, for some $C > 0$,
\begin{equation*}
\begin{split}
\va{e(\alpha,\beta)} \frac{\jap{\va{\alpha}}^{k-m}}{\jap{\va{\beta}}^k} \leq C h^{-n}.
\end{split}
\end{equation*}
Hence, we must estimate (the integral is over $\beta\in \Lambda$)
\begin{equation*}
\begin{split}
h^{-n} \int_{d_{KN}(\alpha,\beta) \leq \eta} e^{-\frac{\Im \Phi_{TS}(\alpha,\beta) - H(\beta) + H(\alpha)}{h}} \mathrm{d}\beta. 
\end{split}
\end{equation*}
Here, provided $\eta$ has been chosen small enough when applying Lemmas \ref{lemma:TPS-far-diagonal} and \ref{lemma:TPS-close-diagonal}, so that Lemma \ref{lemma:PhiTS-Lambda} applies, it follows from the estimate on the Jacobian in Lemma \ref{lemma:uniformity-lagrangians} that the integral \eqref{eq:Schurtest} is controlled by
\begin{equation*}
\begin{split}
h^{-n} \int_{\R^{2n}} \exp\p{- \frac{\jap{\va{\alpha}}x^2 + \jap{\va{\alpha}}^{-1} \xi^2}{h}} \mathrm{d}x \mathrm{d}\xi
\end{split}
\end{equation*}
and the result follows from the changes of variable $x' = h^{-\frac{1}{2}}\jap{\va{\alpha}}^{\frac{1}{2}} x$ and $ \xi' = h^{-\frac{1}{2}}\jap{\va{\alpha}}^{-\frac{1}{2}} \xi$.
\end{proof}

\begin{proof}[Proof of Proposition \ref{lemma:weak-mult}]
We will only prove the first estimate, the other one is obtained similarly (replacing an $\alpha$ by a $\beta$ in the computation below). We use the same notations as in the previous proof, and denote by $\tilde{e}$ the symbol $e$ that appears in the particular case $P = I$. Hence, we see that the reduced kernel of $T_\Lambda P S_\Lambda - p_{\Lambda} \Pi_\Lambda : L^2_k\p{\Lambda} \to L^2_{k-m+ \frac{1}{2}}\p{\Lambda}$ writes
\begin{equation*}
\begin{split}
\widetilde{K}(\alpha,\beta) & \coloneqq e^{\frac{H(\beta) - H(\alpha)}{h}}\p{K_{TPS}\p{\alpha,\beta} - p_{\Lambda}(\alpha)K_{TS}\p{\alpha,\beta}} \frac{\jap{\va{\alpha}}^{k-m+ \frac{1}{2}}}{\jap{\va{\beta}}^k} \\
     & =  e^{ \frac{i \Phi_{TS}(\alpha,\beta) + H(\beta) - H(\alpha)}{h}} \p{e(\alpha,\beta) - p_{\Lambda}(\alpha) \tilde{e}(\alpha,\beta)} \frac{\jap{\va{\alpha}}^{k-m+ \frac{1}{2}}}{\jap{\va{\beta}}^k} \\ & \qquad \qquad \qquad \qquad \qquad \qquad \qquad \qquad \qquad + \O_{\mathcal{C}^\infty}\p{\p{\frac{h}{\jap{\va{\alpha}} + \jap{\va{\beta}}}}^{\infty}}.
\end{split}
\end{equation*}
As in the previous proof, we want to apply Schur's test \cite[Lemma 2.8.4]{Martinez-02-book-IntroductionSemiclassicalMicrolocal}, and we leave as an exercise to the reader to deal with the error term. In order to apply Schur's test, we use the expansion \eqref{eq:premier_ordre_e} to find that
\begin{equation*}
\begin{split}
\va{e(\alpha,\beta) - p_{\Lambda}(\alpha) \tilde{e}(\alpha,\beta)} & \leq \va{e(\alpha,\alpha) - p_{\Lambda}(\alpha) \tilde{e}(\alpha,\alpha)} + C h^{-n} \jap{\va{\alpha}}^m d_{KN}(\alpha,\beta) \\
   & \leq C h^{-n +1} \jap{\va{\alpha}}^{m - 1} + C h^{-n} \jap{\va{\alpha}}^m d_{KN}(\alpha,\beta),
\end{split}
\end{equation*}
and then
\begin{equation*}
\begin{split}
\va{e(\alpha,\beta) - p_{\Lambda}(\alpha) \tilde{e}(\alpha,\beta)} \frac{\jap{\va{\alpha}}^{k-m + \frac{1}{2}}}{\jap{\va{\beta}}^k} \leq C h^{-n +1} \jap{\va{\alpha}}^{- \frac{1}{2}} + C h^{-n} \jap{\va{\alpha}}^{\frac{1}{2}} d_{KN}(\alpha,\beta).
\end{split}
\end{equation*}
We must consequently estimate
\begin{equation*}
\begin{split}
h^{-n+1} \int_{d_{KN}(\alpha,\beta) \leq \eta} \jap{\va{\alpha}}^{- \frac{1}{2}} e^{-\frac{\Im \Phi_{TS}(\alpha,\beta) - H(\beta) + H(\alpha)}{h}} \mathrm{d}\beta
\end{split}
\end{equation*}
and
\begin{equation*}
\begin{split}
h^{-n} \int_{d_{KN}(\alpha,\beta) \leq \eta} \jap{\va{\alpha}}^{\frac{1}{2}} d_{KN}(\alpha,\beta) e^{-\frac{\Im \Phi_{TS}(\alpha,\beta) - H(\beta) + H(\alpha)}{h}} \mathrm{d}\beta.
\end{split}
\end{equation*}
By the same argument as in the proof of Proposition \ref{lemma:boundedness}, we see that these integrals are controlled respectively by
\begin{equation*}
\begin{split}
h^{-n +1} \jap{\va{\alpha}}^{-\frac{1}{2}}\int_{\R^{2n}}\exp\p{ - \frac{\jap{\va{\alpha}}x^2 + \jap{\va{\alpha}}^{-1} \xi^2}{h}} \mathrm{d}x \mathrm{d}\xi
\end{split}
\end{equation*}
and
\begin{equation*}
\begin{split}
h^{-n} \jap{\va{\alpha}}^{\frac{1}{2}} \int_{\R^{2n}}\p{\va{x} + \frac{\va{\xi}}{\jap{\va{\alpha}}}} \exp\p{- \frac{\jap{\va{\alpha}}x^2 + \jap{\va{\alpha}}^{-1} \xi^2}{h}} \mathrm{d}x \mathrm{d}\xi,
\end{split}
\end{equation*}
and the result follows from the same change of variable as in the proof of Proposition \ref{lemma:boundedness}.
\end{proof}

As promised, we deduce Corollaries \ref{corollary:densite} and \ref{corollary:sobolev} from Proposition \ref{lemma:boundedness}.

\begin{proof}[Proof of Corollary \ref{corollary:densite}]
We have a regularizing procedure that is independent of the choice of $\Lambda$. We start by considering $u\in (E^{1,R})'$ with $R>1$ large enough so that Lemmas \ref{lemma:T_BMT} and \ref{lemma:S_BMT} apply. Then, we obtain that for $\epsilon>0$, 
\begin{equation}\label{eq:approximation_jointe}
\begin{split}
u_\epsilon = S_{T^* M} e^{- \epsilon \jap{\alpha}^2} T_{T^*M} u,
\end{split}
\end{equation}
is an element of $E^{1,R}$. Given $\Lambda$ a $(\tau_0,1)$-adapted Lagrangian with $\tau_0>0$ small enough, we can shift contours and obtain that 
\[
u_\epsilon = S_{\Lambda} e^{- \epsilon \jap{\alpha}^2} T_{\Lambda} u.
\]
Let us now assume that $u \in \mathcal{H}_\Lambda^k$ for some such Lagrangian $\Lambda$. Then, by dominated convergence, as $\epsilon$ tends to $0$, we see that $\exp(- \epsilon \jap{\alpha}^2) T_\Lambda u$ converges to $T_\Lambda u$ in $L^2_k\p{\Lambda}$. Then, since $\Pi_\Lambda$ is bounded on that space by Proposition \ref{lemma:boundedness}, we see that $T_\Lambda u_\epsilon = \Pi_\Lambda \exp(- \epsilon \jap{\alpha}^2) T_\Lambda u$ converges to $\Pi_\Lambda T_\Lambda u = T_\Lambda u$ also in that space, and thus in $\mathcal{H}^k_{\Lambda,\FBI}$ as $\epsilon\to 0$.

Recall that $T_{T^* M}$ is an isometry between $L^2\p{M}$ and $L^2\p{T^* M}$ (because $S_{T^* M}$ is the adjoint of $T_{T^* M}$ and $S_{T^* M} T_{T^* M} = I$). In particular, $L^2(M) = \mathcal{H}^0_{T^\ast M}$.
\end{proof}

\begin{proof}[Proof of Corollary \ref{corollary:sobolev}]
We denote by $H^k\p{M}$ the usual Sobolev space of order $k$ on $M$. Since $S = T^*$ is a left inverse for $T$, we see that $T$ is an isometry and consequently $L^2\p{M} = \mathcal{H}_{T^* M}^0$. From the fact that $S = T^*$, it follows that $\mathcal{H}_{T^* M}^{-k}$ is the dual of $\mathcal{H}_{T^*M}^k$ (this is in fact a general fact that is valid for any $\Lambda$, as we will see it in Lemma \ref{lemma:dual}). We can consequently assume that $k > 0$.

Let $u \in \mathcal{H}_\Lambda^k$. It follows from Remark \ref{remark:C_infini} that $u$ is a distribution. According to Proposition \ref{prop:quantization_Gs}, there is a $\G^1$ pseudor $A$ with principal symbol $\jap{\alpha}^k$. By Proposition \ref{lemma:boundedness}, we see that $A u \in \mathcal{H}_{T^* M}^0 = L^2\p{M}$ (since $T^* M$ is $(0,1)$-adapted). Since $k > 0$, we have also $u \in L^2\p{M}$. By elliptic regularity (in the $\mathcal{C}^\infty$ category), it follows that $u \in H^k\p{M}$ with norm controlled by
\begin{equation*}
\begin{split}
\n{u}_{L^2\p{M}} + \n{A u}_{L^2\p{M}} \leq C \n{u}_{\mathcal{H}_\Lambda^k}.
\end{split}
\end{equation*}

Reciprocally, assume that $u \in H^k\p{M}$. By Proposition \ref{prop:quantization_Gs}, we find a $\G^1$ pseudor $B$ whose principal symbol is $\jap{\alpha}^{-k}$. We can use Theorem \ref{thm:parametrix} to construct a parametrix $C$ for $B$, this is a $\G^1$ pseudor of order $k$ such that $B C = C B = I$ for $h$ small enough. Since $C$ has order $k$, the distribution $C u$ is an element of $L^2\p{M}$ and using Proposition \ref{lemma:boundedness}, we see that $u = BCu$ is an element of $\mathcal{H}_\Lambda^k$, with norm controlled by the $H^k\p{M}$ norm of $u$.
\end{proof}

\subsection{FBI transform and wave front set of ultradistributions}\label{sec:wfs}

From the results of \S \ref{sec:continuity-T-S} and \S \ref{subsec:inversion}, we see that the regularity of an ultradistribution may be characterized by the decay of its FBI transform. Indeed,it follows from Lemma \ref{lemma:T_non_stationary} and \ref{lemma:S_BMT} that

\begin{prop}\label{prop:regularite_via_T}
Let $s \geq 1$ and $u \in \U^s\p{M}$. Then $u$ belongs to $\G^s\p{M}$ if and only if $T_{T^* M} u$ belongs to $\GG^s\p{M}$.
\end{prop}

We leave as an exercise to the reader to give a local version of Proposition \ref{prop:regularite_via_T}. These considerations suggest to give the following definition of wave front set.

\begin{definition}\label{def:wavefrontset}
Let $s \geq 1$ and $u \in (E^{1,R_0})'$ for $R_0$ large enough. We define the $\G^s$ wave front set $\WF_{\G^s}\p{u} \subseteq T^* M \setminus \set{0}$ of $u$ in the following way: a point $\alpha \in T^* M \setminus \set{0}$ does not belong to $\WF_{\G^s}\p{u}$ if and only if there is a conical neighbourhood $\Gamma$ of $\alpha$ in $T^*M$ and $C > 0$ such that for some $h>0$,
\begin{equation}\label{eq:not_wave_front_set}
\begin{split}
T_{T^* M} u (\beta) \underset{\substack{\va{\beta} \to + \infty \\ \beta \in \Gamma}}{=} \O\p{\exp\p{- \frac{1}{C} \jap{\beta}^{\frac{1}{s}}}}.
\end{split}
\end{equation}
\end{definition}

One can define similarly the semi-classical wave front set $\WF_{\G^s,h}(u)$ of $u=u(h)$. 
\begin{remark}
In $\R^n$, the usual definition \cite{hormanderUniquenessTheoremsWave1971} of the wave front set $\WF_{\G^s}(u)$ may be rephrased using the flat transform $T_{\R^n}$, defined in \eqref{eq:def-flat-FBI-transform}. The condition \eqref{eq:not_wave_front_set} is then replaced by
\[
\left|T_{\R^n} u(\alpha_x ,\alpha_\xi;h)  \right| \leq C \exp\left( - \frac{1}{Ch^{\frac{1}{s}} }\right), \text{ as }h\to 0,
\]
where $\alpha$ remains in a compact set. With this definition, one has to take $h\to 0$ because the coercive term in the phase is just $|x-\alpha_x |^2$, instead of $\langle|\xi|\rangle |x-\alpha_x |^2$. However the two notions are the same.
\end{remark}

Lemmas \ref{lemma:TPS-far-diagonal} and \ref{lemma:TPS-close-diagonal} allow us to control the wave front set of the elements of $\mathcal{H}_\Lambda^0$ (Lemma \ref{lemma:wavefrontset_HLambda} below will be the main tool in the proof of Proposition \ref{prop:wfs_resonant_state}). Notice that Lemma \ref{lemma:wavefrontset_HLambda} has to be understood with $h$ fixed, so that $G_0$ is assumed to be elliptic as a classical symbol.

\begin{lemma}\label{lemma:wavefrontset_HLambda}
Let $s \geq 1$ and $G_0$ be a symbol on $(T^* M)_{\epsilon_0}$ (for some $\epsilon_0 > 0$) in the class $S_{KN}^{1/s}$. For $\tau>0$ define the symbol $G = h^{1 - 1/s}\tau G_0$ and then $\Lambda$ by \eqref{eq:adapted_Lagrangian}. Let $\alpha \in T^* M \setminus \set{0}$ be such that $G_0$ is negative and classically elliptic of order $1/s$ on a conical neighbourhood of $\alpha$. Then, if $\tau$ and $h$ are small enough, for any $u \in \mathcal{H}_\Lambda^0$ we have that $\alpha \notin \WF_{\G^{s}}\p{u}$.
\end{lemma}

Observe how this is coherent with the heuristics that the space $\mathcal{H}^0_\Lambda$ behaves formally as the space ``$\exp( \Op(G)/h) L^2(M)$''.

\begin{proof}
Notice first that for some $C > 0$, the Lagrangian $\Lambda$ is $\p{C \tau,1}$-adapted, so that the analysis above applies provided $\tau$ is small enough.

Let $\Gamma$ be a conical neighbourhood of $\alpha$ on which $G_0$ is negative and elliptic of order $\delta$. By assumption, there is a constant $C > 0$ such that for $\beta \in \Gamma$ large enough we have
\begin{equation*}
\begin{split}
G_0(\beta) \leq - \frac{\jap{\beta}^{\frac{1}{s}}}{C}.
\end{split}
\end{equation*} 
Choose some small $\epsilon > 0$ and consider $\beta \in \Gamma$ and $\gamma \in \Lambda$ such that the distance between $\beta$ and $\gamma$ for the Kohn--Nirenberg metric is less than $\epsilon$. Then, we have
\begin{align*}
\Im \Phi_{TS}(\beta,\gamma) &  = \Im \Phi_{TS}\p{\beta,e^{- H_{G}^{\omega_I}}(\gamma)} + \Im\p{\Phi_{TS}(\beta,\gamma) - \Phi_{TS}(\beta, e^{- H_{G}^{\omega_I}}(\gamma))}. \\
\intertext{By Lemma \ref{lemma:PhiTS-Lambda} (with $\Lambda=T^\ast M$), the first term in the right-hand side is non-negative, so this is}
      & \geq \mathrm{d}_\beta \Im \Phi_{TS}(\beta,\gamma) \cdot \p{ H_{G}^{\omega_I}(\gamma)} + \O\p{\tau^2 h^{2\p{1 - \frac{1}{s}}} \jap{\beta}^{1 + 2 (1/s - 1)}}, \\
      & \geq \mathrm{d}_\beta \Im \Phi_{TS}(\gamma,\gamma) \cdot\p{ H_{G}^{\omega_I}\p{\gamma}} + \O\p{(\epsilon + \tau h^{1 - \frac{1}{s}}) \tau h^{1 - \frac{1}{s}} \jap{\beta}^{\frac{1}{s}}}, \\
      & \geq - \Im \theta \p{H_{G}^{\omega_I}(\gamma)} +\O\p{(\epsilon + \tau h^{1 - \frac{1}{s}}) \tau h^{1 - \frac{1}{s}} \jap{\beta}^{\frac{1}{s}}}, \\
\intertext{Using \eqref{eq:approximation_of_H}, this is}
      & \geq H(\gamma) - \tau h^{1 - \frac{1}{s}} G_0(\gamma) + \O\p{(\epsilon + \tau h^{1 - \frac{1}{s}}) \tau h^{1 - \frac{1}{s}} \jap{\beta}^{\frac{1}{s}}}, \\
      & \geq H(\gamma) - \tau h^{1 - \frac{1}{s}} G_0(\beta) + \O\p{(\epsilon + \tau h^{1 - \frac{1}{s}}) \tau h^{1 - \frac{1}{s}} \jap{\beta}^{\frac{1}{s}}}, \\
      & \geq H(\gamma) + \frac{\tau h^{1 - \frac{1}{s}}}{C} \jap{\beta}^{\frac{1}{s}},
\end{align*}
provided $\tau,h$ and $\epsilon$ were small enough. Notice here that the quantification on $\epsilon$ does not depend on $\tau$ nor $h$. Now, if $u \in \mathcal{H}_\Lambda^0$ we write
\begin{equation*}
\begin{split}
T_{T^* M} u = T_{T^* M} S_\Lambda T_\Lambda u = T_{T^*M} S_\Lambda e^{\frac{H}{h}} \p{e^{- \frac{H}{h}} T_\Lambda u}.
\end{split}
\end{equation*}
Then, from Lemma \ref{lemma:TPS-far-diagonal}, we see that the kernel of $T_{T^*M} S_\Lambda e^{\frac{H}{h}}$ is exponentially decaying when $d_{KN}\p{\beta,\gamma} \geq \epsilon$ (provided that $\tau$ is small enough, that is why it was important that $\epsilon$ does not depend on $\tau$), and for $\beta \in \Gamma$ and $\gamma \in \Lambda$ such that $d_{KN}(\beta,\gamma) \leq \epsilon$ we have
\begin{equation*}
\begin{split}
T_{T^*M} S_\Lambda(\beta,\gamma) e^{\frac{H(\gamma)}{h}} = e^{\frac{i \Phi_{TS}(\beta,\gamma) + H(\gamma)}{h}} e(\beta,\gamma) + \O\p{\exp\p{- \frac{\jap{\va{\beta}} + \jap{\va{\gamma}} }{ C h} }},
\end{split}
\end{equation*}
for some symbol $e$. Then, the result follows from the fact that $e^{- \frac{H}{h}} T_\Lambda u$ is in $L^2$ and that, if $\beta$ is large enough and with the assumption above,
\begin{equation*}
\begin{split}
\va{e^{\frac{i \Phi_{TS}(\beta,\gamma) + H(\gamma)}{h}}} \leq \exp\p{ - \frac{\tau}{C h^{\frac{1}{s}}} \jap{\beta}^{\frac{1}{s}}}.
\end{split}
\end{equation*}
\end{proof}

The $\Gs$ wave front set interacts nicely with $\Gs$ pseudors. Indeed, we have
\begin{prop}\label{prop:Pseudor-are-pseudo-microlocal}
Let $A$ be a $\Gs$ pseudor, and $u \in \mathcal{U}^s(M)$. Then $\WF_{\Gs}(Au) \subset \WF_{\Gs}(u)$.
\end{prop}

\begin{proof}
We consider $\alpha \notin \WF_{\Gs}(u)$, and seek to prove that $\alpha\notin \WF_{\Gs}(Au)$. Using Lemma \ref{lemma:TPS-far-diagonal} and \ref{lemma:T_BMT}, we find for $\eta>0$ a constant $C>0$ such that for $\beta\in T^\ast M$,
\[
T (Au)(\beta) = \int_{d_{KN}(\gamma,\beta)\leq \eta} K_{TAS}(\beta,\gamma) Tu(\gamma)d\gamma + \O \left(\exp\left(  - \left( \frac{\langle\beta\rangle }{Ch}\right)^{\frac{1}{s}} \right) \right).
\]
From the assumption that $\alpha\notin \WF_{\Gs}(u)$, we deduce that for some $\eta>0$, the quantity $|Tu(\gamma)|$ is controlled by $\exp(- (\langle\gamma\rangle/ C h)^{1/s})$ in the conical neighbourhood of size $2\eta$ of $\alpha$. In particular for $\beta$ in the conical neighbourhood of size $\eta$ of $\alpha$, we find, using Lemma \ref{lemma:PhiTS-Lambda} to see that $\Im \Phi_{TS}(\beta,\gamma)$ is non-negative when $\beta$ and $\gamma$ are real, that 
\[
T (Au)(\beta) =  \O \left(\exp\left(  - \left( \frac{\langle\beta\rangle }{Ch}\right)^{\frac{1}{s}} \right) \right).
\]
\end{proof}

We can also use $I$-Lagrangian spaces to prove elliptic regularity for $\G^s$ pseudors.

\begin{prop}\label{prop:regularite_elliptique}
Let $m \in \R$. Let $A$ be a $\G^s$ pseudor on $M$, semi-classically elliptic of order $m$. Assume that $h$ is small enough. Then, if $u \in \U^s\p{M}$ is such that $Au \in \G^s\p{M}$, we have that $u \in \G^s\p{M}$.
\end{prop}

Proposition \ref{prop:regularite_elliptique} is a result of elliptic regularity that extends previous results \cite{Boutet-Kree-67,zanghiratiPseudodifferentialOperatorsInfinite1985,Rodino-93-book}. The main improvement here is that we use our class of $\G^s \G^s$ pseudors rather than the more common class of $\G^1 \G^s$ pseudors. 

\begin{remark}
As will be clear from the proof, this tool is quite flexible. For example, instead of ``semi-classically elliptic'', one could assume ``elliptic and $\Re A \geq 0$''. This statement can also be microlocalized. 
\end{remark}

The usual way to prove such a result is to construct a parametrix for $A$, i.e. to prove that the inverse for $A$ is itself a $\G^s$ pseudor (we know that this inverse exists when $h$ is small enough by the usual $\mathcal{C}^\infty$ pseudo-differential calculus). However, one can see that Proposition \ref{prop:regularite_elliptique} follows from the following result.

\begin{lemma}
Under the assumptions of Proposition \ref{prop:regularite_elliptique}, if $h$ is small enough then the inverse of $A$ is continuous from $\G^s\p{M}$ to itself and extends as a continuous operator from $\U^s\p{M}$ to itself.
\end{lemma} 

\begin{proof}
Since the formal adjoint of $A$ is also a semi-classically elliptic $\G^s$ pseudor of order $m$, according to Proposition \ref{prop:adjoint_pseudor}, then we only need to prove that $A^{-1}$ is bounded from $\G^s\p{M}$ to itself. To do so, we consider the Lagrangian $\Lambda$ defined by \eqref{eq:adapted_Lagrangian} with $G\p{\alpha} = - c \tau_0 h^{1 - \frac{1}{s}} \jap{\va{\alpha}}^{\frac{1}{s}}$, where $c > 0$ is small (in order to ensure that $\Lambda$ is a $\p{\tau_0,s}$-adapted Lagrangian). Provided that $\tau_0$ and $h$ are small enough, then we know by Proposition \ref{lemma:boundedness} that $A$ is bounded from $\mathcal{H}_\Lambda^k$ to $\mathcal{H}_\Lambda^{k-m}$ for every $k \in \R$. We will start by proving that for $\tau_0$ and $h$ small enough, then $A$ has a bounded inverse on $\mathcal{H}_\Lambda^0$. To do so, apply Proposition \ref{lemma:weak-mult} to write
\begin{equation}\label{eq:mw}
\begin{split}
T_\Lambda A S_\Lambda = a \Pi_\Lambda + \O_{L^2_{m - \frac{1}{2}}\p{\Lambda} \to L^2_{0}\p{\Lambda}}\p{h^{\frac{1}{2}}},
\end{split}
\end{equation}
where $a$ denotes an almost analytic extension for the symbol of $a$. It follows from the ellipticity of $A$ that there is $C > 0$ such that for every $\alpha \in \Lambda$ we have
\begin{equation*}
\begin{split}
\va{a(\alpha)} \geq \frac{\jap{\va{\alpha}}^{m}}{C}.
\end{split}
\end{equation*}
This is true by assumption for $\alpha \in T^* M$ and remains true after a small perturbation (we assume $c \ll 1$). Hence, the multiplication by $a^{-1}$ defines a bounded operator from $L^2_0\p{\Lambda}$ to $L^2_m\p{\Lambda}$. We define the operator
\begin{equation*}
\begin{split}
B = S_\Lambda a^{-1} T_\Lambda,
\end{split}
\end{equation*}
which is consequently bounded from $\mathcal{H}_\Lambda^0$ to $\mathcal{H}_\Lambda^m$. It follows from \eqref{eq:mw} that (the size of the remainder is measured both as an endomorphism of $\mathcal{H}_\Lambda^0$ and of $\mathcal{H}_\Lambda^m$)
\begin{equation*}
\begin{split}
B A = I + \O\p{h^{\frac{1}{2}}}.
\end{split}
\end{equation*}
Hence, we find by Von Neumann's argument that, for $h$ small enough, $A$ has an inverse $A^{-1}$ bounded from $\mathcal{H}_\Lambda^0$ to $\mathcal{H}_\Lambda^m$. Moreover, we see that this result is uniform in the parameter $\tau_0$ that appears in the definition of $G$, so that the inverse $A^{-1}$ exists for $h \leq h_0$ where $h_0$ does not depend on $\tau_0$. 

Now let $R > 0$. If $\tau_0$ is small enough, it follows by Corollary \ref{cor:completude} that $E^{s,R}\p{M} \subseteq \mathcal{H}_\Lambda^0$ with a continuous inclusion. However, the ellipticity of $G$ with Proposition \ref{prop:regularite_via_T} and Lemma \ref{lemma:wavefrontset_HLambda} (or rather a short inspection of their proofs) gives that there is $R_1 > 0$ such that $\mathcal{H}_\Lambda^0 \subseteq  E^{s,R_1}\p{M}$. Hence, if $h < h_0$ the inverse $A^{-1}$ of $A$ from $\mathcal{H}_\Lambda^0$ to $\mathcal{H}_\Lambda^m$ induces an inverse for $A$ from $E^{s,R}\p{M}$ to $E^{s,R_1}\p{M}$. Since both these spaces are included in $\mathcal{C}^\infty\p{M}$, this is nothing else that the inverse of $A$ on $\mathcal{C}^\infty\p{M}$ restricted to $E^{s,R}\p{M}$. Thus, the inverse of $A$ is bounded from $E^{s,R}\p{M}$ to $E^{s,R_1}\p{M}$. Since $R> 0$ is arbitrary, it follows that $A^{-1}$ is bounded from $\G^s\p{M}$ to itself.
\end{proof}

\section{Bergman projector and symbolic calculus}\label{sec:Bergman}

We use the same notations as in the previous section: $T$ is an analytic FBI transform given by Theorem \ref{thm:existence-good-transform}, and $S= T^*$ denotes its adjoint, $\Lambda$ is a $(\tau_0,1)$-adapted Lagrangian manifold with $\tau_0$ small enough. The associated operators $T_\Lambda$ and $S_\Lambda$ have been defined in \S \ref{sec:symplectic-geometry}, and we assume whenever we need that the implicit parameter $h > 0$ is small enough.

In order to understand the action of a $\G^s$ pseudor $P$ on the space $\mathcal{H}_\Lambda^0$ defined in \S \ref{sec:symplectic-geometry}, we only need to study the operator $T_\Lambda P S_\Lambda$ acting on $\mathcal{H}_{\Lambda,\FBI}^0$. The kernel of this operator was already described in Lemmas \ref{lemma:TPS-far-diagonal} and \ref{lemma:TPS-close-diagonal}. With Toeplitz calculus in mind, we would like to compare the operator $T_\Lambda P S_\Lambda$ with a multiplication operator. We gave a first result in this direction, Lemma \ref{lemma:weak-mult}. However, the error term in this result is quite big, and we intend to improve it, giving a multiplication formula valid at any order. 

It would be possible to find a symbol $p$ on $\Lambda$ such that $T_\Lambda P S_\Lambda \simeq \Pi_\Lambda p \Pi_\Lambda$, with a very small error. Here, we recall that $\Pi_\Lambda = T_\Lambda S_\Lambda$ is a bounded projector from $L^2_0\p{\Lambda}$ to $\mathcal{H}_{\Lambda,\FBI}^0$. This is not the most useful form for $T_\Lambda P S_\Lambda$ though. It is in fact much more interesting to replace $\Pi_\Lambda$ by the orthogonal projector $B_\Lambda$ on $\mathcal{H}_{\Lambda,\FBI}^0$ -- it is a legitimate operator since $\mathcal{H}_{\Lambda,\FBI}^0$ is closed. Indeed, if we can find a symbol $\sigma$ on $\Lambda$ such that

\begin{equation}\label{eq:Toeplitz}
\begin{split}
B_\Lambda T_\Lambda P S_\Lambda B_\Lambda \simeq B_\Lambda \sigma B_\Lambda
\end{split}
\end{equation}
up to a very small error, then we can use the symbol $\sigma$ to compute scalar products in $\mathcal{H}_\Lambda^0$:
\begin{equation*}
\begin{split}
\jap{P u ,u}_{\mathcal{H}_\Lambda^0} \simeq \jap{\sigma T_\Lambda u, T_\Lambda u}_{L_0^2\p{\Lambda}}.
\end{split}
\end{equation*}
We can then use the Hilbert structure of $\mathcal{H}_\Lambda^0$ to try to invert $P$ for instance. One may think of the difference between representations in the form $B_\Lambda \sigma B_\Lambda$ and $\Pi_\Lambda p \Pi_\Lambda$ as the difference between the Weyl quantization and the left quantization. Both of them are legitimate quantizations, however one of them is much more practical for the handling of adjoint operators. 

For this whole idea to work, we need to understand the structure of $B_\Lambda$. Guided by the fact that $B_{T^\ast M} = \Pi_{T^\ast M}$, it is reasonable to conjecture that the kernel of $B_\Lambda$ takes in general a form similar to that of $\Pi_\Lambda$. Recall that the latter, up to negligible remainders, takes the form
\begin{equation}\label{eq:forme_operateur}
\begin{split}
e^{i \frac{\Phi_{TS}(\alpha,\beta)}{h}} \sigma_{\Pi_\Lambda}(\alpha,\beta).
\end{split}
\end{equation}
where $\sigma_{\Pi_\Lambda}$ is a symbol supported near the diagonal of $\Lambda \times \Lambda$ and the phase $\Phi_{TS}$ satisfies the estimate Lemma \ref{lemma:PhiTS-Lambda}.

The method to prove it is indeed the case was first elaborated by Boutet de Monvel and Guillemin \cite{boutetdemonvelSpectralTheoryToeplitz1981}, and improved upon by \cite{Helffer-Sjostrand-86} and \cite{Sjostrand-96-convex-obstacle}; we will explain its details.

\begin{itemize}
	\item The first step is to observe that there exists local pseudo-differential operators $Z_j$, $j=1\dots n$, such that $Z_j \Pi_\Lambda = \O(h^\infty)$ as operators from $L^2_k$ to $L^2_{-k}$ for every $k$. The fact that the image of $\Pi_\Lambda$ is (almost) in the kernel of such operators stems from the decomposition $\Pi_\Lambda = T_\Lambda S_\Lambda$. There is no reason for this to be true for a general operator whose kernel is in the form $e^{i \Phi(\alpha,\beta)/h}a(\alpha,\beta)$.
	\item The second step is to consider operators $A$ whose kernel essentially take the form $e^{i \Phi(\alpha,\beta)/h}a(\alpha,\beta)$, and such that $\Pi_\Lambda A = A \Pi_\Lambda^\ast = A$ -- with some technical conditions. One can then use the $Z_j$'s to prove that the phase $\Phi$ is completely determined, and that the amplitude is completely determined by its values on the diagonal. 
	\item Finally, the observation that, given a symbol $\sigma$, the composition $\Pi_\Lambda \sigma \Pi_\Lambda^\ast$ is such an operator with oscillating kernel, leads to a parametrix construction showing that there exists $f$ a symbol of order $0$ such that $B_\Lambda \simeq \Pi_\Lambda f \Pi_\Lambda^\ast$ modulo a very small error.
\end{itemize}

A Toeplitz representation \eqref{eq:Toeplitz} of $\G^s$ pseudor will follow, that can be rephrased as a multiplication formula -- Proposition \ref{propmultform} -- that will be essential in the applications. Finally, we will also state basic results about a Toeplitz calculus that we will use in \S \ref{subsec:singuar_values} to deduce estimates on singular values of inclusions between the spaces $\mathcal{H}_\Lambda^k$ (it will also be very important in the applications).

Let us point out that, now that Lemmas \ref{lemma:TPS-far-diagonal} and \ref{lemma:TPS-close-diagonal} allowed us to replace $T^* M$ by $\Lambda$, we are working in the $\mathcal{C}^\infty$ category on $\Lambda$. We will use tools from \cite{Melin-Sjostrand-75} that require to choose an almost analytic neighbourhood $\widetilde{\Lambda}$, that is an embedding $\Lambda \subseteq \widetilde{\Lambda}$ of $\Lambda$ in a complex analytic manifold of dimension $2n$ that makes $\Lambda$ totally real in $\widetilde{\Lambda}$. We can take for instance $\widetilde{\Lambda} = T^* \widetilde{M}$, but any other choice is legitimate. We can also take $\mathcal{C}^\infty$ coordinates on $\Lambda$ and use $\C^n$ as an almost analytic neighbourhood. Notice also that when considering complex conjugates, we will always mean it in these coordinates. For instance if $f$ is an element of $T^*_\alpha \Lambda \otimes \C \simeq T^*_\alpha \widetilde{\Lambda}$ for some $\alpha \in \Lambda$, then $\bar{f}$ is defined using the tensor product structure. This amounts to say that $\bar{f}$ is the $\C$-linear form on $T_\alpha \Lambda \otimes \C \simeq T_\alpha \widetilde{\Lambda}$ whose restriction to $T_\alpha \Lambda$ is obtained from $f$ by composition with the usual complex conjugacy.

Finally, let us mention that in this section we will identify an operator with its kernel, writing for instance $\Pi_\Lambda(\alpha,\beta)$ for the kernel of $\Pi_\Lambda$.

\subsection{Operators on the image of the transform}\label{sec:FIO}

\subsubsection{Annihilating pseudors}\label{sub:annihilating}

In this section, we will build the $Z_j$'s evoked a few lines above. The existence of such operators vanishing on the image of $T$ should be seen as a generalization of the following relations satisfied by the flat transform defined in \eqref{eq:def-flat-FBI-transform}:
\[
\underset{=Z_{j,\R^n}}{\underbrace{\left( \frac{\partial}{\partial x_j} - i \frac{\partial}{\partial \xi_j}  - \frac{i}{h}\xi_j\right)}} T_{\R^n} = 0.
\]
The collection of these equations forms a twisted $\overline{\partial}$ equation. In our non-flat case, the $Z_j$'s will have to be pseudo-differential, but the fundamental idea is similar. To be prudent, we define beforehand what we mean by a pseudor on $\Lambda$:
\begin{definition}
Let $\Lambda$ be a $(\tau_0,1)$-adapted Lagrangian. Consider $Q$ an operator on $\Lambda$ (maybe depending on $h$ as usual) whose kernel is uniformly properly supported. Assume additionally that in local coordinates near a point $\alpha \in \Lambda$ where the Kohn--Nirenberg metric is uniformly close to the standard metric, the kernel of $Q$ is the kernel of a semi-classical pseudo-differential operator with Planck constant $\hbar = h/\langle|\alpha|\rangle$, whose symbol is uniformly compactly supported in the impulsion variable. Then we say that $Q$ is an pseudor on $\Lambda$.
\end{definition}

Using the map $T^\ast M \to \Lambda$ given by $\exp H^{\omega_I}_G$, and local coordinates $(u,\eta)$ on $T^\ast M$, we have local coordinates $\alpha = \exp H^{\omega_I}_G(u,\eta)$ on $\Lambda$. Taking into account the relevant scaling, $Q$ is a pseudor on $\Lambda$ if its kernel takes in these coordinates the form (near the diagonal)
\[
\frac{1}{(2\pi h)^{2n}}\int e^{\frac{i}{h}(\langle u-u',u^\ast\rangle + \langle \eta -\eta', \eta^\ast \rangle)} q(u,u',\eta,\eta',u^\ast, \eta^\ast) \mathrm{d}u^\ast \mathrm{d}\eta^\ast,
\]
where the symbol $q$ is supported for $|u^\ast| + |\eta-\eta'| \leq C \langle \eta\rangle$, $|\eta^\ast| + |u-u'|\leq C$ and satisfies for $k,\ell \in \N^{2n}$ and $k',\ell' \in \N^n$
\begin{equation}\label{eq:symbole_bizarre}
\left|(\partial_{u,u'})^k (\partial_{\eta,\eta'})^{\ell} (\partial_{u^\ast})^{k'}  (\partial_{\eta^\ast})^{\ell'} q(u,u',\eta,\eta',u^*,\eta^*)\right| \leq C_{k,\ell,k',\ell'} \langle\eta\rangle^{ m -|k'| - |\ell|}.
\end{equation}
If a symbol $q$ satisfies these conditions, we write $q \in S^{m}_{c}(T^\ast \Lambda)$, and say that $Q$ is an pseudor of order $m$.

We pick a cutoff $\chi$ supported near the diagonal of $\Lambda\times\Lambda$. In a chart domain, given a symbol $p\in S^m_{c}(T^\ast \Lambda)$ depending only on $u,\eta,u^\ast,\eta^\ast$, we let
\[
\begin{split}
\Op(p)f(u,\eta) = \frac{1}{(2\pi h)^{2n}} \int e^{\frac{i}{h}(\langle u-u',u^\ast\rangle + \langle \eta -\eta', \eta^\ast \rangle)}& p(u,\eta,u^\ast, \eta^\ast)f(u',\eta') \\
		& \times \chi(u,u',\eta,\eta') \mathrm{d}u^\ast \mathrm{d}\eta^\ast \mathrm{d}u' \mathrm{d}\eta'.
\end{split}
\]
This defines an pseudor of order $m$ (at least in a smaller chart domain). It follows from the results of \cite{Melin-Sjostrand-75} that if a function $f$ on $\Lambda$ has an expansion in the large chart domain in the form $f = e^{i\psi/h}a + \O(h/\langle\eta\rangle)^\infty$, with $\psi$ and $a$ symbols of respective order $1$ and $\tilde{m}$ in the Kohn--Nirenberg class on $\Lambda$ such that $\Im \psi\geq 0$, and $\Im\psi(\alpha_0)=0$, then near $\alpha_0$, in the small chart domain, modulo $\O(h/\langle\eta\rangle)^\infty$ (all these error terms are in $\mathcal{C}^\infty$),
\begin{equation}\label{eq:expansion-action-lagrangian}
\Op(p)f \sim e^{\frac{i}{h}\psi} \sum_{k \geq 0} \frac{h^k}{k!}b_k.
\end{equation}
By the expansion \eqref{eq:expansion-action-lagrangian}, we mean that for every $N \in \N$ we have
\begin{equation*}
\begin{split}
\Op(p)f = e^{\frac{i}{h}\psi} \sum_{k = 0}^N \frac{h^k}{k!}b_k \textup{ mod } h^{N+1} S^{m + \tilde{m} - N - 1}_{KN}\p{\Lambda}.
\end{split}
\end{equation*}
The coefficients in this expansion are given by the expression 
\begin{equation}\label{eq:explicit-formula}
b_k =\left(\frac{\nabla_{u',\eta'}\cdot \nabla_{u^\ast,\eta^\ast}}{i}\right)^k\left[ p(u,\eta ; (u^\ast,\eta^\ast + \rho(u,u';\eta,\eta') ) a(u',\eta') \right]_{| u'=u, \eta'=\eta, u^\ast=\eta^\ast =0},
\end{equation}
where $\rho$ is such that $\psi(u',\eta')-\psi(u,\eta) = \langle (u',\eta') - (u,\eta) , \rho(u,u';\eta,\eta')\rangle$. Here, we use the coordinates on $\Lambda$ described above and we have identified $p$ with one of its almost analytic extensions in the $u^\ast,\eta^\ast$ variable. In particular, the leading term in the expansion \eqref{eq:explicit-formula} is given by
\[
b_0 = p( u,\eta, \mathrm{d}_{u,\eta} \psi(u,\eta)) a(u,\eta).
\]

We will be chiefly interested in the composition of a pseudor $Q = \Op(p)$ with $\Pi_\Lambda$. However, since we are interested in the action of $\Pi_\Lambda$ on $L^2(e^{-2H/h}\mathrm{d}\alpha)$, it is natural to consider the kernel of the composition $Q e^{-H/h}\Pi_\Lambda e^{H/h}$ rather than $Q \Pi_\Lambda$. According to our discussion, its kernel will have a semi-classical expansion, with a principal symbol in the form
\[
p(u,\eta, \mathrm{d}_{u,\eta} \Phi_{TS}^\circ),
\]
where $\Phi_{TS}^\circ$ is the phase of the reduced kernel of $\Pi_\Lambda$:
\begin{equation*}
\begin{split}
\Phi_{TS}^{\circ}\p{\alpha,\beta} = i H(\alpha) + \Phi_{TS}\p{\alpha,\beta} - i H(\beta).
\end{split}
\end{equation*}
This suggest considering the set (for some small $\epsilon > 0$)
\begin{equation}\label{eqdefjlambda}
\J_\Lambda = \set{\p{\alpha, \mathrm{d}_\alpha \Phi_T(\alpha,y) - i \Im \theta(\alpha)} : \alpha \in \Lambda, y \in \widetilde{M}, d_{\widetilde{M}}(\alpha_x,y) < \epsilon}.
\end{equation}
This is a priori a submanifold of $T^\ast \Lambda \otimes \C$ (the complexification of the cotangent space). Indeed, Definition \ref{def:nonstandardphase} of an admissible phase implies that if $\tau_0$ is small enough then $y \mapsto \mathrm{d}_\alpha \Phi_T(\alpha,y)$ is a holomorphic immersion near $\alpha_x$. Recall that for $\beta$ near $\alpha$ we have $\mathrm{d}_\alpha \Phi_{TS}(\alpha,\beta) = \mathrm{d}_\alpha \Phi_T(\alpha,y_c(\alpha,\beta))$ where $y_c(\alpha,\beta)$ denotes the critical point of $y \mapsto \Phi_T(\alpha,y) + \Phi_S(y,\beta)$.  It follows that for $\beta$ near $\alpha$ in $\Lambda$, the``reduced'' phase of $\Pi_\Lambda$ satisfies
\begin{equation*}
\begin{split}
\mathrm{d}_\alpha \Phi_{TS}^\circ \p{\alpha,\beta} \in \J_\Lambda.
\end{split}
\end{equation*}

We will show that
\begin{prop}\label{prop:existence-Z_j}
Let $\Lambda$ be a $(\tau_0,1)$-adapted Lagrangian with $\tau_0$ small enough. Around any fiber $T_x^\ast M$, there exist pseudors $Z_j$, $j=1\dots n$, of order $0$, such that 
\begin{enumerate}[label=(\roman*)]
	\item $Z_j = \Op(\zeta_j)$, with $\zeta_j \sim \sum h^k \zeta_{j,k}$ a symbol of order $0$;
	\item the kernel of $Z_j e^{-H/h} \Pi_\Lambda e^{H/h}$ is $\O(h/\langle\eta\rangle)^\infty$ in $\mathcal{C}^\infty$;
	\item each $\zeta_{j,k}$ is holomorphic in $u^\ast,\eta^\ast$ near $\J_\Lambda\cap T^\ast \Lambda$ and $\zeta_{j,0}$ vanishes on $\J_\Lambda$;
	\item in a uniform neighbourhood of $\J_\Lambda \cap T^\ast \Lambda$, the $\mathrm{d} \zeta_{j,0}$ form a uniformly free family. 
\end{enumerate}

\end{prop}

The asymptotic expansion in (i) is thought in the following sense: for every $N \in \N$, we have
\begin{equation*}
\begin{split}
\zeta_j = \sum_{k = 0}^N h^k \zeta_{j,k} \textup{ mod } h^{N+1} S_c^{- N - 1}\p{T^* \Lambda}.
\end{split}
\end{equation*}
When we write ``around any fiber $T_x^\ast M$'', we mean that the points (i)-(iv) only hold near points $\alpha \in \Lambda$ such that in the coordinates $(u,\eta)$ described above, $u$ is close to $x$. We can then cover $\Lambda$ by a finite number of domains on which the pseudors from Proposition \ref{prop:existence-Z_j} are available.

The difficulty in the proof of Proposition \ref{prop:existence-Z_j} is that we have to use the decomposition $\Pi_\Lambda = T_\Lambda S_\Lambda$, but we are not able to obtain an expansion for $\Op(p)e^{-H/h}T$, because the imaginary part of the phase $\Phi_T + iH(\alpha)$ is not positive. We propose here a solution inspired by the proof of \cite[Proposition 6.7]{Helffer-Sjostrand-86}. We will approximate the $Z_j$'s by differential operators of increasingly large order. It is morally very close to the argument based on formal applications of the stationary phase method from \cite{Helffer-Sjostrand-86}, but we privileged this method because it does not rely on the results from \cite{Sjostrand-82-singularite-analytique-microlocale}, with which the reader may not be familiar.

\begin{proof}[Proof of Proposition \ref{prop:existence-Z_j}]
We will first build the $\zeta_{j,k}$'s as ``formal'' solutions, and then check that they indeed solve the problem. As we said before, we cannot prove an expansion for $\Op(p)e^{-H/h} T$. However, we can still compute the terms that appear in the right hand side of \eqref{eq:expansion-action-lagrangian}. As a formal series in powers of $h$ whose coefficients are functions, it is a well defined object, so that for $p\in S^m_{c}(T^\ast \Lambda)$, we define
\begin{equation}\label{eq:def-formal-pseudor}
\left[e^{\frac{ -i \Phi_T(\cdot,y) + H  }{h} }\Op(p)\right]^{\mathrm{formal}} ( e^{-H/h}T) :=  \sum_{k\geq 0} \frac{h^k}{k!} b_k, 
\end{equation}
with, as in \eqref{eq:explicit-formula}, $b_k$ defined by the expression 
\[
\left(\frac{\nabla_{u',\eta'}\cdot \nabla_{u^\ast,\eta^\ast}}{i}\right)^k\left[ p(u ,\eta ; (u^\ast,\eta^\ast) + \rho(u,u',\eta,\eta') ) a_T(u',\eta',y) \right]_{| u'=u,\eta=\eta', u^\ast = \eta^\ast = 0},
\]
and $\rho$ defined by
\[
\Phi_T(u',\eta',y)-\Phi_T(u,\eta,y) + i( H(u',\eta')-H(u,\eta)) = \langle (u',\eta')-(u,\eta), \rho(u,u',\eta,\eta')\rangle.
\]
Here as usual, we identified $p$ with one of its almost analytic extension \emph{in the coordinates $(u,\eta)$}. Taking $p=\zeta_j$ itself a formal sum of symbols $\zeta_j= \sum h^k \zeta_{j,k}$, we try to solve 
\[
\left[e^{\frac{ -i \Phi_T(\cdot,y) + H  }{h} }\Op(\zeta_j)\right]^{\mathrm{formal}} ( e^{-H/h}T)(\alpha,y) = 0.
\]
If this were an actual composition, since $\Lambda$ and the symbol $a$ of $T$ themselves may depend on $h$, there would be many ways to expand this sum in powers of $h$. However, we are here using non-ambiguously the expansion given by \eqref{eq:expansion-action-lagrangian} and expand $\zeta_j= \sum h^k \zeta_{j,k}$ in the resulting formula, leading to equations of the form
\begin{equation}\label{eq:principal-symbol-condition}
\zeta_{j,0} (u,\eta, \mathrm{d}_{u,\eta} (\Phi_T + i  H) ) = 0,
\end{equation}
and 
\begin{equation}\label{eq:subprincipal-symbol-condition}
\zeta_{j,k} (u,\eta, \mathrm{d}_{u,\eta}(\Phi_T + i  H) ) =  Q_k( \zeta_{j,0}, \zeta_{j,1}, \dots, \zeta_{j,k-1})(u,\eta,y),
\end{equation}
$Q_k$ being some differential operator whose coefficients depend on $a$, $\Phi_T$ and $\Lambda$. However this dependence is symbolic, so its comes with uniform estimates. Since the map
\[
(\alpha,y) \mapsto (\alpha, \mathrm{d}_\alpha (\Phi_T(\alpha,y) + i  H(\alpha)) )
\]
is a local diffeomorphism onto $\J_\Lambda$, there is no obstruction to the equations above having a solution. Quite on the contrary, there are many solutions, because the symbols $\zeta_{j,k}$ are recursively prescribed only \emph{on} $\J_\Lambda$.

We start with $\zeta_{j,0}$. The proof of the fact that $\J_\Lambda$ is an analytic manifold of $T^\ast \Lambda \otimes \C$, of codimension $n$ uses the local inversion theorem, and thus comes with uniform estimates. In particular, we can choose the $\zeta_{j,0}$ so that
\begin{itemize}
	\item each $\zeta_{j,0}$ is a symbol of order $0$, holomorphic in the $u^\ast,\eta^\ast$ variable, near $\J_\Lambda\cap T^\ast \Lambda$ -- the real points of $\J_\Lambda$;
	\item each $\zeta_{j,0}$ vanishes on $\J_\Lambda$;
	\item near $\J_\Lambda\cap T^\ast \Lambda$, the $\mathrm{d} \zeta_{j,0}$ form a uniformly free family;
	\item the $\zeta_{j,0}$'s are supported for $|u^\ast|\leq C\langle\eta\rangle$, $|\eta^\ast|\leq C$.
\end{itemize}
Using the symplectic structure of $\J_\Lambda$, we could additionally assume that $\{ \zeta_{j,0}, \zeta_{\ell,0} \} =0$, but this will not be necessary for us.

Let us now explain how to construct the higher order symbols. Since the differential of
\[
y \mapsto \mathrm{d}_{\alpha_\xi} \Phi_T(\alpha,y)
\]
at $y = \alpha_x$ is the identity (under natural identifications) by assumption, and since $\exp H^{\omega_I}_G$ is close to the identity, $\J_\Lambda$ is uniformly transverse to the foliation tangent to $\partial/\partial u^\ast$. It follows that if a function has its values determined on $\J_\Lambda$ and that function does not depend on $u^\ast$, that function is completely determined in a neighbourhood of $\J_\Lambda$. In particular, we can solve the equations \eqref{eq:subprincipal-symbol-condition} near the real points of $\J_\Lambda$, with symbols $\zeta_{j,k}$, $k\geq 1$, holomorphic in $\eta^\ast$ in a neighbourhood of $\J_\Lambda\cap T^\ast \Lambda$, not depending on $u^\ast$.

Now, we can build $\zeta_j$ some $C^\infty$ symbols of order $0$ so that
\[
\zeta_j \sim \sum h^k \zeta_{j,k}.
\]
We have to prove that for $\beta\in \Lambda$ in the chosen chart, and $\alpha$ close to $\beta$,
\begin{equation}\label{eq:zjpi}
\Op(\zeta_j)e^{-\frac{H}{h}} \Pi_\Lambda e^{\frac{H}{h}} (\alpha,\beta) = \mathcal{O}_{\mathcal{C}^\infty}( (h/\langle|\beta|\rangle)^\infty).
\end{equation}
(the proposition also requires an estimate far from the diagonal, but that estimate is trivial at this point). In order to prove \eqref{eq:zjpi}, we introduce the following truncations of $\zeta_j$:
\begin{equation*}
\begin{split}
\zeta_j^{\leq N} = \sum_{k = 0}^{N} \p{\zeta_{j,k}}^{(2N)},
\end{split}
\end{equation*}
where, $(\zeta_{j,k})^{(2N)}$ is obtained from $\zeta_{j,k}$ by taking its $2N$ Taylor expansion in $u^*$ and $\eta^*$ at $\Re \theta(u,\eta)$. We also write $\zeta_j^{> N} = \zeta_j - \zeta_j^{\leq N}$, and split $\Op(\zeta_j)$ into $\Op(\zeta_j^{\leq N})$ and $\Op(\zeta_j^{> N})$. 

We deal first with $\Op(\zeta_j^{> N})$, to do so let us consider a symbol $p$ on $T^* \Lambda$ (as defined above) such that
\begin{equation}\label{eq:condition_taylor_nul}
\begin{split}
p(u,\eta,u',\eta',u^*,\eta^*) = \O\p{\jap{\eta}^{-2N} \va{(u^*,\eta^*) - \Re \theta(u,\eta)}^{2N}} + \O\p{h^N \jap{\eta}^{-N}}.
\end{split}
\end{equation}
Here, the $\O$'s must be understood in a space of functions satisfying \eqref{eq:symbole_bizarre} and in the norm $\va{(u^*,\eta^*) - \Re \theta(u,\eta)}$ the contribution of the variables corresponding to $u^*$ is rescaled by a factor $\jap{\eta}^{-1}$. Applying the method of \cite{Melin-Sjostrand-75} again, we find that
\begin{equation*}
\begin{split}
\Op(p)e^{-\frac{H}{h}} \Pi_\Lambda e^{\frac{H}{h}}(\alpha,\beta) = e^{\frac{i}{h}\Phi_{TS}^\circ} q(\alpha,\beta) + \O( h/\langle|\beta|\rangle )^\infty,
\end{split}
\end{equation*}
and we have an asymptotic expansion for the symbol $q$
\begin{equation*}
\begin{split}
q \sim \sum_{k \geq 0} h^k q_k.
\end{split}
\end{equation*}
The $q_k$'s write in the coordinates $(u,\eta)$:
\begin{equation}\label{eq:def_qk}
\begin{split}
q_k(u,\eta,u',\eta') =  \left(\frac{\nabla_{u'',\eta''}\cdot \nabla_{u^\ast,\eta^\ast}}{i}\right)^k &\Big[ p\big(u,\eta ; (u^\ast,\eta^\ast)  + \rho(u,u',u'',\eta,\eta',\eta'') \big)\\
				&  \qquad \qquad \times \sigma_{\Pi_\Lambda}(u'',\eta'',u',\eta') \Big]|_{\substack{u''=u,\eta''=\eta \\ u^\ast = \eta^\ast = 0}},
\end{split}
\end{equation}
where $\rho$ satisfies
\[
\Phi_{TS}^\circ( u'',\eta'',u',\eta')-\Phi_{TS}^\circ(u,\eta,u',\eta') = \langle  (u'',\eta'')- (u,\eta), \rho(u,u',u'',\eta,\eta',\eta'') \rangle,
\]
and $\sigma_{\Pi_\Lambda}$ is from \eqref{eq:forme_operateur}. Here, we recall that the phase $\Phi_{TS}$ satisfies the estimate Lemma \ref{lemma:PhiTS-Lambda}, so that the application of the stationary phase method is legitimate. Consequently, we must have
\begin{equation*}
\begin{split}
\rho(u,u',u'' &,\eta,\eta',\eta'') = \mathrm{d}_{u,\eta}\Phi_{TS}^\circ(u,\eta,u',\eta') + \O( \va{u'' - u} + \jap{\eta}^{-1}\va{\eta'' - \eta}) \\
    & = \Re \theta(u,\eta) + \O(\va{u'' - u} + \va{ u - u'} + \jap{\eta}^{-1}\va{\eta'' - \eta} + \jap{\eta}^{-1} \va{\eta - \eta'}).
\end{split}
\end{equation*}
It follows then from our assumption \eqref{eq:condition_taylor_nul} that in \eqref{eq:def_qk} we differentiate at most $2k$ times a term of the form
\begin{equation*}
\begin{split}
& \O\Big(\langle\eta\rangle^{-2N}\big(\va{u'' - u} + \va{ u - u'} + \jap{\eta}^{-1}\va{\eta'' - \eta} \\ & \qquad \qquad \qquad \qquad + \jap{\eta}^{-1} \va{\eta - \eta'} + \jap{\eta}^{-1} \va{u^*} + \va{\eta^*} \big)^{2N}\Big) + \O\Big( \langle\eta\rangle^{-N}h^N\Big).
\end{split}
\end{equation*}
Hence, for $k \leq N$, $q_k(u,n,u',\eta')$ is a 
\begin{equation*}
\begin{split}
\O\Big(\langle\eta\rangle^{-2N}\big(\va{u - u'}^{2(N - k)} + \jap{\va{\eta}}^{2(k - N)} \va{\eta - \eta'}^{2(N - k)} \big) + \langle\eta\rangle^{-N} h^N\Big).
\end{split}
\end{equation*}
Gathering these estimates, we deduce that 
\begin{equation*}
\begin{split}
q(\alpha,\beta) =  \O\left( \left( \langle|\beta|\rangle^{-2} d_{KN}(\alpha,\beta)^2 + \langle|\beta|\rangle^{-1} h \right)^N \right).
\end{split}
\end{equation*}
and it follows from the fact that $\Im \Phi_{TS}^\circ \geq C^{-1} \langle|\beta|\rangle d(\alpha,\beta)^2$ that
\begin{equation*}
\begin{split}
\Op(p) e^{-\frac{H}{h}} \Pi_\Lambda e^{\frac{H}{h}}(\alpha,\beta) = \O(h^{N} \jap{\va{\beta}}^{-N}).
\end{split}
\end{equation*}
Now, we would like to take $p = \zeta_{j}^{>N}$ and apply this estimate. We cannot work directly like that because $\zeta_j^{> N}$ does not satisfy the support condition that we assumed for $p$. However, introducing a cutoff function in impulsion, we can write $\zeta_j^{> N} = p + r$ where $p$ satisfies \eqref{eq:condition_taylor_nul} and satisfies the support condition that we required for symbol on $T^* \Lambda$. The symbol $r$ is supported away from the graph of $\Re \theta$, so that non-stationary phase method proves that
\begin{equation*}
\begin{split}
\Op(r) e^{-\frac{H}{h}} \Pi_\Lambda e^{\frac{H}{h}}(\alpha,\beta) = \O(h^{\infty} \jap{\va{\beta}}^{- \infty}).
\end{split}
\end{equation*}
We have thus proved that
\begin{equation*}
\begin{split}
\Op(\zeta_j^{> N}) e^{-\frac{H}{h}} \Pi_\Lambda e^{\frac{H}{h}}(\alpha,\beta) = \O(h^{N} \jap{\va{\beta}}^{-N}).
\end{split}
\end{equation*}

We consider now the other parts $\zeta_j^{\leq N}$ of $\zeta_j$. To do so, we will use the fact that $\Pi_\Lambda = T_\Lambda S_\Lambda$. It is crucial here that $\Op(\zeta_j^{\leq N})$ is a differential operator with a finite expansion in powers of $h$. Actually, we are slightly abusing notations because $\zeta_j^{\leq N}$ is not compactly supported in $u^\ast,\eta^\ast$, but this is not a problem because it is polynomial. We can certainly compute the action of a differential operator on $T$ and we have 
\[
\Op(\zeta_j^N)e^{-H/h} \Pi_\Lambda e^{H/h} = (\Op(\zeta_j^N) e^{-H/h} T_\Lambda) S_\Lambda e^{H/h}.
\]
Since the series (defined by the procedure above) in 
\[
\left[e^{\frac{ -i \Phi_T(\cdot,y) + H  }{h} }\Op(\zeta_j^N)\right]^{\mathrm{formal}} ( e^{-H/h}T)(\alpha,y)
\]
is finite, it coincides with 
\[
e^{\frac{- i \Phi_T + H}{h} } \Op(\zeta_j^N) e^{-H/h} T_\Lambda
\]
However, using an argument similar to the one above, it also coincides with 
\[
\left[e^{\frac{ -i \Phi_T(\cdot,y) + H  }{h} }\Op(\zeta_j)\right]^{\mathrm{formal}} ( e^{-H/h}T)(\alpha,y) \equiv 0.
\]
up to the order $(\langle | \alpha | \rangle^{- 2} d(\alpha_x,y)^2+ \langle | \alpha | \rangle^{-1} h)^N $ (because $\zeta_j$ and $\zeta_j^{\leq N}$ coincides up to the order $N$ in $h$ and $2N$ in $\va{(u^*,\eta^*) - \Re \theta(u,\eta)}$). Notice that the holomorphy in $y$ is preserved by this procedure, so that we can use the Holomorphic Stationary Phase Method Proposition \ref{prop:HSP_symbol} as in the proof of Lemma \ref{lemma:TPS-close-diagonal} to compute the kernel of $\Op(\zeta_j^{\leq N}) e^{-H/h} \Pi_\Lambda e^{H/h}$. Making explicit the terms in the Stationary Phase Method, we find that
\begin{equation*}
\begin{split}
\Op(\zeta_j^{\leq N}) e^{-H/h} \Pi_\Lambda e^{H/h}(\alpha,\beta) & = e^{i \frac{\Phi_{TS}^\circ(\alpha,\beta)}{h}} \O\left(\left( \jap{\va{\alpha}}^{-1} h + \jap{\va{\alpha}}^{-2} d_{KN}(\alpha,\beta)^2\right)^{N}\right) \\
    & = \O\p{h^{N} \langle|\alpha|\rangle^{-N}}.
\end{split}
\end{equation*}
Here, we used the coercivity of $\Im \Phi_{TS}^\circ$ again. Summing back $\zeta_{j}^{\leq N}$ and $\zeta_j^{> N}$, we find that
\begin{equation*}
\begin{split}
\Op(\zeta_j) e^{-H/h} \Pi_\Lambda(\alpha,\beta) e^{H/h} = \O\p{\p{\frac{h}{\jap{\va{\alpha}}}}^{N}}.
\end{split}
\end{equation*}
Since here $N$ can be taken arbitrarily large, the proof is complete (derivatives are dealt with similarly).
\end{proof}

\subsubsection{A class of FIOs}

Let us consider an operator $A$ whose kernel is essentially supported near the diagonal on $\Lambda$, and whose kernel has an expansion as in \eqref{eq:forme_operateur}. Borrowing the terminology of \cite{Melin-Sjostrand-75}, this means that $A$ is a Fourier Integral Operator -- FIO -- with complex phase, associated with the Lagrangian manifold
\[
L_\Phi:=\left\{ (\alpha, \mathrm{d}_\alpha\Phi(\alpha,\beta); \beta, \mathrm{d}_\beta \Phi(\alpha,\beta))\ |\ \alpha,\beta \in \Lambda\right\} \subset T^\ast \left(\widetilde{\Lambda} \times \widetilde{\Lambda}\right).
\]
From the condition that the imaginary part of $\Phi$ is coercive away from the diagonal, we deduce that the real points of $L_\Phi$ are exactly
\[
(L_\Phi)_\R = \left\{ (\alpha, \mathrm{d}_\alpha\Phi(\alpha,\alpha); \alpha, \mathrm{d}_\beta \Phi(\alpha,\alpha))\ |\ \alpha \in \Lambda\right\}.
\]
In the case of $\Phi_{TS}^\circ(\alpha,\beta) \coloneqq i H(\alpha) + \Phi_{TS}(\alpha,\beta) - i H(\beta)$, from Lemma \ref{lemma:PhiTS-Lambda} and its proof, we deduce that
\[
(L_{\Phi_{TS}^\circ})_\R = \Delta_\Lambda:=\left\{ (\alpha, \Re \theta(\alpha) ; \alpha, -\Re \theta(\alpha) )\ |\ \alpha \in \Lambda\right\}.
\]
The condition of coercivity also implies that $L_\Phi$ is a strictly positive Lagrangian manifold -- see Definition 3.3 in \cite{Melin-Sjostrand-75}. It will turn out that every relevant operator in our setting shares these properties with $\Phi_{TS}^\circ$.

We will study systematically such operators that are left invariant by $\Pi_\Lambda$ and $\Pi_\Lambda^\ast$. Let us start with a definition. During all this section, we fix a $\p{\tau_0,1}$-adapted Lagrangian $\Lambda$ with $\tau_0$ and $h$ small enough. Since we will be working on $\Lambda$, the results and estimates will depend on $\Lambda$. However, the constants appearing will only depend on $\mathcal{C}^k$ estimates on $\Lambda$, and in this sense, we will say that is valid \emph{uniformly in} $\Lambda$. When a $\Gs$ pseudor $P$ appears, implicitly, we assume that $\Lambda$ is $(\tau_0,s)$-adapted.
\begin{definition}\label{def:oscillating_kernel}
Let $m \in \R$. A \emph{complex FIO associated with $\Delta_\Lambda$} -- of order $m$ -- $A$ on $\Lambda$ is an operator whose reduced kernel is of the form
\begin{equation}\label{eq:oscillating_kernel}
\begin{split}
A(\alpha,\beta)e^{\frac{H(\beta) - H(\alpha)}{h}} = e^{\frac{i \Phi_A(\alpha,\beta)}{h}} \sigma_A(\alpha,\beta) + \O_{\mathcal{C}^\infty}\p{\p{\frac{h}{\jap{\va{\alpha}} + \jap{\va{\beta}}}}^\infty}.
\end{split}
\end{equation}
Here, $\Phi_A \in S_{KN}^1\p{\Lambda \times \Lambda}$ and $\sigma_A \in h^{-n} S_{KN}^m\p{\Lambda \times \Lambda}$ are symbols on $\Lambda \times \Lambda$ in Kohn--Nirenberg classes. Moreover, we assume that for $\alpha \in \Lambda$ we have
\begin{equation}\label{eq:first-order-condition}
\Phi(\alpha,\alpha) = 0 \textup{ and } \mathrm{d}_\alpha \Phi(\alpha,\alpha) = \Re \theta(\alpha),
\end{equation}
and that there are $C,\eta > 0$ such that $\sigma_A$ is supported in $$\set{\p{\alpha,\beta} \in \Lambda \times \Lambda : d_{KN}(\alpha,\beta) \leq \eta},$$ and for every $\alpha,\beta \in \Lambda$ such that $d_{KN}(\alpha,\beta) \leq 2 \eta$ we have
\begin{equation}
\frac{1}{C} \jap{\va{\alpha}} d_{KN}(\alpha,\beta)^2 \leq \Im \Phi_A(\alpha,\beta) \leq C \jap{\va{\alpha}} d_{KN}(\alpha,\beta)^2.\label{eq:coercivity_condition}
\end{equation}
We say that $\Phi_A$ is the (reduced) \emph{phase} of $A$ and that $\sigma_A$ is its symbol.
\end{definition}

When studying FIO, we will often need to consider remainders that are negligible in the following sense.

\begin{definition}\label{def:negligible}
An operator $A$ on $\Lambda$ is said to be negligible, or to have negligible kernel, if it is a FIO with symbol $0$, that is if its reduced kernel satisfies
\begin{equation*}
\begin{split}
A(\alpha,\beta)e^{\frac{H(\beta) - H(\alpha)}{h}} = \O_{\mathcal{C}^\infty}\p{\p{\frac{h}{\jap{\va{\alpha}} + \jap{\va{\beta}}}}^\infty}.
\end{split}
\end{equation*}
\end{definition}

The main goal of this section is to build the necessary tools to prove that the orthogonal projector $B_\Lambda$ is a complex FIOs associated with $\Delta_\Lambda$. Insofar as we will not be making use of any other type of FIOs in this section, and will be working with just one Lagrangian $\Lambda$, we will just write ``FIO'' for ``complex FIO associated with $\Delta_\Lambda$''.

\begin{remark}\label{remark:fourre_tout}
Let us mention that the arguments in the proof of Proposition \ref{lemma:boundedness} imply that an FIO of order $m$ is bounded from $L^2_k\p{\Lambda}$ to $L_{k - m}^2\p{\Lambda}$ for every $k \in \R$. In particular, it makes sense to compose these operators. Notice also that, due to the coercivity condition \eqref{eq:coercivity_condition}, the constant $\eta$ in Definition \ref{def:oscillating_kernel} can be taken arbitrarily small (just multiply $\sigma_A$ by a cutoff function supported near the diagonal and put the remainder in the error term in \eqref{eq:oscillating_kernel}).

It follows from Lemmas \ref{lemma:TPS-far-diagonal}, \ref{lemma:TPS-close-diagonal} and \ref{lemma:PhiTS-Lambda} that $\Pi_\Lambda$ is an FIO of order $0$. More generally, if $P$ is a $\G^s$ pseudor of order $m$, then the same lemmas imply that if $\tau_0$ and $h$ are small enough then $T_\Lambda P S_\Lambda$ is an FIO of order $m$ (this fact has been established in the proof of Proposition \ref{lemma:boundedness}). The phase associated with both of these operators is $\Phi_{TS}^\circ$, and it satisfies the coercivity condition \eqref{eq:coercivity_condition} due to Lemma \ref{lemma:PhiTS-Lambda}.

Notice also that a negligible operator is a $\O(h^N)$ as on operator from $L_{-N}^2\p{\Lambda}$ to $L^2_N\p{\Lambda}$ for every $N \in \N$.
\end{remark}

\begin{lemma}\label{lemma:composition_oscillating_kernel}
Let $A$ and $B$ be FIOs, whose phases are respectively $\Phi_A$ and $\Phi_B$. Let $f$ be a symbol in Kohn--Nirenberg class on $\Lambda$. Then $A^\ast$, $A B$, $f A$ and $Af$ also are FIOs. The phase of $AfB$ is given by
\begin{equation}\label{eq:phase_composition}
\begin{split}
\Phi_{AB} (\alpha,\beta) = \textup{v.c.}_{\gamma} \p{\Phi_A\p{\alpha,\gamma} + \Phi_B\p{\gamma,\beta}}.
\end{split}
\end{equation}
Here, v.c. stands for critical value, and this critical value is defined in the sense of almost analytic extension (see \cite{Melin-Sjostrand-75} in particular Lemma 2.1). Moreover, from an expansion for $f$ and the symbols of $A$ and $B$, we deduce an asymptotic expansion for the symbol of $AfB$ given by the stationary phase method. 
\end{lemma}

It will be convenient to give a name to the reduced phase of $\Pi_\Lambda^*$, which is $\Phi_{TS}^*(\alpha,\beta) \coloneqq - \overline{\Phi_{TS}^\circ\p{\beta,\alpha}} = - i H(\alpha) - \overline{\Phi_{TS}\p{\beta,\alpha}} + i H(\beta)$. Recall that, in order to study the orthogonal projector $B_\Lambda$, we want to study operators of the form $\Pi_\Lambda f \Pi_\Lambda^*$ where $f$ is a symbol on $\Lambda$. From Lemma \ref{lemma:composition_oscillating_kernel}, we already know that $\Pi_\Lambda f \Pi_\Lambda^*$ is a complex FIO adapted with $\Delta_\Lambda$.

\begin{remark}\label{remark:expansion_symbole_FIO}
If $A$ is an FIO of order $m$, we say that its symbol satisfies an asymptotic expansion
\begin{equation*}
\begin{split}
\sigma_A \simeq \frac{1}{(2 \pi h)^n} \sum_{k \geq 0} h^k a_k,
\end{split}
\end{equation*}
if, for every $k \in \N$, we have $a_k \in S_{KN}^{m - k}\p{\Lambda \times \Lambda}$, and, for every $N \in \N$, we have that
\begin{equation*}
\begin{split}
\sigma_A - \frac{1}{(2 \pi h)^n} \sum_{k = 0}^{N-1} h^k a_k \in h^{N - n} S_{KN}^{m - N}\p{\Lambda \times \Lambda}.
\end{split}
\end{equation*}
\end{remark}

\begin{proof}[Proof of Lemma \ref{lemma:composition_oscillating_kernel}]
The statements concerning $Af$ and $fA$ are elementary. For the adjoint, it suffices to observe that if $\Phi$ satisfies \eqref{eq:first-order-condition} and \eqref{eq:coercivity_condition}, then so does
\[
\Phi^\ast(\alpha,\beta) := - \overline{\Phi(\beta,\alpha)}.
\]
Thus, we only need to consider the composition $AB$. Since the kernels of $A$ and $B$ are rapidly decaying away from the diagonal and grow at most polynomially near the diagonal, the kernel of $AB$ is rapidly decaying (in $\mathcal{C}^\infty$) at any fixed distance of the diagonal in $\Lambda \times \Lambda$. Let us turn to the reduced kernel of $AB$ near the diagonal. The error term in \eqref{eq:oscillating_kernel} is proved to be negligible as in the non-diagonal case. Consequently, we may assume that the reduced kernel of $AB$ is given by the integral
\begin{equation}\label{eq:noyau_compose}
\begin{split}
\int_{\Lambda} e^{i \frac{\jap{\va{\alpha}}}{h} \Psi_{\alpha,\beta}\p{\gamma}} \sigma_A\p{\alpha,\gamma} \sigma_B\p{\gamma,\beta} \mathrm{d}\gamma,
\end{split}
\end{equation}
where the phase $\Psi_{\alpha,\beta}$ is defined by
\begin{equation*}
\begin{split}
\Psi_{\alpha,\beta}\p{\gamma} & = \frac{ \Phi_A\p{\alpha,\gamma} + \Phi_\beta\p{\gamma,\beta}}{\jap{\va{\alpha}}}.
\end{split}
\end{equation*}
Thanks to the condition \eqref{eq:coercivity_condition} that we imposed on the phases $\Phi_A$ and $\Phi_B$, we see that this phase has non-negative imaginary part. From \eqref{eq:first-order-condition}, we see that when $\alpha = \beta$ the phase $\Psi_{\alpha,\alpha}$ has a critical point at $\gamma = \alpha$, which is non-degenerate since \eqref{eq:coercivity_condition} imposes that the Hessian of the phase has a definite positive imaginary part. Melin--Sjöstrand's $\mathcal{C}^\infty$ version of the stationary phase method with complex phases \cite[Theorem 2.3]{Melin-Sjostrand-75} thus applies. We apply these results after a proper rescaling (in $\alpha_\xi,\beta_\xi,\gamma_\xi$) that replaces the estimates using the Kohn--Nirenberg metric by uniform $\mathcal{C}^\infty$ estimates. Notice that when doing so there is a Jacobian that appears in \eqref{eq:noyau_compose} which is a symbol of order $n$. This is because the volume form $\mathrm{d}\gamma$ in \eqref{eq:noyau_compose} is not the volume form associated with the Kohn--Nirenberg metric but to the symplectic form $\omega_R$. In particular, one can check that the determinant of the Hessian of $\gamma \mapsto \Phi_A(\alpha,\gamma) + \Phi_\beta(\gamma,\beta)$ associated with the volume form $\mathrm{d}\gamma$ is a symbol of order $0$, and not $2n$ as we could expect from the coercivity condition \eqref{eq:coercivity_condition}. Thus, we find that for $\alpha,\beta$ near the diagonal, the kernel of $AB$ is given by
\begin{equation*}
\begin{split}
e^{ \frac{i \Phi_{AB}(\alpha,\beta)}{h}} \sigma_{AB}(\alpha,\beta),
\end{split}
\end{equation*}
where $\sigma_{AB}$ is a symbol in the class $h^{-n} S_{KN}^{m}\p{\Lambda \times \Lambda}$, where $m$ is the sum of the orders of $A$ and $B$ (recall that the dimension of $\Lambda$ is $2n$). The phase $\Phi_{AB}$ is given by \eqref{eq:phase_composition} in the sense of almost analytic extension. Since we already understand the reduced kernel of $AB$ off diagonal, it only remains to prove that the phase $\Phi_{AB}$ satisfies the conditions \eqref{eq:first-order-condition} and \eqref{eq:coercivity_condition} near the diagonal. For \eqref{eq:first-order-condition}, it is not too hard, since we already identified the critical point when $\alpha=\beta$.

For \eqref{eq:coercivity_condition}, we pick $\alpha \in \Lambda$ and consider $\Psi_{\alpha,\beta}(\gamma)$ as a function of $\beta$ and $\gamma$ near $\alpha$. We want to apply the Fundamental Lemma \ref{lemma:fundamental-lemma}. We deduce from \eqref{eq:coercivity_condition}, satisfied by $\Phi_A$ and $\Phi_B$, that $\Im \Psi_{\alpha,\alpha}(\alpha) = 0$ and that, for $\beta$ and $\gamma$ near $\alpha$, we have
\begin{equation*}
\begin{split}
\Im \Psi_{\alpha,\beta}\p{\gamma} \geq d_{KN}\p{\alpha,\beta}^2 + d_{KN}\p{\alpha,\gamma}^2.
\end{split}
\end{equation*}
Consequently, $\p{\beta,\gamma} \mapsto \Im \Psi_{\alpha,\beta}\p{\gamma}$ has a critical point at $\beta = \gamma = \alpha$, and its Hessian is (uniformly) definite positive. Recalling that
\begin{equation*}
\begin{split}
\Phi_{AB}(\alpha,\beta) = \jap{\va{\alpha}} \textup{v.c}_{\gamma} \Psi_{\alpha,\beta}\p{\gamma},
\end{split}
\end{equation*}
we find using Lemma \ref{lemma:fundamental-lemma} (with a proper rescaling) that $\beta \mapsto \Im\Phi_{AB}(\alpha,\beta)$ has a critical point at $\beta = \alpha$ whose Hessian is greater than $C^{-1} \jap{\va{\alpha}} g_{KN}$. The condition \eqref{eq:coercivity_condition} follows then by Taylor's formula using the symbolic estimates on $\Phi_{AB}$.

Notice that we applied here Lemma \ref{lemma:fundamental-lemma} to the phase $\Psi_{\alpha,\beta}$ which is not analytic. However, there is no problem here since the proof of Lemma \ref{lemma:fundamental-lemma} is based on a direct computation on the Hessian of the imaginary part of the phase at $\beta = \alpha$. The result of this computation is still valid here since when $\beta = \alpha$ the critical point of the phase is real and consequently all the almost analytic errors vanish.
\end{proof}

\begin{remark}
As in the more regular cases described in \S \ref{sec:stationary}, the stationary phase method \cite[Theorem 2.3]{Melin-Sjostrand-75} comes with an asymptotic expansion with semi-explicit coefficients. Actually, the expression for these coefficients is the same as in Remarks \ref{remark:expression_coefficients} and \ref{remark:coefficients_symboles}. However, when defining the differential operator $P_m$, all the objects must be replaced by their almost analytic extensions, in particular the Morse coordinates are now defined in the sense of \cite[(2.13)]{Melin-Sjostrand-75}
\end{remark}

\begin{remark}\label{remark:a_propos_des_jets}
From now on, the notion of jets and equations holding at infinite order will be essential. Let $N$ be a $\mathcal{C}^\infty$ manifold. Let $x_0$ be a point of $N$ and $f$ be a $\mathcal{C}^\infty$ function defined on a neighbourhood of $x_0$ in $N$. In general, we cannot define the Taylor expansion of $f$ at $x_0$: the first term in the expansion, given by the derivative $\mathrm{d}f(x_0)$ of $f$, is the only one that is always intrinsically defined. However, when $\mathrm{d}f(x_0) = 0$, then the second derivative $\mathrm{d}^2f(x_0)$ is well-defined. More generally, if the $k$ first derivatives $\mathrm{d}f(x_0),\dots,\mathrm{d}^k f(x_0)$ vanish, then one may define intrinsically the $k+1$th derivative $\mathrm{d}^{k+1}f(x_0)$ of $f$ at $x_0$ (this is a $k+1$th symmetric linear form on $T_{x_0} M$). Hence, while we cannot define the Taylor expansion of $f$ at $x_0$, it makes sense to say that $f$ vanishes to order $k \in \N \cup \set{\infty}$ at $x_0$. Notice that if $d$ is a distance on $M$ induced by a smooth Riemannian metric then it is equivalent to say that $f$ vanishes to order $k$ at $x_0$ and that $f(x) \underset{x \to x_0}{=} \O\p{d(x,x_0)^k}$.

Now, if $P$ is a submanifold of $N$ and $f$ and $g$ are two $\mathcal{C}^\infty$ functions defined on a neighbourhood of $P$, we say that the equation $f = g$ holds at infinite order on $P$, or in the sense of jets on $P$, if the function $f-g$ vanishes to infinite order at every point $x_0 \in P$. The space of jets may then be defined as the quotient $\mathcal{C}^\infty\p{N} / \sim$, where the relation $\sim$ is defined by: $f \sim g$ if and only if $f=g$ to infinite order on $P$.

In the analysis, below, we will often have $N = \Lambda \times \Lambda$ and $P$ the diagonal of $\Lambda \times \Lambda$. Then, if $f : N \to \C$ is a symbol of order $m$ that vanishes to all orders on $P$, for every $M > 0$ and $k \in \N$ there is a constant $C_{k,M}$ such that for every $\p{\alpha,\beta} \in N$ such that $d_{KN}\p{\alpha,\beta} \leq M$ we have
\begin{equation*}
\begin{split}
\va{f(\alpha,\beta)} \leq C_{k,M} \jap{\va{\alpha}}^m d_{KN}\p{\alpha,\beta}^k,
\end{split}
\end{equation*}
where $d_{KN}$ denotes the distance induced by the Kohn--Nirenberg metric. Here, the constant $C_{k,M}$ depends on the symbolic estimates on $f$. In particular, even if the vanishing to infinite order of $f$ on the diagonal looks like a merely algebraic condition, it can be used to derive effective estimates on $f$ since we already have a control on the derivatives of $f$ (in the Kohn--Nirenberg metric).

Finally, if $N_1, N_2$ are two $\mathcal{C}^\infty$ manifolds, $P_1,P_2$ are submanifolds respectively of $N_1$ and $N_2$ and $f$ is a $\mathcal{C}^\infty$ map from $N_1$ to $N_2$, we say that $f$ is valued in $P_2$ -- or loosely $f \in P_2$ -- at infinite order on $P_1$ if, for every $\mathcal{C}^\infty$ function $g : N_2 \to \C$ identically equal to zero on $P_2$, the function $g \circ f$ vanishes at infinite order on $P_1$.
\end{remark}

\subsubsection{``Eikonal'' condition on the phase}

From this section, we will consider $A$, an FIO in the sense of the previous section, and assume that $\Pi_\Lambda A = A \Pi_\Lambda^\ast = A$. We will start by finding some constraints on the phase of $A$. In this direction, the first step is to observe that
\begin{lemma}\label{lemma:phase_left_multiplication}
Let $A$ and $B$ be FIOs. Denote respectively by $\Phi_A$ and $\Phi_B$ the phases of $A$ and $B$. Assume that $\Phi_A$ satisfies
\begin{equation}\label{eq:dans_JLambda}
\begin{split}
\mathrm{d}_\alpha \Phi_{A} \in \J_\Lambda
\end{split}
\end{equation}
at infinite order on the diagonal of $\Lambda \times \Lambda$. Then, if $\Phi_{AB}$ denotes the phase of $AB$ provided by Lemma \ref{lemma:composition_oscillating_kernel}, the phase $\Phi_{AB}$ also satisfies \eqref{eq:dans_JLambda} at infinite order on the diagonal.
\end{lemma}

\begin{proof}
We will use the same stratagem that we used to deduce the differential of $\Phi_{TS}$ from the differential of $\Phi_T$, taking into account the use of almost analytic extensions. Recall that $\Phi_{AB}$ is defined by \eqref{eq:phase_composition} using almost analytic extensions. More precisely, if $\widetilde{\Phi}_A$ and $\widetilde{\Phi}_B$ denote almost analytic extensions respectively for $\Phi_A$ and $\Phi_B$, then, when $\alpha$ is near $\beta$, there is a unique point $\gamma_c\p{\alpha,\beta}$ near $\alpha$ and $\beta$ in an almost analytic neighbourhood of $\Lambda$ such that 
\begin{equation*}
\begin{split}
\partial_\beta \widetilde{\Phi}_A(\alpha, \gamma_c(\alpha,\beta)) + \partial_\alpha \widetilde{\Phi}_B(\gamma_c(\alpha,\beta),\beta) = 0, 
\end{split}
\end{equation*}
where $\partial$ denotes the holomorphic part of the exterior derivative. Then we have to infinite order at the diagonal
\begin{equation*}
\begin{split}
\Phi_{AB}\p{\alpha,\beta} = \widetilde{\Phi}_A\p{\alpha,\gamma_c(\alpha,\beta)} + \widetilde{\Phi}_B\p{\gamma_c(\alpha,\beta),\beta}.
\end{split}
\end{equation*}
Differentiating this expression with respect to $\alpha$, we find that, for $\alpha$ near $\beta$,
\begin{equation}\label{eq:differentielle_composition_phase}
\begin{split}
\mathrm{d}_\alpha \Phi_{AB}\p{\alpha,\beta} & = \mathrm{d}_\alpha \widetilde{\Phi}_A\p{\alpha,\gamma_c(\alpha,\beta)}  \\ & \qquad \qquad + \p{\bar{\partial}_\beta \widetilde{\Phi}_A(\alpha, \gamma_c(\alpha,\beta)) + \bar{\partial}_\alpha \widetilde{\Phi}_B(\gamma_c(\alpha,\beta),\beta)} \circ \mathrm{d}_\alpha \gamma_c (\alpha,\beta).
\end{split}
\end{equation}
Since $\widetilde{\Phi}_A$ and $\widetilde{\Phi}_B$ are almost analytic and $\gamma_c(\alpha,\alpha) = \alpha$ is real, we see that the second term in the right-hand side of \eqref{eq:differentielle_composition_phase} vanish at infinite order on the diagonal of $\Lambda \times \Lambda$. Thus, we see that the equation \eqref{eq:dans_JLambda} satisfied at infinite order on the diagonal by $\mathrm{d}_\alpha \Phi_A$ is also satisfied by $\mathrm{d}_\alpha \Phi_{AB}$.
\end{proof}

Since $\Pi_\Lambda A = A$, we can replace (if necessary) $\Phi_A$ and $\sigma_A$ by the phase and symbol obtained by the application of the methods of \cite{Melin-Sjostrand-75} as in the proof of Lemma \ref{lemma:composition_oscillating_kernel}, so as to assume that $\Phi_A$ satisfies the assumptions of Lemma \ref{lemma:phase_left_multiplication}.

We can consider adjoints, and for this introduce the new submanifolds of $T^* \Lambda \otimes \C$
\begin{equation}\label{eq:def-Jbar-Jstar}
\begin{split}
\overline{\J_\Lambda} := \set{\p{\alpha, \overline{\alpha^*}}: \p{\alpha,\alpha^*} \in \J_\Lambda} \text{ and } \J_\Lambda^* := \set{\p{\alpha, - \overline{\alpha^*}}: \p{\alpha,\alpha^*} \in \J_\Lambda}.
\end{split}
\end{equation}
Here, the complex conjugacy is the one induced by the structure of tensor product on $T^* \Lambda \otimes \C$. The proof that led to Lemma \ref{lemma:phase_left_multiplication} also gives that: 

\begin{lemma}\label{lemma:phase_right_multiplication}
Let $A$ and $B$ be FIOs. Denote respectively by $\Phi_A$ and $\Phi_B$ the phases of $A$ and $B$. Assume that $\Phi_B$ satisfies
\begin{equation}\label{eq:dans_JLambda_star}
\begin{split}
\mathrm{d}_\beta \Phi_{B} \in \J_\Lambda^*
\end{split}
\end{equation}
at infinite order on the diagonal of $\Lambda \times \Lambda$. Then, if $\Phi_{AB}$ denotes the phase of $AB$ provided by Lemma \ref{lemma:composition_oscillating_kernel}, then $\Phi_{AB}$ also satisfies \eqref{eq:dans_JLambda_star} at infinite order on the diagonal.
\end{lemma}

Here, it is important to notice that the manifold $\J_\Lambda^*$ has been defined so that the phase $\Phi_{TS}^*$ of $\Pi_\Lambda^*$ satisfies \eqref{eq:dans_JLambda_star}. Consequently, possibly changing again (and for the last time) the phase $\Phi_A$ as before, we can assume that $\Phi_A$ also satisfies both \eqref{eq:dans_JLambda} and \eqref{eq:dans_JLambda_star}. We will see that this is actually sufficient to determine $\Phi_A$. To understand this, we start by observing that $\J_\Lambda$ is \emph{really} not a real manifold, as
\begin{lemma}\label{lemma:transversality}
Assume that $\Lambda$ is a $(\tau_0,1)$-adapted Lagrangian with $\tau_0$ small enough. The intersection of $\J_\Lambda\cap \overline{\J_\Lambda} = \Sigma_\Lambda$ is transverse in each fiber $T_\alpha \Lambda \otimes \C$, where $\Sigma_\Lambda$ denotes the graph of $\Re \theta$ in $T^\ast \Lambda$.
\end{lemma}

\begin{proof}
If we prove that $\J_{T^\ast M}$ and $\overline{\J_{T^\ast M}}$ are uniformly transverse, then the result will follow by a perturbative argument. Additionally, it suffices to prove the transversality in each fiber of $T^* \p{T^* M} \otimes \C$. Recall that for $\alpha \in T^*M$ we have
\[
\J_{T^\ast M}\cap T_\alpha (T^\ast M) = \{ \mathrm{d}_\alpha \Phi_T(\alpha,y)\ |\ y \in \widetilde{M},\ d(y,\alpha_x)< \epsilon\}.
\]
Letting $L$ denote the tangent space to $\J_{T^\ast M}\cap T_\alpha (T^\ast M)$ at $\Re \theta(\alpha)$, we want to check that $L$ and $\overline{L}$ are uniformly transverse. The space $L$ is the image of $T_{\alpha_x} \widetilde{M}$ by the differential at $y = \alpha_x$ of the application
\begin{equation*}
\begin{split}
y \mapsto \mathrm{d}_\alpha \Phi_T(\alpha,y).
\end{split}
\end{equation*}
Since the phase $\Phi_T$ is an admissible phase (recall Definition \ref{def:nonstandardphase}), we can write $L$ in the form (with local coordinates)
\[
\left\{ \begin{pmatrix}
Au + i Bu \\ u
\end{pmatrix}\ \middle|\ u\in\C^n \right\},
\]
where $A$, $B$ are real and $B$ is invertible. An element of the intersection $L\cap \overline{L}$ corresponds to vectors $u$, $v$ such that $u= \overline{v}$ and $Au + i Bu=A \overline{v} - i B \overline{v}$, so that $u=v=0$. Since the invertibility of $B$ comes with uniform estimates, this comes with a uniform estimate on the transversality of $L$ and $\overline{L}$.
\end{proof} 

\begin{lemma}\label{lmcarphase}
Assume that $\Lambda$ is a $(\tau_0,1)$-adapted Lagrangian with $\tau_0$ small enough. There is a unique jet $\Phi$ on the diagonal of $\Lambda \times \Lambda$ that satisfies \eqref{eq:dans_JLambda}, \eqref{eq:dans_JLambda_star} and
\begin{equation}\label{eq:condtion_base_phase}
\begin{split}
\Phi(\alpha,\alpha) = 0,\ \mathrm{d}_\alpha \Phi(\alpha,\alpha) = - \mathrm{d}_\beta \Phi(\alpha,\alpha) = \Re \theta(\alpha)
\end{split}
\end{equation}
for every $\alpha \in \Lambda$.
\end{lemma}

\begin{remark}\label{remark:phi_lambda}
Notice that it follows from Lemmas \ref{lemma:phase_left_multiplication} and \ref{lemma:phase_right_multiplication} that $\Phi$ from Lemma \ref{lmcarphase} is the jet on the diagonal of the phase of $\Pi_\Lambda \Pi_\Lambda^*$. We will denote by $\Phi_\Lambda$ the phase of $\Pi_\Lambda \Pi_\Lambda^*$. Now, notice that if $A$ is a FIO such that the jet of its phase on the diagonal is $\Phi$ (which will happen by application of Lemma \ref{lmcarphase}), we may always assume that its phase is exactly $\Phi_\Lambda$. Indeed, it follows from the coercivity condition \eqref{eq:coercivity_condition} that the error that occurs when replacing the phase of $A$ by $\Phi_\Lambda$ may be considered as part of the remainder term in \eqref{eq:oscillating_kernel}.

Notice also that the phase $(\alpha,\beta) \mapsto -\overline{\Phi_\Lambda(\beta,\alpha)}$ satisfies the conditions from Lemma \ref{lmcarphase} too. Hence, we may replace $\Phi_\Lambda$ by $(\alpha,\beta) \mapsto (\Phi_\Lambda(\alpha,\beta) - \overline{\Phi_\Lambda(\beta,\alpha)})/2$ and assume that the phase $\Phi_\Lambda$ satisfies $\Phi_\Lambda(\alpha,\beta) = - \overline{\Phi_\Lambda(\beta,\alpha)}$.
\end{remark}

The proof of Lemma \ref{lmcarphase} is based on an almost analytic version of the following elementary observation. Let $\Gamma$ be a Lagrangian submanifold of $T^\ast \R^m$, and $f: T^\ast \R^m \to \R$ a smooth function. Assume that $f$ vanishes on $\Gamma$. Denoting $H_f$ the corresponding Hamiltonian vector field, we pick $u$ a vector tangent to $\Gamma$ and observe that
\[
\omega( u,H_f) = \mathrm{d}f(u) =0,
\]
so that $H_f$ has to be tangent to $\Gamma$. If we have many independent functions that vanish on $\Gamma$, its tangent space is thus determined. We already have the functions $\zeta_{j,0}$ (from Proposition \ref{prop:existence-Z_j}) that vanish on $\J_\Lambda$. It will be useful to introduce the functions
\[
\zeta_{j,0}^\ast(\alpha,\alpha^\ast)= \overline{\zeta_{j,0}(\alpha, - \overline{\alpha^\ast} )}.
\]
They vanish on $\J_\Lambda^\ast$.

\begin{proof}[Proof of Lemma \ref{lmcarphase}]
The existence is immediate since the phase $\Phi_\Lambda$ from Remark \ref{remark:phi_lambda} satisfies the conditions from the lemma. Let us then focus on the uniqueness.

We start by picking an almost analytic extension of $\Phi$ and consider the graph 
\[
\Gamma_{\Phi} = \left\{ (\alpha, \partial_\alpha \Phi ; \beta, \partial_\beta \Phi)\ \middle|\ \alpha,\beta \in \widetilde{\Lambda} \right\} \subset T^\ast (\widetilde{\Lambda} \times \widetilde{\Lambda}),
\]
where $\partial_\alpha$ denotes the $\C$-linear part of the exterior derivative, that is $\mathrm{d}_\alpha = \partial_\alpha + \bar{\partial}_\alpha$. This is an almost analytic Lagrangian manifold, in the sense that the complex analytic symplectic form vanishes at infinite order near the real points of $\Gamma_\Phi$, which are exactly the points of $\Delta_\Lambda$ (for more details on almost analytic Lagrangian manifolds, see \S 3 in \cite{Melin-Sjostrand-75}).

We can also pick almost analytic extensions of the $\zeta_{j,0}^{(\ast)}$'s. From the conditions \eqref{eq:dans_JLambda} and \eqref{eq:dans_JLambda_star}, we deduce that
\[
\zeta_{j,0}(\alpha, \partial_\alpha \Phi) \text{ and } \zeta_{j,0}^\ast(\beta, \partial_\beta \Phi ) \text{ are } \O( |\alpha - \beta| + \va{\Im \alpha} + \va{\Im \beta})^\infty. 
\]
We identify $\zeta_{j,0}$ (resp. $\zeta_{j,0}^\ast$) with $(\alpha,\alpha^\ast,\beta,\beta^\ast) \mapsto \zeta_{j,0}(\alpha,\alpha^\ast)$ (resp. $\zeta_{j,0}^\ast(\beta, \beta^\ast)$). Denoting by $\omega$ the complex symplectic form of $T^\ast ( \widetilde{\Lambda} \times\widetilde{\Lambda})$, we have for $u$ a tangent vector to $\Gamma_\Phi$, and $f$ one of the functions $\zeta_{j,0}$, $\zeta_{j,0}^\ast$, $j = 1 \dots n$
\[
\omega(u,H_f) = \partial f(u) = \mathrm{d}f (u) + \O( |\alpha-\beta|)^\infty = \O( |\alpha - \beta| + \va{\Im \alpha} + \va{\Im \beta})^\infty.
\]
Here, the vector field $H_f$ is defined by $\omega(\cdot,H_f) = \partial f$. We find that $H_f$ is tangent to $\Gamma_\Phi$ to all order on the diagonal. Since $\Gamma_\Phi$ has (almost analytic complex-)dimension $4n$, contains $\Delta_\Lambda$, which has dimension $2n$, and we have $2n$ vector fields tangent to $\Phi$, it suffices to check that these vector fields form a free family, generating a space transverse to $T\Delta_\Lambda\otimes \C$.

The fact that the $\mathrm{d}\zeta_{j,0}$ form a free family on $\J_\Lambda$ near its real points implies that the $H_{\zeta_{j,0}}$, $j= 1\dots n$ are free, and likewise for the $H_{\zeta_{j,0}^\ast}$, $j=1\dots n$. Since they act either on $\alpha,\alpha^\ast$ or on $\beta,\beta^\ast$, the $H_f$'s form a free family. We now have to consider the transversality with $\Delta_\Lambda$. We consider $a,b\in \C^n$ such that
\[
\sum_{j= 1}^n a_j H_{\zeta_{j,0}} + \sum_{j = 1}^n b_j H_{\zeta_{j,0}^\ast} \in T \Delta_\Lambda \otimes \C.
\]
From the structure of $\Delta_\Lambda$, this implies
\[
\sum a_j \mathrm{d}_{\alpha^*}\zeta_{j,0}(\alpha,\Re \theta(\alpha)) + \sum b_j \mathrm{d}_{\alpha^*} \zeta_{j,0}^\ast (\alpha, - \Re \theta(\alpha)) =0.
\]
From the definition of $\zeta_{j,0}^\ast$, this means (here the complex conjugacy makes sense due to the tensor product structure on $T^* \Lambda \otimes \C$)
\[
\sum a_j \mathrm{d}_{\alpha^*}\zeta_{j,0} (\alpha,\Re \theta(\alpha)) - \sum b_j \overline{\mathrm{d}_{\alpha^*}\zeta_{j,0} }(\alpha, \Re \theta(\alpha)) =0
\]
In the fiber $T_\alpha^\ast \Lambda$, recall that the $\zeta_{j,0}$, $j=1\dots n$, form a system of equations for $\J_\Lambda$, so that the linear forms $\mathrm{d}_{\alpha^*}\zeta_{j,0}$ generate the orthogonal dual of $T_{\Re \theta(\alpha)}\J_\Lambda \cap T_\alpha^\ast \Lambda$. Their complex conjugates thus generate the orthogonal dual of $T_{\Re \theta(\alpha)}\overline{\J_\Lambda}\cap T_\alpha^\ast\Lambda$. The transversality lemma \ref{lemma:transversality} thus implies that equation above only admits $0$ as solution, so $a=b=0$ and the proof is complete.
\end{proof}

\subsubsection{Transport equations on the symbol}

Let us now come back to our original datum of $A$, an FIO such that $\Pi_\Lambda A = A \Pi_\Lambda^\ast = A$. We assume, as we may, that $\Phi_A$ satisfies both \eqref{eq:dans_JLambda} and \eqref{eq:dans_JLambda_star}, so that $\Phi_A$ can be replaced by the phase $\Phi_\Lambda$ from Remark \ref{remark:phi_lambda}. Under these conditions, $A$ is determined up to a negligible operator by the value of its symbol on the diagonal of $\Lambda \times \Lambda$, as we prove now.
\begin{lemma}\label{lemma:de_la_diagonale_le_symbole}
Assume that $\Lambda$ is a $(\tau_0,1)$-adapted Lagrangian with $\tau_0$ small enough. Let $A$ be an FIO with phase $\Phi_\Lambda$. Assume that $\Pi_\Lambda A = A$ and $A \Pi_\Lambda^* = A$. Assume in addition that the symbol $\sigma_A$ of $A$ is an $\O(h/\langle\alpha\rangle)^\infty$ \emph{on} the diagonal, then $A$ is a negligible operator in the sense of Definition \ref{def:negligible}.
\end{lemma} 

The first idea may be to compute directly the following product and try to identify the terms:
\[
\Pi_\Lambda A (\alpha,\beta)e^{\frac{H(\beta)-H(\alpha)}{h}}\sim e^{\frac{i\Phi_{\Pi_\Lambda A}(\alpha,\beta)}{h}} \left[ \sum_{k\geq 0} h^k P_k\Big( \sigma_{\Pi_\Lambda}(\alpha,\gamma)\sigma_{A}(\gamma,\beta)\Big)_{|\gamma=\gamma_c(\alpha,\beta)} \right].
\]
Here, the operators $P_k$ are differential operators of order $2k$ in the variable $\gamma$. In particular, for the term of order $h^k$, the coefficient involves $2k$ derivatives of $\sigma_A$. At least formally, the equation $\Pi_\Lambda A = A$ thus gives a collection of PDEs on the symbol of $A$, that are a priori not easy to interpret. This is why we have to take a different approach, relying crucially on the $Z_j$'s constructed in \S~\ref{sub:annihilating}.

\begin{proof}
The proof of this lemma occupies the rest of this section. Far from the diagonal, the kernel of $A$ is very small by definition, so we can concentrate on a neighbourhood of fixed size of the diagonal. From the coercivity condition \eqref{eq:coercivity_condition} satisfied by $\Phi_A$, we see that we only need to prove that all the derivatives of $\sigma_A$ on the diagonal of $\Lambda \times \Lambda$ are $\mathcal{O}\p{(h/\langle|\alpha|\rangle)^\infty}$.

To do so, recall the operators $\Op(\zeta_j)$ from Proposition \ref{prop:existence-Z_j} (we can work ``around a fiber'' in the sense of Proposition \ref{prop:existence-Z_j}) and notice that
\begin{equation*}
\begin{split}
\Op(\zeta_j)e^{-\frac{H}{h}} Ae^{\frac{H}{h}}(\alpha,\beta) & = \p{\Op(\zeta_j) e^{- H/h} \Pi_\Lambda e^{H/h}} e^{-H/h} A e^{H/h} (\alpha,\beta) \\ & = \mathcal{O}\p{\p{\frac{h}{\jap{\va{\alpha}} + \jap{\va{\beta}}}}^\infty}.
\end{split}
\end{equation*}
We want to deduce some information about the symbol of $A$. Certainly, applying \cite{Melin-Sjostrand-75} once again,
\[
\Op(\zeta_j)e^{-\frac{H}{h}} A(\alpha,\beta)e^{\frac{H}{h}} = e^{\frac{i}{h}\Phi_A(\alpha,\beta)} c(\alpha,\beta) + \O( (h/\langle|\beta|\rangle)^\infty),
\]
where $c$ is supported near the diagonal, and has a semi-classical expansion involving derivatives of $\zeta_j$ and $\sigma_A$. Since $e^{i\Phi_A/h}$ is a Gaussian term centered at the diagonal, we deduce that for $d_{KN}(\alpha,\beta) \leq C (h/\jap{\va{\alpha}})^{1/2}$, 
\[
c(\alpha,\beta)=  \O( (h/\langle|\alpha|\rangle)^\infty).
\]
This is a priori a $L^\infty$ estimate. However, we know that $c$ is a symbol, so that using finite differences and Taylor's Formula, we deduce that this estimate actually holds in $\mathcal{C}^\infty$. We want to deduce that $\sigma_A= \O(h/\langle|\alpha|\rangle)^\infty$ at infinite order on the diagonal.

We consider now the semi-classical expansion for $c$, which starts as (we use the same coordinates $(u,\eta)$ as in \S \ref{sub:annihilating})
\[
\zeta_{j,0}(u,\eta, \mathrm{d}_{u,\eta}\Phi_A(u,\eta,u',\eta'))\sigma_{A}(u,\eta,u',\eta') + \O\p{\frac{h}{\jap{\eta}}}.
\]
From the choice of $\zeta_{j,0}$ and Lemma \ref{lemma:phase_left_multiplication}, the first term here vanishes at infinite order on the diagonal. Next, we consider the term of order $h$ in the expansion. For this we recall the usual expansion 
\begin{equation}\label{eq:last_label}
\begin{split}
\int & e^{\frac{i}{h}\langle x-y,\xi\rangle} p(x,\xi) e^{\frac{i}{h}S(y)}a(y)\mathrm{d}y\mathrm{d}\xi = \\
			&e^{\frac{i}{h}S(x)} \left[ p(x,\mathrm{d}_x S)a(x) + \frac{h}{i} \left( \nabla_\xi p\cdot \nabla_x a +\frac{a}{2}\Tr \partial^2_\xi p \partial^2_x S \right) +\O(h^2 a) \right] + \O(h^\infty).
\end{split}
\end{equation}
(here, $\O(h^2 a)$ is a bit of an abuse of notations since it involves derivatives of $a$). It follows that
\[
\begin{split}
c(u,\eta,u',\eta') &= \frac{h}{i}\nabla_{u^\ast,\eta^\ast}\zeta_{j,0}(u,\eta; \mathrm{d}_{u,\eta} \Phi_{TS}^\circ (u,\eta,u',\eta'))\cdot\nabla_{u,\eta}\sigma_{A}(u,\eta,u',\eta') \\ 
& \qquad + h F_j(u,\eta,u',\eta') \sigma_{A}(u,\eta,u',\eta') \\ 
& \qquad + h^2 \jap{\eta}^{-1} P_j \sigma_A(u,\eta,u',\eta') + \O_{C^\infty}\p{\p{\frac{h}{\jap{\eta}}}^\infty}.
\end{split}
\]
The function $F_j$ is the sum of $\zeta_{j,1}(u,\eta,\mathrm{d}_{u,\eta} \Phi_{TS}^{\circ}(u,\eta,u',\eta'))$ and a term involving second derivatives of $\zeta_{j,0}$ and of the phase, according to \eqref{eq:last_label}. The operator $P_j$ comes from the application of the stationary phase method, and in particular we know that if we can prove that the derivatives of the symbol $\sigma_A$ on the diagonal are $\mathcal{O}(\jap{\eta}^{-N} h^N)$ then the  derivatives of the symbol $P_j \sigma (u,\eta,u,\eta)$ satisfy the same estimates. Writing this in a more intrinsic fashion, we deduce that,
\begin{equation}\label{eq:transport_1}
\begin{split}
\mathrm{d}_{\alpha^*} \zeta_{j,0}\p{\alpha, \mathrm{d}_{\alpha} \Phi_{TS}^\circ(\alpha,\beta)} \cdot \mathrm{d}_{\alpha} \sigma_A(\alpha,\beta) + & i F_j(\alpha,\beta) \sigma_A(\alpha,\beta)\\
		 + h \jap{\va{\alpha}}^{-1}P_j & \sigma_A(\alpha,\beta) = \O_{C^\infty}\p{ \p{\frac{h}{\jap{\va{\alpha}}}}^\infty}
\end{split}
\end{equation}
Here, it makes sense to apply $\mathrm{d}_{\alpha^*} \zeta_{j,0}$ to $\mathrm{d}_\alpha \sigma_A$ since they are elements respectively of $(T^*_\alpha \Lambda) ^*$ and $T^*_\alpha \Lambda$.

Now we will find $n$ others equations. Indeed, since $A\Pi_\Lambda^\ast = A$, we have $\Pi_\Lambda A^\ast = A^\ast$, so that \eqref{eq:transport_1} is also satisfied by $\sigma_{A^\ast}(\alpha,\beta) = \overline{\sigma_A(\beta,\alpha)}$. We can rewrite our equations in the form ($j=1\dots 2n$)
\[
X_j \sigma_A + F_j \sigma_A + h \jap{\va{\alpha}}^{-1} Q_j \sigma_A = \O\p{\p{\frac{h}{\jap{\va{\alpha}}}}^\infty}.
\]
The $X_j$'s are complex vector fields on $\Lambda\times\Lambda$ (defined near the diagonal), i.e. sections of $T\p{\Lambda \times \Lambda} \otimes \C$. Using again Lemma \ref{lemma:transversality}, as in the proof of Lemma \ref{lmcarphase}, we deduce that they generate a space which is supplementary to the complexified tangent space to the diagonal. The $Q_j$'s have the same general properties than the $P_j$'s. Denoting $\gamma = \alpha -\beta$ and $\delta = \alpha+\beta$ in local coordinates, we deduce that
\begin{equation}\label{eq:transport_final}
\nabla_\gamma \sigma_A = L(\gamma,\delta) \nabla_\delta \sigma_A(\alpha,\beta) + \sigma_A(\gamma,\delta) N(\gamma,\delta) + h \jap{\va{\alpha}}^{-1} Q\sigma_A(\gamma,\delta)+ \O\p{\p{\frac{h}{\jap{\va{\alpha}}}}^\infty}.
\end{equation}
Here, $L$ and $N$ are respectively a matrix and a vector that satisfy symbolic estimates in $\gamma,\delta$. The operator $Q$ satisfies the same kind of estimates as the $P_j$'s or the $Q_j$'s.

Our assumption on $\sigma_A$ implies that all its derivatives (of any order) with respect to $\delta$ are $\O((h/\jap{\va{\alpha}})^\infty)$ on the diagonal of $\Lambda \times \Lambda$. Hence, differentiating \eqref{eq:transport_final} any number of times and then restricting to the diagonal, we find by induction that the derivatives to all order (with respect to both $\gamma$ and $\delta$) of $\sigma_A$ are $\O(h/\jap{\va{\alpha}})$ on the diagonal.

However, from this newly acquired knowledge, we find that the derivatives of $Q \sigma_A(\alpha,\beta)$ are in fact $\O( h / \jap{\va{\alpha}})$ on the diagonal. We can consequently play the same game to find that the derivatives to all order (with respect to both $\gamma$ and $\delta$) of $\sigma_A$ are in fact $\O((h/\jap{\va{\alpha}})^2)$ on the diagonal. Then, we keep iterating these procedures and we find by induction that all the derivatives of $\sigma_A$ are $\O((h/ \jap{\va{\alpha}})^{\infty})$ on the diagonal of $\Lambda \times \Lambda$. This ends the proof of the lemma.

\end{proof}

\subsection{Orthogonal projector and Toeplitz calculus}\label{sec:Toeplitz}

We are now ready to deduce useful consequences from the analysis of \S \ref{sec:FIO}. First of all, we want to study the orthogonal projector $B_\Lambda$ on $\mathcal{H}_{\Lambda,\FBI}^0$ in $L_0^2\p{\Lambda}$. We start by showing that:

\begin{lemma}\label{lemma:presque_projecteur}
There is a real-valued symbol $f \in S_{KN}^0\p{\Lambda}$ positive and elliptic of order $0$ such that $\Pi_\Lambda f \Pi_{\Lambda}^*$ and $\p{\Pi_\Lambda f \Pi_{\Lambda}^*}^2$ differ by a negligible operator.
\end{lemma}

\begin{proof}
From Lemmas \ref{lemma:phase_left_multiplication}, \ref{lemma:phase_right_multiplication} and \ref{lmcarphase}, we know that the phases of $\Pi_\Lambda f \Pi_\Lambda^*$ and its square coincide with the phase $\Phi_\Lambda$ from Remark \ref{remark:phi_lambda} at infinite order on the diagonal. Hence, we may assume that the phases of $\Pi_\Lambda f \Pi_\Lambda^*$ and its square are exactly $\Phi_\Lambda$. We will be in position to apply Lemma \ref{lemma:de_la_diagonale_le_symbole}  if we can find $f$ such that the symbols of $\Pi_\Lambda f \Pi_\Lambda^*$ and its square coincides on the diagonal up to an $\O((h/\jap{\va{\alpha}})^\infty)$. We want to construct $f$ with an asymptotic expansion
\begin{equation*}
\begin{split}
f \sim \sum_{k \geq 0} h^k f_k.
\end{split}
\end{equation*}
Then, the symbols $c$ and $\tilde{c}$ of $\Pi_\Lambda f \Pi_{\Lambda}^*$ and $\p{\Pi_\Lambda f \Pi_{\Lambda}^*}^2$ respectively have asymptotic expansions (in the sense of Remark \ref{remark:expansion_symbole_FIO})
\begin{equation}\label{eq:asymptotic_B}
\begin{split}
c \sim \frac{1}{(2 \pi h)^n} \sum_{k \geq 0} h^k c_k, \quad \tilde{c} \sim \frac{1}{(2 \pi h)^n} \sum_{k \geq 0} h^k \tilde{c}_k
\end{split}
\end{equation}
that are deduced from the expansion for $f$ (recall Lemma \ref{lemma:composition_oscillating_kernel}). In order to apply Lemma \ref{lemma:de_la_diagonale_le_symbole}, we want to choose $f$ such that the $c_k$ and $\tilde{c}_k$ coincide on the diagonal of $\Lambda \times \Lambda$ for every $k \geq 0$.

From the application of the method of stationary phase in the proof of Lemma \ref{lemma:composition_oscillating_kernel}, we find that $c_0(\alpha,\alpha) = f_0(\alpha) g(\alpha,\alpha)$, where $g(\alpha,\beta)$ a symbol of order $0$, whose restriction to the diagonal is positive and elliptic (we recall that the order of $g$ is discussed in the proof of Lemma \ref{lemma:composition_oscillating_kernel}). In general,
\begin{equation}\label{eq70}
\begin{split}
c_k\p{\alpha,\alpha} = g(\alpha,\alpha) f_k(\alpha) + A_k(f_0,\dots,f_{k-1})(\alpha,\alpha),
\end{split}
\end{equation}
where $A_k(f_0,\dots,f_{k-1})$ is a symbol of order $-k$ that only depends on $f_0,\dots,f_{k-1}$. On the other hand, we also get ($\tilde{g}$ is a symbol of order $0$ with properties that are similar to those of $g$)
\begin{equation*}
\begin{split}
\widetilde{c}_0(\alpha,\alpha) = \tilde{g}(\alpha,\alpha)c_0(\alpha,\alpha)^2,
\end{split}
\end{equation*}
and for $k \geq 1$
\begin{equation*}
\begin{split}
\widetilde{c}_k(\alpha,\alpha) = 2 \tilde{g}(\alpha,\alpha) c_0(\alpha,\alpha) c_k(\alpha,\alpha) + B_k(c_0,\dots,c_{k-1})(\alpha,\alpha),
\end{split}
\end{equation*}
where $B_k(c_0,\dots,c_{k-1})(\alpha,\alpha)$ is a symbol of order $-k$ depending only on $c_0,\dots,c_{k-1}$.

Consequently by choosing $f_0(\alpha) = (\tilde{g}(\alpha,\alpha) g(\alpha,\alpha))^{-1}$, we ensure that $c_0$ and $\tilde{c}_0$ coincide on the diagonal (as well as the ellipticity and positivity of $f$, provided that $h$ is small enough). This imposes for $k \geq 0$
\[
c_k(\alpha,\alpha) = - B_k(c_0,\dots,c_{k-1})(\alpha,\alpha).
\]
In turn, we get
\begin{equation}\label{eq:deffk}
f_k(\alpha)= - g(\alpha,\alpha)^{-1}\p{A_k(f_0,\dots,f_{k-1})(\alpha,\alpha) + B_k(c_0,\dots,c_{k-1})(\alpha,\alpha)}.
\end{equation}
Now, by induction, we prove that the $f_k$'s may be chosen real valued. We gave an explicit expression for $f_0$ that ensures that it is real-valued (and positive). Assume that $f_0,\dots, f_k$ are all real-valued, then $\Pi_\Lambda^\ast(f_0 + \dots + h^k f_k) \Pi_\Lambda$ is self-adjoint. In particular, if $K$ denotes the reduced kernel of this operator we have $K(\alpha,\beta) = \overline{K(\beta,\alpha)}$. Recall that the phase of $\Pi_\Lambda^\ast(f_0 + \dots + h^k f_k) \Pi_\Lambda$ is $\Phi_\Lambda$ and denote by $\sigma$ its symbol. Then we have, up to negligible terms,
\begin{equation*}
\begin{split}
K(\alpha,\beta) & = \frac{K(\alpha,\beta) + \overline{K(\beta,\alpha)}}{2} = e^{i \Phi_\Lambda(\alpha,\beta)} \frac{\sigma(\alpha,\beta) + \overline{\sigma(\beta,\alpha)}}{2}.
\end{split}
\end{equation*}
Here, we recall that $\Phi_\Lambda(\alpha,\beta) = - \overline{\Phi_\Lambda(\beta,\alpha)}$. We may consequently assume that $\sigma(\alpha,\beta) = \overline{\sigma(\beta,\alpha)}$. This implies in particular that $A_{k+1}(f_0,\dots, f_k)(\alpha,\alpha)$ is real. Now, $(\Pi_\Lambda^\ast(f_0 + \dots + h^k f_k) \Pi_\Lambda)^2$ will also be self-adjoint, and we can use the same argument to ensure that $B_{k+1}(c_0, \dots, c_k)(\alpha,\alpha)$ is also real. Using then \eqref{eq:deffk} to define $f_{k+1}$, we find that it is real-valued.
\end{proof}

We can now make explicit the structure of the orthogonal projector $B_\Lambda$.

\begin{lemma}\label{lemma:approximate_projector}
The orthogonal projector $B_\Lambda$ is a FIO. More precisely, if $f$ is as in Lemma \ref{lemma:presque_projecteur}, then $B_\Lambda$ and $\Pi_\Lambda f \Pi_\Lambda^*$ differ by a negligible operator.
\end{lemma}

\begin{proof}
Let $f$ be as in Lemma \ref{lemma:presque_projecteur} and set $B_\Lambda^\circ = \Pi_\lambda f\Pi_{\Lambda}^*$. Notice that, since $f$ is real valued, the operator $B_\Lambda^\circ$ is self-adjoint. Moreover, since $f_0$ is positive and elliptic, for $h$ small enough we have $f \geq C^{-1}$ for some $C > 0$, and hence, if $h$ is small enough, we have for $u \in \mathcal{H}_{\Lambda,\FBI}^0$ (the scalar product is in $L^2_0\p{\Lambda}$)
\begin{equation}\label{eq74}
\begin{split}
\langle B_\Lambda^\circ u ,u \rangle & = \langle f \Pi_\Lambda^* u , \Pi_\Lambda^* u \rangle \geq \frac{1}{C} \n{\Pi_\Lambda^*u}^2 \\
    & \geq \frac{1}{C} \p{\frac{\n{\Pi_\Lambda^*u} \n{u}}{\n{u}}}^2 \geq \frac{1}{C} \frac{\langle \Pi_\Lambda^* u,u\rangle^2}{\n{u}^2} \\
    & \geq \frac{1}{C}\p{\frac{\langle u, \Pi_\Lambda u\rangle}{\n{u}}}^2 = \frac{1}{C}\n{u}^2.
\end{split}
\end{equation}
Hence, the spectrum of $B_\Lambda^\circ$ consists of $0$ and a bounded subset $E$ of $\R_+^*$, bounded from below by $C^{-1}$. This is because the image of $B_\Lambda^\circ$ is contained in $\mathcal{H}_{\Lambda,\FBI}^0$ and $B_\Lambda^\circ$ vanishes on the orthogonal of $\mathcal{H}_{\Lambda,\FBI}^0$. Then, we may choose a loop $\gamma$ around $E$ but not around $0$, and define the spectral projector
\begin{equation*}
\begin{split}
\widetilde{B}_\Lambda = \frac{1}{2 i \pi} \int_{\gamma} \p{ z - B_\Lambda^\circ}^{-1} \mathrm{d}z.
\end{split}
\end{equation*}
We start by showing that $B_\Lambda^\circ$ approximates $\widetilde{B}_\Lambda$. To do so, notice that if $z \in \gamma$ then we have
\begin{equation*}
\begin{split}
\p{z - B_\Lambda^\circ} \p{ \frac{1}{z} + \frac{1}{z(z-1)} B_{\Lambda}^\circ} = 1 + \frac{1}{z(z-1)} \p{B_\Lambda^\circ -(B_\Lambda^\circ)^2}.
\end{split}
\end{equation*}
When $h$ is small enough, applying Lemma \ref{lemma:presque_projecteur}, we may invert the operator $1 + (z(z-1))^{-1} \p{B_\Lambda^\circ - (B_\Lambda^\circ)^2}$ by mean of Von Neumann series (uniformly for $z \in \gamma$). Thus, we have
\begin{equation*}
\begin{split}
\p{1 + \frac{1}{z(z-1)} \p{B_\Lambda^\circ - (B_\Lambda^\circ)^2}}^{-1} = 1 + R(z),
\end{split}
\end{equation*}
where $R(z)$ is a negligible operator (uniformly in $z \in \gamma$). Moreover, $R(z)$ is valued in $\mathcal{H}_{\Lambda,\FBI}^0$ and commute with $B_\Lambda^\circ$. Consequently, we may write
\begin{equation}\label{eq:premiere_approximation}
\begin{split}
\widetilde{B}_{\Lambda} & = \frac{1}{2 i \pi} \int_{\gamma} \p{\frac{1}{z} + \frac{1}{z(z-1)} B_\Lambda^\circ} \p{1 + R(z)} \mathrm{d}z \\
    & = B_\Lambda^\circ  + \frac{1}{2 i \pi} \int_{\gamma} R(z) \p{\frac{1}{z} + \frac{1}{z(z-1)} B_\Lambda^\circ} \mathrm{d}z.
\end{split}
\end{equation}
Since $R(z)$ is negligible, if we prove that $\widetilde{B}_\Lambda = B_\Lambda$, we will be done. However, $\widetilde{B}_\Lambda$ is an orthogonal projector by construction since $B_\Lambda^\circ$ is self-adjoint. Additionally, since it is the projection on the non-zero part of the spectrum of $B_\Lambda^\circ$, the spectral theorem ensures that $B_\Lambda^\circ u= 0$ if and only if $\widetilde{B}_\Lambda u= 0$. It follows that $\ker \widetilde{B}_\Lambda$ is the orthogonal of $\mathcal{H}_{\Lambda,\FBI}^0$ (by \eqref{eq74}), and thus that $\widetilde{B}_\Lambda = B_\Lambda$.
\end{proof}

\begin{remark}\label{remark:boundedness}
Recalling Remark \ref{remark:fourre_tout}, we see that for every $k \in \R$ the operator $B_\Lambda$ is bounded from $L^2_k\p{\Lambda}$ to $\mathcal{H}_{\Lambda,\FBI}^k$. Consequently, if $\sigma$ is a symbol of order $m$ on $\Lambda$ and $k \in \R$, the operator $B_\Lambda \sigma B_\Lambda$ is bounded from $L^2_k\p{\Lambda}$ to $\mathcal{H}_{\Lambda,\FBI}^{k-m}$.
\end{remark}

Before going further into the study of Toeplitz calculus, let us mention that Lemma \ref{lemma:approximate_projector} allows to identify the dual of the spaces $\mathcal{H}_\Lambda^k$.

\begin{lemma}\label{lemma:dual}
Assume that $\tau_0$ and $h$ are small enough and let $k \in \R$. Then, if $u \in \mathcal{H}_\Lambda^{-k}$ and $v \in \mathcal{H}_\Lambda^{k}$ the pairing
\begin{equation}\label{eq:un_autre_pairing}
\begin{split}
\jap{u,v} \coloneqq \jap{T_\Lambda u , T_\Lambda v}_{L^2_0\p{\Lambda}}
\end{split}
\end{equation}
is well-defined and induces an (anti-linear) identification between $\mathcal{H}_\Lambda^{-k}$ and the dual of $\mathcal{H}_\Lambda^k$ (inducing equivalent, but \emph{a priori} not equal norms).
\end{lemma}

\begin{proof}
It is clear from the definition of the spaces $\mathcal{H}_\Lambda^k$ and $\mathcal{H}_\Lambda^{-k}$ that the pairing \eqref{eq:un_autre_pairing} is well-defined and induces a bounded anti-linear map from $\mathcal{H}_\Lambda^{-k}$ to the dual of $\mathcal{H}_\Lambda^k$. Let us denote this map by $A$. We start by proving that $A$ is surjective. Let $\ell$ be a continuous linear form on $\mathcal{H}_\Lambda^{k}$ and let $\tilde{\ell}$ be the linear form on $L^2_k(\Lambda)$ defined by
\begin{equation*}
\begin{split}
\tilde{\ell}(u) = \ell\p{S_\Lambda u}.
\end{split}
\end{equation*}
Recall that $\Pi_\Lambda = T_\Lambda S_\Lambda$ is bounded on $L^2_k(\Lambda)$ by Proposition \ref{lemma:boundedness}, and hence that $S_\Lambda$ is bounded from $L^2_k(\Lambda)$ to $\mathcal{H}_\Lambda^k$. Consequently, $\tilde{\ell}$ is a well-defined and continuous linear form on $L^2_k(\Lambda)$, and thus there exists $v_\ell \in L^2_{-k}(\Lambda)$ such that for every $u$ in the space $L^2_k(\Lambda)$ we have
\begin{equation*}
\begin{split}
\tilde{\ell}(u) = \jap{v_\ell,u}_{L^2_0\p{\Lambda}}.
\end{split}
\end{equation*}
Hence, if $u \in \mathcal{H}_\Lambda^k$ we have
\begin{equation*}
\begin{split}
\ell(u) & = \tilde{\ell}\p{T_\Lambda u} = \jap{v_\ell,T_\Lambda u}_{L^2_0\p{\Lambda}} \\
        & = \jap{T_\Lambda S_\Lambda B_\Lambda v_\ell,T_\Lambda u}_{L^2_0\p{\Lambda}}.
\end{split}
\end{equation*}
Thus, we have $A(S_\Lambda B_\Lambda v_\ell) = \ell$ and $A$ is surjective. Here, we used the fact that $B_\Lambda$ is bounded on $\mathcal{H}_{\Lambda,\FBI}^{-k}$ (see Remark \ref{remark:boundedness}).

The injectivity of $A$ follows from its surjectivity by duality (the result is symmetric in $k$ and $-k$), and thus $A$ is an isomorphism by the Open Mapping Theorem.
\end{proof}

We want now to prove the ``multiplication formula'' (Proposition \ref{propmultform}), that will be the main tool to study the spectral properties of an Anosov vector field acting on a space of anisotropic ultradistributions. To do so we first prove a ``Toeplitz property'' for $\G^s$ pseudors.

\begin{prop}[Toeplitz property]\label{proptoeplitz}
Let $s \geq 1$. Let $P$ be a $\G^s$ pseudor of order $m$. Assume that $\Lambda$ is a $(\tau_0, s)$-Gevrey adapted Lagrangian with $\tau_0$ small enough. Let $f \in S_{KN}^{\tilde{m}}\p{\Lambda}$ be a symbol of order $\tilde{m}$ on $\Lambda$ with uniform estimates when $h \to 0$. Then there is a symbol $\sigma \in S_{KN}^{m+\tilde{m}}\p{\Lambda}$ of order $m+m'$ on $\Lambda$ such that, for $h$ small enough, $B_\Lambda f T_\Lambda P S_\Lambda B_\Lambda$ and $B_\Lambda \sigma B_\Lambda$ differ by a negligible operator.

Moreover, $\sigma$ is given at first order by
\begin{equation*}
\begin{split}
\sigma(\alpha) = f(\alpha) p_\Lambda(\alpha) \textup{ mod } h S^{m+\tilde{m} - 1}_{KN}\p{\Lambda},
\end{split}
\end{equation*}
where $p_\Lambda$ denotes the restriction to $\Lambda$ of an almost analytic extension of the principal symbol of $p$ (see Remark \ref{remark:aae_for_symbol}).
\end{prop}

\begin{proof}[Proof of Proposition \ref{proptoeplitz}]
It follows from Lemma \ref{lemma:composition_oscillating_kernel} that the operators $B_\Lambda f T_\Lambda P S_\Lambda B_\Lambda$ and $B_\Lambda \sigma B_\Lambda$ both are FIOs. Moreover, according to Lemmas \ref{lemma:phase_left_multiplication}, \ref{lemma:phase_right_multiplication} and \ref{lmcarphase}, they share the same phase and it is $\Phi_\Lambda$ (from Remark \ref{remark:phi_lambda}). Finally, $B_\Lambda f T_\Lambda P S_\Lambda B_\Lambda$ and $B_\Lambda \sigma B_\Lambda$ inherit from $B_\Lambda$ the left invariance by $\Pi_\Lambda$ and the right invariance by $\Pi_\Lambda^*$. Consequently, according from Lemma \ref{lemma:de_la_diagonale_le_symbole}, we only need to choose $\sigma$ so that the symbols of $B_\Lambda f T_\Lambda P S_\Lambda B_\Lambda$ and $B_\Lambda \sigma B_\Lambda$ agree up to $\O((h/\jap{\va{\alpha}})^\infty)$ on the diagonal of $\Lambda \times \Lambda$.

As in the proof of Lemma \ref{lemma:presque_projecteur}, we will construct $\sigma \in S_{KN}^{m + \tilde{m}}\p{\Lambda}$ with an asymptotic expansion
\begin{equation*}
\begin{split}
\sigma \sim \sum_{k \geq 0} h^k \sigma_k.
\end{split}
\end{equation*}
The $\sigma_k \in S_{KN}^{m + \tilde{m} - k}\p{\Lambda}$'s will be constructed by induction. It follows from Lemma \ref{lemma:composition_oscillating_kernel} that the symbol $\tilde{\sigma} \in h^{-n} S_{KN}^{m + \tilde{m}}\p{\Lambda \times \Lambda}$ of $B_\Lambda \sigma B_\Lambda$ has an asymptotic expansion (in the sense of Remark \ref{remark:expansion_symbole_FIO})
\begin{equation*}
\begin{split}
\tilde{\sigma} \sim \frac{1}{(2 \pi h)^n} \sum_{k \geq 0} h^k \tilde{\sigma}_k.
\end{split}
\end{equation*}
Here, we see that we are using the same stationary phase argument than when computing $B_\Lambda^2 = B_\Lambda$, the only difference being the multiplication by $\sigma$ in between. Consequently, we find that $\tilde{\sigma}_0$ is given on the diagonal by
\begin{equation*}
\begin{split}
\tilde{\sigma}_0(\alpha,\alpha) = c_0(\alpha,\alpha) \sigma_0(\alpha),
\end{split}
\end{equation*}
where $c_0$ is the first term in the expansion \eqref{eq:asymptotic_B} for the symbol $c$ of $B_\Lambda$. More generally, we have
\begin{equation*}
\begin{split}
\tilde{\sigma}_k(\alpha,\alpha) = c_0(\alpha,\alpha) \sigma_k(\alpha) + C_k(\sigma_0,\dots,\sigma_{k-1})(\alpha,\alpha)
\end{split}
\end{equation*}
with $C_k(\sigma_0,\dots,\sigma_{k-1})$ a symbol of order $m + \tilde{m} - k$. Hence, we only need to take
\begin{equation*}
\begin{split}
\sigma_0(\alpha) = \frac{(2 \pi h)^n e(\alpha,\alpha)}{c_0(\alpha,\alpha)},
\end{split}
\end{equation*}
where $e$ denotes the symbol of $B_\Lambda f T_\Lambda P S_\Lambda B_\Lambda$, and for $k \geq 1$,
\begin{equation*}
\begin{split}
\sigma_k(\alpha) = - \frac{C_k(\sigma_0,\dots,\sigma_{k-1})(\alpha,\alpha)}{c_0(\alpha,\alpha)}.
\end{split}
\end{equation*}
This is possible due to the ellipticity of $c_0(\alpha,\alpha)$.

It remains to check the first order asymptotic for $\sigma$. To do so, we only need to see that $e$ is given at first order on the diagonal by $c_0(\alpha,\alpha) f(\alpha) p_\Lambda(\alpha)$. Recall that $\Pi_\Lambda$ and $T_\Lambda P S_\Lambda$ share the same phase. Moreover,it follows from Lemma \ref{lemma:TPS-close-diagonal} that the symbol of $T_\Lambda P S_\Lambda$ differ from the symbol of $\Pi_\Lambda$ by a factor $p_\Lambda(\alpha)$, on the diagonal and in first order approximation. Hence, the symbols of $B_\Lambda = B_\Lambda \Pi_\Lambda B_\Lambda$ and $B_\Lambda f T_\Lambda P S_\Lambda B_\Lambda$ differ in first order approximation on the diagonal by a factor $f p_\Lambda$.
\end{proof}

We state now what will be the key tool in the spectral analysis of $P$ in the next chapter. Notice that the multiplication formula, Proposition \ref{propmultform}, is a stronger statement than Theorem \ref{thm:deforming-Gs-pseudors}.

\begin{prop}[Multiplication formula]\label{propmultform}
Let $s \geq 1$ and $P$ be a $\G^s$ pseudor of order $m$. Assume that $\Lambda$ is a $(\tau_0,s)$-Gevrey adapted Lagrangian with $\tau_0$ small enough. Let $p_\Lambda$ denotes the restriction to $\Lambda$ of an almost analytic extension of the principal symbol of $P$. Assume that $h$ is small enough. Let $f \in S_{KN}^{\tilde{m}}\p{\Lambda}$ be a symbol of order $\tilde{m}$ on $\Lambda$, uniformly in $h$. Then, if $m_1,m_2 \in \R$ are such that $m_1 + m_2 = m + \tilde{m}-1$, there is a constant $C> 0$ such that for any $u,v \in \mathcal{H}_\Lambda^\infty$, we have (the scalar product is in $L^2_0\p{\Lambda}$)
\begin{equation}\label{eqmultform}
\begin{split}
\va{\langle f T_\Lambda P u , T_\Lambda v \rangle - \langle f p_\Lambda T_\Lambda u ,T_\Lambda v \rangle} \leq C h \n{u}_{\mathcal{H}_{\Lambda}^{m_1}} \n{v}_{\mathcal{H}_{\Lambda}^{m_2}}.
\end{split}
\end{equation}
\end{prop}

\begin{proof}
First, notice that from Proposition \ref{lemma:boundedness} we know that $Pu$ belongs to $\mathcal{H}_\Lambda^\infty$, so that it makes sense to write
\begin{equation*}
\begin{split}
\langle f T_\Lambda P u , T_\Lambda v \rangle = \langle B_\Lambda f T_\Lambda P S_\Lambda B_\Lambda T_\Lambda u , T_\Lambda v \rangle.
\end{split}
\end{equation*}
Then from Proposition \ref{proptoeplitz}, we can find $\sigma$ a symbol on $\Lambda$ such that
\begin{equation*}
B_\Lambda f T_\Lambda P S_\Lambda B_\Lambda = B_\Lambda \sigma B_\Lambda  + \O_{\mathcal{H}_{\Lambda,\FBI}^{-N} \to \mathcal{H}_{\Lambda,\FBI}^N}(h^N),
\end{equation*}
Let $\sigma_0 = f p_\Lambda$ denote the first order approximation of $\sigma$ given by Proposition \ref{proptoeplitz}. Since $\sigma - \sigma_0 = \O(h \jap{\va{\alpha}}^{m+\tilde{m}-1})$, it is clear that $B_\Lambda(\sigma - \sigma_0) B_\Lambda$ is a $\O(h)$ as a bounded operator from $\mathcal{H}_{\Lambda,\FBI}^{m_1}$ to $\mathcal{H}_{\Lambda,\FBI}^{-m_2}$, recalling that $m_1 - m - \tilde{m} +1 = -m_2$. Hence, we see that
\begin{equation*}
\begin{split}
B_\Lambda f T_\Lambda P S_\Lambda B_\Lambda = B_\Lambda f p_\Lambda B_\Lambda + \O_{\mathcal{H}_{\Lambda,\FBI}^{m_1} \to H_{\Lambda,\FBI}^{-m_2}}\p{h},
\end{split}
\end{equation*}
and the result readily follows.
\end{proof}

\begin{remark}
In the applications, we will not use the asymptotic $h \to 0$ that we stated in Proposition \ref{propmultform}, but only the asymptotic for $\alpha$ large, with $h>0$ fixed small enough. However, we will still need $h$ to be small enough to make the machinery work.
\end{remark}

Using the same kind of arguments as in the proof of Proposition \ref{proptoeplitz}, we may prove the following statements about composition of Toeplitz operators. This is left as an exercise to the reader.

\begin{prop}[Composition of Toeplitz operators]\label{propcomptoep}
Let $\sigma_1 \in S_{KN}^{m_1}\p{\Lambda}$ and $\sigma_2 \in S_{KN}^{m_2}\p{\Lambda}$. Then there is a symbol $\sigma_1 \# \sigma_2 \in S_{KN}^{m_1 + m_2}\p{\Lambda}$ such that $ \p{B_\Lambda \sigma_1 B_\Lambda} \circ \p{B_\Lambda \sigma_2 B_\Lambda}$ and $B_\Lambda \sigma_1 \# \sigma_2 B_\Lambda$ differ by a negligible operator. Moreover, $\sigma_1 \# \sigma_2$ coincides with the product $\sigma_1 \sigma_2$ up to $h S_{KN}^{m_1 + m_2-1}\p{\Lambda}$.
\end{prop}

In prevision of the next section, we state and prove a last result about this Toeplitz calculus.

\begin{lemma}\label{lemma:trace_pseudo}
Let $\sigma \in S_{KN}^{m}\p{\Lambda}$ be a symbol of order $m< -n$ on $\Lambda$. Then the operator $B_\Lambda \sigma B_\Lambda$ from $\mathcal{H}_{\Lambda,\FBI}^0$ to itself is trace class, with trace class norm controlled by a semi-norm of $\sigma$ (in the class of symbols of order $m$). Moreover, the trace of $B_\Lambda \sigma B_\Lambda$ is the integral of its kernel on the diagonal of $\Lambda \times \Lambda$.
\end{lemma}

\begin{proof}
The first observation is that if an operator $R$ has a kernel satisfying
\begin{equation}\label{eq:trace-remainder-estimate}
R(\alpha,\beta)e^{\frac{H(\beta)-H(\alpha)}{h}} = \O_{\mathcal{C}^N}( L (\langle|\alpha|\rangle+\langle|\beta|\rangle)^{-N}),
\end{equation}
for $N$ large enough, then it is trace class, with trace norm $\O(L)$. This follows from usual techniques. For example, see \cite[Section 9]{Dimassi-Sjostrand-book}.

Let $\sigma_0$ be the symbol $\sigma_{0}\p{\alpha} = \jap{\va{\alpha}}^{\frac{m}{2}}$ on $\Lambda$. By a parametrix construction ($\sigma_0$ is elliptic), we find a symbol $\sigma_1$ of order $m/2$ such that
\begin{equation*}
\begin{split}
B_\Lambda \sigma B_\Lambda = \p{B_\Lambda \sigma_0 B_\Lambda} \circ \p{B_\Lambda \sigma_1 B_\Lambda} + R,
\end{split}
\end{equation*}
where $R$ is an operator with a smooth kernel that satisfies \eqref{eq:trace-remainder-estimate}. We only proceed at a finite number of steps in the parametrix construction in order to be able to control the derivatives of the kernel of $R$ by some semi-norm of $\sigma$. However, by proceeding at a large enough number of steps in the parametrix construction, we may obtain an estimate of the form \eqref{eq:trace-remainder-estimate} for an arbitrary large $N$. In particular, we control the trace class norm of $R$ by some semi-norm of $\sigma$.

In order to prove that $B_\Lambda \sigma B_\Lambda$ is trace class, we want now to prove that $B_\Lambda \sigma_0 B_\Lambda$ and $B_\Lambda \sigma_1 B_\Lambda$ are Hilbert--Schmidt. To do so, we only need to prove that their reduced kernels are square integrable. When investigating the kernel of $B_\Lambda \sigma_1 B_\Lambda$ by mean of the non-stationary phase method (as in the proof of Lemma \ref{lemma:composition_oscillating_kernel}), one only needs estimates on a finite number of derivatives of $\sigma_1$ to obtain an estimate of the form
\begin{equation}\label{eq:HilbSchmid}
\begin{split}
B_\Lambda \sigma_1 B_\Lambda(\alpha,\beta)e^{\frac{H(\beta) - H(\alpha)}{h}} = e^{i \frac{\Phi_\Lambda(\alpha,\beta)}{h}} c(\alpha,\beta) + \O_{\mathcal{C}^N}\big(\p{\jap{\va{\alpha}} + \jap{\va{\beta}}}^{-N}\big),
\end{split}
\end{equation}
where $N$ may be taken as large as we want. In particular, the remainder in \eqref{eq:HilbSchmid} is square integrable. Here, $c(\alpha,\beta)$ is a symbol of order $m/2$ supported near the diagonal. It follows that the kernel of $B_\Lambda \sigma_1 B_\Lambda$ is square integrable (using that $\Phi_{\Lambda}$ satisfies the coercivity condition \eqref{eq:coercivity_condition} and working as in the proof of Proposition \ref{lemma:boundedness}). Hence, the kernel of $B_\Lambda \sigma_1 B_\Lambda$ is square integrable (with norm controlled by a semi-norm of $\sigma_1$ and hence of $\sigma$). The operator $B_\Lambda \sigma_0 B_\Lambda$ is Hilbert--Schmidt too for the same reason.

We just proved that $B_\Lambda \sigma B_\Lambda$ is trace class. To prove that its trace is the integral of its kernel on the diagonal of $\Lambda \times \Lambda$, just approximate $\sigma$ by compactly supported kernel.
\end{proof}

\subsection{Estimate on singular values}\label{subsec:singuar_values}

The aim of this section is to prove the following proposition. It will be used in the applications to deduce the Schatten property of the resolvent from the ellipticity of the generator of the flow (see Lemma \ref{lmhypoell} and Theorem \ref{thm:spectral-theory}).

\begin{prop}\label{propvalsing}
Assume that $\Lambda$ is a $\p{\tau_0,1}$-adapted Lagrangian with $\tau_0$ small enough. Assume that $h$ is small enough. Let $m > 0$ and $q \in \R$, then the inclusion of $\mathcal{H}_{\Lambda}^{m+q}$ into $\mathcal{H}_{\Lambda}^q$ is compact. In addition, if $\p{\mu_k}_{k \in \N}$ denotes the singular values of this inclusion then we have
\begin{equation*}
\begin{split}
\mu_k \underset{k \to + \infty}{=} \O\p{\frac{1}{k^{\frac{m}{n}}}}.
\end{split}
\end{equation*}
In particular, it belongs to the Schatten class $\mathcal{S}_p$ for any $p > n/m$.
\end{prop}

Notice that we get an equivalent statement if we replace $\mathcal{H}_\Lambda^{q}$ and $\mathcal{H}_\Lambda^{m+q}$ in Proposition \ref{propvalsing} respectively by $\mathcal{H}_{\Lambda,\FBI}^{q}$ and $\mathcal{H}_{\Lambda,\FBI}^{m+q}$. We will rather consider this case in the following. That the inclusion is compact is an easy consequence of the fact that it may be written as the restriction of the operator with smooth kernel $\Pi_\Lambda$ to $\mathcal{H}_{\Lambda,\FBI}^{m+q}$, the proof is left as an exercise to the reader. The following result will be useful in the proof of Proposition \ref{propvalsing}.

\begin{lemma}\label{lmreduc}
Assume that $\Lambda$ is a $\p{\tau_0,1}$-adapted Lagrangian with $\tau_0$ small enough. Assume that $h$ is small enough. Let $m > 0$. Define the unbounded operator $A = B_\Lambda \jap{\va{\alpha}}^m B_\Lambda$ on $\mathcal{H}_{\Lambda,\FBI}^0$ with domain $\mathcal{H}_{\Lambda,\FBI}^m$. The operator $A$ is closed, self-adjoint, positive, and has compact resolvent. Moreover, $0$ does not belong to the spectrum of $A$, and, if $\p{\lambda_k}_{k \in \N}$ denotes the sequence of eigenvalues of $A$ (ordered increasingly), then there is a constant $C$ such that for all $k \in \N$ we have $\mu_k \leq C \lambda_k^{-1}$ -- where the $\mu_k$'s are from Proposition \ref{propvalsing}.
\end{lemma}

\begin{proof}
Let $\p{u_\ell}_{\ell \in \N}$ be a sequence in $\mathcal{H}_{\Lambda,\FBI}^{m}$ such that $\p{u_\ell}_{\ell \in \N}$ converges to some $u$ in $\mathcal{H}_{\Lambda,\FBI}^{0}$ and $(A u_\ell)_{\ell \in \N}$ converges to some $v$ in $\mathcal{H}_{\Lambda,\FBI}^{0}$. Then, since $A$ is bounded from $\mathcal{H}_{\Lambda,\FBI}^{0}$ to $\mathcal{H}_{\Lambda,\FBI}^{-m}$, the sequence $(A u_\ell)_{\ell \in \N}$ converges to $Au$ in $\mathcal{H}_{\Lambda,\FBI}^{-m}$. Since the convergence in $\mathcal{H}_{\Lambda,\FBI}^{-m}$ or in $\mathcal{H}_{\Lambda,\FBI}^{0}$ implies pointwise convergence (it is a consequence of the structure of $\Pi_\Lambda$), then $Au = v$. Using Proposition \ref{propcomptoep} to construct a parametrix for $A$, we see that since $Au = v$ belongs to $\mathcal{H}_{\Lambda,\FBI}^0$, we have that $u \in \mathcal{H}_{\Lambda,\FBI}^m$, and hence $A$ is closed.

Let us prove that $A$ is self-adjoint. For this, we observe that for $u\in L^2_{-m}\p{\Lambda}$ and $v\in L^2_m\p{\Lambda}$, we have
\begin{equation}\label{eq:B-self-adjoint-general}
\langle B_\Lambda u, v\rangle = \langle u , B_\Lambda v \rangle.
\end{equation}
This follows from the fact that $B_\Lambda$ is bounded on every space $L^2_k\p{\Lambda}$, and the density of these spaces in one another. Now, let $u\in \mathcal{H}_{\Lambda,\FBI}^0$ and notice that $A u \in \mathcal{H}_{\Lambda,\FBI}^{-m}$. In particular, if $v \in \mathcal{H}_{\Lambda,\FBI}^{m}$ then $\langle A u , v \rangle$ makes sense, and
\begin{align*}
\langle u, A v\rangle 	&= \langle u, B_\Lambda \langle|\alpha|\rangle^m B_\Lambda v \rangle \\
						&= \langle \langle|\alpha|\rangle^m u, v \rangle\\
						&= \langle B_\Lambda\langle|\alpha|\rangle^m B_\Lambda u, v \rangle= \langle A u, v\rangle.
\end{align*}
From this, we see that $\Dom(A)\subseteq \Dom(A^\ast)$, and in general, that if $u\in \Dom(A^\ast)$, $Au - A^\ast u$ is orthogonal to all $v\in \mathcal{H}^m_{\Lambda,\FBI}$. But since $\mathcal{H}^m_{\Lambda,\FBI}$ is dense in $\mathcal{H}^{-m}_{\Lambda,\FBI}$, we deduce that in that case, $A u= A^\ast u$, so that $u\in \Dom(A)$.

To see that $A$ is positive, just notice that for $u \in \mathcal{H}_{\Lambda,\FBI}^m\p{\Lambda}$ we have
\begin{equation}\label{eq:inverser_A}
\jap{A u, u} \geq \n{u}_{\mathcal{H}_{\Lambda,\FBI}^{0}}^2
\end{equation}
since $\jap{\va{\alpha}}^m \geq 1$. It also follows that $0$ does not belong to the spectrum of $A$. Using a parametrix construction, we see that the resolvent of $A$ sends $\mathcal{H}_{\Lambda,\FBI}^0$ continuously into $\mathcal{H}_{\Lambda,\FBI}^m$, and is hence compact as an operator from $\mathcal{H}_{\Lambda,\FBI}^0$ to itself.

Let $\p{u_k}_{k \in \N}$ be an orthonormal basis of eigenvectors of $A$, that is $A u_k = \lambda_k u_k$. Then if $u \in \mathcal{H}_{\Lambda,\FBI}^{m}$, we have using Plancherel's formula
\begin{equation*}
\begin{split}
\n{ u - \sum_{k=0}^{N-1} \langle u , u_k \rangle u_k}_{\mathcal{H}_{\Lambda,\FBI}^{0}}^2 & = \sum_{k=N}^{+ \infty} \va{\langle u, u_k\rangle}^2 \\
   & = \sum_{k=N}^{+ \infty} \frac{\va{\jap{u, A u_k}}^2}{\lambda_k^2} \\
   & \leq \lambda_N^{-2} \sum_{k=N}^{+ \infty} \va{\jap{A u ,u_k}}^2 \\
   & \leq \lambda_N^{-2} \n{A u}_{\mathcal{H}_{\Lambda,\FBI}^{0}}^2 \\
   & \leq C^2 \lambda_N^{-2} \n{u}_{\mathcal{H}_{\Lambda,\FBI}^{m}}^2.
\end{split}
\end{equation*}
Then the lemma follows from \cite[Theorem IV.2.5]{Gohb}.
\end{proof}

Before starting the proof of Proposition \ref{propvalsing}, we need to do two reductions.

\begin{reduc}\label{reduc1}
Let $m > 0$ and assume that, for this value of $m$, Proposition \ref{propvalsing} holds for $q=0$. Then, for this particular value of $m$, Proposition \ref{propvalsing} holds for any $q \in \R$.
\end{reduc}

\begin{reduc}\label{reduc2}
Assume that Proposition \ref{propvalsing} holds for all $m \in \left]0,1\right[$ and for $q=0$, then Proposition \ref{propvalsing} holds.
\end{reduc}

\begin{proof}[Proof of Reduction \ref{reduc1}]
Let $q \in \R$. Let us deal first with the case $q > 0$. Define the operator $A = B_\Lambda \jap{\va{\alpha}}^q B_\Lambda$ on $\mathcal{H}_{\Lambda,\FBI}^0$ with domain $\mathcal{H}_{\Lambda,\FBI}^q$. We know from Lemma \ref{lmreduc} that $A$ is a closed self-adjoint operator and that $0$ does not belong to its spectrum. Using Proposition \ref{propcomptoep} to construct a parametrix for $A$, we see that the operator $A^{-1}$ is bounded from $\mathcal{H}_{\Lambda,\FBI}^0$ to $\mathcal{H}_{\Lambda,\FBI}^q$. Hence, if $j$ denotes the inclusion of $\mathcal{H}_{\Lambda,\FBI}^{m+q}$ into $\mathcal{H}_{\Lambda,\FBI}^{q}$ and $j'$ denotes the inclusion of $\mathcal{H}_{\Lambda,\FBI}^{m}$ into $\mathcal{H}_{\Lambda,\FBI}^{0}$, then we have $j = A \circ j' \circ A^{-1}$, and the estimate on the singular values of $j'$ that we assumed carries on to the singular values of $j$.

We deal now with the case $q < 0$. We set $A = B_\Lambda \jap{\va{\alpha}}^{-q} B_\Lambda$. As above, we construct an inverse $A^{-1}$ for $A$ (the operator $A^{-1}$ is \emph{a priori} defined on $\mathcal{H}_{\Lambda,\FBI}^0$) and a parametrix $E$ for $A$ (the parametrix $E$ is defined in particular on $\mathcal{H}_{\Lambda,\FBI}^{m+q}$). Then we have $AE = I - R$ where $R$ is a negligible operator and hence $A^{-1} = E + A^{-1} R$ may be extended to an operator from $\mathcal{H}_{\Lambda,\FBI}^{m+q}$ to $\mathcal{H}_{\Lambda,\FBI}^{m}$, which is still an inverse for $A$ by a density argument. From this point, the proof in this case $q < 0$ is similar to in the case $q > 0$, interchanging $A$ and $A^{-1}$.
\end{proof}

\begin{proof}[Proof of Reduction \ref{reduc2}]
Let $m > 0$. By Reduction \ref{reduc1}, we only need to deal with the case $q=0$. Choose $N$ large enough so that $\frac{m}{N} < 1$, and for $\ell = 0,\dots, N-1$ denotes by $j_{\ell}$ the inclusion of $\mathcal{H}_{\Lambda,\FBI}^{\frac{\p{\ell +1} m}{N}}$ into $\mathcal{H}_{\Lambda,\FBI}^{\frac{\ell m}{N}}$. Then if $j$ denotes the inclusion of $\mathcal{H}_{\Lambda,\FBI}^{m}$ into $\mathcal{H}_{\Lambda,\FBI}^{0}$ we have
\begin{equation*}
\begin{split}
j = j_{N-1} \circ \dots \circ j_0.
\end{split}
\end{equation*}
But our assumption, Reduction \ref{reduc1} and \cite[Theorem IV.2.5]{Gohb} imply that for any $k \in \N$ and $\ell \in \set{0,\dots,N-1}$ there is an operator $L_{k,\ell}$ from $\mathcal{H}_{\Lambda,\FBI}^{\frac{\p{\ell +1} m}{N}}$ to $\mathcal{H}_{\Lambda,\FBI}^{\frac{\ell m}{N}}$ of rank at most $k$ such that the operator norm $j_\ell - L_{k,\ell}$ is less than $C\p{k+1}^{-n/Nm}$ (for some $C > 0$). Hence, the operator norm of
\begin{equation}\label{eqapproxfinrank}
\begin{split}
\p{j_{N-1} - L_{k,N-1}} \circ \dots \circ \p{j_0 - L_{k,0}}
\end{split}
\end{equation}
is less than $C^N (k+1)^{-n/m}$. However, expanding we see that the operator \eqref{eqapproxfinrank} differs from $j$ by an operator of rank at most $\p{2^N - 1}k$. Hence, using \cite[Theorem IV.2.5]{Gohb} again, we have
\begin{equation*}
\begin{split}
\mu_{\p{2^N -1}k} \leq \frac{C^N}{(k+1)^{\frac{n}{m}}}.
\end{split}
\end{equation*}
And the result follows since the sequence $\p{\mu_k}_{k \in \N}$ is decreasing.
\end{proof}

We can now start the proof of Proposition \ref{propvalsing}. According to the reductions above, we may assume $m \in \left]0,1\right[$ and $q=0$. Introduce then the unbounded operator $A = B_\Lambda \jap{\va{\alpha}}^m B_\Lambda$ on $\mathcal{H}_{\Lambda,\FBI}^{0}$ with domain $\mathcal{H}_{\Lambda,\FBI}^m$. By Lemma \ref{lmreduc}, the proof of Proposition \ref{propvalsing} reduces to the following lemma.

\begin{lemma}[Upper Weyl's law for $A$]
Recall that $m \in \left]0,1\right[$ and let $N(r)$ denote the number of eigenvalues less than $r$ of $A$. Then $N(r) \underset{r \to + \infty}{=} \O(r^{n/m})$.
\end{lemma}

\begin{proof}
First, take an integer $N > 0$ such that $n+1 > N m > n$ (this is possible thanks to our assumption on $m$). By Proposition \ref{propcomptoep}, we may write 
\begin{equation*}
\begin{split}
A^N = B_\Lambda \sigma B_\Lambda + R,
\end{split}
\end{equation*}
where $\sigma$ is a symbol of order $Nm$ with leading term $\jap{\va{\alpha}}^{Nm}$ and $R$ is a negligible operator. For $r \geq 0$, define the symbol $a_r$ by $a_r(\alpha) = (r + \jap{\va{\alpha}}^{N m})^{-1}$. Then, using a finite number of terms only in Proposition \ref{propcomptoep}, we find that
\begin{equation*}
\begin{split}
A^N B_\Lambda a_r B_\Lambda = B_\Lambda \jap{\va{\alpha}}^{Nm} a_r B_\Lambda + B_\Lambda \sigma_r B_\Lambda + R_r.
\end{split}
\end{equation*}
We will detail later the properties of $\sigma_r$ and $R_r$. Then we can write
\begin{equation*}
\begin{split}
\p{A^N + r}^{-1} = B_\Lambda a_r B_\Lambda - \p{A^N + r}^{-1} B_\Lambda \sigma_r B_\Lambda - \p{A^N + r}^{-1} R_r.
\end{split}
\end{equation*}
Notice that the $a_r$'s form a family of symbols of order $-Nm$ with uniform estimates on any semi-norm for this symbol class. But if we consider now the $a_r$'s like symbols of order $- Nm +1$ then any semi-norm of $a_r$ in this symbol class is dominated by $r^{-1/Nm }$. Since the symbol $\sigma$ of $A^N$ does not depend on $r$, by taking an expansion with a large enough number of terms, we may ensure that the trace class norm of $R_r$ acting on $\mathcal{H}_{\Lambda,\FBI}^{0}$ is dominated by $r^{-1/Nm}$ ($R_r$ is a continuous function of $a_r$ and $\sigma$). Now, since the quadratic form associated with $A^N$ is positive, a standard argument ensures that the norm of $(A^N + r)^{-1}$  acting on $\mathcal{H}_{\Lambda,\FBI}^{0}$ is less than $1/r$. Hence, the trace class operator norm of $(A^N + r)^{-1} R_r$ is dominated by $r^{-1 - 1/Nm}$. 

In order to control the term $(A^N + r)^{-1} B_\Lambda \sigma_r B_\Lambda$, we notice that $\sigma_r$ is a continuous function of $\sigma$ and $a_r$. As above, if we see $\sigma$ like a symbol of order $Nm$ and $a_r$ like a symbol of order $-Nm +1$ then $\sigma_r$, as a symbol of order $0$, does not grow faster than some $r^{- 1/Nm}$. Hence, as a bounded operator on $\mathcal{H}_{\Lambda,\FBI}^{0}$, the Toeplitz $B_\Lambda \sigma_r B_\Lambda$ is dominated by $r^{-1/Nm}$. From the resolvent identity, one sees that the trace class norm of $(A^N + r)^{-1}$ is uniformly bounded when $r$ tends to $+ \infty$.
Hence, we have
\begin{equation*}
\begin{split}
\Tr\p{ \p{A^N + r}^{-1} } \underset{r \to + \infty}{=} \Tr\p{B_\Lambda a_r B_\Lambda} + \O\p{r^{- \frac{1}{Nm}}}.
\end{split}
\end{equation*}
By Lemma \ref{lemma:trace_pseudo}, since the symbol $a_r$ is of order $ - Nm < -n$, we retrieve that $B_\Lambda a_r B_\Lambda$ is trace class and moreover, its trace is the integral of its kernel on the diagonal. By the usual stationary phase argument, we find that the kernel of $B_\Lambda a_r B_\Lambda$ is of the form 
\begin{equation*}
\begin{split}
e^{i \frac{\Phi_\Lambda(\alpha,\beta)}{h}} \p{c_r(\alpha,\beta) + d_r(\alpha,\beta)} + R_r(\alpha,\beta). 
\end{split}
\end{equation*}
Here, $c_r(\alpha,\beta)$ is a symbol of order $-Nm$ such that $c_r(\alpha,\alpha) = (r + \jap{\va{\alpha}}^{Nm})^{-1}$. The symbol $d_r$ is of order $-Nm -1$, but if we see $a_r$ as a symbol of order  $-Nm +1$, then we see that, as a symbol of order $-Nm$, this $d_r$ is dominated by $r^{- 1/Nm}$. We bound the remainder term $R_r$ in the same way. Hence, we find that
\begin{equation}\label{eqpremierpas}
\begin{split}
\Tr\p{ \p{A^N + r}^{-1} } \underset{r \to + \infty}{=} \int_{\Lambda} \frac{1}{r + \jap{\va{\alpha}}^{Nm}} \mathrm{d}\alpha + \O\p{r^{- \frac{1}{Nm}}}.
\end{split}
\end{equation} 
And since $Nm < n+1$ we have $ - \frac{1}{Nm} < \frac{n}{Nm} -1$ and thus the remainder term in \eqref{eqpremierpas} is a small $o$ of $r^{\frac{n}{Nm} - 1}$. In order to estimate the integral in the right hand side of \eqref{eqpremierpas}, we may split it in the integral over $\{\jap{\va{\alpha}} \leq r^{1/Nm}\}$ and the integral over $\{\jap{\va{\alpha}} > r^{1/Nm}\}$. To bound the integral over $\{\jap{\va{\alpha}} \leq r^{1/Nm}\}$, just bound the integrand by $r^{-1}$ and notice that the volume of this set is dominated by $r^{n/Nm}$ (this can be done using the fact that the Jacobian of $\exp(- H_G^{\omega_I})$ is very close to be $1$, see Lemma \ref{lemma:uniformity-lagrangians}, and that $\exp(-H_G^{\omega_I})$ does not change scale, see Lemma \ref{lmdist}, where $G$ is as in Definition \ref{def:adapted-Lagrangian}). Hence, this first integral is dominated by $r^{\frac{n}{Nm} - 1}$. To bound the integral over $\{\jap{\va{\alpha}} > r^{1/Nm}\}$, split it into integrals over annulus where $\jap{\va{\alpha}}$ is roughly equal to a power of $2$ times $r^{1/Nm}$ and compare it to a geometric series to find that it is also dominated by $r^{n/Nm - 1}$. Finally, we find 
\begin{equation*}
\begin{split}
\Tr\p{ \p{A^N + r}^{-1} } \underset{r \to + \infty}{=} \O\p{r^{\frac{n}{Nm} - 1}}.
\end{split}
\end{equation*}
Then using Lidskii's Trace Theorem, we have for large $\lambda$ and $r$, and some constant $C > 0$:
\begin{equation*}
\begin{split}
\frac{N(\lambda)}{\lambda^N + r} & \leq \sum_{k \geq 0} \frac{1}{\lambda_k^N + r} = \Tr\p{A^N + r}^{-1} \leq C r^{\frac{n}{Nm} - 1}.
\end{split}
\end{equation*}
Hence
\begin{equation*}
\begin{split}
N(\lambda) \leq C(\lambda^N + r) r^{\frac{n}{Nm} - 1},
\end{split}
\end{equation*}
and the result follows by taking $r = \lambda^N$.
\end{proof}

\setcounter{equation}{0}

\chapter{Ruelle--Pollicott resonances and Gevrey Anosov flows}
\label{part:Anosov-flow}

This chapter is devoted to the application of the tools that we developed in Chapter \ref{part:FBI} to the study of the spectral theory of Gevrey Anosov flows. After giving an overview of our results, we will in \S~\ref{sec:Flow-adapted-spaces}, given a $\Gs$ Anosov vector field $X$, construct an I-Lagrangian space adapted to the study of its Ruelle--Pollicott spectrum. In \S~\ref{sec:order}, we will prove Theorem \ref{thm:main-strong} which generalizes Theorem \ref{thm:main}. In the last section \S~\ref{sec:perturbative-results}, we will consider perturbation of resonances.

\subsection*[Results on the dynamical determinant]{Results: Ruelle--Pollicott resonances and dynamical determinant for Gevrey Anosov flows}

We start by recalling some well-known facts in the $\mathcal{C}^\infty$ case. Let $M$ be a compact $\mathcal{C}^\infty$ manifold and $X$ be a $\mathcal{C}^\infty$ vector field on $M$ that does not vanish. We denote by $\p{\phi_t}_{t \in \R}$ the flow generated by $X$ and assume that $\phi_t$ is Anosov. Recall that it means that there is a decomposition 
\[
TM = E_0 \oplus E_u \oplus E_s
\]
of the tangent bundle with the following properties:
\begin{enumerate}[label=(\roman*)]
\item the decomposition $TM = E_0 \oplus E_u \oplus E_s$ is invariant by $\mathrm{d} \phi_t$ for every $t \in \R$;
\item $E_0$ is the span of the vector field $X$;
\item there are constants $C ,\theta > 0$ such that for all $v \in E_s$ and $t > 0$ we have $\va{\mathrm{d}\phi_t v} \leq C e^{- \theta t} \va{v}$;
\item there are constants $C ,\theta > 0$ such that for all $v \in E_u$ and $t > 0$ we have $\va{\mathrm{d}\phi_{-t} v} \leq C e^{- \theta t} \va{v}$.
\end{enumerate}
Here $\va{\cdot}$ denotes any Riemannian metric on $TM$. We also choose a $\mathcal{C}^\infty$ weight $V : M \to \C$ (which is real-valued in most of the applications) and define, for $t \geq 0$, the \emph{Koopman operator}
\begin{equation}\label{eq:transfer}
\begin{split}
\mathcal{L}_t u = \exp\p{\int_0^t V \circ \phi_\tau \mathrm{d}\tau} u \circ \phi_t.
\end{split}
\end{equation}
Notice that this definition makes sense when $u$ is a smooth function, but also when $u$ is a distribution $\mathcal{D}'\p{M}$. In fact, when $M$, $\phi_t$ and $V$ are $\G^s$ for some $s\geq1$, the operator $\mathcal{L}_t$ may be defined as an operator on $\U^s\p{M}$. Let us introduce the differential operator $P \coloneqq X+V$, so that
\begin{equation*}
\begin{split}
\frac{\mathrm{d}}{\mathrm{d}t'}\p{\mathcal{L}_{t'} u}_{|_{t' =t}} = P \mathcal{L}_t u.
\end{split}
\end{equation*}

In the case of $C^\infty$ flows, the anisotropic Banach spaces of distributions -- see \cite{Butterley-Liverani-07,bulicor,FauSjo} -- that we mentioned in the introduction always form a scale $(\mathcal{H}^r)_{r > 0}$ of Banach spaces with the following properties:
\begin{enumerate}[label=(\roman*)]
\item $\mathcal{C}^\infty\p{M} \subseteq \mathcal{H}^r \subseteq \mathcal{D}'\p{M}$, both inclusions are continuous and the first one has dense image;
\item $\p{\mathcal{L}_t}_{t \geq 0}$ is a strongly continuous semi-group on $\mathcal{H}^r$ and its generator is $P$;
\item the intersection of the spectrum of $P$ acting on $\mathcal{H}^r$ with the half-plane $\set{z \in \C : \Re z > -r}$ contains only isolated eigenvalues of finite multiplicities.
\end{enumerate}
The property (i) is just a non-triviality assumption: it implies that the elements of $\mathcal{H}^r$ are objects that live on the manifold $M$. This property can easily be softened without any harm in the theory. For instance, working with a $\G^s$ flow, one could replace $\mathcal{C}^\infty\p{M}$ by $\G^s\p{M}$ and $\mathcal{D}'\p{M}$ by $\U^s\p{M}$. The point (ii) is of utmost importance: it is the property that ensures that the spectrum of $P$ on $\mathcal{H}^r$ has a dynamical interpretation and allows to describe the asymptotic of $\mathcal{L}_t$ when $t$ tends to $+ \infty$.

If the spaces $\mathcal{H}^r$ are highly non-canonical, the spectrum of $P$ acting on $\mathcal{H}^r$ is intrinsically defined by $V$ and the vector field $X$. This may be shown using the following argument from \cite{FauSjo} (we will refer to this argument in the proof of Theorem \ref{thm:spectral-theory}): when $\Re z \gg 1$, the resolvent $(z-P)^{-1}$ of $P$, as an operator on $L^2(M)$ for example, is the Laplace transform of $\mathcal{L}_t$, that is
\begin{equation}\label{eq:a_quoi_doit_ressembler_une_resolvante}
\begin{split}
R(z) : u \mapsto \int_{0}^{+ \infty} e^{- z t} \mathcal{L}_t u \mathrm{d}t.
\end{split}
\end{equation}
Now, the point (ii) above implies that, when $\Re z \gg 1$, the resolvent of $P$ acting on $\mathcal{H}^r$ is also given by \eqref{eq:a_quoi_doit_ressembler_une_resolvante}. Since $\mathcal{H}^r$ is intermediate between $\mathcal{C}^\infty\p{M}$ and $\mathcal{D}'\p{M}$, the point (iii) implies that the family of operator $R(z) : \mathcal{C}^\infty\p{M} \to \mathcal{D}'\p{M}$ admits a meromorphic continuation to $\set{z \in \C : \Re z > -r}$ with residues of finite rank. By the analytic continuation principle, this meromorphic continuation does not depend on the choice of space $\mathcal{H}^r$. Moreover, letting $r$ tend to $+ \infty$, we see that $R(z)$ has a meromorphic continuation to the whole complex plane.

As in the case without potential, the poles of the complex continuation of $R(z)$ are called the \emph{Ruelle--Pollicott resonances} of $P$. Moreover, if $\lambda \in \C$ is a Ruelle--Pollicott resonance, then the residue of $R(z)$ is (up to some injections) a spectral projector $\pi_\lambda$ of finite rank for $P$. The rank of $\pi_\lambda$ is the multiplicity of $\lambda$ as a Ruelle resonance and the elements of the image of $\pi_\lambda$ are called \emph{resonant states} associated with $\lambda$ for $P$. Notice that the resonant states are not necessarily eigenvectors for $P$ (there may be Jordan blocks as shown in \cite{cekicResonantSpacesVolume2019}).

We can introduce a generalization of $\zeta_X(z)$ associated to $P$: for $\Re z \gg 1$, we let
\begin{equation}\label{eq:zetaXV}
\begin{split}
\zeta_{X,V}(z) = \exp\p{ - \sum_{\gamma} \frac{T_\gamma^{\#}}{T_\gamma}\frac{e^{\int_\gamma V} e^{-z T_\gamma} }{\va{\det(I- \mathcal{P}_\gamma)}} },
\end{split}
\end{equation}
where the sum $\gamma$ runs over the periodic orbits of $\phi_t$. If $\gamma$ is a periodic orbit then $T_\gamma$ denotes its length, $T_\gamma^{\#}$ its primitive length (i.e. the length of the smallest periodic orbits with the same image), the integral $\int_\gamma V$ is defined by $\int_\gamma V = \int_0^{T_\gamma} V\p{\phi_t(x)} \mathrm{d}t$ for any $x$ in the image of $\gamma$, and $\mathcal{P}_\gamma$ is the linearized Poincaré map associated with $\gamma$, i.e. $\mathcal{P}_\gamma = \mathrm{d}\phi_{T_\gamma}(x)_{|_{E_u(x) \oplus E_s(x)}}$. The map $\mathcal{P}_\gamma$ depends on the point $x$ in the image of $\gamma$, but its conjugacy class is well-defined. It follows from elementary estimates on the number of periodic orbits for $\phi_t$ (see \cite[Lemma 2.2]{DZdet} for instance) that $\zeta_{X,V}(z)$ is well-defined for $\Re z \gg 1$. It is proven in \cite{GLP} (see also \cite{DZdet} for an alternative proof) that $\zeta_{X,V}$ extends to a holomorphic function on $\C$ whose zeros are the Ruelle resonances (counted with multiplicity). 

Let us now state our main result. It extends Theorem \ref{thm:main} to the case with potentials.
\begin{theorem}\label{thm:main-strong}
Let $s \in \left[1,+\infty\right[$. Let $M$ be a $n$-dimensional $\G^s$ compact manifold. Let $X$ be a $\G^s$ vector field that generates an Anosov flow $\p{\phi_t}_{t \in \R}$. Let $V : M \to \C$ be a $\G^s$ potential. Then there is a constant $C > 0$ such that for every $z \in \C$ we have
\begin{equation*}
\begin{split}
\va{\zeta_{X,V}(z)} \leq C \exp\p{C \va{z}^{n s}}.
\end{split}
\end{equation*}
In particular, the order of $\zeta_{X,V}$ is less than $n s$.
\end{theorem}

As a by-product of the proof of Theorem \ref{thm:main-strong}, we deduce the following bound on the number of Ruelle resonances.

\begin{prop}\label{prop:borne-resonances}
Under the assumptions of Theorem \ref{thm:main-strong} and if $N(r)$ denotes the number of Ruelle resonances of $P$ of modulus less than $r$, then we have
\begin{equation*}
\begin{split}
N(r) \underset{r \to + \infty }{=} \O(r^{n s}).
\end{split}
\end{equation*}
\end{prop}

\begin{remark}
Proposition \ref{prop:borne-resonances} is an immediate consequence of Theorem \ref{thm:main-strong} and Jensen's formula \cite[Theorem 1.2.1]{Boas}. However, we will need to prove Proposition \ref{prop:borne-resonances} before Theorem \ref{thm:main-strong}. Notice also that Theorem \ref{thm:main-strong} and Hadamard's Factorization Theorem (see \cite[\S 2.7]{Boas}) imply that there is a canonical factorization for the dynamical determinant.
\end{remark}

\begin{remark}
It is well-known that the finiteness of the order of the dynamical determinant implies that a \emph{trace formula} associated with $P$ holds (see \cite[Proposition 1.5]{jezequelGlobalTraceFormula2019} for a precise statement). In particular, we prove here that the trace formula associated with $P$ holds when the data are Gevrey. However, this is not a new result since \cite[Corollary 1.8]{jezequelGlobalTraceFormula2019} states that the trace formula holds for a larger class of regularity than Gevrey.
\end{remark}

We will also give a statement about the $\G^{s}$ wave front sets (in the sense of \S \ref{sec:wfs}) of the resonant states. To do so, we need to introduce the decomposition $T^* M = E_0^* \oplus E_s^* \oplus E_u^*$ of the cotangent bundle of $M$. We define the space $E_0^*$ as the annihilator of $E_u \oplus E_s$, the space $E_s^*$ as the annihilator of $E_0 \oplus E_s$ and the space $E_u^*$ as the annihilator of $E_0 \oplus E_u$. Then we have the following statement:

\begin{prop}\label{prop:wfs_resonant_state}
Under the assumption of Theorem \ref{thm:main-strong}, if $f$ is a resonant state for $P = X+V$ then the Gevrey wave front set $\WF_{\G^{s}}(f)$ (see Definition \ref{def:wavefrontset}) of $f$ is contained in the stable codirection $E_s^*$.
\end{prop}

\begin{remark}
Proposition \ref{prop:wfs_resonant_state} may seem surprising to readers familiar with papers on Ruelle resonances from the microlocal community (for instance \cite{FauSjo,DZdet,dyatlovRuelleZetaFunction2017}). Indeed, in these papers resonant states have their wave front sets in $E_u^*$. The difference is due to the fact that we study resonant states associated with the Koopman operator (the operator \eqref{eq:transfer} for $t \geq 0$) while \cite{FauSjo,DZdet,dyatlovRuelleZetaFunction2017} study resonant states associated with the transfer operator (the operator \eqref{eq:transfer} for $t \leq 0$). One can go from one to another by reversing time, which exchanges stable and unstable direction, so that our result is coherent with the literature. Notice also that in \cite{Faure-Roy-Sjostrand-08}, while studying the Koopman operator, resonant states have their wave front set contained in $E_u^*$. This is just because \cite{Faure-Roy-Sjostrand-08} uses a different convention for the definition of $E_u^*$ and $E_s^*$ than \cite{FauSjo,DZdet,dyatlovRuelleZetaFunction2017}. Finally, it is well-known in the dynamical community that resonant states are ``smooth in the stable direction'' (if we study the Koopman operator). This is precisely what happens here: in practical terms, having its wave front set in $E_s^*$ means ``being smooth in the stable direction''. This may seem strange and is due to the convention for the definition of $E_u^*$ and $E_s^*$ that we borrowed from \cite{FauSjo} (we have in particular that $\dim E_u^* = \dim E_s$).
\end{remark}

While it may seem obvious to specialists, let us notice that our results extend to the case of vector bundles. Let $F \to M$ be a complex $\G^s$ vector bundle. Assume that $\phi_t$ lifts to a $\G^s$ one parameter subgroup of vector bundle automorphisms $\p{L_t}_{t \in \R}$ of $F$. In this context, we may define the Koopman operator $\mathcal{L}_t$ for $t \in \R$ and $u$ a smooth section of $F$ by $\mathcal{L}_t u(x) = L_{-t}(u\circ \phi_t(x))$ (this is the natural analogue of \eqref{eq:transfer} in the scalar case). Then, we may replace in the results above the operator $P = X +V$ by the operator $P$ define on smooth sections of $F$ by 
\begin{equation*}
\begin{split}
P u = \frac{\mathrm{d}}{\mathrm{d}t}(\mathcal{L}_t u)|_{t = 0}.
\end{split}
\end{equation*}
To highlight the similarity with the scalar case, we may choose a $\G^s$ connection $\nabla$ on $F$. Then the operator $P$ writes $P = \nabla_X + A$ where $A$ is a $\G^s$ section of the bundle of automorphisms of $F$. We see here that the principal symbol of $P$ is scalar, hence the eventual functional analytic complications that could arise when introducing vector bundles are restricted to the sub-principal level and are consequently dealt with easily. 

Concerning the algebra, the traces and determinants computations in the bundle case are carried out using the bundle version of Guillemin's Trace Formula \cite{guillemin-lectures-spectral-theory-77}. Hence, in this context the dynamical determinant writes (for $\Re z  \gg 1$)
\begin{equation*}
\begin{split}
\zeta_{F}(z) = \exp\p{- \sum_{\gamma} \frac{T_\gamma^{\#}}{T_\gamma} \frac{\Tr\p{L^\gamma} e^{- z T_\gamma}}{\va{\det\p{I - \mathcal{P}_\gamma}}}}.
\end{split}
\end{equation*}
Here, if $\gamma$ is a periodic orbit of $\phi_t$ and $x$ a point of the image of $\gamma$, the endomorphism $L^\gamma$ is the restriction of $L_{T_\gamma}$ to $F_x$ (the conjugacy class of $L^\gamma$ is well-defined).

Using a trick due to Ruelle to write zeta functions as alternate products of dynamical determinants (as in \cite{GLP, DZdet}), the bundle version of Theorem \ref{thm:main} implies that the \emph{Ruelle zeta function}, i.e.
\[
\zeta_R(z) := \exp\left( -  \sum_{\gamma} \frac{T_\gamma^\#}{T_\gamma} e^{- z T_\gamma} \right),
\]
associated with a $\G^s$ Anosov flow, has order less than $ns$.

Returning back to the scalar case, we will also use the tools that we developed in order to study stochastic and deterministic perturbations of Gevrey Anosov flows, respectively in \S \ref{sec:viscosite} and in \S \ref{sec:linear_response}. In \S \ref{sec:viscosite}, we consider a perturbation of the operator $P = X + V$ of the form
\begin{equation*}
\begin{split}
P_\epsilon = P + \epsilon \Delta,
\end{split}
\end{equation*}
where $\epsilon \geq 0$ and $\Delta$ is a self-adjoint, negative and elliptic of order $m > 1$ (in the classical sense) $\G^s$ pseudor. The operator $P_\epsilon$ acting on $L^2$ (with its natural domain) has its real part $\Re V + \epsilon \Delta$ bounded from above, hence its resolvent set is non-empty (and it generates a strongly continuous semi-group). When $\epsilon > 0$, the operator $P_\epsilon$ is elliptic, so that its resolvent is compact and $P_\epsilon$ has discrete spectrum $\sigma_{L^2}\p{P_\epsilon}$ on $L^2\p{M}$. We know from \cite[Theorem 1]{dyatlovStochasticStabilityPollicott2015} that $\sigma_{L^2}\p{P_\epsilon}$ converges locally to the Ruelle spectrum of $P$. We will prove a global version of this result. To do so, we need to introduce a new distance to compare spectra. If $z \in \C$, we define the distance $d_z$ on $\C \cup \set{\infty} \setminus {z}$ by
\begin{equation*}
\begin{split}
d_z(x,y) = \va{\frac{1}{z-x} - \frac{1}{z-y}}.
\end{split}
\end{equation*}
We will prove in \S \ref{sec:viscosite} the following ``global'' version of \cite[Theorem 1]{dyatlovStochasticStabilityPollicott2015}.

\begin{theorem}\label{theorem:viscosite}
Under the assumption of Theorem \ref{thm:main-strong}, if $h$ is small enough, then for every $p > n s$ and $z \in \R_+$ large enough, there is a constant $C > 0$ such that for every $\epsilon > 0$ small enough, we have
\begin{equation*}
\begin{split}
d_{z,H}\p{\sigma_{\textup{Ruelle}}\p{P} \cup \set{\infty}, \sigma_{L^2}\p{P_\epsilon} \cup \set{\infty}} \leq C \va{\ln \epsilon}^{- \frac{1}{p}}.
\end{split}
\end{equation*}
Here, $d_{z,H}$ denotes the Hausdorff distance associated to the distance $d_z$ and $\sigma_{\textup{Ruelle}}\p{P}$ the Ruelle spectrum of $P$.
\end{theorem}

\begin{remark}
Theorem \ref{theorem:viscosite} also applies if $\Delta$ is a classical differential operator with $\G^s$ coefficients (for instance, the Laplacian associated to a $\G^s$ metric). Indeed, one only needs to replace $\Delta$ by $h^m \Delta$ and $\epsilon$ by $h^{-m}\epsilon$ in Theorem \ref{theorem:viscosite}.
\end{remark}

In the proof of Theorem \ref{theorem:viscosite}, we will use ideas similar to those of \cite{Zworski-Galkowski-2}, with the necessary modifications in our different context (in particular, we need to deal with the flow direction and the fact that we are dealing with operators of order $1$). The convergence in Theorem \ref{theorem:viscosite} seems to be very weak. However, we think that it is not reasonable to expect too fast a convergence in such a global result. Indeed, when we add the pseudo-differential operator $\Delta$ to $P$ in order to form $P_\epsilon$, since $\Delta$ has higher order, we can expect that the spectrum of $P_\epsilon$ looks globally like the spectrum of $\Delta$, rather than like the Ruelle spectrum of $P$. Indeed, the higher order operator will be predominant at higher frequencies. Furthermore, the spectrum of $\Delta$ is contained in $\R_-$ while we expect some kind of vertical structure for the Ruelle spectrum of $P$ (see for instance \cite{ltfzwor,faureBandStructureRuelle2013}). Hence, we may expect $\sigma_{L^2}\p{P_\epsilon}$  to be some kind of ``flattened'' version of the Ruelle spectrum of $P$, and its global structure is thus very different from the actual Ruelle spectrum of $P$.

Finally, \S \ref{sec:linear_response} will be dedicated to the study of deterministic perturbations of an Anosov flow. Such perturbations have already been studied by functional analytic methods, see \cite{Butterley-Liverani-07,bulicor}, and we will focus on the particularity of our highly regular context. The main results of \S \ref{sec:linear_response} are of technical natures and cannot be stated now. Let us just notice that we will prove in \S \ref{sec:linear_response} a deterministic analogue of Theorem \ref{theorem:viscosite} (see Corollary \ref{cor:lin_resp_global}) and Theorem \ref{theorem:SRB} on the dependence of the SRB measure of a real-analytic Anosov flow on the flow. The main point here is that in the real-analytic case we are able to set up a Kato theory for real-analytic Anosov flow (Theorem \ref{thm:lanalyticite_engendre_lanalyticite}).

\section{I-Lagrangian spaces adapted to a Gevrey Anosov flow}
\label{sec:Flow-adapted-spaces}

From now on, $s \in \left[1,+ \infty\right[$ is fixed, $X$ is a $\G^s$ vector field on a $n$-dimensional $\G^s$ manifold $M$ that generates an Anosov flow $\p{\phi_t}_{t \in \R}$, and $V : M \to \C$ is a $\G^s$ function. We define the differential operator $P = X +V$, and the associated Koopman operator is given by \eqref{eq:transfer}. Without loss of generality, we may assume that $M$ is endowed with a structure of real-analytic Riemannian manifold (compatible with its $\G^s$ structure, see Remark \ref{remark:structure_analytique}). 

The machinery from \S \ref{sec:Grauert} is then available, in particular we denote by $\smash{\widetilde{M}}$ a complex neighbourhood for $M$. According to Theorem \ref{thm:existence-good-transform}, there is an analytic FBI transform $T$ on $M$ such that $T^* T = I$. As above we set $S = T^*$. In order to apply the results from the previous chapter to the operator $P$, we need first to find a suitable $\p{\tau_0,s}$-adapted Lagrangian $\Lambda$. The Lagrangian $\Lambda$ will be defined by \eqref{eq:adapted_Lagrangian} where the symbol $G$ is defined by $G = \tau G_0$ where $\tau \ll h^{1 - 1/s}$ and $G_0$ is a so-called \emph{escape function}. The construction of $G_0$ is detailed in \S\ref{sec:escape_function} (see Lemma \ref{lemma:escape-function}). In \S \ref{sec:spectral_theory}, we will then describe the spectral theory of $P$ on the related I-Lagrangian space.

\subsection{Constructing an escape function}\label{sec:escape_function}

Recall that the decomposition $TM = E_0 \oplus E_u \oplus E_s$ of the tangent bundle induces a dual decomposition $T^* M = E_0^* \oplus E_u^* \oplus E_s^*$ of the cotangent bundle. Here, $E_0^* = (E_u \oplus E_s)^{\perp}$, the space $E_u^*= (E_0 \oplus E_u)^\perp$ and $E_s^* = (E_0 \oplus E_s)^\perp$. We denote by $p : T^* M \to \C$ the principal symbol of the semi-classical differential operator $hP$. We recall for $\alpha = (\alpha_x,\alpha_\xi) \in T^* M$,
\begin{equation}\label{eq:extensionp}
\begin{split}
p(\alpha) = i \alpha_\xi\p{X(\alpha_x)}.
\end{split}
\end{equation}
To apply the machinery presented in the previous part, we will need an almost analytic extension for $p$. We construct it in the following way: we take a $\G^s$ almost analytic extension $\widetilde{X}$ for $X$, given by Lemma \ref{lemma:almost-analytic-extension-Gs} if $s > 1$ (we just take $\widetilde{X} = X$ if $s=1$), and then we set for $\alpha \in \p{T^* M}_{\epsilon_0}$ (for some small $\epsilon_0 > 0$)
\begin{equation}\label{eq:defptilde}
\begin{split}
\tilde{p}(\alpha) = i \alpha_\xi\p{\widetilde{X}\p{\alpha_x}}.
\end{split}
\end{equation}
It will be important when constructing the escape function $G_0$ that this almost analytic extension is linear in $\alpha_\xi$. We are now ready to construct $G_0$.

\begin{lemma}\label{lemma:escape-function}
Let $\mathscr{C}^0$ and $\mathscr{C}^s$ be conical neighbourhoods respectively of $E_0^*$ and of $E_s^*$ in $T^* M$. Let $\delta \geq 0$. Then there are arbitrarily small $\epsilon_0 > \epsilon_1 > 0$ and a symbol $G_0$ of order $\delta$ on $(T^* M)_{\epsilon_0}$, supported in $(T^* M)_{\epsilon_1}$ with the following properties:
\begin{enumerate}[label=(\roman*)]
\item the restriction of $G_0$ to $T^* M$ is negative and classically elliptic of order $\delta$ outside of $\mathscr{C}^s$;
\item the restriction of $\set{G_0,\Re \tilde{p}}$ to $T^* M$ is negative and classically elliptic of order $\delta$ outside of $\mathscr{C}^0$;
\item if $s=1$, there are $C,\epsilon_2 > 0$ such that $\set{G_0,\Re \tilde{p}} \leq C$ on $(T^* M)_{\epsilon_2} $. If $s > 1$, there is $\epsilon_2$ such that for every $N > 0$ there is $C_N > 0$ such that for all $\alpha \in (T^* M)_{\epsilon_2}$ we have $\set{G_0,\Re \tilde{p}}(\alpha) \leq C_N(1 + \va{\Im \alpha_x}^N \jap{\va{\alpha}}^\delta)$.
\end{enumerate}
Here, the Poisson Bracket $\set{G_0,\Re \tilde{p}}$ is the one associated with the real symplectic form $\omega_I$ on $(T^* M)_{\epsilon_0}$.
\end{lemma}

Before proving Lemma \ref{lemma:escape-function}, let us explain why we need our escape function to satisfy these properties. We recall that a symbol is said to be ``classically elliptic of order $\delta$'' if it is greater than $\jap{\va{\alpha}}^\delta$ when $\alpha$ is large enough. The point (i) is here to control the $\G^{s}$ wave front set of elements of the set $\mathcal{H}_\Lambda^0$, in particular of resonant states, using Lemma \ref{lemma:wavefrontset_HLambda}. The point (ii) will be used in the proof of Lemma \ref{lmhypoell} to deduce the hypoellipticity of the operator $P$ acting on $\mathcal{H}_\Lambda^0$ from the multiplication formula, Proposition \ref{propmultform}. Finally, the point (iii) will be used in the proof of Proposition \ref{prop:semigroup} to show that the Koopman operator \eqref{eq:transfer} defines a continuous semi-group on $\mathcal{H}_\Lambda^0$. This property is what ensures that the spectrum of $P$ on our spaces has a dynamical meaning. Even though point (ii) seems to be the most crucial one, since it is the one that allows us to enter the world of Schatten operators and eventually prove Theorem \ref{thm:main-strong}, the importance of (iii) could not be overestimated. Finally, notice that, in order to apply Lemma \ref{lemma:TPS-close-diagonal} in the most favorable case for us, it will be natural in the following to choose $\delta = 1/s$.

The proof of Lemma \ref{lemma:escape-function} will take the rest of this section. 
\begin{proof}[Proof of Lemma \ref{lemma:escape-function}]
We want to understand $\set{G_0,\Re \tilde{p}}$ in order to control how the real part of $\tilde{p}$ evolves under the flow of $H_{G_0}^{\omega_I}$. However, since $\set{G_0,\Re \tilde{p}} = - \set{\Re \tilde{p}, G_0}$, we may understand $\set{G_0,\Re \tilde{p}}$ by controlling how $G_0$ evolves under the flow of $-H_{\Re \tilde{p}}^{\omega_I}$. Hence, we need to understand the dynamics of this flow. To do so, we may multiply $\smash{\widetilde{X}}$ by a bump function identically equals to $1$ near $M$ (since we only claim properties for $G_0$ near $T^* M$ and we will not use the high regularity of $\smash{\widetilde{X}}$ in this proof). Then, it follows from the formula \eqref{eq:I-hamilton-field} that the flow of $-H_{\Re \tilde{p}}^{\omega_I}$ is complete. We denote this flow by $\p{\Theta_t}_{t \in \R}$ and write
\begin{equation*}
\begin{split}
\Theta_t(\alpha) = \p{\Theta_{t,x}(\alpha), \Theta_{t,\xi}(\alpha)}.
\end{split}
\end{equation*}
Using \eqref{eq:I-hamilton-field}, we see that $H_{\Re \tilde{p}}^{\omega_I}$ is given in coordinates $\p{x+iy ,\xi + i \eta}$ with $\smash{\widetilde{X}} = (\smash{\widetilde{X}_1},\dots, \smash{\widetilde{X_n}})$ by
\begin{equation}\label{eq:detaillonsunpeu}
\begin{split}
& -H_{\Re \tilde{p}}^{\omega_I} = \sum_{j = 1}^n \Re \widetilde{X_j} \frac{\partial}{\partial x_j} + \Im \widetilde{X_j} \frac{\partial}{\partial y_j} - \p{\jap{\xi,\frac{\partial \Im \widetilde{X}}{\partial y_j}} + \jap{\eta,\frac{\partial \Re \widetilde{X}}{\partial y_j}}} \frac{\partial}{\partial \xi_j} \\ & \qquad \qquad \qquad \qquad \qquad \qquad \qquad \qquad \quad - \p{\jap{\xi,\frac{\partial \Im \widetilde{X}}{\partial x_j}} + \jap{\eta,\frac{\partial \Re \widetilde{X}}{\partial x_j}}} \frac{\partial}{\partial \eta_j}.
\end{split}
\end{equation}
Indeed, it follows from \eqref{eq:extensionp} that in such coordinates we have
\begin{equation*}
\begin{split}
\Re \tilde{p} = - \jap{\xi,\Im \widetilde{X}} - \jap{\eta,\Re \widetilde{X}}.
\end{split}
\end{equation*}
From \eqref{eq:detaillonsunpeu}, we see that the projection $\Theta_{t,x}$ of $\Theta_t$ is in fact given by the formula 
\begin{equation*}
\begin{split}
\Theta_{t,x}(\alpha) = \tilde{\phi}_t(\alpha_x),
\end{split}
\end{equation*}
where $(\tilde{\phi}_t)_{t \in \R}$ denotes the flow of $\widetilde{X}$ (in particular, the restriction of $\tilde{\phi}_t$ to $M$ is $\phi_t$). Then, we notice that in \eqref{eq:detaillonsunpeu}, the component of $-H_{\Re \tilde{p}}^{\omega_I}$ along $\partial/ \partial \xi$ and $\partial/ \partial \eta$ is linear in $\p{\xi,\eta}$. It implies that
\begin{equation*}
\begin{split}
\Theta_{t,\xi}(\alpha) = L_t(\alpha_x)(\alpha_\xi),
\end{split}
\end{equation*}
where $L_t(\alpha_x)$ is a $\R$-linear application from $T^*_{\alpha_x} \smash{\widetilde{M}}$ to $\smash{T^*_{\tilde{\phi}_t(\alpha_x)} \widetilde{M}}$ (that depends smoothly on $t$ and $\alpha_x$). Now, since $\smash{\widetilde{X}}$ satisfies the Cauchy--Riemann equations and is tangent to $M$ on $M$, we find that, for $y = 0$, in the same system of coordinates as \eqref{eq:detaillonsunpeu}, we have
\begin{equation}\label{eq:CRishere}
\begin{split}
\frac{\partial \Re \widetilde{X}}{\partial y} = - \frac{\partial \Im \widetilde{X}}{\partial x} = 0 \textup{ and } \frac{\partial \Im \widetilde{X}}{\partial y} = \frac{\partial \Re \widetilde{X}}{\partial x} = \frac{\partial X}{\partial x}.
\end{split}
\end{equation} 
By uniqueness in the Cauchy--Lipschitz Theorem, we find by plugging \eqref{eq:CRishere} in \eqref{eq:detaillonsunpeu} that, for $x \in M$ and $t \in \R$ we have
\begin{equation}\label{eq:cestLt}
\begin{split}
L_t(x) = {}^T \! \p{\mathrm{d}\phi_t(x)^{-1}}. 
\end{split}
\end{equation}
Hence, the hyperbolicity of $\phi_t$ will have important consequences on the dynamics of $\Theta_t$. Let us ``complexify'' the bundles $E_{0,u,s}^\ast$. For $x\in M$, we denote by $\smash{E_{0}^{\C,*},E_u^{\C,*}}$ and $\smash{E_s^{\C,*}}$ the complexification of $E_0^*, E_u^*$ and $E_s^*$, considering linear forms valued in $\C$ instead of $\R$. For instance, for $x \in M$, we write $\smash{E_{0,x}^{\C,*}}$ for the subspace of $T_x^* M \otimes \C$ consisting of $\R$-linear maps from $T_x M$ to $\C$ that vanish on $E^u \oplus E^s$ (or, under a natural identification, of $\C$-linear forms on $\smash{T_x^* \widetilde{M}}$ that vanish on $E^u \oplus E^s$). From the fact that $T_ x M = E_{0,x} \oplus E_{u,x} \oplus E_{s,x}$ is a totally real subspace of maximal dimension of $T_x \smash{\widetilde{M}}$, we deduce that $\smash{T_x^* \widetilde{M} =  E_{0,x}^{\C,*} \oplus E_{u,x}^{\C,*} \oplus E_{s,x}^{\C,*}}$. Since $E_{0,x}, E_{u,x}$ and $E_{s,x}$ depends in a Hölder-continuous fashion on $x \in M$, so do $\smash{E_{0,x}^{\C,*}, E_{u,x}^{\C,*}}$ and $\smash{E_{s,x}^{\C,*}}$. Consequently, we may extend continuously $\smash{E_{0}^{\C,*},E_u^{\C,*}}$ and $\smash{E_s^{\C,*}}$ to $\smash{\widetilde{M}}$. Then, if $\smash{\widetilde{M}}$ is small enough, we have $\smash{T_x \widetilde{M} = E_{0,x}^{\C,*} \oplus E_{u,x}^{\C,*} \oplus E_{s,x}^{\C,*}}$ for all $x \in \widetilde{M}$. \emph{A priori}, this decomposition is only invariant under $L_t$ for $t \in \R$ and $x \in M$. If $\sigma \in \smash{T_x^* \widetilde{M}}$ then we write $\sigma = \sigma_0 + \sigma_u + \sigma_s$ for the decomposition of $\sigma$ under $\smash{T_x \widetilde{M} = E_{0,x}^{\C,*} \oplus E_{u,x}^{\C,*} \oplus E_{s,x}^{\C,*}}$. Then, we put a real Riemannian metric on $\smash{\widetilde{M}}$ and define for $x \in \widetilde{M}$ and $\gamma > 0$ the cones 
\begin{equation*}
\begin{split}
C_u^\gamma(x) = \set{ \sigma \in T_x^* \widetilde{M} : \va{\sigma_0} + \va{\sigma_s} \leq \gamma \va{\sigma_u}}
\end{split}
\end{equation*}
and
\begin{equation*}
\begin{split}
C_s^\gamma(x) = \set{ \sigma \in T_x^* \widetilde{M} : \va{\sigma_0} + \va{\sigma_u} \leq \gamma \va{\sigma_s}}.
\end{split}
\end{equation*}
Without loss of generality, we may assume that $\mathscr{C}^0$ is closed and does not intersect $E_u^* \oplus E_s^* \setminus \set{0}$. We may also assume that there is a closed conic neighbourhood $\smash{\mathscr{C}^{\C,0}}$ of $\smash{E_0^{\C,*}}$ in $\smash{T^* \widetilde{M}}$ such that $\smash{\mathscr{C}^{\C,0} \cap T^* M = \mathscr{C}^0}$. Choose then a small closed conic neighbourhood $\mathscr{C}^{0s}$ of $E_s^{*} \oplus E_0^{*}$ in $T^* M$, that does not intersect $E_u^* \setminus \set{0}$.

From \eqref{eq:cestLt} and standard arguments in hyperbolic dynamics, there are a large $T_0 > 0$, small $0 < \gamma_1 < \gamma$, some $\lambda > 1$ and a constant $C_1 >0$ such that if $\alpha = (\alpha_x,\alpha_\xi) \in T^* M \otimes \C$, and $T_1\geq T_0$ then:
\begin{enumerate}[label=\alph*.]
\item either $\Theta_{T_1,\xi}(\alpha) \in C_u^{\gamma_1}(\tilde{\phi}_{T_1}(\alpha_x))$ or $\Theta_{-T_0,\xi}(\alpha) \in C_s^{\gamma_1}(\tilde{\phi}_{-T_0}(\alpha_x))$ or we have $\alpha\in \mathscr{C}^{\C,0}$;
\item if $\Theta_{-T_0,\xi}(\alpha) \in C_u^\gamma(\tilde{\phi}_{-T_0}(\alpha_x))$ then $\Theta_{T_1,\xi}(\alpha) \in C_u^{\gamma_1}(\tilde{\phi}_{T_1}(\alpha_x))$ and we have $\va{\Theta_{T_1,\xi}(\alpha)} \geq \lambda \va{\Theta_{-T_0,\xi}(\alpha)}$;
\item if $\Theta_{T_1,\xi}(\alpha) \in C_s^\gamma(\tilde{\phi}_{T_1}(\alpha_x))$ then $\Theta_{-T_0,\xi}(\alpha) \in C_s^{\gamma_1}(\tilde{\phi}_{-T_0}(\alpha_x))$ and we have $\va{\Theta_{T_1,\xi}(\alpha)} \leq \lambda^{-1} \va{\Theta_{-T_0,\xi}(\alpha)}$;
\item if $\alpha \in T^* M$ does not belong to $\mathscr{C}^{0s}$, then, for $t \geq 0$, we have that $\Theta_{t,\xi}(\alpha) \notin C_s^{\gamma}(\tilde{\phi}_{t}(\alpha_x))$, and, for $t \geq T_0$, we have $\Theta_{t,\xi}(\alpha) \in C_u^{\gamma_1}(\tilde{\phi}_{t}(\alpha_x))$ and $\va{\Theta_{t,\xi}(\alpha)} \geq C_1^{-1} \va{\alpha_\xi}$.
\end{enumerate}
Since we ask here for $\alpha \in T^* M \otimes \C$, these are consequences of the hyperbolicity of $\phi_t$, that is, it only relies on the dynamic on $M$. We want to apply a perturbation argument to show that b and c remain true on a small complex neighbourhood of $M$, but we need first to fix the value of $T_1$. Hence, we fix the value of $T_1$, large enough such that we have
\begin{equation}\label{eq:choix_de_T_1}
\begin{split}
\sup_{\alpha \in T^* M \setminus \set{0}} \frac{1}{\va{\alpha_\xi}^\delta}\int_{-T_0}^0 \va{\Theta_{t,\xi}(\alpha)}^\delta \mathrm{d}t < \frac{T_1 - T_0}{2 C_1^\delta}.
\end{split}
\end{equation}

Now that $T_1$ is fixed, it follows from a perturbation argument that, up to taking a smaller $\lambda$, a smaller $\gamma$, a larger $\gamma_1$ and a larger $C_1$, the properties b and c above remain true when $(x,\sigma) \in T^* \p{M}_{\epsilon_0}$, for some small $\epsilon_0 > 0$ (and \eqref{eq:choix_de_T_1} remains true since we asked for a strict inequality). Then, we choose a symbol $m \in S_{KN}^0(T^* \p{M}_{\epsilon_0})$ of order $0$ on $T^*(M)_{\epsilon_0}$, valued in $\left[-1,1\right]$, with the following properties:
\begin{itemize}
\item if $x \in (M)_{\epsilon_0}$ and $\sigma \in T^*_x \widetilde{M} \setminus \p{C_u^\gamma(x) \cup C_s^\gamma(x)}$ or $\sigma$ is near $0$ then $m(x,\sigma) = 0$;
\item there is $C > 0$ such that if $x \in (M)_{\epsilon_0}$ and $\sigma \in C_s^{\gamma_1}(x)$ satisfies $\va{\sigma} \geq C$ then $m(x,\sigma) = 1$;
\item there is $C > 0$ such that if $x \in (M)_{\epsilon_0}$ and $\sigma \in C_u^{\gamma_1}(x)$ satisfies $\va{\sigma} \geq C$ then $m(x,\sigma) = - 1$;
\item $m$ is non-positive on $\mathcal{C}^\gamma_u$ and non-negative on $\mathcal{C}^\gamma_s$.
\end{itemize}
We also choose a $\mathcal{C}^\infty$ function $\chi : \R \to \left[0,1\right]$ vanishing on $\left]- \infty, 1\right]$ and taking value $1$ on $\left[2,+ \infty \right[$. Then, we may define $G_0$ near $T^* M \otimes \C$ by the formula
\begin{equation}\label{eq:defG}
\begin{split}
G_0(\alpha) = \int_{-T_0}^{T_1} m(\Theta_t(\alpha)) \va{\Theta_{t,\xi}(\alpha)}^\delta \mathrm{d}t - A \chi(\va{\alpha_\xi}) \va{\Re\p{\alpha_\xi\p{\widetilde{X}(\alpha_x)}}}^\delta,
\end{split}
\end{equation}
where $A$ is a large constant to be chosen later. Then we multiply $G_0$ by a bump function to satisfy the claim on the support. It does not change the value of $G_0$ near $T^* M$ and hence it will not interfere with the properties (i),(ii) and (iii). We may consequently use Formula \eqref{eq:defG} to prove (i),(ii) and (iii).

We start by proving (i). Take $\alpha \in T^* M \setminus \mathscr{C}^s$ and write
\begin{equation}\label{eq:ellipG}
\begin{split}
G_0(\alpha) & = \int_{-T_0}^{0} m(\Theta_t(\alpha)) \va{\Theta_{t,\xi}(\alpha)}^\delta \mathrm{d}t + \int_{0}^{T_1} m(\Theta_t(\alpha)) \va{\Theta_{t,\xi}(\alpha)}^\delta \mathrm{d}t \\ & \qquad \qquad \qquad \qquad \qquad - A \chi(\va{\alpha_\xi}) \va{\alpha_\xi\p{\widetilde{X}(\alpha_x)}}^\delta.
\end{split}
\end{equation}
The first two terms in \eqref{eq:ellipG} are symbols of order $\delta$ that do not depend on $A$. The last term is elliptic of order $\delta$ on a conical neighbourhood of $\mathscr{C}^0$ (since we assumed that $\mathscr{C}^0$ is closed and does not intersect $E_u^* \oplus E_s^* \setminus \set{0}$). Hence, if $A$ is large enough, $G_0$ is classically elliptic and negative on a conical neighbourhood of $\mathscr{C}^0$ in $T^* M$.  In fact, we see that this  conical neighbourhood may be chosen arbitrarily large as soon as it does not intersect $E_u^* \oplus E_s^* \setminus \set{0}$. Then, if $\alpha$ does not belong to this neighbourhood of $\mathscr{C}^0$, since it does not belong to $\mathscr{C}^s$ either, we see that $\alpha \notin \mathscr{C}^{0s}$ (provided that $\mathscr{C}^{0s}$ has been chosen narrow enough), and hence property d. above gives that, provided $\alpha_\xi$ is large enough,
\begin{equation*}
\begin{split}
\int_{0}^{T_1} m(\Theta_t(\alpha)) \va{\Theta_{t,\xi}(\alpha)}^\delta \mathrm{d}t & \leq \int_{T_0}^{T_1} m(\Theta_t(\alpha)) \va{\Theta_{t,\xi}(\alpha)}^\delta \mathrm{d}t \\
            & \leq - \int_{T_0}^{T_1} \va{\Theta_{t,\xi}(\alpha)}^\delta \mathrm{d}t \\
            & \leq - \frac{T_1 - T_0}{C_1^\delta} \va{\alpha_\xi}^\delta. 
\end{split}
\end{equation*}
Hence our choice of $T_1$ (see \eqref{eq:choix_de_T_1}) ensures that $G_0$ is negative and elliptic of order $\delta$ outside of $\mathscr{C}^{0s}$ (the last term in \eqref{eq:ellipG} is always non-positive). We just proved (i).

Now, we prove (ii). To do so we compute
\begin{equation}\label{eq:evolution}
\begin{split}
\set{G_0, \Re \tilde{p}} & = - \set{\Re \tilde{p},G_0} = -H_{\Re \tilde{p}}^{\omega_I} G_0 \\
   & = m(\Theta_{T_1}(\alpha)) \va{\Theta_{T_1,\xi}(\alpha)}^\delta - m(\Theta_{-T_0}(\alpha)) \va{\Theta_{-T_0,\xi}(\alpha)}^\delta \\ & \qquad \qquad \qquad + A H_{\Re \tilde{p}}^{\omega_I}\p{\chi(\va{\alpha_\xi}) \va{\Re\p{\alpha_\xi\p{\widetilde{X}(\alpha_x)}}}^\delta}.
\end{split}
\end{equation}
We start by considering the term
\begin{equation}\label{eq:premierterme}
\begin{split}
m(\Theta_{T_1}(\alpha)) \va{\Theta_{T_1,\xi}(\alpha)}^\delta - m(\Theta_{-T_0}(\alpha)) \va{\Theta_{-T_0,\xi}(\alpha)}^\delta.
\end{split}
\end{equation}
Assuming that $\alpha$ does not belong to $\mathscr{C}^0$, we know that $\Theta_{T_1,\xi}(\alpha)$ belongs to $\mathcal{C}_u^{\gamma_1}(\tilde{\phi}_{T_1}(\alpha_x))$ or $\Theta_{-T_0,\xi}(\alpha)$ belongs to $\mathcal{C}_s^{\gamma_1}(\tilde{\phi}_{-T_0}(\alpha_x))$. Let us assume for instance that $\Theta_{T_1,\xi}(\alpha)$ belongs to $\mathcal{C}_u^{\gamma_1}(\tilde{\phi}_{T_1}(\alpha_x))$ (the other case is symmetric). Then again there are two possibilities: either $\Theta_{-T_0,\xi}(\alpha)$ belongs to $\mathcal{C}_u^{\gamma}(\tilde{\phi}_{-T_0}(\alpha_x))$ or it does not. If it does then (for $\va{\alpha_\xi}$ large enough)
\begin{equation*}
\begin{split}
& m(\Theta_{T_1}(\alpha)) \va{\Theta_{T_1,\xi}(\alpha)}^\delta - m(\Theta_{-T_0}(\alpha)) \va{\Theta_{-T_0,\xi}(\alpha)}^\delta \\ & \qquad \qquad \leq - \va{\Theta_{T_1,\xi}(\alpha)}^\delta + \va{\Theta_{-T_0,\xi}(\alpha)}^\delta \\
    & \qquad \qquad \leq - \p{\lambda^\delta - 1} \va{\Theta_{-T_0,\xi}(\alpha)}^\delta \\
    & \qquad \qquad \leq - C^{-1} \va{\alpha_\xi}^\delta,
\end{split}
\end{equation*}
for some $C > 0$. If $\Theta_{-T_0,\xi}(\alpha)$ does not belong to $\mathcal{C}^\gamma_u$, then the situation is even simpler since the term $ - m(\Theta_{-T_0,\xi}(\alpha)) \va{\Theta_{-T_0,\xi}(\alpha)}^\delta$ is non-positive. Hence, the term \eqref{eq:premierterme} is negative and elliptic of order $\delta$ outside of $\mathscr{C}^0$. We focus now on the other term in \eqref{eq:evolution}. To do so, we introduce the Hamiltonian vector field $H_{\Im \tilde{p}}^{\omega_R}$ of $\Im \tilde{p}$ with respect to the real symplectic form $\omega_R$ and write:
\begin{equation}\label{eq:Leibniz}
\begin{split}
& H_{\Re \tilde{p}}^{\omega_I}\p{\chi(\va{\alpha_\xi}) \va{\Re\p{\alpha_\xi\p{\widetilde{X}(\alpha_x)}}}^\delta} \\ & \quad  = H_{\Re \tilde{p}}^{\omega_I}\p{\chi\p{\va{\alpha_\xi}}} \va{\Re\p{\alpha_\xi\p{\widetilde{X}(\alpha_x)}}}^\delta - \chi\p{\va{\alpha_\xi}} H_{\Im \tilde{p}}^{\omega_R}\p{\va{\Re\p{\alpha_\xi\p{\widetilde{X}(\alpha_x)}}}^\delta} \\
     & \qquad \qquad \qquad \qquad \qquad \qquad \qquad + \p{H_{\Im \tilde{p}}^{\omega_R} + H_{\Re \tilde{p}}^{\omega_I}}\p{\va{\Re\p{\alpha_\xi\p{\widetilde{X}(\alpha_x)}}}^\delta}.
\end{split}
\end{equation}
The first term in the right hand side of \eqref{eq:Leibniz} is compactly supported and may consequently be ignored. Since $\Re(\alpha_\xi(\widetilde{X}(\alpha_x))) = \Im \tilde{p}(\alpha)$, the second term is equal to $0$. To estimate the last term in \eqref{eq:Leibniz}, we work in local coordinates $(x+iy,\xi + i \eta)$. In these coordinates, we may compute $H_{\Im \tilde{p}}^{\omega_R}$ as we did for $H_{\Re \tilde{p}}^{\omega_I}$ in \eqref{eq:detaillonsunpeu}, and we find that
\begin{equation}\label{eq:Cragain}
\begin{split}
H_{\Im \tilde{p}}^{\omega_R} + H_{\Re \tilde{p}}^{\omega_I} & = \sum_{j=1}^n\p{ \jap{\xi,\frac{\partial \Im \widetilde{X}}{\partial y_j} - \frac{\partial \Re \widetilde{X}}{\partial x_j}} + \jap{\eta,\frac{\partial \Re \widetilde{X}}{\partial y_j}+ \frac{\partial \Im \widetilde{X}}{\partial x_j} }} \frac{\partial}{\partial \xi_j} \\
      & \qquad + \p{\jap{\xi,\frac{\partial \Re \widetilde{X}}{\partial y_j} + \frac{\partial \Im \widetilde{X}}{\partial x_j}} + \jap{\eta,\frac{\partial \Re \widetilde{X}}{\partial x_j} - \frac{\partial \Im \widetilde{X}}{\partial y_j} }} \frac{\partial}{\partial \eta_j}.
\end{split}
\end{equation}
Since $\widetilde{X}$ satisfies the Cauchy--Riemann equation on $M$, we find that the vector field $H_{\Im \tilde{p}}^{\omega_R} + H_{\Re \tilde{p}}^{\omega_I}$ vanishes on $T^* M$ (even on $T^* M \otimes \C$ in fact). It follows that outside of a compact set, the only non-zero term in the right hand side of \eqref{eq:evolution} is \eqref{eq:premierterme}. The property (ii) follows.

It remains to prove (iii). The analysis is based on \eqref{eq:evolution} again. We start by studying the term \eqref{eq:premierterme}. Let $\alpha \in \p{T^* M}_{\epsilon_2}$. If $\Theta_{T_1,\xi}(\alpha) \in \mathcal{C}^{\gamma_1}_u$ or $\Theta_{-T_0,\xi}(\alpha) \in \mathcal{C}_s^{\gamma_1}$, then the analysis from the proof of (ii) applies, and we see that \eqref{eq:premierterme} is non-positive for $\alpha_\xi$ large enough. Otherwise, $\Theta_{-T_0,\xi}(\alpha) \notin \mathcal{C}_u^{\gamma}$ and $\Theta_{T_1,\xi}(\alpha) \notin \mathcal{C}^{\gamma}_s$ and both terms in \eqref{eq:premierterme} are non-positive. Thus, the term \eqref{eq:premierterme} is always non-positive when $\va{\alpha_\xi}$ is large enough. Hence, this term is bounded from above and we may focus on the other term in \eqref{eq:evolution}, which is given by\eqref{eq:Leibniz}. The two first terms in the right hand side of \eqref{eq:Leibniz} are dealt with as in the proof of (ii): the first is compactly supported and the second is identically equal to $0$. To deal with the last one, we notice that, since $\widetilde{X}$ satisfies the Cauchy--Riemann equation at infinite order on $M$, the expression \eqref{eq:CRishere} for $H_{\Im \tilde{p}}^{\omega_R} + H_{\Re \tilde{p}}^{\omega_I}$ in coordinates imply that the last term in \eqref{eq:Leibniz} is a $\O(\va{\Im \alpha_x}^{\infty} \jap{\alpha}^\delta)$. This settles the case $s > 1$. To deal with the case $s = 1$, just notice that in that case $\widetilde{X}$ is holomorphic on $\smash{\widetilde{M}}$ and consequently $H_{\Im \tilde{p}}^{\omega_R} + H_{\Re \tilde{p}}^{\omega_I}= 0$.  
\end{proof}

Now, that we are equipped with a good escape function, we are in position to apply the tools from the previous chapter to study the spectral theory of $P = X+V$.

\subsection{Spectral theory for the generator of the flow}\label{sec:spectral_theory}

With the notations of the previous section, we set $\delta = 1/s$ and let $G_0$ be an escape function given by Lemma \ref{lemma:escape-function} (for arbitrary $\mathscr{C}^0$ and $\mathscr{C}^u$, we only assume that $\mathscr{C}^0$ is closed and does not intersect $E_u^* \oplus E_s^* \setminus \set{0}$). Then we define $G = \tau G_0$ and $\Lambda = e^{H_G^{\omega_I}} T^* M$ with $\tau = c \tau_0 h^{1 - 1/s}$, where $c$ and $\tau_0$ are small. Notice that if $c$ is small enough, then $\Lambda$ is a $\p{\tau_0,s}$-adapted Lagrangian (in the sense of Definition \ref{def:adapted-Lagrangian}), and hence the results from the previous chapter apply to the $\G^s$ semi-classical pseudor $hP$. We will not consider the asymptotic $h \to 0$, and hence we will assume that $h$ is fixed, small enough so that the results from the previous chapter apply. We will also assume that $\tau_0$ is small enough for the same reason.

We want now to study the spectral theory of the operator $P = X+V$ on the space $\mathcal{H}_\Lambda^0$ defined by \eqref{eq:def-HMk}. Notice that, according to Proposition \ref{lemma:boundedness}, if $h$ and $\tau_0$ are small enough, then the operator $hP$ (and hence $P$) is bounded from $\mathcal{H}_{\Lambda}^k$ to $\mathcal{H}_{\Lambda}^{k-1}$ for every $k \in \R$.

We will start by proving:
\begin{lemma}\label{lemma:well-defined}
$P$ defines a closed operator on $\mathcal{H}_{\Lambda}^0$ with domain
\begin{equation*}
\begin{split}
D(P) = \set{u \in \mathcal{H}_{\Lambda}^{0} : P u \in \mathcal{H}_{\Lambda}^0}.
\end{split}
\end{equation*}
\end{lemma}

In order to ensure that the spectral theory of $P$ on $\mathcal{H}_\Lambda^0$ has a dynamical meaning, we will then prove the following lemma.

\begin{prop}\label{prop:semigroup}
The operator $P$ is the generator of a strongly continuous semi-group $\p{\mathcal{L}_t}_{t \in \R}$ on $\mathcal{H}_\Lambda^0$. Moreover, if $t \geq 0$ and $u \in \mathcal{H}_\Lambda^0 \cap L^2\p{M}$ then $\mathcal{L}_t u$ is given by the expression \eqref{eq:transfer}.
\end{prop}

We will then prove the following key lemma that will be used with Proposition \ref{propvalsing} to prove that the resolvent of $P$ is in a Schatten class. 

\begin{lemma}[Hypo-ellipticity of $P$]\label{lmhypoell}
There is a constant $C > 0$ such that for every $u \in D(P)$ we have $u \in \mathcal{H}_{\Lambda}^{\delta}$ and
\begin{equation*}
\begin{split}
\n{u}_{\mathcal{H}_{\Lambda}^{\delta}} \leq C\p{\n{u}_{\mathcal{H}_{\Lambda}^{0}} + \n{Pu}_{\mathcal{H}_{\Lambda}^{0}}},
\end{split}
\end{equation*}
where we recall that we set $\delta = 1/s$.
\end{lemma}

Notice that in the case $s = 1$, in which we have $\delta = 1$, the operator $P$ is in fact elliptic on $\mathcal{H}_{\Lambda}^0$. With Proposition \ref{propvalsing}, we deduce then from Lemma \ref{lmhypoell} that $P$ has a good spectral theory on $\mathcal{H}_\Lambda^0$. More precisely, we have:

\begin{theorem}\label{thm:spectral-theory}
If $z$ is any element in the resolvent set of $P$, then the resolvent $(z-P)^{-1} : \mathcal{H}_\Lambda^0 \to \mathcal{H}_\Lambda^0$ is compact and if $\p{\sigma_k}_{k \geq 0}$ denotes the sequence of its singular values, we have
\begin{equation*}
\begin{split}
\sigma_k \underset{k \to + \infty}{=} \O\p{k^{-\frac{1}{ns}}}.
\end{split}
\end{equation*}
In particular, the operator $(z-P)^{-1}$ is in the Schatten class $\mathcal{S}_p$ for any $p > ns$. Consequently, $P$ has discrete spectrum on $\mathcal{H}_\Lambda^0$, and this spectrum is the Ruelle spectrum of $P$. The eigenvectors of $P$ acting on $\mathcal{H}_\Lambda^0$ are also the resonant states for $P$. If $N(R)$ denotes the number of Ruelle resonances of modulus less than $R$, we have
\begin{equation*}
\begin{split}
N(R) \underset{R \to + \infty}{=} \O\p{R^{n s}}.
\end{split}
\end{equation*}
\end{theorem}

Notice that Theorem \ref{thm:spectral-theory} implies Proposition \ref{prop:borne-resonances}. We will see that it also implies Proposition \ref{prop:wfs_resonant_state}. Theorem \ref{thm:main-strong} will be proved in the following section as a corollary of Theorem \ref{thm:spectral-theory}. The key tool in the proof of these results will be the multiplication formula Proposition \ref{propmultform}. We start by proving Lemma \ref{lemma:well-defined}.

\begin{proof}[Proof of Lemma \ref{lemma:well-defined}]
Let $(u_m)_{m \in \N}$ be a sequence of elements of $D(P)$ such that $(u_m)_{m \in \N}$ converges to some $u \in \mathcal{H}_{\Lambda}^0$ and $(Pu_m)_{m \in \N}$ converges to some $v \in \mathcal{H}_{\Lambda}^0$. According to Proposition \ref{lemma:boundedness}, $P$ is bounded from $\mathcal{H}_\Lambda^0$ to $\mathcal{H}_\Lambda^{-1}$ so that $(P u_m)_{m \in \N}$ tends to $Pu$ in $\mathcal{H}_\Lambda^{-1}$. Since $\mathcal{H}_\Lambda^0$ is continuously included in $\mathcal{H}_\Lambda^{-1}$, it follows that $Pu = v \in \mathcal{H}_\Lambda^0$, and hence that $P$ is closed.
\end{proof}

Then we prove an approximation lemma that will sometimes be useful (recall that $\mathcal{H}_\Lambda^\infty$ is defined by \eqref{eq:def_H_infini}).

\begin{lemma}[Approximation Lemma]\label{lmapprox}
Let $R_0$ be large enough. Then, the domains $D(P)$ and $D(P^*)$ of $P$ and its adjoint on $\mathcal{H}_\Lambda^0$ both contain $E^{1,R_0}$. Moreover, for every $u \in D(P)$ there is a sequence of $\p{u_n}_{n \in \N}$ of elements of $\mathcal{H}_\Lambda^\infty$ that tends to $u$ in $\mathcal{H}_{\Lambda}^0$ and such that $\p{Pu_n}_{n \in \N}$ tends to $Pu$ in $\mathcal{H}_{\Lambda}^0$. The same is true if we replace $P$ by $P^*$.
\end{lemma}

\begin{proof}
First of all, $D(P)$ contains $\mathcal{H}_\Lambda^1$, and hence $E^{1,R_0}$ for $R_0$ large enough according to Corollary \ref{cor:completude}.

Let $\chi$ be a $\mathcal{C}^\infty$ function from $\R$ to $\left[0,1\right]$ compactly supported and identically equals to $1$ near $0$. Then define for $\epsilon > 0$ small the function $\chi_\epsilon$ on $\Lambda$ by $\chi_\epsilon(\alpha) = \chi(\epsilon \jap{\va{\alpha}})$. Then the $\chi_\epsilon$'s form a family of uniform symbols of order $0$ on $\Lambda$. Hence, the operators
\begin{equation*}
\begin{split}
I_\epsilon = S_\Lambda B_\Lambda \chi_\epsilon B_\Lambda T_\Lambda
\end{split}
\end{equation*}
are uniformly bounded on $\mathcal{H}_{\Lambda}^{0}$. One easily checks that $I_\epsilon$ converges to the identity in strong operator topology when $\epsilon$ tends to $0$. In particular, if $u \in D(P)$, we have that $I_\epsilon u$ tends to $u$ in $\mathcal{H}_{\Lambda}^{0}$ when $\epsilon$ tends to $0$. Moreover, we have
\begin{equation}\label{eqbete}
\begin{split}
P I_\epsilon u = I_\epsilon P u + \left[P,I_\epsilon\right] u
\end{split}
\end{equation}
and $I_\epsilon P u$ tends to $Pu$ in $\mathcal{H}_{\Lambda}^{0}$ when $\epsilon$ tends to $0$. Then, we may write
\begin{equation*}
\begin{split}
\left[P,I_\epsilon\right] = S_\Lambda\left[B_\Lambda T_\Lambda P S_\Lambda B_\Lambda, B_\Lambda \chi_\epsilon B_\Lambda\right]T_\Lambda.
\end{split}
\end{equation*}
Then using Propositions \ref{proptoeplitz} and \ref{propcomptoep}, we see that the $\left[P,I_\epsilon\right]$'s are uniformly bounded on $\mathcal{H}_{\Lambda}^{0}$. But if $v \in E^{1,R_0}$, then $T_\Lambda u$ is rapidly decaying and thus $I_\epsilon u$ tends to $u$ in $\mathcal{H}_{\Lambda}^{1}$ and since $P$ is bounded from $\mathcal{H}_{\Lambda}^{1}$ to $\mathcal{H}_{\Lambda}^{0}$, we see that $\left[P,I_\epsilon\right]v$ tends to $0$ in $\mathcal{H}_{\Lambda}^{0}$. By a standard argument, this implies that $\left[P,I_\epsilon\right]$ tends to $0$ in strong operator topology on $\mathcal{H}_{\Lambda}^{0}$. Hence, from \eqref{eqbete}, we see that $P I_\epsilon u$ tends to $Pu$ in $\mathcal{H}_{\Lambda}^{0}$. Finally, thanks to Remark \ref{remark:boundedness}, the $I_\epsilon u$'s belong to $\mathcal{H}_\Lambda^\infty$.

We turn now to the study of the adjoint $P^*$ of $P$. To do so, we apply Proposition \ref{proptoeplitz} to find a symbol $\sigma$ of order $1$ on $\Lambda$ and a negligible operator $R$ such that 
\begin{equation*}
\begin{split}
B_\Lambda T_\Lambda P S_\Lambda B_\Lambda = B_\Lambda\p{\sigma + R} B_\Lambda.
\end{split}
\end{equation*}
Then define the operator $\widetilde{P}$ by 
\begin{equation*}
\begin{split}
\widetilde{P} = S_\Lambda B_\Lambda \p{\bar{\sigma} + R^*} B_\Lambda T_\Lambda,
\end{split}
\end{equation*}
where $R^*$ is the formal adjoint of $R$. Since $\bar{\sigma}$ is a symbol of order $1$, we find that $\widetilde{P}$ is bounded from $\mathcal{H}_\Lambda^k$ to $\mathcal{H}_\Lambda^{k-1}$ for every $k \in \R$. Now, if $u,v \in \mathcal{H}_\Lambda^\infty$ then we have (the scalar product is in $\mathcal{H}_\Lambda^0$)
\begin{equation}\label{eq:construire_ladjoint}
\begin{split}
\jap{Pu,v} = \jap{u, \widetilde{P}v}.
\end{split}
\end{equation} 
Using the approximation property for $P$ that we just proved, we find that \eqref{eq:construire_ladjoint} remains true for $u \in D(P)$ and $v \in \mathcal{H}_\Lambda^\infty$. Hence, $\mathcal{H}_\Lambda^\infty$ (and in particular $E^{1,R_0}$) is contained in the domain of $P^*$ and if $v \in  \mathcal{H}_\Lambda^\infty$ then $P^*v = \widetilde{P} v$. Now, if $v$ belongs to the domain of $P^*$, we find that, for $u \in \mathcal{H}_\Lambda^\infty$,
\begin{equation}\label{eq:reconnaitre_ladjoint}
\begin{split}
\jap{u,P^*v} = \jap{Pu,v} = \jap{u, \widetilde{P}v}.
\end{split}
\end{equation}
Here, the last bracket makes sense because $u \in \mathcal{H}_\Lambda^{1}$ and $\widetilde{P}v \in \mathcal{H}_\Lambda^{-1}$. Since $E^{1,R_0}$ (and hence $\mathcal{H}_\Lambda^\infty$) is dense in $\mathcal{H}_\Lambda^{1}$, Lemma \ref{lemma:dual} implies with \eqref{eq:reconnaitre_ladjoint} that $P^* v = \widetilde{P}v$ (in particular $\widetilde{P}v$ belongs to $\mathcal{H}_\Lambda^0$). Now, since $\widetilde{P}$ has the same structure as $P$, we may prove the approximation property for $P^*$ as we did for $P$.
\end{proof}

\begin{proof}[Proof of Proposition \ref{prop:semigroup}]
We will apply the Hille--Yosida Theorem to prove that $P$ is the generator of a strongly continuous semi-group. We denote by $p_\Lambda$ the restriction to $\Lambda$ of the almost analytic extension $\tilde{p}$ of the principal symbol of $hP$ given by \eqref{eq:defptilde}. It follows from (iii) in Lemma \ref{lemma:escape-function} that there is a constant $C > 0$ such that $\Re p_\Lambda \leq C$. Indeed, since $\Re \tilde{p}$ vanishes on $T^* M$, the value of $\Re p_\Lambda$ is obtained by integrating $\set{G_0,\Re \tilde{p}}$ on an orbit of time $\tau$ of the flow of $H_{G_0}^{\omega_I}$. This gives the upper bound on $\Re p_\Lambda$ in the case $s=1$. In the case $s >1$, we need in addition to notice that, as in Lemma \ref{lmdist}, we remain at distance at most $\jap{\va{\alpha}}^{\delta - 1}$ of the reals. Hence, $\Re p_\Lambda$ is less than $C_N \jap{\va{\alpha}}^{\delta + N(\delta - 1)}$ (for any $N > 0$ and some constant $C_N > 0$). Since $\delta < 1$ in the case $s > 1$, the bound on $\Re p_\Lambda$ follows by taking $N$ large enough.

We see by Proposition \ref{propmultform} that, up to making $C$ larger (depending on $h$, that we recall is fixed), we have for $u \in \mathcal{H}_\Lambda^\infty$ 
\begin{equation*}
\begin{split}
\re \langle -P u,u \rangle_{\mathcal{H}_{\Lambda}^0} & \geq -\frac{1}{h}\langle \Re p_\Lambda T_\Lambda P u ,T_\Lambda u \rangle_{L_0^2\p{\Lambda}} - C \n{u}_{\mathcal{H}_{\Lambda}^{0}}^2 \\
    & \geq - 2C \n{u}_{\mathcal{H}_{\Lambda}^{0}}^2.
\end{split}
\end{equation*}
Hence, if $z \in \C$, we have
\begin{equation*}
\begin{split}
\re \langle \p{z-P} u,u \rangle_{\mathcal{H}_{\Lambda}^{0}} & \geq \p{\Re z - 2C} \n{u}_{\mathcal{H}_{\Lambda}^{0}}^2.
\end{split}
\end{equation*}
By Cauchy--Schwarz, we find that
\begin{equation}\label{eqinversion}
\begin{split}
\n{\p{z-P}u}_{\mathcal{H}_{\Lambda}^{0}} & = \frac{\n{\p{z-P}u}_{\mathcal{H}_{\Lambda}^{0}} \n{u}_{\mathcal{H}_{\Lambda}^{0}}}{\n{u}_{\mathcal{H}_{\Lambda}^{0}}} \geq \frac{\va{\langle \p{z-P} u,u \rangle_{\mathcal{H}_{\Lambda}^{0}}}}{\n{u}_{\mathcal{H}_{\Lambda}^{0}}} \\
   & \geq \frac{\re\langle \p{z-P} u,u \rangle_{\mathcal{H}_{\Lambda}^{0}}}{\n{u}_{\mathcal{H}_{\Lambda}^{0}}} \geq \p{\Re z - 2C} \n{u}_{\mathcal{H}_{\Lambda}^{0}},
\end{split}
\end{equation}
for $u \in \mathcal{H}_\Lambda^\infty$. By Lemma \ref{lmapprox}, this estimate remains true when $u \in D(P)$. This proves that if $\Re z > 2C$, then the operator $z-P$ is injective and its image is closed. To prove that the image of $z-P$ is dense, notice that if $u \in \mathcal{H}_\Lambda^\infty$  then
\begin{equation*}
\begin{split}
\Re \jap{(z-P)^*u,u}_{\mathcal{H}_\Lambda^0}  = \Re\jap{(z-P)u,u}_{\mathcal{H}_{\Lambda}^0},
\end{split}
\end{equation*}
and consequently \eqref{eqinversion} still holds when $z-P$ is replaced by $(z-P)^*$ (for $u \in \mathcal{H}_\Lambda^\infty$, but it implies the same result for $u \in D(P^*)$ by Lemma \ref{lmapprox}). Hence, $(z-P)^*$ is injective, and thus the image of $z-P$ is closed. Consequently, $z-P$ is invertible and from \eqref{eqinversion}, we see that
\begin{equation*}
\begin{split}
\n{\p{z- P}^{-1}} \leq \frac{1}{\Re z - 2C},
\end{split}
\end{equation*}
for the operator norm on $\mathcal{H}_{\Lambda}^{0}$. Hence, the Hille--Yosida Theorem applies (the domain of $P$ is dense since it contains $E^{1,R_0}$), and we know that $P$ is the generator of a strongly continuous semi-group. 

Denote by $(\smash{\widetilde{\mathcal{L}}_t})_{t \geq 0}$ the semi-group generated by $P$ on $\mathcal{H}_\Lambda^0$ and $\p{\mathcal{L}_t}_{t \geq 0}$ the semi-group on $L^2\p{M}$ defined by \eqref{eq:transfer}. We want to prove that for $t \geq 0$ and $u \in \mathcal{H}_\Lambda^0 \cap L^2\p{M}$ we have $\mathcal{L}_t u = \smash{\widetilde{\mathcal{L}}_t} u$. Thanks to the semi-group property, we only need to prove it for $t \in \left[0,t_0\right]$ for some small $t_0 > 0$. Then, since elements of $L^2\p{M} \cap \mathcal{H}_\Lambda^0$ may be simultaneously approximated in $L^2\p{M}$ and in $\mathcal{H}_\Lambda^0$ by elements of $E^{1,R_0}$ (according to Corollary \ref{corollary:densite}), we only need to prove the equality for $u \in E^{1,R_0}$. Now, there is a $t_0 > 0$ and a $R_1 > 0$ such that for $u \in E^{1,R_0}$, the curve $\gamma : \left[0,t_0\right] \ni t \mapsto \mathcal{L}_t u$ is $\mathcal{C}^1$ in $E^{s,R_1}$ with $\gamma'(t) = P \gamma(t)$. Provided that $\tau_0$ is small enough, $E^{s,R_1}$ is continuously included in $\mathcal{H}_\Lambda^0$ (see Corollary \ref{cor:completude}) and hence the curve $\gamma$ has the same property in $\mathcal{H}_\Lambda^0$. Consequently, we have $\gamma(t) = \widetilde{\mathcal{L}}_t u$ for $t \in \left[0,t_0\right]$, according to \cite[Proposition 3.1.11]{arendtVectorvaluedLaplaceTransforms2011a}, ending the proof of the proposition.
\end{proof}

We turn now to the proof of the hypo-ellipticity of the operator $P$.

\begin{proof}[Proof of Lemma \ref{lmhypoell}]
Assume first that $u \in \mathcal{H}_\Lambda^\infty$. Let $\chi_+,\chi_-$ and $\chi_0$ be $\mathcal{C}^\infty$ functions from $\R \to \left[0,1\right]$ such that $\chi_+ + \chi_0 + \chi_- = 1$, and, for some small $\eta > 0$, the function $\chi_0$ is supported in $\left[- \eta,\eta\right]$, the function $\chi_-$ is supported in $\left]- \infty, -\eta/2\right]$ and the function $\chi_{+}$ is supported in $\left[\eta/2, + \infty \right[$. Then define for $\sigma \in \set{+,-,0}$ the symbol $f_\sigma$ on $\Lambda$ by
\begin{equation*}
\begin{split}
f_\sigma(\alpha) = \chi_{\sigma}\p{-i \frac{p\p{e^{-\tau H_{G_0}^{\omega_I}}(\alpha)}}{\jap{\va{\alpha}}}}.
\end{split}
\end{equation*}
Then notice that if $\alpha$ is in the support of $f_+$ then we have
\begin{equation}\label{eqimfplus}
\begin{split}
\Im p_\Lambda(\alpha) &  = \Im p\p{e^{-\tau H_{G_0}^{\omega_I}}(\alpha)} + \O\p{\tau \jap{\va{\alpha}}} \\
     & \geq \frac{\eta}{4} \jap{\va{\alpha}},
\end{split}
\end{equation}
provided that $\tau$ is small enough (depending on $\eta$). And similarly, if $\alpha$ belongs to the support of $f_-$ we have
\begin{equation}\label{eqimfmoins}
\begin{split}
\Im p_\Lambda(\alpha) \leq - \frac{\eta}{4} \jap{\va{\alpha}}.
\end{split}
\end{equation}
If $\alpha$ belongs to the support of $f_0$ then we have 
$$
\va{p\p{e^{-\tau H_{G_0}}(\alpha)}} \leq \eta \jap{\va{\alpha}} \leq C \eta \jap{\va{e^{- \tau H_{G_0}}(\alpha)}}.
$$
Hence, either $e^{-\tau H_{G_0}}(\alpha)$ is small, either it does not belong to $\mathscr{C}^0$ (provided that $\eta$ is small enough, we use here the assumption that $\mathscr{C}^0$ does not intersect $E_u^* \oplus E_s^*$). In the second case, we may apply (ii) in Lemma \ref{lemma:escape-function} to find that
\begin{equation}\label{eqimf0}
\begin{split}
\Re p_\Lambda(\alpha) & = \Re p\p{e^{- \tau H_{G_0}^{\omega_I}}(\alpha)} + \tau \set{G_0, \Re \tilde{p}} + \O\p{\tau^2 \jap{\va{\alpha}}^{2 \delta-1 }} \\
      & \leq - \frac{1}{C} \jap{\va{\alpha}}^\delta + C,
\end{split}
\end{equation}
provided that $\tau$ is small enough. Here, we added the constant $C$ so that \eqref{eqimf0} remains true for any $\alpha$ in the support $f_0$.

Now, Proposition \ref{propmultform} and \eqref{eqimf0} give that (the constant $C$ may vary from one line to another)
\begin{equation*}
\begin{split}
& - \re \langle \jap{\va{\alpha}}^\delta f_0 T_\Lambda P u , T_\Lambda u \rangle \\ & \qquad \qquad \qquad \geq - \frac{1}{h}\int_{\Lambda} \jap{\va{\alpha}}^\delta f_0(\alpha) \Re p_\Lambda(\alpha) \va{T_\Lambda u(\alpha)}^2 \mathrm{d}\alpha - C \n{u}_{\mathcal{H}_{\Lambda}^{0}} \n{u}_{\mathcal{H}_{\Lambda}^{\delta}} \\
    & \qquad \qquad \qquad \geq \frac{1}{C} \int_{\Lambda} f_0(\alpha) \jap{\va{\alpha}}^{2\delta} \va{T_\Lambda u(\alpha)}^2 \mathrm{d}\alpha - C \n{u}_{\mathcal{H}_{\Lambda}^{0}} \n{u}_{\mathcal{H}_{\Lambda}^{\delta}}.
\end{split}
\end{equation*}
Applying Cauchy--Schwarz formula, we find then that
\begin{equation}\label{eqhypzero}
\begin{split}
\int_{\Lambda} f_0(\alpha) \jap{\va{\alpha}}^{2\delta} \va{T_\Lambda u(\alpha)}^2 \mathrm{d}\alpha \leq C \n{u}_{\mathcal{H}_{\Lambda}^{\delta}}\p{\n{Pu}_{\mathcal{H}_{\Lambda}^{0}} + \n{u}_{\mathcal{H}_{\Lambda}^{0}}}.
\end{split}
\end{equation}
Working similarly with \eqref{eqimf0} replaced by the better estimates \eqref{eqimfplus} and \eqref{eqimfmoins}, we find that \eqref{eqhypzero} still holds when $f_0$ is replaced by $f_+$ or $f_-$. Summing these three estimates, we get
\begin{equation}\label{eqhypcarre}
\begin{split}
\n{u}_{\mathcal{H}_{\Lambda}^{\delta}}^2 \leq C \n{u}_{\mathcal{H}_{\Lambda}^{\delta}}\p{\n{Pu}_{\mathcal{H}_{\Lambda}^{0}} + \n{u}_{\mathcal{H}_{\Lambda}^{0}}}.
\end{split}
\end{equation}
Since the result is trivial when $u = 0$, we may divide by $\n{u}_{\mathcal{H}_{\Lambda}^{\delta}}$ in \eqref{eqhypcarre} to end the proof of the lemma when $u \in \mathcal{H}_\Lambda^\infty$.

We deal now with the general case $u \in D(P)$. Let $\p{u_n}_{n \in \N}$ be a sequence of elements of $\mathcal{H}_\Lambda^\infty$ as in Lemma \ref{lmapprox}. Since we already dealt with the case of elements of $\mathcal{H}_\Lambda^\infty$ we know that for some $C > 0$ and all $n \in \N$ we have
\begin{equation*}
\begin{split}
\n{u_n}_{\mathcal{H}_{\Lambda}^{\delta}} \leq C\p{\n{u}_{\mathcal{H}_{\Lambda}^{0}} + \n{Pu}_{\mathcal{H}_{\Lambda}^{0}}}.
\end{split}
\end{equation*}
In addition, $T_{\Lambda} u_n$ converges pointwise to $T_\Lambda u$ and hence the result follows by Fatou's Lemma.
\end{proof}

We can now prove Theorem \ref{thm:spectral-theory}.

\begin{proof}[Proof of Theorem \ref{thm:spectral-theory}]
Let $z$ be any element of the resolvent set of $P$ (the resolvent set of $P$ is non-empty according to Proposition \ref{prop:semigroup}). If $u \in \mathcal{H}_\Lambda^0$ then we have that
\begin{equation*}
\begin{split}
P(z-P)^{-1} u = z(z-P)^{-1} u - u.
\end{split}
\end{equation*}
Hence, $(z-P)^{-1}$ and $P(z-P)^{-1}$ are both bounded from $\mathcal{H}_\Lambda^0$ to itself and, consequently, Lemma \ref{lmhypoell} implies that $\p{z-P}^{-1}$ is bounded from $\mathcal{H}_\Lambda^0$ to $\mathcal{H}_\Lambda^\delta$. Hence, Proposition \ref{propvalsing} implies that $\p{z-P}^{-1}$, as an operator from $\mathcal{H}_\Lambda^0$ to itself, is compact, with the announced estimates on its singular values.

We prove the estimates on the number of eigenvalues of $P$ before proving that these eigenvalues are indeed the Ruelle resonances. Let $z \in \C$ be any point in the resolvent set of $P$ and denote by $\widetilde{N}(R)$ the number of eigenvalues of $(z-P)^{-1}$ of modulus larger than $R^{-1}$. Then let $(\mu_k)_{k \in \N}$ denote the sequence of eigenvalues of $(z-P)^{-1}$, ordered so that $(|\mu_k|)_{k \in \N}$ is decreasing, and $(\sigma_k)_{k \in \N}$ the sequence of its singular values, and choose $p > 0$ such that $\delta p/n < 1$. According to \cite[Corollary IV.3.4]{Gohb}, we have then for every $R > 0$ that
\begin{equation*}
\begin{split}
\frac{\widetilde{N}(R)}{R^p} & \leq \sum_{k =0}^{\widetilde{N}(R) - 1} \va{\mu_k}^p \leq \sum_{k = 0}^{\widetilde{N}(R) - 1} \sigma_k^p \leq C \sum_{k = 0}^{\widetilde{N}(R) - 1} \p{1+k}^{- \frac{\delta p}{n}} \\
     & \leq C\widetilde{N}(R)^{1 - \frac{\delta p}{n}}.
\end{split}
\end{equation*}
Here, we applied the estimates on singular values that we just proved, and $C$ may vary from one line to another. It follows that $\widetilde{N}(R) \leq C^{n/\delta p} R^{n/\delta}$. The estimates on $N(R)$ follows since, if $\p{\lambda_k}_{k \in \N}$ denotes the sequence of eigenvalues of $P$, we have the relation $\mu_k = (z-\lambda_k)^{-1}$ (up to reordering, and recall that $\delta = 1/s$).

It remains to prove that the eigenvalues of $P$ acting on $\mathcal{H}_\Lambda^0$ are indeed the Ruelle resonances of $P$. To do so let $R_0 > 0$ be large enough, and assume that $\tau_0$ is small enough, so that \eqref{eq:des_inclusions} holds with $s = 1$ and that $E^{1,R_0}$ is dense in $\mathcal{H}_\Lambda^0$ (see Corollary \ref{corollary:densite}). We also assume that $R_0$ is large enough so that $E^{1,R_0}$ is dense in $\mathcal{C}^\infty\p{M}$ (see Corollary \ref{cor:density-E1R}). Denote by $i_0$ the inclusion of $E^{1,R_0}$ in $\mathcal{H}_\Lambda^0$ and by $i_1$ the inclusion of $\mathcal{H}_\Lambda^0$ in $(E^{1,R_0})'$. Then, we define
\begin{equation*}
\begin{split}
\widetilde{R}(z) = i_1 \circ (z-P)^{-1} \circ i_0 : E^{1,R_0} \to \p{E^{1,R_0}}'.
\end{split}
\end{equation*}

Then, if we denote by $i$ the inclusion of $E^{1,R_0}$ into $\mathcal{C}^\infty\p{M}$ and by $j$ the inclusion of $\mathcal{D}'\p{M}$ into $(E^{1,R_0})'$, we see that
\begin{equation}\label{eq:ce_sont_les_memes}
\begin{split}
\widetilde{R}(z) = j \circ R(z) \circ i,
\end{split}
\end{equation}
where $R(z)$ is defined by \eqref{eq:a_quoi_doit_ressembler_une_resolvante} and seen as an operator from $\mathcal{C}^\infty\p{M}$ to $\mathcal{D}'\p{M}$. Indeed, when $\Re z$ is large \eqref{eq:ce_sont_les_memes} follows from Proposition \ref{prop:semigroup} since $\widetilde{R}(z)$ and $R(z)$ are both obtained as the Laplace transform of the family of operators \eqref{eq:transfer}. The equality \eqref{eq:ce_sont_les_memes} then follows for any $z$ by analytic continuation. Integrating on small circles, we see that \eqref{eq:ce_sont_les_memes} is also satisfied by the residues of $\widetilde{R}(z)$ and $R(z)$. Since these residues have finite rank and since $E^{1,R_0}$ is dense in $\mathcal{C}^\infty\p{M}$, it follows that the eigenvalues of $P$ on $\mathcal{H}_\Lambda^0$ (the poles of $\widetilde{R}(z)$) are the Ruelle resonances of $P$ (the poles of $R(z)$) counted with multiplicity (the rank of the associated residues). For the same reason, the resonant spaces (the images of the residues) also coincide.
\end{proof}

As announced, we can now give the proof of Proposition \ref{prop:wfs_resonant_state}.

\begin{proof}[Proof of Proposition \ref{prop:wfs_resonant_state}]
Since the eigenvectors of $P$ acting on $\mathcal{H}_\Lambda^0$ are the resonant states of $P$, they do not depend on $G_0$ (as soon as $G_0$ is as in Lemma \ref{lemma:escape-function}) nor on $\tau$ (provided that $\tau$ is small enough). Hence, if $\mathscr{C}^s$ is a conic neighbourhood of $E_s^*$ in $T^* M$, we may choose $G_0$ negative and elliptic of order $\delta = 1/s$ outside of $\mathscr{C}^s$. Then, taking $\tau$ small enough, we may apply Lemma \ref{lemma:wavefrontset_HLambda} to see that if $u$ is a resonant state for $P$ then $\WF_{\G^{s}}\p{u} \subseteq \mathscr{C}^s$. Since $\mathscr{C}^s$ is an arbitrary conic neighbourhood of $E_s^*$, it follows that $\WF_{\G^{s}}(u) \subseteq E_s^*$.
\end{proof}

\section{Traces and I-Lagrangian spaces}\label{sec:order}

At the heart of many ``microlocal'' results lie a trace formula, that links (purely) spectral information to geometric or dynamical data. Theorems \ref{thm:main} and \ref{thm:main-strong} are no exception. Their proof will be given in this section, and this will give a glimpse of the relation between I-Lagrangian spaces and traces.

In order to show that $\zeta_{X,V}$ has finite order (under our Gevrey assumption), we will relate it to a regularized determinant associated with the resolvent of $P$. This will be based on the following version of Guillemin's trace formula:
\begin{lemma}\label{lmtraceresolvante}
If the real part of $z$ is positive and large enough and $m$ is an integer such that $m > s n$, then the operator $\p{z-P}^{-m}$ acting on $\mathcal{H}_{\Lambda}^{0}$ is trace class and
\begin{equation}\label{eqformtrace}
\begin{split}
\Tr \p{(z-P)^{-m}} = \frac{1}{\p{m-1}!} \sum_{\gamma} \frac{T_\gamma^{\#} e^{\int_\gamma V}}{\va{\det\p{I - \mathcal{P}_\gamma}}} T_\gamma^{m-1} e^{- z T_\gamma}.
\end{split}
\end{equation}
Here, we use the notations defined in the introduction of the chapter (after \eqref{eq:zetaXV}).
\end{lemma}

The fundamental reason for which this result holds is that whenever we can give a reasonable meaning to $\Tr_{|\mathcal{H}^0_\Lambda} (z-P)^{-m}$, it formally does not depend on $\Lambda$. Indeed, this trace should coincide with the ``flat trace'' of $(z- P)^{-m}$ given by formal integration of it Schwartz kernel on the diagonal of $M \times M$. See \cite{DZdet} for an extensive discussion of the notion of flat trace in the context of Anosov flows.

\begin{proof}
That $(z-P)^{-m}$ is trace class results from Theorem \ref{thm:spectral-theory}. For $\Re z$ large enough, the convergence of the right hand side of \eqref{eqformtrace} is provided by Margulis' bound \cite{margulisAspectsTheoryAnosov2004} on the number of closed geodesics for an Anosov flow (see also \cite[Lemma 2.2]{DZdet}).

First, recall the semi-group $\p{\mathcal{L}_t}_{t \geq 0}$ generated by $P$ (see Proposition \ref{prop:semigroup}), given by the formula \eqref{eq:transfer} on $\mathcal{H}_\Lambda^0 \cap L^2\p{M}$. It follows from the semi-group property and from \cite[Theorem IV.5.5]{Gohb} that
\begin{equation*}
\begin{split}
\Tr_{\mathcal{H}^0_\Lambda} (z-P)^{-m} = \lim_{\rho \to 0} \Tr_{\mathcal{H}^0_\Lambda} \mathcal{L}_\rho (z-P)^{-m}.
\end{split}
\end{equation*}
Thus, \eqref{eqformtrace} will follow if we can prove that for $\rho > 0$ small enough we have
\begin{equation*}
\begin{split}
\Tr_{\mathcal{H}^0_\Lambda} \mathcal{L}_\rho (z-P)^{-m} = \frac{1}{\p{m-1}!} \sum_{\gamma} \frac{T_\gamma^{\#} e^{\int_\gamma V}}{\va{\det\p{I - \mathcal{P}_\gamma}}} \p{T_\gamma - \rho}_+^{m-1} e^{- z \p{T_\gamma- \rho}},
\end{split}
\end{equation*}
where $x_+ = \max(x,0)$ denotes the positive part of a real number $x$. Indeed, the convergence of the right hand side when $\rho$ tends to zero is obtained by dominated convergence.

Let then $\rho > 0$ be small enough. Then, as in the proof of Corollary \ref{cor:density-E1R}, introduce the operators $I_\epsilon = S \exp(-\epsilon \langle\alpha\rangle^2) T$ for $\epsilon > 0$. These are regularizing operators that map $(E^{1,R_0})'$ to $E^{1,R_0}$ for $R_0>0$ large enough. In particular, $I_\epsilon \mathcal{L}_\rho (z-P)^{-m} I_\epsilon$ is of trace class on $L^2$ and $\mathcal{H}^0_\Lambda$ (actually on any reasonable space where $\mathcal{L}_\rho (z-P)^{-m}$ is bounded). Here, it makes sense to discuss the operator $I_\epsilon \mathcal{L}_\rho (z-P)^{-m} I_\epsilon$ acting on $L^2$ or $\mathcal{H}^0_\Lambda$ because $I_\epsilon$ is valued in $E^{1,R_0}$. Moreover, thanks to Proposition \ref{prop:semigroup}, we may use the expression \eqref{eq:transfer} for $\mathcal{L}_t$ and then $\mathcal{L}_\rho (z-P)^{-m}$ writes for $\Re z \gg 1$
\begin{equation}\label{eq:une_forme_integrale}
\begin{split}
\mathcal{L}_\rho (z-P)^{-m} = \frac{1}{\p{m-1}!} \int_{0}^{+ \infty} e^{-z (t- \rho)} \p{t - \rho}_+^{m - 1} \mathcal{L}_t \mathrm{d}t.
\end{split}
\end{equation}

On the other hand, $\p{I_\epsilon}_{\epsilon > 0}$ is a family of $\G^1$ pseudors, uniformly of order $0$ as $\epsilon\to 0$ (this is an immediate consequence of the proof of Lemma \ref{lemma:ST-pseudo} in which we used the same regularization procedure to study the composition $ST$), and they converge strongly to the identity when $\epsilon$ tends to $0$, both as operators on $L^2\p{M}$ and $\mathcal{H}_\Lambda^0$ (or any Sobolev space, see Corollary \ref{corollary:sobolev} and the proof of Corollary \ref{corollary:densite}). We deduce then from \cite[Theorem IV.5.5]{Gohb} that
\[
\Tr_{\mathcal{H}^0_\Lambda} \mathcal{L}_\rho (z-P)^{-m} = \lim_{\epsilon\to 0} \Tr_{\mathcal{H}^0_\Lambda} (I_\epsilon \mathcal{L}_\rho (z-P)^{-m} I_\epsilon).
\]

We consider now $u$ a generalized eigenvector of $I_\epsilon \mathcal{L}_\rho(z-P)^{-m} I_\epsilon$, associated to an eigenvalue $\lambda\neq 0$. Since $I_\epsilon$ maps $\mathcal{H}^0_\Lambda$ continuously in $E^{1,R_0}$ for some $R_0>0$, we deduce that if $u\in \mathcal{H}^0_\Lambda$, then $u \in E^{1,R_0}$. In particular, $u\in L^2$. Reciprocally, if $u\in L^2$, then $u\in \mathcal{H}^0_\Lambda$. Since $I_\epsilon \mathcal{L}_\rho(z-P)^{-m} I_\epsilon$ has the same eigenvectors and eigenvalues on $L^2$ and $\mathcal{H}^0_\Lambda$, according to Lidskii's theorem, its trace is the same on both spaces, so that we have
\begin{equation*}
\begin{split}
\Tr_{\mathcal{H}^0_\Lambda} \mathcal{L}_\rho (z-P)^{-m} = \lim_{\epsilon\to 0} \Tr_{L^2} (I_\epsilon \mathcal{L}_\rho (z-P)^{-m} I_\epsilon).
\end{split}
\end{equation*}

In order to compute the limit when $\epsilon$ tends to $0$ of the trace of $I_\epsilon \mathcal{L}_\delta(z-P)^{-m} I_\epsilon$, we will use the notion of flat trace. We refer to \cite{DZdet} and references therein for basic properties of this object and insightful discussion of its properties in the context of Anosov flows.

Let us show that the flat trace of $\mathcal{L}_\rho(z-P)^{-m}$ is well-defined. Recall that the flat trace of an operator $A : \mathcal{C}^\infty\p{M} \to \mathcal{D}'\p{M}$ is defined if the intersection between its wave front set $WF'(A)$ (defined in \cite[C.1]{DZdet}) and the conormal to the diagonal
\begin{equation*}
\begin{split}
\Delta(T^\ast M) = \{ (x,\xi;x,\xi)\ |\ (x,\xi)\in T^\ast M \} \subseteq T^* M \times T^* M
\end{split}
\end{equation*}
is empty. Then, according to \cite[Proposition 3.3]{DZdet}, the wave front set of $(z - P)^{-1}$ is contained in 
\begin{equation}\label{eq:localisastion_WF_resolvante}
\begin{split}
\Delta(T^\ast M) \cup \Omega_- \cup E^\ast_s \times E^\ast_u,
\end{split}
\end{equation}
where
\[
\Omega_- = \{ (x,\xi; \phi_t(x),{}^T \mathrm{d}\phi_t(x) ^{-1} \xi))\ |\ t\geq 0,\ \xi\p{X(x)} = 0 \}.
\]
(we have adjusted the signs because they study the transfer operator instead of the Koopman operator). Then, by \cite[Proposition E.40]{dyatlovMathematicalTheoryScattering2019}, we see that the wave front set of $\p{z - P}^{-m}$ is also contained in \eqref{eq:localisastion_WF_resolvante}. Then, from \cite[Theorem 8.2.4]{hormanderAnalysisLinearPartial2003}, we know that the wave front set of $\mathcal{L}_\rho \p{z - P}^{-m}$ is contained in
\begin{equation*}
\begin{split}
F_\rho \cup \Omega_-^{\rho} \cup E^\ast_s \times E^\ast_u,
\end{split}
\end{equation*}
where $\Omega_-^{\rho}$ is defined by replacing the condition $t \geq 0$ by $t \geq \rho$ in the definition of $\Omega_-$ and
\begin{equation*}
\begin{split}
F_\rho = \{ (x,\xi; \phi_\rho(x),{}^T \mathrm{d}\phi_\rho(x) ^{-1} \xi))\ |\ (x,\xi) \in T^* M \}
\end{split}
\end{equation*}
In particular, provided that $\rho > 0$ is shorter than the length of all periodic orbits of the flow $\phi_t$, we see that the wave front set of $\mathcal{L}_\rho \p{z - P}^{-m}$ does not intersect $\Delta\p{T^* M}$ (the wave front set of an operator does not intersect the zero section), so that the flat trace of $\mathcal{L}_\rho \p{z - P}^{-m}$ is well-defined.

In order to compute this flat trace, we just notice that the argument used in \cite[Section 4]{DZdet} to compute the flat trace of $\mathcal{L}_\rho (z - P)^{-1}$ also applies to $\mathcal{L}_\rho (z - P)^{-m}$. Let us just mention that this argument is based on Guillemin trace formula \cite{guillemin-lectures-spectral-theory-77}: for $t > 0$ we have the distributional equality (see \cite[Appendix B]{DZdet} for details)
\[
\Tr^\flat \mathcal{L}_t = \sum_{\gamma} \frac{ T_\gamma^\# \exp\int_\gamma V }{|\det 1 - \mathcal{P}_\gamma|} \delta(t-T_\gamma).
\]
Consequently, recalling \eqref{eq:une_forme_integrale}, we find without surprise that
\begin{equation*}
\begin{split}
\Tr^\flat \p{\mathcal{L}_\rho (z-P)^{-m}} = \frac{1}{\p{m-1}!} \sum_{\gamma} \frac{T_\gamma^{\#} e^{\int_\gamma V}}{\va{\det\p{I - \mathcal{P}_\gamma}}} \p{T_\gamma - \rho}_+^{m-1} e^{- z \p{T_\gamma- \rho}}.
\end{split}
\end{equation*}
Hence, we only need to prove that
\begin{equation*}
\begin{split}
\lim_{\epsilon \to 0} \Tr_{|L^2}\p{I_\epsilon \mathcal{L}_\rho (z - P)^{-m} I_\epsilon} = \Tr^\flat \p{\mathcal{L}_\rho (z-P)^{-m}}.
\end{split}
\end{equation*}
Since for $\epsilon > 0$ the operator $I_\epsilon \mathcal{L}_\rho (z - P)^{-m} I_\epsilon$ has a $\mathcal{C}^\infty$ kernel, its trace acting on $L^2\p{M}$ coincides with its flat trace. Consequently, we want to prove that
\begin{equation}\label{eq:laconvergencequilfaut}
\begin{split}
\lim_{\epsilon \to 0} \Tr^{\flat}\p{I_\epsilon \mathcal{L}_\rho (z - P)^{-m} I_\epsilon} = \Tr^\flat \p{\mathcal{L}_\rho (z-P)^{-m}}.
\end{split}
\end{equation}

According to \cite[Definition 8.2.2 and Theorem 8.2.4]{Hormander-1}, in order to prove \eqref{eq:laconvergencequilfaut}, we only need to prove that the Schwartz kernel of $I_\epsilon \mathcal{L}_\rho (z - P)^{-m} I_\epsilon$ converges weakly to the kernel of $\mathcal{L}_\rho (z - P)^{-m}$ (when $\epsilon$ tends to $0$) and that the wave front set condition needed to define the flat trace is uniformly satisfied by the $I_\epsilon \mathcal{L}_\rho (z - P)^{-m} I_\epsilon$'s.

In order to check the wave front set condition, we just use the fact that the $I_\epsilon$'s form a family of pseudors uniformly of order $0$, so that the Schwartz kernel of $I_\epsilon \mathcal{L}_\rho (z - P)^{-m} I_\epsilon$ is the image of the kernel of $\mathcal{L}_\rho (z - P)^{-m}$ by a pseudor $J_\epsilon = I_\epsilon \otimes {}^t I_\epsilon$ uniformly of order $0$ when $\epsilon$ tends to $0$. To get the weak convergence, we just need to recall from the proof of Corollary \ref{cor:density-E1R} that $I_\epsilon$ converges pointwise to the identity on any Sobolev space on $M$ (see also Corollary \ref{corollary:sobolev}), so that $J_\epsilon$ has the same property on $M \times M$.
\end{proof}

With Lemma \ref{lmtraceresolvante}, we are ready to relate the dynamical determinant $\zeta_{X,V}$ with a regularized determinant. See \cite[Chapter XI]{Gohb} for the general theory of regularized determinants.

\begin{lemma}\label{lmfactorisation}
Let $z$ be a complex number with large and positive real part and $m$ be the smallest integer strictly larger than $ns$. Let $Q_z$ be the polynomial of order at most $m-1$
\begin{equation*}
\begin{split}
Q_z(\lambda) = - \sum_{\ell = 0}^{m-1} \p{\sum_{\gamma} \frac{T_\gamma^{\#} e^{\int_\gamma V} e^{- z T_\gamma} T_\gamma^{\ell - 1}}{\va{\det\p{I - \mathcal{P}_\gamma}}}}\frac{\p{z - \lambda}^{\ell}}{\ell!}.
\end{split}
\end{equation*}
Then for every $\lambda \in \C$ we have
\begin{equation*}
\begin{split}
\zeta_{X,V}(\lambda) = \textup{det}_m\p{I + (\lambda - z) \p{z - P}^{-1}} \exp\p{Q_z(\lambda)},
\end{split}
\end{equation*}
where $\textup{det}_m$ denotes the regularized determinant of order $m$. 
\end{lemma}

\begin{proof}
By analytic continuation principle, we only need to prove this result for $\lambda$ close to $z$. For such a $\lambda$ the regularized determinant is defined by 
\begin{equation*}
\begin{split}
\\ & \textup{det}_m\p{I + (\lambda - z) \p{z - P}^{-1}} \\& \qquad \qquad = \exp\p{- \sum_{\ell \geq m}\frac{\p{z- \lambda}^\ell}{\ell} \Tr\p{\p{z-P}^{-\ell}}} \\
     & \qquad \qquad = \exp\p{ - \sum_{\ell \geq m}\frac{(z-\lambda)^\ell}{\ell!} \sum_{\gamma} \frac{T_\gamma^{\#} e^{\int_\gamma V}}{\va{\det\p{I - \mathcal{P}_\gamma}}} T_\gamma^{\ell-1} e^{- z T_\gamma}} \\
    \intertext{According to Lemma \ref{lmtraceresolvante}, this is:}
     & \qquad \qquad = \exp\p{ - \sum_{\gamma}\frac{T_\gamma^{\#}}{T_\gamma} \frac{e^{\int_\gamma V} e^{- z T_\gamma}}{\va{\det\p{I - \mathcal{P}_\gamma}}} \sum_{\ell \geq m} \frac{\p{(z-\lambda)T_\gamma}^\ell}{\ell!}} \\
     & \qquad \qquad = \exp\p{ - \sum_{\gamma}\frac{T_\gamma^{\#}}{T_\gamma} \frac{e^{\int_\gamma V} e^{- z T_\gamma}}{\va{\det\p{I - \mathcal{P}_\gamma}}} \p{ e^{\p{z- \lambda} T_\gamma} - \sum_{\ell = 0}^{m-1} \frac{\p{(z-\lambda)T_\gamma}^\ell}{\ell!}}} \\
     & \qquad \qquad = \zeta_{X,V}(\lambda) e^{- Q_z\p{\lambda}}.
\end{split}
\end{equation*}
The applications of Fubini's Theorem are justified when $\Re z  \gg 1$ and $\va{z - \lambda}$ small enough by Margulis' bound.
\end{proof}

We are now ready to prove Theorem \ref{thm:main-strong}.

\begin{proof}[Proof of Theorem \ref{thm:main-strong}]
Let $m$ be as in Lemma \ref{lmfactorisation}. Recall the Weierstrass primary factor (the second expression is valid when $\va{\lambda} < 1$)
\begin{equation*}
\begin{split}
E(\lambda,m-1) = \p{1 - \lambda} \exp\p{\sum_{\ell = 1}^{m-1} \frac{1}{\ell} \lambda^\ell} = \exp\p{- \sum_{\ell = m}^{+ \infty} \frac{1}{\ell} \lambda^\ell}.
\end{split}
\end{equation*}
It follows from Lidskii's Trace Theorem that when $\Re z \gg 1$ and $\lambda \in \C$ we have
\begin{equation}\label{eq:factorisation_canonique}
\begin{split}
\textup{det}_m\p{I - (z - \lambda) (z-P)^{-1}} & = \prod_{k = 0}^{+ \infty} E\p{\frac{\lambda-z}{\lambda_k-z},m-1},
\end{split}
\end{equation}
where $(\lambda_k)_{k \in \N}$ denotes the sequence of Ruelle resonances of $P$. We want to use this expression with Lemma \ref{lmfactorisation} in order to prove Theorem \ref{thm:main-strong}, but let us make an observation first. If $\lambda$ is a complex number such that 
\begin{equation}\label{eq:proche_des_reels_positifs}
\begin{split}
\va{\frac{\lambda}{\va{\lambda}} - 1} \leq \frac{1}{2}
\end{split}
\end{equation}
then we have $\Re \lambda \geq \va{\lambda}/2$. Hence, using the expression \eqref{eq:zetaXV} and dominated convergence, we see that when $\va{\lambda}$ tends to $+ \infty$ while satisfying \eqref{eq:proche_des_reels_positifs}, the function $\zeta_{X,V}(\lambda)$ tends to $1$. In particular, $\zeta_{X,V}(\lambda)$ remains bounded when $\lambda$ satisfies \eqref{eq:proche_des_reels_positifs}. Hence, we may assume in the following that
\begin{equation}\label{eq:loin_des_reels_positifs}
\begin{split}
\va{\frac{\lambda}{\va{\lambda}} - 1} \geq \frac{1}{2},
\end{split}
\end{equation}
and that $\va{\lambda}$ is large of course. When $\va{\lambda}$ is large enough, we may apply Lemma \ref{lmfactorisation} with $z = \va{\lambda}$. Then, notice that $Q_{\va{\lambda}}(\lambda)$ tends to $0$ when $\va{\lambda}$ tends to $+ \infty$ and thus we may ignore the factor $\exp\p{Q_z(\lambda)}$ from Lemma \ref{lmfactorisation}. The other factor is given by \eqref{eq:factorisation_canonique} (with $z = \va{\lambda}$).

Notice that if $\va{\lambda_k} \geq 5 \va{\lambda}$ then $\va{\lambda - \va{\lambda}}/\va{\lambda_k - \va{\lambda}} \leq 1/2$ and hence
\begin{equation}\label{eq:lesgrands}
\begin{split}
\log \va{E\p{\frac{\lambda-\va{\lambda}}{\lambda_k-\va{\lambda}},m-1}} & \leq 2 \va{\frac{\lambda - \va{\lambda}}{\lambda_k - \va{\lambda}}}^m \leq 2^{m+1}\p{\frac{ \va{\lambda}}{\va{\lambda_k} - \va{\lambda}}}^m \\ & \leq 2  \p{\frac{5}{2}}^{m} \va{\frac{\lambda}{\lambda_k}}^{m}.
\end{split}
\end{equation}
On the other hand if $\va{\lambda_k} < 5 \va{\lambda}$, we have, since we assume \eqref{eq:loin_des_reels_positifs},
\begin{equation*}
\begin{split}
\va{\frac{\lambda - \va{\lambda}}{\lambda_k - \va{\lambda}}} & \geq \frac{1}{2} \frac{1}{\va{\frac{\lambda_k}{\va{\lambda}}-1}} \geq \frac{1}{12}.
\end{split}
\end{equation*} 
Hence, we have
\begin{equation}\label{eq:lautre_cas}
\begin{split}
\log \va{E\p{\frac{\lambda-\va{\lambda}}{\lambda_k-\va{\lambda}},m-1}} & \leq \log\va{1 - \frac{\lambda - \va{\lambda}}{\lambda_k - \va{\lambda}} } + \sum_{\ell = 1}^{m-1} \frac{1}{\ell} \va{\frac{\lambda - \va{\lambda}}{\lambda_k - \va{\lambda}}}^{\ell} \\
    & \leq \va{\frac{\lambda - \va{\lambda}}{\lambda_k - \va{\lambda}}} + \sum_{\ell = 1}^{m-1} \va{\frac{\lambda - \va{\lambda}}{\lambda_k - \va{\lambda}}}^{\ell} \\
    & \leq \p{\sum_{\ell = 0}^{m-1} 12^{m-1-\ell}} \va{\frac{\lambda - \va{\lambda}}{\lambda_k - \va{\lambda}}}^{m-1} \\
    & \leq 12^m \va{\frac{\lambda - \va{\lambda}}{\lambda_k - \va{\lambda}}}^{m-1}.
\end{split}
\end{equation}
Then, introduce a constant $C$ such that $\Re \lambda_k \leq C$ for all $k \in \N$ (such a constant exists because $P$ is the generator of a strongly continuous semi-group) and notice that for $\va{\lambda}$ large enough, we have
\begin{equation*}
\begin{split}
\va{\frac{\lambda - \va{\lambda}}{\lambda_k - \va{\lambda}}} \leq \frac{2}{\va{\frac{\lambda_k - C}{\va{\lambda}} - 1 + \frac{C}{\va{\lambda}}}} \leq \frac{2}{1 - \frac{C}{\va{\lambda}}} \leq 4.
\end{split}
\end{equation*}
And thus, \eqref{eq:lautre_cas} becomes:
\begin{equation}\label{eq:lespetits}
\begin{split}
\log \va{E\p{\frac{\lambda-\va{\lambda}}{\lambda_k-\va{\lambda}},m-1}} \leq \frac{48^m}{4}.
\end{split}
\end{equation}
Now, gathering \eqref{eq:lesgrands} and \eqref{eq:lespetits}, that are valid respectively when $\va{\lambda_k} \geq 5 \va{\lambda}$ and $\va{\lambda_k} < 5 \va{\lambda}$, we find that
\begin{equation}\label{eq:on_y_est_presque}
\begin{split}
& \log \va{\textup{det}_m\p{I - (\va{\lambda} - \lambda) (\va{\lambda}-P)^{-1}}} \\ & \qquad \qquad\leq 2 \times \p{\frac{5}{2}}^{m} \va{\lambda}^m \sum_{\va{\lambda_k} \geq 5 \va{\lambda}} \va{\lambda_k}^{-m} + \frac{48^m}{4} \# \set{k \in \N : \va{\lambda_k} < 5 \va{\lambda}}.
\end{split}
\end{equation}
Then, from the counting bound in Theorem \ref{thm:spectral-theory}, we see that
\begin{equation}\label{eq:chtok}
\begin{split}
\# \set{k \in \N : \va{\lambda_k} < 5 \va{\lambda}} = \O\p{\va{\lambda}^{n s}},
\end{split}
\end{equation}
and that,
\begin{equation}\label{eq:chtak}
\begin{split}
\sum_{\va{\lambda_k} \geq 5 \va{\lambda}} \va{\lambda_k}^{-m} = \O\p{\va{\lambda}^{n s-m}}.
\end{split}
\end{equation}
We end the proof of Theorem \ref{thm:main-strong} by plugging \eqref{eq:chtok} and \eqref{eq:chtak} in \eqref{eq:on_y_est_presque}.
\end{proof}

\section{Perturbative results}
\label{sec:perturbative-results}

\subsection{Ruelle resonances and stochastic perturbations}\label{sec:viscosite}

This section is dedicated to the proof of Theorem \ref{theorem:viscosite}. Recall that we are considering perturbations of $P$ of the form
\begin{equation*}
\begin{split}
P_\epsilon \coloneqq P + \epsilon \Delta,
\end{split}
\end{equation*}
where $\epsilon \geq 0$ and $\Delta$ is a $\G^s$ semi-classical pseudor, self-adjoint, negative and classically elliptic of order $m > 1$. The proof of Theorem \ref{theorem:viscosite} will be split into two parts. We will first prove that the spectrum of $P_\epsilon$ on $\mathcal{H}_\Lambda^0$ is discrete and coincides with $\sigma_{L^2}\p{P_\epsilon}$, and then prove the convergence results working directly on $\mathcal{H}_\Lambda^0$. In this section, we will need to work with other adapted Lagrangians than $\Lambda$ that we introduced in \S \ref{sec:spectral_theory}. We will denote these adapted Lagrangians by $\Lambda'$, keeping the notation $\Lambda$ for the Lagrangian that we introduced to study $P$. We start with a technical lemma.

\begin{lemma}\label{lemma:technique_P_epsilon}
Let $\tau_0$ be small enough and $\Lambda'$ be a $\p{\tau_0, s}$-adapted Lagrangian. Then, for $h$ small enough, for every $N > 0, \epsilon > 0, \lambda \in \C$ and $k \in \N$, there is a constant $C$ such that for every $u \in \mathcal{H}_{\Lambda'}^{km}$, we have
\begin{equation*}
\begin{split}
\n{u}_{\mathcal{H}_{\Lambda'}^{k m}} \leq C \p{\n{\p{P_\epsilon - \lambda}^k u}_{\mathcal{H}_{\Lambda'}^0} + \n{u}_{\mathcal{H}_{\Lambda'}^{-N}}}.
\end{split}
\end{equation*}
\end{lemma}

\begin{proof}
We use the assumption that $\tau_0$ and $h$ are small enough to be able to apply Proposition \ref{proptoeplitz} and write 
\begin{equation*}
\begin{split}
B_{\Lambda'} T_{\Lambda'} P S_{\Lambda'} B_{\Lambda'} = B_{\Lambda'} \sigma_{P} B_{\Lambda'} + R_1 \textup{ and } B_{\Lambda'} T_{\Lambda'} \Delta S_{\Lambda'} B_{\Lambda'} = B_{\Lambda'} \sigma_{\Delta} B_{\Lambda'} + R_2,
\end{split}
\end{equation*}
where $\sigma_P$ and $\sigma_{\Delta}$ are symbols on $\Lambda'$ (of order $1$ and $m$ respectively), and $R_1$ and $R_2$ are negligible operators. Then, provided that $\tau_0$ is small enough, it follows from the ellipticity of $\Delta$ that there is a constant $C > 0$ such that for every $\alpha \in \Lambda'$, we have
\begin{equation}\label{eq:ellipticite_de_Delta}
\begin{split}
\Re \sigma_{\Delta}(\alpha) \leq - \frac{1}{C} \jap{\va{\alpha}}^m + C.
\end{split}
\end{equation}
Then, applying Proposition \ref{propcomptoep}, we see that
\begin{equation}\label{eq:ecriture_dun_symbole}
\begin{split}
B_{\Lambda'} \jap{\va{\alpha}}^{mk} B_{\Lambda'} T_\Lambda' \p{P_\epsilon - \lambda}^k S_{\Lambda'} B_{\Lambda'} = B_{\Lambda'} \sigma_{\epsilon,k,\lambda} B_{\Lambda'} + R_3,
\end{split}
\end{equation}
where $R_3$ is negligible and $\sigma_{\epsilon,k,\lambda}$ is a symbol of order $2mk$. Moreover, it follows from \eqref{eq:ellipticite_de_Delta} that for some $C >0$ and every $\alpha \in \Lambda'$ we have
\begin{equation}\label{eq:borne_peu_uniforme}
\begin{split}
\Re \sigma_{\epsilon,k,\lambda}(\alpha) \leq - \frac{1}{C} \jap{\va{\alpha}}^{2mk} + C \jap{\va{\alpha}}^{-N + mk},
\end{split}
\end{equation}
where the constant $C >0$ is allowed to depend on $\epsilon,\Lambda', N,\lambda$ and $k$. By Cauchy--Schwarz inequality, we find that (the scalar product is in $L^2_0\p{\Lambda'}$)
\begin{equation}\label{eq:on_a_utilise_CS}
\begin{split}
\va{\Re \jap{\jap{\va{\alpha}}^{mk} T_{\Lambda'} \p{P_\epsilon - \lambda}^k u, T_{\Lambda'} u }} \leq \n{u}_{\mathcal{H}_{\Lambda'}^{mk}} \n{\p{P_\epsilon - \lambda}^k u}_{\mathcal{H}_{\Lambda'}^0}.
\end{split}
\end{equation}
Then, using \eqref{eq:ecriture_dun_symbole} and \eqref{eq:borne_peu_uniforme}, we find that (for some new constant $C >0$, we also apply Cauchy--Schwarz inequality)
\begin{equation*}
\begin{split}
\Re \jap{\jap{\va{\alpha}}^{mk} T_{\Lambda'} \p{P_\epsilon - \lambda}^k u, T_{\Lambda'} u } & = \Re \jap{\sigma_{\epsilon,k,\lambda} T_{\Lambda'} u, T_{\Lambda'}u} + \Re \jap{R_3 T_{\Lambda'} u , T_{\Lambda'} u } \\
       & \leq - \frac{1}{C} \n{u}_{\mathcal{H}_{\Lambda'}^{mk}}^2 + C \n{u}_{\mathcal{H}_{\Lambda'}^{-N}} \n{u}_{\mathcal{H}_{\Lambda'}^{mk}}.
\end{split}
\end{equation*}
Hence, with \eqref{eq:on_a_utilise_CS}, we find that
\begin{equation*}
\begin{split}
\frac{1}{C} \n{u}_{\mathcal{H}_{\Lambda'}^{mk}}^2  \leq \n{u}_{\mathcal{H}_{\Lambda'}^{mk}} \n{\p{P_\epsilon - \lambda}^k u}_{\mathcal{H}_{\Lambda'}^0} + C \n{u}_{\mathcal{H}_{\Lambda'}^{-N}} \n{u}_{\mathcal{H}_{\Lambda'}^{mk}},
\end{split}
\end{equation*}
and the result follows.
\end{proof}

We prove now that $P_\epsilon$ has still discrete spectrum after the Lagrangian perturbation of $T^* M$.

\begin{lemma}
Assume that $\tau_0$ is small enough and let $\Lambda'$ be a $\p{\tau_0,s}$-adapted Lagrangian. Then, for $h$ small enough and $\epsilon > 0$, the operator $P_\epsilon$ acting on $\mathcal{H}_{\Lambda'}^{0}$ with domain
\begin{equation*}
\begin{split}
D\p{P_\epsilon} = \set{u \in \mathcal{H}_{\Lambda'}^{0} : P_\epsilon u \in \mathcal{H}_{\Lambda'}^{0}} = \mathcal{H}_{\Lambda'}^{m}
\end{split}
\end{equation*}
has a discrete spectrum $\sigma_{\mathcal{H}_{\Lambda'}^0}\p{P_\epsilon}$ made of eigenvalues of finite multiplicity.
\end{lemma}

\begin{proof}
By a parametrix construction (using Proposition \ref{propcomptoep}), we prove that $D\p{P_\epsilon} = \mathcal{H}_{\Lambda'}^{m}$. As in the proof of Proposition \ref{prop:semigroup}, we use the negativity of the real part of the symbol of $\Delta$ to prove that the resolvent set of $P_\epsilon$ acting on $\mathcal{H}_{\Lambda'}^{0}$ is non-empty. Finally, we apply Lemma \ref{lemma:technique_P_epsilon} with $\lambda = 0, k = 1$ and $N = 0$ to see that the resolvent of $P_\epsilon$ sends $\mathcal{H}_{\Lambda'}^0$ continuously into $\mathcal{H}_{\Lambda'}^{m}$, and is consequently compact as an endomorphism on $\mathcal{H}_{\Lambda'}^{0}$ (recall Proposition \ref{propvalsing}).
\end{proof}

We prove now that the spectrum of an elliptic operator is unchanged under small Lagrangian deformations. The proof of Lemma \ref{lemma:equivalence_des_spectres} is an adaptation of the proof of \cite[Lemma 7.8]{Zworski-Galkowski-2}.

\begin{lemma}\label{lemma:equivalence_des_spectres}
Assume that $\tau_0$ is small enough and let $\Lambda'$ be a $\p{\tau_0,s}$-adapted Lagrangian. Then, for $h$ small enough and every $\epsilon > 0$, we have $\sigma_{\mathcal{H}_{\Lambda'}^0}\p{P_\epsilon} = \sigma_{L^2}\p{P_\epsilon}$. The (generalized) eigenvectors also coincide (in particular, they belong to both $L^2\p{M}$ and $\mathcal{H}_{\Lambda'}^0$).
\end{lemma}

\begin{proof}
We only need to prove that, for $\epsilon > 0$, the generalized eigenvectors of $P_\epsilon$ on $\mathcal{H}_{\Lambda'}^0$ and $L^2\p{M}$ coincide. Let us prove for instance that the generalized eigenvectors of $P_\epsilon$ that belong to $\mathcal{H}_{\Lambda'}^0$ also belong to $L^2\p{M}$ (the other way is similar). Let $G'$ be the symbol that defines $\Lambda'$, as in Definition \ref{def:adapted-Lagrangian}. Choose a $\mathcal{C}^\infty$ function $\chi : \R \to \left[0,1\right]$ such that $\chi(x) = 0$ for $x \leq 1$ and $\chi(x) = 1$ for $x \geq 2$. Then define for $r \in \R$ the symbol
\begin{equation*}
\begin{split}
G'_r(\alpha) = \chi\p{r \jap{\va{\alpha}}}G'(\alpha).
\end{split}
\end{equation*}
Notice that for $r > 0$ large enough $G'_r = G'$ and that $G'_0 = 0$. For $r \in \R$, let $\Lambda'_r$ denote the Lagrangian defined by $G'_r$, using \eqref{eq:adapted_Lagrangian}. Notice that, since $\Lambda'$ is $\p{\tau_0,s}$-adapted, then for some $c > 0$ the Lagrangians $\Lambda'_r$ are uniformly $\p{c \tau_0, s}$-adapted. Hence, if $\tau_0$ and $h$ are small enough, we can apply Lemma \ref{lemma:technique_P_epsilon} uniformly to the Lagrangians $\Lambda'_r$. Let $u \in \mathcal{H}_{\Lambda'}^0$ be such that there are $\lambda \in \C, k \in \N^*$ and $\epsilon > 0$ such that $\p{P_\epsilon - \lambda}^k u = 0$. We notice that, for every $r > 0$, the Lagrangians $\Lambda'_r$ and $\Lambda'$ coincide outside of a compact set. Hence, so do $T_{\Lambda'}u$ and $T_{\Lambda'_r} u$. Since $T_{\Lambda'_r} u$ is continuous, it is bounded on any compact set and consequently, for every $r > 0$, we have $u \in \mathcal{H}_{\Lambda'_r}^{0}$. Hence, we may apply Lemma \ref{lemma:technique_P_epsilon} to find a constant $C > 0$ such that for every $r > 0$ we have
\begin{equation*}
\begin{split}
\n{u}_{\mathcal{H}_{\Lambda'_r}^0} \leq C \n{u}_{\mathcal{H}_{\Lambda'_r}^{-1}}.
\end{split}
\end{equation*}
Here, the constant $C$ does not depend on $r > 0$, since the Lagrangians $\Lambda'_r$ satisfy uniformly the hypotheses of Lemma \ref{lemma:technique_P_epsilon}.
Then, for large $M > 0$, we have, for $r > 0$,
\begin{equation*}
\begin{split}
\n{u}_{\mathcal{H}_{\Lambda'_r}^0} & \leq C\p{ \n{\mathbf{1}_{\set{\jap{\va{\alpha}} \leq M}}\jap{\va{\alpha}}^{-1} T_{\Lambda'_r} u}_{L^2} + \n{\mathbf{1}_{\set{\jap{\va{\alpha}} > M}} \jap{\va{\alpha}}^{-1} T_{\Lambda'_r} u}_{L^2}} \\
       & \leq C\p{\n{\mathbf{1}_{\set{\jap{\va{\alpha}} \leq M}} \jap{\va{\alpha}}^{-1} T_{\Lambda'_r} u}_{L^2} + M^{-1} \n{u}_{\mathcal{H}_{\Lambda'_r}^0}}.
\end{split}
\end{equation*}
Thus, if $M \geq 2 C$, we find that
\begin{equation*}
\begin{split}
\n{u}_{\mathcal{H}_{\Lambda'_r}^0} \leq 2C \n{\mathbf{1}_{\set{\jap{\va{\alpha}} \leq M}} \jap{\va{\alpha}}^{-1} T_{\Lambda'_r} u}_{L^2} \leq C',
\end{split}
\end{equation*}
where the constant $C' > 0$ may depend on $u$ but not on $r$ (we only use the fact that $T u$ is continuous on a conical neighbourhood of $T^*M$ and hence bounded on compact sets). Since
\begin{equation*}
\begin{split}
\n{u}_{\mathcal{H}_{\Lambda'_r}^0} = \n{T_{\Lambda'_r} u}_{L^2},
\end{split}
\end{equation*}
we may apply Fatou's Lemma (after a change of variable to write the norm as the square root of an integral over $T^* M$ for instance) to find that (notice that $\Lambda'_0 = T^* M$)
\begin{equation*}
\begin{split}
\n{u}_{L^2\p{M}} = \n{u}_{\mathcal{H}_{\Lambda'_0}^0} \leq C' < + \infty.
\end{split}
\end{equation*}
Hence, $u$ belongs to $L^2\p{M}$, and the proof of the lemma is over.
\end{proof}

By choosing for $\Lambda'$ the adapted Lagrangian defined by the symbol $G'(\alpha) = - \tau_1 \jap{\va{\alpha}}^{s }$ with $\tau_1 \ll 1$, we deduce from Lemma \ref{lemma:equivalence_des_spectres} and Lemma \ref{lemma:wavefrontset_HLambda} the following result, that was already known for less general $\G^s$ pseudors (see for instance \cite[Théorème 4.3]{baouendiRegulariteAnalytiqueIteres1972}, we could adapt our proof to deal with any elliptic $\G^s$ pseudor in the sense of Definition \ref{def:gevrey_pseudor}). 

\begin{corollary}
There is $R > 0$ such that, for every $\epsilon > 0$, the $L^2$ eigenvectors of $P_\epsilon$ belongs to $E^{s,R}\p{M}$.
\end{corollary}

Now that we proved that the spectrum of $P_\epsilon$ is invariant under Lagrangian deformation (in the sense of Lemma \ref{lemma:equivalence_des_spectres}), we may go back to the case $\Lambda' = \Lambda$ (the Lagrangian deformation introduced in \S \ref{sec:spectral_theory} to study $P = P_0$). We assume in addition that $\tau_0$ and $h$ are small enough so that the machinery from the first chapter applies to both $P$ and $\Delta$. First, we prove that the hypo-ellipticity still holds for small $\epsilon \geq 0$.

\begin{lemma}\label{lemma:bornes_resolvantes}
Assume that $h$ and $\tau_0$ are small enough. Then, for every $k \in \R$, if $z$ is a large enough positive real number, then there is a constant $C > 0$ such that, for every $\epsilon \in \left[0,1\right]$ the number $z$ belongs to the resolvent set of $P_\epsilon$ acting on $\mathcal{H}_\Lambda^k$ and
\begin{equation}\label{eq:borne_uniforme_resolvante}
\begin{split}
\n{\p{z-P_\epsilon}^{-1}}_{\mathcal{H}_\Lambda^k \to \mathcal{H}_\Lambda^k} \leq C.
\end{split}
\end{equation}
Moreover, for every $\epsilon \in \left]0,1\right]$, the resolvent $\p{z-P_\epsilon}^{-1}$ is bounded from $\mathcal{H}_\Lambda^k$ to $\mathcal{H}_\Lambda^{m+k}$ and
\begin{equation}\label{eq:borne_resolvante_regularisante}
\begin{split}
\n{\p{z-P_\epsilon}^{-1}}_{\mathcal{H}_\Lambda^k \to \mathcal{H}_\Lambda^{k+m}} \leq \frac{C}{\epsilon}.
\end{split}
\end{equation}
\end{lemma}

\begin{proof}
We may write again (using Proposition \ref{proptoeplitz})
\begin{equation*}
\begin{split}
B_\Lambda T_\Lambda P S_\Lambda B_\Lambda = B_\Lambda p_\Lambda B_\Lambda + R_1 \textup{ and } B_\Lambda T_\Lambda \Delta S_\Lambda B_\Lambda = B_\Lambda \sigma_\Delta B_\Lambda + R_2,
\end{split}
\end{equation*}
where $p_\Lambda$ and $\sigma_\Delta$ are symbols of order respectively $1$ and $m$. The operators $R_1$ and $R_2$ are negligible. Moreover, $\Lambda$ has been constructed so that for some $C > 0$ and all $\alpha \in \Lambda$ we have $\Re p_\Lambda(\alpha) \leq C$ (see the proof of Proposition \ref{prop:semigroup}).
Furthermore, by ellipticity of $\Delta$, up to making $C$ larger, we also have, for every $\alpha \in \Lambda$,
\begin{equation*}
\begin{split}
\Re \sigma_\Delta(\alpha) \leq - \frac{1}{C} \jap{\va{\alpha}}^m + C.
\end{split}
\end{equation*}
Then, we apply Proposition \ref{proptoeplitz} to write
\begin{equation*}
\begin{split}
B_\Lambda \jap{\va{\alpha}}^{2k} T_\Lambda P S_\Lambda B_\Lambda = B_\Lambda \jap{\va{\alpha}}^{2k} \# p_\Lambda B_\Lambda + R_1'
\end{split}
\end{equation*}
and
\begin{equation*}
\begin{split}
B_\Lambda \jap{\va{\alpha}}^{2k} T_\Lambda \Delta S_\Lambda B_\Lambda = B_\Lambda \jap{\va{\alpha}}^{2k} \# \sigma_\Delta B_\Lambda + R_2',
\end{split}
\end{equation*}
where $R_1'$ and $R_2'$ are negligible. The symbols $\jap{\va{\alpha}}^{2k} \# p_\Lambda$ and $\jap{\va{\alpha}}^{2k} \# \sigma_{\Delta}$ are given at leading order by $p_\Lambda \jap{\va{\alpha}}^{2k}$ and $\sigma_{\Delta} \jap{\va{\alpha}}^{2k}$. Hence, for some large $C > 0$ and all $\alpha \in \Lambda$ we have
\begin{equation*}
\begin{split}
\Re \p{\jap{\va{\alpha}}^{2k} \# p_\Lambda} \p{\alpha} \leq C \jap{\va{\alpha}}^{2k} \textup{ and } \Re \p{\jap{\va{\alpha}}^{2k} \# \sigma_\Delta} (\alpha) \leq C.
\end{split}
\end{equation*}
Consequently, for $u \in \mathcal{H}_\Lambda^{m+k}$ we have that (up to making $C$ larger and with $\epsilon \in \left[0,1\right]$)
\begin{equation*}
\begin{split}
\Re \jap{ P_\epsilon u ,u }_{\mathcal{H}_\Lambda^k} & = \Re \jap{\jap{\va{\alpha}}^{2k} T_\Lambda P_\epsilon u , T_\Lambda u } \\
       &  = \Re \jap{ \p{\jap{\va{\alpha}}^{2k} \# p_\Lambda + \epsilon\jap{\va{\alpha}}^{2k} \# \sigma_\Delta} T_\Lambda u , T_\Lambda u} \\ & \qquad \qquad \qquad \qquad \qquad \qquad \qquad + \Re \jap{\p{R_1' + \epsilon R_2'} T_\Lambda u , T_\Lambda u} \\
       & \leq 2 C \jap{\jap{\va{\alpha}}^{2k} T_\Lambda u , T_\Lambda u} + C \n{u}_{\mathcal{H}_\Lambda^k}^2 \\
       & \leq 3 C \n{u}_{\mathcal{H}_\Lambda^k}^2.
\end{split}
\end{equation*}
As in the proof of Proposition \ref{prop:semigroup}, it follows that if $z$ is a real number such that $z > 3C$, then $z$ belongs to the resolvent set of $P_\epsilon$ and
\begin{equation*}
\begin{split}
\n{\p{z - P_\epsilon}^{-1}}_{\mathcal{H}_\Lambda^k \to \mathcal{H}_\Lambda^k} \leq \frac{1}{z - 3C}.
\end{split}
\end{equation*}
In particular, \eqref{eq:borne_uniforme_resolvante} holds for such a $z$. We turn now to the proof of \eqref{eq:borne_resolvante_regularisante}. To do so, we proceed as above, replacing the factor $\jap{\va{\alpha}}^{2k}$ by $\jap{\va{\alpha}}^{2 k + m}$, then as above we have
\begin{equation*}
\begin{split}
\Re \p{\jap{\va{\alpha}}^{2k+ m} \# p_\Lambda} \p{\alpha} \leq C \jap{\va{\alpha}}^{2k+ m}
\end{split}
\end{equation*} 
and
\begin{equation*}
\begin{split}
\Re \p{\jap{\va{\alpha}}^{2k+ m} \# \sigma_\Delta} (\alpha) \leq - \frac{1}{C} \jap{\va{\alpha}}^{2k + 2 m} + C \jap{\va{\alpha}}^{2k + m}.
\end{split}
\end{equation*}
It follows that (for a larger $C > 0$, the scalar product is in $L^2_0\p{\Lambda}$)
\begin{equation*}
\begin{split}
\Re \jap{\jap{\va{\alpha}}^{2k+m} T_\Lambda P_\epsilon u, T_\Lambda u } & \leq C \n{u}_{\mathcal{H}_\Lambda^k} \n{u}_{\mathcal{H}_\Lambda^{m+k}} - \frac{1}{C} \epsilon \n{u}_{\mathcal{H}_\Lambda^{m+k}}^2.
\end{split}
\end{equation*}
And, as in the proof of Lemma \ref{lmhypoell}, we find that
\begin{equation*}
\begin{split}
\epsilon \n{u}_{\mathcal{H}_\Lambda^{m+k}} \leq C\p{\n{P_\epsilon u}_{\mathcal{H}_\Lambda^{k}} + \n{u}_{\mathcal{H}_\Lambda^{k}}},
\end{split}
\end{equation*}
and \eqref{eq:borne_resolvante_regularisante} follows using \eqref{eq:borne_uniforme_resolvante}.
\end{proof}

We want now to use Lemma \ref{lemma:bornes_resolvantes} to deduce other resolvent bounds that are needed for the proof of Theorem \ref{theorem:viscosite}. The idea behind the proof of Lemma \ref{lemma:interpolation_implicite} below is that for $k_0,k_1 \in \R$ and $\theta \in \left[0,1\right]$ then the complex interpolation space $[\mathcal{H}_\Lambda^{k_0}, \mathcal{H}_\Lambda^{k_1} ]_{\theta}$ is $\mathcal{H}_\Lambda^{k_\theta}$ (with equivalent norms), where $k_\theta = \p{1 - \theta} k_0 + \theta k_1$. We will not prove this claim, since the interpolation bound that we need for the proof of Lemma \ref{lemma:interpolation_implicite} follows directly from H\"older's inequality.

\begin{lemma}\label{lemma:interpolation_implicite}
Assume that $h$ and $\tau_0$ are small enough. Let $z$ be a large enough positive real number. Then there is a constant $C> 0$ such that for every $\epsilon \in \left[0,1\right]$ we have (recall that $\delta = 1/s$)
\begin{equation}\label{eq:resolvantes_se_rapprochent}
\begin{split}
\n{\p{z- P}^{-1} - \p{z - P_\epsilon}^{-1}}_{\mathcal{H}_\Lambda^0 \to \mathcal{H}_\Lambda^0} \leq C \epsilon^{\frac{\delta}{m}},
\end{split}
\end{equation}
and
\begin{equation}\label{eq:uniforme_regularise}
\begin{split}
\n{\p{z - P_\epsilon}^{-1}}_{\mathcal{H}_\Lambda^0 \to \mathcal{H}_\Lambda^\delta} \leq C.
\end{split}
\end{equation}
\end{lemma}

\begin{proof}
Write
\begin{equation}\label{eq:differences_resolvantes}
\begin{split}
\p{z- P}^{-1} - \p{z- P_\epsilon}^{-1} = -\epsilon \p{z-P_\epsilon}^{-1} \Delta \p{z - P}^{-1}.
\end{split}
\end{equation}
From Proposition \ref{lemma:boundedness} and Lemma \ref{lmhypoell}, we know that $\Delta$ is bounded from $\mathcal{H}_\Lambda^{\delta}$ to $\mathcal{H}_\Lambda^{\delta - m}$ and that $(z - P)^{-1}$ is bounded from $\mathcal{H}_\Lambda^0$ to $\mathcal{H}_\Lambda^{\delta}$. Hence, we see with \eqref{eq:borne_resolvante_regularisante} that, for some $C > 0$ and every $\epsilon \in \left[0,1\right]$,
\begin{equation*}
\begin{split}
\n{\p{z-P_\epsilon}^{-1} \Delta \p{z - P}^{-1}}_{\mathcal{H}_\Lambda^0 \to \mathcal{H}_\Lambda^\delta} \leq \frac{C}{\epsilon},
\end{split}
\end{equation*}
and \eqref{eq:uniforme_regularise} follows by \eqref{eq:differences_resolvantes}.

Notice that, since $\Delta \p{z - P}^{-1}$ is bounded from $\mathcal{H}_\Lambda^0$ to $\mathcal{H}_\Lambda^{\delta - m}$, the estimate \eqref{eq:resolvantes_se_rapprochent} will follow from \eqref{eq:differences_resolvantes} if we are able to prove that, for some $C > 0$ and every $\epsilon \in \left[0,1\right]$,
\begin{equation}\label{eq:interpolation_resolvante}
\begin{split}
\n{\p{z- P_\epsilon}^{-1}}_{\mathcal{H}_\Lambda^{\delta - m} \to \mathcal{H}_\Lambda^0} \leq C \epsilon^{\frac{\delta}{m} - 1}.
\end{split}
\end{equation}
To see that \eqref{eq:interpolation_resolvante} holds, just notice that if $u \in \mathcal{H}_\Lambda^{\delta - m}$ then by H\"older's inequality we have
\begin{equation*}
\begin{split}
\n{\p{z - P_\epsilon}^{-1}u}_{\mathcal{H}_\Lambda^0} & \leq \n{\p{z - P_\epsilon}^{-1} u}_{\mathcal{H}_\Lambda^\delta}^{1 - \frac{\delta}{m}} \n{\p{z - P_\epsilon}^{-1} u}_{\mathcal{H}_\Lambda^{\delta - m}}^{\frac{\delta}{m}} 
\end{split}
\end{equation*}
and apply Lemma \ref{lemma:bornes_resolvantes}.
\end{proof}

We are now ready to prove Theorem \ref{theorem:viscosite}. 

\begin{proof}[Proof of Theorem \ref{theorem:viscosite}]
The proof relies on the results from \cite{bandtlowEstimatesNormsResolvents2004}. First recall that, $\lambda \in \sigma\p{P_\epsilon} \cup\set{\infty}$ if and only if $1/(z-\lambda)$ belongs to the spectrum of $\sigma((z - P_\epsilon)^{-1})$. Hence, we have
\begin{equation}\label{eq:chgt_de_variable_Hausdorff}
\begin{split}
d_{z,H}\p{\sigma\p{P}\cup\set{\infty},\sigma\p{P_\epsilon}\cup \set{\infty}} = d_{H}\p{\sigma\p{\p{z - P}^{-1}}, \sigma\p{\p{z - P_\epsilon}^{-1}}},
\end{split}
\end{equation}
where $d_{H}$ denotes the usual Hausdorff distance on compact subsets of $\C$. Then choose $p > n/\delta = ns$. It follows from Lemma \ref{lemma:interpolation_implicite} and Proposition \ref{propvalsing} that $\p{ z - P_\epsilon}^{-1}$ (seen as an endomorphism of $\mathcal{H}_\Lambda^0$) is uniformly in the Schatten class $S_p$ for $\epsilon \in \left[0,1\right]$. Consequently, it follows from \cite[Theorem 5.2]{bandtlowEstimatesNormsResolvents2004} and \eqref{eq:resolvantes_se_rapprochent} that, for some $C > 0$ and every $\epsilon \in \left[0,1\right]$,
\begin{equation*}
\begin{split}
d_H\p{\sigma\p{\p{z - P_\epsilon}^{-1}}, \sigma\p{\p{z - P}^{-1}}} \leq C f_p\p{\frac{\epsilon^{- \frac{\delta}{m}}}{C}}^{-1}.
\end{split}
\end{equation*}
Here, $f_p$ is the inverse of the function $ g_p : \R_+ \ni x \mapsto x \exp\p{a_p x^p + b_p} \in \R_+$ for some $a_p,b_p > 0$. We see that $f_p(x) \underset{x \to + \infty}{\sim} \p{\ln x}^{\frac{1}{p}}$. Hence, for some new constant $C > 0$, we have that
\begin{equation*}
\begin{split}
d_H\p{\sigma\p{\p{z - P_\epsilon}^{-1}}, \sigma\p{\p{z - P}^{-1}}} \leq C \va{\ln \epsilon}^{-\frac{1}{p}},
\end{split}
\end{equation*}
and the result follows from \eqref{eq:chgt_de_variable_Hausdorff}.
\end{proof}

\subsection{Linear response}\label{sec:linear_response}

We want now to apply our machinery to study how the Ruelle spectrum vary under a deterministic perturbation of the dynamics. To do so, we consider a perturbation 
\begin{equation*}
\begin{split}
\epsilon \mapsto X_\epsilon
\end{split}
\end{equation*}
of our vector field $X = X_0$. Here, the perturbation is defined for $\epsilon$ in a neighbourhood of $0$ and is assumed to be (at least) $\mathcal{C}^\infty$ from this neighbourhood of zero to a space of $\G^s$ sections of the tangent bundle of $M$ (see Remark \ref{remark:Gs_sections}). We can also consider a perturbation $\epsilon \mapsto V_\epsilon$ with the same features and then form for $\epsilon$ near zero the operator
\begin{equation}\label{eq:P_epsilon_lin_resp}
\begin{split}
P_\epsilon \coloneqq X_\epsilon + V_\epsilon.
\end{split}
\end{equation}
This kind of perturbation of the operator $P$ is very different from the one we considered in the last section.

Notice that, when $\epsilon$ is close enough to $0$, then the vector field $X_\epsilon$ generates an Anosov flow, so that the Ruelle spectrum of $P_\epsilon$ is well-defined. It is then natural to wonder how this spectrum varies with $\epsilon$. The situation is pretty well-understood in finite differentiability and in the $\mathcal{C}^\infty$ case (see \cite{Butterley-Liverani-07,bulicor,bonthonneauFlowindependentAnisotropicSpace2018}).

Let us consider the most simple example: in the $\mathcal{C}^\infty$ case, if $P_0$ has a simple resonances $\lambda_0$, then it will extend to a $\mathcal{C}^\infty$ family of simple resonances $\epsilon \mapsto \lambda_\epsilon$ (provided that the perturbation is $\mathcal{C}^\infty$). However, even if the perturbation $\epsilon \mapsto X_\epsilon$ is real-analytic in the $\mathcal{C}^\infty$ category, then we do not know that the family $\epsilon \mapsto \lambda_\epsilon$ is real-analytic (in fact, it is reasonable to expect that it is not). We will see below that if the perturbation $\epsilon \mapsto X_\epsilon$ is real-analytic in the real-analytic category, then the family $\epsilon \mapsto \lambda_\epsilon$ is real-analytic (this is an immediate consequence of Theorem \ref{thm:lanalyticite_engendre_lanalyticite}). Our first result is the following (the operator $P_\epsilon$ is the one defined by \eqref{eq:P_epsilon_lin_resp}), and the Lagrangian $\Lambda$ is the one that we introduced in \S \ref{sec:spectral_theory}.

\begin{prop}\label{prop:dependance_de_la_resolvante}
Assume that $\delta > 1/2$ (i.e. $s < 2$). Let $\ell \in \R_+ \setminus \N$. Assume that $h$ and $\tau$ are small enough. Then for $\epsilon$ small enough, the spectrum of $P_\epsilon$ acting on $\mathcal{H}_\Lambda^0$ is the Ruelle spectrum of $P_\epsilon$. Moreover, if $k =(\ell+1)\delta - \ell$ and $r$ is a large enough positive real number, then the the map
\begin{equation*}
\begin{split}
\epsilon \mapsto \p{r - P_{\epsilon}}^{-1} \in \mathcal{L}\p{\mathcal{H}_\Lambda^0, \mathcal{H}_\Lambda^{k}}
\end{split}
\end{equation*}
is $\mathcal{C}^\ell$ on a neighbourhood of $0$.
\end{prop}

As in the case of stochastic perturbations, we may prove a global bound on the way the spectrum of $P_\epsilon$ tends to the spectrum of $P$. Indeed, using the same arguments as in \S \ref{sec:viscosite}, it follows from Proposition \ref{prop:dependance_de_la_resolvante} that :

\begin{corollary}\label{cor:lin_resp_global}
Assume that $\delta > 1/2$ (i.e. $s < 2$). Then, for every $p > ns$ and $r \in \R_+$ large enough, there is a constant $C > 0$ such that for every $\epsilon$ close enough to $0$ we have
\begin{equation*}
\begin{split}
d_{r,H}\p{\sigma\p{P} \cup \set{\infty}, \sigma\p{P_\epsilon} \cup \set{\infty}} \leq C \va{\ln \va{\epsilon}}^{- \frac{1}{p}}.
\end{split}
\end{equation*}
\end{corollary} 

Finally, in the real-analytic case we are able to improve Proposition \ref{prop:dependance_de_la_resolvante} in the following way.

\begin{theorem}\label{thm:lanalyticite_engendre_lanalyticite}
Assume that $s=1$ and that the perturbations $\epsilon \mapsto X_\epsilon$ and $\epsilon \mapsto V_\epsilon$ are real-analytic. Assume that $h$ and $\tau$ are small enough. Then, for $r \in \R_+$ large enough, the map
\begin{equation*}
\begin{split}
\epsilon \mapsto \p{r - P_{\epsilon}}^{-1} \in \mathcal{L}\p{\mathcal{H}_\Lambda^0, \mathcal{H}_\Lambda^{1}}
\end{split}
\end{equation*}
is real-analytic on a neighbourhood of zero.
\end{theorem}

Theorem \ref{thm:lanalyticite_engendre_lanalyticite} allows to apply Kato theory on analytic perturbations of operators \cite{Kato} and to deduce in particular Theorem \ref{theorem:SRB}.

In order to prove Proposition \ref{prop:dependance_de_la_resolvante}, we recall that Proposition \ref{lemma:weak-mult} allows us to write for $\epsilon$ near $0$:
\begin{equation}\label{eq:mult_faible}
\begin{split}
T_\Lambda P_\epsilon S_\Lambda = p_{\Lambda,\epsilon} \Pi_\Lambda + R_\epsilon,
\end{split}
\end{equation}
where $p_{\Lambda,\epsilon}$ is the restriction to $\Lambda$ of an almost analytic extension of the principal symbol of $P_\epsilon$ and $R_\epsilon$ is a bounded operator from $L^2_0\p{\Lambda}$ to $L^2_{-\frac{1}{2}}\p{\Lambda}$ (the bound is uniform when $\epsilon$ remains in a neighbourhood of $0$).

When constructing $\Lambda$, we took care to ensure that the real part of $p_{\Lambda,0}$ is bounded from above. Unfortunately, it does not need to be true anymore when $\epsilon \neq 0$. This issue will create some technical difficulties: it is not clear anymore that $P_\epsilon$ is the generator of a strongly continuous semi-group on $\mathcal{H}_\Lambda^0$ for instance. This is why we need the additional assumption $\delta > 1/2$.

For technical reasons, we need to introduce a negative elliptic self-adjoint $\G^s$ semi-classical pseudor $\Delta$. We denote by $m$ the order of $\Delta$ and assume for convenience that $1 < m \leq 2 \delta$. Then for $\nu \geq 0$ and $\epsilon$ near $0$, we form the operator
\begin{equation*}
\begin{split}
P_{\epsilon,\nu} = P_\epsilon + \nu \Delta.
\end{split}
\end{equation*}
According to Lemma \ref{lemma:equivalence_des_spectres}, when $\nu > 0$, the operator $P_{\epsilon,\nu}$ has discrete spectrum on $\mathcal{H}_\Lambda^0$, and this spectrum converges to the Ruelle spectrum of $P$ when $\nu$ tends to $0$ (see the proof of Theorem \ref{theorem:viscosite}). Notice that we can also apply Proposition \ref{lemma:weak-mult} to $\Delta$ in order to write
\begin{equation}\label{eq:mult_faible_delta}
\begin{split}
T_\Lambda \Delta S_\Lambda = \sigma_{\Delta} \Pi_\Lambda + R_\Delta,
\end{split}
\end{equation}
where $\sigma_\Delta$ is a symbol of order $m$ and the operator $R_\Delta$ is bounded between the spaces $L^2_0\p{\Lambda}$ and $L^2_{- m + \frac{1}{2}}\p{\Lambda}$. We need now to prove the following key lemma.

\begin{lemma}\label{lemma:hypo_elliptique_symbole}
Assume that $\tau$ is small enough. There are $r_0 > 0$ and a constant $C > 0$, such that for every real number $r \geq r_0$, every $\alpha \in \Lambda$, every $\nu \geq 0$ and $\epsilon$ near $0$, we have
\begin{equation*}
\begin{split}
\va{r - p_{\Lambda,\epsilon} (\alpha) - \nu \sigma_{\Delta}(\alpha)} \geq \frac{1}{C}\p{r + \max\p{\jap{\va{\alpha}}^\delta, \nu \jap{\va{\alpha}}^m}}.
\end{split}
\end{equation*}
\end{lemma}

\begin{proof}
The definition of the Lagrangian $\Lambda$ ensures that there are constants $C, C_1 > 0$ such that for every $\alpha \in \Lambda$ we have
\begin{equation*}
\begin{split}
\Re \sigma_{\Delta} \leq - \frac{1}{C} \jap{\va{\alpha}}^m + C, \qquad \va{\Im \sigma_{\Delta}(\alpha)} \leq \frac{1}{C_1} \jap{\va{\alpha}}^m +C,
\end{split}
\end{equation*}
and
\begin{equation*}
\begin{split}
\va{\Re p_{\Lambda,\epsilon}(\alpha)} \leq \frac{1}{C_1} \jap{\va{\alpha}} +C.
\end{split}
\end{equation*}
Moreover, the constant $C_1 > 0$ may be chosen arbitrarily large by imposing $\tau$ to be small. In addition, the ellipticity conditions that are satisfied by $p_{\Lambda,0}$ are preserved under small perturbations, so that for $\epsilon$ small enough and up to making $C$ larger we have
\begin{equation}\label{eq:alternative_elliptique}
\begin{split}
\Re p_{\Lambda,\epsilon} (\alpha)\leq - \frac{1}{C} \jap{\va{\alpha}}^\delta + C \textup{ or } \va{\Im p_{\Lambda,\epsilon}} \geq \frac{1}{C} \jap{\va{\alpha}} - C.
\end{split}
\end{equation}

We start by writing that
\begin{equation}\label{eq:hashtag_Cauchy_Schwarz}
\begin{split}
\va{r - p_{\Lambda,\epsilon} - \nu \sigma_{\Delta}} \geq \frac{\sqrt{2}}{2}\p{ \va{r - \Re p_{\Lambda,\epsilon}  - \nu \Re \sigma_{\Delta}} + \va{\Im p_{\Lambda,\epsilon} + \nu \Im \sigma_{\Delta}}}.
\end{split}
\end{equation}
Now, if the first alternative holds in \eqref{eq:alternative_elliptique}, we just write that (for $\nu \leq 1$ and $r \geq 6 C$)
\begin{equation*}
\begin{split}
r - \Re p_{\Lambda,\epsilon}(\alpha)  - \nu \Re \sigma_{\Delta}(\alpha) & \geq r + \frac{1}{C} \jap{\va{\alpha}}^\delta - C + \frac{\nu}{C} \jap{\va{\alpha}}^m - C \\
        & \geq \min \p{\frac{1}{3}, \frac{1}{C}}\p{r+ \max\p{\jap{\va{\alpha}}^\delta, \nu \jap{\va{\alpha}}^m}}.
\end{split}
\end{equation*}
Hence, we may focus on the second case in \eqref{eq:alternative_elliptique}. In that case, we have
\begin{equation*}
\begin{split}
\va{\Im p_{\Lambda,\epsilon}\p{\alpha} + \nu \Im \sigma_{\Delta}\p{\alpha}} & \geq \va{\Im p_{\Lambda,\epsilon}\p{\alpha} } - \nu \va{\Im \sigma_{\Delta}\p{\alpha}} \\
    & \geq \frac{1}{C} \jap{\va{\alpha}} - \frac{\nu}{C_1} \jap{\va{\alpha}}^m - 2C.
\end{split}
\end{equation*}
On the other hand, we still have the general bound
\begin{equation*}
\begin{split}
\va{ r - \Re p_{\Lambda,\epsilon}\p{\alpha} - \nu \Re \sigma_{\Delta}\p{\alpha}} & \geq r - \frac{1}{C_1} \jap{\va{\alpha}} + \frac{\nu}{C} \jap{\va{\alpha}}^m - (1+\nu) C, 
\end{split}
\end{equation*}
and consequently \eqref{eq:hashtag_Cauchy_Schwarz} gives that
\begin{equation*}
\begin{split}
& \va{r - p_{\Lambda,\epsilon}\p{\alpha} - \nu \sigma_{\Delta}\p{\alpha}} \\ & \qquad \qquad \qquad \geq \frac{\sqrt{2}}{2}\p{r + \p{\frac{1}{C} - \frac{1}{C_1}} \jap{\va{\alpha}} + \nu \p{\frac{1}{C} - \frac{1}{C_1}} \jap{\va{\alpha}}^m - C\p{3 + \nu}}.
\end{split}
\end{equation*}
Then, by taking $\tau$ small enough, we ensure that $C_1 > C$, and the result follows (we get rid of the term $C\p{3 + \nu}$ by taking $r_0$ large enough).
\end{proof}

Now, we use Lemma \ref{lemma:hypo_elliptique_symbole} to get a new construction for the resolvent $\p{r- P_{\epsilon,\nu}}^{-1}$ (the positivity argument \emph{a priori} does not work when $\nu = 0$ and $\epsilon \neq 0$).

\begin{lemma}\label{lemma:des_resolvantes_partout}
Assume that $h$ and $\tau$ are small enough. Then there is $r_0$ such that, for $\epsilon$ and $\nu \geq 0$ small enough, if $r \geq r_0$ then $r$ belongs to the resolvent set of $P_{\epsilon,\nu}$ acting on $\mathcal{H}_\Lambda^0$ (with its natural domain). Moreover, for every $k \in \R$, if $r > 0$ is large enough then the resolvent $(r - P_{\epsilon,\nu})^{-1}$ is bounded from $\mathcal{H}_\Lambda^k$ to $\mathcal{H}_{\Lambda}^{k+\delta}$ with uniform bound in $\epsilon,\nu$.
\end{lemma}

\begin{proof}
Let $r \in \R_+$ be large enough. The formulae \eqref{eq:mult_faible} and \eqref{eq:mult_faible_delta} suggest to consider the following approximate inverse for $r - P_{\epsilon,\nu}$ :
\begin{equation*}
\begin{split}
A_{r,\epsilon,\nu} \coloneqq S_\Lambda \frac{1}{r - p_{\Lambda,\epsilon} - \nu \sigma_{\Delta}} T_\Lambda.
\end{split}
\end{equation*}
According to Lemma \ref{lemma:hypo_elliptique_symbole}, if $r$ is large enough and $\epsilon$ and $\nu$ are close enough to $0$, the operator $A_{r,\epsilon,\nu}$ is well-defined and is bounded from $\mathcal{H}_\Lambda^0$ to $\mathcal{H}_{\Lambda}^\delta$ and from $\mathcal{H}_\Lambda^{- \delta}$ to $\mathcal{H}_\Lambda^0$ with uniform bounds. Then, we compute
\begin{equation*}
\begin{split}
A_{r,\epsilon,\nu} \p{r - P_{\epsilon,\nu}} & =  S_\Lambda \frac{1}{r - p_{\Lambda,\epsilon} - \nu \sigma_{\Delta}} T_\Lambda (r - P_{\epsilon,\nu}) S_\Lambda T_\Lambda \\
    & = I + S_\Lambda \frac{1}{r - p_{\Lambda,\epsilon} - \nu \sigma_{\Delta}} \p{R_\epsilon + \nu R_{\Delta}} T_\Lambda.
\end{split}
\end{equation*}
In order to control the remainder term, notice using Lemma \ref{lemma:hypo_elliptique_symbole} that the multiplication by $(r - p_{\Lambda,\epsilon} - \nu \sigma_{\Delta})^{-1}$ is bounded from $L^2_{- \frac{1}{2}}\p{\Lambda}$ to $L^2_0\p{\Lambda}$ with norm
\begin{equation*}
\begin{split}
\sup_{\alpha \in \Lambda} \frac{\jap{\va{\alpha}}^{\frac{1}{2}}}{|r - p_{\Lambda,\epsilon}(\alpha) - \nu \sigma_{\Delta}(\alpha)|} \leq C \sup_{t \in \R_+} \frac{t}{r + t^{2 \delta}} \leq C \frac{\p{2 \delta - 1}^{1 - \frac{1}{2 \delta}}}{2 \delta} r^{\frac{1}{2 \delta} - 1}.
\end{split}
\end{equation*}
Similarly, if $0 < \nu \leq 1$, then the multiplication by $(r - p_{\Lambda,\epsilon} - \nu \sigma_{\Delta})^{-1}$ is bounded from $L^2_{- m + \frac{1}{2}}\p{\Lambda}$ to $L^2_0\p{\Lambda}$ with norm
\begin{equation*}
\begin{split}
\sup_{\alpha \in \Lambda} \frac{\jap{\va{\alpha}}^{m - \frac{1}{2}}}{|r - p_{\Lambda,\epsilon}(\alpha) - \nu \sigma_{\Delta}(\alpha)|} \leq \frac{C}{\nu} \sup_{t \in \R_+} \frac{t}{r + t^{\frac{m}{m - \frac{1}{2}}}} \leq C \frac{\p{2 m -1}^{1 - \frac{1}{2m}}}{2m} \frac{r^{-\frac{1}{2 m}}}{\nu}.
\end{split}
\end{equation*}
Hence, we see that the operator
\begin{equation*}
\begin{split}
S_\Lambda \frac{1}{r - p_{\Lambda,\epsilon} - \nu \sigma_{\Delta}} \p{R_\epsilon + \nu R_{\Delta}} T_\Lambda
\end{split}
\end{equation*}
is bounded from $\mathcal{H}_\Lambda^0$ to itself with norm an $\O(r^{-\frac{1}{2m}} + r^{\frac{1}{2 \delta} - 1})$ when $r \to + \infty$ (uniformly in $\epsilon$ and $\nu$ near $0$). Consequently, if $r$ is large enough, we may invert the operator
\begin{equation*}
\begin{split}
I + S_\Lambda \frac{1}{r - p_{\Lambda,\epsilon} - \nu \sigma_{\Delta}} \p{R_\epsilon + \nu R_{\Delta}} T_\Lambda
\end{split}
\end{equation*}
by mean of Neumann series and
\begin{equation*}
\begin{split}
\p{I + S_\Lambda \frac{1}{r - p_{\Lambda,\epsilon} - \nu \sigma_{\Delta}} \p{R_\epsilon + \nu R_{\Delta}} T_\Lambda}^{-1} A_{r,\epsilon,\nu}
\end{split}
\end{equation*}
is a left inverse for $r - P_{\epsilon,\nu}$. We construct similarly a right inverse for $r - P_{\epsilon,\nu}$. Indeed, in \eqref{eq:mult_faible} we may replace $p_{\Lambda,\epsilon} \Pi_\Lambda$ by $\Pi_{\Lambda} p_{\Lambda,\epsilon}$ (with a different remainder of course), and similarly in \eqref{eq:mult_faible_delta}.
Finally, if $r$ is large enough, it belongs to the resolvent set of $P_{\epsilon,\nu}$ for $\epsilon$ and $\nu$ near $0$, and the resolvent writes
\begin{equation}\label{eq:une_autre_resolvante}
\begin{split}
\p{ r - P_{\epsilon,\nu}}^{-1} = A_{r,\epsilon,\nu} \p{I + \widetilde{R}_{\epsilon,\nu}}^{-1},
\end{split}
\end{equation}
where the operator $\widetilde{R}_{\epsilon,\nu}$ is bounded on $\mathcal{H}_\Lambda^0$ with norm less than $\frac{1}{2}$. Proceeding as above, we see that, for $k \in \R$ and $r > 0$ large enough, the operator $\widetilde{R}_{\epsilon,\nu}$ is also bounded on $\mathcal{H}_\Lambda^{k}$ with norm less than $\frac{1}{2}$ and it follows from \eqref{eq:une_autre_resolvante} and Lemma \ref{lemma:hypo_elliptique_symbole} that $\p{ r- P_{\epsilon,\nu}}^{-1}$ is bounded from $\mathcal{H}_\Lambda^k$ to $\mathcal{H}_\Lambda^{k + \delta}$.
\end{proof}

\begin{proof}[Proof of Proposition \ref{prop:dependance_de_la_resolvante}]
First of all, from Lemma \ref{lemma:des_resolvantes_partout}, we know that the resolvent $\p{r- P_{\epsilon,\nu}}^{-1}$ is compact and hence $P_{\epsilon,\nu}$ has discrete spectrum on $\mathcal{H}_\Lambda^0$. This fact holds in particular for $P_{\epsilon} = P_{\epsilon,0}$. In order to see that the spectrum of $P_\epsilon$ acting on $\mathcal{H}_\Lambda^0$ coincides with its Ruelle spectrum, notice that
\begin{equation*}
\begin{split}
\p{r - P_{\epsilon,\nu}}^{-1} = \p{r - P_{\epsilon}}^{-1} + \nu \p{r - P_{\epsilon}}^{-1} \Delta \p{r - P_{\epsilon,\nu}}^{-1}.
\end{split}
\end{equation*}
Then, using Proposition \ref{lemma:boundedness} and Lemma \ref{lemma:des_resolvantes_partout} (recall that the order of $\Delta$ is less than $2 \delta$), we see that $\p{r - P_{\epsilon,\nu}}^{-1}$ converges to $\p{r - P_{\epsilon}}^{-1}$ in operator norm when $\nu$ tends to $0$. It follows that the spectrum of $P_{\epsilon,\nu}$ converges to the spectrum of $P_\epsilon$, but thanks to Lemmas \ref{lemma:equivalence_des_spectres} and \ref{lemma:interpolation_implicite}, the spectrum of $P_{\epsilon,\nu}$ converges to the Ruelle spectrum of $P_\epsilon$. Hence, the spectrum of $P_\epsilon$ acting on $\mathcal{H}_\Lambda^0$ is the Ruelle spectrum of $P_\epsilon$.

We prove now the regularity of the map $\epsilon \mapsto \p{r - P_\epsilon}^{-1}$. We start with the case $\ell \in \left]0,1\right[$ by writing
\begin{equation*}
\begin{split}
\p{r- P_{\epsilon'}}^{-1} = \p{ r- P_\epsilon}^{-1} + \p{r - P_\epsilon}^{-1}\p{P_{\epsilon'} - P_\epsilon} \p{ r - P_{\epsilon'}}^{-1}. 
\end{split}
\end{equation*}
Then we notice that, for some constant $C > 0$, we have
\begin{equation}\label{eq:distance_resolvante_lip}
\begin{split}
& \n{\p{r - P_{\epsilon'}}^{-1} - \p{r - P_\epsilon}^{-1}}_{\mathcal{H}_\Lambda^0 \to \mathcal{H}_\Lambda^{2 \delta - 1}} \\ & \qquad \qquad \quad \leq \n{\p{r - P_{\epsilon'}}^{-1}}_{\mathcal{H}_\Lambda^{\delta - 1} \to \mathcal{H}_\Lambda^{2 \delta - 1}} \n{P_\epsilon' - P_\epsilon}_{\mathcal{H}_\Lambda^\delta \to \mathcal{H}_{\Lambda}^{\delta - 1}} \n{\p{ r - P_{\epsilon'}}^{-1}}_{\mathcal{H}_\Lambda^0 \to \mathcal{H}_\Lambda^\delta} \\
     & \qquad \qquad \quad \leq C \va{\epsilon - \epsilon'}.
\end{split}
\end{equation}
Here, we applied Proposition \ref{lemma:boundedness} and Lemma \ref{lemma:des_resolvantes_partout}. It follows also from Lemma \ref{lemma:des_resolvantes_partout} that, up to making $C$ larger, we have
\begin{equation}\label{eq:distance_resolvante_bornee}
\begin{split}
\n{\p{r - P_{\epsilon'}}^{-1} - \p{r - P_\epsilon}^{-1}}_{\mathcal{H}_\Lambda^0 \to \mathcal{H}_\Lambda^{\delta}} \leq C. 
\end{split}
\end{equation}
Applying H\"older's inequality as in the proof of Lemma \ref{lemma:interpolation_implicite}, we deduce from \eqref{eq:distance_resolvante_lip} and \eqref{eq:distance_resolvante_bornee} that, for $\epsilon,\epsilon'$ near $0$, we have
\begin{equation*}
\begin{split}
\n{\p{r - P_{\epsilon'}}^{-1} - \p{r - P_\epsilon}^{-1}}_{\mathcal{H}_\Lambda^0 \to \mathcal{H}_\Lambda^{\p{1 - \ell}\delta + \ell \p{2 \delta - 1}}} \leq C \va{\epsilon - \epsilon'}^\ell. 
\end{split}
\end{equation*}
The result is then proved when $\ell \in \left]0,1\right[$.

We turn now to the case $\ell \in \left]1,2\right[$ (the general result will follow by induction). Letting $\dot{P}_\epsilon$ denotes the derivative of $\epsilon \mapsto P_\epsilon$, we write, for $\epsilon,\epsilon'$ near $0$,
\begin{equation}\label{eq:dev_lineaire_resolvante}
\begin{split}
& \p{r - P_{\epsilon'}}^{-1} - \p{ r - P_\epsilon}^{-1} - \p{\epsilon' - \epsilon} \p{r - P_\epsilon}^{-1} \dot{P}_\epsilon \p{r - P_\epsilon}^{-1} \\
     & \qquad \qquad \qquad = \p{ r - P_\epsilon}^{-1} \p{ P_{\epsilon'} - P_\epsilon - (\epsilon' - \epsilon) \dot{P}_\epsilon} \p{ r - P_\epsilon}^{-1} \\
     & \qquad \qquad \qquad \qquad \qquad + \p{ r - P_\epsilon}^{-1}\p{P_{\epsilon'} - P_\epsilon} \p{\p{r - P_{\epsilon'}}^{-1} - \p{r - P_{\epsilon}}^{-1}}.
\end{split}
\end{equation}
From Proposition \ref{lemma:boundedness} and Taylor's formula, we see that (provided $\tau$ is small enough)
\begin{equation}\label{eq:Taylor_ordre_2}
\begin{split}
\n{ P_{\epsilon'} - P_\epsilon - (\epsilon' - \epsilon) \dot{P}_\epsilon}_{\mathcal{H}_\Lambda^\delta \to \mathcal{H}_\Lambda^{\delta - 1}} \leq C \va{\epsilon - \epsilon'}^{\ell},
\end{split}
\end{equation}
and
\begin{equation}\label{eq:Taylor_ordre_1}
\begin{split}
\n{P_{\epsilon'} - P_\epsilon}_{\mathcal{H}_{\Lambda}^{k +1 - \delta} \to \mathcal{H}_\Lambda^{k - \delta}} \leq C \va{\epsilon - \epsilon'}.
\end{split}
\end{equation}
Then, we find, applying the previous case (with $\ell- 1$ instead of $\ell$), that
\begin{equation}\label{eq:holder_resolvante}
\begin{split}
\n{\p{r - P_\epsilon}^{-1} - \p{r - P_{\epsilon'}}^{-1}}_{\mathcal{H}_\Lambda^0 \to \mathcal{H}_\Lambda^{k+1 - \delta}} \leq C \va{\epsilon' - \epsilon}^{\ell - 1}.
\end{split}
\end{equation}
Putting \eqref{eq:Taylor_ordre_2}, \eqref{eq:Taylor_ordre_1}, \eqref{eq:holder_resolvante} and Lemma \ref{lemma:des_resolvantes_partout} in \eqref{eq:dev_lineaire_resolvante}, we find that, for a new constant $C > 0$ and $\epsilon,\epsilon'$ near $0$, we have
\begin{equation*}
\begin{split}
& \n{\p{r - P_{\epsilon'}}^{-1} - \p{ r - P_\epsilon}^{-1} - \p{\epsilon' - \epsilon} \p{r - P_\epsilon}^{-1} \dot{P}_\epsilon \p{r - P_\epsilon}^{-1}}_{\mathcal{H}_\Lambda^0 \to \mathcal{H}_\Lambda^k} \\ & \qquad \qquad \qquad \qquad \qquad \qquad \qquad \qquad \qquad \qquad \qquad \qquad \qquad \qquad \leq C \va{\epsilon' - \epsilon}^\ell.
\end{split}
\end{equation*}
It follows that the map $\epsilon \mapsto \p{r- P_\epsilon}^{-1} \in \mathcal{L}\p{\mathcal{H}_\Lambda^0, \mathcal{H}_\Lambda^k}$ is differentiable with derivative
\begin{equation}\label{eq:ceci_est_une_derivee}
\begin{split}
\epsilon \mapsto \p{r- P_\epsilon}^{-1} \dot{P}_\epsilon \p{r - P_\epsilon}^{-1}.
\end{split}
\end{equation}
Moreover, this derivative is $\p{\ell-1}$-H\"older, since, for $\epsilon,\epsilon'$ near $0$, we may write
\begin{equation*}
\begin{split}
& \p{\epsilon - \epsilon'} \p{\p{r- P_\epsilon}^{-1} \dot{P}_\epsilon \p{r - P_\epsilon}^{-1} - \p{r- P_{\epsilon'}}^{-1} \dot{P}_{\epsilon'} \p{r - P_{\epsilon'}}^{-1}} \\
      & \qquad \quad = \p{\p{r - P_{\epsilon}}^{-1} - \p{ r - P_{\epsilon'}}^{-1} - \p{\epsilon - \epsilon'} \p{r - P_{\epsilon'}}^{-1} \dot{P}_{\epsilon'} \p{r - P_{\epsilon'}}^{-1}} \\
      & \qquad \qquad \quad \quad  + \p{\p{r - P_{\epsilon'}}^{-1} - \p{ r - P_\epsilon}^{-1} - \p{\epsilon' - \epsilon} \p{r - P_\epsilon}^{-1} \dot{P}_\epsilon \p{r - P_\epsilon}^{-1}}.
\end{split}
\end{equation*}

To get the result in the case $\ell > 2$, we proceed by induction. By applying the case $\ell - 1$, we find as above that the map $\epsilon \mapsto \p{r- P_\epsilon}^{-1} \in \mathcal{L}\p{\mathcal{H}_\Lambda^0, \mathcal{H}_\Lambda^k}$ is $\lfloor \ell \rfloor - 1$ times differentiable. Moreover, its $(\lfloor l \rfloor - 1)$th derivative is given by a formula of the form 
\begin{equation*}
\begin{split}
\sum_{\ell_1 + \dots + \ell_r = \lfloor \ell \rfloor - 1} a_{\ell_1,\dots,\ell_r} \p{r - P_\epsilon}^{-1} \frac{\mathrm{d}^{\ell_1}}{\mathrm{d}\epsilon^{\ell_1}}\p{P_\epsilon} \dots \p{r - P_\epsilon}^{-1} \frac{\mathrm{d}^{\ell_r}}{\mathrm{d}\epsilon^{\ell_r}}\p{P_\epsilon} \p{r - P_\epsilon}^{-1},
\end{split}
\end{equation*}
where the $a_{\ell_1,\dots,\ell_r}$'s are integral coefficients. Now, reasoning as in the case $\ell \in \left]1,2\right[$, we see that each term in this sum is $\mathcal{C}^{\ell - \lfloor \ell \rfloor - 1}$ with the expected derivative, ending the proof of the proposition.
\end{proof}

\begin{proof}[Proof of Theorem \ref{thm:lanalyticite_engendre_lanalyticite}]
One could prove Theorem \ref{thm:lanalyticite_engendre_lanalyticite} by showing that $\epsilon \mapsto \p{r - P_\epsilon}^{-1}$ is the sum of its Taylor series at $0$ on a neighbourhood of $0$. However, we will rather rely on Cauchy's formula.

From the analyticity assumption on $\epsilon \mapsto P_\epsilon$, we know that we may extend this map to a holomorphic map on a neighbourhood of $0$. It follows then from Proposition \ref{lemma:boundedness} that $\epsilon \mapsto P_\epsilon$ defines a holomorphic family of operators from $\mathcal{H}_\Lambda^1$ to $\mathcal{H}_\Lambda^0$ on a neighbourhood of $0$ (in particular, it satisfies Cauchy's formula). Notice however that when $\epsilon$ is complex then $X_\epsilon$ is a first order differential operator that may not necessarily be interpreted as a vector field on $M$ (this is a section of $T M \otimes \C$).

Working as in the proof of Proposition \ref{prop:dependance_de_la_resolvante}, we see that if $r \in \R_+$ is large enough then, for $\epsilon$ in a complex neighbourhood of $0$, the operator $(r - P_\epsilon)^{-1}$ is well-defined and sends $\mathcal{H}_\Lambda^0$ into $\mathcal{H}_\Lambda^1$ with uniform bound. The fact that $X_\epsilon$ is not necessarily a vector field does not play any role here, since the proof of Proposition \ref{prop:dependance_de_la_resolvante} in the case $s=1$ only relies on the fact that $P_\epsilon$ is a small perturbation of $P_0$ (this is true since $\delta = 1$ so that the ellipticity condition on the real part of $p_{\Lambda,\epsilon}$ is stable by small perturbations). Consequently, we may define for $\epsilon$ near $0$ the operator
\begin{equation*}
\begin{split}
Q_\epsilon = \frac{1}{2 i \pi} \int_{\gamma} \frac{\p{r - P_w}^{-1}}{w - \epsilon} \mathrm{d}w,
\end{split}
\end{equation*}
where $\gamma$ is a small circle around $0$ in $\C$. By interverting series and integral, we see that $Q_\epsilon$ is holomorphic in $\epsilon$ near $0$. Moreover, for $\epsilon$ near $0$, we have
\begin{equation*}
\begin{split}
Q_\epsilon \p{ r- P_\epsilon} & = \frac{1}{2 i \pi}\int_\gamma \frac{\p{r-P_w}^{-1}}{w- \epsilon}\p{ \frac{1}{2i \pi} \int_\gamma \frac{\p{r -P_z}}{z-\epsilon} \mathrm{d}z } \mathrm{d}w \\
               & = \frac{1}{2 i \pi}\int_\gamma \frac{1}{w- \epsilon}\p{ \frac{1}{2i \pi} \int_\gamma \frac{1}{z-\epsilon} \mathrm{d}z } \mathrm{d}w \\ & \qquad \qquad \qquad \qquad + \frac{1}{2 i \pi}\int_\gamma \frac{\p{r-P_w}^{-1}}{w- \epsilon}\p{ \frac{1}{2i \pi} \int_\gamma \frac{\p{P_w -P_z}}{z-\epsilon} \mathrm{d}z } \mathrm{d}w \\
               & = I + \frac{1}{2 i \pi}\int_\gamma \frac{\p{r-P_w}^{-1}}{w- \epsilon}\p{P_w - P_\epsilon} \mathrm{d}w = I.
\end{split}
\end{equation*}
Thus, $Q_\epsilon = \p{r - P_\epsilon}^{-1}$ and the result follows.
\end{proof}

We deduce now Theorem \ref{theorem:SRB} from Theorem \ref{thm:lanalyticite_engendre_lanalyticite}.

\begin{proof}[Proof of Theorem \ref{thm:lanalyticite_engendre_lanalyticite}]
The uniqueness of the SRB measure for $\epsilon$ near $0$ follows from \cite[Lemma 5.1]{Butterley-Liverani-07} and the fact that Ruelle resonances are intrinsically defined.

From \cite[Lemma 5.1]{Butterley-Liverani-07}, we also know that for $\epsilon$ near $0$, the point $0$ is a simple eigenvalues for $X_\epsilon$ acting on $\mathcal{H}_\Lambda^0$, and that the associated normalized left eigenvector $\mu_\epsilon$ is the SRB measure for the Anosov flow generated by $X_\epsilon$. Noticing that $\mu_\epsilon$ is an eigenvector for $\p{r - X_\epsilon}^{-1}$ associated with the eigenvalue $r^{-1}$, it follows from Theorem \ref{thm:lanalyticite_engendre_lanalyticite} and Kato theory that the dependence of $\mu_\epsilon$ on $\epsilon$ is real-analytic, when we see $\mu_\epsilon$ as an element of the dual of $\mathcal{H}_\Lambda^0$. Now, for every $R > 0$, the space $E^{1,R}$ is continuously contained in $\mathcal{H}_\Lambda^0$ when $\tau$ is small enough, and the result follows.
\end{proof}

\bibliographystyle{alpha}
\bibliography{biblio}

\end{document}